\documentclass[12pt,reqno]{amsart}

\usepackage{amssymb}

\catcode`\@=11

\thinlines
\newskip\Einheit \Einheit=0.6cm
\newcount\xcoord \newcount\ycoord
\newdimen\xdim \newdimen\ydim \newdimen\PfadD@cke \newdimen\Pfadd@cke
\def\PfadDicke#1{\PfadD@cke#1 \divide\PfadD@cke by2 \Pfadd@cke\PfadD@cke \multiply\PfadD@cke by2}
\long\def\LOOP#1\REPEAT{\def\BODY{#1}\ITERATE}
\def\ITERATE{\BODY \let\next\ITERATE \else\let\next\relax\fi \next}
\let\REPEAT=\fi
\def\Punkt{\hbox{\raise-2pt\hbox to0pt{\hss\scriptsize$\bullet$\hss}}}
\def\DuennPunkt(#1,#2){\unskip
  \raise#2 \Einheit\hbox to0pt{\hskip#1 \Einheit
          \raise-2.5pt\hbox to0pt{\hss\normalsize$\bullet$\hss}\hss}}
\def\NormalPunkt(#1,#2){\unskip
  \raise#2 \Einheit\hbox to0pt{\hskip#1 \Einheit
          \raise-3pt\hbox to0pt{\hss\large$\bullet$\hss}\hss}}
\def\DickPunkt(#1,#2){\unskip
  \raise#2 \Einheit\hbox to0pt{\hskip#1 \Einheit
          \raise-4pt\hbox to0pt{\hss\Large$\bullet$\hss}\hss}}
\def\Kreis(#1,#2){\unskip
  \raise#2 \Einheit\hbox to0pt{\hskip#1 \Einheit
          \raise-4pt\hbox to0pt{\hss\Large$\circ$\hss}\hss}}
\def\Diagonale(#1,#2)#3{\unskip\leavevmode
  \xcoord#1\relax \ycoord#2\relax
      \raise\ycoord \Einheit\hbox to0pt{\hskip\xcoord \Einheit
         \unitlength\Einheit
         \line(1,1){#3}\hss}}
\def\AntiDiagonale(#1,#2)#3{\unskip\leavevmode
  \xcoord#1\relax \ycoord#2\relax 
      \raise\ycoord \Einheit\hbox to0pt{\hskip\xcoord \Einheit
         \unitlength\Einheit
         \line(1,-1){#3}\hss}}
\def\Pfad(#1,#2),#3\endPfad{\unskip\leavevmode
  \xcoord#1 \ycoord#2 \thicklines\ZeichnePfad#3\endPfad\thinlines}
\def\ZeichnePfad#1{\ifx#1\endPfad\let\next\relax
  \else\let\next\ZeichnePfad
    \ifnum#1=1
      \raise\ycoord \Einheit\hbox to0pt{\hskip\xcoord \Einheit
         \vrule height\Pfadd@cke width1 \Einheit depth\Pfadd@cke\hss}%
      \advance\xcoord by 1
    \else\ifnum#1=2
      \raise\ycoord \Einheit\hbox to0pt{\hskip\xcoord \Einheit
        \hbox{\hskip-\PfadD@cke\vrule height1 \Einheit width\PfadD@cke depth0pt}\hss}%
      \advance\ycoord by 1
    \else\ifnum#1=3
      \raise\ycoord \Einheit\hbox to0pt{\hskip\xcoord \Einheit
         \unitlength\Einheit
         \line(1,1){1}\hss}
      \advance\xcoord by 1
      \advance\ycoord by 1
    \else\ifnum#1=4
      \raise\ycoord \Einheit\hbox to0pt{\hskip\xcoord \Einheit
         \unitlength\Einheit
         \line(1,-1){1}\hss}
      \advance\xcoord by 1
      \advance\ycoord by -1
    \else\ifnum#1=5
      \advance\xcoord by -1
      \raise\ycoord \Einheit\hbox to0pt{\hskip\xcoord \Einheit
         \vrule height\Pfadd@cke width1 \Einheit depth\Pfadd@cke\hss}%
    \else\ifnum#1=6
      \advance\ycoord by -1
      \raise\ycoord \Einheit\hbox to0pt{\hskip\xcoord \Einheit
        \hbox{\hskip-\PfadD@cke\vrule height1 \Einheit width\PfadD@cke depth0pt}\hss}%
    \else\ifnum#1=7
      \advance\xcoord by -1
      \advance\ycoord by -1
      \raise\ycoord \Einheit\hbox to0pt{\hskip\xcoord \Einheit
         \unitlength\Einheit
         \line(1,1){1}\hss}
    \else\ifnum#1=8
      \advance\xcoord by -1
      \advance\ycoord by +1
      \raise\ycoord \Einheit\hbox to0pt{\hskip\xcoord \Einheit
         \unitlength\Einheit
         \line(1,-1){1}\hss}
    \fi\fi\fi\fi
    \fi\fi\fi\fi
  \fi\next}
\def\hSSchritt{\leavevmode\raise-.4pt\hbox to0pt{\hss.\hss}\hskip.2\Einheit
  \raise-.4pt\hbox to0pt{\hss.\hss}\hskip.2\Einheit
  \raise-.4pt\hbox to0pt{\hss.\hss}\hskip.2\Einheit
  \raise-.4pt\hbox to0pt{\hss.\hss}\hskip.2\Einheit
  \raise-.4pt\hbox to0pt{\hss.\hss}\hskip.2\Einheit}
\def\vSSchritt{\vbox{\baselineskip.2\Einheit\lineskiplimit0pt
\hbox{.}\hbox{.}\hbox{.}\hbox{.}\hbox{.}}}
\def\DSSchritt{\leavevmode\raise-.4pt\hbox to0pt{%
  \hbox to0pt{\hss.\hss}\hskip.2\Einheit
  \raise.2\Einheit\hbox to0pt{\hss.\hss}\hskip.2\Einheit
  \raise.4\Einheit\hbox to0pt{\hss.\hss}\hskip.2\Einheit
  \raise.6\Einheit\hbox to0pt{\hss.\hss}\hskip.2\Einheit
  \raise.8\Einheit\hbox to0pt{\hss.\hss}\hss}}
\def\dSSchritt{\leavevmode\raise-.4pt\hbox to0pt{%
  \hbox to0pt{\hss.\hss}\hskip.2\Einheit
  \raise-.2\Einheit\hbox to0pt{\hss.\hss}\hskip.2\Einheit
  \raise-.4\Einheit\hbox to0pt{\hss.\hss}\hskip.2\Einheit
  \raise-.6\Einheit\hbox to0pt{\hss.\hss}\hskip.2\Einheit
  \raise-.8\Einheit\hbox to0pt{\hss.\hss}\hss}}
\def\SPfad(#1,#2),#3\endSPfad{\unskip\leavevmode
  \xcoord#1 \ycoord#2 \ZeichneSPfad#3\endSPfad}
\def\ZeichneSPfad#1{\ifx#1\endSPfad\let\next\relax
  \else\let\next\ZeichneSPfad
    \ifnum#1=1
      \raise\ycoord \Einheit\hbox to0pt{\hskip\xcoord \Einheit
         \hSSchritt\hss}%
      \advance\xcoord by 1
    \else\ifnum#1=2
      \raise\ycoord \Einheit\hbox to0pt{\hskip\xcoord \Einheit
        \hbox{\hskip-2pt \vSSchritt}\hss}%
      \advance\ycoord by 1
    \else\ifnum#1=3
      \raise\ycoord \Einheit\hbox to0pt{\hskip\xcoord \Einheit
         \DSSchritt\hss}
      \advance\xcoord by 1
      \advance\ycoord by 1
    \else\ifnum#1=4
      \raise\ycoord \Einheit\hbox to0pt{\hskip\xcoord \Einheit
         \dSSchritt\hss}
      \advance\xcoord by 1
      \advance\ycoord by -1
    \else\ifnum#1=5
      \advance\xcoord by -1
      \raise\ycoord \Einheit\hbox to0pt{\hskip\xcoord \Einheit
         \hSSchritt\hss}%
    \else\ifnum#1=6
      \advance\ycoord by -1
      \raise\ycoord \Einheit\hbox to0pt{\hskip\xcoord \Einheit
        \hbox{\hskip-2pt \vSSchritt}\hss}%
    \else\ifnum#1=7
      \advance\xcoord by -1
      \advance\ycoord by -1
      \raise\ycoord \Einheit\hbox to0pt{\hskip\xcoord \Einheit
         \DSSchritt\hss}
    \else\ifnum#1=8
      \advance\xcoord by -1
      \advance\ycoord by 1
      \raise\ycoord \Einheit\hbox to0pt{\hskip\xcoord \Einheit
         \dSSchritt\hss}
    \fi\fi\fi\fi
    \fi\fi\fi\fi
  \fi\next}
\def\Koordinatenachsen(#1,#2){\unskip
 \hbox to0pt{\hskip-.5pt\vrule height#2 \Einheit width.5pt depth1 \Einheit}%
 \hbox to0pt{\hskip-1 \Einheit \xcoord#1 \advance\xcoord by1
    \vrule height0.25pt width\xcoord \Einheit depth0.25pt\hss}}
\def\Koordinatenachsen(#1,#2)(#3,#4){\unskip
 \hbox to0pt{\hskip-.5pt \ycoord-#4 \advance\ycoord by1
    \vrule height#2 \Einheit width.5pt depth\ycoord \Einheit}%
 \hbox to0pt{\hskip-1 \Einheit \hskip#3\Einheit 
    \xcoord#1 \advance\xcoord by1 \advance\xcoord by-#3 
    \vrule height0.25pt width\xcoord \Einheit depth0.25pt\hss}}
\def\Gitter(#1,#2){\unskip \xcoord0 \ycoord0 \leavevmode
  \LOOP\ifnum\ycoord<#2
    \loop\ifnum\xcoord<#1
      \raise\ycoord \Einheit\hbox to0pt{\hskip\xcoord \Einheit\Punkt\hss}%
      \advance\xcoord by1
    \repeat
    \xcoord0
    \advance\ycoord by1
  \REPEAT}
\def\Gitter(#1,#2)(#3,#4){\unskip \xcoord#3 \ycoord#4 \leavevmode
  \LOOP\ifnum\ycoord<#2
    \loop\ifnum\xcoord<#1
      \raise\ycoord \Einheit\hbox to0pt{\hskip\xcoord \Einheit\Punkt\hss}%
      \advance\xcoord by1
    \repeat
    \xcoord#3
    \advance\ycoord by1
  \REPEAT}
\def\Label#1#2(#3,#4){\unskip \xdim#3 \Einheit \ydim#4 \Einheit
  \def\lo{\advance\xdim by-.5 \Einheit \advance\ydim by.5 \Einheit}%
  \def\llo{\advance\xdim by-.25cm \advance\ydim by.5 \Einheit}%
  \def\loo{\advance\xdim by-.5 \Einheit \advance\ydim by.25cm}%
  \def\o{\advance\ydim by.25cm}%
  \def\ro{\advance\xdim by.5 \Einheit \advance\ydim by.5 \Einheit}%
  \def\rro{\advance\xdim by.25cm \advance\ydim by.5 \Einheit}%
  \def\roo{\advance\xdim by.5 \Einheit \advance\ydim by.25cm}%
  \def\l{\advance\xdim by-.30cm}%
  \def\r{\advance\xdim by.30cm}%
  \def\lu{\advance\xdim by-.5 \Einheit \advance\ydim by-.6 \Einheit}%
  \def\llu{\advance\xdim by-.25cm \advance\ydim by-.6 \Einheit}%
  \def\luu{\advance\xdim by-.5 \Einheit \advance\ydim by-.30cm}%
  \def\u{\advance\ydim by-.30cm}%
  \def\ru{\advance\xdim by.5 \Einheit \advance\ydim by-.6 \Einheit}%
  \def\rru{\advance\xdim by.25cm \advance\ydim by-.6 \Einheit}%
  \def\ruu{\advance\xdim by.5 \Einheit \advance\ydim by-.30cm}%
  #1\raise\ydim\hbox to0pt{\hskip\xdim
     \vbox to0pt{\vss\hbox to0pt{\hss$#2$\hss}\vss}\hss}%
}
\catcode`\@=12

\usepackage{color}

\input texdraw

\def\Ringerl(#1 #2){\move(#1 #2)\fcir f:0 r:.15}
\def\ringerl(#1 #2){\move(#1 #2)\fcir f:0 r:.1}
\def\Mark(#1 #2){\move(#1 #2)\fcir f:0 r:.2}

\newtheorem{theorem}{Theorem}
\newtheorem{proposition}[theorem]{Proposition}
\newtheorem{lemma}[theorem]{Lemma}
\newtheorem{corollary}[theorem]{Corollary}

\theoremstyle{remark}
\newtheorem{remark}{Remark}

\numberwithin{equation}{section}

\newcounter{saveeqn}
\newcommand{\alphaeqn}{\setcounter{saveeqn}{\value{equation}}%
\setcounter{equation}{0}%
\global\def\theequation{\mbox{\thesection.\arabic{saveeqn}\alph{equation}}}}
\newcommand{\reseteqn}{\setcounter{equation}{\value{saveeqn}}%
\global\def\theequation{\thesection.\arabic{equation}}}

\def\({\left(}
\def\){\right)}

\def\Fix{\operatorname{Fix}}
\def\Cent{\operatorname{Cent}}
\def\Cat{\operatorname{Cat}}

\def\rk{\operatorname{rk}}

\def\ep{\varepsilon}
\def\al{\alpha}
\def\be{\beta}
\def\ga{\gamma}

\begin{document}

\title[Cyclic sieving for generalised non-crossing partitions]
{Cyclic sieving for generalised non-crossing partitions associated
with complex reflection groups of exceptional type --- the details}
\newbox\Aut
\setbox\Aut\vbox{
\centerline{\sc 
Christian Krattenthaler$^{\dagger}$ \rm and \sc Thomas W.
M\"uller$^\ddagger$}
\vskip18pt
\centerline{$^\dagger$ Fakult\"at f\"ur Mathematik, Universit\"at Wien,}
\centerline{Nordbergstra\ss e 15, A-1090 Vienna, Austria.}
\centerline{\footnotesize WWW: \footnotesize\tt
http://www.mat.univie.ac.at/\lower0.5ex\hbox{\~{}}kratt} 
\vskip18pt
\centerline{$^\ddagger$ School of Mathematical Sciences,}
\centerline{Queen Mary \& Westfield College, University of London,}
\centerline{Mile End Road, London E1 4NS, United Kingdom.}
\centerline{\footnotesize WWW: \tt http://www.maths.qmw.ac.uk/\lower0.5ex\hbox{\~{}}twm/}
}
\author[C. Krattenthaler and T. W. M\"uller]{\box\Aut}

\address{Fakult\"at f\"ur Mathematik, Universit\"at Wien,
Nordbergstra{\ss}e~15, A-1090 Vienna, Austria.
WWW: \tt http://www.mat.univie.ac.at/\lower0.5ex\hbox{\~{}}kratt.}

\address{School of Mathematical Sciences, Queen Mary
\& Westfield College, University of London,
Mile End Road, London E1 4NS, United Kingdom.\newline
\tt http://www.maths.qmw.ac.uk/\lower0.5ex\hbox{\~{}}twm/.}

\keywords{complex reflection groups, unitary reflection groups,
$m$-divisible non-crossing partitions,
generalised non-crossing partitions, 
Fu\ss--Catalan numbers, cyclic sieving}

\dedicatory{Dedicated to the memory of Herb Wilf}

\subjclass [2000]{Primary 05E15; Secondary 05A10 05A15 05A18 06A07 20F55}

\thanks{$^\dagger$Research partially supported by the Austrian
Science Foundation FWF, grants Z130-N13 and S9607-N13,
the latter in the framework of the National Research Network
``Analytic Combinatorics and Probabilistic Number Theory."\newline
\indent
$^\ddagger$Research supported by the Austrian
Science Foundation FWF, Lise Meitner grant M1201-N13}

\begin{abstract}
We prove that the generalised non-crossing partitions associated with
well-generated complex reflection groups of exceptional type
obey two different cyclic sieving phenomena, as conjectured by
Armstrong, and by Bessis and Reiner.
This manuscript accompanies the paper 
{\it``Cyclic sieving for generalised non-crossing partitions associated
with complex reflection groups of exceptional type"}
[{\tt ar$\chi$iv:1001.0028}], for which it provides
the computational details.
\end{abstract}

\maketitle

\section{Introduction}

In his memoir \cite{ArmDAA}, Armstrong introduced {\em generalised
non-crossing partitions} associated with finite (real) reflection groups,
thereby embedding Kreweras' non-crossing partitions \cite{KrewAC},
Edelman's $m$-divisible non-crossing partitions \cite{EdelAA}, 
the non-crossing partitions associated with reflection groups due to
Bessis \cite{BesDAA} and Brady and Watt \cite{BRWaAA} into one
uniform framework. Bessis and Reiner \cite{BeReAA} observed that
Armstrong's definition can be straightforwardly extended to {\em
well-generated complex reflection groups} (see Section~\ref{sec:prel} 
for the precise definition). 
These generalised non-crossing partitions possess a wealth of
beautiful properties, and they display deep and surprising 
relations to other combinatorial objects defined for reflection
groups (such as the generalised cluster complex of Fomin and
Reading \cite{FoReAA}, or the extended Shi arrangement and
the geometric multichains of filters
of Athanasiadis \cite{AthaAG,AthaAH}); 
see Armstrong's memoir \cite{ArmDAA} and the references given therein.

On the other hand,
{\it cyclic sieving} is a phenomenon brought to light by Reiner,
Stanton and White \cite{ReSWAA}.
It extends the so-called
``$(-1)$-phenomenon" of Stembridge \cite{StemAL,StemAP}. 
Cyclic sieving can be
defined in three equivalent ways (cf.\ \cite[Prop.~2.1]{ReSWAA}). 
The one which gives the name 
can be described as follows: given a set $S$ of combinatorial
objects, an action on $S$ of a cyclic group $G=\langle g\rangle$ with
generator $g$ of order $n$, and a polynomial $P(q)$ in $q$
with non-negative integer coefficients, we say
that the triple $(S,P,G)$ {\it exhibits the cyclic sieving
phenomenon}, if the number of elements of $S$ fixed by $g^k$ equals
$P(e^{2\pi ik/n})$. In \cite{ReSWAA} it is shown that this phenomenon
occurs in surprisingly many contexts, and several further instances
have been discovered since then.

In \cite[Conj.~5.4.7]{ArmDAA} (also appearing in \cite[Conj.~6.4]{BeReAA}) and 
\cite[Conj.~6.5]{BeReAA},
Armstrong, respectively Bessis and Reiner, conjecture that generalised
non-crossing partitions for irreducible 
well-generated complex reflection groups
exhibit two different cyclic sieving phenomena (see
Sections~\ref{sec:siev1} and \ref{sec:siev2} for the precise
statements). 

According to the classification of these groups due to
Shephard and Todd \cite{ShToAA}, there are two infinite families of
irreducible well-generated complex reflection groups, namely the
groups $G(d,1,n)$ and $G(e,e,n)$, where $n,d,e$ are positive integers, 
and there are 26 exceptional groups.
For the infinite families of types $G(d,1,n)$ and $G(e,e,n)$, the two
cyclic sieving conjectures follow from the results
in \cite{KratCG}. 

The purpose of the present article is to prove the cyclic sieving
conjectures of Armstrong, and of Bessis and Reiner, for the
26 exceptional types, thus completing the proof of these conjectures. 
Since the generalised non-crossing partitions
feature a parameter $m$, from the outset 
this is {\it not\/} a finite problem. Consequently, we first need
several auxiliary results to reduce the conjectures for each of the
26 exceptional types to a {\it finite} problem. Subsequently, we use
Stembridge's {\sl Maple} package {\tt coxeter} \cite{StemAZ} and the
{\sl GAP} package {\tt CHEVIE} \cite{chevAA,MichAA} to carry out the remaining
{\it finite} computations. It is interesting to observe that, for the
verification of the type $E_8$ case, it is essential to use the
decomposition numbers in the sense of \cite{KratCB,KratCF,KrMuAB} because,
otherwise, the necessary computations would not be feasible 
in reasonable time with the currently available computer facilities.
We point out that, for the special case where the aforementioned
parameter $m$ is equal to $1$, the first cyclic sieving conjecture
has been proved in a uniform fashion by Bessis and Reiner in \cite{BeReAA}. 
(See \cite{ArSTAA} for a --- non-uniform --- proof of cyclic sieving
for non-crossing partitions associated with {\it real\/} reflection groups
under the action of the so-called Kreweras map, a special case of
the second cyclic sieving phenomenon discussed in the present paper.)
The crucial result on which the proof of Bessis and Reiner 
is based is \eqref{eq:7}
below, and it plays an important role in our reduction of the
conjectures for the 26 exceptional groups to a finite problem.

Our paper is organised as follows.
In the next section, we recall the definition of generalised
non-crossing partitions for well-generated complex reflection
groups and of decomposition numbers in the sense of
\cite{KratCB,KratCF,KrMuAB}, and we review some basic facts. 
The first cyclic sieving conjecture is subsequently stated in
Section~\ref{sec:siev1}. In Section~\ref{sec:pol}, we outline an
elementary proof that the $q$-Fu\ss--Catalan number, which is
the polynomial $P$ in the cyclic sieving phenomena concerning
the generalised non-crossing partitions for well-generated 
complex reflection groups, is always a polynomial with non-negative
integer coefficients, as required by the definition of cyclic
sieving. 
(The reader is referred to the first paragraph of 
Section~\ref{sec:pol} for comments on other
approaches for establishing polynomiality with non-negative
coefficients.)
Section~\ref{sec:aux1} contains the
announced auxiliary results which, for the 26 exceptional types, 
allow a reduction of the conjecture to a finite problem. 
The remaining case-by-case verification of the conjecture is then
carried out in Section~\ref{sec:Beweis1}. 
The second cyclic sieving conjecture is stated in
Section~\ref{sec:siev2}. Section~\ref{sec:aux2} contains the
auxiliary results which, for the 26 exceptional types, 
allow a reduction of the conjecture to a finite problem, while
Section~\ref{sec:Beweis2} contains
the remaining case-by-case verification of the conjecture.

\section{Preliminaries}
\label{sec:prel}

A {\it complex reflection group} is a 
group generated by (complex) reflections in
$\mathbb C^n$. (Here, a reflection is a non-trivial element of
$GL_n(\mathbb C)$ which fixes a hyperplane pointwise and which 
has finite order.) We refer to \cite{LeTaAA} for an in-depth 
exposition of the theory complex reflection groups.

Shephard and Todd provided a complete classification of all {\it
finite} complex reflection groups in \cite{ShToAA} (see also
\cite[Ch.~8]{LeTaAA}). According to this
classification, an arbitrary complex reflection group $W$ decomposes into
a direct product of {\it irreducible} complex reflection groups, 
acting on mutually orthogonal subspaces of the complex vector space
on which $W$ is acting. Moreover, the list of irreducible complex
reflection groups consists of the infinite family of groups
$G(m,p,n)$, where $m,p,n$ are positive integers, and $34$ exceptional
groups, denoted $G_4,G_5,\dots,G_{37}$ by Shephard and Todd.

In this paper, we are only interested in finite complex
reflection groups which are {\it well-generated}. 
A complex reflection group of rank $n$ is called {\it well-generated\/} if
it is generated by $n$ reflections.\footnote{We refer to
\cite[Def.~1.29]{LeTaAA} for the precise definition of ``rank."
Roughly speaking, the rank of a complex reflection group $W$
is the minimal $n$ such that
$W$ can be realized as reflection group on $\mathbb C^n$.}
Well-generation can be equivalently characterised by a duality
property due to Orlik and Solomon \cite{OS}. Namely, a
complex reflection group of rank $n$ has two sets of distinguished integers
$d_1\leq d_2\leq \cdots \leq d_n$ and $d_1^*\geq d_2^*\geq \cdots \geq
d_n^*$, called its {\it degrees} and {\it codegrees}, respectively
(see \cite[p.~51 and Def.~10.27]{LeTaAA}). Orlik and 
Solomon observed, using case-by-case checking, that an irreducible
complex reflection group $W$ of rank $n$ is
well-generated if and only if its degrees and codegrees satisfy        
\begin{equation*}
d_i+d_i^*=d_n
\end{equation*}
for all $i=1,2,\dots,n$. The reader is referred to 
\cite[App.~D.2]{LeTaAA} for a table of
the degrees and codegrees of all irreducible complex reflection
groups. Together with the classification of 
Shephard and Todd \cite{ShToAA}, this constitutes a classification of 
well-generated complex reflection groups: 
the irreducible
well-generated complex reflection groups are 
\begin{enumerate}
\item[---]the two infinite
families $G(d,1,n)$ and $G(e,e,n)$, where $d,e,n$ are positive
integers, 
\item[---]the exceptional groups 
$G_4,G_5,G_6,G_8,G_9,G_{10},G_{14},G_{16},G_{17},G_{18},G_{20},G_{21}$ 
of\break rank~$2$,
\item[---]the exceptional groups 
$G_{23}=H_3,G_{24},G_{25},G_{26},G_{27}$ of rank $3$,
\item[---]the exceptional groups 
$G_{28}=F_4,G_{29},G_{30}=H_4,G_{32}$ of rank $4$,
\item[---]the exceptional group
$G_{33}$ of rank $5$,
\item[---]the exceptional groups 
$G_{34},G_{35}=E_6$ of rank $6$,
\item[---]the exceptional group 
$G_{36}=E_7$ of rank $7$,
\item[---]and the exceptional group 
$G_{37}=E_8$ of rank $8$.
\end{enumerate}

\noindent
In this list, we have made visible the groups
$H_3,F_4,H_4,E_6,E_7,E_8$ which appear as exceptional groups
in the classification of all
irreducible {\it real\/} reflection groups (cf.\ \cite{HumpAC}).

Let $W$ be a well-generated complex reflection group of rank $n$,
and let $T\subseteq W$ denote the set of {\it all\/} 
(complex) reflections in the group. 
Let $\ell_T:W\to\mathbb{Z}$ denote the word length in terms of the 
generators $T$. This word length is called {\it absolute length\/} or
{\it reflection length}. Furthermore,
we define a partial order $\le_T$ on $W$ by
\begin{equation} \label{eq:absord} 
u\le_T w\quad \text{if and only if}\quad
\ell_T(w)=\ell_T(u)+\ell_T(u^{-1}w).
\end{equation}
This partial order is called {\it absolute order} or {\it reflection
order}. As is well-known and easy to see, 
the equation in \eqref{eq:absord} is equivalent to the statement that every
shortest representation of $u$ by reflections
occurs as an initial segment in some shortest product representation
of $w$ by reflections. 

Now fix a (generalised) Coxeter 
element\footnote{An element of an irreducible well-generated complex
reflection group $W$ of rank $n$ is called a
{\it Coxeter element} if it is {\it regular} in the sense of
Springer \cite{SpriAA} (see also \cite[Def.~11.21]{LeTaAA}) 
and of order $d_n$. An element of $W$ is called regular if it
has an eigenvector which lies in no reflecting
hyperplane of a reflection of $W$. It follows from 
an observation of Lehrer and Springer, proved uniformly by Lehrer and 
Michel \cite{LeMiAA} (see \cite[Theorem~11.28]{LeTaAA}), that there is
always a regular element of order $d_n$
in an irreducible well-generated complex reflection group $W$ of rank $n$.
More generally, if a well-generated complex reflection group $W$
decomposes as $W\cong W_1\times W_2\times\dots\times W_k$, where the
$W_i$'s are irreducible, then a Coxeter element of $W$ is an element
of the form $c=c_1c_2\cdots c_k$, where $c_i$ is a Coxeter element of
$W_i$, $i=1,2,\dots,k$.
If $W$ is a {\it
real\/} reflection group, that is, if all generators in $T$ have order
$2$, then the notion of generalised Coxeter element given above 
reduces to that of a Coxeter element in the classical sense
(cf.\ \cite[Sec.~3.16]{HumpAC}).} $c\in W$ 
and a positive integer $m$. 
The {\it $m$-divisible non-crossing 
partitions} $NC^m(W)$ are defined as the set
\begin{multline*}
NC^m(W)=\big\{(w_0;w_1,\dots,w_m):w_0w_1\cdots w_m=c\text{ and }\\
\ell_T(w_0)+\ell_T(w_1)+\dots+\ell_T(w_m)=\ell_T(c)\big\}.
\end{multline*}
A partial order is defined on this set by 
$$
(w_0;w_1,\ldots,w_m) \leq (u_0;u_1,\ldots,u_m) \quad \text{if and only if}\quad u_i\le_T w_i \text{ for } 1\leq i\leq m.
$$
We have suppressed the dependence on $c$, since we understand this
definition up to isomorphism of posets. To be more precise,
it can be shown that any two Coxeter elements are related to each other
by conjugation and (possibly) an automorphism on the field of complex
numbers
(see \cite[Theorem~4.2]{SpriAA} or \cite[Cor.~11.25]{LeTaAA}), 
and hence the resulting posets $NC^m(W)$ are isomorphic to each other. 
If $m=1$, then $NC^1(W)$ can be identified with the set
$NC(W)$ of 
non-crossing partitions for the (complex) reflection group $W$ as defined
by Bessis and Corran (cf.\ \cite{BeCoAA} and
\cite[Sec.~13]{BesDAB}; their definition 
extends the earlier definition by Bessis
\cite{BesDAA} and Brady and Watt \cite{BRWaAA} for real reflection
groups).

The following result has been proved by a collaborative
effort of several authors (see \cite[Prop.~13.1]{BesDAB}).

\begin{theorem} \label{thm:2}
Let $W$ be an irreducible well-generated complex 
reflection group, and let
$d_1\le d_2\le\dots\le d_n$ be its degrees and $h:=d_n$ its Coxeter number. 
Then
\begin{equation} \label{eq:F-C}
\vert NC^m(W)\vert=\prod_{i=1}^n \frac {mh+d_i} {d_i}.
\end{equation}
\end{theorem}

\begin{remark} \label{rem:0}
(1)
The number in \eqref{eq:F-C} is called the {\it Fu\ss--Catalan number}
for the reflection group $W$.

\smallskip
(2) If $c$ is a Coxeter element of a well-generated complex
reflection group $W$ of rank $n$, then $\ell_T(c)=n$.
(This follows from \cite[Sec.~7]{BesDAB}.)
\end{remark}

We conclude this section by recalling the definition of decomposition
numbers from \cite{KratCB,KratCF,KrMuAB}. Although we need them here
only for (very small) real reflection groups, and although, strictly
speaking, they have been only defined for real reflection groups in
\cite{KratCB,KratCF,KrMuAB}, this definition can be extended to
well-generated complex reflection groups without any extra effort, 
which we do now.

Given a well-generated complex reflection group $W$ of rank $n$, types
$T_1,T_2,\dots,T_d$ (in the sense of the classification of well-generated 
complex reflection groups) such that the sum of the ranks of the
$T_i$'s equals $n$, and a Coxeter element $c$,
the {\it decomposition number} $N_W(T_1,T_2,\dots,T_d)$ is defined
as the number of ``minimal" factorisations $c=c_1c_2\cdots c_d$,
``minimal" meaning that 
$\ell_T(c_1)+\ell_T(c_2)+\dots+\ell_T(c_d)=\ell_T(c)=n$, 
such that, for $i=1,2,\dots,d$, the
type of $c_i$ as a parabolic Coxeter element is $T_i$.
(Here, the term ``parabolic Coxeter element" means a Coxeter element
in some parabolic subgroup. It follows from \cite[Prop.6.3]{RipoAA}
that any element $c_i$ is indeed a Coxeter element in a 
unique parabolic subgroup
of $W$.\footnote{The uniqueness can be argued as follows: suppose that
$c_i$ were a Coxeter element in two parabolic subgroups of $W$,
say $U_1$ and $U_2$. Then it must also be a Coxeter element in the
intersection $U_1\cap U_2$. On the other hand, the absolute length
of a Coxeter element of a complex reflection group $U$ 
is always equal to $\rk(U)$, the rank of $U$.
(This follows from the fact that, for each element
$u$ of $U$, we have
$\ell_T(u)=\text{codim}\big(\text{ker}(u-\text{id})\big)$, with id denoting
the identity element in $U$; see e.g.\ \cite[Prop.~1.3]{RipoAA}).
We conclude that $\ell_T(c_i)=\rk(U_1)=\rk(U_2)=\rk(U_1\cap U_2)$,
This implies that $U_1=U_2$.} 
By definition, the type of $c_i$ is the type of this
parabolic subgroup.) Since any two Coxeter elements are related 
to each other by
conjugation plus field automorphism, the decomposition numbers are independent
of the choice of the Coxeter element $c$.

The decomposition numbers for real reflection groups have been
computed in \cite{KratCB,KratCF,KrMuAB}. To compute the decomposition
numbers for well-generated
complex reflection groups is a task that remains to be done.

\section{Cyclic sieving I}
\label{sec:siev1}

In this section we present the first cyclic sieving conjecture due to
Armstrong \cite[Conj.~5.4.7]{ArmDAA}, and to Bessis and Reiner
\cite[Conj.~6.4]{BeReAA}.

Let $\phi:NC^m(W)\to NC^m(W)$ be the map defined by
\begin{equation} \label{eq:phi}
(w_0;w_1,\dots,w_m)\mapsto
\big((cw_mc^{-1})w_0(cw_mc^{-1})^{-1};
cw_mc^{-1},w_1,w_2,\dots,w_{m-1}\big).
\end{equation}
It is indeed not difficult to see that, if the $(m+1)$-tuple
on the left-hand side is an element of $NC^m(W)$, then so is
the $(m+1)$-tuple on the right-hand side.
For $m=1$, this action reduces to conjugation by the Coxeter element
$c$ (applied to $w_1$). Cyclic sieving arising from conjugation by $c$ 
has been the subject of \cite{BeReAA}.

It is easy to see that $\phi^{mh}$ acts as the identity,
where $h$ is the Coxeter number of $W$ (see \eqref{eq:Aktion}
and Lemma~\ref{lem:4} below).
By slight abuse of notation, let $C_1$ be the cyclic group of order $mh$
generated by $\phi$. (The slight abuse consists in the fact that we
insist on $C_1$ to be a cyclic group of order $mh$, while it may
happen that the order of the action of $\phi$ given in \eqref{eq:phi}
is actually a proper divisor of $mh$.)

Given these definitions, we are now in the position to state
the first cyclic sieving conjecture of
Armstrong, respectively of Bessis and Reiner.
By the results of \cite{KratCG} and of this paper, it becomes the
following theorem.

\begin{theorem} \label{thm:1}
For an irreducible well-generated complex reflection group $W$ and any $m\ge1$, 
the triple $(NC^m(W),\Cat^m(W;q),C_1)$, where $\Cat^m(W;q)$ is
the $q$-analogue of the Fu\ss--Catalan number defined by 
\begin{equation} \label{eq:FCZahl}
\Cat^m(W;q):=\prod_{i=1}^n \frac {[mh+d_i]_q} {[d_i]_q},
\end{equation}
exhibits the cyclic sieving phenomenon in the sense of
Reiner, Stanton and White \cite{ReSWAA}.
Here, $n$ is the rank of $W$, $d_1,d_2,\dots,d_n$ are the
degrees of $W$, $h$ is the Coxeter number of $W$, 
and $[\alpha]_q:=(1-q^\alpha)/(1-q)$.
\end{theorem}

\begin{remark}
We write $\Cat^m(W)$ for $\Cat^m(W;1)$.
\end{remark}

By definition of the cyclic sieving phenomenon, we have to
prove that $\Cat^m(W;q)$ is a polynomial in $q$ with non-negative
integer coefficients, and that
\begin{equation} \label{eq:1}
\vert\Fix_{NC^m(W)}(\phi^{p})\vert = 
\Cat^m(W;q)\big\vert_{q=e^{2\pi i p/mh}}, 
\end{equation}
for all $p$ in the range $0\le p<mh$.
The first fact is established in the next section, while the 
proof of the second is achieved by making use of 
several auxiliary results, given
in Section~\ref{sec:aux1}, to reduce the proof to a finite problem,
and a subsequent case-by-case analysis,
which occupies Section~\ref{sec:Beweis1}.

\section{The $q$-Fusz--Catalan numbers $\Cat^m(W;q)$}
\label{sec:pol}

The purpose of this section is to provide an elementary, self-contained
proof of the fact that, for all irreducible complex reflection groups
$W$, the $q$-Fu\ss--Catalan number $\Cat^m(W;q)$
is a polynomial in $q$ with non-negative integer coefficients.
For most of the groups, this is a known property. However, aside from the
fact that, for many of the known cases, the proof is very indirect and uses
deep algebraic results on rational Cherednik algebras, there still remained
some cases where this property had not been formally established.
The reader is referred to the ``Theorem" in Section~1.6 of \cite{GoGrAA},
which says that, under the assumption of a certain rank condition
(\cite[Hypothesis~2.4]{GoGrAA}), the $q$-Fu\ss--Catalan number
$\Cat^m(W;q)$ is a Hilbert series of a 
finite-dimensional quotient of the ring of invariants
of $W$ and also the graded character of a finite-dimensional irreducible
representation of a spherical rational Cherednik algebra associated with
$W$. At present, this rank condition has been proven
for all irreducible well-generated complex reflection groups apart
from $G_{17},G_{18},G_{29},G_{33},G_{34}$; see \cite[Tables~8 and 9, 
column~``rank"]{MaMiAA}, and the recent paper \cite{MariAA},
which establishes the result in the case of $G_{32}$.

In the sequel, aside from the standard notation
$[\al]_q=(1-q^\al)/(1-q)$ for $q$-integers, we shall also use the
$q$-binomial coefficient, which is defined by
$$
\begin{bmatrix} n\\k\end{bmatrix}_q:=\begin{cases} 1,&\text{if $k=0$,}\\
\frac {[n]_q\,[n-1]_q\cdots [n-k+1]_q} {[k]_q\,[k-1]_q\cdots [1]_q},
&\text{if $k>0$.}\end{cases}
$$

We begin with several auxiliary results.

\begin{proposition} \label{prop:1}
For all non-negative integers $n$ and $k$,
the $q$-binomial coefficient $\left[\begin{smallmatrix}
n\\k \end{smallmatrix}\right]_q$
is a polynomial in $q$ with non-negative
integer coefficients.
\end{proposition}

\begin{proof}This is a well-known fact, which can be derived either
from the recurrence relation(s) satisfied by the $q$-binomial
coefficients  (generalising Pascal's recurrence
relation for binomial coefficients;
cf.\ \cite[eqs.~(3.3.3) and (3.3.4)]{AndrAF}), 
or from the fact that the
$q$-binomial coefficient $\left[\begin{smallmatrix}
n\\k \end{smallmatrix}\right]_q$ is the generating function for
(integer) partitions with at most $k$ parts all of which are at
most $n-k$ (cf.\ \cite[Theorem~3.1]{AndrAF}).
\end{proof}

\begin{proposition} \label{prop:2}
For all non-negative integers $m$ and $n$, the $q$-Fu\ss--Catalan
number of type $A_n$,
$$
\frac {1} {[(m+1)n+1]_q}\begin{bmatrix} (m+1)n+1\\n\end{bmatrix}_q,
$$
is a polynomial in $q$ with non-negative
integer coefficients.
\end{proposition}

\begin{proof}
In \cite[Sec.~3.3]{LoehAA}, Loehr proves that
\begin{multline} \label{eq:Loehr}
\frac {1} {[(m+1)n+1]_q}\begin{bmatrix} (m+1)n+1\\n\end{bmatrix}_q\\
=\sum_{v\in\mathcal V_n^{(m)}}
q^{m\binom n2+\sum_{i\ge0}\left(m\binom {v_i}2-iv_i\right)}
\prod_{i\ge1}q^{v_i\sum_{j=1}^m (m-j)v_{i-j}}
\begin{bmatrix} v_i+v_{i-1}+\dots+v_{i-m}-1\\v_i\end{bmatrix}_q,
\end{multline}
where $\mathcal V_n^{(m)}$ denotes the set of all sequences
$v=(v_0,v_1,\dots,v_s)$ (for some $s$) of non-negative integers with
$v_0>0$, $v_s>0$, and $v_0+v_1+\dots+v_s=n$, and such that there is
never a string of $m$ or more
consecutive zeroes in $v$. By convention, $v_i = 0$ for all negative
$i$. His proof works by showing that the expressions on both sides of
\eqref{eq:Loehr} satisfy the same recurrence relation and initial
conditions, using classical $q$-binomial identities. We refer the
reader to \cite{LoehAA} for details. By Proposition~\ref{prop:1},
the expression on the right-hand side of \eqref{eq:Loehr} is
manifestly a polynomial in $q$ with non-negative integer coefficients.
\end{proof}

\begin{lemma} \label{lem:A}
If $a$ and $b$ are coprime positive integers, then
\begin{equation} \label{eq:ab} 
\frac {[ab]_q} {[a]_q\,[b]_q}
\end{equation}
is a polynomial in $q$ of degree $(a-1)(b-1)$, 
all of whose coefficients are in $\{0,1,-1\}$.
Moreover, if one disregards the coefficients which are $0$,
then $+1$'s and $(-1)$'s alternate, and the constant coefficient
as well as the leading coefficient of the polynomial equal $+1$.
\end{lemma}

\begin{proof}
Let $\Phi_n(q)$ denote the $n$-th cyclotomic polynomial in $q$.
Using the classical formula
$$1-q^n=
\prod _{d\mid n} ^{}\Phi_d(q),$$
we see that
$$
\frac {(1-q)(1-q^{ab})} {(1-q^a)(1-q^b)}=
\underset{{d_2\mid a,\,d_2\ne1}}{\prod _{d_1\mid a,\,d_1\ne1} ^{}}
\Phi_{d_1d_2}(q),
$$
so that, manifestly, the expression in \eqref{eq:ab} is a polynomial
in $q$. The claim concerning the degree of this polynomial is obvious.

In order to establish the claim on the coefficients,
we start with a sub-expression of \eqref{eq:ab},
\begin{equation} \label{eq:ab0}
\frac {(1-q^{ab})} {(1-q^a)(1-q^b)}=
\bigg(\sum_{i=0}^{b-1}q^{ia}\bigg)
\bigg(\sum_{j=0}^{\infty}q^{jb}\bigg)=
\sum_{k=0}^{\infty}C_kq^k,
\end{equation}
say. The assumption that $a$ and $b$ are coprime
implies that $0\le C_k\le1$ for  $k\le (a-1)(b-1)$.
Multiplying both sides of \eqref{eq:ab0} by $1-q$, we obtain
the equation
\begin{equation} \label{eq:ab1} 
\frac {[ab]_q} {[a]_q\,[b]_q}
=(1-q)\sum_{k=0}^{(a-1)(b-1)}C_kq^k
+(1-q)\sum_{k=(a-1)(b-1)+1}^{\infty}C_kq^k.
\end{equation}
By our previous observation on the coefficients $C_k$ with
$k\le(a-1)(b-1)$, it is obvious that the coefficients of the first
expression on the right-hand side of \eqref{eq:ab1} are alternately
$+1$ and $-1$, when $0$'s are disregarded. Since we already know
that the left-hand side is a polynomial in $q$ of degree $(a-1)(b-1)$, 
we may ignore the second expression.

The proof is concluded by observing that the claims on the constant
and leading coefficients are obvious.
\end{proof}

\begin{corollary} \label{cor:A}
Let $a$ and $b$ be coprime positive integers, and let $\ga$ be an integer
with $\ga\ge (a-1)(b-1)$. Then the expression
$$
\frac {[\ga]_q\,[ab]_q} {[a]_q\,[b]_q}
$$
is a polynomial in $q$ with non-negative
integer coefficients.
\end{corollary}

\begin{proof}Let
$$
\frac {[ab]_q} {[a]_q\,[b]_q}
=\sum_{k=0}^{(a-1)(b-1)}D_kq^k.
$$
We then have
\begin{equation} \label{eq:abc} 
\frac {[\ga]_q\,[ab]_q} {[a]_q\,[b]_q}
=\sum_{N=0}^{(a-1)(b-1)+\ga-1}q^N\sum_{k=\max\{0,N-\ga+1\}}^{N}D_k.
\end{equation}
If $N\le \ga-1$, then, by Lemma~\ref{lem:A}, 
the sum over $k$ on the right-hand side of \eqref{eq:abc} equals
$1-1+1-1+\cdots$, which is manifestly non-negative. 
On the other hand, if $N> \ga-1$, then we may rewrite
the sum over $k$ on the right-hand side of \eqref{eq:abc} as
$$
\sum_{k=\max\{0,N-\ga+1\}}^{N}D_k=
\sum_{k=N-\ga+1}^{(a-1)(b-1)}D_k=
\sum_{k=0}^{(a-1)(b-1)+\ga-1-N}D_{(a-1)(b-1)-k}.
$$
Again, by Lemma~\ref{lem:A}, 
this sum equals
$1-1+1-1+\cdots$, which is manifestly non-negative. 
\end{proof}

\begin{lemma} \label{lem:B}
Let $\al$ and $\be$ be positive integers with $\al\ge 6$ and $\be\ge
8$. Then the expression
$$
\left[\al\right]_{q^3}\left[\be\right]_{q^4}
\frac {\left[72\right]_q\left[3\right]_q\left[4\right]_q} 
{\left[8\right]_q\left[9\right]_q\left[12\right]_q}
$$
is a polynomial in $q$ with non-negative
integer coefficients.
\end{lemma}

\begin{proof}
We have
\begin{multline*}
\frac {\left[72\right]_q\left[3\right]_q\left[4\right]_q} 
{\left[8\right]_q\left[9\right]_q\left[12\right]_q}\\
=(1
-q^3
+q^9
-q^{15}
+q^{18}
)
(1
-q^4
+q^8
-q^{12}
+q^
   {16}
-q^{20}
+q^{24}
-q^{28}
+q^{32}
).
\end{multline*}
It should be observed that both factors on the right-hand side
have the property that coefficients are in $\{0,1,-1\}$ and that
$(+1)$'s and $(-1)$'s alternate, if one disregards the coefficients
which are $0$.
If we now apply the same idea as in the proof of Corollary~\ref{cor:A},
then we see that $[\al]_{q^3}$ times the first factor 
is a polynomial in $q$ with
non-negative integer coefficients, as is
$[\be]_{q^4}$ times the second factor.
Taken together, this establishes
the claim.
\end{proof}

\begin{lemma} \label{lem:G}
Let $\al$ and $\be$ be positive integers with $\al\ge 26$ and $\be\ge
8$. Then the expression
$$
\left[\al\right]_{q}\left[\be\right]_{q^4}
\frac {\left[15\right]_q}
{\left[3\right]_q\left[5\right]_q}
\frac {\left[72\right]_q\left[3\right]_q\left[4\right]_q} 
{\left[8\right]_q\left[9\right]_q\left[12\right]_q}
$$
is a polynomial in $q$ with non-negative
integer coefficients.
\end{lemma}

\begin{proof}
We have
\begin{align*}
&\frac {\left[15\right]_q\left[72\right]_q\left[4\right]_q} 
{\left[5\right]_q\left[8\right]_q\left[9\right]_q\left[12\right]_q}\\
&\kern.5cm
=(
1 - q + q^5 - q^6 + q^9 - q^{11} + q^{12} - q^{13} + q^{14} 
- q^{15} + q^{17} -
q^{20} + q^{21} - q^{25} + q^{26}
)
\\
&\kern2cm
\times
(1
-q^4
+q^8
-q^{12}
+q^
   {16}
-q^{20}
+q^{24}
-q^{28}
+q^{32}
).
\end{align*}
Again,
if we apply the same idea as in the proof of Corollary~\ref{cor:A},
then we see that $[\al]_{q}$ times the first factor on the the
right-hand side of the above equation is a polynomial in $q$ with
non-negative integer coefficients, as is
$[\be]_{q^4}$ times the second factor.
Taken together, this establishes
the claim.
\end{proof}

\begin{lemma} \label{lem:C}
Let $\al$ and $\be$ be positive integers with $\al\ge 18$ and $\be\ge
3$. Then the expression
$$
\left[\al\right]_{q^3}\left[\be\right]_{q^4}
\frac {\left[90\right]_q\left[3\right]_q\left[4\right]_q} 
{\left[5\right]_q\left[6\right]_q\left[9\right]_q}
$$
is a polynomial in $q$ with non-negative
integer coefficients.
\end{lemma}

\begin{proof}
We have
\begin{align*}
&\frac {\left[90\right]_q\left[3\right]_q\left[4\right]_q} 
{\left[5\right]_q\left[6\right]_q\left[9\right]_q}\\
&\kern1cm
=
(1
-q^3
+q^9
-q^{12}
+q
   ^{18}
-q^{21}
+q^{27}
-q^{33}
+q^
   {36}
-q^{42}
+q^{45}
-q^{51}
+q^{54}
)\\
&\kern2cm
\times
(1
-q^4
+
   q^5
+q^6
-q^9
+q^{11}
+q^{12}
-q^
   {14}
+q^{17}
+q^{18}
-q^{19}
+q^{23}
)\\
&\kern1cm
=
(1
-q^3
+q^9
-q^{12}
+q
   ^{18}
-q^{21}
+q^{27}
-q^{33}
+q^
   {36}
-q^{42}
+q^{45}
-q^{51}
+q^{54}
)\\
&\kern2cm
\times
\big(1
-q^4
+q^{12}
+q^5(1-q^4+q^{12})
+q^6(1-q^8+q^{12})
+q^{11}(1-q^8+q^{12})
\big).
\end{align*}
Again,
if we apply the same idea as in the proof of Corollary~\ref{cor:A},
then we see that $[\al]_{q^3}$ times the first factor on the the
right-hand side of the above equation is a polynomial in $q$ with
non-negative integer coefficients, as is
$[\be]_{q^4}$ times the second factor.
Taken together, this establishes
the claim.
\end{proof}

\begin{lemma} \label{lem:F}
Let $\al$ and $\be$ be positive integers with $\al\ge 20$ and $\be\ge
18$. Then the expression
$$
\left[\al\right]_{q}\left[\be\right]_{q^3}
\frac {\left[90\right]_q\left[3\right]_q}
{\left[5\right]_q\left[6\right]_q\left[9\right]_q}
$$
is a polynomial in $q$ with non-negative
integer coefficients.
\end{lemma}

\begin{proof}
We have
\begin{align*}
&\frac {\left[90\right]_q\left[3\right]_q}
{\left[5\right]_q\left[6\right]_q\left[9\right]_q}\\
&\kern1cm
=
(1
-q^3
+q^9
-q^{12}
+q
   ^{18}
-q^{21}
+q^{27}
-q^{33}
+q^
   {36}
-q^{42}
+q^{45}
-q^{51}
+q^{54}
)\\
&\kern2cm
\times
(1
-q
+q^5
-q^7
+q^
   {10}
-q^{13}
+q^{15}
-q^{19}
+q^{20}
)
.
\end{align*}
Again,
if we apply the same idea as in the proof of Corollary~\ref{cor:A},
then we see that $[\al]_{q^3}$ times the second factor on the the
right-hand side of the above equation is a polynomial in $q$ with
non-negative integer coefficients, as is
$[\be]_{q^3}$ times the first factor.
Taken together, this establishes
the claim.
\end{proof}

\begin{lemma} \label{lem:D}
Let $\al$ be a positive integer with $\al\ge 26$.
Then the expression
$$
\left[\al\right]_{q}
\frac {\left[15\right]_q}
{\left[3\right]_q\left[5\right]_q}
\frac {\left[12\right]_{q^3}}
{\left[3\right]_{q^3}\left[4\right]_{q^3}}
$$
is a polynomial in $q$ with non-negative
integer coefficients.
\end{lemma}

\begin{proof}
We have
\begin{multline*}
\frac {\left[15\right]_q}
{\left[3\right]_q\left[5\right]_q}
\frac {\left[12\right]_{q^3}}
{\left[3\right]_{q^3}\left[4\right]_{q^3}}\\
=1
-q
+q^5
-q^6
+q^9
-q^{11}
+q
   ^{12}
-q^{13}
+q^{14}
-q^{15}
+q^
   {17}
-q^{20}
+q^{21}
-q^{25}
+q^{26}
.
\end{multline*}
Once again, the coefficients of the polynomial on the right-hand side 
are in $\{0,1,-1\}$ and 
$(+1)$'s and $(-1)$'s alternate, if one disregards the coefficients
which are $0$.
The argument from the proof of Corollary~\ref{cor:A}
then implies that this
is a polynomial in $q$ with
non-negative integer coefficients.
\end{proof}

\begin{lemma} \label{lem:E}
Let $\al$ be a positive integer with $\al\ge 14$.
Then the expression
$$
\left[\al\right]_{q}
\frac {\left[15\right]_q}
{\left[3\right]_q\left[5\right]_q}
\frac {\left[6\right]_{q^3}}
{\left[2\right]_{q^3}\left[3\right]_{q^3}}
$$
is a polynomial in $q$ with non-negative
integer coefficients.
\end{lemma}

\begin{proof}
We have
$$
\frac {\left[15\right]_q}
{\left[3\right]_q\left[5\right]_q}
\frac {\left[6\right]_{q^3}}
{\left[2\right]_{q^3}\left[3\right]_{q^3}}
=1
-q
+q^5
-q^7
+q^9
-q^{13}
+q^{14}
.
$$
Repeating the arguments of the previous proof, this establishes
the claim.
\end{proof}

\begin{lemma} \label{lem:H}
Let $\al$ and $\be$ be positive integers with $\al\ge 30$ and $\be\ge
20$. Then the expression
$$
\left[\al\right]_{q}\left[\be\right]_{q^2}
\frac {\left[84\right]_q\left[2\right]_q}
{\left[4\right]_q\left[6\right]_q\left[7\right]_q}
$$
is a polynomial in $q$ with non-negative
integer coefficients.
\end{lemma}

\begin{proof}
We have
\begin{align*}
&
\frac {\left[84\right]_q\left[2\right]_q}
{\left[4\right]_q\left[6\right]_q\left[7\right]_q}\\
&\kern1cm
=
(
1 - q + q^6 - q^8 + q^{12} - q^{15} + q^{18} - q^{22} 
+ q^{24} - q^{29} + q^{30}
)\\
&\kern2cm
\times
(
1 - q^2 + q^4 - q^6 + q^8 - q^{10} + q^{12} 
- q^{14} + q^{16} - q^{18} + q^{20} -
q^{22} \\
&\kern3cm
+ q^{24} - q^{26} + q^{28} - q^{30} 
+ q^{32} - q^{34} + q^{36} - q^{38} + q^{40}
).
\end{align*}
Once again,
if we apply the same idea as in the proof of Corollary~\ref{cor:A},
then we see that $[\al]_{q}$ times the first factor on the the
right-hand side of the above equation is a polynomial in $q$ with
non-negative integer coefficients, as is
$[\be]_{q^2}$ times the second factor.
Taken together, this establishes
the claim.
\end{proof}

\begin{lemma} \label{lem:I}
Let $\al$ and $\be$ be positive integers with $\al\ge 24$ and $\be\ge
68$. Then the expression
$$
\left[\al\right]_{q}\left[\be\right]_{q}
\frac {\left[105\right]_q}
{\left[3\right]_q\left[5\right]_q\left[7\right]_q}
$$
is a polynomial in $q$ with non-negative
integer coefficients.
\end{lemma}

\begin{proof}
We have
\begin{align*}
&
\frac {\left[105\right]_q}
{\left[3\right]_q\left[5\right]_q\left[7\right]_q}\\
&\kern1cm
=
(
1 - q + q^5 - q^6 + q^7 - q^8 + q^{10} - q^{11} 
+ q^{12} - q^{13} + q^{14} - 
q^{16}\\
&\kern2.7cm
 + q^{17} - q^{18} + q^{19} - q^{23} + q^{24}
)\\
&\kern2cm
\times
(
1 - q + q^3 - q^4 + q^6 - q^7 + q^9 - q^{10} + q^{12} - q^{13} + q^{15} - 
q^{16} + q^{18} - q^{19}\\
&\kern2.7cm
 + q^{21} - q^{22} + q^{24} - q^{25} + q^{27} - q^{28} + q^{30} -
q^{31} + q^{33} - q^{34} + q^{35} - q^{37}\\
&\kern2.7cm
 + q^{38} - q^{40} + q^{41} - q^{43} + q^{44} -
q^{46} + q^{47} - q^{49} + q^{50} - q^{52} + q^{53} - q^{55}\\
&\kern2.7cm
 + q^{56} - q^{58} + q^{59} -
q^{61} + q^{62} - q^{64} + q^{65} - q^{67} + q^{68}
).
\end{align*}
Once again,
if we apply the same idea as in the proof of Corollary~\ref{cor:A},
then we see that $[\al]_{q}$ times the first factor on the the
right-hand side of the above equation is a polynomial in $q$ with
non-negative integer coefficients, as is
$[\be]_{q}$ times the second factor.
Taken together, this establishes
the claim.
\end{proof}

\begin{lemma} \label{lem:J}
Let $\al$ and $\be$ be positive integers with $\al\ge 24$ and $\be\ge
34$. Then the expression
$$
\left[\al\right]_{q}\left[\be\right]_{q}
\frac {\left[70\right]_q}
{\left[2\right]_q\left[5\right]_q\left[7\right]_q}
$$
is a polynomial in $q$ with non-negative
integer coefficients.
\end{lemma}

\begin{proof}
We have
\begin{align*}
&
\frac {\left[70\right]_q}
{\left[2\right]_q\left[5\right]_q\left[7\right]_q}\\
&\kern1cm
=
(
1 - q + q^5 - q^6 + q^7 - q^8 + q^{10} - q^{11} 
+ q^{12} - q^{13} \\
&\kern2.7cm
+ q^{14} - 
q^{16} + q^{17} - q^{18} + q^{19} - q^{23} + q^{24}
)\\
&\kern2cm
\times
(
1 - q + q^2 - q^3 + q^4 - q^5 + q^6 - q^7 + q^8 - q^9 + q^{10} - q^{11} +
q^{12} - q^{13} \\
&\kern2.7cm
+ q^{14} - q^{15} + q^{16} - q^{17} + q^{18} - q^{19} + q^{20} - q^{21} +
q^{22} - q^{23} + q^{24} - q^{25}\\
&\kern2.7cm + q^{26} - q^{27} + q^{28} - q^{29} + q^{30} - q^{31} +
q^{32} - q^{33} + q^{34}
).
\end{align*}
Also here,
if we apply the same idea as in the proof of Corollary~\ref{cor:A},
then we see that $[\al]_{q}$ times the first factor on the the
right-hand side of the above equation is a polynomial in $q$ with
non-negative integer coefficients, as is
$[\be]_{q}$ times the second factor.
Taken together, this establishes
the claim.
\end{proof}

\begin{lemma} \label{lem:61015}
Let $\al$ and $\be$ be positive integers with $\al\ge 4$ and $\be\ge
2$. Then the expression
$$
\left[\al\right]_{q^2}\left[\be\right]_{q^5}
\frac {\left[30\right]_q\left[2\right]_q\left[3\right]_q\left[5\right]_q}
{\left[6\right]_q\left[10\right]_q\left[15\right]_q}
$$
is a polynomial in $q$ with non-negative
integer coefficients.
\end{lemma}

\begin{proof}
We have
$$
\frac {\left[30\right]_q\left[2\right]_q\left[3\right]_q\left[5\right]_q}
{\left[6\right]_q\left[10\right]_q\left[15\right]_q}
=
1 + q - q^3 - q^4 - q^5 + q^7 + q^8.
$$
If we multiply this expression by $[\al]_{q^2}$, then, for $\al=4$ we obtain
$$
1 + q + q^2 - q^5 - q^9 + q^{12} + q^{13} + q^{14},
$$
for $\al=5$ we obtain
$$
1 + q + q^2 - q^5 + q^8 - q^{11} + q^{14} + q^{15} + q^{16},
$$
and, for $\al\ge6$, we obtain 
$$
1 + q + q^2 - q^5 + 
q^8+q^{10}+p_1(q)+q^{2\al-4}+q^{2\al-2}
 - q^{2\al+1} + q^{2\al+4} + q^{2\al+5} + q^{2\al+6},
$$
where $p_1(q)$ is a polynomial
in $q$ with non-negative coefficients of order 
at least $11$ and degree at most $2\al-5$. 
In all cases it is obvious that the product of the result
and $[\be]_{q^5}$, with $\be\ge2$, 
is a polynomial in $q$ with non-negative coefficients.
\end{proof}

\begin{lemma} \label{lem:61015B}
Let $\al$ and $\be$ be positive integers with $\al\ge 14$ and $\be\ge
2$. Then the expression
$$
\left[\al\right]_{q}\left[\be\right]_{q^5}
\frac {\left[14\right]_q}
{\left[2\right]_q\left[7\right]_q}
\frac {\left[30\right]_q\left[2\right]_q\left[3\right]_q\left[5\right]_q}
{\left[6\right]_q\left[10\right]_q\left[15\right]_q}
$$
is a polynomial in $q$ with non-negative
integer coefficients.
\end{lemma}

\begin{proof}
We have
$$
\frac {\left[14\right]_q}
{\left[2\right]_q\left[7\right]_q}
\frac {\left[30\right]_q\left[2\right]_q\left[3\right]_q\left[5\right]_q}
{\left[6\right]_q\left[10\right]_q\left[15\right]_q}
=
1 - q^3 - q^5 + q^6 + q^7 + q^8 - q^9 - q^{11} + q^{14}.
$$
If we multiply this expression by $[\al]_{q}$, then, 
for $\al\ge14$, we obtain 
\begin{multline*}
1 + q + q^2 - q^5 + q^7 + 2 q^8 + q^9 + q^{10}  
+p_2(q)\\
+q^{\al+3} + q^{\al+4} + 
 2 q^{\al+5} + q^{\al+6} - q^{\al+8} + q^{\al+11} + q^{\al+12} + q^{\al+13},
\end{multline*}
where $p_2(q)$ is a polynomial
in $q$ with non-negative coefficients of order 
at least $11$ and degree at most $\al+2$. 
It is obvious that the product of the result
and $[\be]_{q^5}$, with $\be\ge2$, 
is a polynomial in $q$ with non-negative coefficients.
\end{proof}

\begin{lemma} \label{lem:61015C}
Let $\al$ and $\be$ be positive integers with $\al\ge 32$ and $\be\ge
12$. Then the expression
$$
\left[\al\right]_{q}\left[\be\right]_{q^2}
\frac {\left[35\right]_q}
{\left[5\right]_q\left[7\right]_q}
\frac {\left[30\right]_q\left[2\right]_q\left[3\right]_q\left[5\right]_q}
{\left[6\right]_q\left[10\right]_q\left[15\right]_q}
$$
is a polynomial in $q$ with non-negative
integer coefficients.
\end{lemma}

\begin{proof}
We have
\begin{multline*}
\frac {\left[35\right]_q}
{\left[5\right]_q\left[7\right]_q}
\frac {\left[30\right]_q\left[2\right]_q\left[3\right]_q\left[5\right]_q}
{\left[6\right]_q\left[10\right]_q\left[15\right]_q}
=
1 - q^2 - q^3 + q^5 + q^6 + q^7 - q^8 - 
 2 q^9 + q^{11} + q^{12} - q^{16} + q^{20} + q^{21} \\- 
 2 q^{23} - q^{24} + q^{25} + q^{26} + q^{27} - q^{29} - q^{30} + q^{32}.
\end{multline*}
If we multiply this expression by $[\al]_{q}$, then, 
for $\al\ge32$, we obtain 
\begin{multline*}
1 + q - q^3 - q^4 + q^6 + 
 2 q^7 + q^8 - q^9 - q^{10} + q^{12} + q^{13} + q^{14} + q^{15} + q^{20} + 
 2 q^{21} + 2 q^{22} \\- q^{24} + q^{26} + 2 q^{27} + 2 q^{28} + q^{29}
 + 
p_3(q)
+q^{\al+2} + 
 2 q^{\al+4} + 2 q^{\al+5} + q^{\al+6} - q^{\al+8} \\
+ 2 q^{\al+9} + 
 2 q^{\al+10} + q^{\al+11} + q^{\al+16} + q^{\al+17} + q^{\al+18} + q^{\al+19} - q^{\al+21} - q^{\al+22} + q^{\al+23}\\ + 
 2 q^{\al+24} + q^{\al+25} - q^{\al+27} - q^{{\al+28}} + q^{\al+30} + q^{\al+31},
\end{multline*}
where $p_3(q)$ is a polynomial
in $q$ with non-negative coefficients of order 
at least $30$ and degree at most $\al+1$. 
It is obvious that the product of the result
and $[\be]_{q^2}$, with $\be\ge12$, 
is a polynomial in $q$ with non-negative coefficients.
\end{proof}

\begin{lemma} \label{lem:101215}
Let $\al$ and $\be$ be positive integers with $\al\ge 16$ and $\be\ge
2$. Then the expression
$$
\left[\al\right]_{q^2}\left[\be\right]_{q^5}
\frac {\left[60\right]_q\left[2\right]_q\left[3\right]_q\left[5\right]_q}
{\left[10\right]_q\left[12\right]_q\left[15\right]_q}
$$
is a polynomial in $q$ with non-negative
integer coefficients.
\end{lemma}

\begin{proof}
We have
\begin{multline*}
\frac {\left[60\right]_q\left[2\right]_q\left[3\right]_q\left[5\right]_q}
{\left[10\right]_q\left[12\right]_q\left[15\right]_q}=
1 + q - q^3 - q^4 - q^5 - q^6 + q^8 + q^9 + q^{10} + q^{11} + q^{12} - q^{14}
- q^{15} - q^{16}\\
 - q^{17} - q^{18} + q^{20} + q^{21} + q^{22} + q^{23} + q^{24} - q^{26}
- q^{27} - q^{28} - q^{29} + q^{31} + q^{32}.
\end{multline*}
If we multiply this expression by $[\al]_{q^2}$, then, 
for $\al\ge16$, we obtain 
\begin{multline*}
1 + q + q^2 - q^5 - q^6 - q^7 + q^{10} + q^{11} + 
 2 q^{12} + q^{13} + q^{14}\\
 - q^{17} - q^{18} - q^{19} + q^{22} + q^{23} + 2 q^{24}
+p_4(q)+s_4(q),
\end{multline*}
where $p_4(q)$ is a polynomial
in $q$ with non-negative coefficients of order 
at least $25$ and degree at most $2\al+5$ and $s_4(q)$ completes
the above expression to a symmetric polynomial. 
It is obvious that the product of the result
and $[\be]_{q^5}$, with $\be\ge2$, 
is a polynomial in $q$ with non-negative coefficients.
\end{proof}

\begin{lemma} \label{lem:101215B}
Let $\al$ and $\be$ be positive integers with $\al\ge 56$ and $\be\ge
4$. Then the expression
$$
\left[\al\right]_{q}\left[\be\right]_{q^2}
\frac {\left[35\right]_q}
{\left[5\right]_q\left[7\right]_q}
\frac {\left[60\right]_q\left[2\right]_q\left[3\right]_q\left[5\right]_q}
{\left[10\right]_q\left[12\right]_q\left[15\right]_q}
$$
is a polynomial in $q$ with non-negative
integer coefficients.
\end{lemma}

\begin{proof}
We have
\begin{align*}
&
\frac {\left[35\right]_q}
{\left[5\right]_q\left[7\right]_q}
\frac {\left[60\right]_q\left[2\right]_q\left[3\right]_q\left[5\right]_q}
{\left[10\right]_q\left[12\right]_q\left[15\right]_q}
=
1 - q^2 - q^3 + q^5 + q^7 - q^9 + q^{12} - q^{13} + q^{15} - q^{16} - q^{18} \\
&\kern2cm
+
q^{19} + q^{20} + q^{22} - 2 q^{23} + q^{27} - q^{28} + q^{29} - 
 2 q^{33} + q^{34} + q^{36} + q^{37} - q^{38}\\
&\kern2cm
 - q^{40} + q^{41} - q^{43} + q^{44} - 
q^{47} + q^{49} + q^{51} - q^{53} - q^{54} + q^{56}.
\end{align*}
If we multiply this expression by $[\al]_{q}$, then, 
for $\al\ge56$, we obtain 
\begin{multline*}
1 + q - q^3 - q^4 + q^7 + q^8 + q^{12} + q^{15} - q^{18} + q^{20} + q^{21} + 
 2 q^{22} + q^{27}\\
 + q^{29} + q^{30} + q^{31} + q^{32} - q^{33} + q^{36} + 
 2 q^{37} + q^{38}
+p_5(q)+s_5(q),
\end{multline*}
where $p_5(q)$ is a polynomial
in $q$ with non-negative coefficients of order 
at least $39$ and degree at most $\al+16$ and $s_5(q)$ completes
the above expression to a symmetric polynomial. 
It is obvious that the product of the result
and $[\be]_{q^2}$, with $\be\ge4$, 
is a polynomial in $q$ with non-negative coefficients.
\end{proof}

\begin{lemma} \label{lem:101215C}
Let $\al$ and $\be$ be positive integers with $\al\ge 38$ and $\be\ge
2$. Then the expression
$$
\left[\al\right]_{q}\left[\be\right]_{q^5}
\frac {\left[14\right]_q}
{\left[2\right]_q\left[7\right]_q}
\frac {\left[60\right]_q\left[2\right]_q\left[3\right]_q\left[5\right]_q}
{\left[10\right]_q\left[12\right]_q\left[15\right]_q}
$$
is a polynomial in $q$ with non-negative
integer coefficients.
\end{lemma}

\begin{proof}
We have
$$
\frac {\left[14\right]_q}
{\left[2\right]_q\left[7\right]_q}
\frac {\left[60\right]_q\left[2\right]_q\left[3\right]_q\left[5\right]_q}
{\left[10\right]_q\left[12\right]_q\left[15\right]_q}
=
1 - q^3 - q^5 + q^7 + q^8 - q^{13} + q^{19} - q^{25} + q^{30} + q^{31} - q^{33} -
q^{35} + q^{38}.
$$
If we multiply this expression by $[\al]_{q}$, then, 
for $\al\ge38$, we obtain 
$$
1 + q + q^2 - q^5 - q^6 + q^8 + q^9+q^{10}+q^{11}
+p_6(q)+s_6(q),
$$
where $p_6(q)$ is a polynomial
in $q$ with non-negative coefficients of order 
at least $12$ and degree at most $\al+25$ and $s_6(q)$ completes
the above expression to a symmetric polynomial. 
It is obvious that the product of the result
and $[\be]_{q^2}$, with $\be\ge2$, 
is a polynomial in $q$ with non-negative coefficients.
\end{proof}

\begin{lemma} \label{lem:679}
Let $\al$ and $\be$ be positive integers with $\al\ge 30$ and $\be\ge
26$. Then the expression
$$
\left[\al\right]_{q}\left[\be\right]_{q^3}
\frac {\left[126\right]_q\left[3\right]_q}
{\left[6\right]_q\left[7\right]_q\left[9\right]_q}
$$
is a polynomial in $q$ with non-negative
integer coefficients.
\end{lemma}

\begin{proof}
We have
\begin{align*}
&
\frac {\left[126\right]_q\left[3\right]_q}
{\left[6\right]_q\left[7\right]_q\left[9\right]_q}
=
(
1 - q + q^6 - q^8 + q^{12} - q^{15} + q^{18} - q^{22} + q^{24} - q^{29} + q^{30}
)\\
&\kern2cm
\times
(
1 - q^3 + q^9 - q^{12} + q^{18} - q^{21} + q^{27} - q^{30} + q^{36} -
q^{39} \\
&\kern3cm
+ q^{42}- q^{48} + q^{51} - q^{57} + q^{60} - q^{66} + q^{69} - q^{75} +
q^{78}
).
\end{align*}
If we apply the same idea as in the proof of Corollary~\ref{cor:A},
then we see that $[\al]_{q}$ times the first factor on the the
right-hand side of the above equation is a polynomial in $q$ with
non-negative integer coefficients, as is
$[\be]_{q^3}$ times the second factor.
Taken together, this establishes
the claim.
\end{proof}

\begin{lemma} \label{lem:7912}
Let $\al$ and $\be$ be positive integers with $\al\ge 66$ and $\be\ge
54$. Then the expression
$$
\left[\al\right]_{q}\left[\be\right]_{q^3}
\frac {\left[252\right]_q\left[3\right]_q}
{\left[7\right]_q\left[9\right]_q\left[12\right]_q}
$$
is a polynomial in $q$ with non-negative
integer coefficients.
\end{lemma}

\begin{proof}
We have
\begin{align*}
&
\frac {\left[252\right]_q\left[3\right]_q}
{\left[7\right]_q\left[9\right]_q\left[12\right]_q}\\
&\kern1cm
=
(
1 - q + q^7 - q^8 + q^{12} - q^{13} + q^{14} - q^{15} + q^{19} - q^{20} + q^{21} - 
q^{22} + q^{24} - q^{25} \\
&\kern3cm
+ q^{26} - q^{27}
 + q^{28} - q^{29} + q^{31} - q^{32} + q^{33} -
q^{34} + q^{35} - q^{37} + q^{38} \\
&\kern3cm
- q^{39} + q^{40} - q^{41} + q^{42}
 - q^{44} + q^{45} -
q^{46} + q^{47} - q^{51} + q^{52} - q^{53}\\
&\kern3cm
 + q^{54} - q^{58} + q^{59}
- q^{65} + q^{66}
)\\
&\kern2cm
\times
(
1 - q^3 + q^9 - q^{12} + q^{18} - q^{21} + q^{27} - q^{30} + q^{36} - q^{39} + q^{45}
- q^{48}\\
&\kern3cm
 + q^{54} - q^{57}
 + q^{63} - q^{66} + q^{72} - q^{75} + q^{81} - q^{87} + q^{90}
- q^{96} + q^{99} \\
&\kern3cm
- q^{105} + q^{108} - q^{114}
 + q^{117} - q^{123} + q^{126} - q^{132}
+ q^{135} - q^{141}\\
&\kern3cm
 + q^{144} - q^{150} + q^{153} - q^{159} + q^{162}
).
\end{align*}
If we apply the same idea as in the proof of Corollary~\ref{cor:A},
then we see that $[\al]_{q}$ times the first factor on the the
right-hand side of the above equation is a polynomial in $q$ with
non-negative integer coefficients, as is
$[\be]_{q^3}$ times the second factor.
Taken together, this establishes
the claim.
\end{proof}

\begin{lemma} \label{lem:4710}
Let $\al$ and $\be$ be positive integers with $\al\ge 54$ and $\be\ge
34$. Then the expression
$$
\left[\al\right]_{q}\left[\be\right]_{q^2}
\frac {\left[140\right]_q\left[2\right]_q}
{\left[4\right]_q\left[7\right]_q\left[10\right]_q}
$$
is a polynomial in $q$ with non-negative
integer coefficients.
\end{lemma}

\begin{proof}
We have
\begin{align*}
&
\frac {\left[140\right]_q\left[2\right]_q}
{\left[4\right]_q\left[7\right]_q\left[10\right]_q}\\
&\kern1cm
=
(
1 - q + q^7 - q^8 + q^{10} - q^{11} + q^{14} - q^{15} + q^{17} - q^{18} + q^{20} - 
q^{22} + q^{24} - q^{25} \\
&\kern3cm
+ q^{27} - q^{29}
 + q^{30} - q^{32} + q^{34} - q^{36} + q^{37} -
q^{39} + q^{40} - q^{43} + q^{44} - q^{46} \\
&\kern3cm
+ q^{47} - q^{53} + q^{54}
)\\
&\kern2cm
\times
(
1 - q^2 + q^4 - q^6 + q^8 - q^{10} + q^{12} - q^{14} + q^{16} - q^{18} + q^{20} -
q^{22} + q^{24} \\
&\kern3cm
- q^{26} + q^{28}
 - q^{30} + q^{32} - q^{34} + q^{36} - q^{38} + q^{40} -
q^{42} + q^{44} - q^{46} \\
&\kern3cm
+ q^{48} 
- q^{50} + q^{52} - q^{54}
 + q^{56} - q^{58} + q^{60} -
q^{62} + q^{64} - q^{66} + q^{68}
),
\end{align*}
If we apply the same idea as in the proof of Corollary~\ref{cor:A},
then we see that $[\al]_{q}$ times the first factor on the the
right-hand side of the above equation is a polynomial in $q$ with
non-negative integer coefficients, as is
$[\be]_{q^2}$ times the second factor.
Taken together, this establishes
the claim.
\end{proof}

%

We are now ready for the proof of the main result of this section.

\begin{theorem} \label{thm:0}
For all irreducible well-generated complex reflection groups
and positive integers $m$,
the $q$-Fu\ss--Catalan number $\Cat^m(W;q)$ 
is a polynomial in $q$ with non-negative
integer coefficients.
\end{theorem}

\begin{proof}
First, let $W=A_n$. In this case, the degrees are $2,3,\dots,n+1$,
and hence
$$
\Cat^m(A_n;q)=\frac {1} {[(m+1)n+1]_q}\begin{bmatrix}
  (m+1)n+1\\n\end{bmatrix}_q, 
$$
which, by Proposition~\ref{prop:2}, is a polynomial in $q$ with 
non-negative integer coefficients.

Next, let $W=G(d,1,n)$. In this case, the degrees are 
$d,2d,\dots,nd$, and hence
$$
\Cat^m(G(d,1,n);q)=\begin{bmatrix}
  (m+1)n\\n\end{bmatrix}_{q^d}, 
$$
which, by Proposition~\ref{prop:1}, is a polynomial in $q$ with 
non-negative integer coefficients.

Now, let $W=G(e,e,n)$. In this case, the degrees are 
$e,2e,\dots,(n-1)e,n$, and hence
\begin{align*}
\Cat^m(G(e,e,n);q)&=
\frac {[m(n-1)e+n]_q} {[n]_q}
\prod _{i=1} ^{n-1}\frac {[m(n-1)e+ie]_q} {[ie]_q}\\
&=
\begin{bmatrix}
  (m+1)(n-1)\\n-1\end{bmatrix}_{q^e}+
q^n[e]_{q^n}\begin{bmatrix}
  (m+1)(n-1)\\n\end{bmatrix}_{q^e}, 
\end{align*}
which, by Proposition~\ref{prop:1}, is a polynomial in $q$ with 
non-negative integer coefficients.

It remains to verify the claim for the exceptional groups.

For $W=G_4$, the degrees are $4,6$, and hence
$$
\Cat^m(G_4;q)=\frac {[6m+4]_q\,[6m+6]_q} {[4]_q\,[6]_q}=
\frac {1} {[3m+4]_{q^2}}
\begin{bmatrix}
3m+4\\3\end{bmatrix}_{q^2}, 
$$
which, by Proposition~\ref{prop:2}, is a polynomial in $q$ with 
non-negative integer coefficients.

For $W=G_5$, the degrees are $6,12$, and hence
$$
\Cat^m(G_5;q)=\frac {[12m+6]_q\,[12m+12]_q} {[6]_q\,[12]_q}
=\begin{bmatrix} 2m+2\\2\end{bmatrix}_{q^6},
$$
which, by Proposition~\ref{prop:1}, is a polynomial in $q$ with 
non-negative integer coefficients.

For $W=G_6$, the degrees are $4,12$, and hence
$$
\Cat^m(G_6;q)=\frac {[12m+4]_q\,[12m+12]_q} {[4]_q\,[12]_q}
=[3m+1]_{q^4}\,[m+1]_{q^{12}},
$$
which is manifestly a polynomial in $q$ with 
non-negative integer coefficients.

For $W=G_8$, the degrees are $8,12$, and hence
$$
\Cat^m(G_8;q)=\frac {[12m+8]_q\,[12m+12]_q} {[8]_q\,[12]_q}=
\frac {1} {[3m+4]_{q^4}}
\begin{bmatrix}
3m+4\\3\end{bmatrix}_{q^4}, 
$$
which, by Proposition~\ref{prop:2}, is a polynomial in $q$ with 
non-negative integer coefficients.

For $W=G_9$, the degrees are $8,24$, and hence
$$
\Cat^m(G_9;q)=\frac {[24m+8]_q\,[24m+24]_q} {[8]_q\,[24]_q}
=[3m+1]_{q^8}\,[m+1]_{q^{24}},
$$
which is manifestly a polynomial in $q$ with 
non-negative integer coefficients.

For $W=G_{10}$, the degrees are $12,24$, and hence
$$
\Cat^m(G_{10};q)=\frac {[24m+12]_q\,[24m+24]_q} {[12]_q\,[24]_q}
=\begin{bmatrix} 2m+2\\2\end{bmatrix}_{q^{12}},
$$
which, by Proposition~\ref{prop:1}, is a polynomial in $q$ with 
non-negative integer coefficients.

For $W=G_{14}$, the degrees are $6,24$, and hence
$$
\Cat^m(G_{14};q)=\frac {[24m+6]_q\,[24m+24]_q} {[6]_q\,[24]_q}
=[4m+1]_{q^6}\,[m+1]_{q^{24}},
$$
which is manifestly a polynomial in $q$ with 
non-negative integer coefficients.

For $W=G_{16}$, the degrees are $20,30$, and hence
$$
\Cat^m(G_{16};q)=\frac {[30m+20]_q\,[30m+30]_q} {[20]_q\,[30]_q}=
\frac {1} {[3m+4]_{q^{10}}}
\begin{bmatrix}
3m+4\\3\end{bmatrix}_{q^{10}}, 
$$
which, by Proposition~\ref{prop:2}, is a polynomial in $q$ with 
non-negative integer coefficients.

For $W=G_{17}$, the degrees are $20,60$, and hence
$$
\Cat^m(G_{17};q)=\frac {[60m+20]_q\,[60m+60]_q} {[20]_q\,[60]_q}
=[3m+1]_{q^{20}}\,[m+1]_{q^{60}},
$$
which is manifestly a polynomial in $q$ with 
non-negative integer coefficients.

For $W=G_{18}$, the degrees are $30,60$, and hence
$$
\Cat^m(G_{18};q)=\frac {[60m+30]_q\,[60m+60]_q} {[30]_q\,[60]_q}
=\begin{bmatrix} 2m+2\\2\end{bmatrix}_{q^{30}},
$$
which, by Proposition~\ref{prop:1}, is a polynomial in $q$ with 
non-negative integer coefficients.

For $W=G_{20}$, the degrees are $12,30$, and hence
\begin{align*}
\Cat^m(G_{20};q)&=\frac {[30m+12]_q\,[30m+30]_q} {[12]_q\,[30]_q}\\
&=
\begin{cases} 
\left[\frac {5m+2} {2}\right]_{q^{12}}\,
\left[m+1\right]_{q^{30}},&\text{if $m$ is even,}\\
\left[5m+2\right]_{q^{6}}\,
\left[\frac {m+1} {2}\right]_{q^{60}}
\frac {[10]_{q^6}} {[2]_{q^6}\left[5\right]_{q^6}},
&\text{if $m$ is odd,}
\end{cases}
\end{align*}
which, by Corollary~\ref{cor:A}, are polynomials in $q$ with 
non-negative integer coefficients in both cases.

For $W=G_{21}$, the degrees are $12,60$, and hence
$$
\Cat^m(G_{21};q)=\frac {[60m+12]_q\,[60m+60]_q} {[12]_q\,[60]_q}
=[5m+1]_{q^{12}}\,[m+1]_{q^{60}},
$$
which is manifestly a polynomial in $q$ with 
non-negative integer coefficients.

For $W=G_{23}=H_3$, the degrees are $2,6,10$, and hence
\begin{align*}
\Cat^m(H_3;q)&=\frac {[10m+2]_q\,[10m+6]_q\,[10m+10]_q} 
{[2]_q\,[6]_q\,[10]_q}\\
&=\begin{cases} 
\left[5m+2\right]_{q^2}\,
\left[\tfrac {5m+3} {3}\right]_{q^6}\,
\left[m+1\right]_{q^{10}},&\text{if $m\equiv0$ (mod 3),}\\
\left[\tfrac {5m+1} {3}\right]_{q^6}\,
\left[5m+3\right]_{q^2}\,
\left[m+1\right]_{q^{10}},&\text{if $m\equiv1$ (mod 3),}\\
\left[5m+2\right]_{q^2}\,
\left[{5m+3} \right]_{q^2}\,
\left[\frac {m+1} {3}\right]_{q^{30}}
\frac {[15]_{q^2}} {[3]_{q^2}\left[5\right]_{q^2}}
,&\text{if $m\equiv2$ (mod 3),}
\end{cases}
\end{align*}
which, by Corollary~\ref{cor:A}, are polynomials in $q$ with 
non-negative integer coefficients in all cases.

For $W=G_{24}$, the degrees are $4,6,14$, and hence
$$
\Cat^m(G_{24};q)=\frac {[14m+4]_q\,[14m+6]_q\,[14m+14]_q} 
{[4]_q\,[6]_q\,[14]_q}.
$$
We have
$$
\Cat^m(G_{24};q)=
\begin{cases} 
\left[\tfrac {7m} {2}+1\right]_{q^4}\,
\left[\tfrac {14m} {6}+1\right]_{q^6}\,
\left[m+1\right]_{q^{14}},&\text{if $m\equiv0$ (mod 6),}\\
\left[\tfrac {7m+2} {3}\right]_{q^6}\,
\left[\tfrac {7m+3} {2}\right]_{q^4}\,
\left[m+1\right]_{q^{14}},&\text{if $m\equiv1$ (mod 6),}\\
\left[\tfrac {7m} {2}+1\right]_{q^4}\,
\left[7m+3\right]_{q^2}\,
\left[\frac {m+1} {3}\right]_{q^{42}}
\frac {[21]_{q^2}} {[3]_{q^2}\left[7\right]_{q^2}}
,&\text{if $m\equiv2$ (mod 6),}\\
\left[{7m}+ {2}\right]_{q^2}\,
\left[\frac{7m}3+1\right]_{q^6}\,
\left[\frac {m+1} {2}\right]_{q^{28}}
\frac {[14]_{q^2}} {[2]_{q^2}\left[7\right]_{q^2}}
,&\text{if $m\equiv3$ (mod 6),}\\
\left[\tfrac {7m+2} {6}\right]_{q^{12}}
\frac {[6]_{q^2}} {[2]_{q^2}\left[3\right]_{q^2}}
\left[7m+3\right]_{q^2}\,
\left[{m+1} \right]_{q^{14}}
,&\text{if $m\equiv4$ (mod 6),}\\
\left[ {7m} +{2}\right]_{q^2}\,
\left[\frac{7m+3}2\right]_{q^4}\,
\left[\frac {m+1} {3}\right]_{q^{42}}
\frac {[21]_{q^2}} {[3]_{q^2}\left[7\right]_{q^2}}
,&\text{if $m\equiv5$ (mod 6),}
\end{cases}
$$
which, by Corollary~\ref{cor:A}, are polynomials in $q$ with 
non-negative integer coefficients in all cases.

For $W=G_{25}$, the degrees are $6,9,12$, and hence
$$
\Cat^m(G_{25};q)=\frac {[12m+6]_q\,[12m+9]_q\,[12m+12]_q} {[6]_q\,[9]_q\,[12]_q}=
\frac {1} {[4m+5]_{q^{3}}}
\begin{bmatrix}
4m+5\\4\end{bmatrix}_{q^{3}}, 
$$
which, by Proposition~\ref{prop:2}, is a polynomial in $q$ with 
non-negative integer coefficients.

For $W=G_{26}$, the degrees are $6,12,18$, and hence
$$
\Cat^m(G_{26};q)=\frac {[18m+6]_q\,[18m+12]_q\,[18m+18]_q} {[6]_q\,[12]_q\,[18]_q}
=\begin{bmatrix} 3m+3\\3\end{bmatrix}_{q^{6}},
$$
which, by Proposition~\ref{prop:1}, is a polynomial in $q$ with 
non-negative integer coefficients.

For $W=G_{27}$, the degrees are $6,12,30$, and hence
$$
\Cat^m(G_{27};q)=\frac {[30m+6]_q\,[30m+12]_q\,[30m+30]_q} {[6]_q\,[12]_q\,[30]_q}
=[m+1]_{q^{30}}
\begin{bmatrix} 5m+2\\2\end{bmatrix}_{q^{6}},
$$
which, by Proposition~\ref{prop:1}, is a polynomial in $q$ with 
non-negative integer coefficients.

For $W=G_{28}=F_4$, the degrees are $2,6,8,12$, and hence
\begin{align*}
\Cat^m(F_4;q)&=\frac {[12m+2]_q\,[12m+6]_q\,[12m+8]_q\,[12m+12]_q}
{[2]_q\,[6]_q\,[8]_q\,[12]_q}\\
&=
\begin{cases} 
\left[6m+1\right]_{q^{2}}\,
\left[2m+1\right]_{q^{6}}\,
\left[\tfrac {3} {2}m+1\right]_{q^{8}}\,
\left[m+1\right]_{q^{12}},&\text{if $m$ is even,}\\
\left[6m+1\right]_{q^{2}}\,
\left[2m+1\right]_{q^{6}}\,
\left[\tfrac {m+1} {2}\right]_{q^{24}}\,
\left[3m+2\right]_{q^4}\,
\dfrac {[6]_{q^4}} {[2]_{q^4}\,[3]_{q^4}},&\text{if $m$ is odd,}\\
\end{cases}
\end{align*}
which in both cases is a polynomial in $q$ with 
non-negative integer coefficients; in the second case this is
due to Corollary~\ref{cor:A}.

For $W=G_{29}$, the degrees are $4,8,12,20$, and hence
\begin{align*}
\Cat^m(G_{29};q)&=\frac {[20m+4]_q\,[20m+8]_q\,[20m+12]_q\,[20m+20]_q} 
{[4]_q\,[8]_q\,[12]_q\,[20]_q}\\
&=[m+1]_{q^{20}}
\begin{bmatrix} 5m+3\\3\end{bmatrix}_{q^{4}},
\end{align*}
which, by Proposition~\ref{prop:1}, is a polynomial in $q$ with 
non-negative integer coefficients.

For $W=G_{30}=H_4$, the degrees are $2,12,20,30$, and hence
$$
\Cat^m(H_4;q)=\frac {[30m+2]_q\,[30m+12]_q\,[30m+20]_q\,[30m+30]_q}
{[2]_q\,[12]_q\,[20]_q\,[30]_q}.
$$
If $m$ is even, then we have
$$
\Cat^m(H_4;q)=
\left[15m+1\right]_{q^{2}}\,
\left[\tfrac {5} {2}m+1\right]_{q^{12}}\,
\left[\tfrac {3} {2}m+1\right]_{q^{20}}\,
\left[m+1\right]_{q^{30}},
$$
which is manifestly a polynomial in $q$ with 
non-negative integer coefficients. On the other hand,
if $m$ is odd, then we may write
\begin{align*}
\Cat^m(H_4;q)&=
\left[\tfrac{15m+1}2\right]_{q^{4}}\,
\left[5m+2\right]_{q^{6}}\,
\left[3m+2\right]_{q^{10}}\,
\left[\tfrac {m+1} {2}\right]_{q^{60}}\,
\frac {[30]_{q^2}\left[2\right]_{q^2}
\left[3\right]_{q^2}\left[5\right]_{q^2}} 
{[6]_{q^6}\left[10\right]_{q^2}\left[15\right]_{q^2}},
\end{align*}
which, by Lemma~\ref{lem:61015}, is a polynomial in $q$ with 
non-negative integer coefficients. 

For $W=G_{32}$, the degrees are $12,18,24,30$, and hence
\begin{align*}
\Cat^m(G_{32};q)&=\frac
    {[30m+12]_q\,[30m+18]_q\,[30m+24]_q\,[30m+30]_q} 
{[12]_q\,[18]_q\,[24]_q\,[30]_q}\\
&=
\frac {1} {[5m+6]_{q^{6}}}
\begin{bmatrix}
5m+6\\5\end{bmatrix}_{q^{6}}, 
\end{align*}
which, by Proposition~\ref{prop:2}, is a polynomial in $q$ with 
non-negative integer coefficients.

For $W=G_{33}$, the degrees are $4,6,10,12,18$, and hence
$$
\Cat^m(G_{33};q)=\frac {[18m+4]_q\,[18m+6]_q\,[18m+10]_q\,
[18m+12]_q\,[18m+18]_q}
{[4]_q\,[6]_q\,[10]_q\,[12]_q\,[18]_q}.
$$
If $m\equiv0~(\text{mod }10)$, then we have
$$
\Cat^m(G_{33};q)=
\left[\tfrac {9} {2}m+1\right]_{q^4}\,
\left[3m+1\right]_{q^6}\,
\left[\tfrac {9} {5}m+1\right]_{q^{10}}\,
\left[\tfrac {3} {2}m+1\right]_{q^{12}}\,
\left[m+1\right]_{q^{18}},
$$
which is manifestly a polynomial in $q$ with 
non-negative integer coefficients.
If $m\equiv1~(\text{mod }10)$, then we have
$$
\Cat^m(G_{33};q)=
\left[\tfrac {3m+1} {2}\right]_{q^{12}}\,
\left[\tfrac {9m+5} {2}\right]_{q^{4}}\,
\left[\tfrac {3m+2} {5}\right]_{q^{30}}\,
\left[m+1\right]_{q^{18}}\,
\left[9m+2\right]_{q^2}\,
\frac {[15]_{q^2}} {[3]_{q^2}\,[5]_{q^2}},
$$
which, by Corollary~\ref{cor:A}, is a polynomial in $q$ with 
non-negative integer coefficients.
If $m\equiv2~(\text{mod }10)$, then we have
$$
\Cat^m(G_{33};q)=
\left[\tfrac {9m+2} {10}\right]_{q^{20}}\,
\left[3m+1\right]_{q^{6}}\,
\left[\tfrac {3} {2}m+1\right]_{q^{12}}\,
\left[m+1\right]_{q^{18}}\,
\left[9m+5\right]_{q^{2}}\,
\frac {[10]_{q^2}} {[2]_{q^2}\,[5]_{q^2}},
$$
which, by Corollary~\ref{cor:A}, is a polynomial in $q$ with 
non-negative integer coefficients.
If $m\equiv3~(\text{mod }10)$, then we have
$$
\Cat^m(G_{33};q)=
\left[\tfrac {3m+1} {10}\right]_{q^{60}}\,
\left[\tfrac {9m+5} {2}\right]_{q^{4}}\,
\left[3m+2\right]_{q^{6}}\,
\left[m+1\right]_{q^{18}}\,
\left[9m+2\right]_{q^{2}}\,
\frac {[30]_{q^2}} {[5]_{q^2}\,[6]_{q^2}},
$$
which, by Corollary~\ref{cor:A}, is a polynomial in $q$ with 
non-negative integer coefficients.
If $m\equiv4~(\text{mod }10)$, then we have
$$
\Cat^m(G_{33};q)=
\left[\tfrac {9} {2}m+1\right]_{q^{4}}\,
\left[3m+1\right]_{q^{6}}\,
\left[\tfrac {3} {2}m+1\right]_{q^{12}}\,
\left[\tfrac {m+1} {5}\right]_{q^{90}}\,
\left[9m+5\right]_{q^{2}}\,
\frac {[45]_{q^2}} {[5]_{q^2}\,[9]_{q^2}},
$$
which, by Corollary~\ref{cor:A}, is a polynomial in $q$ with 
non-negative integer coefficients.
If $m\equiv5~(\text{mod }10)$, then we have
$$
\Cat^m(G_{33};q)=
\left[\tfrac {3m+1} {2}\right]_{q^{12}}\,
\left[\tfrac {9m+5} {10}\right]_{q^{20}}\,
\left[3m+2\right]_{q^{6}}\,
\left[m+1\right]_{q^{18}}\,
\left[9m+2\right]_{q^{2}}\,
\frac {[10]_{q^2}} {[2]_{q^2}\,[5]_{q^2}},
$$
which, by Corollary~\ref{cor:A}, is a polynomial in $q$ with 
non-negative integer coefficients.
If $m\equiv6~(\text{mod }10)$, then we have
$$
\Cat^m(G_{33};q)=
\left[\tfrac {9} {2}m+1\right]_{q^{4}}\,
\left[3m+1\right]_{q^{6}}\,
\left[\tfrac {3m+2} {10}\right]_{q^{60}}\,
\left[m+1\right]_{q^{18}}\,
\left[9m+5\right]_{q^{2}}\,
\frac {[30]_{q^2}} {[5]_{q^2}\,[6]_{q^2}},
$$
which, by Corollary~\ref{cor:A}, is a polynomial in $q$ with 
non-negative integer coefficients.
If $m\equiv7~(\text{mod }10)$, then we have
$$
\Cat^m(G_{33};q)=
\left[\tfrac {9m+2} {5}\right]_{q^{10}}\,
\left[\tfrac {3m+1} {2}\right]_{q^{12}}\,
\left[\tfrac {9m+5} {2}\right]_{q^{4}}\,
\left[3m+2\right]_{q^{6}}\,
\left[m+1\right]_{q^{18}},
$$
which is manifestly a polynomial in $q$ with 
non-negative integer coefficients.
If $m\equiv8~(\text{mod }10)$, then we have
$$
\Cat^m(G_{33};q)=
\left[\tfrac {9} {2}m+1\right]_{q^{4}}\,
\left[\tfrac {3m+1} {5}\right]_{q^{30}}\,
\left[\tfrac {3} {2}m+1\right]_{q^{12}}\,
\left[m+1\right]_{q^{18}}\,
\left[9m+5\right]_{q^{2}}\,
\frac {[15]_{q^2}} {[3]_{q^2}\,[5]_{q^2}},
$$
which, by Corollary~\ref{cor:A}, is a polynomial in $q$ with 
non-negative integer coefficients.
Finally, if $m\equiv9~(\text{mod }10)$, then we have
$$
\Cat^m(G_{33};q)=
\left[\tfrac {3m+1} {2}\right]_{q^{12}}\,
\left[\tfrac {9m+5} {2}\right]_{q^{4}}\,
\left[3m+2\right]_{q^{6}}\,
\left[\tfrac {m+1} {5}\right]_{q^{90}}\,
\left[9m+2\right]_{q^{2}}\,
\frac {[45]_{q^2}} {[5]_{q^2}\,[9]_{q^2}},
$$
which, by Corollary~\ref{cor:A}, is a polynomial in $q$ with 
non-negative integer coefficients.

For $W=G_{34}$, the degrees are $6,12,18,24,30,42$, and hence
\begin{align*}
\Cat^m(G_{34};q)&=\frac {[42m+6]_q\,[42m+12]_q\,[42m+18]_q\,
[42m+24]_q\,[42m+30]_q\,[42m+42]_q} 
{[6]_q\,[12]_q\,[18]_q\,[24]_q\,[30]_q\,[42]_q}\\
&=[m+1]_{q^{42}}
\begin{bmatrix} 7m+5\\5\end{bmatrix}_{q^{6}},
\end{align*}
which, by Proposition~\ref{prop:1}, is a polynomial in $q$ with 
non-negative integer coefficients.

For $W=G_{35}=E_6$, the degrees are $2,5,6,8,9,12$, and hence
$$
\Cat^m(E_6;q)=\frac {[12m+2]_q\,[12m+5]_q\,[12m+6]_q\,[12m+8]_q\,
[12m+9]_q\,[12m+12]_q}
{[2]_q\,[5]_q\,[6]_q\,[8]_q\,[9]_q\,[12]_q}.
$$
If $m\equiv
0
~(\text{mod }30),$ then we have $$
\Cat^m(E_6;q)=
{\left[ 6m+1 \right]_{q^{ 2}}} {\left[ \tfrac{12m+5}{5}  
  \right]_{q^{ 5}}} {\left[ 2 m+1 \right]_{q^{ 6}}} {\left[ 
  \tfrac{3m+2}{2}  \right]_{q^{ 8}}} {\left[ \tfrac{4m+3}{3}  
  \right]_{q^{ 9}}} {\left[ m+1 \right]_{q^{ 12}}}
,$$
which is manifestly a polynomial in $q$ with 
non-negative integer coefficients.  
If $m\equiv
1
~(\text{mod }30),$ then we have 
\begin{multline*}\Cat^m(E_6;q)=
{\left[ 6m+1 \right]_{q^{ 2 }}}{\left[ \tfrac{2m+1}{3}  
   \right]_{q^{ 18}}} {\left[ 4m+3 \right]_{q^3}} 
\frac {{ {\left[ 6 \right]_{q^{ 3}}} 
  }} { {\left[ 2 \right]_{q^{ 3}}} {\left[ 3 
   \right]_{q^{ 3}}}}\\
\times {\left[ \tfrac{3m+2}{5}  \right]_{q^{ 20 
  }}} {\left[ 12m+5 
  \right]_{q}}  \frac {[20]_q} {[4]_q\left[5\right]_q}
 {\left[ \tfrac{m+1}{2} \right]_{q^{ 24}}} \frac {{ 
  {\left[ 6 \right]_{q^{ 4}}} }} { {\left[ 2 
   \right]_{q^{ 4}}} {\left[ 3 \right]_{q^{ 4}}}}
.\end{multline*} 
If one decomposes $[6m+1]_{q^2}$ as $[3m+1]_{q^4}+q^2[3m]_{q^4}$,
then one sees that,
by Corollary~\ref{cor:A}, the above expression is a polynomial in $q$ with 
non-negative integer coefficients.
If $m\equiv
2
~(\text{mod }30),$ then we have \begin{multline*}\Cat^m(E_6;q)=
{\left[ 6m+1 \right]_{q^{ 2 }}} {\left[ 12m+5 
  \right]_{q}} {\left[ \tfrac{2m+1}{5}  
   \right]_{q^{ 30}}} \frac {{ {\left[ 30 \right]_{q}} 
  }} { {\left[ 5 \right]_{q}} {\left[ 6 
   \right]_{q}}} \\
\times{\left[ \tfrac{3m+2}{2}  \right]_{q^{ 8 
  }}} {\left[ 4m+3 \right]_{q^3}} 
 {\left[ \tfrac{m+1}{3} \right]_{q^{ 36}}} \frac {{ 
  {\left[ 12 \right]_{q^{ 3}}} }} { {\left[ 3 
   \right]_{q^{ 3}}} {\left[ 4 \right]_{q^{ 3}}}}
,\end{multline*} 
which, by Corollary~\ref{cor:A}, is a polynomial in $q$ with 
non-negative integer coefficients.
If $m\equiv
3
~(\text{mod }30),$ then we have \begin{multline*}\Cat^m(E_6;q)=
{\left[ 6m+1 \right]_{q^{ 2 }}} {\left[ 12m+5 
  \right]_{q}} {\left[ 2 m+1 \right]_{q^{ 6 
  }}} {\left[ 3m+2 \right]_{q^4}} \\
\times
 {\left[ \tfrac{4m+3}{15}  \right]_{q^{ 45}}} \frac {{ 
  {\left[ 45 \right]_{q}} }} { {\left[ 5 
   \right]_{q}} {\left[ 9 \right]_{q}}} {\left[ 
   \tfrac{m+1}{2} \right]_{q^{ 24}}} \frac {{ {\left[ 6 
   \right]_{q^{ 4}}} }} { {\left[ 2 \right]_{q^{ 4}}} 
  {\left[ 3 \right]_{q^{ 4}}}}
,\end{multline*} 
which, by Corollary~\ref{cor:A}, is a polynomial in $q$ with 
non-negative integer coefficients.
If $m\equiv
4
~(\text{mod }30),$ then we have 
\begin{multline*}\Cat^m(E_6;q)=
{\left[ \tfrac{12m+2}{5}  \right]_{q^{ 5 }}} {\left[ 
  \tfrac{12m+5}{2}  \right]_{q^{ 2 }}} {\left[ 
   \tfrac{2m+1}{3}  \right]_{q^{ 18}}} \frac {{ {\left[ 6 
   \right]_{q^{ 3}}} }} { {\left[ 2 \right]_{q^{ 3}}} 
  {\left[ 3 \right]_{q^{ 3}}}} \\
\times{\left[ \tfrac{3m+2}{2}  
  \right]_{q^{ 8 }}} {\left[ 4m+3 \right]_{q^3}}
 {\left[ m+1 \right]_{q^{ 12 }}}
,\end{multline*}
which, by Corollary~\ref{cor:A}, is a polynomial in $q$ with 
non-negative integer coefficients.
If $m\equiv
5
~(\text{mod }30),$ then we have \begin{multline*}\Cat^m(E_6;q)=
{\left[ 6m+1 \right]_{q^{ 2}}} 
{\left[ \tfrac{12m+5}{5}  
   \right]_{q^{ 5}}} 
{\left[ 2 m+1 \right]_{q^{ 6}}} \\
\times
 {\left[ 3m+2 \right]_{q^{ 4 }}} {\left[ 
  {4m+3}  \right]_{q^{ 3 }}} 
{\left[\tfrac { m+1 }6
   \right]_{q^{ 72}}} 
\frac {\left[72\right]_q\left[3\right]_q\left[4\right]_q} 
{\left[8\right]_q\left[9\right]_q\left[12\right]_q}
,\end{multline*} 
which, by Lemma~\ref{lem:B}, 
is a polynomial in $q$ with 
non-negative integer coefficients.
If $m\equiv
6
~(\text{mod }30),$ then we have $$\Cat^m(E_6;q)=
{\left[ 6m+1 \right]_{q^{ 2 }}} {\left[ 12m+5 
  \right]_{q}} {\left[ 2 m+1 \right]_{q^{ 6 
  }}} {\left[ \tfrac{3m+2}{10}  \right]_{q^{ 40}}} 
  \frac {{ {\left[ 40 \right]_{q}} }} { 
  {\left[ 5 \right]_{q}} {\left[ 8 \right]_{q}}} 
 {\left[ \tfrac{4m+3}{3}  \right]_{q^{ 9 }}} {\left[ 
  m+1 \right]_{q^{ 12 }}}
,$$ 
which, by Corollary~\ref{cor:A}, is a polynomial in $q$ with 
non-negative integer coefficients.
If $m\equiv
7
~(\text{mod }30),$ then we have 
\begin{multline*}\Cat^m(E_6;q)=
{\left[\tfrac { 6m+1}2 \right]_{q^{ 4}}} 
{\left[ {12m+5}  
  \right]_{q}} {\left[ \tfrac {2 m+1}{15} \right]_{q^{ 90}}}\\
\times
\frac {\left[90\right]_q\left[3\right]_q\left[4\right]_q} 
{\left[5\right]_q\left[6\right]_q\left[9\right]_q}
 {\left[ 3m+2 \right]_{q^{ 4 }}} {\left[ 
  {4m+3} \right]_{q^{ 3 }}} {\left[ \tfrac {m+1}2 
   \right]_{q^{ 24}}} 
\frac {[6]_{q^4}} {[2]_{q^4}\left[3\right]_{q^4}}
,\end{multline*} 
which, by Corollary~\ref{cor:A} and Lemma~\ref{lem:C}, 
is a polynomial in $q$ with 
non-negative integer coefficients.
If $m\equiv
8
~(\text{mod }30),$ then we have 
\begin{multline*}\Cat^m(E_6;q)=
{\left[ 6m+1 \right]_{q^{ 2 }}} {\left[ 12m+5 
  \right]_{q}} {\left[ 2 m+1 \right]_{q^{ 6 
  }}} {\left[ \tfrac{3m+2}{2}  \right]_{q^{ 8 }}} \\
\times
 {\left[ \tfrac{4m+3}{5}  \right]_{q^{ 15 }}}
\frac {[15]_q} {[3]_q\left[5\right]_q} 
 {\left[ \tfrac{m+1}{3} \right]_{q^{ 36}}} \frac {{ 
  {\left[ 12 \right]_{q^{ 3}}} }} { {\left[ 3 
   \right]_{q^{ 3}}} {\left[ 4 \right]_{q^{ 3}}}}
,\end{multline*}
which, by Lemma~\ref{lem:D}, is a polynomial in $q$ with 
non-negative integer coefficients.
If $m\equiv
9
~(\text{mod }30),$ then we have $$\Cat^m(E_6;q)=
{\left[ \tfrac{12m+2}{5}  \right]_{q^{ 5 }}} {\left[ 
  \tfrac{12m+5}{2}  \right]_{q^{ 2 }}} {\left[ 2 m+1 
  \right]_{q^{ 6 }}} {\left[ 3m+2 \right]_{q^{ 4
  }}} {\left[ \tfrac{4m+3}{3}  \right]_{q^{ 9 }}} 
 {\left[ \tfrac{m+1}{2} \right]_{q^{ 24}}} \frac {{ 
  {\left[ 6 \right]_{q^{ 4}}} }} { {\left[ 2 
   \right]_{q^{ 4}}} {\left[ 3 \right]_{q^{ 4}}}}
,$$ 
which, by Corollary~\ref{cor:A}, is a polynomial in $q$ with 
non-negative integer coefficients.
If $m\equiv
10
~(\text{mod }30),$ then we have $$\Cat^m(E_6;q)=
{\left[ 6m+1 \right]_{q^{ 2 }}} {\left[ 
  \tfrac{12m+5}{5}  \right]_{q^{ 5 }}} {\left[ 
   \tfrac{2m+1}{3}  \right]_{q^{ 18}}} \frac {{ {\left[ 6 
   \right]_{q^{ 3}}} }} { {\left[ 2 \right]_{q^{ 3}}} 
  {\left[ 3 \right]_{q^{ 3}}}} {\left[ \tfrac{3m+2}{2}  
  \right]_{q^{ 8 }}} {\left[ 4m+3 \right]_{q^{ 3 
  }}} {\left[ m+1 \right]_{q^{ 12 }}}
,$$ 
which, by Corollary~\ref{cor:A}, is a polynomial in $q$ with 
non-negative integer coefficients.
If $m\equiv
11
~(\text{mod }30),$ then we have \begin{multline*}\Cat^m(E_6;q)=
{\left[ 6m+1 \right]_{q^{ 2}}}
 {\left[ 12m+5
  \right]_{q}} {\left[ 2 m+1 \right]_{q^{ 6}}} \\
\times
 {\left[ \tfrac {3m+2}5 \right]_{q^{ 20}}}
\frac {[20]_q} {[4]_q\left[5\right]_q} 
{\left[ 4m+3  
  \right]_{q^3}} {\left[ \tfrac {m+1}6 \right]_{q^{ 72}}} 
\frac {\left[72\right]_q\left[3\right]_q\left[4\right]_q} 
{\left[8\right]_q\left[9\right]_q\left[12\right]_q}
.\end{multline*} 
If one decomposes $[6m+1]_{q^2}$ as $[3m+1]_{q^4}+q^2[3m]_{q^4}$,
then one sees that, by Corollary~\ref{cor:A} and Lemma~\ref{lem:B}, 
this is a polynomial in $q$ with 
non-negative integer coefficients.
If $m\equiv
12
~(\text{mod }30),$ then we have $$\Cat^m(E_6;q)=
{\left[ 6m+1 \right]_{q^{ 2 }}} {\left[ 12m+5 
  \right]_{q}} {\left[ \tfrac{2m+1}{5}  
   \right]_{q^{ 30}}} \frac {{ {\left[ 30 \right]_{q}} 
  }} { {\left[ 5 \right]_{q}} {\left[ 6 
   \right]_{q}}} {\left[ \tfrac{3m+2}{2}  \right]_{q^{ 8 
  }}} {\left[ \tfrac{4m+3}{3}  \right]_{q^{ 9 }}} 
 {\left[ m+1 \right]_{q^{ 12 }}}
,$$ 
which, by Corollary~\ref{cor:A}, is a polynomial in $q$ with 
non-negative integer coefficients.
If $m\equiv
13
~(\text{mod }30),$ then we have \begin{multline*}\Cat^m(E_6;q)=
{\left[ 6m+1 \right]_{q^{ 2 }}} {\left[ 12m+5 
  \right]_{q}} {\left[ \tfrac{2m+1}{3}  
   \right]_{q^{ 18}}} \frac {{ {\left[ 6 \right]_{q^{ 3}}} 
  }} { {\left[ 2 \right]_{q^{ 3}}} {\left[ 3 
   \right]_{q^{ 3}}}} \\
\times{\left[ 3m+2 \right]_{q^4}} 
 {\left[ \tfrac{4m+3}{5}  \right]_{q^{ 15 }}} 
\frac {[15]_q} {[3]_q\left[5\right]_q}
 {\left[ \tfrac{m+1}{2} \right]_{q^{ 24}}} \frac {{ 
  {\left[ 6 \right]_{q^{ 4}}} }} { {\left[ 2 
   \right]_{q^{ 4}}} {\left[ 3 \right]_{q^{ 4}}}}
,\end{multline*} 
which, by Lemma~\ref{lem:E}, is a polynomial in $q$ with 
non-negative integer coefficients.
If $m\equiv
14
~(\text{mod }30),$ then we have $$\Cat^m(E_6;q)=
{\left[ \tfrac{12m+2}{5}  \right]_{q^{ 5 }}} {\left[ 
  \tfrac{12m+5}{2}  \right]_{q^{ 2 }}} {\left[ 2 m+1 
  \right]_{q^{ 6 }}} {\left[ \tfrac{3m+2}{2}  
  \right]_{q^{ 8 }}} {\left[ 4m+3 \right]_{q^3}}
 {\left[ \tfrac{m+1}{3} \right]_{q^{ 36}}} 
  \frac {{ {\left[ 12 \right]_{q^{ 3}}} }} { 
  {\left[ 3 \right]_{q^{ 3}}} {\left[ 4 \right]_{q^{ 3}}}}
,$$ 
which, by Corollary~\ref{cor:A}, is a polynomial in $q$ with 
non-negative integer coefficients.
If $m\equiv
15
~(\text{mod }30),$ then we have $$\Cat^m(E_6;q)=
{\left[ 6m+1 \right]_{q^{ 2 }}} {\left[ 
  \tfrac{12m+5}{5}  \right]_{q^{ 5 }}} {\left[ 2 m+1 
  \right]_{q^{ 6 }}} {\left[ 3m+2 \right]_{q^4}}
 {\left[ \tfrac{4m+3}{3}  \right]_{q^{ 9 }}} 
 {\left[ \tfrac{m+1}{2} \right]_{q^{ 24}}} \frac {{ 
  {\left[ 6 \right]_{q^{ 4}}} }} { {\left[ 2 
   \right]_{q^{ 4}}} {\left[ 3 \right]_{q^{ 4}}}}
,$$ 
which, by Corollary~\ref{cor:A}, is a polynomial in $q$ with 
non-negative integer coefficients.
If $m\equiv
16
~(\text{mod }30),$ then we have \begin{multline*}\Cat^m(E_6;q)=
{\left[ 6m+1 \right]_{q^{ 2 }}} {\left[ 12m+5 
  \right]_{q}} {\left[ \tfrac{2m+1}{3}  
   \right]_{q^{ 18}}} \frac {{ {\left[ 6 \right]_{q^{ 3}}} 
  }} { {\left[ 2 \right]_{q^{ 3}}} {\left[ 3 
   \right]_{q^{ 3}}}} \\
\times{\left[ \tfrac{3m+2}{10}  \right]_{q^{ 40}}} 
  \frac {{ {\left[ 40 \right]_{q}} }} { 
  {\left[ 5 \right]_{q}} {\left[ 8 \right]_{q}}} 
 {\left[ 4m+3 \right]_{q^3}} {\left[ m+1 
  \right]_{q^{ 12 }}}
,\end{multline*} 
which, by Corollary~\ref{cor:A}, is a polynomial in $q$ with 
non-negative integer coefficients.
If $m\equiv
17
~(\text{mod }30),$ then we have \begin{multline*}\Cat^m(E_6;q)=
{\left[ 6m+1 \right]_{q^{ 2}}} 
{\left[ {12m+5}
  \right]_{q}} {\left[ \tfrac {2 m+1}5 \right]_{q^{ 30}}}
\frac {[30]_q} {[5]_q\left[6\right]_q}\\
\times
 {\left[ 3m+2 \right]_{q^{ 4 }}} {\left[ 
  {4m+3}  \right]_{q^{ 3 }}} {\left[ \tfrac {m+1}6 
   \right]_{q^{ 72}}} 
\frac {\left[72\right]_q\left[3\right]_q\left[4\right]_q} 
{\left[8\right]_q\left[9\right]_q\left[12\right]_q}
,\end{multline*} 
which, by Corollary~\ref{cor:A} and Lemma~\ref{lem:B}, 
is a polynomial in $q$ with 
non-negative integer coefficients.
If $m\equiv
18
~(\text{mod }30),$ then we have $$\Cat^m(E_6;q)=
{\left[ 6m+1 \right]_{q^{ 2 }}} {\left[ 12m+5 
  \right]_{q}} {\left[ 2 m+1 \right]_{q^{ 6 
  }}} {\left[ \tfrac{3m+2}{2}  \right]_{q^{ 8 }}} 
 {\left[ \tfrac{4m+3}{15}  \right]_{q^{ 45}}} \frac {{ 
  {\left[ 45 \right]_{q}} }} { {\left[ 5 
   \right]_{q}} {\left[ 9 \right]_{q}}} {\left[ m+1 
  \right]_{q^{ 12 }}}
,$$ 
which, by Corollary~\ref{cor:A}, is a polynomial in $q$ with 
non-negative integer coefficients.
If $m\equiv
19
~(\text{mod }30),$ then we have \begin{multline*}\Cat^m(E_6;q)=
{\left[ \tfrac{12m+2}{5}  \right]_{q^{ 5 }}} {\left[ 
  \tfrac{12m+5}{2}  \right]_{q^{ 2 }}} {\left[ 
   \tfrac{2m+1}{3}  \right]_{q^{ 18}}} \frac {{ {\left[ 6 
   \right]_{q^{ 3}}} }} { {\left[ 2 \right]_{q^{ 3}}} 
  {\left[ 3 \right]_{q^{ 3}}}}\\
\times {\left[ 3m+2 \right]_{q^4}}
 {\left[ 4m+3 \right]_{q^3}} 
 {\left[ \tfrac{m+1}{2} \right]_{q^{ 24}}} \frac {{ 
  {\left[ 6 \right]_{q^{ 4}}} }} { {\left[ 2 
   \right]_{q^{ 4}}} {\left[ 3 \right]_{q^{ 4}}}}
,\end{multline*} 
which, by Corollary~\ref{cor:A}, is a polynomial in $q$ with 
non-negative integer coefficients.
If $m\equiv
20
~(\text{mod }30),$ then we have $$\Cat^m(E_6;q)=
{\left[ 6m+1 \right]_{q^{ 2 }}} {\left[ 
  \tfrac{12m+5}{5}  \right]_{q^{ 5 }}} {\left[ 2 m+1 
  \right]_{q^{ 6 }}} {\left[ \tfrac{3m+2}{2}  
  \right]_{q^{ 8 }}} {\left[ 4m+3 \right]_{q^3}}
 {\left[ \tfrac{m+1}{3} \right]_{q^{ 36}}} 
  \frac {{ {\left[ 12 \right]_{q^{ 3}}} }} { 
  {\left[ 3 \right]_{q^{ 3}}} {\left[ 4 \right]_{q^{ 3}}}}
,$$ 
which, by Corollary~\ref{cor:A}, is a polynomial in $q$ with 
non-negative integer coefficients.
If $m\equiv
21
~(\text{mod }30),$ then we have 
\begin{multline*}\Cat^m(E_6;q)=
{\left[ 6m+1 \right]_{q^{ 2 }}} {\left[ 12m+5 
  \right]_{q}} {\left[ 2 m+1 \right]_{q^{ 6 
  }}} \\
\times{\left[ \tfrac{3m+2}{5}  \right]_{q^{ 20 }}}
\frac {[20]_q} {[4]_q\left[5\right]_q} 
 {\left[ \tfrac{4m+3}{3}  \right]_{q^{ 9 }}} 
 {\left[ \tfrac{m+1}{2} \right]_{q^{ 24}}} \frac {{ 
  {\left[ 6 \right]_{q^{ 4}}} }} { {\left[ 2 
   \right]_{q^{ 4}}} {\left[ 3 \right]_{q^{ 4}}}}
,\end{multline*} 
which, by Corollary~\ref{cor:A}, is a polynomial in $q$ with 
non-negative integer coefficients.
If $m\equiv
22
~(\text{mod }30),$ then we have \begin{multline*}\Cat^m(E_6;q)=
{\left[ 6m+1 \right]_{q^{ 2 }}} {\left[ 12m+5 
  \right]_{q}} {\left[ \tfrac{2 m+1}{15}  
   \right]_{q^{ 90}}} 
\frac {{ {\left[ 90 \right]_{q}} {\left[ 3 \right]_{q}} 
  }} { {\left[ 5 \right]_{q}} {\left[ 6 \right]_{q}} {\left[ 9 
   \right]_{q}}}\\
\times {\left[ \tfrac{3m+2}{2}  \right]_{q^{ 8 
  }}} {\left[ 4m+3 \right]_{q^3}} 
 {\left[ m+1 \right]_{q^{ 12}}} 
,\end{multline*} 
which, by Lemma~\ref{lem:F}, is a polynomial in $q$ with 
non-negative integer coefficients.
If $m\equiv
23
~(\text{mod }30),$ then we have \begin{multline*}\Cat^m(E_6;q)=
{\left[ 6m+1 \right]_{q^{ 2}}} 
{\left[ 12m+5
  \right]_{q}} {\left[ 2 m+1 \right]_{q^{ 6}}} \\
\times
 {\left[ {3m+2} \right]_{q^{ 4 }}} {\left[ 
   \tfrac{4m+3}5  \right]_{q^{ 15}}} 
\frac {[15]_q} {[3]_q\left[5\right]_q}
 {\left[ \tfrac {m+1}6 \right]_{q^{ 
   72}}} 
\frac {\left[72\right]_q\left[3\right]_q\left[4\right]_q} 
{\left[8\right]_q\left[9\right]_q\left[12\right]_q}
,\end{multline*} 
which, by Lemma~\ref{lem:G}, 
is a polynomial in $q$ with 
non-negative integer coefficients.
If $m\equiv
24
~(\text{mod }30),$ then we have $$\Cat^m(E_6;q)=
{\left[ \tfrac{12m+2}{5}  \right]_{q^{ 5}}} {\left[ 
  \tfrac{12m+5}{2}  \right]_{q^{ 2}}} {\left[ 2 m+1 \right]_{q^{ 
  6}}} {\left[ \tfrac{3m+2}{2}  \right]_{q^{ 8}}} {\left[ 
  \tfrac{4m+3}{3}  \right]_{q^{ 9}}} {\left[ m+1 \right]_{q^{ 12}}}
,$$ 
which is manifestly a polynomial in $q$ with 
non-negative integer coefficients. 
If $m\equiv
25
~(\text{mod }30),$ then we have \begin{multline*}\Cat^m(E_6;q)=
{\left[ 6m+1 \right]_{q^{ 2 }}} {\left[ 
  \tfrac{12m+5}{5}  \right]_{q^{ 5 }}} {\left[ 
   \tfrac{2m+1}{3}  \right]_{q^{ 18}}} \frac {{ {\left[ 6 
   \right]_{q^{ 3}}} }} { {\left[ 2 \right]_{q^{ 3}}} 
  {\left[ 3 \right]_{q^{ 3}}}}\\
\times {\left[ 3m+2 \right]_{q^{ 
  4 }}} {\left[ 4m+3 \right]_{q^3}} 
 {\left[ \tfrac{m+1}{2} \right]_{q^{ 24}}} \frac {{ 
  {\left[ 6 \right]_{q^{ 4}}} }} { {\left[ 2 
   \right]_{q^{ 4}}} {\left[ 3 \right]_{q^{ 4}}}}
,\end{multline*} 
which, by Corollary~\ref{cor:A}, is a polynomial in $q$ with 
non-negative integer coefficients.
If $m\equiv
26
~(\text{mod }30),$ then we have \begin{multline*}\Cat^m(E_6;q)=
{\left[ 6m+1 \right]_{q^{ 2 }}} {\left[ 12m+5 
  \right]_{q}} {\left[ 2 m+1 \right]_{q^{ 6 
  }}}\\
\times {\left[ \tfrac{3m+2}{10}  \right]_{q^{ 40}}} 
  \frac {{ {\left[ 40 \right]_{q}} }} { 
  {\left[ 5 \right]_{q}} {\left[ 8 \right]_{q}}} 
 {\left[ 4m+3 \right]_{q^3}} {\left[ 
   \tfrac{m+1}{3} \right]_{q^{ 36}}} \frac {{ {\left[ 12 
   \right]_{q^{ 3}}} }} { {\left[ 3 \right]_{q^{ 3}}} 
  {\left[ 4 \right]_{q^{ 3}}}}
,\end{multline*} 
which, by Corollary~\ref{cor:A}, is a polynomial in $q$ with 
non-negative integer coefficients.
If $m\equiv
27
~(\text{mod }30),$ then we have \begin{multline*}\Cat^m(E_6;q)=
{\left[ 6m+1 \right]_{q^{ 2 }}} {\left[ 12m+5 
  \right]_{q}} {\left[ \tfrac{2m+1}{5}  
   \right]_{q^{ 30}}} \frac {{ {\left[ 30 \right]_{q}} 
  }} { {\left[ 5 \right]_{q}} {\left[ 6 
   \right]_{q}}} \\
\times{\left[ 3m+2 \right]_{q^4}} 
 {\left[ \tfrac{4m+3}{3}  \right]_{q^{ 9 }}} 
 {\left[ \tfrac{m+1}{2} \right]_{q^{ 24}}} \frac {{ 
  {\left[ 6 \right]_{q^{ 4}}} }} { {\left[ 2 
   \right]_{q^{ 4}}} {\left[ 3 \right]_{q^{ 4}}}}
,\end{multline*} 
which, by Corollary~\ref{cor:A}, is a polynomial in $q$ with 
non-negative integer coefficients.
If $m\equiv
28
~(\text{mod }30),$ then we have 
\begin{multline*}\Cat^m(E_6;q)=
{\left[ 6m+1 \right]_{q^{ 2 }}} {\left[ 12m+5 
  \right]_{q}} {\left[ \tfrac{2m+1}{3}  
   \right]_{q^{ 18}}} \frac {{ {\left[ 6 \right]_{q^{ 3}}} 
  }} { {\left[ 2 \right]_{q^{ 3}}} {\left[ 3 
   \right]_{q^{ 3}}}}\\
\times {\left[ \tfrac{3m+2}{2}  \right]_{q^{ 8 
  }}} {\left[ \tfrac{4m+3}{5}  \right]_{q^{ 15 }}} 
\frac {[15]_q} {[3]_q\left[5\right]_q}
 {\left[ m+1 \right]_{q^{ 12 }}}
,\end{multline*}
which, by Lemma~\ref{lem:E}, is a polynomial in $q$ with 
non-negative integer coefficients.
If $m\equiv
29
~(\text{mod }30),$ then we have \begin{multline*}\Cat^m(E_6;q)=
{\left[ \tfrac{6m+1}5 \right]_{q^{ 10}}} 
\frac {[10]_q} {[2]_q\left[5\right]_q}
{\left[ 12m+5  
  \right]_{q}} {\left[ 2 m+1 \right]_{q^{ 6}}} \\
\times
 {\left[ 3m+2 \right]_{q^{ 4 }}} {\left[ 
  4m+3  \right]_{q^{ 3 }}} {\left[ \tfrac {m+1}6
   \right]_{q^{ 72}}} 
\frac {\left[72\right]_q\left[3\right]_q\left[4\right]_q} 
{\left[8\right]_q\left[9\right]_q\left[12\right]_q},
\end{multline*}
which, by Corollary~\ref{cor:A} and Lemma~\ref{lem:B}, 
is a polynomial in $q$ with 
non-negative integer coefficients.

For $W=G_{36}=E_7$, the degrees are $2,6,8,10,12,14,18$, and hence
\begin{multline*}
\Cat^m(E_7;q)=\frac {[18m+2]_q\,[18m+6]_q\,[18m+8]_q\,[18m+10]_q\,
}
{[2]_q\,[6]_q\,[8]_q\,[10]_q}\\
\times
\frac {
[18m+12]_q\,[18m+14]_q\,[18m+18]_q}
{[12]_q\,[14]_q\,[18]_q}.
\end{multline*}
If $m\equiv
0
~(\text{mod }140),$ then we have $$\Cat^m(E_7;q)=
{{\left[ 9 m+1 \right]_{q^{ 2}}} {\left[ 3 m+1 \right]_{q^{ 
  6}}} {\left[ \tfrac{9m+4}{4}  \right]_{q^{ 8}}} {\left[ 
  \tfrac{9m+5}{5}  \right]_{q^{ 10}}} {\left[ \tfrac{3m+2}{2}  
  \right]_{q^{ 12}}} {\left[ \tfrac{9m+7}{7}  \right]_{q^{ 14}}} 
 {\left[ m+1 \right]_{q^{ 18}}}}
,$$ 
which is manifestly a polynomial in $q$ with 
non-negative integer coefficients.  
If $m\equiv
1
~(\text{mod }140),$ then we have $$\Cat^m(E_7;q)=
{{\left[ \tfrac{9m+1}{5}  \right]_{q^{ 10}}} {\left[ 
  \tfrac{3m+1}{2}  \right]_{q^{ 12}}} {\left[ 9m+4 \right]_{q^{ 
  2}}} {\left[ \tfrac{9m+5}{7}  \right]_{q^{ 14}}} {\left[ 
  3m+2 \right]_{q^{ 6}}} {\left[ \tfrac{9m+7}{4}  \right]_{q^{ 
  8}}} {\left[ m+1 \right]_{q^{ 18}}}}
,$$ 
which is manifestly a polynomial in $q$ with 
non-negative integer coefficients.  
If $m\equiv
2
~(\text{mod }140),$ then we have \begin{multline*}\Cat^m(E_7;q)=
{{\left[ 9 m+1 \right]_{q^{ 2}}}} {{\left[ \tfrac{3m+1}{7}  
   \right]_{q^{ 42}}}} \frac {{ {\left[ 21 \right]_{q^{ 2}}}}
  } { {\left[ 3 \right]_{q^{ 2}}} {\left[ 7 
   \right]_{q^{ 2} }}} {{\left[ \tfrac{9m+4}2 \right]_{q^4}}} 
 {{\left[ 9 m+5 \right]_{q^2}}} \\
\times{{\left[ 
   \tfrac{3m+2}{4}  \right]_{q^{ 24}}}} \frac {{ {\left[ 6 
   \right]_{q^{ 4}}} }} { {\left[ 2 \right]_{q^{ 4}}} 
  {\left[ 3 \right]_{q^{ 4}} }} {\left[ \tfrac{9m+7}{5} 
   \right]_{q^{ 10}}} {\left[ m+1 \right]_{q^{ 18}}}
,\end{multline*} 
which, by Corollary~\ref{cor:A}, is a polynomial in $q$ with 
non-negative integer coefficients.
If $m\equiv
3
~(\text{mod }140),$ then we have \begin{multline*}\Cat^m(E_7;q)=
{{\left[ \tfrac{9m+1}{7}  \right]_{q^{ 14}}}} {{\left[ 
   \tfrac{3m+1}{10}  \right]_{q^{ 60}}}} \frac {{ {\left[ 30
    \right]_{q^{ 2}}}} } { {\left[ 5 \right]_{q^{ 2}}} 
  {\left[ 6 \right]_{q^{ 2}}}} {\left[ 9m+4 
  \right]_{q^{ 2}}}\\
\times {\left[ \tfrac{9m+5}{4}  \right]_{q^{ 8}}} 
 {\left[ 3m+2 \right]_{q^{ 6}}} {{\left[ 9 m+7 
  \right]_{q^2}}} {\left[ m+1 \right]_{q^{ 18}}}
,\end{multline*} 
which, by Corollary~\ref{cor:A}, is a polynomial in $q$ with 
non-negative integer coefficients.
If $m\equiv
4
~(\text{mod }140),$ then we have \begin{multline*}\Cat^m(E_7;q)=
{\left[ 9 m+1 \right]_{q^{ 2}}} {\left[ 3 m+1 \right]_{q^{ 
  6}}} {\left[ \tfrac{9m+4}{20}  \right]_{q^{ 40}}} 
  \frac {{ {\left[ 20 \right]_{q^{ 2}}} }} { 
  {\left[ 4 \right]_{q^{ 2}}} {\left[ 5 \right]_{q^{ 2}}} 
  } {\left[ 9 m+5 \right]_{q^2}} \\
\times{\left[ 
   \tfrac{3m+2}{14}  \right]_{q^{ 84}}} \frac {{ {\left[ 42
    \right]_{q^{ 2}}} }} { {\left[ 6 \right]_{q^{ 2}}} 
  {\left[ 7 \right]_{q^{ 2}}} } {\left[ 9 m+7 
  \right]_{q^2}} {\left[ m+1 \right]_{q^{ 18}}}
,\end{multline*} 
which, by Corollary~\ref{cor:A}, is a polynomial in $q$ with 
non-negative integer coefficients.
If $m\equiv
5
~(\text{mod }140),$ then we have $$\Cat^m(E_7;q)=
{\left[ 9 m+1 \right]_{q^{ 2}}} {\left[ \tfrac{3m+1}{2}  
  \right]_{q^{ 12}}} {\left[ \tfrac{9m+4}{7}  \right]_{q^{ 14}}} 
 {\left[ \tfrac{9m+5}{5}  \right]_{q^{ 10}}} {\left[ 3m+2 
  \right]_{q^{ 6}}} {\left[ \tfrac{9m+7}{4}  \right]_{q^{ 8}}} 
 {\left[ m+1 \right]_{q^{ 18}}}
,$$ 
which is manifestly a polynomial in $q$ with 
non-negative integer coefficients.  
If $m\equiv
6
~(\text{mod }140),$ then we have \begin{multline*}\Cat^m(E_7;q)=
{\left[ \tfrac{9m+1}{5}  \right]_{q^{ 10}}} {\left[ 3 m+1 
  \right]_{q^{ 6}}} {\left[ \tfrac{9m+4}2 \right]_{q^{ 4}}} {\left[ 
  9 m+5 \right]_{q^2}} \\
\times{\left[ \tfrac{3m+2}{4}  
   \right]_{q^{ 24}}} \frac {{ {\left[ 6 \right]_{q^{ 4}}} 
  }} { {\left[ 2 \right]_{q^{ 4}}} {\left[ 3 
   \right]_{q^{ 4}}}}  {\left[ 9 m+7 \right]_{q^2}} 
 {\left[ \tfrac{m+1}{7} \right]_{q^{ 126}}} \frac {{ 
  {\left[ 63 \right]_{q^{ 2}}} }} { {\left[ 7 
   \right]_{q^{ 2}}} {\left[ 9 \right]_{q^{ 2}}} }
,\end{multline*} 
which, by Corollary~\ref{cor:A}, is a polynomial in $q$ with 
non-negative integer coefficients.
If $m\equiv
7
~(\text{mod }140),$ then we have \begin{multline*}\Cat^m(E_7;q)=
{\left[ \tfrac{9m+1}{4}  \right]_{q^{ 8}}} {\left[ 
  \tfrac{3m+1}{2}  \right]_{q^{ 12}}} {\left[ 9m+4 \right]_{q^{ 
  2}}}\\
\times {\left[ 9 m+5 \right]_{q^2}} {\left[ 3m+2 
  \right]_{q^{ 6}}} {\left[ \tfrac{9m+7}{35}  \right]_{q^{ 70}}} 
  \frac {{ {\left[ 35 \right]_{q^{ 2}}} }} { 
  {\left[ 5 \right]_{q^{ 2}}} {\left[ 7 \right]_{q^{ 2}}} 
  } {\left[ m+1 \right]_{q^{ 18}}}
,\end{multline*} 
which, by Corollary~\ref{cor:A}, is a polynomial in $q$ with 
non-negative integer coefficients.
If $m\equiv
8
~(\text{mod }140),$ then we have \begin{multline*}\Cat^m(E_7;q)=
{\left[ 9 m+1 \right]_{q^{ 2}}} {\left[ \tfrac{3m+1}{5}  
   \right]_{q^{ 30}}} \frac {{ {\left[ 15 \right]_{q^{ 2}}} 
  }} { {\left[ 3 \right]_{q^{ 2}}} {\left[ 5 
   \right]_{q^{ 2}}} } {\left[ \tfrac{9m+4}{4}  \right]_{q^{ 
  8}}} \\
\times{\left[ \tfrac{9m+5}{7}  \right]_{q^{ 14}}} {\left[ 
  \tfrac{3m+2}{2}  \right]_{q^{ 12}}} {\left[ 9 m+7 
  \right]_{q^2}} {\left[ m+1 \right]_{q^{ 18}}}
,\end{multline*} 
which, by Corollary~\ref{cor:A}, is a polynomial in $q$ with 
non-negative integer coefficients.
If $m\equiv
9
~(\text{mod }140),$ then we have \begin{multline*}\Cat^m(E_7;q)=
{\left[ 9 m+1 \right]_{q^{ 2}}} {\left[ \tfrac{3m+1}{14}  
   \right]_{q^{ 84}}} \frac {{ {\left[ 42 \right]_{q^{ 2}}} 
  }} { {\left[ 6 \right]_{q^{ 2}}} {\left[ 7 
   \right]_{q^{ 2}}} } {\left[ \tfrac{9m+4}{5}  \right]_{q^{ 
  10}}}\\
\times {\left[ 9 m+5 \right]_{q^2}} {\left[ 3m+2 
  \right]_{q^{ 6}}} {\left[ \tfrac{9m+7}{4}  \right]_{q^{ 8}}} 
 {\left[ m+1 \right]_{q^{ 18}}}
,\end{multline*} 
which, by Corollary~\ref{cor:A}, is a polynomial in $q$ with 
non-negative integer coefficients.
If $m\equiv
10
~(\text{mod }140),$ then we have \begin{multline*}\Cat^m(E_7;q)=
{\left[ \tfrac{9m+1}{7}  \right]_{q^{ 14}}} {\left[ 3 m+1 
  \right]_{q^{ 6}}} {\left[ \tfrac{9m+4}2 \right]_{q^{ 4}}} {\left[ 
  \tfrac{9m+5}{5}  \right]_{q^{ 10}}} \\
\times{\left[ \tfrac{3m+2}{4}  
   \right]_{q^{ 24}}} \frac {{ {\left[ 6 \right]_{q^{ 4}}} 
  }} { {\left[ 2 \right]_{q^{ 4}}} {\left[ 3 
   \right]_{q^{ 4}}} } {\left[9 m+7 \right]_{q^2}} 
 {\left[ m+1 \right]_{q^{ 18}}}
,\end{multline*} 
which, by Corollary~\ref{cor:A}, is a polynomial in $q$ with 
non-negative integer coefficients.
If $m\equiv
11
~(\text{mod }140),$ then we have \begin{multline*}\Cat^m(E_7;q)=
{\left[ \tfrac{9m+1}{5}  \right]_{q^{ 10}}} {\left[ 
   \tfrac{3m+1}{2}  \right]_{q^{ 12}}} 
 {\left[ 9m+4 
  \right]_{q^{ 2}}} {\left[ \tfrac{9m+5}{4}  \right]_{q^{ 8}}} \\
\times
 {\left[ \tfrac{3m+2}{7}  \right]_{q^{ 42}}}
\frac {[21]_{q^2}} {[3]_{q^2}\left[7\right]_{q^2}}
 {\left[ 
  9 m+7 \right]_{q^2}} {\left[ m+1 \right]_{q^{ 18}}}
,\end{multline*} 
which, by Corollary~\ref{cor:A}, is a polynomial in $q$ with 
non-negative integer coefficients.
If $m\equiv
12
~(\text{mod }140),$ then we have \begin{multline*}\Cat^m(E_7;q)=
{\left[ 9 m+1 \right]_{q^{ 2}}} {\left[ 3 m+1 \right]_{q^{ 
  6}}} {\left[ \tfrac{9m+4}{28}  \right]_{q^{ 56}}} 
  \frac {{ {\left[ 28 \right]_{q^{ 2}}} }} { 
  {\left[ 4 \right]_{q^{ 2}}} {\left[ 7 \right]_{q^{ 2}}
  }} {\left[ 9 m+5 \right]_{q^2}}\\
\times {\left[ 
  \tfrac{3m+2}{2}  \right]_{q^{ 12}}} {\left[ \tfrac{9m+7}{5}  
  \right]_{q^{ 10}}} {\left[ m+1 \right]_{q^{ 18}}}
,\end{multline*} 
which, by Corollary~\ref{cor:A}, is a polynomial in $q$ with 
non-negative integer coefficients.
If $m\equiv
13
~(\text{mod }140),$ then we have \begin{multline*}\Cat^m(E_7;q)=
{\left[ 9 m+1 \right]_{q^{ 2}}} {\left[ \tfrac{3m+1}{10}  
   \right]_{q^{ 60}}} \frac {{ {\left[ 30 \right]_{q^{ 2}}} 
  }} { {\left[ 5 \right]_{q^{ 2}}} {\left[ 6 
   \right]_{q^{ 2}}} } {\left[ 9m+4 \right]_{q^2}} \\
\times
 {\left[ 9 m+5 \right]_{q^2}} {\left[ 3m+2 
  \right]_{q^{ 6}}} {\left[ \tfrac{9m+7}{4}  \right]_{q^{ 8}}} 
 {\left[ \tfrac{m+1}{7} \right]_{q^{ 126}}} \frac {{ 
  {\left[ 63 \right]_{q^{ 2}}} }} { {\left[ 7 
   \right]_{q^{ 2}}} {\left[ 9 \right]_{q^{ 2}}} }
,\end{multline*} 
which, by Corollary~\ref{cor:A}, is a polynomial in $q$ with 
non-negative integer coefficients.
If $m\equiv
14
~(\text{mod }140),$ then we have \begin{multline*}\Cat^m(E_7;q)=
{\left[ 9 m+1 \right]_{q^{ 2}}} {\left[ 3 m+1 \right]_{q^{ 
  6}}} {\left[ \tfrac{9m+4}{10}  \right]_{q^{ 20}}}
\frac {[10]_{q^2}} {[2]_{q^2}\left[5\right]_{q^2}} \\
\times{\left[ 
  9 m+5 \right]_{q^2}} {\left[ \tfrac{3m+2}{4}  
   \right]_{q^{ 24}}} \frac {{ {\left[ 6 \right]_{q^{ 4}}} 
  }} { {\left[ 2 \right]_{q^{ 4}}} {\left[ 3 
   \right]_{q^{ 4}}} } {\left[ \tfrac{9m+7}{7}  \right]_{q^{ 
  14}}} {\left[ m+1 \right]_{q^{ 18}}}
.\end{multline*} 
If one decomposes $[9m+5]_{q^2}$ as 
$[\frac {9m} {2}+3]_{q^4}+q^2[\frac {9m} {2}+2]_{q^4}$,
then one sees that,  by Corollary~\ref{cor:A},
this is a polynomial in $q$ with 
non-negative integer coefficients.
If $m\equiv
15
~(\text{mod }140),$ then we have \begin{multline*}\Cat^m(E_7;q)=
{\left[ \tfrac{9m+1}{4}  \right]_{q^{ 8}}} {\left[ 
  \tfrac{3m+1}{2}  \right]_{q^{ 12}}} {\left[ 9m+4 \right]_{q^{ 
  2}}} \\
\times{\left[ \tfrac{9m+5}{35}  \right]_{q^{ 70}}} 
  \frac {{ {\left[ 35 \right]_{q^{ 2}}} }} { 
  {\left[ 5 \right]_{q^{ 2}}} {\left[ 7 \right]_{q^{ 2}}} 
  } {\left[ 3m+2 \right]_{q^{ 6}}} {\left[ 9 m+7 
  \right]_{q^2}} {\left[ m+1 \right]_{q^{ 18}}}
,\end{multline*} 
which, by Corollary~\ref{cor:A}, is a polynomial in $q$ with 
non-negative integer coefficients.
If $m\equiv
16
~(\text{mod }140),$ then we have \begin{multline*}\Cat^m(E_7;q)=
{\left[ \tfrac{9m+1}{5}  \right]_{q^{ 10}}} {\left[ 
   \tfrac{3m+1}{7}  \right]_{q^{ 42}}} \frac {{ {\left[ 21 
   \right]_{q^{ 2}}} }} { {\left[ 3 \right]_{q^{ 2}}} 
  {\left[ 7 \right]_{q^{ 2}}} } {\left[ \tfrac{9m+4}{4} 
   \right]_{q^{ 8}}} {\left[ 9m+5 \right]_{q^{ 2}}} \\
\times{\left[ 
  \tfrac{3m+2}{2}  \right]_{q^{ 12}}} {\left[ 9 m+7 
  \right]_{q^2}} {\left[ m+1 \right]_{q^{ 18}}}
,\end{multline*} 
which, by Corollary~\ref{cor:A}, is a polynomial in $q$ with 
non-negative integer coefficients.
If $m\equiv
17
~(\text{mod }140),$ then we have \begin{multline*}\Cat^m(E_7;q)=
{\left[ \tfrac{9m+1}{7}  \right]_{q^{ 14}}} {\left[ 
   \tfrac{3m+1}{4}  \right]_{q^{ 24}}} \frac {{ {\left[ 6 
   \right]_{q^{ 4}}} }} { {\left[ 2 \right]_{q^{ 4}}} 
  {\left[ 3 \right]_{q^{ 4}}} } {\left[ 9m+4 
  \right]_{q^{ 2}}} {\left[\tfrac{ 9 m+5}2 \right]_{q^4}}\\
\times {\left[ 
  3m+2 \right]_{q^{ 6}}} {\left[ \tfrac{9m+7}{5}  \right]_{q^{ 
  10}}} {\left[ m+1 \right]_{q^{ 18}}}
,\end{multline*} 
which, by Corollary~\ref{cor:A}, is a polynomial in $q$ with 
non-negative integer coefficients.
If $m\equiv
18
~(\text{mod }140),$ then we have \begin{multline*}\Cat^m(E_7;q)=
{\left[ 9 m+1 \right]_{q^{ 2}}} 
{\left[ \tfrac{3m+1}5 
   \right]_{q^{ 30}}} 
\frac {[15]_{q^2}} {[3]_{q^2}\left[5\right]_{q^2}}\\
\times
{\left[ \tfrac{9m+4}2 \right]_{q^{ 4}}} 
 {\left[ 9m+5 \right]_{q^{ 2}}} {\left[ \tfrac{3m+2}{28} 
   \right]_{q^{ 168}}} 
\frac {[84]_{q^2}\left[2\right]_{q^2}} 
{[4]_{q^2}\left[6\right]_{q^2}\left[7\right]_{q^2}}
{\left[ 9m+7 \right]_{q^{ 2}}} 
 {\left[ m+1 \right]_{q^{ 18}}}
,\end{multline*} 
which, by Corollary~\ref{cor:A} and Lemma~\ref{lem:H}, 
is a polynomial in $q$ with 
non-negative integer coefficients.
If $m\equiv
19
~(\text{mod }140),$ then we have \begin{multline*}\Cat^m(E_7;q)=
{\left[ \tfrac{9m+1}{4}  \right]_{q^{ 8}}} {\left[ 
  \tfrac{3m+1}{2}  \right]_{q^{ 12}}} {\left[ \tfrac{9m+4}{35} 
    \right]_{q^{ 70}}} \frac {{ {\left[ 35 \right]_{q^{ 2}}} 
  }} { {\left[ 5 \right]_{q^{ 2}}} {\left[ 7 
   \right]_{q^{ 2}}} } {\left[ 9m+5 \right]_{q^{ 2}}} \\
\times
 {\left[ 3m+2 \right]_{q^{ 6}}} {\left[ 9 m+7 
  \right]_{q^2}} {\left[ m+1 \right]_{q^{ 18}}}
,\end{multline*} 
which, by Corollary~\ref{cor:A}, is a polynomial in $q$ with 
non-negative integer coefficients.
If $m\equiv
20
~(\text{mod }140),$ then we have \begin{multline*}\Cat^m(E_7;q)=
{\left[ 9 m+1 \right]_{q^{ 2}}} {\left[ 3 m+1 \right]_{q^{ 
  6}}} {\left[ \tfrac{9m+4}{4}  \right]_{q^{ 8}}} {\left[ 
  \tfrac{9m+5}{5}  \right]_{q^{ 10}}} {\left[ \tfrac{3m+2}{2}  
  \right]_{q^{ 12}}} \\
\times{\left[ 9 m+7 \right]_{q^2}} 
 {\left[ \tfrac{m+1}{7} \right]_{q^{ 126}}} \frac {{ 
  {\left[ 63 \right]_{q^{ 2}}} }} { {\left[ 7 
   \right]_{q^{ 2}}} {\left[ 9 \right]_{q^{ 2}}} }
,\end{multline*} 
which, by Corollary~\ref{cor:A}, is a polynomial in $q$ with 
non-negative integer coefficients.
If $m\equiv
21
~(\text{mod }140),$ then we have \begin{multline*}\Cat^m(E_7;q)=
{\left[ \tfrac{9m+1}{5}  \right]_{q^{ 10}}} {\left[ 
   \tfrac{3m+1}{4}  \right]_{q^{ 24}}} \frac {{ {\left[ 6 
   \right]_{q^{ 4}}} }} { {\left[ 2 \right]_{q^{ 4}}} 
  {\left[ 3 \right]_{q^{ 4}}} } {\left[ 9m+4 
  \right]_{q^{ 2}}} {\left[ \tfrac{9 m+5}2 \right]_{q^4}}\\
\times {\left[ 
  3m+2 \right]_{q^{ 6}}} {\left[ \tfrac{9m+7}{7}  \right]_{q^{ 
  14}}} {\left[ m+1 \right]_{q^{ 18}}}
,\end{multline*} 
which, by Corollary~\ref{cor:A}, is a polynomial in $q$ with 
non-negative integer coefficients.
If $m\equiv
22
~(\text{mod }140),$ then we have \begin{multline*}\Cat^m(E_7;q)=
{\left[ 9 m+1 \right]_{q^{ 2}}} {\left[ 3 m+1 \right]_{q^{ 
  6}}} {\left[ \tfrac{9m+4}2 \right]_{q^4}} {\left[ 
  \tfrac{9m+5}{7}  \right]_{q^{ 14}}} {\left[ \tfrac{3m+2}{4}  
   \right]_{q^{ 24}}} \frac {{ {\left[ 6 \right]_{q^{ 4}}} 
  }} { {\left[ 2 \right]_{q^{ 4}}} {\left[ 3 
   \right]_{q^{ 4}}} } \\
\times{\left[ \tfrac{9m+7}{5}  \right]_{q^{ 
  10}}} {\left[ m+1 \right]_{q^{ 18}}}
,\end{multline*} 
which, by Corollary~\ref{cor:A}, is a polynomial in $q$ with 
non-negative integer coefficients.
If $m\equiv
23
~(\text{mod }140),$ then we have \begin{multline*}\Cat^m(E_7;q)=
{\left[ \tfrac{9m+1}{4}  \right]_{q^{ 8}}} {\left[ 
   \tfrac{3 m+1}{35}  \right]_{q^{ 210}}} \frac {{ {\left[ 105
    \right]_{q^{ 2}}} }} { {\left[ 3 \right]_{q^{ 2}}}
{\left[ 5 \right]_{q^{ 2}}} 
  {\left[ 7 \right]_{q^{ 2}}} } {\left[ 9m+4 
  \right]_{q^{ 2}}} {\left[ 9 m+5 \right]_{q^2}}\\
\times {\left[ 
  3m+2 \right]_{q^{ 6}}} {\left[ 9 m+7 \right]_{q^2}} 
 {\left[ \tfrac{m+1}{2} \right]_{q^{ 36}}} \frac {{ 
  {\left[ 6 \right]_{q^{ 6}}} }} { {\left[ 2 
   \right]_{q^{ 6}}} {\left[ 3 \right]_{q^{ 6}}} }
,\end{multline*} 
which, by Corollary~\ref{cor:A} and Lemma~\ref{lem:I}, 
is a polynomial in $q$ with 
non-negative integer coefficients.
If $m\equiv
24
~(\text{mod }140),$ then we have \begin{multline*}\Cat^m(E_7;q)=
{\left[ \tfrac{9m+1}{7}  \right]_{q^{ 14}}} {\left[ 3 m+1 
  \right]_{q^{ 6}}} {\left[ \tfrac{9m+4}{20}  \right]_{q^{ 40}}} 
  \frac {{ {\left[ 20 \right]_{q^{ 2}}} }} { 
  {\left[ 4 \right]_{q^{ 2}}} {\left[ 5 \right]_{q^{ 2}}} }
   {\left[ 9m+5 \right]_{q^{ 2}}} \\
\times{\left[ 
  \tfrac{3m+2}{2}  \right]_{q^{ 12}}} {\left[ 9 m+7 
  \right]_{q^2}} {\left[ m+1 \right]_{q^{ 18}}}
,\end{multline*} 
which, by Corollary~\ref{cor:A}, is a polynomial in $q$ with 
non-negative integer coefficients.
If $m\equiv
25
~(\text{mod }140),$ then we have \begin{multline*}\Cat^m(E_7;q)=
{\left[ 9 m+1 \right]_{q^{ 2}}} {\left[ \tfrac{3m+1}{2}  
   \right]_{q^{ 12}}} \\
\times
 {\left[ 9m+4 \right]_{q^2}} 
 {\left[ \tfrac{9m+5}{5}  \right]_{q^{ 10}}} {\left[ 
  \tfrac{3m+2}{7}  \right]_{q^{ 42}}}
\frac {[21]_{q^2}} {[3]_{q^2}\left[7\right]_{q^2}}
 {\left[ \tfrac{9m+7}{4}  
  \right]_{q^{ 8}}} {\left[ m+1 \right]_{q^{ 18}}}
,\end{multline*} 
which, by Corollary~\ref{cor:A}, is a polynomial in $q$ with 
non-negative integer coefficients.
If $m\equiv
26
~(\text{mod }140),$ then we have \begin{multline*}\Cat^m(E_7;q)=
{\left[ \tfrac{9m+1}{5}  \right]_{q^{ 10}}} {\left[ 3 m+1 
  \right]_{q^{ 6}}} {\left[ \tfrac{9m+4}{14}  \right]_{q^{ 28}}}
\frac {[14]_{q^2}} {[2]_{q^2}\left[7\right]_{q^2}} 
 {\left[ 9m+5 \right]_{q^{ 2}}} \\
\times{\left[ \tfrac{3m+2}{4}  
   \right]_{q^{ 24}}} \frac {{ {\left[ 6 \right]_{q^{ 4}}} 
  }} { {\left[ 2 \right]_{q^{ 4}}} {\left[ 3 
   \right]_{q^{ 4}}} } {\left[ 9 m+7 \right]_{q^2}} 
 {\left[ m+1 \right]_{q^{ 18}}}
.\end{multline*} 
If one decomposes $[9m+7]_{q^2}$ as 
$[\frac {9m} {2}+4]_{q^4}+q^2[\frac {9m} {2}+3]_{q^4}$,
then one sees that, by Corollary~\ref{cor:A},
this is a polynomial in $q$ with 
non-negative integer coefficients.
If $m\equiv
27
~(\text{mod }140),$ then we have \begin{multline*}\Cat^m(E_7;q)=
{\left[ \tfrac{9m+1}{4}  \right]_{q^{ 8}}} {\left[ 
  \tfrac{3m+1}{2}  \right]_{q^{ 12}}} {\left[ 9m+4 \right]_{q^{ 
  2}}}\\
\times {\left[ 9 m+5 \right]_{q^2}} {\left[ 3m+2 
  \right]_{q^{ 6}}} {\left[ \tfrac{9m+7}{5}  \right]_{q^{ 10}}} 
 {\left[ \tfrac{m+1}{7} \right]_{q^{ 126}}} \frac {{ 
  {\left[ 63 \right]_{q^{ 2}}} }} { {\left[ 7 
   \right]_{q^{ 2}}} {\left[ 9 \right]_{q^{ 2}}} }
,\end{multline*} 
which, by Corollary~\ref{cor:A}, is a polynomial in $q$ with 
non-negative integer coefficients.
If $m\equiv
28
~(\text{mod }140),$ then we have \begin{multline*}\Cat^m(E_7;q)=
{\left[ 9 m+1 \right]_{q^{ 2}}} {\left[ \tfrac{3m+1}{5}  
   \right]_{q^{ 30}}} \frac {{ {\left[ 15 \right]_{q^{ 2}}} 
  }} { {\left[ 3 \right]_{q^{ 2}}} {\left[ 5 
   \right]_{q^{ 2}}} } {\left[ \tfrac{9m+4}{4}  \right]_{q^{ 
  8}}} {\left[ 9 m+5 \right]_{q^2}}\\
\times {\left[ 
  \tfrac{3m+2}{2}  \right]_{q^{ 12}}} {\left[ \tfrac{9m+7}{7}  
  \right]_{q^{ 14}}} {\left[ m+1 \right]_{q^{ 18}}}
,\end{multline*} 
which, by Corollary~\ref{cor:A}, is a polynomial in $q$ with 
non-negative integer coefficients.
If $m\equiv
29
~(\text{mod }140),$ then we have $$\Cat^m(E_7;q)=
{\left[ 9 m+1 \right]_{q^{ 2}}} {\left[ \tfrac{3m+1}{2}  
  \right]_{q^{ 12}}} {\left[ \tfrac{9m+4}{5}  \right]_{q^{ 10}}} 
 {\left[ \tfrac{9m+5}{7}  \right]_{q^{ 14}}} {\left[ 3m+2 
  \right]_{q^{ 6}}} {\left[ \tfrac{9m+7}{4}  \right]_{q^{ 8}}} 
 {\left[ m+1 \right]_{q^{ 18}}}
,$$ 
which is manifestly a polynomial in $q$ with 
non-negative integer coefficients.  
If $m\equiv
30
~(\text{mod }140),$ then we have \begin{multline*}\Cat^m(E_7;q)=
{\left[ 9 m+1 \right]_{q^{ 2}}} {\left[ \tfrac{3m+1}{7}  
   \right]_{q^{ 42}}} \frac {{ {\left[ 21 \right]_{q^{ 2}}} 
  }} { {\left[ 3 \right]_{q^{ 2}}} {\left[ 7 
   \right]_{q^{ 2}}} } {\left[ \tfrac{9m+4}2 \right]_{q^4}} 
 {\left[ \tfrac{9m+5}{5}  \right]_{q^{ 10}}} \\
\times{\left[ 
   \tfrac{3m+2}{4}  \right]_{q^{ 24}}} \frac {{ {\left[ 6 
   \right]_{q^{ 4}}} }} { {\left[ 2 \right]_{q^{ 4}}} 
  {\left[ 3 \right]_{q^{ 4}}} } {\left[ 9 m+7 
  \right]_{q^2}} {\left[ m+1 \right]_{q^{ 18}}}
,\end{multline*} 
which, by Corollary~\ref{cor:A}, is a polynomial in $q$ with 
non-negative integer coefficients.
If $m\equiv
31
~(\text{mod }140),$ then we have \begin{multline*}\Cat^m(E_7;q)=
{\left[ \tfrac{9m+1}{35}  \right]_{q^{ 70}}} \frac {{ 
  {\left[ 35 \right]_{q^{ 2}}} }} { {\left[ 5 
   \right]_{q^{ 2}}} {\left[ 7 \right]_{q^{ 2}}} } {\left[ 
  \tfrac{3m+1}{2}  \right]_{q^{ 12}}} {\left[ 9m+4 \right]_{q^{ 
  2}}}\\
\times {\left[ \tfrac{9m+5}{4}  \right]_{q^{ 8}}} {\left[ 3m+2
   \right]_{q^{ 6}}} {\left[ 9 m+7 \right]_{q^2}} 
 {\left[ m+1 \right]_{q^{ 18}}}
,\end{multline*} 
which, by Corollary~\ref{cor:A}, is a polynomial in $q$ with 
non-negative integer coefficients.
If $m\equiv
32
~(\text{mod }140),$ then we have \begin{multline*}\Cat^m(E_7;q)=
{\left[ 9 m+1 \right]_{q^{ 2}}} {\left[ 3 m+1 \right]_{q^{ 
  6}}} {\left[ \tfrac{9m+4}{4}  \right]_{q^{ 8}}} \\
\times{\left[ 
  9 m+5 \right]_{q^2}} {\left[ \tfrac{3m+2}{14}  
   \right]_{q^{ 84}}} \frac {{ {\left[ 42 \right]_{q^{ 2}}} 
  }} { {\left[ 6 \right]_{q^{ 2}}} {\left[ 7 
   \right]_{q^{ 2}}} } {\left[ \tfrac{9m+7}{5}  \right]_{q^{ 
  10}}} {\left[ m+1 \right]_{q^{ 18}}}
,\end{multline*} 
which, by Corollary~\ref{cor:A}, is a polynomial in $q$ with 
non-negative integer coefficients.
If $m\equiv
33
~(\text{mod }140),$ then we have \begin{multline*}\Cat^m(E_7;q)=
{\left[ 9 m+1 \right]_{q^{ 2}}} {\left[ \tfrac{3m+1}{10}  
   \right]_{q^{ 60}}} \frac {{ {\left[ 30 \right]_{q^{ 2}}} 
  }} { {\left[ 5 \right]_{q^{ 2}}} {\left[ 6 
   \right]_{q^{ 2}}} } {\left[ \tfrac{9m+4}{7}  \right]_{q^{ 
  14}}}\\
\times {\left[ 9 m+5 \right]_{q^2}} {\left[ 3m+2 
  \right]_{q^{ 6}}} {\left[ \tfrac{9m+7}{4}  \right]_{q^{ 8}}} 
 {\left[ m+1 \right]_{q^{ 18}}}
,\end{multline*} 
which, by Corollary~\ref{cor:A}, is a polynomial in $q$ with 
non-negative integer coefficients.
If $m\equiv
34
~(\text{mod }140),$ then we have \begin{multline*}\Cat^m(E_7;q)=
{\left[ 9 m+1 \right]_{q^{ 2}}} {\left[ 3 m+1 \right]_{q^{ 
  6}}} {\left[ \tfrac{9m+4}{10}  \right]_{q^{ 20}}}
\frac {[10]_{q^2}} {[2]_{q^2}\left[5\right]_{q^2}}
 {\left[ 
  9 m+5 \right]_{q^2}}\\
\times {\left[ \tfrac{3m+2}{4}  
   \right]_{q^{ 24}}} \frac {{ {\left[ 6 \right]_{q^{ 4}}} 
  }} { {\left[ 2 \right]_{q^{ 4}}} {\left[ 3 
   \right]_{q^{ 4}}} } {\left[ 9 m+7 \right]_{q^2}} 
 {\left[ \tfrac{m+1}{7} \right]_{q^{ 126}}} \frac {{ 
  {\left[ 63 \right]_{q^{ 2}}} }} { {\left[ 7 
   \right]_{q^{ 2}}} {\left[ 9 \right]_{q^{ 2}}} }
.\end{multline*} 
If one decomposes $[9m+7]_{q^2}$ as 
$[\frac {9m} {2}+4]_{q^4}+q^2[\frac {9m} {2}+3]_{q^4}$,
then one sees that, by Corollary~\ref{cor:A},
this is a polynomial in $q$ with 
non-negative integer coefficients.
If $m\equiv
35
~(\text{mod }140),$ then we have $$\Cat^m(E_7;q)=
{\left[ \tfrac{9m+1}{4}  \right]_{q^{ 8}}} {\left[ 
  \tfrac{3m+1}{2}  \right]_{q^{ 12}}} {\left[ 9m+4 \right]_{q^{ 
  2}}} {\left[ \tfrac{9m+5}{5}  \right]_{q^{ 10}}} {\left[ 
  3m+2 \right]_{q^{ 6}}} {\left[ \tfrac{9m+7}{7}  \right]_{q^{ 
  14}}} {\left[ m+1 \right]_{q^{ 18}}}
,$$ 
which is manifestly a polynomial in $q$ with 
non-negative integer coefficients.  
If $m\equiv
36
~(\text{mod }140),$ then we have $$\Cat^m(E_7;q)=
{\left[ \tfrac{9m+1}{5}  \right]_{q^{ 10}}} {\left[ 3 m+1 
  \right]_{q^{ 6}}} {\left[ \tfrac{9m+4}{4}  \right]_{q^{ 8}}} 
 {\left[ \tfrac{9m+5}{7}  \right]_{q^{ 14}}} {\left[ 
  \tfrac{3m+2}{2}  \right]_{q^{ 12}}} {\left[ 9m+7 \right]_{q^{ 
  2}}} {\left[ m+1 \right]_{q^{ 18}}}
,$$ 
which is manifestly a polynomial in $q$ with 
non-negative integer coefficients.  
If $m\equiv
37
~(\text{mod }140),$ then we have \begin{multline*}\Cat^m(E_7;q)=
{\left[ 9 m+1 \right]_{q^{ 2}}} {\left[ \tfrac{3m+1}{14}  
   \right]_{q^{ 84}}} \frac {{ {\left[ 42 \right]_{q^{ 2}}} 
  }} { {\left[ 6 \right]_{q^{ 2}}} {\left[ 7 
   \right]_{q^{ 2}}} } {\left[ 9m+4 \right]_{q^2}}\\
\times 
 {\left[ 9 m+5 \right]_{q^2}} {\left[ 3m+2 
  \right]_{q^{ 6}}} {\left[ \tfrac{9m+7}{20}  \right]_{q^{ 40}}} 
  \frac {{ {\left[ 20 \right]_{q^{ 2}}} }} { 
  {\left[ 4 \right]_{q^{ 2}}} {\left[ 5 \right]_{q^{ 2}}} 
  } {\left[ m+1 \right]_{q^{ 18}}}
,\end{multline*} 
which, by Corollary~\ref{cor:A}, is a polynomial in $q$ with 
non-negative integer coefficients.
If $m\equiv
38
~(\text{mod }140),$ then we have \begin{multline*}\Cat^m(E_7;q)=
{\left[ \tfrac{9m+1}{7}  \right]_{q^{ 14}}} {\left[ 
   \tfrac{3m+1}{5}  \right]_{q^{ 30}}} \frac {{ {\left[ 15 
   \right]_{q^{ 2}}} }} { {\left[ 3 \right]_{q^{ 2}}} 
  {\left[ 5 \right]_{q^{ 2}}} } {\left[ \tfrac{9m+4}2 
  \right]_{q^{ 4}}} {\left[ 9 m+5 \right]_{q^2}} \\
\times
 {\left[ \tfrac{3m+2}{4}  \right]_{q^{ 24}}} \frac {{ 
  {\left[ 6 \right]_{q^{ 4}}} }} { {\left[ 2 
   \right]_{q^{ 4}}} {\left[ 3 \right]_{q^{ 4}}} } 
 {\left[ 9 m+7 \right]_{q^2}} {\left[ m+1 \right]_{q^{ 
  18}}}
,\end{multline*} 
which, by Corollary~\ref{cor:A}, is a polynomial in $q$ with 
non-negative integer coefficients.
If $m\equiv
39
~(\text{mod }140),$ then we have \begin{multline*}\Cat^m(E_7;q)=
{\left[ \tfrac{9m+1}{4}  \right]_{q^{ 8}}} {\left[ 
   \tfrac{3m+1}{2}  \right]_{q^{ 12}}} 
{\left[ \tfrac{9m+4}{5} 
   \right]_{q^{ 10}}} {\left[ 9m+5 \right]_{q^{ 2}}} \\
\times{\left[ 
  \tfrac{3m+2}{7}  \right]_{q^{ 42}}} 
\frac {[21]_{q^2}} {[3]_{q^2}\left[7\right]_{q^2}}
{\left[ 9 m+7
  \right]_{q^2}} {\left[ m+1 \right]_{q^{ 18}}}
,\end{multline*} 
which, by Corollary~\ref{cor:A}, is a polynomial in $q$ with 
non-negative integer coefficients.
If $m\equiv
40
~(\text{mod }140),$ then we have \begin{multline*}\Cat^m(E_7;q)=
{\left[ 9 m+1 \right]_{q^{ 2}}} {\left[ 3 m+1 \right]_{q^{ 
  6}}} {\left[ \tfrac{9m+4}{28}  \right]_{q^{ 56}}} 
  \frac {{ {\left[ 28 \right]_{q^{ 2}}} }} { 
  {\left[ 4 \right]_{q^{ 2}}} {\left[ 7 \right]_{q^{ 2}}} 
  }\\
\times {\left[ \tfrac{9m+5}{5}  \right]_{q^{ 10}}} {\left[ 
  \tfrac{3m+2}{2}  \right]_{q^{ 12}}} {\left[ 9 m+7 
  \right]_{q^2}} {\left[ m+1 \right]_{q^{ 18}}}
,\end{multline*} 
which, by Corollary~\ref{cor:A}, is a polynomial in $q$ with 
non-negative integer coefficients.
If $m\equiv
41
~(\text{mod }140),$ then we have \begin{multline*}\Cat^m(E_7;q)=
{\left[ \tfrac{9m+1}{5}  \right]_{q^{ 10}}} {\left[ 
  \tfrac{3m+1}{2}  \right]_{q^{ 12}}} {\left[ 9m+4 \right]_{q^{ 
  2}}}\\
\times {\left[ 9 m+5 \right]_{q^2}} {\left[ 3m+2 
  \right]_{q^{ 6}}} {\left[ \tfrac{9m+7}{4}  \right]_{q^{ 8}}} 
 {\left[ \tfrac{m+1}{7} \right]_{q^{ 126}}} \frac {{ 
  {\left[ 63 \right]_{q^{ 2}}} }} { {\left[ 7 
   \right]_{q^{ 2}}} {\left[ 9 \right]_{q^{ 2}}} }
,\end{multline*} 
which, by Corollary~\ref{cor:A}, is a polynomial in $q$ with 
non-negative integer coefficients.
If $m\equiv
42
~(\text{mod }140),$ then we have \begin{multline*}\Cat^m(E_7;q)=
{\left[ 9 m+1 \right]_{q^{ 2}}} {\left[ 3 m+1 \right]_{q^{ 
  6}}}\\
\times {\left[ \tfrac{9m+4}2 \right]_{q^4}} {\left[ 9 m+5 
  \right]_{q^2}} {\left[ \tfrac{3m+2}{4}  \right]_{q^{ 24}}} 
  \frac {{ {\left[ 6 \right]_{q^{ 4}}} }} { 
  {\left[ 2 \right]_{q^{ 4}}} {\left[ 3 \right]_{q^{ 4}}} 
  } {\left[ \tfrac{9m+7}{35}  \right]_{q^{ 70}}} 
  \frac {{ {\left[ 35 \right]_{q^{ 2}}} }} { 
  {\left[ 5 \right]_{q^{ 2}}} {\left[ 7 \right]_{q^{ 2}}} 
  } {\left[ m+1 \right]_{q^{ 18}}}
,\end{multline*} 
which, by Corollary~\ref{cor:A}, is a polynomial in $q$ with 
non-negative integer coefficients.
If $m\equiv
43
~(\text{mod }140),$ then we have \begin{multline*}\Cat^m(E_7;q)=
{\left[ \tfrac{9m+1}{4}  \right]_{q^{ 8}}} {\left[ 
   \tfrac{3m+1}{10}  \right]_{q^{ 60}}} \frac {{ {\left[ 30
    \right]_{q^{ 2}}} }} { {\left[ 5 \right]_{q^{ 2}}} 
  {\left[ 6 \right]_{q^{ 2}}} } {\left[ 9m+4 
  \right]_{q^{ 2}}}\\
\times {\left[ \tfrac{9m+5}{7}  \right]_{q^{ 14}}} 
 {\left[ 3m+2 \right]_{q^{ 6}}} {\left[ 9 m+7 
  \right]_{q^2}} {\left[ m+1 \right]_{q^{ 18}}}
,\end{multline*} 
which, by Corollary~\ref{cor:A}, is a polynomial in $q$ with 
non-negative integer coefficients.
If $m\equiv
44
~(\text{mod }140),$ then we have \begin{multline*}\Cat^m(E_7;q)=
{\left[ 9 m+1 \right]_{q^{ 2}}} {\left[ \tfrac{3m+1}{7}  
   \right]_{q^{ 42}}} \frac {{ {\left[ 21 \right]_{q^{ 2}}} 
  }} { {\left[ 3 \right]_{q^{ 2}}} {\left[ 7 
   \right]_{q^{ 2}}} } {\left[ \tfrac{9m+4}{20}  
   \right]_{q^{ 40}}} \frac {{ {\left[ 20 \right]_{q^{ 2}}} 
  }} { {\left[ 4 \right]_{q^{ 2}}} {\left[ 5 
   \right]_{q^{ 2}}} } {\left[ 9 m+5 \right]_{q^2}}\\
\times 
 {\left[ \tfrac{3m+2}{2}  \right]_{q^{ 12}}} {\left[ 
  9 m+7 \right]_{q^2}} {\left[ m+1 \right]_{q^{ 18}}}
,\end{multline*} 
which, by Corollary~\ref{cor:A}, is a polynomial in $q$ with 
non-negative integer coefficients.
If $m\equiv
45
~(\text{mod }140),$ then we have $$\Cat^m(E_7;q)=
{\left[ \tfrac{9m+1}{7}  \right]_{q^{ 14}}} {\left[ 
  \tfrac{3m+1}{2}  \right]_{q^{ 12}}} {\left[ 9m+4 \right]_{q^{ 
  2}}} {\left[ \tfrac{9m+5}{5}  \right]_{q^{ 10}}} {\left[ 
  3m+2 \right]_{q^{ 6}}} {\left[ \tfrac{9m+7}{4}  \right]_{q^{ 
  8}}} {\left[ m+1 \right]_{q^{ 18}}}
,$$ 
which is manifestly a polynomial in $q$ with 
non-negative integer coefficients.  
If $m\equiv
46
~(\text{mod }140),$ then we have \begin{multline*}\Cat^m(E_7;q)=
{\left[ \tfrac{9 m+1}5 \right]_{q^{ 10}}} 
{\left[ {3m+1} 
   \right]_{q^{ 6}}} \\
\times
{\left[ \tfrac{9m+4}2 \right]_{q^{ 4}}} 
 {\left[ 9m+5 \right]_{q^{ 2}}} {\left[ \tfrac{3m+2}{28} 
   \right]_{q^{ 168}}} 
\frac {[84]_{q^2}\left[2\right]_{q^2}} 
{[4]_{q^2}\left[6\right]_{q^2}\left[7\right]_{q^2}}
{\left[ 9m+7 \right]_{q^{ 2}}} 
 {\left[ m+1 \right]_{q^{ 18}}}
,\end{multline*} 
which, by Lemma~\ref{lem:H}, is a polynomial in $q$ with 
non-negative integer coefficients.
If $m\equiv
47
~(\text{mod }140),$ then we have $$\Cat^m(E_7;q)=
{\left[ \tfrac{9m+1}{4}  \right]_{q^{ 8}}} {\left[ 
  \tfrac{3m+1}{2}  \right]_{q^{ 12}}} {\left[ \tfrac{9m+4}{7}  
  \right]_{q^{ 14}}} {\left[ 9m+5 \right]_{q^{ 2}}} {\left[ 
  3m+2 \right]_{q^{ 6}}} {\left[ \tfrac{9m+7}{5}  \right]_{q^{ 
  10}}} {\left[ m+1 \right]_{q^{ 18}}}
,$$ 
which is manifestly a polynomial in $q$ with 
non-negative integer coefficients.  
If $m\equiv
48
~(\text{mod }140),$ then we have \begin{multline*}\Cat^m(E_7;q)=
{\left[ 9 m+1 \right]_{q^{ 2}}} {\left[ \tfrac{3m+1}{5}  
   \right]_{q^{ 30}}} \frac {{ {\left[ 15 \right]_{q^{ 2}}} 
  }} { {\left[ 3 \right]_{q^{ 2}}} {\left[ 5 
   \right]_{q^{ 2}}} } {\left[ \tfrac{9m+4}{4}  \right]_{q^{ 
  8}}} {\left[ 9 m+5 \right]_{q^2}}\\
\times {\left[ 
  \tfrac{3m+2}{2}  \right]_{q^{ 12}}} {\left[ 9 m+7 
  \right]_{q^2}} {\left[ \tfrac{m+1}{7} \right]_{q^{ 126}}} 
  \frac {{ {\left[ 63 \right]_{q^{ 2}}} }} { 
  {\left[ 7 \right]_{q^{ 2}}} {\left[ 9 \right]_{q^{ 2}}} 
  }
,\end{multline*} 
which, by Corollary~\ref{cor:A}, is a polynomial in $q$ with 
non-negative integer coefficients.
If $m\equiv
49
~(\text{mod }140),$ then we have \begin{multline*}\Cat^m(E_7;q)=
{\left[ 9 m+1 \right]_{q^{ 2}}} {\left[ \tfrac{3m+1}{4}  
   \right]_{q^{ 24}}} \frac {{ {\left[ 6 \right]_{q^{ 4}}} 
  }} { {\left[ 2 \right]_{q^{ 4}}} {\left[ 3 
   \right]_{q^{ 4}}} } {\left[ \tfrac{9m+4}{5}  \right]_{q^{ 
  10}}} {\left[ \tfrac{9 m+5}2 \right]_{q^4}}\\
\times {\left[ 3m+2 
  \right]_{q^{ 6}}} {\left[ \tfrac{9m+7}{7}  \right]_{q^{ 14}}} 
 {\left[ m+1 \right]_{q^{ 18}}}
,\end{multline*} 
which, by Corollary~\ref{cor:A}, is a polynomial in $q$ with 
non-negative integer coefficients.
If $m\equiv
50
~(\text{mod }140),$ then we have \begin{multline*}\Cat^m(E_7;q)=
{\left[ 9 m+1 \right]_{q^{ 2}}} {\left[ 3 m+1 \right]_{q^{ 
  6}}} {\left[ \tfrac{9m+4}2 \right]_{q^4}}\\
\times {\left[ 
   \tfrac{9m+5}{35}  \right]_{q^{ 70}}} \frac {{ {\left[ 35
    \right]_{q^{ 2}}} }} { {\left[ 5 \right]_{q^{ 2}}} 
  {\left[ 7 \right]_{q^{ 2}}} } {\left[ 
   \tfrac{3m+2}{4}  \right]_{q^{ 24}}} \frac {{ {\left[ 6 
   \right]_{q^{ 4}}} }} { {\left[ 2 \right]_{q^{ 4}}} 
  {\left[ 3 \right]_{q^{ 4}}} } {\left[ 9 m+7 
  \right]_{q^2}} {\left[ m+1 \right]_{q^{ 18}}}
,\end{multline*} 
which, by Corollary~\ref{cor:A}, is a polynomial in $q$ with 
non-negative integer coefficients.
If $m\equiv
51
~(\text{mod }140),$ then we have \begin{multline*}\Cat^m(E_7;q)=
{\left[ \tfrac{9m+1}{5}  \right]_{q^{ 10}}} {\left[ 
   \tfrac{3m+1}{14}  \right]_{q^{ 84}}} \frac {{ {\left[ 42
    \right]_{q^{ 2}}} }} { {\left[ 6 \right]_{q^{ 2}}} 
  {\left[ 7 \right]_{q^{ 2}}} } {\left[ 9m+4 
  \right]_{q^{ 2}}}\\
\times {\left[ \tfrac{9m+5}{4}  \right]_{q^{ 8}}} 
 {\left[ 3m+2 \right]_{q^{ 6}}} {\left[ 9 m+7 
  \right]_{q^2}} {\left[ m+1 \right]_{q^{ 18}}}
,\end{multline*} 
which, by Corollary~\ref{cor:A}, is a polynomial in $q$ with 
non-negative integer coefficients.
If $m\equiv
52
~(\text{mod }140),$ then we have $$\Cat^m(E_7;q)=
{\left[ \tfrac{9m+1}{7}  \right]_{q^{ 14}}} {\left[ 3 m+1 
  \right]_{q^{ 6}}} {\left[ \tfrac{9m+4}{4}  \right]_{q^{ 8}}} 
 {\left[ 9m+5 \right]_{q^{ 2}}} {\left[ \tfrac{3m+2}{2}  
  \right]_{q^{ 12}}} {\left[ \tfrac{9m+7}{5}  \right]_{q^{ 10}}} 
 {\left[ m+1 \right]_{q^{ 18}}}
,$$ 
which is manifestly a polynomial in $q$ with 
non-negative integer coefficients.  
If $m\equiv
53
~(\text{mod }140),$ then we have \begin{multline*}\Cat^m(E_7;q)=
{\left[ 9 m+1 \right]_{q^{ 2}}} {\left[ \tfrac{3m+1}{10}  
   \right]_{q^{ 60}}} \frac {{ {\left[ 30 \right]_{q^{ 2}}} 
  }} { {\left[ 5 \right]_{q^{ 2}}} {\left[ 6 
   \right]_{q^{ 2}}} } {\left[ 9m+4 \right]_{q^2}}\\
\times 
 {\left[ 9 m+5 \right]_{q^2}} {\left[ 
   \tfrac{3m+2}{7}  \right]_{q^{ 42}}} \frac {{ {\left[ 21 
   \right]_{q^{ 2}}} }} { {\left[ 3 \right]_{q^{ 2}}} 
  {\left[ 7 \right]_{q^{ 2}}} } {\left[ \tfrac{9m+7}{4} 
   \right]_{q^{ 8}}} {\left[ m+1 \right]_{q^{ 18}}}
,\end{multline*} 
which, by Corollary~\ref{cor:A}, is a polynomial in $q$ with 
non-negative integer coefficients.
If $m\equiv
54
~(\text{mod }140),$ then we have \begin{multline*}\Cat^m(E_7;q)=
{\left[ 9 m+1 \right]_{q^{ 2}}} {\left[ 3 m+1 \right]_{q^{ 
  6}}} {\left[ \tfrac{9m+4}{70}  \right]_{q^{ 140}}} 
  \frac {{ {\left[ 70 \right]_{q^{ 2}}} }} { 
  {\left[ 2 \right]_{q^{ 2}}}
  {\left[ 5 \right]_{q^{ 2}}} {\left[ 7 \right]_{q^{ 2}}} 
  }
 {\left[ 9 m+5 \right]_{q^2}}\\
\times {\left[ 
   \tfrac{3m+2}{4}  \right]_{q^{ 24}}} \frac {{ {\left[ 6 
   \right]_{q^{ 4}}} }} { {\left[ 2 \right]_{q^{ 4}}} 
  {\left[ 3 \right]_{q^{ 4}}} } {\left[ 9 m+7 
  \right]_{q^2}} {\left[ m+1 \right]_{q^{ 18}}}
.\end{multline*} 
If one decomposes $[9m+7]_{q^2}$ as 
$[\frac {9m} {2}+4]_{q^4}+q^2[\frac {9m} {2}+3]_{q^4}$,
then one sees that, by Corollary~\ref{cor:A} and Lemma~\ref{lem:J},
this is a polynomial in $q$ with 
non-negative integer coefficients.
If $m\equiv
55
~(\text{mod }140),$ then we have \begin{multline*}\Cat^m(E_7;q)=
{\left[ \tfrac{9m+1}{4}  \right]_{q^{ 8}}} {\left[ 
  \tfrac{3m+1}{2}  \right]_{q^{ 12}}} {\left[ 9m+4 \right]_{q^{ 
  2}}} {\left[ \tfrac{9m+5}{5}  \right]_{q^{ 10}}} {\left[ 
  3m+2 \right]_{q^{ 6}}} \\
\times{\left[ 9 m+7 \right]_{q^2}} 
 {\left[ \tfrac{m+1}{7} \right]_{q^{ 126}}} \frac {{ 
  {\left[ 63 \right]_{q^{ 2}}} }} { {\left[ 7 
   \right]_{q^{ 2}}} {\left[ 9 \right]_{q^{ 2}}} }
,\end{multline*} 
which, by Corollary~\ref{cor:A}, is a polynomial in $q$ with 
non-negative integer coefficients.
If $m\equiv
56
~(\text{mod }140),$ then we have $$\Cat^m(E_7;q)=
{\left[ \tfrac{9m+1}{5}  \right]_{q^{ 10}}} {\left[ 3 m+1 
  \right]_{q^{ 6}}} {\left[ \tfrac{9m+4}{4}  \right]_{q^{ 8}}} 
 {\left[ 9m+5 \right]_{q^{ 2}}} {\left[ \tfrac{3m+2}{2}  
  \right]_{q^{ 12}}} {\left[ \tfrac{9m+7}{7}  \right]_{q^{ 14}}} 
 {\left[ m+1 \right]_{q^{ 18}}}
,$$ 
which is manifestly a polynomial in $q$ with 
non-negative integer coefficients.  
If $m\equiv
57
~(\text{mod }140),$ then we have \begin{multline*}\Cat^m(E_7;q)=
{\left[ \tfrac{9 m+1}2 \right]_{q^{ 4}}} {\left[ \tfrac{3m+1}{4}  
   \right]_{q^{ 24}}} \frac {{ {\left[ 6 \right]_{q^{ 4}}} 
  }} { {\left[ 2 \right]_{q^{ 4}}} {\left[ 3 
   \right]_{q^{ 4}}} } \\
\times{\left[ 9m+4 \right]_{q^2}} 
 {\left[ \tfrac{9m+5}{7}  \right]_{q^{ 14}}} {\left[ 3m+2 
  \right]_{q^{ 6}}} {\left[ \tfrac{9m+7}{5}  \right]_{q^{ 10}}} 
 {\left[ m+1 \right]_{q^{ 18}}}
,\end{multline*} 
which, by Corollary~\ref{cor:A}, is a polynomial in $q$ with 
non-negative integer coefficients.
If $m\equiv
58
~(\text{mod }140),$ then we have \begin{multline*}\Cat^m(E_7;q)=
{\left[ 9 m+1 \right]_{q^{ 2}}} {\left[ \tfrac{3 m+1}{35}  
   \right]_{q^{ 210}}} \frac {{ {\left[ 105 \right]_{q^{ 2}}} 
  }} { {\left[ 3 \right]_{q^{ 2}}}{\left[ 5 \right]_{q^{ 2}}} {\left[ 7 
   \right]_{q^{ 2}}} } {\left[ \tfrac{9m+4}2 \right]_{q^4}} 
 {\left[ 9 m+5 \right]_{q^2}} \\
\times{\left[ 
   \tfrac{3m+2}{4}  \right]_{q^{ 24}}} \frac {{ {\left[ 6 
   \right]_{q^{ 4}}} }} { {\left[ 2 \right]_{q^{ 4}}} 
  {\left[ 3 \right]_{q^{ 4}}} } {\left[ 9 m+7 
  \right]_{q^2}} {\left[ m+1 \right]_{q^{ 18}}} 
,\end{multline*} 
which, by Corollary~\ref{cor:A} and Lemma~\ref{lem:I}, 
is a polynomial in $q$ with 
non-negative integer coefficients.
If $m\equiv
59
~(\text{mod }140),$ then we have $$\Cat^m(E_7;q)=
{\left[ \tfrac{9m+1}{7}  \right]_{q^{ 14}}} {\left[ 
  \tfrac{3m+1}{2}  \right]_{q^{ 12}}} {\left[ \tfrac{9m+4}{5}  
  \right]_{q^{ 10}}} {\left[ \tfrac{9m+5}{4}  \right]_{q^{ 8}}} 
 {\left[ 3m+2 \right]_{q^{ 6}}} {\left[ 9m+7 \right]_{q^{ 
  2}}} {\left[ m+1 \right]_{q^{ 18}}}
,$$ 
which is manifestly a polynomial in $q$ with 
non-negative integer coefficients.  
If $m\equiv
60
~(\text{mod }140),$ then we have \begin{multline*}\Cat^m(E_7;q)=
{\left[ 9 m+1 \right]_{q^{ 2}}} {\left[ 3 m+1 \right]_{q^{ 
  6}}} {\left[ \tfrac{9m+4}{4}  \right]_{q^{ 8}}} {\left[ 
  \tfrac{9m+5}{5}  \right]_{q^{ 10}}} \\
\times{\left[ \tfrac{3m+2}{14} 
    \right]_{q^{ 84}}} \frac {{ {\left[ 42 \right]_{q^{ 2}}} 
  }} { {\left[ 6 \right]_{q^{ 2}}} {\left[ 7 
   \right]_{q^{ 2}}} } {\left[ 9 m+7 \right]_{q^2}} 
 {\left[ m+1 \right]_{q^{ 18}}}
,\end{multline*} 
which, by Corollary~\ref{cor:A}, is a polynomial in $q$ with 
non-negative integer coefficients.
If $m\equiv
61
~(\text{mod }140),$ then we have $$\Cat^m(E_7;q)=
{\left[ \tfrac{9m+1}{5}  \right]_{q^{ 10}}} {\left[ 
  \tfrac{3m+1}{2}  \right]_{q^{ 12}}} {\left[ \tfrac{9m+4}{7}  
  \right]_{q^{ 14}}} {\left[ 9m+5 \right]_{q^{ 2}}} {\left[ 
  3m+2 \right]_{q^{ 6}}} {\left[ \tfrac{9m+7}{4}  \right]_{q^{ 
  8}}} {\left[ m+1 \right]_{q^{ 18}}}
,$$ 
which is manifestly a polynomial in $q$ with 
non-negative integer coefficients.  
If $m\equiv
62
~(\text{mod }140),$ then we have \begin{multline*}\Cat^m(E_7;q)=
{\left[ 9 m+1 \right]_{q^{ 2}}} {\left[ 3 m+1 \right]_{q^{ 
  6}}}\\
\times {\left[ \tfrac{9m+4}2 \right]_{q^4}} {\left[ 9 m+5 
  \right]_{q^2}} {\left[ \tfrac{3m+2}{4}  \right]_{q^{ 24}}} 
  \frac {{ {\left[ 6 \right]_{q^{ 4}}} }} { 
  {\left[ 2 \right]_{q^{ 4}}} {\left[ 3 \right]_{q^{ 4}}} 
  } {\left[ \tfrac{9m+7}{5}  \right]_{q^{ 10}}} 
 {\left[ \tfrac{m+1}{7} \right]_{q^{ 126}}} \frac {{ 
  {\left[ 63 \right]_{q^{ 2}}} }} { {\left[ 7 
   \right]_{q^{ 2}}} {\left[ 9 \right]_{q^{ 2}}} }
,\end{multline*} 
which, by Corollary~\ref{cor:A}, is a polynomial in $q$ with 
non-negative integer coefficients.
If $m\equiv
63
~(\text{mod }140),$ then we have \begin{multline*}\Cat^m(E_7;q)=
{\left[ \tfrac{9m+1}{4}  \right]_{q^{ 8}}} {\left[ 
   \tfrac{3m+1}{10}  \right]_{q^{ 60}}} \frac {{ {\left[ 30
    \right]_{q^{ 2}}} }} { {\left[ 5 \right]_{q^{ 2}}} 
  {\left[ 6 \right]_{q^{ 2}}} } {\left[ 9m+4 
  \right]_{q^{ 2}}}\\
\times {\left[ 9 m+5 \right]_{q^2}} {\left[ 
  3m+2 \right]_{q^{ 6}}} {\left[ \tfrac{9m+7}{7}  \right]_{q^{ 
  14}}} {\left[ m+1 \right]_{q^{ 18}}}
,\end{multline*} 
which, by Corollary~\ref{cor:A}, is a polynomial in $q$ with 
non-negative integer coefficients.
If $m\equiv
64
~(\text{mod }140),$ then we have \begin{multline*}\Cat^m(E_7;q)=
{\left[ 9 m+1 \right]_{q^{ 2}}} {\left[ 3 m+1 \right]_{q^{ 
  6}}} {\left[ \tfrac{9m+4}{20}  \right]_{q^{ 40}}} 
  \frac {{ {\left[ 20 \right]_{q^{ 2}}} }} { 
  {\left[ 4 \right]_{q^{ 2}}} {\left[ 5 \right]_{q^{ 2}}} 
  } \\
\times{\left[ \tfrac{9m+5}{7}  \right]_{q^{ 14}}} {\left[ 
  \tfrac{3m+2}{2}  \right]_{q^{ 12}}} {\left[ 9 m+7 
  \right]_{q^2}} {\left[ m+1 \right]_{q^{ 18}}}
,\end{multline*} 
which, by Corollary~\ref{cor:A}, is a polynomial in $q$ with 
non-negative integer coefficients.
If $m\equiv
65
~(\text{mod }140),$ then we have \begin{multline*}\Cat^m(E_7;q)=
{\left[ 9 m+1 \right]_{q^{ 2}}} {\left[ \tfrac{3m+1}{14}  
   \right]_{q^{ 84}}} \frac {{ {\left[ 42 \right]_{q^{ 2}}} 
  }} { {\left[ 6 \right]_{q^{ 2}}} {\left[ 7 
   \right]_{q^{ 2}}} } {\left[ 9m+4 \right]_{q^2}}\\
\times 
 {\left[ \tfrac{9m+5}{5}  \right]_{q^{ 10}}} {\left[ 3m+2 
  \right]_{q^{ 6}}} {\left[ \tfrac{9m+7}{4}  \right]_{q^{ 8}}} 
 {\left[ m+1 \right]_{q^{ 18}}}
,\end{multline*} 
which, by Corollary~\ref{cor:A}, is a polynomial in $q$ with 
non-negative integer coefficients.
If $m\equiv
66
~(\text{mod }140),$ then we have \begin{multline*}\Cat^m(E_7;q)=
{\left[ \tfrac{9m+1}{35}  \right]_{q^{ 70}}} \frac {{ 
  {\left[ 35 \right]_{q^{ 2}}} }} { {\left[ 5 
   \right]_{q^{ 2}}} {\left[ 7 \right]_{q^{ 2}}} } {\left[ 
  3 m+1 \right]_{q^{ 6}}} {\left[ \tfrac{9m+4}2 \right]_{q^{ 4}}} 
 {\left[ 9 m+5 \right]_{q^2}}\\
\times {\left[ 
   \tfrac{3m+2}{4}  \right]_{q^{ 24}}} \frac {{ {\left[ 6 
   \right]_{q^{ 4}}} }} { {\left[ 2 \right]_{q^{ 4}}} 
  {\left[ 3 \right]_{q^{ 4}}} } {\left[ 9 m+7 
  \right]_{q^2}} {\left[ m+1 \right]_{q^{ 18}}}
,\end{multline*} 
which, by Corollary~\ref{cor:A}, is a polynomial in $q$ with 
non-negative integer coefficients.
If $m\equiv
67
~(\text{mod }140),$ then we have \begin{multline*}\Cat^m(E_7;q)=
{\left[ \tfrac{9m+1}{4}  \right]_{q^{ 8}}} {\left[ 
   \tfrac{3m+1}{2}  \right]_{q^{ 12}}} 
{\left[ 9m+4 
  \right]_{q^{ 2}}} \\
\times{\left[ 9 m+5 \right]_{q^2}} {\left[ 
  \tfrac{3m+2}{7}  \right]_{q^{ 42}}}
\frac {[21]_{q^2}} {[3]_{q^2}\left[7\right]_{q^2}}
 {\left[ \tfrac{9m+7}{5}  
  \right]_{q^{ 10}}} {\left[ m+1 \right]_{q^{ 18}}}
,\end{multline*} 
which, by Corollary~\ref{cor:A}, is a polynomial in $q$ with 
non-negative integer coefficients.
If $m\equiv
68
~(\text{mod }140),$ then we have \begin{multline*}\Cat^m(E_7;q)=
{\left[ 9 m+1 \right]_{q^{ 2}}} {\left[ \tfrac{3m+1}{5}  
   \right]_{q^{ 30}}} \frac {{ {\left[ 15 \right]_{q^{ 2}}} 
  }} { {\left[ 3 \right]_{q^{ 2}}} {\left[ 5 
   \right]_{q^{ 2}}} } \\
\times{\left[ \tfrac{9m+4}{28}  
   \right]_{q^{ 56}}} \frac {{ {\left[ 28 \right]_{q^{ 2}}} 
  }} { {\left[ 4 \right]_{q^{ 2}}} {\left[ 7 
   \right]_{q^{ 2}}} } {\left[ 9 m+5 \right]_{q^2}} 
 {\left[ \tfrac{3m+2}{2}  \right]_{q^{ 12}}} {\left[ 
  9 m+7 \right]_{q^2}} {\left[ m+1 \right]_{q^{ 18}}}
,\end{multline*} 
which, by Corollary~\ref{cor:A}, is a polynomial in $q$ with 
non-negative integer coefficients.
If $m\equiv
69
~(\text{mod }140),$ then we have \begin{multline*}\Cat^m(E_7;q)=
{\left[ 9 m+1 \right]_{q^{ 2}}} {\left[ \tfrac{3m+1}{2}  
  \right]_{q^{ 12}}} {\left[ \tfrac{9m+4}{5}  \right]_{q^{ 10}}} \\
\times
 {\left[ 9 m+5 \right]_{q^2}} {\left[ 3m+2 
  \right]_{q^{ 6}}} {\left[ \tfrac{9m+7}{4}  \right]_{q^{ 8}}} 
 {\left[ \tfrac{m+1}{7} \right]_{q^{ 126}}} \frac {{ 
  {\left[ 63 \right]_{q^{ 2}}} }} { {\left[ 7 
   \right]_{q^{ 2}}} {\left[ 9 \right]_{q^{ 2}}} }
,\end{multline*} 
which, by Corollary~\ref{cor:A}, is a polynomial in $q$ with 
non-negative integer coefficients.
If $m\equiv
70
~(\text{mod }140),$ then we have \begin{multline*}\Cat^m(E_7;q)=
{\left[ 9 m+1 \right]_{q^{ 2}}} {\left[ 3 m+1 \right]_{q^{ 
  6}}} {\left[ \tfrac{9m+4}2 \right]_{q^4}} {\left[ 
  \tfrac{9m+5}{5}  \right]_{q^{ 10}}} {\left[ \tfrac{3m+2}{4}  
   \right]_{q^{ 24}}} \frac {{ {\left[ 6 \right]_{q^{ 4}}} 
  }} { {\left[ 2 \right]_{q^{ 4}}} {\left[ 3 
   \right]_{q^{ 4}}} } \\
\times{\left[ \tfrac{9m+7}{7}  \right]_{q^{ 
  14}}} {\left[ m+1 \right]_{q^{ 18}}}
,\end{multline*} 
which, by Corollary~\ref{cor:A}, is a polynomial in $q$ with 
non-negative integer coefficients.
If $m\equiv
71
~(\text{mod }140),$ then we have \begin{multline*}\Cat^m(E_7;q)=
{\left[ \tfrac{9m+1}{20}  \right]_{q^{ 40}}} \frac {{ 
  {\left[ 20 \right]_{q^{ 2}}} }} { {\left[ 4 
   \right]_{q^{ 2}}} {\left[ 5 \right]_{q^{ 2}}} } {\left[ 
  \tfrac{3m+1}{2}  \right]_{q^{ 12}}} {\left[ 9m+4 \right]_{q^{ 
  2}}}\\
\times {\left[ \tfrac{9m+5}{7}  \right]_{q^{ 14}}} {\left[ 
  3m+2 \right]_{q^{ 6}}} {\left[ 9 m+7 \right]_{q^2}} 
 {\left[ m+1 \right]_{q^{ 18}}}
,\end{multline*} 
which, by Corollary~\ref{cor:A}, is a polynomial in $q$ with 
non-negative integer coefficients.
If $m\equiv
72
~(\text{mod }140),$ then we have \begin{multline*}\Cat^m(E_7;q)=
{\left[ 9 m+1 \right]_{q^{ 2}}} {\left[ \tfrac{3m+1}{7}  
   \right]_{q^{ 42}}} \frac {{ {\left[ 21 \right]_{q^{ 2}}} 
  }} { {\left[ 3 \right]_{q^{ 2}}} {\left[ 7 
   \right]_{q^{ 2}}} } {\left[ \tfrac{9m+4}{4}  \right]_{q^{ 
  8}}}\\
\times {\left[ 9 m+5 \right]_{q^2}} {\left[ 
  \tfrac{3m+2}{2}  \right]_{q^{ 12}}} {\left[ \tfrac{9m+7}{5}  
  \right]_{q^{ 10}}} {\left[ m+1 \right]_{q^{ 18}}}
,\end{multline*} 
which, by Corollary~\ref{cor:A}, is a polynomial in $q$ with 
non-negative integer coefficients.
If $m\equiv
73
~(\text{mod }140),$ then we have \begin{multline*}\Cat^m(E_7;q)=
{\left[ \tfrac{9m+1}{7}  \right]_{q^{ 14}}} {\left[ 
   \tfrac{3m+1}{10}  \right]_{q^{ 60}}} \frac {{ {\left[ 30
    \right]_{q^{ 2}}} }} { {\left[ 5 \right]_{q^{ 2}}} 
  {\left[ 6 \right]_{q^{ 2}}} } {\left[ 9m+4 
  \right]_{q^{ 2}}}\\
\times {\left[ 9 m+5 \right]_{q^2}} {\left[ 
  3m+2 \right]_{q^{ 6}}} {\left[ \tfrac{9m+7}{4}  \right]_{q^{ 
  8}}} {\left[ m+1 \right]_{q^{ 18}}}
,\end{multline*} 
which, by Corollary~\ref{cor:A}, is a polynomial in $q$ with 
non-negative integer coefficients.
If $m\equiv
74
~(\text{mod }140),$ then we have \begin{multline*}\Cat^m(E_7;q)=
{\left[ {9 m+1} \right]_{q^{ 2}}} 
{\left[ {3m+1} 
   \right]_{q^{ 6}}} 
{\left[ \tfrac{9m+4}{10} \right]_{q^{ 20}}}
\frac {[10]_{q^2}} {[2]_{q^2}\left[5\right]_{q^2}}\\
\times 
 {\left[ 9m+5 \right]_{q^{ 2}}} {\left[ \tfrac{3m+2}{28} 
   \right]_{q^{ 168}}} 
\frac {[84]_{q^2}\left[2\right]_{q^2}} 
{[4]_{q^2}\left[6\right]_{q^2}\left[7\right]_{q^2}}
{\left[ 9m+7 \right]_{q^{ 2}}} 
 {\left[ m+1 \right]_{q^{ 18}}}
.\end{multline*} 
If one decomposes $[9m+7]_{q^2}$ as 
$[\frac {9m} {2}+4]_{q^4}+q^2[\frac {9m} {2}+3]_{q^4}$,
then one sees that, by Corollary~\ref{cor:A} and Lemma~\ref{lem:H},
this is a polynomial in $q$ with 
non-negative integer coefficients.
If $m\equiv
75
~(\text{mod }140),$ then we have $$\Cat^m(E_7;q)=
{\left[ \tfrac{9m+1}{4}  \right]_{q^{ 8}}} {\left[ 
  \tfrac{3m+1}{2}  \right]_{q^{ 12}}} {\left[ \tfrac{9m+4}{7}  
  \right]_{q^{ 14}}} {\left[ \tfrac{9m+5}{5}  \right]_{q^{ 10}}} 
 {\left[ 3m+2 \right]_{q^{ 6}}} {\left[ 9m+7 \right]_{q^{ 
  2}}} {\left[ m+1 \right]_{q^{ 18}}}
,$$ 
which is manifestly a polynomial in $q$ with 
non-negative integer coefficients.  
If $m\equiv
76
~(\text{mod }140),$ then we have \begin{multline*}\Cat^m(E_7;q)=
{\left[ \tfrac{9m+1}{5}  \right]_{q^{ 10}}} {\left[ 3 m+1 
  \right]_{q^{ 6}}} {\left[ \tfrac{9m+4}{4}  \right]_{q^{ 8}}} 
 {\left[ 9m+5 \right]_{q^{ 2}}} \\
\times{\left[ \tfrac{3m+2}{2}  
  \right]_{q^{ 12}}} {\left[ 9 m+7 \right]_{q^2}} 
 {\left[ \tfrac{m+1}{7} \right]_{q^{ 126}}} \frac {{ 
  {\left[ 63 \right]_{q^{ 2}}} }} { {\left[ 7 
   \right]_{q^{ 2}}} {\left[ 9 \right]_{q^{ 2}}} }
,\end{multline*} 
which, by Corollary~\ref{cor:A}, is a polynomial in $q$ with 
non-negative integer coefficients.
If $m\equiv
77
~(\text{mod }140),$ then we have \begin{multline*}\Cat^m(E_7;q)=
{\left[ 9 m+1 \right]_{q^{ 2}}} {\left[ \tfrac{3m+1}{4}  
   \right]_{q^{ 24}}} \frac {{ {\left[ 6 \right]_{q^{ 4}}} 
  }} { {\left[ 2 \right]_{q^{ 4}}} {\left[ 3 
   \right]_{q^{ 4}}} } {\left[ {9m+4} \right]_{q^2}} \\
\times
 {\left[ \tfrac{9 m+5}2 \right]_{q^4}} {\left[ 3m+2 
  \right]_{q^{ 6}}} {\left[ \tfrac{9m+7}{35}  \right]_{q^{ 70}}} 
  \frac {{ {\left[ 35 \right]_{q^{ 2}}} }} { 
  {\left[ 5 \right]_{q^{ 2}}} {\left[ 7 \right]_{q^{ 2}}} 
  } {\left[ m+1 \right]_{q^{ 18}}}
,\end{multline*} 
which, by Corollary~\ref{cor:A}, is a polynomial in $q$ with 
non-negative integer coefficients.
If $m\equiv
78
~(\text{mod }140),$ then we have \begin{multline*}\Cat^m(E_7;q)=
{\left[ 9 m+1 \right]_{q^{ 2}}} {\left[ \tfrac{3m+1}{5}  
   \right]_{q^{ 30}}} \frac {{ {\left[ 15 \right]_{q^{ 2}}} 
  }} { {\left[ 3 \right]_{q^{ 2}}} {\left[ 5 
   \right]_{q^{ 2}}} } \\
\times{\left[ \tfrac{9m+4}2 \right]_{q^4}} 
 {\left[ \tfrac{9m+5}{7}  \right]_{q^{ 14}}} {\left[ 
   \tfrac{3m+2}{4}  \right]_{q^{ 24}}} \frac {{ {\left[ 6 
   \right]_{q^{ 4}}} }} { {\left[ 2 \right]_{q^{ 4}}} 
  {\left[ 3 \right]_{q^{ 4}}} } {\left[ 9 m+7 
  \right]_{q^2}} {\left[ m+1 \right]_{q^{ 18}}}
,\end{multline*} 
which, by Corollary~\ref{cor:A}, is a polynomial in $q$ with 
non-negative integer coefficients.
If $m\equiv
79
~(\text{mod }140),$ then we have \begin{multline*}\Cat^m(E_7;q)=
{\left[ \tfrac{9m+1}{4}  \right]_{q^{ 8}}} {\left[ 
   \tfrac{3m+1}{14}  \right]_{q^{ 84}}} \frac {{ {\left[ 42
    \right]_{q^{ 2}}} }} { {\left[ 6 \right]_{q^{ 2}}} 
  {\left[ 7 \right]_{q^{ 2}}} } {\left[ \tfrac{9m+4}{5} 
   \right]_{q^{ 10}}} {\left[ 9m+5 \right]_{q^{ 2}}}\\
\times {\left[ 
  3m+2 \right]_{q^{ 6}}} {\left[ 9 m+7 \right]_{q^2}} 
 {\left[ m+1 \right]_{q^{ 18}}}
,\end{multline*} 
which, by Corollary~\ref{cor:A}, is a polynomial in $q$ with 
non-negative integer coefficients.
If $m\equiv
80
~(\text{mod }140),$ then we have $$\Cat^m(E_7;q)=
{\left[ \tfrac{9m+1}{7}  \right]_{q^{ 14}}} {\left[ 3 m+1 
  \right]_{q^{ 6}}} {\left[ \tfrac{9m+4}{4}  \right]_{q^{ 8}}} 
 {\left[ \tfrac{9m+5}{5}  \right]_{q^{ 10}}} {\left[ 
  \tfrac{3m+2}{2}  \right]_{q^{ 12}}} {\left[ 9m+7 \right]_{q^{ 
  2}}} {\left[ m+1 \right]_{q^{ 18}}}
,$$ 
which is manifestly a polynomial in $q$ with 
non-negative integer coefficients.  
If $m\equiv
81
~(\text{mod }140),$ then we have \begin{multline*}\Cat^m(E_7;q)=
{\left[ \tfrac{9m+1}{5}  \right]_{q^{ 10}}} {\left[ 
   \tfrac{3m+1}{2}  \right]_{q^{ 12}}} 
{\left[ 9m+4 
  \right]_{q^{ 2}}} \\
\times{\left[ 9 m+5 \right]_{q^2}} {\left[ 
  \tfrac{3m+2}{7}  \right]_{q^{ 42}}} 
\frac {[21]_{q^2}} {[3]_{q^2}\left[7\right]_{q^2}}
{\left[ \tfrac{9m+7}{4}  
  \right]_{q^{ 8}}} {\left[ m+1 \right]_{q^{ 18}}}
,\end{multline*} 
which, by Corollary~\ref{cor:A}, is a polynomial in $q$ with 
non-negative integer coefficients.
If $m\equiv
82
~(\text{mod }140),$ then we have \begin{multline*}\Cat^m(E_7;q)=
{\left[ 9 m+1 \right]_{q^{ 2}}} {\left[ 3 m+1 \right]_{q^{ 
  6}}} {\left[ \tfrac{9m+4}{14}  \right]_{q^{ 28}}}
\frac {[14]_{q^2}} {[2]_{q^2}\left[7\right]_{q^2}}\\
\times
 {\left[ 
  9 m+5 \right]_{q^2}} {\left[ \tfrac{3m+2}{4}  
   \right]_{q^{ 24}}} \frac {{ {\left[ 6 \right]_{q^{ 4}}} 
  }} { {\left[ 2 \right]_{q^{ 4}}} {\left[ 3 
   \right]_{q^{ 4}}} } {\left[ \tfrac{9m+7}{5}  \right]_{q^{ 
  10}}} {\left[ m+1 \right]_{q^{ 18}}}
.\end{multline*} 
If one decomposes $[9m+5]_{q^2}$ as 
$[\frac {9m} {2}+4]_{q^4}+q^2[\frac {9m} {2}+2]_{q^4}$,
then one sees that, by Corollary~\ref{cor:A},
this is a polynomial in $q$ with 
non-negative integer coefficients.
If $m\equiv
83
~(\text{mod }140),$ then we have \begin{multline*}\Cat^m(E_7;q)=
{\left[ \tfrac{9m+1}{4}  \right]_{q^{ 8}}} {\left[ 
   \tfrac{3m+1}{10}  \right]_{q^{ 60}}} \frac {{ {\left[ 30
    \right]_{q^{ 2}}} }} { {\left[ 5 \right]_{q^{ 2}}} 
  {\left[ 6 \right]_{q^{ 2}}} } {\left[ 9m+4 
  \right]_{q^{ 2}}} {\left[ 9 m+5 \right]_{q}}\\
\times {\left[ 
  3m+2 \right]_{q^{ 6}}} {\left[ 9 m+7 \right]_{q^2}} 
 {\left[ \tfrac{m+1}{7} \right]_{q^{ 126}}} \frac {{ 
  {\left[ 63 \right]_{q^{ 2}}} }} { {\left[ 7 
   \right]_{q^{ 2}}} {\left[ 9 \right]_{q^{ 2}}} }
,\end{multline*} 
which, by Corollary~\ref{cor:A}, is a polynomial in $q$ with 
non-negative integer coefficients.
If $m\equiv
84
~(\text{mod }140),$ then we have \begin{multline*}\Cat^m(E_7;q)=
{\left[ 9 m+1 \right]_{q^{ 2}}} {\left[ 3 m+1 \right]_{q^{ 
  6}}} {\left[ \tfrac{9m+4}{20}  \right]_{q^{ 40}}} 
  \frac {{ {\left[ 20 \right]_{q^{ 2}}} }} { 
  {\left[ 4 \right]_{q^{ 2}}} {\left[ 5 \right]_{q^{ 2}}} 
  }\\
\times {\left[ 9 m+5 \right]_{q^2}} {\left[ 
  \tfrac{3m+2}{2}  \right]_{q^{ 12}}} {\left[ \tfrac{9m+7}{7}  
  \right]_{q^{ 14}}} {\left[ m+1 \right]_{q^{ 18}}}
,\end{multline*} 
which, by Corollary~\ref{cor:A}, is a polynomial in $q$ with 
non-negative integer coefficients.
If $m\equiv
85
~(\text{mod }140),$ then we have \begin{multline*}\Cat^m(E_7;q)=
{\left[ 9 m+1 \right]_{q^{ 2}}} {\left[ \tfrac{3m+1}{2}  
  \right]_{q^{ 12}}} \\
\times{\left[ 9m+4 \right]_{q^2}} 
 {\left[ \tfrac{9m+5}{35}  \right]_{q^{ 70}}} \frac {{ 
  {\left[ 35 \right]_{q^{ 2}}} }} { {\left[ 5 
   \right]_{q^{ 2}}} {\left[ 7 \right]_{q^{ 2}}} } {\left[ 
  3m+2 \right]_{q^{ 6}}} {\left[ \tfrac{9m+7}{4}  \right]_{q^{ 
  8}}} {\left[ m+1 \right]_{q^{ 18}}}
,\end{multline*} 
which, by Corollary~\ref{cor:A}, is a polynomial in $q$ with 
non-negative integer coefficients.
If $m\equiv
86
~(\text{mod }140),$ then we have \begin{multline*}\Cat^m(E_7;q)=
{\left[ \tfrac{9m+1}{5}  \right]_{q^{ 10}}} {\left[ 
   \tfrac{3m+1}{7}  \right]_{q^{ 42}}} \frac {{ {\left[ 21 
   \right]_{q^{ 2}}} }} { {\left[ 3 \right]_{q^{ 2}}} 
  {\left[ 7 \right]_{q^{ 2}}} } {\left[ \tfrac{9m+4}2 
  \right]_{q^{ 4}}} {\left[ 9 m+5 \right]_{q^2}}\\
\times 
 {\left[ \tfrac{3m+2}{4}  \right]_{q^{ 24}}} \frac {{ 
  {\left[ 6 \right]_{q^{ 4}}} }} { {\left[ 2 
   \right]_{q^{ 4}}} {\left[ 3 \right]_{q^{ 4}}} } 
 {\left[ 9 m+7 \right]_{q^2}} {\left[ m+1 \right]_{q^{ 
  18}}}
,\end{multline*} 
which, by Corollary~\ref{cor:A}, is a polynomial in $q$ with 
non-negative integer coefficients.
If $m\equiv
87
~(\text{mod }140),$ then we have $$\Cat^m(E_7;q)=
{\left[ \tfrac{9m+1}{7}  \right]_{q^{ 14}}} {\left[ 
  \tfrac{3m+1}{2}  \right]_{q^{ 12}}} {\left[ 9m+4 \right]_{q^{ 
  2}}} {\left[ \tfrac{9m+5}{4}  \right]_{q^{ 8}}} {\left[ 3m+2
   \right]_{q^{ 6}}} {\left[ \tfrac{9m+7}{5}  \right]_{q^{ 10}}} 
 {\left[ m+1 \right]_{q^{ 18}}}
,$$ 
which is manifestly a polynomial in $q$ with 
non-negative integer coefficients.  
If $m\equiv
88
~(\text{mod }140),$ then we have \begin{multline*}\Cat^m(E_7;q)=
{\left[ 9 m+1 \right]_{q^{ 2}}} {\left[ \tfrac{3m+1}{5}  
   \right]_{q^{ 30}}} \frac {{ {\left[ 15 \right]_{q^{ 2}}} 
  }} { {\left[ 3 \right]_{q^{ 2}}} {\left[ 5 
   \right]_{q^{ 2}}} } {\left[ \tfrac{9m+4}{4}  \right]_{q^{ 
  8}}} {\left[ 9 m+5 \right]_{q^2}}\\
\times {\left[ 
   \tfrac{3m+2}{14}  \right]_{q^{ 84}}} \frac {{ {\left[ 42
    \right]_{q^{ 2}}} }} { {\left[ 6 \right]_{q^{ 2}}} 
  {\left[ 7 \right]_{q^{ 2}}} } {\left[ 9 m+7 
  \right]_{q^2}} {\left[ m+1 \right]_{q^{ 18}}}
,\end{multline*} 
which, by Corollary~\ref{cor:A}, is a polynomial in $q$ with 
non-negative integer coefficients.
If $m\equiv
89
~(\text{mod }140),$ then we have \begin{multline*}\Cat^m(E_7;q)=
{\left[ 9 m+1 \right]_{q^{ 2}}} {\left[ \tfrac{3m+1}{2}  
  \right]_{q^{ 12}}} {\left[ \tfrac{9m+4}{35}  \right]_{q^{ 70}}} 
  \frac {{ {\left[ 35 \right]_{q^{ 2}}} }} { 
  {\left[ 5 \right]_{q^{ 2}}} {\left[ 7 \right]_{q^{ 2}}} 
  }\\
\times {\left[ 9 m+5 \right]_{q^2}} {\left[ 3m+2 
  \right]_{q^{ 6}}} {\left[ \tfrac{9m+7}{4}  \right]_{q^{ 8}}} 
 {\left[ m+1 \right]_{q^{ 18}}}
,\end{multline*} 
which, by Corollary~\ref{cor:A}, is a polynomial in $q$ with 
non-negative integer coefficients.
If $m\equiv
90
~(\text{mod }140),$ then we have \begin{multline*}\Cat^m(E_7;q)=
{\left[ 9 m+1 \right]_{q^{ 2}}} {\left[ 3 m+1 \right]_{q^{ 
  6}}} {\left[ \tfrac{9m+4}2 \right]_{q^4}} {\left[ 
  \tfrac{9m+5}{5}  \right]_{q^{ 10}}} {\left[ \tfrac{3m+2}{4}  
   \right]_{q^{ 24}}} \frac {{ {\left[ 6 \right]_{q^{ 4}}} 
  }} { {\left[ 2 \right]_{q^{ 4}}} {\left[ 3 
   \right]_{q^{ 4}}} } \\
\times{\left[ 9 m+7 \right]_{q^2}} 
 {\left[ \tfrac{m+1}{7} \right]_{q^{ 126}}} \frac {{ 
  {\left[ 63 \right]_{q^{ 2}}} }} { {\left[ 7 
   \right]_{q^{ 2}}} {\left[ 9 \right]_{q^{ 2}}} }
,\end{multline*} 
which, by Corollary~\ref{cor:A}, is a polynomial in $q$ with 
non-negative integer coefficients.
If $m\equiv
91
~(\text{mod }140),$ then we have $$\Cat^m(E_7;q)=
{\left[ \tfrac{9m+1}{5}  \right]_{q^{ 10}}} {\left[ 
  \tfrac{3m+1}{2}  \right]_{q^{ 12}}} {\left[ 9m+4 \right]_{q^{ 
  2}}} {\left[ \tfrac{9m+5}{4}  \right]_{q^{ 8}}} {\left[ 3m+2
   \right]_{q^{ 6}}} {\left[ \tfrac{9m+7}{7}  \right]_{q^{ 14}}} 
 {\left[ m+1 \right]_{q^{ 18}}}
,$$ 
which is manifestly a polynomial in $q$ with 
non-negative integer coefficients.  
If $m\equiv
92
~(\text{mod }140),$ then we have $$\Cat^m(E_7;q)=
{\left[ 9 m+1 \right]_{q^{ 2}}} {\left[ 3 m+1 \right]_{q^{ 
  6}}} {\left[ \tfrac{9m+4}{4}  \right]_{q^{ 8}}} {\left[ 
  \tfrac{9m+5}{7}  \right]_{q^{ 14}}} {\left[ \tfrac{3m+2}{2}  
  \right]_{q^{ 12}}} {\left[ \tfrac{9m+7}{5}  \right]_{q^{ 10}}} 
 {\left[ m+1 \right]_{q^{ 18}}}
,$$ 
which is manifestly a polynomial in $q$ with 
non-negative integer coefficients.  
If $m\equiv
93
~(\text{mod }140),$ then we have \begin{multline*}\Cat^m(E_7;q)=
{\left[ 9 m+1 \right]_{q^{ 2}}} {\left[ \tfrac{3 m+1}{35}  
   \right]_{q^{ 210}}} \frac {{ {\left[ 105 \right]_{q^{ 2}}} 
  }} { {\left[ 3 \right]_{q^{ 2}}} {\left[ 5 \right]_{q^{ 2}}} {\left[ 7 
   \right]_{q^{ 2}}} } {\left[ 9m+4 \right]_{q^2}}\\
\times 
 {\left[ 9 m+5 \right]_{q^2}} {\left[ 3m+2 
  \right]_{q^{ 6}}} {\left[ \tfrac{9m+7}{4}  \right]_{q^{ 8}}} 
 {\left[ \tfrac{m+1}{2} \right]_{q^{ 36}}} \frac {{ 
  {\left[ 6 \right]_{q^{ 6}}} }} { {\left[ 2 
   \right]_{q^{ 6}}} {\left[ 3 \right]_{q^{ 6}}} }
,\end{multline*} 
which, by Corollary~\ref{cor:A} and Lemma~\ref{lem:I}, 
is a polynomial in $q$ with 
non-negative integer coefficients.
If $m\equiv
94
~(\text{mod }140),$ then we have \begin{multline*}\Cat^m(E_7;q)=
{\left[ \tfrac{9m+1}{7}  \right]_{q^{ 14}}} {\left[ 3 m+1 
  \right]_{q^{ 6}}} {\left[ \tfrac{9m+4}{10}  \right]_{q^{ 20}}}
\frac {[10]_{q^2}} {[2]_{q^2}\left[5\right]_{q^2}} 
 {\left[ 9m+5 \right]_{q^{ 2}}} \\
\times{\left[ \tfrac{3m+2}{4}  
   \right]_{q^{ 24}}} \frac {{ {\left[ 6 \right]_{q^{ 4}}} 
  }} { {\left[ 2 \right]_{q^{ 4}}} {\left[ 3 
   \right]_{q^{ 4}}} } {\left[ 9 m+7 \right]_{q^2}} 
 {\left[ m+1 \right]_{q^{ 18}}}
.\end{multline*} 
If one decomposes $[9m+7]_{q^2}$ as 
$[\frac {9m} {2}+4]_{q^4}+q^2[\frac {9m} {2}+3]_{q^4}$,
then one sees that, by Corollary~\ref{cor:A},
this is a polynomial in $q$ with 
non-negative integer coefficients.
If $m\equiv
95
~(\text{mod }140),$ then we have \begin{multline*}\Cat^m(E_7;q)=
{\left[ \tfrac{9m+1}{4}  \right]_{q^{ 8}}} {\left[ 
   \tfrac{3m+1}{2}  \right]_{q^{ 12}}}
{\left[ 9m+4 
  \right]_{q^{ 2}}} {\left[ \tfrac{9m+5}{5}  \right]_{q^{ 10}}} \\
\times
 {\left[ \tfrac{3m+2}{7}  \right]_{q^{ 42}}} 
\frac {[21]_{q^2}} {[3]_{q^2}\left[7\right]_{q^2}}
{\left[ 
  9 m+7 \right]_{q^2}} {\left[ m+1 \right]_{q^{ 18}}}
,\end{multline*} 
which, by Corollary~\ref{cor:A}, is a polynomial in $q$ with 
non-negative integer coefficients.
If $m\equiv
96
~(\text{mod }140),$ then we have \begin{multline*}\Cat^m(E_7;q)=
{\left[ \tfrac{9m+1}{5}  \right]_{q^{ 10}}} {\left[ 3 m+1 
  \right]_{q^{ 6}}} {\left[ \tfrac{9m+4}{28}  \right]_{q^{ 56}}} 
  \frac {{ {\left[ 28 \right]_{q^{ 2}}} }} { 
  {\left[ 4 \right]_{q^{ 2}}} {\left[ 7 \right]_{q^{ 2}}} 
  } {\left[ 9m+5 \right]_{q^{ 2}}}\\
\times {\left[ 
  \tfrac{3m+2}{2}  \right]_{q^{ 12}}} {\left[ 9 m+7 
  \right]_{q^2}} {\left[ m+1 \right]_{q^{ 18}}}
,\end{multline*} 
which, by Corollary~\ref{cor:A}, is a polynomial in $q$ with 
non-negative integer coefficients.
If $m\equiv
97
~(\text{mod }140),$ then we have \begin{multline*}\Cat^m(E_7;q)=
{\left[ 9 m+1 \right]_{q^{ 2}}} {\left[ \tfrac{3m+1}{4}  
   \right]_{q^{ 24}}} \frac {{ {\left[ 6 \right]_{q^{ 4}}} 
  }} { {\left[ 2 \right]_{q^{ 4}}} {\left[ 3 
   \right]_{q^{ 4}}} } {\left[ 9m+4 \right]_{q^2}} 
 {\left[ \tfrac{9 m+5}2 \right]_{q^4}}\\
\times {\left[ 3m+2 
  \right]_{q^{ 6}}} {\left[ \tfrac{9m+7}{5}  \right]_{q^{ 10}}} 
 {\left[ \tfrac{m+1}{7} \right]_{q^{ 126}}} \frac {{ 
  {\left[ 63 \right]_{q^{ 2}}} }} { {\left[ 7 
   \right]_{q^{ 2}}} {\left[ 9 \right]_{q^{ 2}}} }
,\end{multline*} 
which, by Corollary~\ref{cor:A}, is a polynomial in $q$ with 
non-negative integer coefficients.
If $m\equiv
98
~(\text{mod }140),$ then we have \begin{multline*}\Cat^m(E_7;q)=
{\left[ 9 m+1 \right]_{q^{ 2}}} {\left[ \tfrac{3m+1}{5}  
   \right]_{q^{ 30}}} \frac {{ {\left[ 15 \right]_{q^{ 2}}} 
  }} { {\left[ 3 \right]_{q^{ 2}}} {\left[ 5 
   \right]_{q^{ 2}}} } \\
\times{\left[ \tfrac{9m+4}2 \right]_{q^4}} 
 {\left[ 9 m+5 \right]_{q^2}} {\left[ 
   \tfrac{3m+2}{4}  \right]_{q^{ 24}}} \frac {{ {\left[ 6 
   \right]_{q^{ 4}}} }} { {\left[ 2 \right]_{q^{ 4}}} 
  {\left[ 3 \right]_{q^{ 4}}} } {\left[ \tfrac{9m+7}{7} 
   \right]_{q^{ 14}}} {\left[ m+1 \right]_{q^{ 18}}}
,\end{multline*} 
which, by Corollary~\ref{cor:A}, is a polynomial in $q$ with 
non-negative integer coefficients.
If $m\equiv
99
~(\text{mod }140),$ then we have $$\Cat^m(E_7;q)=
{\left[ \tfrac{9m+1}{4}  \right]_{q^{ 8}}} {\left[ 
  \tfrac{3m+1}{2}  \right]_{q^{ 12}}} {\left[ \tfrac{9m+4}{5}  
  \right]_{q^{ 10}}} {\left[ \tfrac{9m+5}{7}  \right]_{q^{ 14}}} 
 {\left[ 3m+2 \right]_{q^{ 6}}} {\left[ 9m+7 \right]_{q^{ 
  2}}} {\left[ m+1 \right]_{q^{ 18}}}
,$$ 
which is manifestly a polynomial in $q$ with 
non-negative integer coefficients.  
If $m\equiv
100
~(\text{mod }140),$ then we have \begin{multline*}\Cat^m(E_7;q)=
{\left[ 9 m+1 \right]_{q^{ 2}}} {\left[ \tfrac{3m+1}{7}  
   \right]_{q^{ 42}}} \frac {{ {\left[ 21 \right]_{q^{ 2}}} 
  }} { {\left[ 3 \right]_{q^{ 2}}} {\left[ 7 
   \right]_{q^{ 2}}} } {\left[ \tfrac{9m+4}{4}  \right]_{q^{ 
  8}}} \\
\times{\left[ \tfrac{9m+5}{5}  \right]_{q^{ 10}}} {\left[ 
  \tfrac{3m+2}{2}  \right]_{q^{ 12}}} {\left[ 9 m+7 
  \right]_{q^2}} {\left[ m+1 \right]_{q^{ 18}}}
,\end{multline*} 
which, by Corollary~\ref{cor:A}, is a polynomial in $q$ with 
non-negative integer coefficients.
If $m\equiv
101
~(\text{mod }140),$ then we have \begin{multline*}\Cat^m(E_7;q)=
{\left[ \tfrac{9m+1}{35}  \right]_{q^{ 70}}} \frac {{ 
  {\left[ 35 \right]_{q^{ 2}}} }} { {\left[ 5 
   \right]_{q^{ 2}}} {\left[ 7 \right]_{q^{ 2}}} } {\left[ 
  \tfrac{3m+1}{2}  \right]_{q^{ 12}}} {\left[ 9m+4 \right]_{q^{ 
  2}}} \\
\times{\left[ 9 m+5 \right]_{q^2}} {\left[ 3m+2 
  \right]_{q^{ 6}}} {\left[ \tfrac{9m+7}{4}  \right]_{q^{ 8}}} 
 {\left[ m+1 \right]_{q^{ 18}}}
,\end{multline*} 
which, by Corollary~\ref{cor:A}, is a polynomial in $q$ with 
non-negative integer coefficients.
If $m\equiv
102
~(\text{mod }140),$ then we have \begin{multline*}\Cat^m(E_7;q)=
{\left[ {9 m+1} \right]_{q^{ 2}}} 
{\left[ {3m+1} 
   \right]_{q^{ 6}}} \\
\times
{\left[ \tfrac{9m+4}2 \right]_{q^{ 4}}} 
 {\left[ 9m+5 \right]_{q^{ 2}}} {\left[ \tfrac{3m+2}{28} 
   \right]_{q^{ 168}}} 
\frac {[84]_{q^2}\left[2\right]_{q^2}} 
{[4]_{q^2}\left[6\right]_{q^2}\left[7\right]_{q^2}}
{\left[ \tfrac{9m+7}5 \right]_{q^{ 10}}} 
 {\left[ m+1 \right]_{q^{ 18}}}
,\end{multline*} 
which, by Lemma~\ref{lem:H}, is a polynomial in $q$ with 
non-negative integer coefficients.
If $m\equiv
103
~(\text{mod }140),$ then we have \begin{multline*}\Cat^m(E_7;q)=
{\left[ \tfrac{9m+1}{4}  \right]_{q^{ 8}}} {\left[ 
   \tfrac{3m+1}{10}  \right]_{q^{ 60}}} \frac {{ {\left[ 30
    \right]_{q^{ 2}}} }} { {\left[ 5 \right]_{q^{ 2}}} 
  {\left[ 6 \right]_{q^{ 2}}} } {\left[ \tfrac{9m+4}{7} 
   \right]_{q^{ 14}}} {\left[ 9m+5 \right]_{q^{ 2}}}\\
\times {\left[ 
  3m+2 \right]_{q^{ 6}}} {\left[ 9 m+7 \right]_{q^2}} 
 {\left[ m+1 \right]_{q^{ 18}}}
,\end{multline*} 
which, by Corollary~\ref{cor:A}, is a polynomial in $q$ with 
non-negative integer coefficients.
If $m\equiv
104
~(\text{mod }140),$ then we have \begin{multline*}\Cat^m(E_7;q)=
{\left[ 9 m+1 \right]_{q^{ 2}}} {\left[ 3 m+1 \right]_{q^{ 
  6}}} {\left[ \tfrac{9m+4}{20}  \right]_{q^{ 40}}} 
  \frac {{ {\left[ 20 \right]_{q^{ 2}}} }} { 
  {\left[ 4 \right]_{q^{ 2}}} {\left[ 5 \right]_{q^{ 2}}} 
  } {\left[ 9 m+5 \right]_{q^2}}\\
\times {\left[ 
  \tfrac{3m+2}{2}  \right]_{q^{ 12}}} {\left[ 9 m+7 
  \right]_{q^2}} {\left[ \tfrac{m+1}{7} \right]_{q^{ 126}}} 
  \frac {{ {\left[ 63 \right]_{q^{ 2}}} }} { 
  {\left[ 7 \right]_{q^{ 2}}} {\left[ 9 \right]_{q^{ 2}}} 
  }
,\end{multline*} 
which, by Corollary~\ref{cor:A}, is a polynomial in $q$ with 
non-negative integer coefficients.
If $m\equiv
105
~(\text{mod }140),$ then we have \begin{multline*}\Cat^m(E_7;q)=
{\left[ 9 m+1 \right]_{q^{ 2}}} {\left[ \tfrac{3m+1}{4}  
   \right]_{q^{ 24}}} \frac {{ {\left[ 6 \right]_{q^{ 4}}} 
  }} { {\left[ 2 \right]_{q^{ 4}}} {\left[ 3 
   \right]_{q^{ 4}}} } \\
\times{\left[ 9m+4 \right]_{q^2}} 
 {\left[ \tfrac{9m+5}{10}  \right]_{q^{ 20}}}
\frac {[10]_{q^2}} {[2]_{q^2}\left[5\right]_{q^2}}
 {\left[ 3m+2 
  \right]_{q^{ 6}}} {\left[ \tfrac{9m+7}{7}  \right]_{q^{ 14}}} 
 {\left[ m+1 \right]_{q^{ 18}}}
.\end{multline*} 
If one decomposes $[9m+1]_{q^2}$ as 
$[\frac {9m+1} {2}]_{q^4}+q^2[\frac {9m+1} {2}]_{q^4}$,
then one sees that, by Corollary~\ref{cor:A},
this is a polynomial in $q$ with 
non-negative integer coefficients.
If $m\equiv
106
~(\text{mod }140),$ then we have \begin{multline*}\Cat^m(E_7;q)=
{\left[ \tfrac{9m+1}{5}  \right]_{q^{ 10}}} {\left[ 3 m+1 
  \right]_{q^{ 6}}} {\left[ \tfrac{9m+4}2 \right]_{q^{ 4}}} {\left[ 
  \tfrac{9m+5}{7}  \right]_{q^{ 14}}} {\left[ \tfrac{3m+2}{4}  
   \right]_{q^{ 24}}} \frac {{ {\left[ 6 \right]_{q^{ 4}}} 
  }} { {\left[ 2 \right]_{q^{ 4}}} {\left[ 3 
   \right]_{q^{ 4}}} } \\
\times{\left[ 9 m+7 \right]_{q^2}} 
 {\left[ m+1 \right]_{q^{ 18}}}
,\end{multline*} 
which, by Corollary~\ref{cor:A}, is a polynomial in $q$ with 
non-negative integer coefficients.
If $m\equiv
107
~(\text{mod }140),$ then we have \begin{multline*}\Cat^m(E_7;q)=
{\left[ \tfrac{9m+1}{4}  \right]_{q^{ 8}}} {\left[ 
   \tfrac{3m+1}{14}  \right]_{q^{ 84}}} \frac {{ {\left[ 42
    \right]_{q^{ 2}}} }} { {\left[ 6 \right]_{q^{ 2}}} 
  {\left[ 7 \right]_{q^{ 2}}} } {\left[ 9m+4 
  \right]_{q^{ 2}}} \\
\times{\left[ 9 m+5 \right]_{q^2}} {\left[ 
  3m+2 \right]_{q^{ 6}}} {\left[ \tfrac{9m+7}{5}  \right]_{q^{ 
  10}}} {\left[ m+1 \right]_{q^{ 18}}}
,\end{multline*} 
which, by Corollary~\ref{cor:A}, is a polynomial in $q$ with 
non-negative integer coefficients.
If $m\equiv
108
~(\text{mod }140),$ then we have \begin{multline*}\Cat^m(E_7;q)=
{\left[ \tfrac{9m+1}{7}  \right]_{q^{ 14}}} {\left[ 
   \tfrac{3m+1}{5}  \right]_{q^{ 30}}} \frac {{ {\left[ 15 
   \right]_{q^{ 2}}} }} { {\left[ 3 \right]_{q^{ 2}}} 
  {\left[ 5 \right]_{q^{ 2}}} } {\left[ \tfrac{9m+4}{4} 
   \right]_{q^{ 8}}} {\left[ 9m+5 \right]_{q^{ 2}}}\\
\times {\left[ 
  \tfrac{3m+2}{2}  \right]_{q^{ 12}}} {\left[ 9 m+7 
  \right]_{q^2}} {\left[ m+1 \right]_{q^{ 18}}}
,\end{multline*} 
which, by Corollary~\ref{cor:A}, is a polynomial in $q$ with 
non-negative integer coefficients.
If $m\equiv
109
~(\text{mod }140),$ then we have \begin{multline*}\Cat^m(E_7;q)=
{\left[ 9 m+1 \right]_{q^{ 2}}} {\left[ \tfrac{3m+1}{2}  
   \right]_{q^{ 12}}} 
 {\left[ \tfrac{9m+4}{5}  \right]_{q^{ 
  10}}}\\
\times {\left[ 9 m+5 \right]_{q^2}} {\left[ 
  \tfrac{3m+2}{7}  \right]_{q^{ 42}}}
\frac {[21]_{q^2}} {[3]_{q^2}\left[7\right]_{q^2}}
 {\left[ \tfrac{9m+7}{4}  
  \right]_{q^{ 8}}} {\left[ m+1 \right]_{q^{ 18}}}
,\end{multline*} 
which, by Corollary~\ref{cor:A}, is a polynomial in $q$ with 
non-negative integer coefficients.
If $m\equiv
110
~(\text{mod }140),$ then we have \begin{multline*}\Cat^m(E_7;q)=
{\left[ 9 m+1 \right]_{q^{ 2}}} {\left[ 3 m+1 \right]_{q^{ 
  6}}} {\left[ \tfrac{9m+4}{14}  \right]_{q^{ 28}}}
\frac {[14]_{q^2}} {[2]_{q^2}\left[7\right]_{q^2}}\\
\times
 {\left[ 
  \tfrac{9m+5}{5}  \right]_{q^{ 10}}} {\left[ \tfrac{3m+2}{4}  
   \right]_{q^{ 24}}} \frac {{ {\left[ 6 \right]_{q^{ 4}}} 
  }} { {\left[ 2 \right]_{q^{ 4}}} {\left[ 3 
   \right]_{q^{ 4}}} } {\left[ 9 m+7 \right]_{q^2}} 
 {\left[ m+1 \right]_{q^{ 18}}}
.\end{multline*} 
If one decomposes $[9m+7]_{q^2}$ as 
$[\frac {9m} {2}+4]_{q^4}+q^2[\frac {9m} {2}+3]_{q^4}$,
then one sees that, by Corollary~\ref{cor:A},
this is a polynomial in $q$ with 
non-negative integer coefficients.
If $m\equiv
111
~(\text{mod }140),$ then we have \begin{multline*}\Cat^m(E_7;q)=
{\left[ \tfrac{9m+1}{5}  \right]_{q^{ 10}}} {\left[ 
  \tfrac{3m+1}{2}  \right]_{q^{ 12}}} {\left[ 9m+4 \right]_{q^{ 
  2}}} {\left[ \tfrac{9m+5}{4}  \right]_{q^{ 8}}} {\left[ 3m+2
   \right]_{q^{ 6}}} \\
\times{\left[ 9 m+7 \right]_{q^2}} 
 {\left[ \tfrac{m+1}{7} \right]_{q^{ 126}}} \frac {{ 
  {\left[ 63 \right]_{q^{ 2}}} }} { {\left[ 7 
   \right]_{q^{ 2}}} {\left[ 9 \right]_{q^{ 2}}} }
,\end{multline*} 
which, by Corollary~\ref{cor:A}, is a polynomial in $q$ with 
non-negative integer coefficients.
If $m\equiv
112
~(\text{mod }140),$ then we have \begin{multline*}\Cat^m(E_7;q)=
{\left[ 9 m+1 \right]_{q^{ 2}}} {\left[ 3 m+1 \right]_{q^{ 
  6}}} {\left[ \tfrac{9m+4}{4}  \right]_{q^{ 8}}} \\
\times{\left[ 
  9 m+5 \right]_{q^2}} {\left[ \tfrac{3m+2}{2}  \right]_{q^{ 
  12}}} {\left[ \tfrac{9m+7}{35}  \right]_{q^{ 70}}} 
  \frac {{ {\left[ 35 \right]_{q^{ 2}}} }} { 
  {\left[ 5 \right]_{q^{ 2}}} {\left[ 7 \right]_{q^{ 2}}} 
  } {\left[ m+1 \right]_{q^{ 18}}}
,\end{multline*} 
which, by Corollary~\ref{cor:A}, is a polynomial in $q$ with 
non-negative integer coefficients.
If $m\equiv
113
~(\text{mod }140),$ then we have \begin{multline*}\Cat^m(E_7;q)=
{\left[ 9 m+1 \right]_{q^{ 2}}} {\left[ \tfrac{3m+1}{10}  
   \right]_{q^{ 60}}} \frac {{ {\left[ 30 \right]_{q^{ 2}}} 
  }} { {\left[ 5 \right]_{q^{ 2}}} {\left[ 6 
   \right]_{q^{ 2}}} } {\left[ 9m+4 \right]_{q^2}}\\
\times 
 {\left[ \tfrac{9m+5}{7}  \right]_{q^{ 14}}} {\left[ 3m+2 
  \right]_{q^{ 6}}} {\left[ \tfrac{9m+7}{4}  \right]_{q^{ 8}}} 
 {\left[ m+1 \right]_{q^{ 18}}}
,\end{multline*} 
which, by Corollary~\ref{cor:A}, is a polynomial in $q$ with 
non-negative integer coefficients.
If $m\equiv
114
~(\text{mod }140),$ then we have \begin{multline*}\Cat^m(E_7;q)=
{\left[ 9 m+1 \right]_{q^{ 2}}} {\left[ \tfrac{3m+1}{7}  
   \right]_{q^{ 42}}} \frac {{ {\left[ 21 \right]_{q^{ 2}}} 
  }} { {\left[ 3 \right]_{q^{ 2}}} {\left[ 7 
   \right]_{q^{ 2}}} } {\left[ \tfrac{9m+4}{10}  \right]_{q^{ 
  20}}}
\frac {[10]_{q^2}} {[2]_{q^2}\left[5\right]_{q^2}}
 {\left[ 9 m+5 \right]_{q^2}} \\
\times{\left[ 
   \tfrac{3m+2}{4}  \right]_{q^{ 24}}} \frac {{ {\left[ 6 
   \right]_{q^{ 4}}} }} { {\left[ 2 \right]_{q^{ 4}}} 
  {\left[ 3 \right]_{q^{ 4}}} } {\left[ 9 m+7 
  \right]_{q^2}} {\left[ m+1 \right]_{q^{ 18}}}
.\end{multline*} 
If one decomposes $[9m+7]_{q^2}$ as 
$[\frac {9m} {2}+4]_{q^4}+q^2[\frac {9m} {2}+3]_{q^4}$,
then one sees that, by Corollary~\ref{cor:A},
this is a polynomial in $q$ with 
non-negative integer coefficients.
If $m\equiv
115
~(\text{mod }140),$ then we have \begin{multline*}\Cat^m(E_7;q)=
{\left[ \tfrac{9m+1}{28}  \right]_{q^{ 56}}} \frac {{ 
  {\left[ 28 \right]_{q^{ 2}}} }} { {\left[ 4 
   \right]_{q^{ 2}}} {\left[ 7 \right]_{q^{ 2}}} } {\left[ 
  \tfrac{3m+1}{2}  \right]_{q^{ 12}}} {\left[ 9m+4 \right]_{q^{ 
  2}}}\\
\times {\left[ \tfrac{9m+5}{5}  \right]_{q^{ 10}}} {\left[ 
  3m+2 \right]_{q^{ 6}}} {\left[ 9 m+7 \right]_{q}} 
 {\left[ m+1 \right]_{q^{ 18}}}
,\end{multline*} 
which, by Corollary~\ref{cor:A}, is a polynomial in $q$ with 
non-negative integer coefficients.
If $m\equiv
116
~(\text{mod }140),$ then we have \begin{multline*}\Cat^m(E_7;q)=
{\left[ \tfrac{9m+1}{5}  \right]_{q^{ 10}}} {\left[ 3 m+1 
  \right]_{q^{ 6}}} {\left[ \tfrac{9m+4}{4}  \right]_{q^{ 8}}} 
 {\left[ 9m+5 \right]_{q^{ 2}}} \\
\times{\left[ \tfrac{3m+2}{14}  
   \right]_{q^{ 84}}} \frac {{ {\left[ 42 \right]_{q^{ 2}}} 
  }} { {\left[ 6 \right]_{q^{ 2}}} {\left[ 7 
   \right]_{q^{ 2}}} } {\left[ 9 m+7 \right]_{q^2}} 
 {\left[ m+1 \right]_{q^{ 18}}}
,\end{multline*} 
which, by Corollary~\ref{cor:A}, is a polynomial in $q$ with 
non-negative integer coefficients.
If $m\equiv
117
~(\text{mod }140),$ then we have \begin{multline*}\Cat^m(E_7;q)=
{\left[ 9 m+1 \right]_{q^{ 2}}} {\left[ \tfrac{3m+1}{4}  
   \right]_{q^{ 24}}} \frac {{ {\left[ 6 \right]_{q^{ 4}}} 
  }} { {\left[ 2 \right]_{q^{ 4}}} {\left[ 3 
   \right]_{q^{ 4}}} } {\left[ \tfrac{9m+4}{7}  \right]_{q^{ 
  14}}} {\left[ \tfrac{9 m+5}2 \right]_{q^4}}\\
\times {\left[ 3m+2 
  \right]_{q^{ 6}}} {\left[ \tfrac{9m+7}{5}  \right]_{q^{ 10}}} 
 {\left[ m+1 \right]_{q^{ 18}}}
,\end{multline*} 
which, by Corollary~\ref{cor:A}, is a polynomial in $q$ with 
non-negative integer coefficients.
If $m\equiv
118
~(\text{mod }140),$ then we have \begin{multline*}\Cat^m(E_7;q)=
{\left[ 9 m+1 \right]_{q^{ 2}}} {\left[ \tfrac{3m+1}{5}  
   \right]_{q^{ 30}}} \frac {{ {\left[ 15 \right]_{q^{ 2}}} 
  }} { {\left[ 3 \right]_{q^{ 2}}} {\left[ 5 
   \right]_{q^{ 2}}} }\\
\times {\left[ \tfrac{9m+4}2 \right]_{q^4}} 
 {\left[ 9 m+5 \right]_{q^2}} {\left[ 
   \tfrac{3m+2}{4}  \right]_{q^{ 24}}} \frac {{ {\left[ 6 
   \right]_{q^{ 4}}} }} { {\left[ 2 \right]_{q^{ 4}}} 
  {\left[ 3 \right]_{q^{ 4}}} } {\left[ 9 m+7 
  \right]_{q^2}} {\left[ \tfrac{m+1}{7} \right]_{q^{ 126}}} 
  \frac {{ {\left[ 63 \right]_{q^{ 2}}} }} { 
  {\left[ 7 \right]_{q^{ 2}}} {\left[ 9 \right]_{q^{ 2}}} 
  }
,\end{multline*} 
which, by Corollary~\ref{cor:A}, is a polynomial in $q$ with 
non-negative integer coefficients.
If $m\equiv
119
~(\text{mod }140),$ then we have $$\Cat^m(E_7;q)=
{\left[ \tfrac{9m+1}{4}  \right]_{q^{ 8}}} {\left[ 
  \tfrac{3m+1}{2}  \right]_{q^{ 12}}} {\left[ \tfrac{9m+4}{5}  
  \right]_{q^{ 10}}} {\left[ 9m+5 \right]_{q^{ 2}}} {\left[ 
  3m+2 \right]_{q^{ 6}}} {\left[ \tfrac{9m+7}{7}  \right]_{q^{ 
  14}}} {\left[ m+1 \right]_{q^{ 18}}}
,$$ 
which is manifestly a polynomial in $q$ with 
non-negative integer coefficients.  
If $m\equiv
120
~(\text{mod }140),$ then we have \begin{multline*}\Cat^m(E_7;q)=
{\left[ 9 m+1 \right]_{q^{ 2}}} {\left[ 3 m+1 \right]_{q^{ 
  6}}} {\left[ \tfrac{9m+4}{4}  \right]_{q^{ 8}}}\\
\times {\left[ 
   \tfrac{9m+5}{35}  \right]_{q^{ 70}}} \frac {{ {\left[ 35
    \right]_{q^{ 2}}} }} { {\left[ 5 \right]_{q^{ 2}}} 
  {\left[ 7 \right]_{q^{ 2}}} } {\left[ \tfrac{3m+2}{2} 
   \right]_{q^{ 12}}} {\left[ 9 m+7 \right]_{q^2}} 
 {\left[ m+1 \right]_{q^{ 18}}}
,\end{multline*} 
which, by Corollary~\ref{cor:A}, is a polynomial in $q$ with 
non-negative integer coefficients.
If $m\equiv
121
~(\text{mod }140),$ then we have \begin{multline*}\Cat^m(E_7;q)=
{\left[ \tfrac{9m+1}{5}  \right]_{q^{ 10}}} {\left[ 
   \tfrac{3m+1}{14}  \right]_{q^{ 84}}} \frac {{ {\left[ 42
    \right]_{q^{ 2}}} }} { {\left[ 6 \right]_{q^{ 2}}} 
  {\left[ 7 \right]_{q^{ 2}}} } {\left[ 9m+4 
  \right]_{q^{ 2}}} \\
\times{\left[ 9 m+5 \right]_{q^2}} {\left[ 
  3m+2 \right]_{q^{ 6}}} {\left[ \tfrac{9m+7}{4}  \right]_{q^{ 
  8}}} {\left[ m+1 \right]_{q^{ 18}}}
,\end{multline*} 
which, by Corollary~\ref{cor:A}, is a polynomial in $q$ with 
non-negative integer coefficients.
If $m\equiv
122
~(\text{mod }140),$ then we have \begin{multline*}\Cat^m(E_7;q)=
{\left[ \tfrac{9m+1}{7}  \right]_{q^{ 14}}} {\left[ 3 m+1 
  \right]_{q^{ 6}}}\\
\times {\left[ \tfrac{9m+4}2 \right]_{q^{ 4}}} {\left[ 
  9 m+5 \right]_{q^2}} {\left[ \tfrac{3m+2}{4}  
   \right]_{q^{ 24}}} \frac {{ {\left[ 6 \right]_{q^{ 4}}} 
  }} { {\left[ 2 \right]_{q^{ 4}}} {\left[ 3 
   \right]_{q^{ 4}}} } {\left[ \tfrac{9m+7}{5}  \right]_{q^{ 
  10}}} {\left[ m+1 \right]_{q^{ 18}}}
,\end{multline*} 
which, by Corollary~\ref{cor:A}, is a polynomial in $q$ with 
non-negative integer coefficients.
If $m\equiv
123
~(\text{mod }140),$ then we have \begin{multline*}\Cat^m(E_7;q)=
{\left[ \tfrac{9m+1}{4}  \right]_{q^{ 8}}} {\left[ 
   \tfrac{3m+1}{10}  \right]_{q^{ 60}}} \frac {{ {\left[ 30
    \right]_{q^{ 2}}} }} { {\left[ 5 \right]_{q^{ 2}}} 
  {\left[ 6 \right]_{q^{ 2}}} } {\left[ 9m+4 
  \right]_{q^{ 2}}} {\left[ 9 m+5 \right]_{q^2}}\\
\times 
 {\left[ \tfrac{3m+2}{7}  \right]_{q^{ 42}}} \frac {{ 
  {\left[ 21 \right]_{q^{ 2}}} }} { {\left[ 3 
   \right]_{q^{ 2}}} {\left[ 7 \right]_{q^{ 2}}} } 
 {\left[ 9 m+7 \right]_{q^2}} {\left[ m+1 \right]_{q^{ 
  18}}}
,\end{multline*} 
which, by Corollary~\ref{cor:A}, is a polynomial in $q$ with 
non-negative integer coefficients.
If $m\equiv
124
~(\text{mod }140),$ then we have \begin{multline*}\Cat^m(E_7;q)=
{\left[ 9 m+1 \right]_{q^{ 2}}} {\left[ 3 m+1 \right]_{q^{ 
  6}}} {\left[ \tfrac{9m+4}{28}  \right]_{q^{ 56}}} 
  \frac {{ {\left[ 28 \right]_{q^{ 2}}} }} { 
  {\left[ 4 \right]_{q^{ 2}}} {\left[ 7 \right]_{q^{ 2}}} 
  } {\left[ 9 m+5 \right]_{q^2}}\\
\times {\left[ 
  \tfrac{3m+2}{2}  \right]_{q^{ 12}}} {\left[ 9 m+7 
  \right]_{q^2}} {\left[ \tfrac{m+1}{5} \right]_{q^{ 90}}} 
  \frac {{ {\left[ 45 \right]_{q^{ 2}}} }} { 
  {\left[ 5 \right]_{q^{ 2}}} {\left[ 9 \right]_{q^{ 2}}} 
  }
,\end{multline*} 
which, by Corollary~\ref{cor:A}, is a polynomial in $q$ with 
non-negative integer coefficients.
If $m\equiv
125
~(\text{mod }140),$ then we have \begin{multline*}\Cat^m(E_7;q)=
{\left[ 9 m+1 \right]_{q^{ 2}}} {\left[ \tfrac{3m+1}{2}  
  \right]_{q^{ 12}}} \\
\times{\left[ 9m+4 \right]_{q^2}} {\left[ 
  \tfrac{9m+5}{5}  \right]_{q^{ 10}}} {\left[ 3m+2 \right]_{q^{ 
  6}}} {\left[ \tfrac{9m+7}{4}  \right]_{q^{ 8}}} {\left[ 
   \tfrac{m+1}{7} \right]_{q^{ 126}}} \frac {{ {\left[ 63 
   \right]_{q^{ 2}}} }} { {\left[ 7 \right]_{q^{ 2}}} 
  {\left[ 9 \right]_{q^{ 2}}} }
,\end{multline*} 
which, by Corollary~\ref{cor:A}, is a polynomial in $q$ with 
non-negative integer coefficients.
If $m\equiv
126
~(\text{mod }140),$ then we have \begin{multline*}\Cat^m(E_7;q)=
{\left[ \tfrac{9m+1}{5}  \right]_{q^{ 10}}} {\left[ 3 m+1 
  \right]_{q^{ 6}}}\\
\times {\left[\tfrac{ 9m+4}2 \right]_{q^{ 4}}} {\left[ 
  9 m+5 \right]_{q^2}} {\left[ \tfrac{3m+2}{4}  
   \right]_{q^{ 24}}} \frac {{ {\left[ 6 \right]_{q^{ 4}}} 
  }} { {\left[ 2 \right]_{q^{ 4}}} {\left[ 3 
   \right]_{q^{ 4}}} } {\left[ \tfrac{9m+7}{7}  \right]_{q^{ 
  14}}} {\left[ m+1 \right]_{q^{ 18}}}
,\end{multline*} 
which, by Corollary~\ref{cor:A}, is a polynomial in $q$ with 
non-negative integer coefficients.
If $m\equiv
127
~(\text{mod }140),$ then we have $$\Cat^m(E_7;q)=
{\left[ \tfrac{9m+1}{4}  \right]_{q^{ 8}}} {\left[ 
  \tfrac{3m+1}{2}  \right]_{q^{ 12}}} {\left[ 9m+4 \right]_{q^{ 
  2}}} {\left[ \tfrac{9m+5}{7}  \right]_{q^{ 14}}} {\left[ 
  3m+2 \right]_{q^{ 6}}} {\left[ \tfrac{9m+7}{5}  \right]_{q^{ 
  10}}} {\left[ m+1 \right]_{q^{ 18}}}
,$$ 
which is manifestly a polynomial in $q$ with 
non-negative integer coefficients.  
If $m\equiv
128
~(\text{mod }140),$ then we have \begin{multline*}\Cat^m(E_7;q)=
{\left[ 9 m+1 \right]_{q^{ 2}}} {\left[ \tfrac{3 m+1}{35}  
   \right]_{q^{ 210}}} \frac {{ {\left[ 105 \right]_{q^{ 2}}} 
  }} { {\left[ 3 \right]_{q^{ 2}}}{\left[ 5 \right]_{q^{ 2}}} {\left[ 7 
   \right]_{q^{ 2}}} } {\left[ \tfrac{9m+4}{4}  \right]_{q^{ 
  8}}} {\left[ 9 m+5 \right]_{q^2}}\\
\times {\left[ 
   \tfrac{3m+2}{2}  \right]_{q^{ 12}}} 
{\left[ 9 m+7 
  \right]_{q^2}} {\left[ m+1 \right]_{q^{ 18}}}
,\end{multline*} 
which, by Lemma~\ref{lem:I}, is a polynomial in $q$ with 
non-negative integer coefficients.
If $m\equiv
129
~(\text{mod }140),$ then we have $$\Cat^m(E_7;q)=
{\left[ \tfrac{9m+1}{7}  \right]_{q^{ 14}}} {\left[ 
  \tfrac{3m+1}{2}  \right]_{q^{ 12}}} {\left[ \tfrac{9m+4}{5}  
  \right]_{q^{ 10}}} {\left[ 9m+5 \right]_{q^{ 2}}} {\left[ 
  3m+2 \right]_{q^{ 6}}} {\left[ \tfrac{9m+7}{4}  \right]_{q^{ 
  8}}} {\left[ m+1 \right]_{q^{ 18}}}
,$$ 
which is manifestly a polynomial in $q$ with 
non-negative integer coefficients.  
If $m\equiv
130
~(\text{mod }140),$ then we have \begin{multline*}\Cat^m(E_7;q)=
{\left[ {9 m+1} \right]_{q^{ 2}}} 
{\left[ {3m+1} 
   \right]_{q^{ 6}}} \\
\times
{\left[ \tfrac{9m+4}2 \right]_{q^{ 4}}} 
 {\left[ \tfrac{9m+5}5 \right]_{q^{ 10}}} {\left[ \tfrac{3m+2}{28} 
   \right]_{q^{ 168}}} 
\frac {[84]_{q^2}\left[2\right]_{q^2}} 
{[4]_{q^2}\left[6\right]_{q^2}\left[7\right]_{q^2}}
{\left[ 9m+7 \right]_{q^{ 2}}} 
 {\left[ m+1 \right]_{q^{ 18}}}
,\end{multline*} 
which, by Lemma~\ref{lem:H}, is a polynomial in $q$ with 
non-negative integer coefficients.
If $m\equiv
131
~(\text{mod }140),$ then we have $$\Cat^m(E_7;q)=
{\left[ \tfrac{9m+1}{5}  \right]_{q^{ 10}}} {\left[ 
  \tfrac{3m+1}{2}  \right]_{q^{ 12}}} {\left[ \tfrac{9m+4}{7}  
  \right]_{q^{ 14}}} {\left[ \tfrac{9m+5}{4}  \right]_{q^{ 8}}} 
 {\left[ 3m+2 \right]_{q^{ 6}}} {\left[ 9m+7 \right]_{q^{ 
  2}}} {\left[ m+1 \right]_{q^{ 18}}}
,$$ 
which is manifestly a polynomial in $q$ with 
non-negative integer coefficients.  
If $m\equiv
132
~(\text{mod }140),$ then we have \begin{multline*}\Cat^m(E_7;q)=
{\left[ 9 m+1 \right]_{q^{ 2}}} {\left[ 3 m+1 \right]_{q^{ 
  6}}} {\left[ \tfrac{9m+4}{4}  \right]_{q^{ 8}}} \\
\times{\left[ 
  9 m+5 \right]_{q^2}} {\left[ \tfrac{3m+2}{2}  \right]_{q^{ 
  12}}} {\left[ \tfrac{9m+7}{5}  \right]_{q^{ 10}}} {\left[ 
   \tfrac{m+1}{7} \right]_{q^{ 126}}} \frac {{ {\left[ 63 
   \right]_{q^{ 2}}} }} { {\left[ 7 \right]_{q^{ 2}}} 
  {\left[ 9 \right]_{q^{ 2}}} }
,\end{multline*} 
which, by Corollary~\ref{cor:A}, is a polynomial in $q$ with 
non-negative integer coefficients.
If $m\equiv
133
~(\text{mod }140),$ then we have \begin{multline*}\Cat^m(E_7;q)=
{\left[ 9 m+1 \right]_{q^{ 2}}} {\left[ \tfrac{3m+1}{10}  
   \right]_{q^{ 60}}} \frac {{ {\left[ 30 \right]_{q^{ 2}}} 
  }} { {\left[ 5 \right]_{q^{ 2}}} {\left[ 6 
   \right]_{q^{ 2}}} } {\left[ 9m+4 \right]_{q^2}}\\
\times 
 {\left[ 9 m+5 \right]_{q^2}} {\left[ 3m+2 
  \right]_{q^{ 6}}} {\left[ \tfrac{9m+7}{28}  \right]_{q^{ 56}}} 
  \frac {{ {\left[ 28 \right]_{q^{ 2}}} }} { 
  {\left[ 4 \right]_{q^{ 2}}} {\left[ 7 \right]_{q^{ 2}}} 
  } {\left[ m+1 \right]_{q^{ 18}}}
,\end{multline*} 
which, by Corollary~\ref{cor:A}, is a polynomial in $q$ with 
non-negative integer coefficients.
If $m\equiv
134
~(\text{mod }140),$ then we have \begin{multline*}\Cat^m(E_7;q)=
{\left[ 9 m+1 \right]_{q^{ 2}}} {\left[ 3 m+1 \right]_{q^{ 
  6}}} {\left[ \tfrac{9m+4}{10}  \right]_{q^{ 20}}}
\frac {[10]_{q^2}} {[2]_{q^2}\left[5\right]_{q^2}}\\
\times
 {\left[ 
  \tfrac{9m+5}{7}  \right]_{q^{ 14}}} {\left[ \tfrac{3m+2}{4}  
   \right]_{q^{ 24}}} \frac {{ {\left[ 6 \right]_{q^{ 4}}} 
  }} { {\left[ 2 \right]_{q^{ 4}}} {\left[ 3 
   \right]_{q^{ 4}}} } {\left[ 9 m+7 \right]_{q^2}} 
 {\left[ m+1 \right]_{q^{ 18}}}
.\end{multline*} 
If one decomposes $[9m+7]_{q^2}$ as 
$[\frac {9m} {2}+4]_{q^4}+q^2[\frac {9m} {2}+3]_{q^4}$,
then one sees that, by Corollary~\ref{cor:A},
this is a polynomial in $q$ with 
non-negative integer coefficients.
If $m\equiv
135
~(\text{mod }140),$ then we have \begin{multline*}\Cat^m(E_7;q)=
{\left[ \tfrac{9m+1}{4}  \right]_{q^{ 8}}} {\left[ 
   \tfrac{3m+1}{14}  \right]_{q^{ 84}}} \frac {{ {\left[ 42
    \right]_{q^{ 2}}} }} { {\left[ 6 \right]_{q^{ 2}}} 
  {\left[ 7 \right]_{q^{ 2}}} } {\left[ 9m+4 
  \right]_{q^{ 2}}}\\
\times {\left[ \tfrac{9m+5}{5}  \right]_{q^{ 10}}} 
 {\left[ 3m+2 \right]_{q^{ 6}}} {\left[ 9 m+7 
  \right]_{q^2}} {\left[ m+1 \right]_{q^{ 18}}}
,\end{multline*} 
which, by Corollary~\ref{cor:A}, is a polynomial in $q$ with 
non-negative integer coefficients.
If $m\equiv
136
~(\text{mod }140),$ then we have \begin{multline*}\Cat^m(E_7;q)=
{\left[ \tfrac{9m+1}{35}  \right]_{q^{ 70}}} \frac {{ 
  {\left[ 35 \right]_{q^{ 2}}} }} { {\left[ 5 
   \right]_{q^{ 2}}} {\left[ 7 \right]_{q^{ 2}}} } {\left[ 
  3 m+1 \right]_{q^{ 6}}} {\left[ \tfrac{9m+4}{4}  \right]_{q^{ 
  8}}} {\left[ 9m+5 \right]_{q^{ 2}}}\\
\times {\left[ \tfrac{3m+2}{2} 
   \right]_{q^{ 12}}} {\left[ 9 m+7 \right]_{q^2}} 
 {\left[ m+1 \right]_{q^{ 18}}}
,\end{multline*} 
which, by Corollary~\ref{cor:A}, is a polynomial in $q$ with 
non-negative integer coefficients.
If $m\equiv
137
~(\text{mod }140),$ then we have \begin{multline*}\Cat^m(E_7;q)=
{\left[ 9 m+1 \right]_{q^{ 2}}} {\left[ \tfrac{3m+1}{4}  
   \right]_{q^{ 24}}} \frac {{ {\left[ 6 \right]_{q^{ 4}}} 
  }} { {\left[ 2 \right]_{q^{ 4}}} {\left[ 3 
   \right]_{q^{ 4}}} } {\left[ 9m+4 \right]_{q^2}} 
 {\left[ \tfrac{9 m+5}2 \right]_{q^4}}\\
\times {\left[ 
   \tfrac{3m+2}{7}  \right]_{q^{ 42}}} \frac {{ {\left[ 21 
   \right]_{q^{ 2}}} }} { {\left[ 3 \right]_{q^{ 2}}} 
  {\left[ 7 \right]_{q^{ 2}}} } {\left[ \tfrac{9m+7}{5} 
   \right]_{q^{ 10}}} {\left[ m+1 \right]_{q^{ 18}}}
,\end{multline*} 
which, by Corollary~\ref{cor:A}, is a polynomial in $q$ with 
non-negative integer coefficients.
If $m\equiv
138
~(\text{mod }140),$ then we have \begin{multline*}\Cat^m(E_7;q)=
{\left[ 9 m+1 \right]_{q^{ 2}}} {\left[ \tfrac{3m+1}{5}  
   \right]_{q^{ 30}}} \frac {{ {\left[ 15 \right]_{q^{ 2}}} 
  }} { {\left[ 3 \right]_{q^{ 2}}} {\left[ 5 
   \right]_{q^{ 2}}} } {\left[ \tfrac{9m+4}{14}  \right]_{q^{ 
  28}}}
\frac {[14]_{q^2}} {[2]_{q^2}\left[7\right]_{q^2}}
 {\left[ 9 m+5 \right]_{q^2}}\\
\times {\left[ 
   \tfrac{3m+2}{4}  \right]_{q^{ 24}}} \frac {{ {\left[ 6 
   \right]_{q^{ 4}}} }} { {\left[ 2 \right]_{q^{ 4}}} 
  {\left[ 3 \right]_{q^{ 4}}} } {\left[ 9 m+7 
  \right]_{q^2}} {\left[ m+1 \right]_{q^{ 18}}}
.\end{multline*} 
If one decomposes $[9m+7]_{q^2}$ as 
$[\frac {9m} {2}+4]_{q^4}+q^2[\frac {9m} {2}+3]_{q^4}$,
then one sees that, by Corollary~\ref{cor:A},
this is a polynomial in $q$ with 
non-negative integer coefficients.
If $m\equiv
139
~(\text{mod }140),$ then we have \begin{multline*}\Cat^m(E_7;q)=
{\left[ \tfrac{9m+1}{4}  \right]_{q^{ 8}}} {\left[ 
  \tfrac{3m+1}{2}  \right]_{q^{ 12}}} {\left[ \tfrac{9m+4}{5}  
  \right]_{q^{ 10}}} {\left[ 9m+5 \right]_{q^{ 2}}}\\
\times {\left[ 
  3m+2 \right]_{q^{ 6}}} {\left[ 9 m+7 \right]_{q^2}} 
 {\left[ \tfrac{m+1}{7} \right]_{q^{ 126}}} \frac {{ 
  {\left[ 63 \right]_{q^{ 2}}} }} { {\left[ 7 
   \right]_{q^{ 2}}} {\left[ 9 \right]_{q^{ 2}}} },
\end{multline*}
which, by Corollary~\ref{cor:A}, is a polynomial in $q$ with 
non-negative integer coefficients.

For $W=G_{37}=E_8$, the degrees are $2,8,12,14,18,20,24,30$, and hence
\begin{multline*}
\Cat^m(E_7;q)=\frac {[30m+2]_q\,[30m+8]_q\,[30m+12]_q\,[30m+14]_q\,
}
{[2]_q\,[8]_q\,[12]_q\,[14]_q}\\
\times
\frac {
[30m+18]_q\,[30m+20]_q\,[30m+24]_q\,[30m+30]_q}
{[18]_q\,[20]_q\,[24]_q\,[30]_q}.
\end{multline*}
If $m\equiv
0
~(\text{mod }84),$ then we have \begin{multline*} \Cat^m(E_8;q)=
{\left[ 15 m+1 \right]_{q^{ 2}}} {\left[ \tfrac{15m+4}{4}  
  \right]_{q^{ 8}}} {\left[ \tfrac{5m+2}{2}  \right]_{q^{ 12}}} 
 {\left[ \tfrac{15m+7}{7}  \right]_{q^{ 14}}}\\
\times
 {\left[ 
  \tfrac{5m+3}{3}  \right]_{q^{ 18}}} {\left[ \tfrac{3m+2}{2}  
  \right]_{q^{ 20}}} {\left[ \tfrac{5m+4}{4}  \right]_{q^{ 24}}} 
 {\left[ m+1 \right]_{q^{ 30}}}
,\end{multline*} 
which is manifestly a polynomial in $q$ with 
non-negative integer coefficients.  
If $m\equiv
1
~(\text{mod }84),$ then we have 
\begin{multline*} \Cat^m(E_8;q)=
{\left[ \tfrac{15m+1}{4}  \right]_{q^{ 8 }}}
 {\left[ 15m+4 \right]_{q^{ 2 }}} {\left[ 
  \tfrac{5m+2}{7}  \right]_{q^{ 42 }}}
\frac {[21]_{q^2}} {[3]_{q^2}\left[7\right]_{q^2}}
 {\left[ \tfrac{15 m+7}2 
  \right]_{q^4}} {\left[ \tfrac{5m+3}{4}  
   \right]_{q^{ 24}}} 
   \\
\times {\left[ 3m+2 \right]_{q^{10}}} 
 {\left[ \tfrac{5m+4}{3}  \right]_{q^{ 18 }}}
 {\left[ \tfrac{m+1}{2} \right]_{q^{ 60}}} \frac {{ 
  {\left[ 30 \right]_{q^{ 2}}}
  {\left[ 2 \right]_{q^{ 2}}}
  {\left[ 3 \right]_{q^{ 2}}}
  {\left[ 5 \right]_{q^{ 2}}} }} { {\left[ 6 
   \right]_{q^{ 2}}}{\left[ 10
   \right]_{q^{ 2}}} {\left[ 15 \right]_{q^{ 2}}}}
,\end{multline*} 
which, by Corollary~\ref{cor:A} and Lemma~\ref{lem:61015}, 
is a polynomial in $q$ with 
non-negative integer coefficients.
If $m\equiv
2
~(\text{mod }84),$ then we have \begin{multline*} \Cat^m(E_8;q)=
{\left[ 15 m+1 \right]_{q^{ 2 }}} {\left[ \tfrac{15m+4}2 
  \right]_{q^4}} {\left[ \tfrac{5m+2}{12}  
   \right]_{q^{ 72}}} \frac {{ {\left[ 12 \right]_{q^{ 6}}} 
  }} { {\left[ 3 \right]_{q^{ 6}}} {\left[ 4 
   \right]_{q^{ 6}}}} {\left[ 15 m+7 \right]_{q^2}} 
 {\left[ 5 m+3 \right]_{q^6}} \\
\times {\left[ 
   \tfrac{3m+2}{4}  \right]_{q^{ 40}}} \frac {{ {\left[ 10 
   \right]_{q^{ 4}}} }} { {\left[ 2 \right]_{q^{ 4}}} 
  {\left[ 5 \right]_{q^{ 4}}}} {\left[ \tfrac{5m+4}{14}  
   \right]_{q^{ 84}}} \frac {{ {\left[ 42 \right]_{q^{ 2}}} 
  }} { {\left[ 6 \right]_{q^{ 2}}} {\left[ 7 
   \right]_{q^{ 2}}}} {\left[ m+1 \right]_{q^{ 30 }}}
,\end{multline*} 
which, by Corollary~\ref{cor:A}, is a polynomial in $q$ with 
non-negative integer coefficients.
If $m\equiv
3
~(\text{mod }84),$ then we have 
\begin{multline*} \Cat^m(E_8;q)=
{\left[ \tfrac{15 m+1}2 \right]_{q^{ 4}}} 
{\left[ \tfrac{15m+4}7 \right]_{q^{ 
   14}}} 
 {\left[ 5m+2 \right]_{q^{ 6 }}}
 {\left[ \tfrac{15m+7}4 
   \right]_{q^{ 8}}} 
 {\left[ \tfrac{5m+3}6 \right]_{q^{ 36}}}
\frac {[6]_{q^6}} {[2]_{q^6}\left[3\right]_{q^6}}
\\
\times 
 {\left[ 3m+2 \right]_{q^{ 10 }}} {\left[ 5m+4
   \right]_{q^{ 6 }}} {\left[ \tfrac{m+1}4 \right]_{q^{ 120}}}
\frac {[60]_{q^2}\left[2\right]_{q^2}\left[3\right]_{q^2}\left[5\right]_{q^2}} 
{\left[10\right]_{q^2}\left[12\right]_{q^2}\left[15\right]_{q^2}} 
,\end{multline*} 
which, by Corollary~\ref{cor:A} and Lemma~\ref{lem:101215}, 
is a polynomial in $q$ with 
non-negative integer coefficients.
If $m\equiv
4
~(\text{mod }84),$ then we have \begin{multline*} \Cat^m(E_8;q)=
{\left[ 15 m+1 \right]_{q^{ 2 }}} {\left[ 
  \tfrac{15m+4}{4}  \right]_{q^{ 8 }}} {\left[ 
  \tfrac{5m+2}{2}  \right]_{q^{ 12 }}} \\
\times{\left[ 15 m+7 
  \right]_{q^2}} {\left[ 5 m+3 \right]_{q^{ 6 
  }}} 
{\left[ \tfrac{3m+2}{14}  \right]_{q^{ 140}}} 
  \frac {{ {\left[ 70 \right]_{q^{ 2}}} }} { 
  {\left[ 7 \right]_{q^{ 2}}} {\left[ 10 \right]_{q^{ 2}}}} 
 {\left[ \tfrac{5m+4}{12}  \right]_{q^{ 72}}} \frac {{ 
  {\left[ 12 \right]_{q^{ 6}}} }} { {\left[ 3 
   \right]_{q^{ 6}}} {\left[ 4 \right]_{q^{ 6}}}} {\left[ m+1 
  \right]_{q^{ 30 }}}
,\end{multline*} 
which, by Corollary~\ref{cor:A}, is a polynomial in $q$ with 
non-negative integer coefficients.
If $m\equiv
5
~(\text{mod }84),$ then we have \begin{multline*} \Cat^m(E_8;q)=
{\left[ \tfrac{15m+1}{4}  \right]_{q^{ 8 }}}
 {\left[ 15m+4 \right]_{q^{ 2 }}} {\left[ 
  \tfrac{5m+2}3  \right]_{q^{ 18 }}}
 {\left[ \tfrac{15 m+7}2 
  \right]_{q^4}} {\left[ \tfrac{5m+3}{28}  
   \right]_{q^{ 168}}}
\frac {[84]_{q^2}} {[7]_{q^2}\left[12\right]_{q^2}} 
   \\
\times {\left[ 3m+2 \right]_{q^{10}}} 
 {\left[ {5m+4}  \right]_{q^{ 6 }}}
 {\left[ \tfrac{m+1}{2} \right]_{q^{ 60}}} \frac {{ 
  {\left[ 30 \right]_{q^{ 2}}}
  {\left[ 2 \right]_{q^{ 2}}}
  {\left[ 3 \right]_{q^{ 2}}}
  {\left[ 5 \right]_{q^{ 2}}} }} { {\left[ 6 
   \right]_{q^{ 2}}}{\left[ 10
   \right]_{q^{ 2}}} {\left[ 15 \right]_{q^{ 2}}}}
,\end{multline*} 
which, by Corollary~\ref{cor:A} and Lemma~\ref{lem:61015}, 
is a polynomial in $q$ with 
non-negative integer coefficients.
If $m\equiv
6
~(\text{mod }84),$ then we have \begin{multline*} \Cat^m(E_8;q)=
{\left[ \tfrac{15m+1}{7}  \right]_{q^{ 14 }}} 
 {\left[\tfrac{ 15m+4}2 \right]_{q^{ 4 }}} {\left[ 
   \tfrac{5m+2}{4}  \right]_{q^{ 24}}} 
{\left[ 15 m+7 \right]_{q^{ 
  2 }}} \\
\times
{\left[ \tfrac{5m+3}{3}  \right]_{q^{ 18 
  }}} {\left[ \tfrac{3m+2}{4}  \right]_{q^{ 40 }}} 
\frac {[10]_{q^4}} {[2]_{q^4}\left[5\right]_{q^4}}
 {\left[ \tfrac{5m+4}{2}  \right]_{q^{ 12 }}} 
 {\left[ m+1 \right]_{q^{ 30 }}}
,\end{multline*} 
which, by Corollary~\ref{cor:A}, is a polynomial in $q$ with 
non-negative integer coefficients.
If $m\equiv
7
~(\text{mod }84),$ then we have 
\begin{multline*} \Cat^m(E_8;q)=
{\left[ \tfrac{15 m+1}2 \right]_{q^{ 4}}} 
{\left[ {15m+4} \right]_{q^{ 
   2}}} 
 {\left[ 5m+2 \right]_{q^{ 6 }}}
 {\left[ \tfrac{15m+7}28 
   \right]_{q^{ 56}}} 
\frac {[28]_{q^2}} {[4]_{q^2}\left[7\right]_{q^2}}
 {\left[ \tfrac{5m+3}2 \right]_{q^{ 12}}}\\
\times 
 {\left[ 3m+2 \right]_{q^{ 10 }}} {\left[ \tfrac{5m+4}3
   \right]_{q^{ 18 }}} {\left[ \tfrac{m+1}4 \right]_{q^{ 120}}}
\frac {[60]_{q^2}\left[2\right]_{q^2}\left[3\right]_{q^2}\left[5\right]_{q^2}} 
{\left[10\right]_{q^2}\left[12\right]_{q^2}\left[15\right]_{q^2}} 
,\end{multline*} 
which, by Corollary~\ref{cor:A} and Lemma~\ref{lem:101215}, 
is a polynomial in $q$ with 
non-negative integer coefficients.
If $m\equiv
8
~(\text{mod }84),$ then we have \begin{multline*} \Cat^m(E_8;q)=
{\left[ 15 m+1 \right]_{q^{ 2 }}} {\left[ 
  \tfrac{15m+4}{4}  \right]_{q^{ 8 }}} {\left[ 
   \tfrac{5m+2}{42}  \right]_{q^{252}}} \frac {{ {\left[ 126
    \right]_{q^{ 2}}}{\left[ 3
    \right]_{q^{ 2}}} }} { {\left[ 6
    \right]_{q^{ 2}}}{\left[ 7 \right]_{q^{ 2}}} 
  {\left[ 9 \right]_{q^{ 2}}}}
 {\left[ 15 m+7 \right]_{q^{ 
  2 }}} {\left[ 5 m+3 \right]_{q^6}} \\
\times
 {\left[ \tfrac{3m+2}{2}  \right]_{q^{ 20 }}} 
 {\left[ \tfrac{5m+4}{4}  \right]_{q^{ 24}}}
{\left[ m+1 
  \right]_{q^{ 30 }}}
,\end{multline*} 
which, by Lemma~\ref{lem:679}, is a polynomial in $q$ with 
non-negative integer coefficients.
If $m\equiv
9
~(\text{mod }84),$ then we have \begin{multline*} \Cat^m(E_8;q)=
{\left[ \tfrac{15m+1}{4}  \right]_{q^{ 8 }}}
 {\left[ 15m+4 \right]_{q^{ 2 }}} {\left[ 
  {5m+2}  \right]_{q^{ 6 }}}
 {\left[ \tfrac{15 m+7}2 
  \right]_{q^4}} {\left[ \tfrac{5m+3}{12}  
   \right]_{q^{ 72}}}
\frac {[12]_{q^6}} {[3]_{q^6}\left[4\right]_{q^6}} 
   \\
\times {\left[ 3m+2 \right]_{q^{10}}} 
 {\left[ \tfrac{5m+4}7  \right]_{q^{ 42 }}}
\frac {[21]_{q^2}} {[3]_{q^2}\left[7\right]_{q^2}}
 {\left[ \tfrac{m+1}{2} \right]_{q^{ 60}}} \frac {{ 
  {\left[ 30 \right]_{q^{ 2}}}
  {\left[ 2 \right]_{q^{ 2}}}
  {\left[ 3 \right]_{q^{ 2}}}
  {\left[ 5 \right]_{q^{ 2}}} }} { {\left[ 6 
   \right]_{q^{ 2}}}{\left[ 10
   \right]_{q^{ 2}}} {\left[ 15 \right]_{q^{ 2}}}}
,\end{multline*} 
which, by Corollary~\ref{cor:A} and Lemma~\ref{lem:61015}, 
is a polynomial in $q$ with 
non-negative integer coefficients.
If $m\equiv
10
~(\text{mod }84),$ then we have \begin{multline*} \Cat^m(E_8;q)=
{\left[ 15 m+1 \right]_{q^{ 2 }}} {\left[ 
  \tfrac{15m+4}{14}  \right]_{q^{ 28}}}
\frac {[14]_{q^2}} {[2]_{q^2}\left[7\right]_{q^2}}
 {\left[ 
   \tfrac{5m+2}{4}  \right]_{q^{ 24}}} \\
\times
{\left[ 15 m+7 
  \right]_{q^2}} {\left[ 5 m+3 \right]_{q^{ 6
  }}} 
{\left[ \tfrac{3m+2}{4}  \right]_{q^{ 40}}} 
  \frac {{ {\left[ 10 \right]_{q^{ 4}}} }} { 
  {\left[ 2 \right]_{q^{ 4}}} {\left[ 5 \right]_{q^{ 4}}}} 
 {\left[ \tfrac{5m+4}{6}  \right]_{q^{ 36 }}}
\frac {[6]_{q^6}} {[2]_{q^6}\left[3\right]_{q^6}} 
 {\left[ m+1 \right]_{q^{ 30 }}}
.\end{multline*} 
If one decomposes $[15m+7]_{q^2}$ as 
$[\frac {15m} {2}+4]_{q^4}+q^2[\frac {15m} {2}+3]_{q^4}$,
then one sees that, by Corollary~\ref{cor:A},
this is a polynomial in $q$ with 
non-negative integer coefficients.
If $m\equiv
11
~(\text{mod }84),$ then we have 
\begin{multline*} \Cat^m(E_8;q)=
{\left[ \tfrac{15 m+1}2 \right]_{q^{ 4}}} 
{\left[ {15m+4} \right]_{q^{ 
   2}}} 
 {\left[ \tfrac{5m+2}3 \right]_{q^{ 18 }}}
 {\left[ \tfrac{15m+7}4 
   \right]_{q^{ 8}}} 
 {\left[ \tfrac{5m+3}2 \right]_{q^{ 12}}}\\
\times 
 {\left[ \tfrac{3m+2}7 \right]_{q^{ 70 }}} 
\frac {[35]_{q^2}} {[5]_{q^2}\left[7\right]_{q^2}}
{\left[ 5m+4
   \right]_{q^{ 6 }}} {\left[ \tfrac{m+1}4 \right]_{q^{ 120}}}
\frac {[60]_{q^2}\left[2\right]_{q^2}\left[3\right]_{q^2}\left[5\right]_{q^2}} 
{\left[10\right]_{q^2}\left[12\right]_{q^2}\left[15\right]_{q^2}} 
,\end{multline*} 
which, by Corollary~\ref{cor:A} and Lemma~\ref{lem:101215B}, 
is a polynomial in $q$ with 
non-negative integer coefficients.
If $m\equiv
12
~(\text{mod }84),$ then we have \begin{multline*} \Cat^m(E_8;q)=
{\left[ 15 m+1 \right]_{q^{ 2 }}} {\left[ 
  \tfrac{15m+4}{4}  \right]_{q^{ 8 }}} {\left[ 
  \tfrac{5m+2}{2}  \right]_{q^{ 12 }}} \\
\times{\left[ 15 m+7 
  \right]_{q^2}} {\left[ \tfrac{5m+3}{21}  
   \right]_{q^{ 126}}}
 \frac {{ {\left[ 63 \right]_{q^{ 2}}} 
  }} { {\left[ 7 \right]_{q^{ 2}}} {\left[ 9 
   \right]_{q^{ 2}}}} {\left[ \tfrac{3m+2}{2}  \right]_{q^{ 20 
  }}} {\left[ \tfrac{5m+4}{4}  \right]_{q^{ 24 }}} 
 {\left[ m+1 \right]_{q^{ 30 }}}
,\end{multline*} 
which, by Corollary~\ref{cor:A}, is a polynomial in $q$ with 
non-negative integer coefficients.
If $m\equiv
13
~(\text{mod }84),$ then we have \begin{multline*} \Cat^m(E_8;q)=
{\left[ \tfrac{15m+1}{28}  \right]_{q^{ 56 }}}
\frac {[28]_{q^2}} {[4]_{q^2}\left[7\right]_{q^2}}
 {\left[ 15m+4 \right]_{q^{ 2 }}} {\left[ 
  {5m+2}  \right]_{q^{ 6 }}}
 {\left[ \tfrac{15 m+7}2 
  \right]_{q^4}} {\left[ \tfrac{5m+3}{4}  
   \right]_{q^{ 24}}} 
   \\
\times {\left[ 3m+2 \right]_{q^{10}}} 
 {\left[ \tfrac{5m+4}3  \right]_{q^{ 18 }}}
 {\left[ \tfrac{m+1}{2} \right]_{q^{ 60}}} \frac {{ 
  {\left[ 30 \right]_{q^{ 2}}}
  {\left[ 2 \right]_{q^{ 2}}}
  {\left[ 3 \right]_{q^{ 2}}}
  {\left[ 5 \right]_{q^{ 2}}} }} { {\left[ 6 
   \right]_{q^{ 2}}}{\left[ 10
   \right]_{q^{ 2}}} {\left[ 15 \right]_{q^{ 2}}}}
,\end{multline*} 
which, by Corollary~\ref{cor:A} and Lemma~\ref{lem:61015}, 
is a polynomial in $q$ with 
non-negative integer coefficients.
If $m\equiv
14
~(\text{mod }84),$ then we have \begin{multline*} \Cat^m(E_8;q)=
{\left[ 15 m+1 \right]_{q^{ 2 }}} {\left[ \tfrac{15m+4}2 
  \right]_{q^4}} {\left[ \tfrac{5m+2}{12}  
   \right]_{q^{ 72}}} \frac {{ {\left[ 12 \right]_{q^{ 6}}} 
  }} { {\left[ 3 \right]_{q^{ 6}}} {\left[ 4 
   \right]_{q^{ 6}}}} {\left[ \tfrac{15m+7}{7}  \right]_{q^{ 14 
  }}} \\
\times
{\left[ 5 m+3 \right]_{q^6}} 
 {\left[ \tfrac{3m+2}{4}  \right]_{q^{ 40}}} \frac {{ 
  {\left[ 10 \right]_{q^{ 4}}} }} { {\left[ 2 
   \right]_{q^{ 4}}} {\left[ 5 \right]_{q^{ 4}}}} {\left[ 
  \tfrac{5m+4}{2}  \right]_{q^{ 12 }}} {\left[ m+1 
  \right]_{q^{ 30 }}}
,\end{multline*} 
which, by Corollary~\ref{cor:A}, is a polynomial in $q$ with 
non-negative integer coefficients.
If $m\equiv
15
~(\text{mod }84),$ then we have 
\begin{multline*} \Cat^m(E_8;q)=
{\left[ \tfrac{15 m+1}2 \right]_{q^{ 4}}} 
{\left[ {15m+4} \right]_{q^{ 
   2}}} 
 {\left[ \tfrac{5m+2}7 \right]_{q^{ 42 }}}
\frac {[21]_{q^2}} {[3]_{q^2}\left[7\right]_{q^2}}
 {\left[ \tfrac{15m+7}4 
   \right]_{q^{ 8}}} 
 {\left[ \tfrac{5m+3}6 \right]_{q^{ 36}}}
\frac {[6]_{q^6}} {[2]_{q^6}\left[3\right]_{q^6}}
\\
\times 
 {\left[ 3m+2 \right]_{q^{ 10 }}} {\left[ 5m+4
   \right]_{q^{ 6 }}} {\left[ \tfrac{m+1}4 \right]_{q^{ 120}}}
\frac {[60]_{q^2}\left[2\right]_{q^2}\left[3\right]_{q^2}\left[5\right]_{q^2}} 
{\left[10\right]_{q^2}\left[12\right]_{q^2}\left[15\right]_{q^2}} 
,\end{multline*} 
which, by Corollary~\ref{cor:A} and Lemma~\ref{lem:101215}, 
is a polynomial in $q$ with 
non-negative integer coefficients.
If $m\equiv
16
~(\text{mod }84),$ then we have \begin{multline*} \Cat^m(E_8;q)=
{\left[ {15m+1}  \right]_{q^{ 2 }}}
 {\left[ \tfrac{15m+4}4 \right]_{q^{ 8 }}} {\left[ 
  \tfrac{5m+2}2  \right]_{q^{ 12 }}}
 {\left[ {15 m+7} 
  \right]_{q^2}} {\left[ {5m+3}  
   \right]_{q^{ 6}}}
   \\
\times {\left[ \tfrac{3m+2}2 \right]_{q^{20}}} 
 {\left[ \tfrac{5m+4}{84}  \right]_{q^{ 504 }}}
\frac {[252]_{q^2}\left[3\right]_{q^2}} 
{[7]_{q^2}\left[9\right]_{q^2}\left[12\right]_{q^2}}
 {\left[ {m+1} \right]_{q^{ 30}}} 
,\end{multline*} 
which, by Lemma~\ref{lem:7912}, 
is a polynomial in $q$ with 
non-negative integer coefficients.
If $m\equiv
17
~(\text{mod }84),$ then we have \begin{multline*} \Cat^m(E_8;q)=
{\left[ \tfrac{15m+1}{4}  \right]_{q^{ 8 }}}
 {\left[ \tfrac{15m+4}7 \right]_{q^{ 14 }}} {\left[ 
  \tfrac{5m+2}3  \right]_{q^{ 18 }}}
 {\left[ \tfrac{15 m+7}2 
  \right]_{q^4}} {\left[ \tfrac{5m+3}{4}  
   \right]_{q^{ 24}}} 
   \\
\times {\left[ 3m+2 \right]_{q^{10}}} 
 {\left[ {5m+4}  \right]_{q^{ 6 }}}
 {\left[ \tfrac{m+1}{2} \right]_{q^{ 60}}} \frac {{ 
  {\left[ 30 \right]_{q^{ 2}}}
  {\left[ 2 \right]_{q^{ 2}}}
  {\left[ 3 \right]_{q^{ 2}}}
  {\left[ 5 \right]_{q^{ 2}}} }} { {\left[ 6 
   \right]_{q^{ 2}}}{\left[ 10
   \right]_{q^{ 2}}} {\left[ 15 \right]_{q^{ 2}}}}
,\end{multline*} 
which, by Lemma~\ref{lem:61015}, 
is a polynomial in $q$ with 
non-negative integer coefficients.
If $m\equiv
18
~(\text{mod }84),$ then we have \begin{multline*} \Cat^m(E_8;q)=
{\left[ 15 m+1 \right]_{q^{ 2 }}} {\left[ \tfrac{15m+4}2 
  \right]_{q^4}} {\left[ \tfrac{5m+2}{4}  
   \right]_{q^{ 24}}} 
{\left[ 15 m+7 \right]_{q^2}} 
 {\left[ \tfrac{5m+3}{3}  \right]_{q^{ 18 }}} \\
 {\left[ \tfrac{3m+2}{28}  \right]_{q^{ 280}}} 
\frac {{ 
  {\left[ 140 \right]_{q^{ 2}}}  {\left[ 2 \right]_{q^{ 2}}} }}
 {  {\left[ 4 \right]_{q^{ 2}}} {\left[ 7 
   \right]_{q^{ 2}}} {\left[ 10 \right]_{q^{ 2}}}} {\left[ 
  \tfrac{5m+4}{2}  \right]_{q^{ 12 }}} {\left[ m+1 
  \right]_{q^{ 30 }}}
,\end{multline*} 
which, by Lemma~\ref{lem:4710}, 
is a polynomial in $q$ with 
non-negative integer coefficients.
If $m\equiv
19
~(\text{mod }84),$ then we have 
\begin{multline*} \Cat^m(E_8;q)=
{\left[ \tfrac{15 m+1}2 \right]_{q^{ 4}}} 
{\left[ {15m+4} \right]_{q^{ 
   2}}} 
 {\left[ 5m+2 \right]_{q^{ 6 }}}
 {\left[ \tfrac{15m+7}4 
   \right]_{q^{ 8}}} 
 {\left[ \tfrac{5m+3}14 \right]_{q^{ 84}}}
\frac {[42]_{q^2}} {[6]_{q^2}\left[7\right]_{q^2}}
\\
\times 
 {\left[ 3m+2 \right]_{q^{ 10 }}} {\left[ \tfrac{5m+4}3
   \right]_{q^{ 18 }}} {\left[ \tfrac{m+1}4 \right]_{q^{ 120}}}
\frac {[60]_{q^2}\left[2\right]_{q^2}\left[3\right]_{q^2}\left[5\right]_{q^2}} 
{\left[10\right]_{q^2}\left[12\right]_{q^2}\left[15\right]_{q^2}} 
,\end{multline*} 
which, by Corollary~\ref{cor:A} and Lemma~\ref{lem:101215}, 
is a polynomial in $q$ with 
non-negative integer coefficients.
If $m\equiv
20
~(\text{mod }84),$ then we have \begin{multline*} \Cat^m(E_8;q)=
{\left[ \tfrac{15m+1}{7}  \right]_{q^{ 14 }}} 
 {\left[ \tfrac{15m+4}{4}  \right]_{q^{ 8 }}} 
 {\left[ \tfrac{5m+2}{6}  \right]_{q^{ 36}}} \frac {{ 
  {\left[ 6 \right]_{q^{ 6}}} }} { {\left[ 2 
   \right]_{q^{ 6}}} {\left[ 3 \right]_{q^{ 6}}}} {\left[ 15m+7
   \right]_{q^{ 2 }}} {\left[ 5 m+3 \right]_{q^{ 6
  }}} \\
\times
{\left[ \tfrac{3m+2}{2}  \right]_{q^{ 20 }}} 
 {\left[ \tfrac{5m+4}{4}  \right]_{q^{ 24 }}} 
 {\left[ m+1 \right]_{q^{ 30 }}}
,\end{multline*} 
which, by Corollary~\ref{cor:A}, is a polynomial in $q$ with 
non-negative integer coefficients.
If $m\equiv
21
~(\text{mod }84),$ then we have \begin{multline*} \Cat^m(E_8;q)=
{\left[ \tfrac{15m+1}{4}  \right]_{q^{ 8 }}}
 {\left[ 15m+4 \right]_{q^{ 2 }}} {\left[ 
  {5m+2}  \right]_{q^{ 6 }}}
 {\left[ \tfrac{15 m+7}{14} 
  \right]_{q^{28}}}
\frac {[14]_{q^2}} {[2]_{q^2}\left[7\right]_{q^2}}
 {\left[ \tfrac{5m+3}{12}  
   \right]_{q^{ 72}}}
\frac {[12]_{q^6}} {[3]_{q^6}\left[4\right]_{q^6}} 
   \\
\times {\left[ 3m+2 \right]_{q^{10}}} 
 {\left[ {5m+4}  \right]_{q^{ 6 }}}
 {\left[ \tfrac{m+1}{2} \right]_{q^{ 60}}} \frac {{ 
  {\left[ 30 \right]_{q^{ 2}}}
  {\left[ 2 \right]_{q^{ 2}}}
  {\left[ 3 \right]_{q^{ 2}}}
  {\left[ 5 \right]_{q^{ 2}}} }} { {\left[ 6 
   \right]_{q^{ 2}}}{\left[ 10
   \right]_{q^{ 2}}} {\left[ 15 \right]_{q^{ 2}}}}
,\end{multline*} 
which, by Corollary~\ref{cor:A} and Lemma~\ref{lem:61015B}, 
is a polynomial in $q$ with 
non-negative integer coefficients.
If $m\equiv
22
~(\text{mod }84),$ then we have \begin{multline*} \Cat^m(E_8;q)=
{\left[ 15 m+1 \right]_{q^{ 2 }}} {\left[ \tfrac{15m+4}2 
  \right]_{q^4}} {\left[ \tfrac{5m+2}{28}  
   \right]_{q^{ 168}}} \frac {{ {\left[ 84 \right]_{q^{ 2}}} 
  }} { {\left[ 7 \right]_{q^{ 2}}} {\left[ 12 
   \right]_{q^{ 2}}}} {\left[ 15 m+7 \right]_{q^2}} \\
\times
 {\left[ 5 m+3 \right]_{q^6}} {\left[ 
   \tfrac{3m+2}{4}  \right]_{q^{ 40}}} \frac {{ {\left[ 10 
   \right]_{q^{ 4}}} }} { {\left[ 2 \right]_{q^{ 4}}} 
  {\left[ 5 \right]_{q^{ 4}}}} {\left[ \tfrac{5m+4}{6}  
   \right]_{q^{ 36}}} \frac {{ {\left[ 6 \right]_{q^{ 6}}} 
  }} { {\left[ 2 \right]_{q^{ 6}}} {\left[ 3 
   \right]_{q^{ 6}}}} {\left[ m+1 \right]_{q^{ 30 }}}
,\end{multline*} 
which, by Corollary~\ref{cor:A}, is a polynomial in $q$ with 
non-negative integer coefficients.
If $m\equiv
23
~(\text{mod }84),$ then we have 
\begin{multline*} \Cat^m(E_8;q)=
{\left[ \tfrac{15 m+1}2 \right]_{q^{ 4}}} 
{\left[ {15m+4} \right]_{q^{ 
   2}}} 
 {\left[ \tfrac{5m+2}3 \right]_{q^{ 18 }}}
 {\left[ \tfrac{15m+7}4 
   \right]_{q^{ 8}}} 
 {\left[ \tfrac{5m+3}2 \right]_{q^{ 12}}}\\
\times 
 {\left[ 3m+2 \right]_{q^{ 10 }}} {\left[ \tfrac{5m+4}7
   \right]_{q^{ 42 }}}
\frac {[21]_{q^2}} {[3]_{q^2}\left[7\right]_{q^2}}
 {\left[ \tfrac{m+1}4 \right]_{q^{ 120}}}
\frac {[60]_{q^2}\left[2\right]_{q^2}\left[3\right]_{q^2}\left[5\right]_{q^2}} 
{\left[10\right]_{q^2}\left[12\right]_{q^2}\left[15\right]_{q^2}} 
,\end{multline*} 
which, by Corollary~\ref{cor:A} and Lemma~\ref{lem:101215}, 
is a polynomial in $q$ with 
non-negative integer coefficients.
If $m\equiv
24
~(\text{mod }84),$ then we have \begin{multline*} \Cat^m(E_8;q)=
{\left[ 15 m+1 \right]_{q^{ 2 }}} {\left[ 
   \tfrac{15m+4}{28}  \right]_{q^{ 56}}} \frac {{ {\left[ 28
    \right]_{q^{ 2}}} }} { {\left[ 4 \right]_{q^{ 2}}} 
  {\left[ 7 \right]_{q^{ 2}}}} {\left[ \tfrac{5m+2}{2}  
  \right]_{q^{ 12 }}} {\left[ 15 m+7 \right]_{q^{ 2
  }}}\\
\times {\left[ \tfrac{5m+3}{3}  \right]_{q^{ 18 }}} 
 {\left[ \tfrac{3m+2}{2}  \right]_{q^{ 20 }}} 
 {\left[ \tfrac{5m+4}{4}  \right]_{q^{ 24 }}} 
 {\left[ m+1 \right]_{q^{ 30 }}}
,\end{multline*} 
which, by Corollary~\ref{cor:A}, is a polynomial in $q$ with 
non-negative integer coefficients.
If $m\equiv
25
~(\text{mod }84),$ then we have \begin{multline*} \Cat^m(E_8;q)=
{\left[ \tfrac{15m+1}{4}  \right]_{q^{ 8 }}}
 {\left[ 15m+4 \right]_{q^{ 2 }}} {\left[ 
  {5m+2}  \right]_{q^{ 6 }}}
 {\left[ \tfrac{15 m+7}2 
  \right]_{q^4}} {\left[ \tfrac{5m+3}{4}  
   \right]_{q^{ 24}}} 
   \\
\times {\left[ \tfrac{3m+2}7 \right]_{q^{70}}} 
\frac {[35]_{q^2}} {[5]_{q^2}\left[7\right]_{q^2}}
 {\left[ \tfrac{5m+4}3  \right]_{q^{ 18 }}}
 {\left[ \tfrac{m+1}{2} \right]_{q^{ 60}}} \frac {{ 
  {\left[ 30 \right]_{q^{ 2}}}
  {\left[ 2 \right]_{q^{ 2}}}
  {\left[ 3 \right]_{q^{ 2}}}
  {\left[ 5 \right]_{q^{ 2}}} }} { {\left[ 6 
   \right]_{q^{ 2}}}{\left[ 10
   \right]_{q^{ 2}}} {\left[ 15 \right]_{q^{ 2}}}}
,\end{multline*} 
which, by Lemma~\ref{lem:61015C}, 
is a polynomial in $q$ with 
non-negative integer coefficients.
If $m\equiv
26
~(\text{mod }84),$ then we have \begin{multline*} \Cat^m(E_8;q)=
{\left[ 15 m+1 \right]_{q^{ 2 }}} {\left[ \tfrac{15m+4}2 
  \right]_{q^4}} {\left[ \tfrac{5m+2}{12}  
   \right]_{q^{ 72}}} \frac {{ {\left[ 12 \right]_{q^{ 6}}} 
  }} { {\left[ 3 \right]_{q^{ 6}}} {\left[ 4 
   \right]_{q^{ 6}}}} {\left[ 15 m+7 \right]_{q^2}} 
 {\left[ \tfrac{5m+3}{7}  \right]_{q^{ 42 }}}
\frac {[21]_{q^2}} {[3]_{q^2}\left[7\right]_{q^2}}
 \\
\times
 {\left[ \tfrac{3m+2}{4}  \right]_{q^{ 40}}} \frac {{ 
  {\left[ 10 \right]_{q^{ 4}}} }} { {\left[ 2 
   \right]_{q^{ 4}}} {\left[ 5 \right]_{q^{ 4}}}} {\left[ 
  \tfrac{5m+4}{2}  \right]_{q^{ 12 }}} {\left[ m+1 
  \right]_{q^{ 30 }}}
.\end{multline*} 
If one decomposes $[15m+1]_{q^2}$ as 
$[5m+1]_{q^6}+q^2[5m]_{q^6}+q^4[5m]_{q^6}$,
then one sees that, by Corollary~\ref{cor:A},
this is a polynomial in $q$ with 
non-negative integer coefficients.
If $m\equiv
27
~(\text{mod }84),$ then we have 
\begin{multline*} \Cat^m(E_8;q)=
{\left[ \tfrac{15 m+1}{14} \right]_{q^{ 28}}}
\frac {[14]_{q^2}} {[2]_{q^2}\left[7\right]_{q^2}} 
{\left[ {15m+4} \right]_{q^{ 
   2}}} 
 {\left[ 5m+2 \right]_{q^{ 6 }}}
 {\left[ \tfrac{15m+7}4 
   \right]_{q^{ 8}}} 
 {\left[ \tfrac{5m+3}6 \right]_{q^{ 36}}}
\frac {[6]_{q^6}} {[2]_{q^6}\left[3\right]_{q^6}}
\\
\times 
 {\left[ 3m+2 \right]_{q^{ 10 }}} {\left[ 5m+4
   \right]_{q^{ 6 }}} {\left[ \tfrac{m+1}4 \right]_{q^{ 120}}}
\frac {[60]_{q^2}\left[2\right]_{q^2}\left[3\right]_{q^2}\left[5\right]_{q^2}} 
{\left[10\right]_{q^2}\left[12\right]_{q^2}\left[15\right]_{q^2}} 
,\end{multline*} 
which, by Corollary~\ref{cor:A} and Lemma~\ref{lem:101215C}, 
is a polynomial in $q$ with 
non-negative integer coefficients.
If $m\equiv
28
~(\text{mod }84),$ then we have \begin{multline*} \Cat^m(E_8;q)=
{\left[ 15 m+1 \right]_{q^{ 2 }}} {\left[ 
  \tfrac{15m+4}{4}  \right]_{q^{ 8 }}} {\left[ 
  \tfrac{5m+2}{2}  \right]_{q^{ 12 }}} {\left[ 
  \tfrac{15m+7}{7}  \right]_{q^{ 14 }}} \\
\times{\left[ 5 m+3
   \right]_{q^6}} {\left[ \tfrac{3m+2}{2}  
  \right]_{q^{ 20 }}}
 {\left[ \tfrac{5m+4}{12}  
   \right]_{q^{ 72}}} \frac {{ {\left[ 12 \right]_{q^{ 6}}} 
  }} { {\left[ 3 \right]_{q^{ 6}}} {\left[ 4 
   \right]_{q^{ 6}}}} {\left[ m+1 \right]_{q^{ 30 }}}
,\end{multline*} 
which, by Corollary~\ref{cor:A}, is a polynomial in $q$ with 
non-negative integer coefficients.
If $m\equiv
29
~(\text{mod }84),$ then we have \begin{multline*} \Cat^m(E_8;q)=
{\left[ \tfrac{15m+1}{4}  \right]_{q^{ 8 }}}
 {\left[ 15m+4 \right]_{q^{ 2 }}} {\left[ 
  \tfrac{5m+2}{21}  \right]_{q^{ 126 }}}
\frac {[63]_{q^2}} {[7]_{q^2}\left[9\right]_{q^2}}
 {\left[ \tfrac{15 m+7}2 
  \right]_{q^4}} {\left[ \tfrac{5m+3}{4}  
   \right]_{q^{ 24}}} 
   \\
\times {\left[ 3m+2 \right]_{q^{10}}} 
 {\left[ {5m+4}  \right]_{q^{ 6 }}}
 {\left[ \tfrac{m+1}{2} \right]_{q^{ 60}}} \frac {{ 
  {\left[ 30 \right]_{q^{ 2}}}
  {\left[ 2 \right]_{q^{ 2}}}
  {\left[ 3 \right]_{q^{ 2}}}
  {\left[ 5 \right]_{q^{ 2}}} }} { {\left[ 6 
   \right]_{q^{ 2}}}{\left[ 10
   \right]_{q^{ 2}}} {\left[ 15 \right]_{q^{ 2}}}}
,\end{multline*} 
which, by Corollary~\ref{cor:A} and Lemma~\ref{lem:61015}, 
is a polynomial in $q$ with 
non-negative integer coefficients.
If $m\equiv
30
~(\text{mod }84),$ then we have \begin{multline*} \Cat^m(E_8;q)=
{\left[ 15 m+1 \right]_{q^{ 2 }}} {\left[ \tfrac{15m+4}2 
  \right]_{q^4}} {\left[ \tfrac{5m+2}{4}  
   \right]_{q^{ 24}}} 
 {\left[ 15 m+7 \right]_{q^2}} 
 {\left[ \tfrac{5m+3}{3}  \right]_{q^{ 18 }}} \\
 {\left[ \tfrac{3m+2}{4}  \right]_{q^{ 40}}} \frac {{ 
  {\left[ 10 \right]_{q^{ 4}}} }} { {\left[ 2 
   \right]_{q^{ 4}}} {\left[ 5 \right]_{q^{ 4}}}} {\left[ 
  \tfrac{5m+4}{14}  \right]_{q^{ 84 }}}
\frac {[42]_{q^2}} {[6]_{q^2}\left[7\right]_{q^2}}
 {\left[ m+1 
  \right]_{q^{ 30 }}}
,\end{multline*} 
which, by Corollary~\ref{cor:A}, 
is a polynomial in $q$ with 
non-negative integer coefficients.
If $m\equiv
31
~(\text{mod }84),$ then we have 
\begin{multline*} \Cat^m(E_8;q)=
{\left[ \tfrac{15 m+1}2 \right]_{q^{ 4}}} 
{\left[ \tfrac{15m+4}7 \right]_{q^{ 
   14}}} 
 {\left[ 5m+2 \right]_{q^{ 6 }}}
 {\left[ \tfrac{15m+7}4 
   \right]_{q^{ 8}}} 
 {\left[ \tfrac{5m+3}2 \right]_{q^{ 12}}}\\
\times 
 {\left[ 3m+2 \right]_{q^{ 10 }}} {\left[ \tfrac{5m+4}3
   \right]_{q^{ 18 }}} {\left[ \tfrac{m+1}4 \right]_{q^{ 120}}}
\frac {[60]_{q^2}\left[2\right]_{q^2}\left[3\right]_{q^2}\left[5\right]_{q^2}} 
{\left[10\right]_{q^2}\left[12\right]_{q^2}\left[15\right]_{q^2}} 
,\end{multline*} 
which, by Lemma~\ref{lem:101215}, 
is a polynomial in $q$ with 
non-negative integer coefficients.
If $m\equiv
32
~(\text{mod }84),$ then we have \begin{multline*} \Cat^m(E_8;q)=
{\left[ 15 m+1 \right]_{q^{ 2 }}} {\left[ 
  \tfrac{15m+4}{4}  \right]_{q^{ 8 }}} {\left[ 
   \tfrac{5m+2}{6}  \right]_{q^{ 36}}} \frac {{ {\left[ 6 
   \right]_{q^{ 6}}} }} { {\left[ 2 \right]_{q^{ 6}}} 
  {\left[ 3 \right]_{q^{ 6}}}} {\left[ 15 m+7 \right]_{q^{ 
  2 }}} {\left[ 5 m+3 \right]_{q^6}} \\
\times
 {\left[ \tfrac{3m+2}{14}  \right]_{q^{ 140}}} \frac {{ 
  {\left[ 70 \right]_{q^{ 2}}} }} { {\left[ 7 
   \right]_{q^{ 2}}} {\left[ 10 \right]_{q^{ 2}}}} {\left[ 
  \tfrac{5m+4}{4}  \right]_{q^{ 24 }}} {\left[ m+1 
  \right]_{q^{ 30 }}}
,\end{multline*} 
which, by Corollary~\ref{cor:A}, is a polynomial in $q$ with 
non-negative integer coefficients.
If $m\equiv
33
~(\text{mod }84),$ then we have \begin{multline*} \Cat^m(E_8;q)=
{\left[ \tfrac{15m+1}{4}  \right]_{q^{ 8 }}}
 {\left[ 15m+4 \right]_{q^{ 2 }}} {\left[ 
  {5m+2}  \right]_{q^{ 6 }}}
 {\left[ \tfrac{15 m+7}2 
  \right]_{q^4}} {\left[ \tfrac{5m+3}{84}  
   \right]_{q^{ 504}}}
\frac {[252]_{q^2}\left[3\right]_{q^2}} 
{[7]_{q^2}\left[9\right]_{q^2}\left[12\right]_{q^2}}
   \\
\times {\left[ 3m+2 \right]_{q^{10}}} 
 {\left[ {5m+4}  \right]_{q^{ 6 }}}
 {\left[ \tfrac{m+1}{2} \right]_{q^{ 60}}} \frac {{ 
  {\left[ 30 \right]_{q^{ 2}}}
  {\left[ 2 \right]_{q^{ 2}}}
  {\left[ 3 \right]_{q^{ 2}}}
  {\left[ 5 \right]_{q^{ 2}}} }} { {\left[ 6 
   \right]_{q^{ 2}}}{\left[ 10
   \right]_{q^{ 2}}} {\left[ 15 \right]_{q^{ 2}}}}
,\end{multline*} 
which, by Lemmas~\ref{lem:61015} and \ref{lem:7912}, 
is a polynomial in $q$ with 
non-negative integer coefficients.
If $m\equiv
34
~(\text{mod }84),$ then we have \begin{multline*} \Cat^m(E_8;q)=
{\left[ \tfrac{15m+1}{7}  \right]_{q^{ 14 }}} 
 {\left[ \tfrac{15m+4}2 \right]_{q^{ 4 }}} {\left[ 
   \tfrac{5m+2}{4}  \right]_{q^{ 24}}} 
{\left[ 15 m+7 
  \right]_{q^2}}\\
 {\left[ 5 m+3 \right]_{q^{ 6 
  }}} {\left[ \tfrac{3m+2}{4}  \right]_{q^{ 40}}} 
  \frac {{ {\left[ 10 \right]_{q^{ 4}}} }} { 
  {\left[ 2 \right]_{q^{ 4}}} {\left[ 5 \right]_{q^{ 4}}}} 
 {\left[ \tfrac{5m+4}{6}  \right]_{q^{ 36 }}}
\frac {[6]_{q^6}} {[2]_{q^6}\left[3\right]_{q^6}} 
 {\left[ m+1 \right]_{q^{ 30 }}}
,\end{multline*} 
which, by Corollary~\ref{cor:A}, 
is a polynomial in $q$ with 
non-negative integer coefficients.
If $m\equiv
35
~(\text{mod }84),$ then we have 
\begin{multline*} \Cat^m(E_8;q)=
{\left[ \tfrac{15 m+1}2 \right]_{q^{ 4}}} 
{\left[ {15m+4} \right]_{q^{ 
   2}}} 
 {\left[ \tfrac{5m+2}3 \right]_{q^{ 18 }}}
 {\left[ \tfrac{15m+7}{28} 
   \right]_{q^{ 56}}}
\frac {[28]_{q^2}} {[4]_{q^2}\left[7\right]_{q^2}} 
 {\left[ \tfrac{5m+3}2 \right]_{q^{ 12}}}\\
\times 
 {\left[ 3m+2 \right]_{q^{ 10 }}} {\left[ 5m+4
   \right]_{q^{ 6 }}} {\left[ \tfrac{m+1}4 \right]_{q^{ 120}}}
\frac {[60]_{q^2}\left[2\right]_{q^2}\left[3\right]_{q^2}\left[5\right]_{q^2}} 
{\left[10\right]_{q^2}\left[12\right]_{q^2}\left[15\right]_{q^2}} 
,\end{multline*} 
which, by Corollary~\ref{cor:A} and Lemma~\ref{lem:101215}, 
is a polynomial in $q$ with 
non-negative integer coefficients.
If $m\equiv
36
~(\text{mod }84),$ then we have \begin{multline*} \Cat^m(E_8;q)=
{\left[ 15 m+1 \right]_{q^{ 2 }}} {\left[ 
  \tfrac{15m+4}{4}  \right]_{q^{ 8 }}} {\left[ 
   \tfrac{5m+2}{14}  \right]_{q^{ 84}}} \frac {{ {\left[ 42
    \right]_{q^{ 2}}} }} { {\left[ 6 \right]_{q^{ 2}}} 
  {\left[ 7 \right]_{q^{ 2}}}} {\left[ 15 m+7 \right]_{q^{ 
  2 }}} {\left[ \tfrac{5m+3}{3}  \right]_{q^{ 18 
  }}}\\
\times
 {\left[ \tfrac{3m+2}{2}  \right]_{q^{ 20 }}} 
 {\left[ \tfrac{5m+4}{4}  \right]_{q^{ 24 }}} 
 {\left[ m+1 \right]_{q^{ 30 }}}
,\end{multline*} 
which, by Corollary~\ref{cor:A}, is a polynomial in $q$ with 
non-negative integer coefficients.
If $m\equiv
37
~(\text{mod }84),$ then we have \begin{multline*} \Cat^m(E_8;q)=
{\left[ \tfrac{15m+1}{4}  \right]_{q^{ 8 }}}
 {\left[ 15m+4 \right]_{q^{ 2 }}} {\left[ 
  {5m+2}  \right]_{q^{ 6 }}}
 {\left[ \tfrac{15 m+7}2 
  \right]_{q^4}} {\left[ \tfrac{5m+3}{4}  
   \right]_{q^{ 24}}} 
   \\
\times {\left[ 3m+2 \right]_{q^{10}}} 
 {\left[ \frac{5m+4}{21}  \right]_{q^{ 126 }}}
\frac {[63]_{q^2}} {[7]_{q^2}\left[9\right]_{q^2}}
 {\left[ \tfrac{m+1}{2} \right]_{q^{ 60}}} \frac {{ 
  {\left[ 30 \right]_{q^{ 2}}}
  {\left[ 2 \right]_{q^{ 2}}}
  {\left[ 3 \right]_{q^{ 2}}}
  {\left[ 5 \right]_{q^{ 2}}} }} { {\left[ 6 
   \right]_{q^{ 2}}}{\left[ 10
   \right]_{q^{ 2}}} {\left[ 15 \right]_{q^{ 2}}}}
,\end{multline*} 
which, by Corollary~\ref{cor:A} and Lemma~\ref{lem:61015}, 
is a polynomial in $q$ with 
non-negative integer coefficients.
If $m\equiv
38
~(\text{mod }84),$ then we have \begin{multline*} \Cat^m(E_8;q)=
{\left[ 15 m+1 \right]_{q^{ 2 }}} {\left[ 
  \tfrac{15m+4}{14}  \right]_{q^{ 28 }}}
\frac {[14]_{q^2}} {[2]_{q^2}\left[7\right]_{q^2}}
 {\left[ 
   \tfrac{5m+2}{12}  \right]_{q^{ 72}}} \frac {{ {\left[ 12
    \right]_{q^{ 6}}} }} { {\left[ 3 \right]_{q^{ 6}}} 
  {\left[ 4 \right]_{q^{ 6}}}} \\
\times
{\left[ 15 m+7 \right]_{q^{ 
  2 }}} {\left[ 5 m+3 \right]_{q^6}} 
 {\left[ \tfrac{3m+2}{4}  \right]_{q^{ 40}}} \frac {{ 
  {\left[ 10 \right]_{q^{ 4}}} }} { {\left[ 2 
   \right]_{q^{ 4}}} {\left[ 5 \right]_{q^{ 4}}}} {\left[ 
  \tfrac{5m+4}{2}  \right]_{q^{ 12 }}} {\left[ m+1 
  \right]_{q^{ 30 }}}
.\end{multline*} 
If one decomposes $[15m+7]_{q^2}$ as 
$[\frac {15m} {2}+4]_{q^4}+q^2[\frac {15m} {2}+3]_{q^4}$,
then one sees that, by Corollary~\ref{cor:A},
this is a polynomial in $q$ with 
non-negative integer coefficients.
If $m\equiv
39
~(\text{mod }84),$ then we have 
\begin{multline*} \Cat^m(E_8;q)=
{\left[ \tfrac{15 m+1}2 \right]_{q^{ 4}}} 
{\left[ {15m+4} \right]_{q^{ 
   2}}} 
 {\left[ 5m+2 \right]_{q^{ 6 }}}
 {\left[ \tfrac{15m+7}4 
   \right]_{q^{ 8}}} 
 {\left[ \tfrac{5m+3}6 \right]_{q^{ 36}}}
\frac {[6]_{q^6}} {[2]_{q^6}\left[3\right]_{q^6}}
\\
\times 
 {\left[ \tfrac{3m+2}7 \right]_{q^{ 70 }}}
\frac {[35]_{q^2}} {[5]_{q^2}\left[7\right]_{q^2}}
 {\left[ 5m+4
   \right]_{q^{ 6 }}} {\left[ \tfrac{m+1}4 \right]_{q^{ 120}}}
\frac {[60]_{q^2}\left[2\right]_{q^2}\left[3\right]_{q^2}\left[5\right]_{q^2}} 
{\left[10\right]_{q^2}\left[12\right]_{q^2}\left[15\right]_{q^2}} 
,\end{multline*} 
which, by Corollary~\ref{cor:A} and Lemma~\ref{lem:101215B}, 
is a polynomial in $q$ with 
non-negative integer coefficients.
If $m\equiv
40
~(\text{mod }84),$ then we have \begin{multline*} \Cat^m(E_8;q)=
{\left[ 15 m+1 \right]_{q^{ 2 }}} {\left[ 
  \tfrac{15m+4}{4}  \right]_{q^{ 8 }}} {\left[ 
  \tfrac{5m+2}{2}  \right]_{q^{ 12 }}} {\left[ 15 m+7 
  \right]_{q^2}} {\left[ \tfrac{5m+3}{7}  
  \right]_{q^{ 42 }}}
\frac {[21]_{q^2}} {[3]_{q^2}\left[7\right]_{q^2}}\\
\times
 {\left[ \tfrac{3m+2}{2}  
  \right]_{q^{ 20 }}}
 {\left[ \tfrac{5m+4}{12}  
   \right]_{q^{ 72}}} \frac {{ {\left[ 12 \right]_{q^{ 6}}} 
  }} { {\left[ 3 \right]_{q^{ 6}}} {\left[ 4 
   \right]_{q^{ 6}}}} {\left[ m+1 \right]_{q^{ 30 }}}
.\end{multline*} 
If one decomposes $[15m+7]_{q^2}$ as 
$[5m+1]_{q^6}+q^2[5m]_{q^6}+q^4[5m]_{q^6}$,
then one sees that, by Corollary~\ref{cor:A},
this is a polynomial in $q$ with 
non-negative integer coefficients.
If $m\equiv
41
~(\text{mod }84),$ then we have \begin{multline*} \Cat^m(E_8;q)=
{\left[ \tfrac{15m+1}{28}  \right]_{q^{ 56 }}}
\frac {[28]_{q^2}} {[4]_{q^2}\left[7\right]_{q^2}}
 {\left[ 15m+4 \right]_{q^{ 2 }}} {\left[ 
  \tfrac{5m+2}3  \right]_{q^{ 18 }}}
 {\left[ \tfrac{15 m+7}2 
  \right]_{q^4}} {\left[ \tfrac{5m+3}{4}  
   \right]_{q^{ 24}}} 
   \\
\times {\left[ 3m+2 \right]_{q^{10}}} 
 {\left[ {5m+4}  \right]_{q^{ 6 }}}
 {\left[ \tfrac{m+1}{2} \right]_{q^{ 60}}} \frac {{ 
  {\left[ 30 \right]_{q^{ 2}}}
  {\left[ 2 \right]_{q^{ 2}}}
  {\left[ 3 \right]_{q^{ 2}}}
  {\left[ 5 \right]_{q^{ 2}}} }} { {\left[ 6 
   \right]_{q^{ 2}}}{\left[ 10
   \right]_{q^{ 2}}} {\left[ 15 \right]_{q^{ 2}}}}
,\end{multline*} 
which, by Corollary~\ref{cor:A} and Lemma~\ref{lem:61015}, 
is a polynomial in $q$ with 
non-negative integer coefficients.
If $m\equiv
42
~(\text{mod }84),$ then we have \begin{multline*} \Cat^m(E_8;q)=
{\left[ 15 m+1 \right]_{q^{ 2 }}} {\left[ \tfrac{15m+4 }2
  \right]_{q^4}} {\left[ \tfrac{5m+2}{4}  
   \right]_{q^{ 24}}}
 {\left[ \tfrac{15m+7}{7}  \right]_{q^{ 14 
  }}}\\
\times
 {\left[ \tfrac{5m+3}{3}  \right]_{q^{ 18 }}} 
 {\left[ \tfrac{3m+2}{4}  \right]_{q^{ 40 }}} 
\frac {[10]_{q^4}} {[2]_{q^4}\left[5\right]_{q^4}}
 {\left[ \tfrac{5m+4}{2}  \right]_{q^{ 12 }}} 
 {\left[ m+1 \right]_{q^{ 30 }}}
,\end{multline*} 
which, by Corollary~\ref{cor:A}, is a polynomial in $q$ with 
non-negative integer coefficients.
If $m\equiv
43
~(\text{mod }84),$ then we have 
\begin{multline*} \Cat^m(E_8;q)=
{\left[ \tfrac{15 m+1}2 \right]_{q^{ 4}}} 
{\left[ {15m+4} \right]_{q^{ 
   2}}} 
 {\left[ \tfrac{5m+2}7 \right]_{q^{ 42 }}}
\frac {[21]_{q^2}} {[3]_{q^2}\left[7\right]_{q^2}}
 {\left[ \tfrac{15m+7}4 
   \right]_{q^{ 8}}} 
 {\left[ \tfrac{5m+3}2 \right]_{q^{ 12}}}\\
\times 
 {\left[ 3m+2 \right]_{q^{ 10 }}} {\left[ \tfrac{5m+4}3
   \right]_{q^{ 18 }}} {\left[ \tfrac{m+1}4 \right]_{q^{ 120}}}
\frac {[60]_{q^2}\left[2\right]_{q^2}\left[3\right]_{q^2}\left[5\right]_{q^2}} 
{\left[10\right]_{q^2}\left[12\right]_{q^2}\left[15\right]_{q^2}} 
,\end{multline*} 
which, by Corollary~\ref{cor:A} and Lemma~\ref{lem:101215}, 
is a polynomial in $q$ with 
non-negative integer coefficients.
If $m\equiv
44
~(\text{mod }84),$ then we have \begin{multline*} \Cat^m(E_8;q)=
{\left[ 15 m+1 \right]_{q^{ 2 }}} {\left[ 
  \tfrac{15m+4}{4}  \right]_{q^{ 8 }}} {\left[ 
   \tfrac{5m+2}{6}  \right]_{q^{ 36}}} \frac {{ {\left[ 6 
   \right]_{q^{ 6}}} }} { {\left[ 2 \right]_{q^{ 6}}} 
  {\left[ 3 \right]_{q^{ 6}}}} {\left[ 15 m+7 \right]_{q^{ 
  2 }}}
 {\left[ 5 m+3 \right]_{q^6}}\\
\times 
 {\left[ \tfrac{3m+2}{2}  \right]_{q^{ 20 }}} 
 {\left[ \tfrac{5m+4}{28}  \right]_{q^{ 168}}} \frac {{ 
  {\left[ 84 \right]_{q^{ 2}}} }} { {\left[ 7 
   \right]_{q^{ 2}}} {\left[ 12 \right]_{q^{ 2}}}} {\left[ m+1 
  \right]_{q^{ 30 }}}
,\end{multline*} 
which, by Corollary~\ref{cor:A}, is a polynomial in $q$ with 
non-negative integer coefficients.
If $m\equiv
45
~(\text{mod }84),$ then we have \begin{multline*} \Cat^m(E_8;q)=
{\left[ \tfrac{15m+1}{4}  \right]_{q^{ 8 }}}
 {\left[ \tfrac{15m+4}7 \right]_{q^{ 14 }}} {\left[ 
  {5m+2}  \right]_{q^{ 6 }}}
 {\left[ \tfrac{15 m+7}2 
  \right]_{q^4}} {\left[ \tfrac{5m+3}{12}  
   \right]_{q^{ 72}}}
\frac {[12]_{q^6}} {[3]_{q^6}\left[4\right]_{q^6}} 
   \\
\times {\left[ 3m+2 \right]_{q^{10}}} 
 {\left[ {5m+4}  \right]_{q^{ 6 }}}
 {\left[ \tfrac{m+1}{2} \right]_{q^{ 60}}} \frac {{ 
  {\left[ 30 \right]_{q^{ 2}}}
  {\left[ 2 \right]_{q^{ 2}}}
  {\left[ 3 \right]_{q^{ 2}}}
  {\left[ 5 \right]_{q^{ 2}}} }} { {\left[ 6 
   \right]_{q^{ 2}}}{\left[ 10
   \right]_{q^{ 2}}} {\left[ 15 \right]_{q^{ 2}}}}
,\end{multline*} 
which, by Corollary~\ref{cor:A} and Lemma~\ref{lem:61015}, 
is a polynomial in $q$ with 
non-negative integer coefficients.
If $m\equiv
46
~(\text{mod }84),$ then we have \begin{multline*} \Cat^m(E_8;q)=
{\left[ 15 m+1 \right]_{q^{ 2 }}} {\left[ \tfrac{15m+4}2 
  \right]_{q^4}} {\left[ \tfrac{5m+2}{4}  
   \right]_{q^{ 24}}} 
{\left[ 15 m+7 \right]_{q^2}} \\
 {\left[ 5 m+3 \right]_{q^6}} {\left[ 
   \tfrac{3m+2}{28}  \right]_{q^{ 280}}} 
\frac {{ {\left[ 140
    \right]_{q^{ 2}}}{\left[ 2
    \right]_{q^{ 2}}} }} { {\left[ 4
    \right]_{q^{ 2}}}{\left[ 7 \right]_{q^{ 2}}} 
  {\left[ 10 \right]_{q^{ 2}}}} {\left[ \tfrac{5m+4}{6}  
   \right]_{q^{ 36}}} \frac {{ {\left[ 6 \right]_{q^{ 6}}} 
  }} { {\left[ 2 \right]_{q^{ 6}}} {\left[ 3 
   \right]_{q^{ 6}}}} {\left[ m+1 \right]_{q^{ 30 }}}
,\end{multline*} 
which, by Corollary~\ref{cor:A} and Lemma~\ref{lem:4710}, 
is a polynomial in $q$ with 
non-negative integer coefficients.
If $m\equiv
47
~(\text{mod }84),$ then we have 
\begin{multline*} \Cat^m(E_8;q)=
{\left[ \tfrac{15 m+1}2 \right]_{q^{ 4}}} 
{\left[ {15m+4} \right]_{q^{ 
   2}}} 
 {\left[ \tfrac{5m+2}3 \right]_{q^{ 18 }}}
 {\left[ \tfrac{15m+7}4 
   \right]_{q^{ 8}}} 
 {\left[ \tfrac{5m+3}{14} \right]_{q^{ 84}}}
\frac {[42]_{q^2}} {[6]_{q^2}\left[7\right]_{q^2}}
\\
\times 
 {\left[ 3m+2 \right]_{q^{ 10 }}} {\left[ 5m+4
   \right]_{q^{ 6 }}} {\left[ \tfrac{m+1}4 \right]_{q^{ 120}}}
\frac {[60]_{q^2}\left[2\right]_{q^2}\left[3\right]_{q^2}\left[5\right]_{q^2}} 
{\left[10\right]_{q^2}\left[12\right]_{q^2}\left[15\right]_{q^2}} 
,\end{multline*} 
which, by Corollary~\ref{cor:A} and Lemma~\ref{lem:101215}, 
is a polynomial in $q$ with 
non-negative integer coefficients.
If $m\equiv
48
~(\text{mod }84),$ then we have \begin{multline*} \Cat^m(E_8;q)=
{\left[ \tfrac{15m+1}{7}  \right]_{q^{ 14}}} {\left[ 
  \tfrac{15m+4}{4}  \right]_{q^{ 8}}} {\left[ \tfrac{5m+2}{2}  
  \right]_{q^{ 12}}} {\left[ 15m+7 \right]_{q^{ 2}}} \\
{\left[ 
  \tfrac{5m+3}{3}  \right]_{q^{ 18}}} {\left[ \tfrac{3m+2}{2}  
  \right]_{q^{ 20}}} {\left[ \tfrac{5m+4}{4}  \right]_{q^{ 24}}} 
 {\left[ m+1 \right]_{q^{ 30}}}
,\end{multline*} 
which is manifestly a polynomial in $q$ with 
non-negative integer coefficients.  
If $m\equiv
49
~(\text{mod }84),$ then we have \begin{multline*} \Cat^m(E_8;q)=
{\left[ \tfrac{15m+1}{4}  \right]_{q^{ 8 }}}
 {\left[ 15m+4 \right]_{q^{ 2 }}} {\left[ 
  {5m+2}  \right]_{q^{ 6 }}}
 {\left[ \tfrac{15 m+7}{14} 
  \right]_{q^{28}}}
\frac {[14]_{q^2}} {[2]_{q^2}\left[7\right]_{q^2}}
 {\left[ \tfrac{5m+3}{4}  
   \right]_{q^{ 24}}} 
   \\
\times {\left[ 3m+2 \right]_{q^{10}}} 
 {\left[ \tfrac{5m+4}3  \right]_{q^{ 18 }}}
 {\left[ \tfrac{m+1}{2} \right]_{q^{ 60}}} \frac {{ 
  {\left[ 30 \right]_{q^{ 2}}}
  {\left[ 2 \right]_{q^{ 2}}}
  {\left[ 3 \right]_{q^{ 2}}}
  {\left[ 5 \right]_{q^{ 2}}} }} { {\left[ 6 
   \right]_{q^{ 2}}}{\left[ 10
   \right]_{q^{ 2}}} {\left[ 15 \right]_{q^{ 2}}}}
,\end{multline*} 
which, by Corollary~\ref{cor:A} and Lemma~\ref{lem:61015}, 
is a polynomial in $q$ with 
non-negative integer coefficients.
If $m\equiv
50
~(\text{mod }84),$ then we have \begin{multline*} \Cat^m(E_8;q)=
{\left[ 15 m+1 \right]_{q^{ 2 }}} {\left[ \tfrac{15m+4}2 
  \right]_{q^4}} {\left[ \tfrac{5m+2}{28}  
   \right]_{q^{ 168}}} \frac {{ {\left[ 84 \right]_{q^{ 2}}} 
  }} { {\left[ 7 \right]_{q^{ 2}}} {\left[ 12 
   \right]_{q^{ 2}}}} {\left[ 15 m+7 \right]_{q^2}} \\
\times
 {\left[ 5 m+3 \right]_{q^6}} {\left[ 
   \tfrac{3m+2}{4}  \right]_{q^{ 40}}} \frac {{ {\left[ 10 
   \right]_{q^{ 4}}} }} { {\left[ 2 \right]_{q^{ 4}}} 
  {\left[ 5 \right]_{q^{ 4}}}} {\left[ \tfrac{5m+4}{2}  
  \right]_{q^{ 12 }}} {\left[ \tfrac{m+1}{3} \right]_{q^{ 
   90}}} \frac {{ {\left[ 15 \right]_{q^{ 6}}} }} { 
  {\left[ 3 \right]_{q^{ 6}}} {\left[ 5 \right]_{q^{ 6}}}}
,\end{multline*} 
which, by Corollary~\ref{cor:A}, is a polynomial in $q$ with 
non-negative integer coefficients.
If $m\equiv
51
~(\text{mod }84),$ then we have 
\begin{multline*} \Cat^m(E_8;q)=
{\left[ \tfrac{15 m+1}2 \right]_{q^{ 4}}} 
{\left[ {15m+4} \right]_{q^{ 
   2}}} 
 {\left[ 5m+2 \right]_{q^{ 6 }}}
 {\left[ \tfrac{15m+7}4 
   \right]_{q^{ 8}}} 
 {\left[ \tfrac{5m+3}6 \right]_{q^{ 36}}}
\frac {[6]_{q^6}} {[2]_{q^6}\left[3\right]_{q^6}}
\\
\times 
 {\left[ 3m+2 \right]_{q^{ 10 }}} {\left[ \tfrac{5m+4}7
   \right]_{q^{ 42 }}}
\frac {[21]_{q^2}} {[3]_{q^2}\left[7\right]_{q^2}}
 {\left[ \tfrac{m+1}4 \right]_{q^{ 120}}}
\frac {[60]_{q^2}\left[2\right]_{q^2}\left[3\right]_{q^2}\left[5\right]_{q^2}} 
{\left[10\right]_{q^2}\left[12\right]_{q^2}\left[15\right]_{q^2}} 
,\end{multline*} 
which, by Corollary~\ref{cor:A} and Lemma~\ref{lem:101215}, 
is a polynomial in $q$ with 
non-negative integer coefficients.
If $m\equiv
52
~(\text{mod }84),$ then we have \begin{multline*} \Cat^m(E_8;q)=
{\left[ 15 m+1 \right]_{q^{ 2 }}} {\left[ 
   \tfrac{15m+4}{28}  \right]_{q^{ 56}}} \frac {{ {\left[ 28
    \right]_{q^{ 2}}} }} { {\left[ 4 \right]_{q^{ 2}}} 
  {\left[ 7 \right]_{q^{ 2}}}} {\left[ \tfrac{5m+2}{2}  
  \right]_{q^{ 12 }}} {\left[ 15 m+7 \right]_{q^{ 2
  }}}\\
\times {\left[ 5 m+3 \right]_{q^6}} 
 {\left[ \tfrac{3m+2}{2}  \right]_{q^{ 20 }}} 
 {\left[ \tfrac{5m+4}{12}  \right]_{q^{ 72}}} \frac {{ 
  {\left[ 12 \right]_{q^{ 6}}} }} { {\left[ 3 
   \right]_{q^{ 6}}} {\left[ 4 \right]_{q^{ 6}}}} {\left[ m+1 
  \right]_{q^{ 30 }}}
,\end{multline*} 
which, by Corollary~\ref{cor:A}, is a polynomial in $q$ with 
non-negative integer coefficients.
If $m\equiv
53
~(\text{mod }84),$ then we have \begin{multline*} \Cat^m(E_8;q)=
{\left[ \tfrac{15m+1}{4}  \right]_{q^{ 8 }}}
 {\left[ 15m+4 \right]_{q^{ 2 }}} {\left[ 
  \tfrac{5m+2}3  \right]_{q^{ 18 }}}
 {\left[ \tfrac{15 m+7}2 
  \right]_{q^4}} {\left[ \tfrac{5m+3}{4}  
   \right]_{q^{ 24}}} 
   \\
\times {\left[ \tfrac{3m+2}7 \right]_{q^{70}}}
\frac {[35]_{q^2}} {[5]_{q^2}\left[7\right]_{q^2}} 
 {\left[ {5m+4}  \right]_{q^{ 6 }}}
 {\left[ \tfrac{m+1}{2} \right]_{q^{ 60}}} \frac {{ 
  {\left[ 30 \right]_{q^{ 2}}}
  {\left[ 2 \right]_{q^{ 2}}}
  {\left[ 3 \right]_{q^{ 2}}}
  {\left[ 5 \right]_{q^{ 2}}} }} { {\left[ 6 
   \right]_{q^{ 2}}}{\left[ 10
   \right]_{q^{ 2}}} {\left[ 15 \right]_{q^{ 2}}}}
,\end{multline*} 
which, by Lemma~\ref{lem:61015C}, 
is a polynomial in $q$ with 
non-negative integer coefficients.
If $m\equiv
54
~(\text{mod }84),$ then we have \begin{multline*} \Cat^m(E_8;q)=
{\left[ 15 m+1 \right]_{q^{ 2 }}} {\left[ \tfrac{15m+4}2 
  \right]_{q^4}} {\left[ \tfrac{5m+2}{4}  
   \right]_{q^{ 24}}} 
{\left[ 15 m+7 \right]_{q^2}} \\
 {\left[ \tfrac{5m+3}{21}  \right]_{q^{ 126}}} \frac {{ 
  {\left[ 63 \right]_{q^{ 2}}} }} { {\left[ 7 
   \right]_{q^{ 2}}} {\left[ 9 \right]_{q^{ 2}}}} {\left[ 
  \tfrac{3m+2}{4}  \right]_{q^{ 40 }}}
\frac {[10]_{q^4}} {[2]_{q^4}\left[5\right]_{q^4}}
 {\left[ 
  \tfrac{5m+4}{2}  \right]_{q^{ 12 }}} {\left[ m+1 
  \right]_{q^{ 30 }}}
,\end{multline*} 
which, by Corollary~\ref{cor:A}, is a polynomial in $q$ with 
non-negative integer coefficients.
If $m\equiv
55
~(\text{mod }84),$ then we have 
\begin{multline*} \Cat^m(E_8;q)=
{\left[ \tfrac{15 m+1}{14} \right]_{q^{ 28}}}
\frac {[14]_{q^2}} {[2]_{q^2}\left[7\right]_{q^2}} 
{\left[ {15m+4} \right]_{q^{ 
   2}}} 
 {\left[ 5m+2 \right]_{q^{ 6 }}}
 {\left[ \tfrac{15m+7}4 
   \right]_{q^{ 8}}} 
 {\left[ \tfrac{5m+3}2 \right]_{q^{ 12}}}\\
\times 
 {\left[ 3m+2 \right]_{q^{ 10 }}} {\left[ \tfrac{5m+4}3
   \right]_{q^{ 18 }}} {\left[ \tfrac{m+1}4 \right]_{q^{ 120}}}
\frac {[60]_{q^2}\left[2\right]_{q^2}\left[3\right]_{q^2}\left[5\right]_{q^2}} 
{\left[10\right]_{q^2}\left[12\right]_{q^2}\left[15\right]_{q^2}} 
,\end{multline*} 
which, by Lemma~\ref{lem:101215C}, 
is a polynomial in $q$ with 
non-negative integer coefficients.
If $m\equiv
56
~(\text{mod }84),$ then we have \begin{multline*} \Cat^m(E_8;q)=
{\left[ 15 m+1 \right]_{q^{ 2 }}} {\left[ 
  \tfrac{15m+4}{4}  \right]_{q^{ 8 }}} {\left[ 
   \tfrac{5m+2}{6}  \right]_{q^{ 36}}} \frac {{ {\left[ 6 
   \right]_{q^{ 6}}} }} { {\left[ 2 \right]_{q^{ 6}}} 
  {\left[ 3 \right]_{q^{ 6}}}} {\left[ \tfrac{15m+7}{7}  
  \right]_{q^{ 14 }}} {\left[ 5 m+3 \right]_{q^{ 6 
  }}} \\
\times
{\left[ \tfrac{3m+2}{2}  \right]_{q^{ 20 }}} 
 {\left[ \tfrac{5m+4}{4}  \right]_{q^{ 24 }}} 
 {\left[ m+1 \right]_{q^{ 30 }}}
,\end{multline*} 
which, by Corollary~\ref{cor:A}, is a polynomial in $q$ with 
non-negative integer coefficients.
If $m\equiv
57
~(\text{mod }84),$ then we have \begin{multline*} \Cat^m(E_8;q)=
{\left[ \tfrac{15m+1}{4}  \right]_{q^{ 8 }}}
 {\left[ 15m+4 \right]_{q^{ 2 }}} {\left[ 
  \tfrac{5m+2}7  \right]_{q^{ 42 }}}
\frac {[21]_{q^2}} {[3]_{q^2}\left[7\right]_{q^2}}
 {\left[ \tfrac{15 m+7}2 
  \right]_{q^4}} {\left[ \tfrac{5m+3}{12}  
   \right]_{q^{ 72}}}
\frac {[12]_{q^6}} {[3]_{q^6}\left[4\right]_{q^6}} 
   \\
\times {\left[ 3m+2 \right]_{q^{10}}} 
 {\left[ {5m+4}  \right]_{q^{ 6 }}}
 {\left[ \tfrac{m+1}{2} \right]_{q^{ 60}}} \frac {{ 
  {\left[ 30 \right]_{q^{ 2}}}
  {\left[ 2 \right]_{q^{ 2}}}
  {\left[ 3 \right]_{q^{ 2}}}
  {\left[ 5 \right]_{q^{ 2}}} }} { {\left[ 6 
   \right]_{q^{ 2}}}{\left[ 10
   \right]_{q^{ 2}}} {\left[ 15 \right]_{q^{ 2}}}}
,\end{multline*} 
which, by Corollary~\ref{cor:A} and Lemma~\ref{lem:61015}, 
is a polynomial in $q$ with 
non-negative integer coefficients.
If $m\equiv
58
~(\text{mod }84),$ then we have \begin{multline*} \Cat^m(E_8;q)=
{\left[ 15 m+1 \right]_{q^{ 2 }}} {\left[ \tfrac{m+4}2 
  \right]_{q^4}} {\left[ \tfrac{5m+2}{4}  
   \right]_{q^{ 24}}} 
 {\left[ 15 m+7 \right]_{q^2}} 
 {\left[ 5 m+3 \right]_{q^6}} \\
{\left[ 
   \tfrac{3m+2}{4}  \right]_{q^{ 40}}} \frac {{ {\left[ 10 
   \right]_{q^{ 4}}} }} { {\left[ 2 \right]_{q^{ 4}}} 
  {\left[ 5 \right]_{q^{ 4}}}} {\left[ \tfrac{5m+4}{42}  
   \right]_{q^{ 252}}} 
\frac {{ {\left[ 126 \right]_{q^{ 2}}} {\left[ 3 \right]_{q^{ 2}}} 
  }} { {\left[ 6 \right]_{q^{ 2}}} {\left[ 7 
   \right]_{q^{ 2}}} {\left[ 9 \right]_{q^{ 2}}} }
 {\left[ m+1 \right]_{q^{ 30 }}}
,\end{multline*} 
which, by Lemma~\ref{lem:679}, 
is a polynomial in $q$ with 
non-negative integer coefficients.
If $m\equiv
59
~(\text{mod }84),$ then we have 
\begin{multline*} \Cat^m(E_8;q)=
{\left[ \tfrac{15 m+1}2 \right]_{q^{ 4}}} 
{\left[ \tfrac{15m+4}7 \right]_{q^{ 
   14}}} 
 {\left[ \tfrac{5m+2}3 \right]_{q^{ 18 }}}
 {\left[ \tfrac{15m+7}4 
   \right]_{q^{ 8}}} 
 {\left[ \tfrac{5m+3}2 \right]_{q^{ 12}}}\\
\times 
 {\left[ 3m+2 \right]_{q^{ 10 }}} {\left[ 5m+4
   \right]_{q^{ 6 }}} {\left[ \tfrac{m+1}4 \right]_{q^{ 120}}}
\frac {[60]_{q^2}\left[2\right]_{q^2}\left[3\right]_{q^2}\left[5\right]_{q^2}} 
{\left[10\right]_{q^2}\left[12\right]_{q^2}\left[15\right]_{q^2}} 
,\end{multline*} 
which, by Lemma~\ref{lem:101215}, 
is a polynomial in $q$ with 
non-negative integer coefficients.
If $m\equiv
60
~(\text{mod }84),$ then we have \begin{multline*} \Cat^m(E_8;q)=
{\left[ 15 m+1 \right]_{q^{ 2 }}} {\left[ 
  \tfrac{15m+4}{4}  \right]_{q^{ 8 }}} {\left[ 
  \tfrac{5m+2}{2}  \right]_{q^{ 12 }}} {\left[ 15 m+7 
  \right]_{q^2}} {\left[ \tfrac{5m+3}{3}  
  \right]_{q^{ 18 }}}\\
\times {\left[ \tfrac{3m+2}{14}  
   \right]_{q^{ 140}}}
 \frac {{ {\left[ 70 \right]_{q^{ 2}}} 
  }} { {\left[ 7 \right]_{q^{ 2}}} {\left[ 10 
   \right]_{q^{ 2}}}} {\left[ \tfrac{5m+4}{4}  \right]_{q^{ 24 
  }}} {\left[ m+1 \right]_{q^{ 30 }}}
,\end{multline*} 
which, by Corollary~\ref{cor:A}, is a polynomial in $q$ with 
non-negative integer coefficients.
If $m\equiv
61
~(\text{mod }84),$ then we have \begin{multline*} \Cat^m(E_8;q)=
{\left[ \tfrac{15m+1}{4}  \right]_{q^{ 8 }}}
 {\left[ 15m+4 \right]_{q^{ 2 }}} {\left[ 
  {5m+2}  \right]_{q^{ 6 }}}
 {\left[ \tfrac{15 m+7}2 
  \right]_{q^4}} {\left[ \tfrac{5m+3}{28}  
   \right]_{q^{ 168}}}
\frac {[84]_{q^2}} {[7]_{q^2}\left[12\right]_{q^2}} 
   \\
\times {\left[ 3m+2 \right]_{q^{10}}} 
 {\left[ \tfrac{5m+4} 3 \right]_{q^{ 18 }}}
 {\left[ \tfrac{m+1}{2} \right]_{q^{ 60}}} \frac {{ 
  {\left[ 30 \right]_{q^{ 2}}}
  {\left[ 2 \right]_{q^{ 2}}}
  {\left[ 3 \right]_{q^{ 2}}}
  {\left[ 5 \right]_{q^{ 2}}} }} { {\left[ 6 
   \right]_{q^{ 2}}}{\left[ 10
   \right]_{q^{ 2}}} {\left[ 15 \right]_{q^{ 2}}}}
,\end{multline*} 
which, by Corollary~\ref{cor:A} and Lemma~\ref{lem:61015}, 
is a polynomial in $q$ with 
non-negative integer coefficients.
If $m\equiv
62
~(\text{mod }84),$ then we have \begin{multline*} \Cat^m(E_8;q)=
{\left[ \tfrac{15m+1}{7}  \right]_{q^{ 14 }}} 
 {\left[ \tfrac{15m+4}2 \right]_{q^{ 4 }}} {\left[ 
   \tfrac{5m+2}{12}  \right]_{q^{ 72}}} \frac {{ {\left[ 12
    \right]_{q^{ 6}}} }} { {\left[ 3 \right]_{q^{ 6}}} 
  {\left[ 4 \right]_{q^{ 6}}}} {\left[ 15 m+7 \right]_{q^{ 
  2 }}} {\left[ 5 m+3 \right]_{q^6}} \\
\times
 {\left[ \tfrac{3m+2}{4}  \right]_{q^{ 40}}} \frac {{ 
  {\left[ 10 \right]_{q^{ 4}}} }} { {\left[ 2 
   \right]_{q^{ 4}}} {\left[ 5 \right]_{q^{ 4}}}} {\left[ 
  \tfrac{5m+4}{2}  \right]_{q^{ 12 }}} {\left[ m+1 
  \right]_{q^{ 30 }}}
,\end{multline*} 
which, by Corollary~\ref{cor:A}, is a polynomial in $q$ with 
non-negative integer coefficients.
If $m\equiv
63
~(\text{mod }84),$ then we have 
\begin{multline*} \Cat^m(E_8;q)=
{\left[ \tfrac{15 m+1}2 \right]_{q^{ 4}}} 
{\left[ {15m+4} \right]_{q^{ 
   2}}} 
 {\left[ 5m+2 \right]_{q^{ 6 }}}
 {\left[ \tfrac{15m+7}{28} 
   \right]_{q^{ 56}}}
\frac {[28]_{q^2}} {[4]_{q^2}\left[7\right]_{q^2}} 
 {\left[ \tfrac{5m+3}6 \right]_{q^{ 36}}}
\frac {[6]_{q^6}} {[2]_{q^6}\left[3\right]_{q^6}}
\\
\times 
 {\left[ 3m+2 \right]_{q^{ 10 }}} {\left[ 5m+4
   \right]_{q^{ 6 }}} {\left[ \tfrac{m+1}4 \right]_{q^{ 120}}}
\frac {[60]_{q^2}\left[2\right]_{q^2}\left[3\right]_{q^2}\left[5\right]_{q^2}} 
{\left[10\right]_{q^2}\left[12\right]_{q^2}\left[15\right]_{q^2}} 
,\end{multline*} 
which, by Corollary~\ref{cor:A} and Lemma~\ref{lem:101215}, 
is a polynomial in $q$ with 
non-negative integer coefficients.
If $m\equiv
64
~(\text{mod }84),$ then we have \begin{multline*} \Cat^m(E_8;q)=
{\left[ 15 m+1 \right]_{q^{ 2 }}} {\left[ 
  \tfrac{15m+4}{4}  \right]_{q^{ 8 }}} {\left[ 
   \tfrac{5m+2}{14}  \right]_{q^{ 84}}} \frac {{ {\left[ 42
    \right]_{q^{ 2}}} }} { {\left[ 6 \right]_{q^{ 2}}} 
  {\left[ 7 \right]_{q^{ 2}}}} {\left[ 15 m+7 \right]_{q^{ 
  2 }}}\\
\times {\left[ 5 m+3 \right]_{q^6}} 
 {\left[ \tfrac{3m+2}{2}  \right]_{q^{ 20 }}} 
 {\left[ \tfrac{5m+4}{12}  \right]_{q^{ 72}}} \frac {{ 
  {\left[ 12 \right]_{q^{ 6}}} }} { {\left[ 3 
   \right]_{q^{ 6}}} {\left[ 4 \right]_{q^{ 6}}}} {\left[ m+1 
  \right]_{q^{ 30 }}}
,\end{multline*} 
which, by Corollary~\ref{cor:A}, is a polynomial in $q$ with 
non-negative integer coefficients.
If $m\equiv
65
~(\text{mod }84),$ then we have \begin{multline*} \Cat^m(E_8;q)=
{\left[ \tfrac{15m+1}{4}  \right]_{q^{ 8 }}}
 {\left[ 15m+4 \right]_{q^{ 2 }}} {\left[ 
  \tfrac{5m+2}3  \right]_{q^{ 18 }}}
 {\left[ \tfrac{15 m+7}2 
  \right]_{q^4}} {\left[ \tfrac{5m+3}{4}  
   \right]_{q^{ 24}}} 
   \\
\times {\left[ 3m+2 \right]_{q^{10}}} 
 {\left[ \tfrac{5m+4}7  \right]_{q^{ 42 }}}
\frac {[21]_{q^2}} {[3]_{q^2}\left[7\right]_{q^2}}
 {\left[ \tfrac{m+1}{2} \right]_{q^{ 60}}} \frac {{ 
  {\left[ 30 \right]_{q^{ 2}}}
  {\left[ 2 \right]_{q^{ 2}}}
  {\left[ 3 \right]_{q^{ 2}}}
  {\left[ 5 \right]_{q^{ 2}}} }} { {\left[ 6 
   \right]_{q^{ 2}}}{\left[ 10
   \right]_{q^{ 2}}} {\left[ 15 \right]_{q^{ 2}}}}
,\end{multline*} 
which, by Corollary~\ref{cor:A} and Lemma~\ref{lem:61015}, 
is a polynomial in $q$ with 
non-negative integer coefficients.
If $m\equiv
66
~(\text{mod }84),$ then we have \begin{multline*} \Cat^m(E_8;q)=
{\left[ 15 m+1 \right]_{q^{ 2 }}} {\left[ 
  \tfrac{15m+4}{14}  \right]_{q^{ 28}}}
\frac {[14]_{q^2}} {[2]_{q^2}\left[7\right]_{q^2}}
 {\left[ 
   \tfrac{5m+2}{4}  \right]_{q^{ 24}}} 
 {\left[ 15 m+7 \right]_{q^{ 
  2 }}}\\
\times {\left[ \tfrac{5m+3}{3}  \right]_{q^{ 18 
  }}} {\left[ \tfrac{3m+2}{4}  \right]_{q^{ 40 }}} 
\frac {[10]_{q^4}} {[2]_{q^4}\left[5\right]_{q^4}}
 {\left[ \tfrac{5m+4}{2}  \right]_{q^{ 12 }}} 
 {\left[ m+1 \right]_{q^{ 30 }}}
,\end{multline*} 
which, by Corollary~\ref{cor:A}, is a polynomial in $q$ with 
non-negative integer coefficients.
If $m\equiv
67
~(\text{mod }84),$ then we have 
\begin{multline*} \Cat^m(E_8;q)=
{\left[ \tfrac{15 m+1}2 \right]_{q^{ 4}}} 
{\left[ {15m+4} \right]_{q^{ 
   2}}} 
 {\left[ 5m+2 \right]_{q^{ 6 }}}
 {\left[ \tfrac{15m+7}4 
   \right]_{q^{ 8}}} 
 {\left[ \tfrac{5m+3}2 \right]_{q^{ 12}}}\\
\times 
 {\left[ \tfrac{3m+2}7 \right]_{q^{ 70 }}}
\frac {[35]_{q^2}} {[5]_{q^2}\left[7\right]_{q^2}}
 {\left[ \tfrac{5m+4}3
   \right]_{q^{ 18 }}} {\left[ \tfrac{m+1}4 \right]_{q^{ 120}}}
\frac {[60]_{q^2}\left[2\right]_{q^2}\left[3\right]_{q^2}\left[5\right]_{q^2}} 
{\left[10\right]_{q^2}\left[12\right]_{q^2}\left[15\right]_{q^2}} 
,\end{multline*} 
which, by Lemma~\ref{lem:101215B}, 
is a polynomial in $q$ with 
non-negative integer coefficients.
If $m\equiv
68
~(\text{mod }84),$ then we have \begin{multline*} \Cat^m(E_8;q)=
{\left[ 15 m+1 \right]_{q^{ 2 }}} {\left[ 
  \tfrac{15m+4}{4}  \right]_{q^{ 8 }}} {\left[ 
   \tfrac{5m+2}{6}  \right]_{q^{ 36}}} \frac {{ {\left[ 6 
   \right]_{q^{ 6}}} }} { {\left[ 2 \right]_{q^{ 6}}} 
  {\left[ 3 \right]_{q^{ 6}}}} \\
\times{\left[ 15 m+7 \right]_{q^{ 
  2 }}} {\left[ \tfrac{5m+3}{7}  \right]_{q^{ 42 
  }}}
\frac {[21]_{q^2}} {[3]_{q^2}\left[7\right]_{q^2}}
 {\left[ \tfrac{3m+2}{2}  \right]_{q^{ 20 }}} 
 {\left[ \tfrac{5m+4}{4}  \right]_{q^{ 24 }}} 
 {\left[ m+1 \right]_{q^{ 30 }}}
.\end{multline*} 
If one decomposes $[15m+1]_{q^2}$ as 
$[5m+1]_{q^6}+q^2[5m]_{q^6}+q^4[5m]_{q^6}$,
then one sees that, by Corollary~\ref{cor:A},
this is a polynomial in $q$ with 
non-negative integer coefficients.
If $m\equiv
69
~(\text{mod }84),$ then we have \begin{multline*} \Cat^m(E_8;q)=
{\left[ \tfrac{15m+1}{28}  \right]_{q^{ 56 }}}
\frac {[28]_{q^2}} {[4]_{q^2}\left[7\right]_{q^2}}
 {\left[ 15m+4 \right]_{q^{ 2 }}} {\left[ 
  {5m+2}  \right]_{q^{ 6 }}}
 {\left[ \tfrac{15 m+7}2 
  \right]_{q^4}} {\left[ \tfrac{5m+3}{12}  
   \right]_{q^{ 72}}}
\frac {[12]_{q^6}} {[3]_{q^6}\left[4\right]_{q^6}} 
   \\
\times {\left[ 3m+2 \right]_{q^{10}}} 
 {\left[ {5m+4}  \right]_{q^{ 6 }}}
 {\left[ \tfrac{m+1}{2} \right]_{q^{ 60}}} \frac {{ 
  {\left[ 30 \right]_{q^{ 2}}}
  {\left[ 2 \right]_{q^{ 2}}}
  {\left[ 3 \right]_{q^{ 2}}}
  {\left[ 5 \right]_{q^{ 2}}} }} { {\left[ 6 
   \right]_{q^{ 2}}}{\left[ 10
   \right]_{q^{ 2}}} {\left[ 15 \right]_{q^{ 2}}}}
,\end{multline*} 
which, by Corollary~\ref{cor:A} and Lemma~\ref{lem:61015}, 
is a polynomial in $q$ with 
non-negative integer coefficients.
If $m\equiv
70
~(\text{mod }84),$ then we have \begin{multline*} \Cat^m(E_8;q)=
{\left[ 15 m+1 \right]_{q^{ 2 }}} {\left[ \tfrac{15m+4}2 
  \right]_{q^4}} {\left[ \tfrac{5m+2}{4}  
   \right]_{q^{ 24}}}
 {\left[ \tfrac{15m+7}{7}  \right]_{q^{ 14 
  }}}\\
\times {\left[ 5 m+3 \right]_{q^6}} 
 {\left[ \tfrac{3m+2}{4}  \right]_{q^{ 40}}} \frac {{ 
  {\left[ 10 \right]_{q^{ 4}}} }} { {\left[ 2 
   \right]_{q^{ 4}}} {\left[ 5 \right]_{q^{ 4}}}} {\left[ 
  \tfrac{5m+4}{6}  \right]_{q^{ 36 }}}
\frac {[6]_{q^6}} {[2]_{q^6}\left[3\right]_{q^6}}
 {\left[ m+1 
  \right]_{q^{ 30 }}}
,\end{multline*} 
which, by Corollary~\ref{cor:A}, is a polynomial in $q$ with 
non-negative integer coefficients.
If $m\equiv
71
~(\text{mod }84),$ then we have 
\begin{multline*} \Cat^m(E_8;q)=
{\left[ \tfrac{15 m+1}2 \right]_{q^{ 4}}} 
{\left[ {15m+4} \right]_{q^{ 
   2}}} 
 {\left[ \tfrac{5m+2}{21} \right]_{q^{ 126 }}}
\frac {[63]_{q^2}} {[7]_{q^2}\left[9\right]_{q^2}}
 {\left[ \tfrac{15m+7}4 
   \right]_{q^{ 8}}} 
 {\left[ \tfrac{5m+3}2 \right]_{q^{ 12}}}\\
\times 
 {\left[ 3m+2 \right]_{q^{ 10 }}} {\left[ 5m+4
   \right]_{q^{ 6 }}} {\left[ \tfrac{m+1}4 \right]_{q^{ 120}}}
\frac {[60]_{q^2}\left[2\right]_{q^2}\left[3\right]_{q^2}\left[5\right]_{q^2}} 
{\left[10\right]_{q^2}\left[12\right]_{q^2}\left[15\right]_{q^2}} 
,\end{multline*} 
which, by Corollary~\ref{cor:A} and Lemma~\ref{lem:101215}, 
is a polynomial in $q$ with 
non-negative integer coefficients.
If $m\equiv
72
~(\text{mod }84),$ then we have \begin{multline*} \Cat^m(E_8;q)=
{\left[ 15 m+1 \right]_{q^{ 2 }}} {\left[ 
  \tfrac{15m+4}{4}  \right]_{q^{ 8 }}} {\left[ 
  \tfrac{5m+2}{2}  \right]_{q^{ 12 }}} {\left[ 15 m+7 
  \right]_{q^2}} {\left[ \tfrac{5m+3}{3}  
  \right]_{q^{ 18 }}} {\left[ \tfrac{3m+2}{2}  
  \right]_{q^{ 20 }}} \\
{\left[ \tfrac{5m+4}{28}  
   \right]_{q^{ 168}}} \frac {{ {\left[ 84 \right]_{q^{ 2}}} 
  }} { {\left[ 7 \right]_{q^{ 2}}} {\left[ 12 
   \right]_{q^{ 2}}}} {\left[ m+1 \right]_{q^{ 30 }}}
,\end{multline*} 
which, by Corollary~\ref{cor:A}, is a polynomial in $q$ with 
non-negative integer coefficients.
If $m\equiv
73
~(\text{mod }84),$ then we have \begin{multline*} \Cat^m(E_8;q)=
{\left[ \tfrac{15m+1}{4}  \right]_{q^{ 8 }}}
 {\left[ \tfrac{15m+4}7 \right]_{q^{ 14 }}} {\left[ 
  {5m+2}  \right]_{q^{ 6 }}}
 {\left[ \tfrac{15 m+7}2 
  \right]_{q^4}} {\left[ \tfrac{5m+3}{4}  
   \right]_{q^{ 24}}} 
   \\
\times {\left[ 3m+2 \right]_{q^{10}}} 
 {\left[ \tfrac{5m+4}3  \right]_{q^{ 18 }}}
 {\left[ \tfrac{m+1}{2} \right]_{q^{ 60}}} \frac {{ 
  {\left[ 30 \right]_{q^{ 2}}}
  {\left[ 2 \right]_{q^{ 2}}}
  {\left[ 3 \right]_{q^{ 2}}}
  {\left[ 5 \right]_{q^{ 2}}} }} { {\left[ 6 
   \right]_{q^{ 2}}}{\left[ 10
   \right]_{q^{ 2}}} {\left[ 15 \right]_{q^{ 2}}}}
,\end{multline*} 
which, by Lemma~\ref{lem:61015}, 
is a polynomial in $q$ with 
non-negative integer coefficients.
If $m\equiv
74
~(\text{mod }84),$ then we have \begin{multline*} \Cat^m(E_8;q)=
{\left[ 15 m+1 \right]_{q^{ 2 }}} {\left[ \tfrac{15m+4}2 
  \right]_{q^4}} {\left[ \tfrac{5m+2}{4}  
   \right]_{q^{ 24}}} 
{\left[ 15 m+7 \right]_{q^2}} 
 {\left[ 5 m+3 \right]_{q^6}} \\
\times{\left[ 
   \tfrac{3m+2}{28}  \right]_{q^{ 280}}} 
\frac {{ {\left[ 140
    \right]_{q^{ 2}}} {\left[ 2
    \right]_{q^{ 2}}} }} { {\left[ 4
    \right]_{q^{ 2}}}{\left[ 7 \right]_{q^{ 2}}} 
  {\left[ 10 \right]_{q^{ 2}}}} {\left[ \tfrac{5m+4}{2}  
  \right]_{q^{ 12 }}} {\left[ \tfrac{m+1}{3} \right]_{q^{ 
   90}}} \frac {{ {\left[ 15 \right]_{q^{ 6}}} }} { 
  {\left[ 3 \right]_{q^{ 6}}} {\left[ 5 \right]_{q^{ 6}}}}
,\end{multline*} 
which, by Corollary~\ref{cor:A} and Lemma~\ref{lem:4710}, 
is a polynomial in $q$ with 
non-negative integer coefficients.
If $m\equiv
75
~(\text{mod }84),$ then we have 
\begin{multline*} \Cat^m(E_8;q)=
{\left[ \tfrac{15 m+1}2 \right]_{q^{ 4}}} 
{\left[ {15m+4} \right]_{q^{ 
   2}}} 
 {\left[ 5m+2 \right]_{q^{ 6 }}}
 {\left[ \tfrac{15m+7}4 
   \right]_{q^{ 8}}} 
 {\left[ \tfrac{5m+3}{42} \right]_{q^{ 252}}}
\frac {[126]_{q^2}\left[3\right]_{q^2}} 
{[6]_{q^2}\left[7\right]_{q^2}\left[9\right]_{q^2}}
\\
\times 
 {\left[ 3m+2 \right]_{q^{ 10 }}} {\left[ 5m+4
   \right]_{q^{ 6 }}} {\left[ \tfrac{m+1}4 \right]_{q^{ 120}}}
\frac {[60]_{q^2}\left[2\right]_{q^2}\left[3\right]_{q^2}\left[5\right]_{q^2}} 
{\left[10\right]_{q^2}\left[12\right]_{q^2}\left[15\right]_{q^2}} 
,\end{multline*} 
which, by Lemmas~\ref{lem:101215} and \ref{lem:679}, 
is a polynomial in $q$ with 
non-negative integer coefficients.
If $m\equiv
76
~(\text{mod }84),$ then we have \begin{multline*} \Cat^m(E_8;q)=
{\left[ \tfrac{15m+1}{7}  \right]_{q^{ 14 }}} 
 {\left[ \tfrac{15m+4}{4}  \right]_{q^{ 8 }}} 
 {\left[ \tfrac{5m+2}{2}  \right]_{q^{ 12 }}} 
 {\left[ 15m+7 \right]_{q^{ 2 }}} \\
\times{\left[ 5 m+3 
  \right]_{q^6}} {\left[ \tfrac{3m+2}{2}  
  \right]_{q^{ 20 }}}
 {\left[ \tfrac{5m+4}{12}  
   \right]_{q^{ 72}}} \frac {{ {\left[ 12 \right]_{q^{ 6}}} 
  }} { {\left[ 3 \right]_{q^{ 6}}} {\left[ 4 
   \right]_{q^{ 6}}}} {\left[ m+1 \right]_{q^{ 30 }}}
,\end{multline*} 
which, by Corollary~\ref{cor:A}, is a polynomial in $q$ with 
non-negative integer coefficients.
If $m\equiv
77
~(\text{mod }84),$ then we have \begin{multline*} \Cat^m(E_8;q)=
{\left[ \tfrac{15m+1}{4}  \right]_{q^{ 8 }}}
 {\left[ 15m+4 \right]_{q^{ 2 }}} {\left[ 
  \tfrac{5m+2}3  \right]_{q^{ 18 }}}
 {\left[ \tfrac{15 m+7}{14} 
  \right]_{q^{28}}}
\frac {[14]_{q^2}} {[2]_{q^2}\left[7\right]_{q^2}}
 {\left[ \tfrac{5m+3}{4}  
   \right]_{q^{ 24}}} 
   \\
\times {\left[ 3m+2 \right]_{q^{10}}} 
 {\left[ {5m+4}  \right]_{q^{ 6 }}}
 {\left[ \tfrac{m+1}{2} \right]_{q^{ 60}}} \frac {{ 
  {\left[ 30 \right]_{q^{ 2}}}
  {\left[ 2 \right]_{q^{ 2}}}
  {\left[ 3 \right]_{q^{ 2}}}
  {\left[ 5 \right]_{q^{ 2}}} }} { {\left[ 6 
   \right]_{q^{ 2}}}{\left[ 10
   \right]_{q^{ 2}}} {\left[ 15 \right]_{q^{ 2}}}}
,\end{multline*} 
which, by Corollary~\ref{cor:A} and Lemma~\ref{lem:61015}, 
is a polynomial in $q$ with 
non-negative integer coefficients.
If $m\equiv
78
~(\text{mod }84),$ then we have \begin{multline*} \Cat^m(E_8;q)=
{\left[ 15 m+1 \right]_{q^{ 2 }}} {\left[ \tfrac{15m+4}2 
  \right]_{q^4}} {\left[ \tfrac{5m+2}{28}  
   \right]_{q^{ 168}}} \frac {{ {\left[ 84 \right]_{q^{ 2}}} 
  }} { {\left[ 7 \right]_{q^{ 2}}} {\left[ 12 
   \right]_{q^{ 2}}}} {\left[ 15 m+7 \right]_{q^2}} 
 {\left[ \tfrac{5m+3}{3}  \right]_{q^{ 18 }}} \\
\times
 {\left[ \tfrac{3m+2}{4}  \right]_{q^{ 40}}} \frac {{ 
  {\left[ 10 \right]_{q^{ 4}}} }} { {\left[ 2 
   \right]_{q^{ 4}}} {\left[ 5 \right]_{q^{ 4}}}} {\left[ 
  \tfrac{5m+4}{2}  \right]_{q^{ 12 }}} {\left[ m+1 
  \right]_{q^{ 30 }}}
,\end{multline*} 
which, by Corollary~\ref{cor:A}, is a polynomial in $q$ with 
non-negative integer coefficients.
If $m\equiv
79
~(\text{mod }84),$ then we have 
\begin{multline*} \Cat^m(E_8;q)=
{\left[ \tfrac{15 m+1}2 \right]_{q^{ 4}}} 
{\left[ {15m+4} \right]_{q^{ 
   2}}} 
 {\left[ 5m+2 \right]_{q^{ 6 }}}
 {\left[ \tfrac{15m+7}4 
   \right]_{q^{ 8}}} 
 {\left[ \tfrac{5m+3}2 \right]_{q^{ 12}}}\\
\times 
 {\left[ 3m+2 \right]_{q^{ 10 }}} {\left[ \tfrac{5m+4}{21}
   \right]_{q^{ 126 }}}
\frac {[63]_{q^2}} {[7]_{q^2}\left[9\right]_{q^2}}
 {\left[ \tfrac{m+1}4 \right]_{q^{ 120}}}
\frac {[60]_{q^2}\left[2\right]_{q^2}\left[3\right]_{q^2}\left[5\right]_{q^2}} 
{\left[10\right]_{q^2}\left[12\right]_{q^2}\left[15\right]_{q^2}} 
,\end{multline*} 
which, by Corollary~\ref{cor:A} and Lemma~\ref{lem:101215}, 
is a polynomial in $q$ with 
non-negative integer coefficients.
If $m\equiv
80
~(\text{mod }84),$ then we have \begin{multline*} \Cat^m(E_8;q)=
{\left[ 15 m+1 \right]_{q^{ 2 }}} {\left[ 
   \tfrac{15m+4}{28}  \right]_{q^{ 56}}} \frac {{ {\left[ 28
    \right]_{q^{ 2}}} }} { {\left[ 4 \right]_{q^{ 2}}} 
  {\left[ 7 \right]_{q^{ 2}}}} {\left[ \tfrac{5m+2}{6}  
   \right]_{q^{ 36}}} \frac {{ {\left[ 6 \right]_{q^{ 6}}} 
  }} { {\left[ 2 \right]_{q^{ 6}}} {\left[ 3 
   \right]_{q^{ 6}}}}
 {\left[ 15 m+7 \right]_{q^2}} 
 {\left[ 5 m+3 \right]_{q^6}}\\
\times {\left[ 
  \tfrac{3m+2}{2}  \right]_{q^{ 20 }}} {\left[ 
  \tfrac{5m+4}{4}  \right]_{q^{ 24 }}} {\left[ m+1 
  \right]_{q^{ 30 }}}
,\end{multline*} 
which, by Corollary~\ref{cor:A}, is a polynomial in $q$ with 
non-negative integer coefficients.
If $m\equiv
81
~(\text{mod }84),$ then we have \begin{multline*} \Cat^m(E_8;q)=
{\left[ \tfrac{15m+1}{4}  \right]_{q^{ 8 }}}
 {\left[ 15m+4 \right]_{q^{ 2 }}} {\left[ 
  {5m+2}  \right]_{q^{ 6 }}}
 {\left[ \tfrac{15 m+7}2 
  \right]_{q^4}} {\left[ \tfrac{5m+3}{12}  
   \right]_{q^{ 72}}}
\frac {[12]_{q^6}} {[3]_{q^6}\left[4\right]_{q^6}} 
   \\
\times {\left[ \tfrac{3m+2}7 \right]_{q^{70}}} 
\frac {[35]_{q^2}} {[5]_{q^2}\left[7\right]_{q^2}}
 {\left[ {5m+4}  \right]_{q^{ 6 }}}
 {\left[ \tfrac{m+1}{2} \right]_{q^{ 60}}} \frac {{ 
  {\left[ 30 \right]_{q^{ 2}}}
  {\left[ 2 \right]_{q^{ 2}}}
  {\left[ 3 \right]_{q^{ 2}}}
  {\left[ 5 \right]_{q^{ 2}}} }} { {\left[ 6 
   \right]_{q^{ 2}}}{\left[ 10
   \right]_{q^{ 2}}} {\left[ 15 \right]_{q^{ 2}}}}
,\end{multline*} 
which, by Corollary~\ref{cor:A} and Lemma~\ref{lem:61015C}, 
is a polynomial in $q$ with 
non-negative integer coefficients.
If $m\equiv
82
~(\text{mod }84),$ then we have \begin{multline*} \Cat^m(E_8;q)=
{\left[ 15 m+1 \right]_{q^{ 2 }}} {\left[ \tfrac{15m+4}2 
  \right]_{q^4}} {\left[ \tfrac{5m+2}{4}  
   \right]_{q^{ 24}}}
 {\left[ 15 m+7 \right]_{q^2}} 
 {\left[ \tfrac{5m+3}{7}  \right]_{q^{ 42 }}}
\frac {[21]_{q^2}} {[3]_{q^2}\left[7\right]_{q^2}}
 \\
 {\left[ \tfrac{3m+2}{4}  \right]_{q^{ 40}}} \frac {{ 
  {\left[ 10 \right]_{q^{ 4}}} }} { {\left[ 2 
   \right]_{q^{ 4}}} {\left[ 5 \right]_{q^{ 4}}}} {\left[ 
  \tfrac{5m+4}{6}  \right]_{q^{ 36 }}}
\frac {[6]_{q^6}} {[2]_{q^6}\left[3\right]_{q^6}}
 {\left[ m+1 
  \right]_{q^{ 30 }}}
.\end{multline*} 
If one decomposes $[15m+1]_{q^2}$ as 
$[5m+1]_{q^6}+q^2[5m]_{q^6}+q^4[5m]_{q^6}$,
then one sees that, by Corollary~\ref{cor:A},
this is a polynomial in $q$ with 
non-negative integer coefficients.
If $m\equiv
83
~(\text{mod }84),$ then we have 
\begin{multline*} \Cat^m(E_8;q)=
{\left[ \tfrac{15 m+1}{14} \right]_{q^{ 28}}}
\frac {[14]_{q^2}} {[2]_{q^2}\left[7\right]_{q^2}} 
{\left[ {15m+4} \right]_{q^{ 
   2}}} 
 {\left[ \tfrac{5m+2}3 \right]_{q^{ 18 }}}
 {\left[ \tfrac{15m+7}4 
   \right]_{q^{ 8}}} 
 {\left[ \tfrac{5m+3}2 \right]_{q^{ 12}}}\\
\times 
 {\left[ 3m+2 \right]_{q^{ 10 }}} {\left[ 5m+4
   \right]_{q^{ 6 }}} {\left[ \tfrac{m+1}4 \right]_{q^{ 120}}}
\frac {[60]_{q^2}\left[2\right]_{q^2}\left[3\right]_{q^2}\left[5\right]_{q^2}} 
{\left[10\right]_{q^2}\left[12\right]_{q^2}\left[15\right]_{q^2}} 
,\end{multline*} 
which, by Corollary~\ref{cor:A} and Lemma~\ref{lem:101215C}, 
is a polynomial in $q$ with 
non-negative integer coefficients.

\end{proof}

\section{Auxiliary results I}
\label{sec:aux1}

This section collects several auxiliary results which allow us to
reduce the problem of proving Theorem~\ref{thm:1}, or the
equivalent statement
\eqref{eq:1}, for the 26 exceptional groups listed in
Section~\ref{sec:prel} to a finite problem. While Lemmas~\ref{lem:2}
and \ref{lem:3} cover special choices of the parameters, 
Lemmas~\ref{lem:1} and
\ref{lem:6} afford an inductive procedure. More precisely, 
if we assume that we have already verified Theorem~\ref{thm:1} for all 
groups of smaller rank, then Lemmas~\ref{lem:1} and \ref{lem:6}, together 
with Lemmas~\ref{lem:2} and \ref{lem:7}, 
reduce the verification of Theorem~\ref{thm:1}
for the group that we are currently considering to a finite problem;
see Remark~\ref{rem:1}.
The final lemma of this section, Lemma~\ref{lem:8}, disposes of 
complex reflection groups with a special property satisfied by their degrees.

Let $p=am+b$, $0\le b<m$. We have
\begin{multline}
\phi^p\big((w_0;w_1,\dots,w_m)\big)\\
=(*;
c^{a+1}w_{m-b+1}c^{-a-1},c^{a+1}w_{m-b+2}c^{-a-1},
\dots,c^{a+1}w_{m}c^{-a-1},\\
c^{a}w_{1}c^{-a},\dots,
c^{a}w_{m-b}c^{-a}\big),
\label{eq:Aktion}
\end{multline}
where $*$ stands for the element of $W$ which is needed to 
complete the product of the components to $c$.

\begin{lemma} \label{lem:1}
It suffices to check \eqref{eq:1} for $p$ a divisor of $mh$.
More precisely, let $p$ be a divisor of $mh$, and let $k$ be another positive integer with 
$\gcd(k,mh/p)=1$, then we have
\begin{equation} \label{eq:2}
\Cat^m(W;q)\big\vert_{q=e^{2\pi i p/mh}}
= 
\Cat^m(W;q)\big\vert_{q=e^{2\pi i kp/mh}}
\end{equation}
and
\begin{equation} \label{eq:3}
\vert\Fix_{NC^m(W)}(\phi^{p})\vert =
\vert\Fix_{NC^m(W)}(\phi^{kp})\vert 	.
\end{equation}
\end{lemma}

\begin{proof}
For \eqref{eq:2}, this follows immediately from
\begin{equation} \label{eq:limit}
\lim_{q\to\zeta} \frac {[\alpha]_q} {[\beta]_q}=
\begin{cases}
\frac \alpha \beta&\text{if }\alpha\equiv\beta\equiv0\pmod d,\\
1&\text{otherwise},
\end{cases}	
\end{equation}
where $\zeta$ is a $d$-th root of unity and $\alpha,\beta$ are
non-negative integers such that 
$\alpha\equiv\beta\pmod d$.

In order to establish \eqref{eq:3}, suppose that 
$x\in\Fix_{NC^m(W)}(\phi^{p})$, that is, $x\in NC^m(W)$ and $\phi^p(x)=x$.
It obviously follows that $\phi^{kp}(x)=x$, so that
$x\in\Fix_{NC^m(W)}(\phi^{kp})$. To establish the converse, note that,  
if $\gcd(k,mh/p)=1$, then there exists $k'$ with $k'k\equiv 
1$~(mod~$\frac {mh} p$). It follows that, if 
$x\in\Fix_{NC^m(W)}(\phi^{kp})$, that is, if $x\in NC^m(W)$ and
$\phi^{kp}(x)=x$, then
$x=\phi^{k'kp}(x)=\phi^{p}(x)$, whence
$x\in\Fix_{NC^m(W)}(\phi^{p})$. 
\end{proof}

\begin{lemma} \label{lem:2}
Let $p$ be a divisor of $mh$.
If $p$ is divisible by $m$, then \eqref{eq:1} is true.
\end{lemma}


\begin{proof}
According to \eqref{eq:Aktion}, the action of $\phi^p$ on
$NC^m(W)$ is described by
\begin{equation*}
\phi^p\big((w_0;w_1,\dots,w_m)\big)
=(*;c^{p/m}w_{1}c^{-p/m},\dots,
c^{p/m}w_{m}c^{-p/m}\big).
\end{equation*}
Hence, if $(w_0;w_1,\dots,w_m)$ is fixed by $\phi^p$, then
each individual $w_i$ must be fixed under conjugation by $c^{p/m}$.

Using the notation $W'=\Cent_W(c^{p/m})$, 
the previous observation means that $w_i\in W'$, 
$i=1,2,\dots,m$. Springer \cite[Theorem~4.2]{SpriAA} 
(see also \cite[Theorem~11.24(iii)]{LeTaAA}) proved that
$W'$ is a well-generated complex reflection group whose degrees
coincide with those degrees of $W$ that are divisible by $mh/p$.
It was furthermore shown in \cite[Lemma~3.3]{BeReAA} that
\begin{equation} \label{eq:7}
NC(W)\cap W'=NC(W').
\end{equation}
Hence, the tuples $(w_0;w_1,\dots,w_m)$ fixed by $\phi^p$
are in fact identical with the elements of $NC^m(W')$, which
implies that
\begin{equation} \label{eq:6}
\vert\Fix_{NC^m(W)}(\phi^{p})\vert=\vert NC^m(W')\vert.
\end{equation}
Application of
Theorem~\ref{thm:2} with $W$ replaced by $W'$ and of the
``limit rule" \eqref{eq:limit} then yields that
\begin{equation} \label{eq:5}
\vert NC^m(W')\vert=
\underset{\frac {mh} p\mid d_i}{\prod_{1\le i\le n}} \frac {mh+d_i} {d_i}=\Cat^m(W;q)\big\vert_{q=e^{2\pi i p/mh}}.
\end{equation}
Combining \eqref{eq:6} and \eqref{eq:5}, we obtain \eqref{eq:1}.
This finishes the proof of the lemma.
\end{proof}

\begin{lemma} \label{lem:3}
Equation \eqref{eq:1} holds for all divisors $p$ of $m$.
\end{lemma}

\begin{proof}
Using \eqref{eq:limit} and the fact that the degrees of 
irreducible well-generated complex reflection groups satisfy
$d_i<h$ for all $i<n$, we see that
$$
\Cat^m(W;q)\big\vert_{q=e^{2\pi i p/mh}}=\begin{cases}
m+1&\text{if }m=p,\\
1&\text{if }m\ne p.
\end{cases}
$$
On the other hand, if $(w_0;w_1,\dots,w_m)$ is fixed by $\phi^p$,
then, because of the action \eqref{eq:Aktion}, we must have
$w_1=w_{p+1}=\dots=w_{m-p+1}$ and $w_1=cw_{m-p+1}c^{-1}$. In particular,
$w_1\in\Cent_W(c)$. By the theorem of Springer cited in the
proof of Lemma~\ref{lem:2}, the subgroup $\Cent_W(c)$
is itself a complex reflection group whose degrees are those degrees
of $W$ that are divisible by $h$. The only such degree is $h$
itself, hence $\Cent_W(c)$ is the cyclic group generated by
$c$. Moreover, by \eqref{eq:7}, we obtain that $w_1=\ep$,
the identity element of $W$, or
$w_1=c$. Therefore, for $m=p$ the set $\Fix_{NC^m(W)}(\phi^p)$
consists of the $m+1$ elements $(w_0;w_1,\dots,w_m)$ obtained by choosing 
$w_i=c$ for a particular $i$ between $0$ and $m$, all other $w_j$'s
being equal to $\ep$, while, for $m\ne p$, we have
$$\Fix_{NC^m(W)}(\phi^p)=\big\{(c;\ep,\dots,\ep)\big\},$$
whence the result.
\end{proof}

\begin{lemma} \label{lem:4}
Let $W$ be an irreducible well-generated complex reflection group 
all of whose degrees are divisible by $d$. 
Then each element of $W$ is fixed under conjugation by $c^{h/d}$.
\end{lemma}

\begin{proof}
By the theorem of Springer cited in the
proof of Lemma~\ref{lem:2}, the subgroup $W'=\Cent_W(c^{h/d})$
is itself a complex reflection group whose degrees are those degrees
of $W$ that are divisible by $d$. Thus, by our assumption, the
degrees of $W'$ coincide with the degrees of $W$, and hence $W'$ must
be equal to $W$.
Phrased differently, each element of $W$ is fixed under conjugation
by $c^{h/d}$, as claimed.
\end{proof}

\begin{lemma} \label{lem:6}
Let $W$ be an irreducible 
well-generated complex reflection group of rank $n$, 
and let $p=m_1h_1$ be a divisor of $mh$, where $m=m_1m_2$ and
$h=h_1h_2$. Without loss of generality, we assume that 
$\gcd(h_1,m_2)=1$. Suppose that
Theorem~{\em\ref{thm:1}} has already been verified for all 
irreducible well-generated
complex reflection groups with rank $<n$. 
If $h_2$ does not divide all degrees $d_i$,
then Equation~\eqref{eq:1} is satisfied.
\end{lemma}

\begin{proof}
Let us write $h_1=am_2+b$, with $0\le b<m_2$. The condition 
$\gcd(h_1,m_2)=1$ translates into $\gcd(b,m_2)=1$.
From \eqref{eq:Aktion}, we infer that
\begin{multline} \label{eq:m2Aktion}
\phi^p\big((w_0;w_1,\dots,w_m)\big)\\=
(*;
c^{a+1}w_{m-m_1b+1}c^{-a-1},c^{a+1}w_{m-m_1b+2}c^{-a-1},
\dots,c^{a+1}w_{m}c^{-a-1},\\
c^aw_{1}c^{-a},\dots,
c^aw_{m-m_1b}c^{-a}\big).
\end{multline}
Supposing that 
$(w_0;w_1,\dots,w_m)$ is fixed by $\phi^p$, we obtain
the system of equations
\begin{align*} 
w_i&=c^{a+1}w_{i+m-m_1b}c^{-a-1}, \quad i=1,2,\dots,m_1b,\\
w_i&=c^aw_{i-m_1b}c^{-a}, \quad i=m_1b+1,m_1b+2,\dots,m,
\end{align*}
which, after iteration, implies in particular that
$$
w_i=c^{b(a+1)+(m_2-b)a}w_ic^{-b(a+1)-(m_2-b)a}=c^{h_1}w_ic^{-h_1},
\quad i=1,2,\dots,m.
$$
It is at this point where we need $\gcd(b,m_2)=1$.
The last equation shows that each $w_i$, $i=1,2,\dots,m$, and thus
also $w_0$, lies in $\Cent_{W}(c^{h_1})$. 
By the theorem of Springer cited in the
proof of Lemma~\ref{lem:2}, this centraliser subgroup 
is itself a complex reflection group, $W'$ say, 
whose degrees are those degrees
of $W$ that are divisible by $h/h_1=h_2$. Since, by assumption, $h_2$
does not divide {\em all\/} degrees, $W'$ has 
rank strictly less than $n$. Again by assumption, we know that
Theorem~\ref{thm:1} is true for $W'$, so that in particular,
$$
\vert\Fix_{NC^m(W')}(\phi^{p})\vert = 
\Cat^m(W';q)\big\vert_{q=e^{2\pi i p/mh}}.
$$
The arguments above together with \eqref{eq:7} show that 
$\Fix_{NC^m(W)}(\phi^{p})=\Fix_{NC^m(W')}(\phi^{p})$.
On the other hand, using \eqref{eq:limit} 
it is straightforward to see that
$$\Cat^m(W;q)\big\vert_{q=e^{2\pi i p/mh}}=
\Cat^m(W';q)\big\vert_{q=e^{2\pi i p/mh}}.$$
This proves \eqref{eq:1} for our particular $p$, as required. 
\end{proof}

\begin{lemma} \label{lem:7}
Let $W$ be an irreducible 
well-generated complex reflection group of rank $n$, 
and let $p=m_1h_1$ be a divisor of $mh$, where $m=m_1m_2$ and
$h=h_1h_2$. We assume that 
$\gcd(h_1,m_2)=1$. If $m_2>n$ then
$$
\Fix_{NC^m(W)}(\phi^{p})=\big\{(c;\ep,\dots,\ep)\big\}.
$$
\end{lemma}

\begin{proof}
Let us suppose that $(w_0;w_1,\dots,w_m)\in 
\Fix_{NC^m(W)}(\phi^{p})$ and that there exists a $j\ge1$ such
that $w_j\ne\ep$. By \eqref{eq:m2Aktion}, it then follows
for such a $j$ that
also $w_k\ne\ep$ for all $k\equiv j-lm_1b$~(mod~$m$), where, as before,
$b$ is defined as the unique integer with $h_1=am_2+b$ and
$0\le b<m_2$. Since, by assumption, $\gcd(b,m_2)=1$, there are
exactly $m_2$ such $k$'s which are distinct mod~$m$. 
However, this implies that the sum of the absolute lengths
of the $w_i$'s, $0\le i\le m$, is at least $m_2>n$, a
contradiction to Remark~\ref{rem:0}.(2).
\end{proof}

\begin{remark} \label{rem:1}
(1)
If we put ourselves in the situation of the assumptions of
Lemma~\ref{lem:6}, then we may conclude that equation~\eqref{eq:1} 
only needs to be checked for pairs $(m_2,h_2)$ subject to the 
following restrictions: 
\begin{equation} 
m_2\ge2,\quad 
\gcd(h_1,m_2)=1,\quad
\text{and $h_2$ divides all degrees of $W$}. 
\label{eq:restr}
\end{equation}
Indeed, Lemmas~\ref{lem:2} and \ref{lem:6} 
together
imply that equation~\eqref{eq:1} is always satisfied 
in all other cases.

\smallskip
(2) Still putting ourselves in the situation of Lemma~\ref{lem:6},
if $m_2>n$ and $m_2h_2$ does not divide any of the degrees of $W$,
then equation~\eqref{eq:1} is satisfied. Indeed, Lemma~\ref{lem:7}
says that in this case the left-hand side of \eqref{eq:1} equals
$1$, while a straightforward computation using \eqref{eq:limit}
shows that in this case the right-hand side of \eqref{eq:1}
equals $1$ as well. 

\smallskip
(3) It should be observed that this leaves a finite
number of choices for $m_2$ to consider, whence a finite number of
choices for $(m_1,m_2,h_1,h_2)$. Altogether, there remains a finite
number of choices for $p=h_1m_1$ to be checked.
\end{remark}

\begin{lemma} \label{lem:8}
Let $W$ be an irreducible 
well-generated complex reflection group of rank $n$ 
with the property that $d_i\mid h$ for $i=1,2,\dots,n$.
Then Theorem~{\em\ref{thm:1}} is true for this group $W$.
\end{lemma}

\begin{proof}
By Lemma~\ref{lem:1}, we may restrict ourselves to divisors
$p$ of $mh$. 

Suppose that $e^{2\pi ip/mh}$ is a $d_i$-th root of unity
for some $i$. In other words, $mh/p$ divides $d_i$.
Since $d_i$ is a divisor of $h$ by assumption, 
the integer $mh/p$ also divides $h$. But this is equivalent to
saying that $m$ divides $p$, and equation \eqref{eq:1} holds by 
Lemma~\ref{lem:2}.

Now assume that $mh/p$ does not divide any of the $d_i$'s.
Then, by \eqref{eq:limit}, the right-hand side of \eqref{eq:1}
equals $1$.  On the other hand, $(c;\ep,\dots,\ep)$ is always an
element of $\Fix_{NC^m(W)}(\phi^{p})$. To see that there are no
others, we make appeal to the classification of all irreducible 
well-generated complex reflection groups, which we recalled in
Section~\ref{sec:prel}. Inspection reveals that all groups
satisfying the hypotheses of the lemma have rank $n\le2$.
Except for the groups contained in the infinite series $G(d,1,n)$ 
and $G(e,e,n)$ for which Theorem~\ref{thm:1} has been established in
\cite{KratCG}, these are the groups $G_5,G_6,G_9,G_{10},
G_{14},G_{17},G_{18},G_{21}$. We now discuss these groups case
by case, keeping the notation of Lemma~\ref{lem:6}.
In order to simplify the argument, we note that Lemma~\ref{lem:7}
implies that equation~\eqref{eq:1} holds if $m_2>2$, so that
in the following arguments we always may assume that $m_2=2$.


\smallskip
{\sc Case $G_5$}. The degrees are $6,12$, and
therefore Remark~\ref{rem:1}.(1) implies that equation~\eqref{eq:1}
is always satisfied.

\smallskip
{\sc Case $G_6$}. The degrees are $4,12$, and
therefore, according to Remark~\ref{rem:1}.(1), we need only consider
the case where $h_2=4$ and $m_2=2$, that is, $p=3m/2$. Then \eqref{eq:m2Aktion} becomes
\begin{equation} \label{eq:3m2Aktion}
\phi^p\big((w_0;w_1,\dots,w_m)\big)
=(*;
c^{2}w_{\frac m2+1}c^{-2},
c^{2}w_{\frac m2+2}c^{-2},
\dots,c^{2}w_{m}c^{-2},
cw_{1}c^{-1},\dots,
cw_{\frac m2}c^{-1}\big).
\end{equation}
If $(w_0;w_1,\dots,w_m)$ is fixed by $\phi^p$ and not
equal to $(c;\ep,\dots,\ep)$, there must exist an $i$
with $1\le i\le \frac m2$ such that 
$\ell_T(w_i)=\ell_T(w_{\frac {m} {2}+i})=1$,
$w_{\frac {m} {2}+i}=cw_ic^{-1}$,
$w_iw_{\frac {m} {2}+i}=w_icw_{i}c^{-1}=c$, and all $w_j$, 
with $j\ne i,\frac m2+i$,
equal $\ep$. However, 
with the help of the {\sl GAP} package {\tt CHEVIE}
\cite{chevAA,MichAA}, one verifies that there is no $w_i$
in $G_6$ such that
$$
\ell_T(w_i)=1\quad \text{and}\quad 
w_icw_{i}c^{-1}=c
$$
are simultaneously satisfied. Hence,
the left-hand side of \eqref{eq:1} is equal to $1$, as required.

\smallskip
{\sc Case $G_9$}. The degrees are $8,24$, and
therefore, according to Remark~\ref{rem:1}.(1), we need only consider
the case where $h_2=8$ and $m_2=2$, that is, $p=3m/2$. 
This is the same $p$ as for $G_6$. Again, 
{\tt CHEVIE}
finds no solution. Hence,
the left-hand side of \eqref{eq:1} is equal to $1$, as required.

\smallskip
{\sc Case $G_{10}$}. The degrees are $12,24$, and
therefore Remark~\ref{rem:1}.(1) implies that equation~\eqref{eq:1}
is always satisfied.

\smallskip
{\sc Case $G_{14}$}. The degrees are $6,24$, and
therefore Remark~\ref{rem:1}.(1) implies that equation~\eqref{eq:1}
is always satisfied.

\smallskip
{\sc Case $G_{17}$}. The degrees are $20,60$, and
therefore, according to Remark~\ref{rem:1}.(1), we need only consider
the cases where $h_2=20$ or $h_2=4$. 
In the first case, $p=3m/2$, which is
the same $p$ as for $G_6$. Again, 
{\tt CHEVIE} finds no solution. 
In the second case, $p=15m/2$. Then \eqref{eq:m2Aktion} becomes
\begin{multline} \label{eq:15m2Aktion}
\phi^p\big((w_0;w_1,\dots,w_m)\big)\\
=(*;
c^{8}w_{\frac m2+1}c^{-8},
c^{8}w_{\frac m2+2}c^{-8},
\dots,c^{8}w_{m}c^{-8},
c^7w_{1}c^{-7},\dots,
c^7w_{\frac m2}c^{-7}\big).
\end{multline}
By Lemma~\ref{lem:4}, every element of $NC(W)$ is fixed under 
conjugation by $c^3$, and, thus, on elements fixed by $\phi^p$,
the above action of $\phi^p$
reduces to the one in \eqref{eq:3m2Aktion}. This action was already
discussed in the first case.
Hence, in both cases,
the left-hand side of \eqref{eq:1} is equal to $1$, as required.

\smallskip
{\sc Case $G_{18}$}. The degrees are $30,60$, and
therefore Remark~\ref{rem:1}.(1) implies that equation~\eqref{eq:1}
is always satisfied.

\smallskip
{\sc Case $G_{21}$}. The degrees are $12,60$, and
therefore, according to Remark~\ref{rem:1}.(1), we need only consider
the cases where $h_2=12$ or $h_2=4$. 
In the first case, $p=5m/2$, 
so that \eqref{eq:m2Aktion} becomes
\begin{multline} \label{eq:5m2Aktion}
\phi^p\big((w_0;w_1,\dots,w_m)\big)\\
=(*;
c^{3}w_{\frac m2+1}c^{-3},
c^{3}w_{\frac m2+2}c^{-3},
\dots,c^{3}w_{m}c^{-3},
c^2w_{1}c^{-2},\dots,
c^2w_{\frac m2}c^{-2}\big).
\end{multline}
If $(w_0;w_1,\dots,w_m)$ is fixed by $\phi^p$ and not
equal to $(c;\ep,\dots,\ep)$, there must exist an $i$
with $1\le i\le \frac m2$ such that $\ell_T(w_i)=1$ and
$w_ic^2w_{i}c^{-2}=c$. However, 
with the help of the {\sl GAP} package {\tt CHEVIE}
\cite{chevAA,MichAA}, one verifies that there is no such
solution to this equation. 
In the second case, $p=15m/2$. Then \eqref{eq:m2Aktion} becomes
the action in \eqref{eq:15m2Aktion}.
By Lemma~\ref{lem:4}, every element of $NC(W)$ is fixed under 
conjugation by $c^5$, and, thus, on elements fixed by $\phi^p$,
the action of $\phi^p$ in
\eqref{eq:15m2Aktion}
reduces to the one in the first case. 
Hence, in both cases,
the left-hand side of \eqref{eq:1} is equal to $1$, as required.

\smallskip
This completes the proof of the lemma.
\end{proof}

\section{Case-by-case verification of Theorem~\ref{thm:1}}
\label{sec:Beweis1}

In the sequel we write $\zeta_d$ for a primitive $d$-th root of 
unity.

\subsection*{\sc Case $G_4$}
The degrees are $4,6$, and hence we have
$$
\Cat^m(G_4;q)=\frac 
{[6m+6]_q\, [6m+4]_q} 
{[6]_q\, [4]_q} .
$$
Let $\zeta$ be a $6m$-th root of unity. 
In what follows, 
we abbreviate the assertion that ``$\zeta$ is a primitive $d$-th root of
unity" as ``$\zeta=\zeta_d$."
The following cases on the right-hand side of \eqref{eq:1}
occur:
{\refstepcounter{equation}\label{eq:G4}}
\alphaeqn
\begin{align} 
\label{eq:G4.2}
\lim_{q\to\zeta}\Cat^m(G_4;q)&=m+1,
\quad\text{if }\zeta=\zeta_6,\zeta_3,\\
\label{eq:G4.3}
\lim_{q\to\zeta}\Cat^m(G_4;q)&=\tfrac {3m+2}2,
\quad\text{if }\zeta=\zeta_4,\ 2\mid m,\\
\label{eq:G4.4}
\lim_{q\to\zeta}\Cat^m(G_4;q)&=\Cat^m(G_4),
\quad\text{if }\zeta=-1\text{ or }\zeta=1,\\
\label{eq:G4.1}
\lim_{q\to\zeta}\Cat^m(G_4;q)&=1,
\quad\text{otherwise.}
\end{align}
\reseteqn

We must now prove that the left-hand side of \eqref{eq:1} in
each case agrees with the values exhibited in 
\eqref{eq:G4}. The only cases not covered by
Lemmas~\ref{lem:2} and \ref{lem:3} are the ones in \eqref{eq:G4.3}
and \eqref{eq:G4.1}. On the other hand, the only case left
to consider according to Remark~\ref{rem:1} is
the case where $h_2=m_2=2$, that is the case \eqref{eq:G4.3} where
$p=3m/2$. In particular, $m$ must be divisible by $2$.
The action of $\phi^p$ is the same as the one in 
\eqref{eq:3m2Aktion}. 
With the help of {\tt CHEVIE}, one finds that each of the
$3$ (complex) reflections in $G_4$ which are less than the (chosen) 
Coxeter element is a valid choice for $w_i$,
and each of these choices gives rise to $m/2$ elements in
$NC^m(G_4)$ since the index $i$ ranges from $1$ to $m/2$. 

Hence, in total, we obtain 
$1+3\frac m2=\frac {3m+2}2$ elements in
$\Fix_{NC^m(G_4)}(\phi^p)$, which agrees with the limit in
\eqref{eq:G4.3}.

\subsection*{\sc Case $G_8$}
The degrees are $8,12$, and hence we have
$$
\Cat^m(G_8;q)=\frac 
{[12m+12]_q\, [12m+8]_q} 
{[12]_q\, [8]_q} .
$$
Let $\zeta$ be a $12m$-th root of unity. 
The following cases on the right-hand side of \eqref{eq:1}
occur:
{\refstepcounter{equation}\label{eq:G8}}
\alphaeqn
\begin{align} 
\label{eq:G8.2}
\lim_{q\to\zeta}\Cat^m(G_8;q)&=m+1,
\quad\text{if }\zeta=\zeta_{12},\zeta_6,\zeta_3,\\
\label{eq:G8.3}
\lim_{q\to\zeta}\Cat^m(G_8;q)&=\tfrac {3m+2}2,
\quad\text{if }\zeta=\zeta_8,\ 2\mid m,\\
\label{eq:G8.4}
\lim_{q\to\zeta}\Cat^m(G_8;q)&=\Cat^m(G_8),
\quad\text{if }\zeta=\zeta_4,-1,1,\\
\label{eq:G8.1}
\lim_{q\to\zeta}\Cat^m(G_8;q)&=1,
\quad\text{otherwise.}
\end{align}
\reseteqn

We must now prove that the left-hand side of \eqref{eq:1} in
each case agrees with the values exhibited in 
\eqref{eq:G8}. The only cases not covered by
Lemmas~\ref{lem:2} and \ref{lem:3} are the ones in \eqref{eq:G8.3}
and \eqref{eq:G8.1}. On the other hand, the only case left
to consider according to Remark~\ref{rem:1} is
the case where $h_2=4$ and $m_2=2$, that is the case \eqref{eq:G8.3} where
$p=3m/2$. In particular, $m$ must be divisible by $2$.
The action of $\phi^p$ is the same as the one in 
\eqref{eq:3m2Aktion}. 
With the help of {\tt CHEVIE}, one finds that each of the
$3$ (complex) reflections in $G_8$ which are less than the (chosen) 
Coxeter element is a valid choice for $w_i$,
and each of these choices gives rise to $m/2$ elements in
$NC^m(G_8)$ since the index $i$ ranges from $1$ to $m/2$. 

Hence, in total, we obtain 
$1+3\frac m2=\frac {3m+2}2$ elements in
$\Fix_{NC^m(G_8)}(\phi^p)$, which agrees with the limit in
\eqref{eq:G8.3}.

\subsection*{\sc Case $G_{16}$}
The degrees are $20,30$, and hence we have
$$
\Cat^m(G_{16};q)=\frac 
{[30m+30]_q\, [30m+20]_q} 
{[30]_q\, [20]_q} .
$$
Let $\zeta$ be a $30m$-th root of unity. 
The following cases on the right-hand side of \eqref{eq:1}
occur:
{\refstepcounter{equation}\label{eq:G16}}
\alphaeqn
\begin{align} 
\label{eq:G16.2}
\lim_{q\to\zeta}\Cat^m(G_{16};q)&=m+1,
\quad\text{if }\zeta=\zeta_{30},\zeta_{15},\zeta_6,\zeta_3,\\
\label{eq:G16.3}
\lim_{q\to\zeta}\Cat^m(G_{16};q)&=\tfrac {3m+2}2,
\quad\text{if }\zeta=\zeta_{20},\zeta_4,\ 2\mid m,\\
\label{eq:G16.4}
\lim_{q\to\zeta}\Cat^m(G_{16};q)&=\Cat^m(G_{16}),
\quad\text{if }\zeta=\zeta_{10},\zeta_5,-1,1,\\
\label{eq:G16.1}
\lim_{q\to\zeta}\Cat^m(G_{16};q)&=1,
\quad\text{otherwise.}
\end{align}
\reseteqn

We must now prove that the left-hand side of \eqref{eq:1} in
each case agrees with the values exhibited in 
\eqref{eq:G16}. The only cases not covered by
Lemmas~\ref{lem:2} and \ref{lem:3} are the ones in \eqref{eq:G16.3}
and \eqref{eq:G16.1}. On the other hand, the only cases left
to consider according to Remark~\ref{rem:1} are
the cases where $h_2=10$ and $m_2=2$, respectively $h_2=m_2=2$.
Both cases belong to \eqref{eq:G16.3}. In the first case,
we have $p=3m/2$, while in the second case we have $p=15m/2$. 
In particular, $m$ must be divisible by $2$. In the first case,
the action of $\phi^p$ is the same as the one in 
\eqref{eq:3m2Aktion}. 
With the help of {\tt CHEVIE}, one finds that each of the
$3$ (complex) reflections in $G_{16}$ which are less than the (chosen) 
Coxeter element is a valid choice for $w_i$,
and each of these choices gives rise to $m/2$ elements in
$NC^m(G_{16})$ since the index $i$ ranges from $1$ to $m/2$. 
On the other hand, if $p=15m/2$, then the action of $\phi^p$ is the same as the one in \eqref{eq:15m2Aktion}. 
By Lemma~\ref{lem:4}, every element of $NC(W)$ is fixed under 
conjugation by $c^3$, and, thus, on elements fixed by $\phi^p$,
the action of $\phi^p$
reduces to the one in the first case. 

Hence, in total, we obtain 
$1+3\frac m2=\frac {3m+2}2$ elements in
$\Fix_{NC^m(G_{16})}(\phi^p)$, which agrees with the limit in
\eqref{eq:G16.3}.

\subsection*{\sc Case $G_{20}$}
The degrees are $12,30$, and hence we have
$$
\Cat^m(G_{20};q)=\frac 
{[30m+30]_q\, [30m+12]_q} 
{[30]_q\, [12]_q} .
$$
Let $\zeta$ be a $30m$-th root of unity. 
The following cases on the right-hand side of \eqref{eq:1}
occur:
{\refstepcounter{equation}\label{eq:G20}}
\alphaeqn
\begin{align} 
\label{eq:G20.2}
\lim_{q\to\zeta}\Cat^m(G_{20};q)&=m+1,
\quad\text{if }\zeta=\zeta_{30},\zeta_{15},\zeta_{10},\zeta_5,\\
\label{eq:G20.3}
\lim_{q\to\zeta}\Cat^m(G_{20};q)&=\tfrac {5m+2}2,
\quad\text{if }\zeta=\zeta_{12},\zeta_4,\ 2\mid m,\\
\label{eq:G20.4}
\lim_{q\to\zeta}\Cat^m(G_{20};q)&=\Cat^m(G_{20}),
\quad\text{if }\zeta=\zeta_6,\zeta_3,-1,1,\\
\label{eq:G20.1}
\lim_{q\to\zeta}\Cat^m(G_{20};q)&=1,
\quad\text{otherwise.}
\end{align}
\reseteqn

We must now prove that the left-hand side of \eqref{eq:1} in
each case agrees with the values exhibited in 
\eqref{eq:G20}. The only cases not covered by
Lemmas~\ref{lem:2} and \ref{lem:3} are the ones in \eqref{eq:G20.3}
and \eqref{eq:G20.1}. On the other hand, the only cases left
to consider according to Remark~\ref{rem:1} are
the cases where $h_2=6$ and $m_2=2$, respectively $h_2=m_2=2$.
Both cases belong to \eqref{eq:G20.3}. In the first case,
we have $p=5m/2$, while in the second case we have $p=15m/2$. 
In particular, $m$ must be divisible by $2$. In the first case,
the action of $\phi^p$ is the same as the one in 
\eqref{eq:5m2Aktion}. 
With the help of {\tt CHEVIE}, one finds that each of the
$5$ (complex) reflections in $G_{20}$ which are less than the (chosen) 
Coxeter element is a valid choice for $w_i$,
and each of these choices gives rise to $m/2$ elements in
$NC^m(G_{20})$ since the index $i$ ranges from $1$ to $m/2$. 
On the other hand, if $p=15m/2$, then the action of $\phi^p$ is the same as the one in \eqref{eq:15m2Aktion}. 
By Lemma~\ref{lem:4}, every element of $NC(W)$ is fixed under 
conjugation by $c^5$, and, thus, on elements fixed by $\phi^p$,
the action of $\phi^p$
reduces to the one in the first case. 

Hence, in total, we obtain 
$1+5\frac m2=\frac {5m+2}2$ elements in
$\Fix_{NC^m(G_{20})}(\phi^p)$, which agrees with the limit in
\eqref{eq:G20.3}.

\subsection*{\sc Case $G_{23}=H_3$}
The degrees are $2,6,10$, and hence we have
$$
\Cat^m(H_3;q)=\frac 
{[10m+10]_q\, [10m+6]_q\, [10m+2]_q} 
{[10]_q\, [6]_q\, [2]_q} .
$$
Let $\zeta$ be a $10m$-th root of unity. 
The following cases on the right-hand side of \eqref{eq:1}
occur:
{\refstepcounter{equation}\label{eq:H3}}
\alphaeqn
\begin{align} 
\label{eq:H3.2}
\lim_{q\to\zeta}\Cat^m(H_3;q)&=m+1,
\quad\text{if }\zeta=\zeta_{10},\zeta_5,\\
\label{eq:H3.3}
\lim_{q\to\zeta}\Cat^m(H_3;q)&=\tfrac {10m+6}6,
\quad\text{if }\zeta= \zeta_{6},\zeta_3,\ 3\mid m,\\
\label{eq:H3.4}
\lim_{q\to\zeta}\Cat^m(H_3;q)&=\Cat^m(H_3),
\quad\text{if }\zeta=-1\text{ or }\zeta=1,\\
\label{eq:H3.1}
\lim_{q\to\zeta}\Cat^m(H_3;q)&=1,
\quad\text{otherwise.}
\end{align}
\reseteqn

We must now prove that the left-hand side of \eqref{eq:1} in
each case agrees with the values exhibited in 
\eqref{eq:H3}. The only cases not covered by
Lemmas~\ref{lem:2} and \ref{lem:3} are the ones 
in \eqref{eq:H3.3} and \eqref{eq:H3.1}.
By Lemma~\ref{lem:1}, we are free to choose $p=5m/3$ if 
$\zeta=\zeta_6$, respectively $p=10m/3$ if 
$\zeta=\zeta_3$. In both cases, $m$ must be divisible by $3$.

We start with the case that $p=5m/3$.
From \eqref{eq:Aktion}, we infer
\begin{equation} \label{eq:5m3Aktion}
\phi^p\big((w_0;w_1,\dots,w_m)\big)=
(*;
c^{2}w_{\frac m3+1}c^{-2},c^{2}w_{\frac m3+2}c^{-2},
\dots,c^{2}w_{m}c^{-2},
cw_{1}c^{-1},\dots,
cw_{\frac m3}c^{-1}\big).
\end{equation}
Supposing that 
$(w_0;w_1,\dots,w_m)$ is fixed by $\phi^p$, we obtain
the system of equations
{\refstepcounter{equation}\label{eq:H3A}}
\alphaeqn
\begin{align} \label{eq:H3Aa}
w_i&=c^2w_{\frac m3+i}c^{-2}, \quad i=1,2,\dots,\tfrac {2m}3,\\
w_i&=cw_{i-\frac {2m}3}c^{-1}, \quad i=\tfrac {2m}3+1,\tfrac {2m}3+2,\dots,m.
\label{eq:H3Ab}
\end{align}
\reseteqn
There are two distinct possibilities for choosing
the $w_i$'s, $1\le i\le m$: 
either all the $w_i$'s are equal to $\ep$, or
there is an $i$ with $1\le i\le \frac m3$ such that
$$\ell_T(w_i)=\ell_T(w_{i+\frac m3})=\ell_T(w_{i+\frac {2m}3})=1.$$
Writing $t_1,t_2,t_3$ for $w_i,w_{i+\frac m3},w_{i+\frac {2m}3}$, 
respectively, the equations \eqref{eq:H3A} 
reduce to
{\refstepcounter{equation}\label{eq:H3B}}
\alphaeqn
\begin{align} \label{eq:H3Ba}
t_1&=c^2t_2c^{-2},\\
\label{eq:H3Bb}
t_2&=c^2t_3c^{-2},\\
t_3&=ct_1c^{-1}.
\label{eq:H3Bc}
\end{align}
\reseteqn
One of these equations is in fact superfluous: if we substitute
\eqref{eq:H3Bb} and \eqref{eq:H3Bc} in \eqref{eq:H3Ba}, then
we obtain $t_1=c^5t_1c^{-5}$ which is automatically satisfied due
to Lemma~\ref{lem:4} with $d=2$.

Since $(w_0;w_1,\dots,w_m)\in NC^m(H_3)$, we must have $t_1t_2t_3=c$.
Combining this with \eqref{eq:H3B}, we infer that
\begin{equation} \label{eq:H3D}
t_1(c^{-2}t_1c^2)(ct_1c^{-1})=c.
\end{equation}
With the help of Stembridge's {\sl Maple} package {\tt coxeter}
\cite{StemAZ}, one obtains five solutions for $t_1$ in this equation:
{\small
\begin{equation} \label{eq:H3sol1}
t_1\in\big\{[2],\,[3],\,[2,1,2],\,
[1,2,3,2,1],\,[1,3,2,1,2,1,3]\big\}.
\end{equation}}%
Here we have used the short notation of {\tt coxeter}:
if $\{s_1,s_2,s_3\}$ is a simple system of generators of $H_3$,
corresponding to the Dynkin diagram displayed in Figure~\ref{fig:H3},
then $[j_1,j_2,\dots,j_k]$ stands for the element
$s_{j_1}s_{j_2}\dots s_{j_k}$.

\begin{figure}[h]
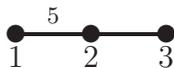

$$
\Einheit1cm
\Pfad(0,0),11\endPfad
\DickPunkt(0,0)
\DickPunkt(1,0)
\DickPunkt(2,0)
\Label\u{1}(0,0)
\Label\u{2}(1,0)
\Label\u{3}(2,0)
\Label\ro{\raise-20pt\hbox{\scriptsize 5}}(0,0)
\hskip2cm
$$ 
\caption{\scriptsize The Dynkin diagram for $H_3$}
\label{fig:H3}
\end{figure}

We claim that each of the above five solutions gives rise to 
$m/3$ elements of\break 
$\Fix_{NC^m(H_3)}(\phi^{p})$. Indeed, given $t_1$, the elements 
$t_2$ and $t_3$ can be computed by \eqref{eq:H3Ba} and
\eqref{eq:H3Bc}, and there are $m/3$ possibilities to choose the
index $i$ for $w_i$. 

In total, we obtain  
$1+5\frac m3=\frac {10m+6}6$ elements in
$\Fix_{NC^m(H_3)}(\phi^p)$, which agrees with the limit in
\eqref{eq:H3.3}.

In the case that $p=10m/3$, we infer from \eqref{eq:Aktion} that
\begin{multline} \label{eq:10m3Aktion}
\phi^p\big((w_0;w_1,\dots,w_m)\big)\\=
(*;
c^{4}w_{\frac {2m}3+1}c^{-4},c^{4}w_{\frac {2m}3+2}c^{-4},
\dots,c^{4}w_{m}c^{-4},
c^3w_{1}c^{-3},\dots,
c^3w_{\frac {2m}3}c^{-3}\big).
\end{multline}
Supposing that 
$(w_0;w_1,\dots,w_m)$ is fixed by $\phi^p$, we obtain
the system of equations
{\refstepcounter{equation}\label{eq:H3'A}}
\alphaeqn
\begin{align} \label{eq:H3'Aa}
w_i&=c^4w_{\frac {2m}3+i}c^{-4}, \quad i=1,2,\dots,\tfrac {m}3,\\
w_i&=c^3w_{i-\frac {m}3}c^{-3}, \quad i=\tfrac {m}3+1,\tfrac {m}3+2,\dots,m.
\label{eq:H3'Ab}
\end{align}
\reseteqn
There are two distinct possibilities for choosing
the $w_i$'s, $1\le i\le m$: 
either all the $w_i$'s are equal to $\ep$, or
there is an $i$ with $1\le i\le \frac m3$ such that
$$\ell_T(w_i)=\ell_T(w_{i+\frac m3})=\ell_T(w_{i+\frac {2m}3})=1.$$
Writing $t_1,t_2,t_3$ for $w_i,w_{i+\frac m3},w_{i+\frac {2m}3}$, 
respectively, the equations \eqref{eq:H3'A} 
reduce to
{\refstepcounter{equation}\label{eq:H3'B}}
\alphaeqn
\begin{align} \label{eq:H3'Ba}
t_1&=c^4t_3c^{-4},\\
\label{eq:H3'Bb}
t_2&=c^3t_1c^{-3},\\
t_3&=c^3t_2c^{-3}.
\label{eq:H3'Bc}
\end{align}
\reseteqn
One of these equations is in fact superfluous: if we substitute
\eqref{eq:H3'Bb} and \eqref{eq:H3'Bc} in \eqref{eq:H3'Ba}, then
we obtain $t_1=c^{10}t_1c^{-10}$ which is automatically satisfied 
since $c^{10}=\ep$.

Since $(w_0;w_1,\dots,w_m)\in NC^m(H_3)$, we must have $t_1t_2t_3=c$.
Combining this with \eqref{eq:H3'B}, we infer that
\begin{equation} \label{eq:H3'D}
t_1(c^{3}t_1c^{-3})(c^{-4}t_1c^{4})=c.
\end{equation}
Using that $c^5t_1c^{-5}=t_1$, due to Lemma~\ref{lem:4}
with $d=2$, we see that
this equation is equivalent with \eqref{eq:H3D}. 
Therefore, we are facing exactly the same enumeration 
problem here as for
$p=5m/3$, and, consequently, the number of solutions to \eqref{eq:H3'D} is the same, namely
$\frac {5m+3}3$, as required.

\smallskip
Finally, we turn to \eqref{eq:H3.1}. By Remark~\ref{rem:1},
the only choices for $h_2$ and $m_2$ to be considered
are $h_2=1$ and $m_2=3$, $h_2=m_2=2$, respectively $h_2=2$
and $m_2=3$. These correspond to the choices $p=10m/3$,
$p=5m/2$, respectively $p=5m/3$, out of which only $p=5m/2$
has not yet been discussed and belongs to the current case.
The corresponding action of $\phi^p$ is given by 
\eqref{eq:5m2Aktion}.
A computation with Stembridge's {\sl Maple} package {\tt coxeter}
\cite{StemAZ} finds no solution. 
Hence, 
the left-hand side of \eqref{eq:1} is equal to $1$, as required.

\subsection*{\sc Case $G_{24}$}
The degrees are $4,6,14$, and hence we have
$$
\Cat^m(G_{24};q)=\frac 
{[14m+14]_q\, [14m+6]_q\, [14m+4]_q} 
{[14]_q\, [6]_q\, [4]_q} .
$$
Let $\zeta$ be a $14m$-th root of unity. 
The following cases on the right-hand side of \eqref{eq:1}
occur:
{\refstepcounter{equation}\label{eq:G24}}
\alphaeqn
\begin{align} 
\label{eq:G24.2}
\lim_{q\to\zeta}\Cat^m(G_{24};q)&=m+1,
\quad\text{if }\zeta=\zeta_{14},\zeta_7,\\
\label{eq:G24.3}
\lim_{q\to\zeta}\Cat^m(G_{24};q)&=\tfrac {7m+3}3,
\quad\text{if }\zeta=\zeta_{6},\zeta_3,\ 3\mid m,\\
\label{eq:G24.4}
\lim_{q\to\zeta}\Cat^m(G_{24};q)&=\tfrac {7m+2}2,
\quad\text{if }\zeta=\zeta_4,\ 2\mid m,\\
\label{eq:G24.5}
\lim_{q\to\zeta}\Cat^m(G_{24};q)&=\Cat^m(G_{24}),
\quad\text{if }\zeta=-1\text{ or }\zeta=1,\\
\label{eq:G24.1}
\lim_{q\to\zeta}\Cat^m(G_{24};q)&=1,
\quad\text{otherwise.}
\end{align}
\reseteqn

We must now prove that the left-hand side of \eqref{eq:1} in
each case agrees with the values exhibited in 
\eqref{eq:G24}. The only cases not covered by
Lemmas~\ref{lem:2} and \ref{lem:3} are the ones in \eqref{eq:G24.3},
\eqref{eq:G24.4},
and \eqref{eq:G24.1}. 

We first consider \eqref{eq:G24.3}. 
By Lemma~\ref{lem:1}, we are free to choose $p=7m/3$ if 
$\zeta=\zeta_6$, respectively $p=14m/3$ if 
$\zeta=\zeta_3$. In both cases, $m$ must be divisible by $3$.

We start with the case that $p=7m/3$.
From \eqref{eq:Aktion}, we infer
\begin{multline*}
\phi^p\big((w_0;w_1,\dots,w_m)\big)\\=
(*;
c^{3}w_{\frac {2m}3+1}c^{-3},c^{3}w_{\frac {2m}3+2}c^{-3},
\dots,c^{3}w_{m}c^{-3},
c^2w_{1}c^{-2},\dots,
c^2w_{\frac {2m}3}c^{-2}\big).
\end{multline*}
Supposing that 
$(w_0;w_1,\dots,w_m)$ is fixed by $\phi^p$, we obtain
the system of equations
{\refstepcounter{equation}\label{eq:G24A}}
\alphaeqn
\begin{align} \label{eq:G24Aa}
w_i&=c^3w_{\frac {2m}3+i}c^{-3}, \quad i=1,2,\dots,\tfrac {m}3,\\
w_i&=c^2w_{i-\frac {m}3}c^{-2}, \quad i=\tfrac {m}3+1,\tfrac {m}3+2,\dots,m.
\label{eq:G24Ab}
\end{align}
\reseteqn
There are two distinct possibilities for choosing
the $w_i$'s, $1\le i\le m$: 
either all the $w_i$'s are equal to $\ep$, or
there is an $i$ with $1\le i\le \frac m3$ such that
$$\ell_T(w_i)=\ell_T(w_{i+\frac m3})=\ell_T(w_{i+\frac {2m}3})=1.$$
Writing $t_1,t_2,t_3$ for $w_i,w_{i+\frac m3},w_{i+\frac {2m}3}$, 
respectively, the equations \eqref{eq:G24A} 
reduce to
{\refstepcounter{equation}\label{eq:G24B}}
\alphaeqn
\begin{align} \label{eq:G24Ba}
t_1&=c^3t_3c^{-3},\\
\label{eq:G24Bb}
t_2&=c^2t_1c^{-2},\\
t_3&=c^2t_2c^{-2}.
\label{eq:G24Bc}
\end{align}
\reseteqn
One of these equations is in fact superfluous: if we substitute
\eqref{eq:G24Bb} and \eqref{eq:G24Bc} in \eqref{eq:G24Ba}, then
we obtain $t_1=c^7t_1c^{-7}$ which is automatically satisfied due
to Lemma~\ref{lem:4} with $d=2$.

Since $(w_0;w_1,\dots,w_m)\in NC^m(G_{24})$, we must have $t_1t_2t_3=c$.
Combining this with \eqref{eq:G24B}, we infer that
\begin{equation} \label{eq:G24D}
t_1(c^{2}t_1c^{-2})(c^4t_1c^{-4})=c.
\end{equation}
With the help of {\tt CHEVIE}, 
one obtains 7 solutions for $t_1$ in this equation:
{\small
\begin{equation} \label{eq:G24sol1}
t_1\in\big\{[ 1 ],\,
[ 2 ],\,
[ 3 ],\,
[ 15 ],\,
[ 16 ],\,
[ 19 ],\,
[ 21 ]\big\},
\end{equation}}%
each of them giving rise to $m/3$ elements of
$\Fix_{NC^m(G_{24})}(\phi^{p})$ since $i$ ranges from $1$ to $m/3$.
Here we have used the short notation of {\tt CHEVIE}:
$[j_1,j_2,\dots,j_k]$ stands for the element
$r_{j_1}r_{j_2}\dots r_{j_k}$, where $r_i$ is the $i$-th (complex)
reflection corresponding to the $i$-th root in the internal
ordering of the roots of $G_{24}$ in {\tt CHEVIE}.

In total, we obtain  
$1+7\frac m3=\frac {7m+3}3$ elements in
$\Fix_{NC^m(G_{24})}(\phi^p)$, which agrees with the limit in
\eqref{eq:G24.3}.

In the case that $p=14m/3$, we infer from \eqref{eq:Aktion} that
\begin{multline*}
\phi^p\big((w_0;w_1,\dots,w_m)\big)\\=
(*;
c^{5}w_{\frac {m}3+1}c^{-5},c^{5}w_{\frac {m}3+2}c^{-5},
\dots,c^{5}w_{m}c^{-5},
c^4w_{1}c^{-4},\dots,
c^4w_{\frac {m}3}c^{-4}\big).
\end{multline*}
Supposing that 
$(w_0;w_1,\dots,w_m)$ is fixed by $\phi^p$, we obtain
the system of equations
{\refstepcounter{equation}\label{eq:G24'A}}
\alphaeqn
\begin{align} \label{eq:G24'Aa}
w_i&=c^5w_{\frac {m}3+i}c^{-5}, \quad i=1,2,\dots,\tfrac {2m}3,\\
w_i&=c^4w_{i-\frac {2m}3}c^{-4}, \quad i=\tfrac {2m}3+1,\tfrac {2m}3+2,\dots,m.
\label{eq:G24'Ab}
\end{align}
\reseteqn
There are two distinct possibilities for choosing
the $w_i$'s, $1\le i\le m$: 
either all the $w_i$'s are equal to $\ep$, or
there is an $i$ with $1\le i\le \frac m3$ such that
$$\ell_T(w_i)=\ell_T(w_{i+\frac m3})=\ell_T(w_{i+\frac {2m}3})=1.$$
Writing $t_1,t_2,t_3$ for $w_i,w_{i+\frac m3},w_{i+\frac {2m}3}$, 
respectively, the equations \eqref{eq:G24'A} 
reduce to
{\refstepcounter{equation}\label{eq:G24'B}}
\alphaeqn
\begin{align} \label{eq:G24'Ba}
t_1&=c^5t_2c^{-5},\\
\label{eq:G24'Bb}
t_2&=c^5t_3c^{-5},\\
t_3&=c^4t_1c^{-4}.
\label{eq:G24'Bc}
\end{align}
\reseteqn
One of these equations is in fact superfluous: if we substitute
\eqref{eq:G24'Bb} and \eqref{eq:G24'Bc} in \eqref{eq:G24'Ba}, then
we obtain $t_1=c^{14}t_1c^{-14}$ which is automatically satisfied 
since $c^{14}=\ep$.

Since $(w_0;w_1,\dots,w_m)\in NC^m(G_{24})$, we must have $t_1t_2t_3=c$.
Combining this with \eqref{eq:G24'B}, we infer that
\begin{equation} \label{eq:G24'D}
t_1(c^{9}t_1c^{-9})(c^{-4}t_1c^{4})=c.
\end{equation}
Using that $c^7t_1c^{-7}=t_1$, due to Lemma~\ref{lem:4}
with $d=2$, we see that
this equation is equivalent with \eqref{eq:G24D}. 
Therefore, we are facing exactly the same enumeration 
problem here as for
$p=7m/3$, and, consequently, the number of solutions to \eqref{eq:G24'D} is the same, namely
$\frac {7m+3}3$, as required.

\smallskip
Our next case is \eqref{eq:G24.4}. 
By Lemma~\ref{lem:1}, we are free to choose $p=7m/2$. In particular, 
$m$ must be divisible by $2$.
From \eqref{eq:Aktion}, we infer
\begin{multline*}
\phi^p\big((w_0;w_1,\dots,w_m)\big)\\=
(*;
c^{4}w_{\frac {m}2+1}c^{-4},c^{4}w_{\frac {m}2+2}c^{-4},
\dots,c^{4}w_{m}c^{-4},
c^3w_{1}c^{-3},\dots,
c^3w_{\frac {m}2}c^{-3}\big).
\end{multline*}
Supposing that 
$(w_0;w_1,\dots,w_m)$ is fixed by $\phi^p$, we obtain
the system of equations
{\refstepcounter{equation}\label{eq:G24''A}}
\alphaeqn
\begin{align} \label{eq:G24''Aa}
w_i&=c^4w_{\frac {m}2+i}c^{-4}, \quad i=1,2,\dots,\tfrac {m}2,\\
w_i&=c^3w_{i-\frac {m}2}c^{-3}, \quad i=\tfrac {m}2+1,\tfrac {m}2+2,\dots,m.
\label{eq:G24''Ab}
\end{align}
\reseteqn
There are two distinct possibilities for choosing
the $w_i$'s, $1\le i\le m$: 
either all the $w_i$'s are equal to $\ep$, or
there is an $i$ with $1\le i\le \frac m2$ such that
$$\ell_T(w_i)=\ell_T(w_{i+\frac m2})=1.$$
Writing $t_1,t_2$ for $w_i,w_{i+\frac m2}$, 
respectively, the equations \eqref{eq:G24''A} 
reduce to
{\refstepcounter{equation}\label{eq:G24''B}}
\alphaeqn
\begin{align} \label{eq:G24''Ba}
t_1&=c^4t_2c^{-4},\\
\label{eq:G24''Bb}
t_2&=c^3t_1c^{-3}.
\end{align}
\reseteqn
One of these equations is in fact superfluous: if we substitute
\eqref{eq:G24''Bb} in \eqref{eq:G24''Ba}, then
we obtain $t_1=c^7t_1c^{-7}$ which is automatically satisfied due
to Lemma~\ref{lem:4} with $d=2$.

Since $(w_0;w_1,\dots,w_m)\in NC^m(G_{24})$, we must have $t_1t_2\le_T c$,
where $\le_T$ is the partial order defined in \eqref{eq:absord}.
Combining this with \eqref{eq:G24''B}, we infer that
\begin{equation} \label{eq:G24''D}
t_1(c^{3}t_1c^{-3})\le_T c.
\end{equation}
With the help of {\tt CHEVIE}, 
one obtains 7 solutions for $t_1$ in this relation:
{\small
\begin{equation} \label{eq:G24sol2}
t_1\in\big\{[ 5 ],\,
[ 6 ],\,
[ 7 ],\,
[ 9 ],\,
[ 12 ],\,
[ 29 ],\,
[ 32 ]\big\},
\end{equation}}%
each of them giving rise to $m/2$ elements of
$\Fix_{NC^m(G_{24})}(\phi^{p})$ since $i$ ranges from $1$ to $m/2$.
Here we have used again the short notation of {\tt CHEVIE}
referring to the internal
ordering of the roots of $G_{24}$ in {\tt CHEVIE}.

In total, we obtain  
$1+7\frac m2=\frac {7m+2}2$ elements in
$\Fix_{NC^m(G_{24})}(\phi^p)$, which agrees with the limit in
\eqref{eq:G24.4}.

\smallskip
Finally, we turn to \eqref{eq:G24.1}. By Remark~\ref{rem:1},
the only choices for $h_2$ and $m_2$ to be considered
are $h_2=1$ and $m_2=3$, $h_2=m_2=2$, and $h_2=2$
and $m_2=3$. These correspond to the choices $p=14m/3$,
$p=7m/2$, respectively $p=7m/3$, all of which have already been 
discussed as they do not belong to \eqref{eq:G24.1}. Hence, 
\eqref{eq:1} must necessarily hold, as required.

\subsection*{\sc Case $G_{25}$}
The degrees are $6,9,12$, and hence we have
$$
\Cat^m(G_{25};q)=\frac 
{[12m+12]_q\, [12m+9]_q\, [12m+6]_q} 
{[12]_q\, [9]_q\, [6]_q} .
$$
Let $\zeta$ be a $12m$-th root of unity. 
The following cases on the right-hand side of \eqref{eq:1}
occur:
{\refstepcounter{equation}\label{eq:G25}}
\alphaeqn
\begin{align} 
\label{eq:G25.2}
\lim_{q\to\zeta}\Cat^m(G_{25};q)&=m+1,
\quad\text{if }\zeta=\zeta_{12},\zeta_4,\\
\label{eq:G25.3}
\lim_{q\to\zeta}\Cat^m(G_{25};q)&=\tfrac {4m+3}3,
\quad\text{if }\zeta=\zeta_{9},\ 3\mid m,\\
\label{eq:G25.4}
\lim_{q\to\zeta}\Cat^m(G_{25};q)&=(m+1)(2m+1),
\quad\text{if }\zeta=\zeta_6,-1\\
\label{eq:G25.5}
\lim_{q\to\zeta}\Cat^m(G_{25};q)&=\Cat^m(G_{25}),
\quad\text{if }\zeta=\zeta_3,1,\\
\label{eq:G25.1}
\lim_{q\to\zeta}\Cat^m(G_{25};q)&=1,
\quad\text{otherwise.}
\end{align}
\reseteqn

We must now prove that the left-hand side of \eqref{eq:1} in
each case agrees with the values exhibited in 
\eqref{eq:G25}. The only cases not covered by
Lemmas~\ref{lem:2} and \ref{lem:3} are the ones in \eqref{eq:G25.3}
and \eqref{eq:G25.1}. 

We first consider \eqref{eq:G25.3}. 
By Lemma~\ref{lem:1}, we are free to choose $p=4m/3$. 
In particular, $m$ must be divisible by $3$.
From \eqref{eq:Aktion}, we infer
\begin{multline*}
\phi^p\big((w_0;w_1,\dots,w_m)\big)\\=
(*;
c^{2}w_{\frac {2m}3+1}c^{-2},c^{2}w_{\frac {2m}3+2}c^{-2},
\dots,c^{2}w_{m}c^{-2},
cw_{1}c^{-1},\dots,
cw_{\frac {2m}3}c^{-1}\big).
\end{multline*}
Supposing that 
$(w_0;w_1,\dots,w_m)$ is fixed by $\phi^p$, we obtain
the system of equations
{\refstepcounter{equation}\label{eq:G25A}}
\alphaeqn
\begin{align} \label{eq:G25Aa}
w_i&=c^2w_{\frac {2m}3+i}c^{-2}, \quad i=1,2,\dots,\tfrac {m}3,\\
w_i&=cw_{i-\frac {m}3}c^{-1}, \quad i=\tfrac {m}3+1,\tfrac {m}3+2,\dots,m.
\label{eq:G25Ab}
\end{align}
\reseteqn
There are two distinct possibilities for choosing
the $w_i$'s, $1\le i\le m$: 
either all the $w_i$'s are equal to $\ep$, or
there is an $i$ with $1\le i\le \frac m3$ such that
$$\ell_T(w_i)=\ell_T(w_{i+\frac m3})=\ell_T(w_{i+\frac {2m}3})=1.$$
Writing $t_1,t_2,t_3$ for $w_i,w_{i+\frac m3},w_{i+\frac {2m}3}$, 
respectively, the equations \eqref{eq:G25A} 
reduce to
{\refstepcounter{equation}\label{eq:G25B}}
\alphaeqn
\begin{align} \label{eq:G25Ba}
t_1&=c^2t_3c^{-2},\\
\label{eq:G25Bb}
t_2&=ct_1c^{-1},\\
t_3&=ct_2c^{-1}.
\label{eq:G25Bc}
\end{align}
\reseteqn
One of these equations is in fact superfluous: if we substitute
\eqref{eq:G25Bb} and \eqref{eq:G25Bc} in \eqref{eq:G25Ba}, then
we obtain $t_1=c^4t_1c^{-4}$ which is automatically satisfied due
to Lemma~\ref{lem:4} with $d=3$.

Since $(w_0;w_1,\dots,w_m)\in NC^m(G_{25})$, we must have $t_1t_2t_3=c$.
Combining this with \eqref{eq:G25B}, we infer that
\begin{equation} \label{eq:G25D}
t_1(ct_1c^{-1})(c^2t_1c^{-2})=c.
\end{equation}
With the help of {\tt CHEVIE}, 
one obtains four solutions for $t_1$ in this equation:
{\small
\begin{equation} \label{eq:G25sol1}
t_1\in\big\{[ 1 ],\,
[ 2 ],\,
[ 3 ],\,
[ 14 ]\big\},
\end{equation}}%
each of them giving rise to $m/3$ elements of
$\Fix_{NC^m(G_{25})}(\phi^{p})$ since $i$ ranges from $1$ to $m/3$.
Here we have used again the short notation of {\tt CHEVIE}
referring to the internal
ordering of the roots of $G_{25}$ in {\tt CHEVIE}.

In total, we obtain  
$1+4\frac m3=\frac {4m+3}3$ elements in
$\Fix_{NC^m(G_{25})}(\phi^p)$, which agrees with the limit in
\eqref{eq:G25.3}.

\smallskip
Finally, we turn to \eqref{eq:G25.1}. By Remark~\ref{rem:1},
the only choice for $h_2$ and $m_2$ to be considered
are $h_2=m_2=3$. This corresponds to the choice $p=4m/3$, 
which has already been discussed as they do not belong to 
\eqref{eq:G25.1}. Hence, 
\eqref{eq:1} must necessarily hold, as required. 

\subsection*{\sc Case $G_{26}$}
The degrees are $6,12,18$, and hence we have
$$
\Cat^m(G_{26};q)=\frac 
{[18m+18]_q\, [18m+12]_q\, [18m+6]_q} 
{[18]_q\, [12]_q\, [6]_q} .
$$
Let $\zeta$ be a $14m$-th root of unity. 
The following cases on the right-hand side of \eqref{eq:1}
occur:
{\refstepcounter{equation}\label{eq:G26}}
\alphaeqn
\begin{align} 
\label{eq:G26.2}
\lim_{q\to\zeta}\Cat^m(G_{26};q)&=m+1,
\quad\text{if }\zeta=\zeta_{18},\zeta_9,\\
\label{eq:G26.3}
\lim_{q\to\zeta}\Cat^m(G_{26};q)&=\tfrac {3m+2}2,
\quad\text{if }\zeta=\zeta_{12},\zeta_4,\ 2\mid m,\\
\label{eq:G26.5}
\lim_{q\to\zeta}\Cat^m(G_{26};q)&=\Cat^m(G_{26}),
\quad\text{if }\zeta=\zeta_6,\zeta_3,-1,1,\\
\label{eq:G26.1}
\lim_{q\to\zeta}\Cat^m(G_{26};q)&=1,
\quad\text{otherwise.}
\end{align}
\reseteqn

We must now prove that the left-hand side of \eqref{eq:1} in
each case agrees with the values exhibited in 
\eqref{eq:G26}. The only cases not covered by
Lemmas~\ref{lem:2} and \ref{lem:3} are the ones in \eqref{eq:G26.3}
and \eqref{eq:G26.1}. 

We first consider \eqref{eq:G26.3}. 
By Lemma~\ref{lem:1}, we are free to choose $p=3m/2$ if 
$\zeta=\zeta_{12}$, respectively $p=9m/2$ if 
$\zeta=\zeta_4$. In both cases, $m$ must be divisible by $2$.

We start with the case that $p=3m/2$.
From \eqref{eq:Aktion}, we infer
$$
\phi^p\big((w_0;w_1,\dots,w_m)\big)\\=
(*;
c^{2}w_{\frac {m}2+1}c^{-2},c^{2}w_{\frac {m}2+2}c^{-2},
\dots,c^{2}w_{m}c^{-2},
cw_{1}c^{-1},\dots,
cw_{\frac {m}2}c^{-1}\big).
$$
Supposing that 
$(w_0;w_1,\dots,w_m)$ is fixed by $\phi^p$, we obtain
the system of equations
{\refstepcounter{equation}\label{eq:G26A}}
\alphaeqn
\begin{align} \label{eq:G26Aa}
w_i&=c^2w_{\frac {m}2+i}c^{-2}, \quad i=1,2,\dots,\tfrac {m}2,\\
w_i&=cw_{i-\frac {m}2}c^{-1}, \quad i=\tfrac {m}2+1,\tfrac {m}2+2,\dots,m.
\label{eq:G26Ab}
\end{align}
\reseteqn
There are two distinct possibilities for choosing
the $w_i$'s, $1\le i\le m$: 
either all the $w_i$'s are equal to $\ep$, or
there is an $i$ with $1\le i\le \frac m2$ such that
$$\ell_T(w_i)=\ell_T(w_{i+\frac m2})=1.$$
Writing $t_1,t_2$ for $w_i,w_{i+\frac m2}$, 
respectively, the equations \eqref{eq:G26A} 
reduce to
{\refstepcounter{equation}\label{eq:G26B}}
\alphaeqn
\begin{align} \label{eq:G26Ba}
t_1&=c^2t_2c^{-2},\\
\label{eq:G26Bb}
t_2&=c^1t_1c^{-1}.
\end{align}
\reseteqn
One of these equations is in fact superfluous: if we substitute
\eqref{eq:G26Bb} in \eqref{eq:G26Ba}, then
we obtain $t_1=c^3t_1c^{-3}$ which is automatically satisfied due
to Lemma~\ref{lem:4} with $d=6$.

Since $(w_0;w_1,\dots,w_m)\in NC^m(G_{26})$, we must have $t_1t_2\le_T
c$.
Combining this with \eqref{eq:G26B}, we infer that
\begin{equation} \label{eq:G26D}
t_1(ct_1c^{-1})\le_T c.
\end{equation}
With the help of {\tt CHEVIE}, 
one obtains three solutions for $t_1$ in this equation:
{\small
\begin{equation*} 
t_1\in\big\{[ 2 ],\,
[ 3 ],\,
[ 12 ]\big\},
\end{equation*}}%
each of them giving rise to $m/2$ elements of
$\Fix_{NC^m(G_{26})}(\phi^{p})$ since $i$ ranges from $1$ to $m/2$.
Here we have again used the short notation of {\tt CHEVIE}
referring to the internal
ordering of the roots of $G_{26}$ in {\tt CHEVIE}.

In total, we obtain  
$1+3\frac m2=\frac {3m+2}2$ elements in
$\Fix_{NC^m(G_{26})}(\phi^p)$, which agrees with the limit in
\eqref{eq:G26.3}.

In the case that $p=9m/2$, we infer from \eqref{eq:Aktion} that
\begin{multline*}
\phi^p\big((w_0;w_1,\dots,w_m)\big)\\=
(*;
c^{5}w_{\frac {m}2+1}c^{-5},c^{5}w_{\frac {m}2+2}c^{-5},
\dots,c^{5}w_{m}c^{-5},
c^4w_{1}c^{-4},\dots,
c^4w_{\frac {m}2}c^{-4}\big).
\end{multline*}
Supposing that 
$(w_0;w_1,\dots,w_m)$ is fixed by $\phi^p$, we obtain
the system of equations
{\refstepcounter{equation}\label{eq:G26'A}}
\alphaeqn
\begin{align} \label{eq:G26'Aa}
w_i&=c^5w_{\frac {m}2+i}c^{-5}, \quad i=1,2,\dots,\tfrac {m}2,\\
w_i&=c^4w_{i-\frac {m}2}c^{-4}, \quad i=\tfrac {m}2+1,\tfrac {m}2+2,\dots,m.
\label{eq:G26'Ab}
\end{align}
\reseteqn
There are two distinct possibilities for choosing
the $w_i$'s, $1\le i\le m$: 
either all the $w_i$'s are equal to $\ep$, or
there is an $i$ with $1\le i\le \frac m3$ such that
$$\ell_T(w_i)=\ell_T(w_{i+\frac m2})=1.$$
Writing $t_1,t_2$ for $w_i,w_{i+\frac m2}$, 
respectively, the equations \eqref{eq:G26'A} 
reduce to
{\refstepcounter{equation}\label{eq:G26'B}}
\alphaeqn
\begin{align} \label{eq:G26'Ba}
t_1&=c^5t_2c^{-5},\\
\label{eq:G26'Bb}
t_2&=c^4t_1c^{-4}.
\end{align}
\reseteqn
One of these equations is in fact superfluous: if we substitute
\eqref{eq:G26'Bb} in \eqref{eq:G26'Ba}, then
we obtain $t_1=c^{9}t_1c^{-9}$ which is automatically satisfied 
due to Lemma~\ref{lem:4} with $d=2$.
Since $(w_0;w_1,\dots,w_m)\in NC^m(G_{26})$, we must have $t_1t_2\le_T
c$.
Combining this with \eqref{eq:G26'B}, we infer that
\begin{equation} \label{eq:G26'D}
t_1(c^{4}t_1c^{-4})\le_T c.
\end{equation}
Using that $c^3t_1c^{-3}=t_1$, due to Lemma~\ref{lem:4}
with $d=6$, we see that
this equation is equivalent with \eqref{eq:G26D}. 
Therefore, we are facing exactly the same enumeration 
problem here as for
$p=3m/2$, and, consequently, the number of solutions to \eqref{eq:G26'D} is the same, namely
$\frac {3m+2}2$, as required.

\smallskip
Finally, we turn to \eqref{eq:G26.1}. By Remark~\ref{rem:1},
the only choices for $h_2$ and $m_2$ to be considered
are $h_2=6$ and $m_2=2$, respectively $h_2=m_2=2$. 
These correspond to the choices $p=3m/2$,
respectively $p=9m/2$, all of which have already been 
discussed as they do not belong to \eqref{eq:G26.1}. Hence, 
\eqref{eq:1} must necessarily hold, as required.

\subsection*{\sc Case $G_{27}$}
The degrees are $6,12,30$, and hence we have
$$
\Cat^m(G_{27};q)=\frac 
{[30m+30]_q\, [30m+12]_q\, [30m+6]_q} 
{[30]_q\, [12]_q\, [6]_q} .
$$
Let $\zeta$ be a $14m$-th root of unity. 
The following cases on the right-hand side of \eqref{eq:1}
occur:
{\refstepcounter{equation}\label{eq:G27}}
\alphaeqn
\begin{align} 
\label{eq:G27.2}
\lim_{q\to\zeta}\Cat^m(G_{27};q)&=m+1,
\quad\text{if }\zeta=\zeta_{30},\zeta_{15},\zeta_{10},\zeta_5,\\
\label{eq:G27.3}
\lim_{q\to\zeta}\Cat^m(G_{27};q)&=\tfrac {5m+2}2,
\quad\text{if }\zeta=\zeta_{12},\zeta_4,\ 2\mid m,\\
\label{eq:G27.4}
\lim_{q\to\zeta}\Cat^m(G_{27};q)&=\Cat^m(G_{27}),
\quad\text{if }\zeta=\zeta_6,\zeta_3,-1,1,\\
\label{eq:G27.1}
\lim_{q\to\zeta}\Cat^m(G_{27};q)&=1,
\quad\text{otherwise.}
\end{align}
\reseteqn

We must now prove that the left-hand side of \eqref{eq:1} in
each case agrees with the values exhibited in 
\eqref{eq:G27}. The only cases not covered by
Lemmas~\ref{lem:2} and \ref{lem:3} are the ones in \eqref{eq:G27.3}
and \eqref{eq:G27.1}. 

We first consider \eqref{eq:G27.3}. 
By Lemma~\ref{lem:1}, we are free to choose $p=5m/2$ if 
$\zeta=\zeta_{12}$, respectively $p=15m/2$ if 
$\zeta=\zeta_4$. In both cases, $m$ must be divisible by $2$.

We start with the case that $p=5m/2$.
From \eqref{eq:Aktion}, we infer
\begin{multline*}
\phi^p\big((w_0;w_1,\dots,w_m)\big)\\=
(*;
c^{3}w_{\frac {m}2+1}c^{-3},c^{3}w_{\frac {m}2+2}c^{-3},
\dots,c^{3}w_{m}c^{-3},
c^2w_{1}c^{-2},\dots,
c^2w_{\frac {m}2}c^{-2}\big).
\end{multline*}
Supposing that 
$(w_0;w_1,\dots,w_m)$ is fixed by $\phi^p$, we obtain
the system of equations
{\refstepcounter{equation}\label{eq:G27A}}
\alphaeqn
\begin{align} \label{eq:G27Aa}
w_i&=c^3w_{\frac {m}2+i}c^{-3}, \quad i=1,2,\dots,\tfrac {m}2,\\
w_i&=c^2w_{i-\frac {m}2}c^{-2}, \quad i=\tfrac {m}2+1,\tfrac {m}2+2,\dots,m.
\label{eq:G27Ab}
\end{align}
\reseteqn
There are two distinct possibilities for choosing
the $w_i$'s, $1\le i\le m$: 
either all the $w_i$'s are equal to $\ep$, or
there is an $i$ with $1\le i\le \frac m2$ such that
$$\ell_T(w_i)=\ell_T(w_{i+\frac m2})=1.$$
Writing $t_1,t_2$ for $w_i,w_{i+\frac m2}$, 
respectively, the equations \eqref{eq:G27A} 
reduce to
{\refstepcounter{equation}\label{eq:G27B}}
\alphaeqn
\begin{align} \label{eq:G27Ba}
t_1&=c^3t_2c^{-3},\\
\label{eq:G27Bb}
t_2&=c^2t_1c^{-2}.
\end{align}
\reseteqn
One of these equations is in fact superfluous: if we substitute
\eqref{eq:G27Bb} in \eqref{eq:G27Ba}, then
we obtain $t_1=c^5t_1c^{-5}$ which is automatically satisfied due
to Lemma~\ref{lem:4} with $d=6$.

Since $(w_0;w_1,\dots,w_m)\in NC^m(G_{27})$, we must have $t_1t_2\le_T
c$.
Combining this with \eqref{eq:G27B}, we infer that
\begin{equation} \label{eq:G27D}
t_1(c^{2}t_1c^{-2})\le_T c.
\end{equation}
With the help of {\tt CHEVIE}, 
one obtains five solutions for $t_1$ in this equation:
{\small
\begin{equation*} 
t_1\in\big\{[ 1 ],\,
[ 2 ],\,
[ 15 ],\,
[ 16 ],\,
[ 28 ]\big\},
\end{equation*}}%
each of them giving rise to $m/2$ elements of
$\Fix_{NC^m(G_{27})}(\phi^{p})$ since $i$ ranges from $1$ to $m/2$.
Here we have used the short notation of {\tt CHEVIE}
referring to the internal
ordering of the roots of $G_{27}$ in {\tt CHEVIE}.

In total, we obtain  
$1+5\frac m2=\frac {5m+2}2$ elements in
$\Fix_{NC^m(G_{27})}(\phi^p)$, which agrees with the limit in
\eqref{eq:G27.3}.

In the case that $p=15m/2$, we infer from \eqref{eq:Aktion} that
\begin{multline*}
\phi^p\big((w_0;w_1,\dots,w_m)\big)\\=
(*;
c^{8}w_{\frac {m}2+1}c^{-8},c^{8}w_{\frac {m}2+2}c^{-8},
\dots,c^{8}w_{m}c^{-8},
c^7w_{1}c^{-7},\dots,
c^7w_{\frac {m}2}c^{-7}\big).
\end{multline*}
Supposing that 
$(w_0;w_1,\dots,w_m)$ is fixed by $\phi^p$, we obtain
the system of equations
{\refstepcounter{equation}\label{eq:G27'A}}
\alphaeqn
\begin{align} \label{eq:G27'Aa}
w_i&=c^8w_{\frac {m}2+i}c^{-8}, \quad i=1,2,\dots,\tfrac {m}2,\\
w_i&=c^7w_{i-\frac {m}2}c^{-7}, \quad i=\tfrac {m}2+1,\tfrac {m}2+2,\dots,m.
\label{eq:G27'Ab}
\end{align}
\reseteqn
There are two distinct possibilities for choosing
the $w_i$'s, $1\le i\le m$: 
either all the $w_i$'s are equal to $\ep$, or
there is an $i$ with $1\le i\le \frac m2$ such that
$$\ell_T(w_i)=\ell_T(w_{i+\frac m2})=1.$$
Writing $t_1,t_2$ for $w_i,w_{i+\frac m2}$, 
respectively, the equations \eqref{eq:G27'A} 
reduce to
{\refstepcounter{equation}\label{eq:G27'B}}
\alphaeqn
\begin{align} \label{eq:G27'Ba}
t_1&=c^8t_2c^{-8},\\
\label{eq:G27'Bb}
t_2&=c^7t_1c^{-7}.
\end{align}
\reseteqn
One of these equations is in fact superfluous: if we substitute
\eqref{eq:G27'Bb} in \eqref{eq:G27'Ba}, then
we obtain $t_1=c^{15}t_1c^{-15}$ which is automatically satisfied 
due to Lemma~\ref{lem:4} with $d=2$.

Since $(w_0;w_1,\dots,w_m)\in NC^m(G_{27})$, we must have $t_1t_2\le_T
c$.
Combining this with \eqref{eq:G27'B}, we infer that
\begin{equation} \label{eq:G27'D}
t_1(c^{7}t_1c^{-7})\le_T c.
\end{equation}
Using that $c^5t_1c^{-5}=t_1$, due to Lemma~\ref{lem:4}
with $d=6$, we see that
this equation is equivalent with \eqref{eq:G27D}. 
Therefore, we are facing exactly the same enumeration 
problem here as for
$p=5m/2$, and, consequently, the number of solutions to \eqref{eq:G27'D} is the same, namely
$\frac {5m+2}2$, as required.

\smallskip
Finally, we turn to \eqref{eq:G27.1}. By Remark~\ref{rem:1},
the only choices for $h_2$ and $m_2$ to be considered
are $h_2=6$ and $m_2=3$, $h_2=6$ and $m_2=2$, $h_2=m_2=3$, respectively $h_2=m_2=2$. These correspond to the choices
$p=5m/3$, $5m/2$, $10m/3$, respectively $15m/2$, 
out of which only $p=5m/3$ and $p=10m/3$
have not yet been discussed and belong to the current case.
If $p=5m/3$, the corresponding action of $\phi^p$ is given by 
\eqref{eq:5m3Aktion}, so that we have to solve for $t_1$ with
$\ell_T(t_1)=1$ in the equation \eqref{eq:H3D}.
A computation with the help of 
{\tt CHEVIE}
finds no solution. 
If $p=10m/3$, the corresponding action of $\phi^p$ is given by 
\eqref{eq:10m3Aktion}, so that we have to solve for $t_1$ with
$\ell_T(t_1)$ in the equation \eqref{eq:H3'D}. 
Using that $c^5t_1c^{-5}=t_1$, 
due to Lemma~\ref{lem:4} with $d=6$, we see that
this equation is equivalent with the one in \eqref{eq:H3D}. 
Hence, in both cases, 
the left-hand side of \eqref{eq:1} is equal to $1$, as required.

\subsection*{\sc Case $G_{28}=F_4$}
The degrees are $2,6,8,12$, and hence we have
$$
\Cat^m(F_4;q)=\frac 
{[12m+12]_q\, [12m+8]_q\, [12m+6]_q\, [12m+2]_q} 
{[12]_q\, [8]_q\, [6]_q\, [2]_q} .
$$
Let $\zeta$ be a $12m$-th root of unity. 
The following cases on the right-hand side of \eqref{eq:1}
occur:
{\refstepcounter{equation}\label{eq:F4}}
\alphaeqn
\begin{align} 
\label{eq:F4.2}
\lim_{q\to\zeta}\Cat^m(F_4;q)&=m+1,
\quad\text{if }\zeta=\zeta_{12},\\
\label{eq:F4.3}
\lim_{q\to\zeta}\Cat^m(F_4;q)&=\tfrac {3m+2}2,
\quad\text{if }\zeta=\zeta_{8},\ 2\mid m,\\
\label{eq:F4.6}
\lim_{q\to\zeta}\Cat^m(F_4;q)&=(m+1)(2m+1),
\quad\text{if }\zeta= \zeta_{6},\zeta_{3},\\
\label{eq:F4.5}
\lim_{q\to\zeta}\Cat^m(F_4;q)&=\tfrac {(m+1)(3m+2)}2,
\quad\text{if }\zeta= \zeta_{4},\\
\label{eq:F4.8}
\lim_{q\to\zeta}\Cat^m(F_4;q)&=\Cat^m(F_4),
\quad\text{if }\zeta=-1\text{ or }\zeta=1,\\
\label{eq:F4.1}
\lim_{q\to\zeta}\Cat^m(F_4;q)&=1,
\quad\text{otherwise.}
\end{align}
\reseteqn

We must now prove that the left-hand side of \eqref{eq:1} in
each case agrees with the values exhibited in 
\eqref{eq:F4}. 
The only cases not covered by
Lemmas~\ref{lem:2} and \ref{lem:3} are the ones in 
\eqref{eq:F4.3} and \eqref{eq:F4.1}.
By Lemma~\ref{lem:1}, we are free to choose $p=3m/2$. In particular,
$m$ must be divisible by $2$.
From \eqref{eq:Aktion}, we infer
$$
\phi^p\big((w_0;w_1,\dots,w_m)\big)\\=
(*;
c^{2}w_{\frac m2+1}c^{-2},c^{2}w_{\frac m2+2}c^{-2},
\dots,c^{2}w_{m}c^{-2},
cw_{1}c^{-1},\dots,
cw_{\frac m2}c^{-1}\big).
$$
Supposing that 
$(w_0;w_1,\dots,w_m)$ is fixed by $\phi^p$, we obtain
the system of equations
{\refstepcounter{equation}\label{eq:F4A}}
\alphaeqn
\begin{align} \label{eq:F4Aa}
w_i&=c^2w_{\frac m2+i}c^{-2}, \quad i=1,2,\dots,\tfrac {m}2,\\
w_i&=cw_{i-\frac {m}2}c^{-1}, \quad i=\tfrac {m}2+1,\tfrac {m}2+2,\dots,m.
\label{eq:F4Ab}
\end{align}
\reseteqn
There are four distinct possibilities for choosing
the $w_i$'s, $1\le i\le m$: 
{\refstepcounter{equation}\label{eq:F4C}}
\alphaeqn
\begin{enumerate}
\item[(i)]
all the $w_i$'s are equal to $\ep$ (and $w_0=c$), 
\item[(ii)]
there is an $i$ with $1\le i\le \frac m2$ such that
\begin{equation} \label{eq:F4Cii}
\ell_T(w_i)=\ell_T(w_{i+\frac m2})=2,
\end{equation}
and all other $w_j$'s are equal to $\ep$,
\item[(iii)]
there is an $i$ with $1\le i\le \frac m2$ such that
\begin{equation} \label{eq:F4Ciii}
\ell_T(w_i)=\ell_T(w_{i+\frac m2})=1,
\end{equation}
and the other $w_j$'s, $1\le j\le m$, are equal to $\ep$,
\item[(iv)]
there are $i_1$ and $i_2$ with $1\le i_1<i_2\le \frac m2$ such that
\begin{equation} \label{eq:F4Civ}
\ell_T(w_{i_1})=\ell_T(w_{i_2})=
\ell_T(w_{i_1+\frac m2})=\ell_T(w_{i_2+\frac m2})=1,
\end{equation}
and all other $w_j$'s are equal to $\ep$.
\end{enumerate}
\reseteqn

Moreover, since $(w_0;w_1,\dots,w_m)\in NC^m(F_4)$, we must have 
$w_iw_{i+\frac m2}\le_T c$, respectively 
$w_{i_1}w_{i_2}w_{i_1+\frac m2}w_{i_2+\frac m2}=c$.
Together with equations~\eqref{eq:F4A}--\eqref{eq:F4C}, 
this implies that
\begin{equation} \label{eq:F4D}
w_i=c^3w_ic^{-3}\quad\text{and}\quad 
w_i(cw_ic^{-1})=c,
\end{equation}
respectively that
\begin{equation} \label{eq:F4D2}
w_i=c^3w_ic^{-3},\quad 
w_i(cw_ic^{-1})\le_T c,\quad\text{and}\quad
\ell_T(w_i)=1,
\end{equation}
respectively that
\begin{equation} \label{eq:F4E}
w_{i_1}=c^3w_{i_1}c^{-3},\quad 
w_{i_1}(cw_{i_1}c^{-1})\le_T c,
\quad\text{and}\quad \ell_T(w_{i_1})=1.
\end{equation}
With the help of Stembridge's {\sl Maple} package {\tt coxeter}
\cite{StemAZ}, one obtains three solutions for $w_i$ in 
\eqref{eq:F4D}:
{\small
$$w_i\in\big\{                             [1, 2, 3, 4, 3, 2],\,
                                   [2, 3],\,
                             [1, 3, 2, 1, 3, 4]
\big\},$$}%
where we have again used the short notation of {\tt coxeter},
$\{s_1,s_2,s_3,s_4\}$ being a simple system of generators of 
$F_4$,
corresponding to the Dynkin diagram displayed in Figure~\ref{fig:F4}.
Each of the above solutions for $w_i$ gives rise to $m/2$ elements of
$\Fix_{NC^m(F_4)}(\phi^{p})$ since $i$ ranges from $1$ to $m/2$.

\begin{figure}[h]
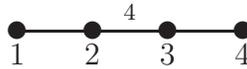

$$
\Einheit1cm
\Pfad(0,0),111\endPfad
\DickPunkt(0,0)
\DickPunkt(1,0)
\DickPunkt(2,0)
\DickPunkt(3,0)
\Label\u{1}(0,0)
\Label\u{2}(1,0)
\Label\u{3}(2,0)
\Label\u{4}(3,0)
\Label\ro{\raise-20pt\hbox{\scriptsize 4}}(1,0)
\hskip3cm
$$ 
\caption{\scriptsize The Dynkin diagram for $F_4$}
\label{fig:F4}
\end{figure}

There are no solutions for $w_i$ in \eqref{eq:F4D2} and
for $w_{i_1}$ in \eqref{eq:F4E}.

In total, we obtain 
$1+3\frac m2=\frac {3m+2}2$ elements in
$\Fix_{NC^m(F_4)}(\phi^p)$, which agrees with the limit in
\eqref{eq:F4.3}.

\smallskip
Finally, we turn to \eqref{eq:F4.1}. By Remark~\ref{rem:1},
the are no choices for $h_2$ and $m_2$ to be considered.
Hence, 
the left-hand side of \eqref{eq:1} is equal to $1$, as required.

\subsection*{\sc Case $G_{29}$}
The degrees are $4,8,12,20$, and hence we have
$$
\Cat^m(G_{29};q)=\frac 
{[20m+20]_q\, [20m+12]_q\, [20m+8]_q\, [20m+4]_q} 
{[20]_q\, [12]_q\, [8]_q\, [4]_q} .
$$
Let $\zeta$ be a $20m$-th root of unity. 
The following cases on the right-hand side of \eqref{eq:1}
occur:
{\refstepcounter{equation}\label{eq:G29}}
\alphaeqn
\begin{align} 
\label{eq:G29.2}
\lim_{q\to\zeta}\Cat^m(G_{29};q)&=m+1,
\quad\text{if }\zeta=\zeta_{20},\zeta_{10},\zeta_5,\\
\label{eq:G29.3}
\lim_{q\to\zeta}\Cat^m(G_{29};q)&=\tfrac {5m+3}3,
\quad\text{if }\zeta=\zeta_{12},\zeta_6,\zeta_3,\ 3\mid m,\\
\label{eq:G29.4}
\lim_{q\to\zeta}\Cat^m(G_{29};q)&=\tfrac {5m+2}2,
\quad\text{if }\zeta=\zeta_{8},\ 2\mid m,\\
\label{eq:G29.5}
\lim_{q\to\zeta}\Cat^m(G_{29};q)&=\Cat^m(G_{29}),
\quad\text{if }\zeta=\zeta_4,-1,1,\\
\label{eq:G29.1}
\lim_{q\to\zeta}\Cat^m(G_{29};q)&=1,
\quad\text{otherwise.}
\end{align}
\reseteqn

We must now prove that the left-hand side of \eqref{eq:1} in
each case agrees with the values exhibited in 
\eqref{eq:G29}. The only cases not covered by
Lemmas~\ref{lem:2} and \ref{lem:3} are the ones in 
\eqref{eq:G29.3}, \eqref{eq:G29.4}, 
and \eqref{eq:G29.1}.

\smallskip
We begin with the case in \eqref{eq:G29.3}.
By Lemma~\ref{lem:1}, 
we are free to choose $p=5m/3$ if $\zeta=\zeta_{12}$, 
we are free to choose $p=10m/3$ if $\zeta=\zeta_{6}$, 
we are free to choose $p=20m/3$ if $\zeta=\zeta_{3}$. 
In particular, in all three cases,
$m$ must be divisible by $3$.

We start with the case that $p=5m/3$.
From \eqref{eq:Aktion}, we infer
$$
\phi^p\big((w_0;w_1,\dots,w_m)\big)\\=
(*;
c^{2}w_{\frac m3+1}c^{-2},c^{2}w_{\frac m3+2}c^{-2},
\dots,c^{2}w_{m}c^{-2},
cw_{1}c^{-1},\dots,
cw_{\frac m3}c^{-1}\big).
$$
Supposing that 
$(w_0;w_1,\dots,w_m)$ is fixed by $\phi^p$, we obtain
the system of equations
{\refstepcounter{equation}\label{eq:G29A}}
\alphaeqn
\begin{align} \label{eq:G29Aa}
w_i&=c^2w_{\frac m3+i}c^{-2}, \quad i=1,2,\dots,\tfrac {2m}3,\\
w_i&=cw_{i-\frac {2m}3}c^{-1}, \quad i=\tfrac {2m}3+1,\tfrac {2m}3+2,\dots,m.
\label{eq:G29Ab}
\end{align}
\reseteqn
There are two distinct possibilities for choosing
the $w_i$'s, $1\le i\le m$: 
{\refstepcounter{equation}\label{eq:G29C}}
\alphaeqn
\begin{enumerate}
\item[(i)]
all the $w_i$'s are equal to $\ep$ (and $w_0=c$), 
\item[(ii)]
there is an $i$ with $1\le i\le \frac m3$ such that
\begin{equation} \label{eq:G29Cii}
\ell_T(w_i)=\ell_T(w_{i+\frac m3})=\ell_T(w_{i+\frac {2m}3})=1,
\end{equation}
and the other $w_j$'s, $1\le j\le m$, are equal to $\ep$.
\end{enumerate}
\reseteqn

Moreover, since $(w_0;w_1,\dots,w_m)\in NC^m(G_{29})$, we must have 
$w_iw_{i+\frac m3}w_{i+\frac {2m}3}\le_T c$.
Together with equations~\eqref{eq:G29A}--\eqref{eq:G29C}, 
this implies that
\begin{equation} \label{eq:G29D}
w_i=c^5w_ic^{-5}\quad\text{and}\quad 
w_i(c^3w_ic^{-3})(cw_ic^{-1})=c.
\end{equation}
With the help of {\tt CHEVIE}, one obtains five solutions for $w_i$ in 
\eqref{eq:G29D}:
{\small
\begin{equation*} 
w_i\in\big\{[1],\,
[2],\,
[8],\,
[25],\,
[31]\big\},
\end{equation*}}%
where we have again used the short notation of {\tt CHEVIE}
referring to the internal
ordering of the roots of $G_{29}$ in {\tt CHEVIE}.
Each of the above solutions for $w_i$ gives rise to $m/3$ elements of
$\Fix_{NC^m(G_{29})}(\phi^{p})$ since $i$ ranges from $1$ to $m/3$.

In total, we obtain 
$1+5\frac m3=\frac {5m+3}3$ elements in
$\Fix_{NC^m(G_{29})}(\phi^p)$, which agrees with the limit in
\eqref{eq:G29.3}.

In the case that $p=10m/3$, we infer from \eqref{eq:Aktion} that
\begin{multline*}
\phi^p\big((w_0;w_1,\dots,w_m)\big)\\=
(*;
c^{4}w_{\frac {2m}3+1}c^{-4},c^{4}w_{\frac {2m}3+2}c^{-4},
\dots,c^{4}w_{m}c^{-4},
c^3w_{1}c^{-3},\dots,
c^3w_{\frac {2m}3}c^{-3}\big).
\end{multline*}
Supposing that 
$(w_0;w_1,\dots,w_m)$ is fixed by $\phi^p$, we obtain
the system of equations
{\refstepcounter{equation}\label{eq:G29'''A}}
\alphaeqn
\begin{align} \label{eq:G29'''Aa}
w_i&=c^4w_{\frac {2m}3+i}c^{-4}, \quad i=1,2,\dots,\tfrac {m}3,\\
w_i&=c^3w_{i-\frac {m}3}c^{-3}, \quad i=\tfrac {m}3+1,\tfrac {m}3+2,\dots,m.
\label{eq:G29'''Ab}
\end{align}
\reseteqn
There are two distinct possibilities for choosing
the $w_i$'s, $1\le i\le m$: 
either all the $w_i$'s are equal to $\ep$, or
there is an $i$ with $1\le i\le \frac m3$ such that
$$\ell_T(w_i)=\ell_T(w_{i+\frac m3})=\ell_T(w_{i+\frac {2m}3})=1.$$
Writing $t_1,t_2,t_3$ for $w_i,w_{i+\frac m3},w_{i+\frac {2m}3}$, 
respectively, the equations \eqref{eq:G29'''A} 
reduce to
{\refstepcounter{equation}\label{eq:G29'''B}}
\alphaeqn
\begin{align} \label{eq:G29'''Ba}
t_1&=c^4t_3c^{-4},\\
\label{eq:G29'''Bb}
t_2&=c^3t_1c^{-3},\\
t_3&=c^3t_2c^{-3}.
\label{eq:G29'''Bc}
\end{align}
\reseteqn
One of these equations is in fact superfluous: if we substitute
\eqref{eq:G29'''Bb} and \eqref{eq:G29'''Bc} in \eqref{eq:G29'''Ba}, then
we obtain $t_1=c^{10}t_1c^{-10}$ which is automatically satisfied 
due to Lemma~\ref{lem:4} with $d=2$.

Since $(w_0;w_1,\dots,w_m)\in NC^m(G_{29})$, we must have
$t_1t_2t_3\le_T c$.
Combining this with \eqref{eq:G29'''B}, we infer that
\begin{equation} \label{eq:G29'''D}
t_1(c^{3}t_1c^{-3})(c^{6}t_1c^{-6})\le_T c.
\end{equation}
Using that $c^5t_1c^{-5}=t_1$, due to Lemma~\ref{lem:4}
with $d=4$, we see that
this equation is equivalent with \eqref{eq:G29D}. 
Therefore, we are facing exactly the same enumeration 
problem here as for
$p=5m/3$, and, consequently, the number of solutions to \eqref{eq:G29'''D} is the same, namely
$\frac {5m+3}3$, as required.

In the case that $p=20m/3$, we infer from \eqref{eq:Aktion} that
\begin{multline*}
\phi^p\big((w_0;w_1,\dots,w_m)\big)\\=
(*;
c^{7}w_{\frac {m}3+1}c^{-7},c^{7}w_{\frac {m}3+2}c^{-7},
\dots,c^{7}w_{m}c^{-7},
c^6w_{1}c^{-6},\dots,
c^6w_{\frac {m}3}c^{-6}\big).
\end{multline*}
Supposing that 
$(w_0;w_1,\dots,w_m)$ is fixed by $\phi^p$, we obtain
the system of equations
{\refstepcounter{equation}\label{eq:G29''''A}}
\alphaeqn
\begin{align} \label{eq:G29''''Aa}
w_i&=c^7w_{\frac {m}3+i}c^{-7}, \quad i=1,2,\dots,\tfrac {2m}3,\\
w_i&=c^6w_{i-\frac {2m}3}c^{-6}, \quad i=\tfrac {2m}3+1,\tfrac {2m}3+2,\dots,m.
\label{eq:G29''''Ab}
\end{align}
\reseteqn
There are two distinct possibilities for choosing
the $w_i$'s, $1\le i\le m$: 
either all the $w_i$'s are equal to $\ep$, or
there is an $i$ with $1\le i\le \frac m3$ such that
$$\ell_T(w_i)=\ell_T(w_{i+\frac m3})=\ell_T(w_{i+\frac {2m}3})=1.$$
Writing $t_1,t_2,t_3$ for $w_i,w_{i+\frac m3},w_{i+\frac {2m}3}$, 
respectively, the equations \eqref{eq:G29''''A} 
reduce to
{\refstepcounter{equation}\label{eq:G29''''B}}
\alphaeqn
\begin{align} \label{eq:G29''''Ba}
t_1&=c^7t_2c^{-7},\\
\label{eq:G29''''Bb}
t_2&=c^7t_3c^{-7},\\
t_3&=c^6t_1c^{-6}.
\label{eq:G29''''Bc}
\end{align}
\reseteqn
One of these equations is in fact superfluous: if we substitute
\eqref{eq:G29''''Bb} and \eqref{eq:G29''''Bc} in \eqref{eq:G29''''Ba}, then
we obtain $t_1=c^{20}t_1c^{-20}$ which is automatically satisfied 
since $c^{20}=\ep$.

Since $(w_0;w_1,\dots,w_m)\in NC^m(G_{29})$, we must have
$t_1t_2t_3\le_T c$.
Combining this with \eqref{eq:G29''''B}, we infer that
\begin{equation} \label{eq:G29''''D}
t_1(c^{13}t_1c^{-13})(c^{6}t_1c^{-6})\le_T c.
\end{equation}
Using that $c^5t_1c^{-5}=t_1$, due to Lemma~\ref{lem:4}
with $d=4$, we see that
this equation is equivalent with \eqref{eq:G29D}. 
Therefore, we are facing exactly the same enumeration 
problem here as for
$p=5m/3$, and, consequently, the number of solutions to \eqref{eq:G29''''D} is the same, namely
$\frac {5m+3}3$, as required.

\smallskip
Next we discuss the case in \eqref{eq:G29.4}.
By Lemma~\ref{lem:1}, we are free to choose $p=5m/2$. In particular,
$m$ must be divisible by $2$.
From \eqref{eq:Aktion}, we infer
\begin{multline*}
\phi^p\big((w_0;w_1,\dots,w_m)\big)\\=
(*;
c^{3}w_{\frac m2+1}c^{-3},c^{3}w_{\frac m2+2}c^{-3},
\dots,c^{3}w_{m}c^{-3},
c^2w_{1}c^{-2},\dots,
c^2w_{\frac m2}c^{-2}\big).
\end{multline*}
Supposing that 
$(w_0;w_1,\dots,w_m)$ is fixed by $\phi^p$, we obtain
the system of equations
{\refstepcounter{equation}\label{eq:G29'A}}
\alphaeqn
\begin{align} \label{eq:G29'Aa}
w_i&=c^3w_{\frac m2+i}c^{-3}, \quad i=1,2,\dots,\tfrac {m}2,\\
w_i&=c^2w_{i-\frac {m}2}c^{-2}, \quad i=\tfrac {m}2+1,\tfrac {m}2+2,\dots,m.
\label{eq:G29'Ab}
\end{align}
\reseteqn
There are four distinct possibilities for choosing
the $w_i$'s, $1\le i\le m$: 
{\refstepcounter{equation}\label{eq:G29'C}}
\alphaeqn
\begin{enumerate}
\item[(i)]
all the $w_i$'s are equal to $\ep$ (and $w_0=c$), 
\item[(ii)]
there is an $i$ with $1\le i\le \frac m2$ such that
\begin{equation} \label{eq:G29'Cii}
\ell_T(w_i)=\ell_T(w_{i+\frac m2})=2,
\end{equation}
and all other $w_j$'s are equal to $\ep$,
\item[(iii)]
there is an $i$ with $1\le i\le \frac m2$ such that
\begin{equation} \label{eq:G29'Ciii}
\ell_T(w_i)=\ell_T(w_{i+\frac m2})=1,
\end{equation}
and the other $w_j$'s, $1\le j\le m$, are equal to $\ep$,
\item[(iv)]
there are $i_1$ and $i_2$ with $1\le i_1<i_2\le \frac m2$ such that
\begin{equation} \label{eq:G29'Civ}
\ell_T(w_{i_1})=\ell_T(w_{i_2})=
\ell_T(w_{i_1+\frac m2})=\ell_T(w_{i_2+\frac m2})=1,
\end{equation}
and all other $w_j$'s are equal to $\ep$.
\end{enumerate}
\reseteqn

Moreover, since $(w_0;w_1,\dots,w_m)\in NC^m(G_{29})$, we must have 
$w_iw_{i+\frac m2}\le_T c$, respectively 
$w_{i_1}w_{i_2}w_{i_1+\frac m2}w_{i_2+\frac m2}=c$.
Together with equations~\eqref{eq:G29'A}--\eqref{eq:G29'C}, 
this implies that
\begin{equation} \label{eq:G29'D}
w_i=c^3w_ic^{-3}\quad\text{and}\quad 
w_i(c^2w_ic^{-2})=c,
\end{equation}
respectively that
\begin{equation} \label{eq:G29'D2}
w_i=c^3w_ic^{-3},\quad 
w_i(c^2w_ic^{-2})\le_T c,\quad\text{and}\quad
\ell_T(w_i)=1,
\end{equation}
respectively that
\begin{multline} \label{eq:G29'E}
w_{i_1}=c^3w_{i_1}c^{-3},\quad 
w_{i_2}=c^3w_{i_2}c^{-3},\\
w_{i_1}w_{i_2}(c^2w_{i_1}c^{-2})(c^2w_{i_2}c^{-2})= c,
\quad\text{and}\quad \ell_T(w_{i_1})=\ell_T(w_{i_2})=1.
\end{multline}
With the help of {\tt CHEVIE}, one obtains five solutions for $w_i$ in 
\eqref{eq:G29'D2}:
{\small
\begin{equation*}
w_i\in\big\{[4],\,
[9],\,
[14],\,
[27],\,
[32]\big\},
\end{equation*}}%
where we have again used the short notation of {\tt CHEVIE}
referring to the internal
ordering of the roots of $G_{29}$ in {\tt CHEVIE}.
Each of these solutions for $w_i$ gives rise to $m/2$ elements of
$\Fix_{NC^m(G_{29})}(\phi^{p})$ since $i$ ranges from $1$ to $m/2$.

There are no solutions for $w_i$ in \eqref{eq:G29'D} and
for $(w_{i_1},w_{i_2})$ in \eqref{eq:G29'E}.

In total, we obtain 
$1+5\frac m2=\frac {5m+2}2$ elements in
$\Fix_{NC^m(G_{29})}(\phi^p)$, which agrees with the limit in
\eqref{eq:G29.4}.

\smallskip
Finally, we turn to \eqref{eq:G29.1}. By Remark~\ref{rem:1},
the only choices for $h_2$ and $m_2$ to be considered
are $h_2=1$ and $m_2=3$, $h_2=2$ and $m_2=3$, 
$h_2=4$ and $m_2=2$, $h_2=4$ and $m_2=3$, 
respectively $h_2=m_2=4$. 
These correspond to the choices $p=20m/3$, $p=10m/3$, 
$p=5m/2$, $p=5m/3$,
respectively $p=5m/4$, out of which only $p=5m/4$
has not yet been discussed and belongs to the current case.
The corresponding action of $\phi^p$ is given by 
\begin{multline*}
\phi^p\big((w_0;w_1,\dots,w_m)\big)\\=
(*;
c^{2}w_{\frac {3m}4+1}c^{-2},c^{2}w_{\frac {3m}4+2}c^{-2},
\dots,c^{2}w_{m}c^{-2},
cw_{1}c^{-1},\dots,
cw_{\frac {3m}4}c^{-1}\big),
\end{multline*}
so that we have to solve
$$t_1(ct_1c^{-1})(c^{2}t_1c^{-2})(c^{3}t_1c^{-3})=c$$
for $t_1$ with $\ell_T(t_1)$.
A computation with the help of {\tt CHEVIE} finds no solution. 
Hence, 
the left-hand side of \eqref{eq:1} is equal to $1$, as required.

\subsection*{\sc Case $G_{30}=H_4$}
The degrees are $2,12,20,30$, and hence we have
$$
\Cat^m(H_4;q)=\frac 
{[30m+30]_q\, [30m+20]_q\, [30m+12]_q\, [30m+2]_q} 
{[30]_q\, [20]_q\, [12]_q\, [2]_q} .
$$
Let $\zeta$ be a $30m$-th root of unity. 
The following cases on the right-hand side of \eqref{eq:1}
occur:
{\refstepcounter{equation}\label{eq:H4}}
\alphaeqn
\begin{align} 
\label{eq:H4.2}
\lim_{q\to\zeta}\Cat^m(H_4;q)&=m+1,
\quad\text{if }\zeta=\zeta_{30},\zeta_{15},\\
\label{eq:H4.3}
\lim_{q\to\zeta}\Cat^m(H_4;q)&=\tfrac {3m+2}2,
\quad\text{if }\zeta=\zeta_{20},\ 2\mid m,\\
\label{eq:H4.4}
\lim_{q\to\zeta}\Cat^m(H_4;q)&=\tfrac {5m+2}2,
\quad\text{if }\zeta=\zeta_{12},\ 2\mid m,\\
\label{eq:H4.5}
\lim_{q\to\zeta}\Cat^m(H_4;q)&=\tfrac {(m+1)(3m+2)}2,
\quad\text{if }\zeta= \zeta_{10},\zeta_{5},\\
\label{eq:H4.6}
\lim_{q\to\zeta}\Cat^m(H_4;q)&=\tfrac {(m+1)(5m+2)}2,
\quad\text{if }\zeta= \zeta_{6},\zeta_{3},\\
\label{eq:H4.7}
\lim_{q\to\zeta}\Cat^m(H_4;q)&=\tfrac {(3m+2)(5m+2)}4,
\quad\text{if }\zeta= \zeta_{4},\ 2\mid m,\\
\label{eq:H4.8}
\lim_{q\to\zeta}\Cat^m(H_4;q)&=\Cat^m(H_4),
\quad\text{if }\zeta=-1\text{ or }\zeta=1,\\
\label{eq:H4.1}
\lim_{q\to\zeta}\Cat^m(H_4;q)&=1,
\quad\text{otherwise.}
\end{align}
\reseteqn

We must now prove that the left-hand side of \eqref{eq:1} in
each case agrees with the values exhibited in 
\eqref{eq:H4}. The only cases not covered by
Lemmas~\ref{lem:2} and \ref{lem:3} are the ones in 
\eqref{eq:H4.3}, \eqref{eq:H4.4}, \eqref{eq:H4.7},
and \eqref{eq:H4.1}.

\smallskip
We begin with the case in \eqref{eq:H4.3}.
By Lemma~\ref{lem:1}, we are free to choose $p=3m/2$. In particular,
$m$ must be divisible by $2$.
From \eqref{eq:Aktion}, we infer
$$
\phi^p\big((w_0;w_1,\dots,w_m)\big)\\=
(*;
c^{2}w_{\frac m2+1}c^{-2},c^{2}w_{\frac m2+2}c^{-2},
\dots,c^{2}w_{m}c^{-2},
cw_{1}c^{-1},\dots,
cw_{\frac m2}c^{-1}\big).
$$
Supposing that 
$(w_0;w_1,\dots,w_m)$ is fixed by $\phi^p$, we obtain
the system of equations
{\refstepcounter{equation}\label{eq:H4A}}
\alphaeqn
\begin{align} \label{eq:H4Aa}
w_i&=c^2w_{\frac m2+i}c^{-2}, \quad i=1,2,\dots,\tfrac {m}2,\\
w_i&=cw_{i-\frac {m}2}c^{-1}, \quad i=\tfrac {m}2+1,\tfrac {m}2+2,\dots,m.
\label{eq:H4Ab}
\end{align}
\reseteqn
There are four distinct possibilities for choosing
the $w_i$'s, $1\le i\le m$: 
{\refstepcounter{equation}\label{eq:H4C}}
\alphaeqn
\begin{enumerate}
\item[(i)]
all the $w_i$'s are equal to $\ep$ (and $w_0=c$), 
\item[(ii)]
there is an $i$ with $1\le i\le \frac m2$ such that
\begin{equation} \label{eq:H4Cii}
\ell_T(w_i)=\ell_T(w_{i+\frac m2})=2,
\end{equation}
and all other $w_j$'s are equal to $\ep$,
\item[(iii)]
there is an $i$ with $1\le i\le \frac m2$ such that
\begin{equation} \label{eq:H4Ciii}
\ell_T(w_i)=\ell_T(w_{i+\frac m2})=1,
\end{equation}
and the other $w_j$'s, $1\le j\le m$, are equal to $\ep$,
\item[(iv)]
there are $i_1$ and $i_2$ with $1\le i_1<i_2\le \frac m2$ such that
\begin{equation} \label{eq:H4Civ}
\ell_T(w_{i_1})=\ell_T(w_{i_2})=
\ell_T(w_{i_1+\frac m2})=\ell_T(w_{i_2+\frac m2})=1,
\end{equation}
and all other $w_j$'s are equal to $\ep$.
\end{enumerate}
\reseteqn

Moreover, since $(w_0;w_1,\dots,w_m)\in NC^m(H_4)$, we must have 
$w_iw_{i+\frac m2}\le_T c$, respectively 
$w_{i_1}w_{i_2}w_{i_1+\frac m2}w_{i_2+\frac m2}=c$.
Together with equations~\eqref{eq:H4A}--\eqref{eq:H4C}, 
this implies that
\begin{equation} \label{eq:H4D}
w_i=c^3w_ic^{-3}\quad\text{and}\quad 
w_i(cw_ic^{-1})=c,
\end{equation}
respectively that
\begin{equation} \label{eq:H4D2}
w_i=c^3w_ic^{-3},\quad 
w_i(cw_ic^{-1})\le_T c,\quad\text{and}\quad
\ell_T(w_i)=1,
\end{equation}
respectively that
\begin{equation} \label{eq:H4E}
w_{i_1}=c^3w_{i_1}c^{-3},\quad 
w_{i_1}(cw_{i_1}c^{-1})\le_T c,
\quad\text{and}\quad \ell_T(w_{i_1})=1.
\end{equation}
With the help of Stembridge's {\sl Maple} package {\tt coxeter}
\cite{StemAZ}, one obtains three solutions for $w_i$ in 
\eqref{eq:H4D}:
{\small
\begin{equation*} 
w_i\in\big\{[1,2,3,4,3,2,1,2],[2,1,2,3],
[1,3,2,1,2,1,3,4]\big\},
\end{equation*}}%
where we have again used the short notation of {\tt coxeter},
$\{s_1,s_2,s_3,s_4\}$ being a simple system of generators of 
$H_4$,
corresponding to the Dynkin diagram displayed in Figure~\ref{fig:H4}.
Each of the above solutions for $w_i$ gives rise to $m/2$ elements of
$\Fix_{NC^m(H_4)}(\phi^{p})$ since $i$ ranges from $1$ to $m/2$.

\begin{figure}[h]
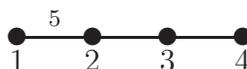

$$
\Einheit1cm
\Pfad(0,0),111\endPfad
\DickPunkt(0,0)
\DickPunkt(1,0)
\DickPunkt(2,0)
\DickPunkt(3,0)
\Label\u{1}(0,0)
\Label\u{2}(1,0)
\Label\u{3}(2,0)
\Label\u{4}(3,0)
\Label\ro{\raise-20pt\hbox{\scriptsize 5}}(0,0)
\hskip3cm
$$ 
\caption{\scriptsize The Dynkin diagram for $H_4$}
\label{fig:H4}
\end{figure}

There are no solutions for $w_i$ in \eqref{eq:H4D2} and
for $w_{i_1}$ in \eqref{eq:H4E}.

In total, we obtain 
$1+3\frac m2=\frac {3m+2}2$ elements in
$\Fix_{NC^m(H_4)}(\phi^p)$, which agrees with the limit in
\eqref{eq:H4.3}.

\smallskip
Next we discuss the case in \eqref{eq:H4.4}.
By Lemma~\ref{lem:1}, we are free to choose $p=5m/2$. In particular,
$m$ must be divisible by $2$.
From \eqref{eq:Aktion}, we infer
\begin{multline*}
\phi^p\big((w_0;w_1,\dots,w_m)\big)\\=
(*;
c^{3}w_{\frac m2+1}c^{-3},c^{3}w_{\frac m2+2}c^{-3},
\dots,c^{3}w_{m}c^{-3},
c^2w_{1}c^{-2},\dots,
c^2w_{\frac m2}c^{-2}\big).
\end{multline*}
Supposing that 
$(w_0;w_1,\dots,w_m)$ is fixed by $\phi^p$, we obtain
the system of equations
{\refstepcounter{equation}\label{eq:H4'A}}
\alphaeqn
\begin{align} \label{eq:H4'Aa}
w_i&=c^3w_{\frac m2+i}c^{-3}, \quad i=1,2,\dots,\tfrac {m}2,\\
w_i&=c^2w_{i-\frac {m}2}c^{-2}, \quad i=\tfrac {m}2+1,\tfrac {m}2+2,\dots,m.
\label{eq:H4'Ab}
\end{align}
\reseteqn
There are four distinct possibilities for choosing
the $w_i$'s, $1\le i\le m$: 
{\refstepcounter{equation}\label{eq:H4'C}}
\alphaeqn
\begin{enumerate}
\item[(i)]
all the $w_i$'s are equal to $\ep$ (and $w_0=c$), 
\item[(ii)]
there is an $i$ with $1\le i\le \frac m2$ such that
\begin{equation} \label{eq:H4'Cii}
\ell_T(w_i)=\ell_T(w_{i+\frac m2})=2,
\end{equation}
and all other $w_j$'s are equal to $\ep$,
\item[(iii)]
there is an $i$ with $1\le i\le \frac m2$ such that
\begin{equation} \label{eq:H4'Ciii}
\ell_T(w_i)=\ell_T(w_{i+\frac m2})=1,
\end{equation}
and the other $w_j$'s, $1\le j\le m$, are equal to $\ep$,
\item[(iv)]
there are $i_1$ and $i_2$ with $1\le i_1<i_2\le \frac m2$ such that
\begin{equation} \label{eq:H4'Civ}
\ell_T(w_{i_1})=\ell_T(w_{i_2})=
\ell_T(w_{i_1+\frac m2})=\ell_T(w_{i_2+\frac m2})=1,
\end{equation}
and all other $w_j$'s are equal to $\ep$.
\end{enumerate}
\reseteqn

Moreover, since $(w_0;w_1,\dots,w_m)\in NC^m(H_4)$, we must have 
$w_iw_{i+\frac m2}\le_T c$, respectively 
$w_{i_1}w_{i_2}w_{i_1+\frac m2}w_{i_2+\frac m2}=c$.
Together with equations~\eqref{eq:H4'A}--\eqref{eq:H4'C}, 
this implies that
\begin{equation} \label{eq:H4'D}
w_i=c^3w_ic^{-3}\quad\text{and}\quad 
w_i(c^2w_ic^{-2})=c,
\end{equation}
respectively that
\begin{equation} \label{eq:H4'D2}
w_i=c^3w_ic^{-3},\quad 
w_i(c^2w_ic^{-2})\le_T c,\quad\text{and}\quad
\ell_T(w_i)=1,
\end{equation}
respectively that
\begin{equation} \label{eq:H4'E}
w_{i_1}=c^3w_{i_1}c^{-3},\quad 
w_{i_1}(c^2w_{i_1}c^{-2})\le_T c,
\quad\text{and}\quad \ell_T(w_{i_1})=1.
\end{equation}
With the help of Stembridge's {\sl Maple} package {\tt coxeter}
\cite{StemAZ}, one obtains five solutions for $w_i$ in 
\eqref{eq:H4'D}:
{\small
\begin{multline*} 
w_i\in\big\{[1,3,2,1,2,1,3,2],[
1,2,1,4,3,2,1,2,1,3,2,1,4,3],\\
[2,1,2,3,2,1,2,4],[
2,1,2,1,4,3,2,1,2,1,3,4],[
1,2,3,2,1,4,3,2,1,2,1,3]
\big\},
\end{multline*}}%
where we used again {\tt coxeter}'s short notation, $\{s_1,s_2,s_3,s_4\}$ being a simple system of generators of 
$H_4$,
corresponding to the Dynkin diagram displayed in Figure~\ref{fig:H4}.
Each of these solutions for $w_i$ gives rise to $m/2$ elements of
$\Fix_{NC^m(H_4)}(\phi^{p})$ since $i$ ranges from $1$ to $m/2$.

There are no solutions for $w_i$ in \eqref{eq:H4'D2} and
for $w_{i_1}$ in \eqref{eq:H4'E}.

In total, we obtain 
$1+5\frac m2=\frac {5m+2}2$ elements in
$\Fix_{NC^m(H_4)}(\phi^p)$, which agrees with the limit in
\eqref{eq:H4.4}.

\smallskip
Finally we discuss the case in \eqref{eq:H4.7}.
By Lemma~\ref{lem:1}, we are free to choose $p=15m/2$. In particular,
$m$ must be divisible by $2$.
From \eqref{eq:Aktion}, we infer
\begin{multline*}
\phi^p\big((w_0;w_1,\dots,w_m)\big)\\=
(*;
c^{8}w_{\frac m2+1}c^{-8},c^{8}w_{\frac m2+2}c^{-8},
\dots,c^{8}w_{m}c^{-8},
c^7w_{1}c^{-7},\dots,
c^7w_{\frac m2}c^{-7}\big).
\end{multline*}
Supposing that 
$(w_0;w_1,\dots,w_m)$ is fixed by $\phi^p$, we obtain
the system of equations
{\refstepcounter{equation}\label{eq:H4''A}}
\alphaeqn
\begin{align} \label{eq:H4''Aa}
w_i&=c^8w_{\frac m2+i}c^{-8}, \quad i=1,2,\dots,\tfrac {m}2,\\
w_i&=c^7w_{i-\frac {m}2}c^{-7}, \quad i=\tfrac {m}2+1,\tfrac {m}2+2,\dots,m.
\label{eq:H4''Ab}
\end{align}
\reseteqn
There are four distinct possibilities for choosing
the $w_i$'s, $1\le i\le m$: {\refstepcounter{equation}\label{eq:H4''C}}
\alphaeqn
\begin{enumerate}
\item[(i)]
all the $w_i$'s are equal to $\ep$ (and $w_0=c$), 
\item[(ii)]
there is an $i$ with $1\le i\le \frac m2$ such that
\begin{equation} \label{eq:H4''Cii}
\ell_T(w_i)=\ell_T(w_{i+\frac m2})=2,
\end{equation}
and all other $w_j$'s are equal to $\ep$,
\item[(iii)]
there is an $i$ with $1\le i\le \frac m2$ such that
\begin{equation} \label{eq:H4''Ciii}
\ell_T(w_i)=\ell_T(w_{i+\frac m2})=1,
\end{equation}
and the other $w_j$'s, $1\le j\le m$, are equal to $\ep$,
\item[(iv)]
there are $i_1$ and $i_2$ with $1\le i_1<i_2\le \frac m2$ such that
\begin{equation} \label{eq:H4''Civ}
\ell_T(w_{i_1})=\ell_T(w_{i_2})=
\ell_T(w_{i_1+\frac m2})=\ell_T(w_{i_2+\frac m2})=1,
\end{equation}
and all other $w_j$'s are equal to $\ep$.
\end{enumerate}
\reseteqn

Moreover, since $(w_0;w_1,\dots,w_m)\in NC^m(H_4)$, we must have 
$w_iw_{i+\frac m2}\le_T c$, respectively 
$w_{i_1}w_{i_2}w_{i_1+\frac m2}w_{i_2+\frac m2}=c$.
Together with equations~\eqref{eq:H4''A}--\eqref{eq:H4''C}, 
this implies that
\begin{equation} \label{eq:H4''D}
w_i=c^{15}w_ic^{-15}\quad\text{and}\quad 
w_i(c^7w_ic^{-7})=c,
\end{equation}
respectively that
\begin{equation} \label{eq:H4''D2}
w_i=c^{15}w_ic^{-15},\quad 
w_i(c^7w_ic^{-7})=c,\quad\text{and}\quad
\ell_T(w_i)=1,
\end{equation}
respectively that
\begin{equation} \label{eq:H4''E}
w_{i_1}=c^{15}w_{i_1}c^{-15},\quad 
w_{i_2}=c^{15}w_{i_2}c^{-15},\quad w_{i_1}w_{i_2}(c^7w_{i_1}c^{-7})(c^7w_{i_2}c^{-7})=c.
\end{equation}
Here, the first equations in both \eqref{eq:H4''D} and 
\eqref{eq:H4''D2}, and the first two equations in \eqref{eq:H4''E} are automatically satisfied due to 
Lemma~\ref{lem:4} with $d=2$.

With the help of Stembridge's {\sl Maple} package {\tt coxeter}
\cite{StemAZ}, one obtains eight solutions for $w_i$ in 
\eqref{eq:H4''D}:
{\small
\begin{multline} \label{eq:H4sol3}
w_i\in\big\{
                          [1, 3, 2, 1, 2, 1, 3, 2],\,
                          [1, 2, 3, 4, 3, 2, 1, 2],\,
                                [2, 1, 2, 3],\\
                 [1, 2, 1, 4, 3, 2, 1, 2, 1, 3, 2, 1, 4, 3],\,
                          [1, 3, 2, 1, 2, 1, 3, 4],\,
                          [2, 1, 2, 3, 2, 1, 2, 4],\\
                    [2, 1, 2, 1, 4, 3, 2, 1, 2, 1, 3, 4],\,
                    [1, 2, 3, 2, 1, 4, 3, 2, 1, 2, 1, 3]
\big\},
\end{multline}}%
where $\{s_1,s_2,s_3,s_4\}$ is a simple system of generators of 
$H_4$,
corresponding to the Dynkin diagram displayed in Figure~\ref{fig:H4},
and each of them gives rise to $m/2$ elements of
$\Fix_{NC^m(H_4)}(\phi^{p})$ since $i$ ranges from $1$ to $m/2$.
Furthermore, one obtains 15 solutions for $w_i$ in 
\eqref{eq:H4''D2}:
{\small
\begin{multline*}
w_i\in\big\{                         [2],\,
                                     [3],\,
                                     [4],\,
                                  [2, 1, 2],\,
                               [3, 2, 1, 2, 3],\,
                            [1, 3, 2, 1, 2, 1, 3],\,
                            [2, 1, 2, 3, 2, 1, 2],\,
                            [1, 2, 3, 4, 3, 2, 1],\\
                         [2, 1, 3, 2, 1, 2, 1, 3, 2],\,
                         [1, 4, 3, 2, 1, 2, 1, 3, 4],\,
                         [2, 1, 2, 3, 4, 3, 2, 1, 2],\\
                   [1, 2, 1, 4, 3, 2, 1, 2, 1, 3, 2, 1, 4],\,
                   [1, 3, 2, 1, 2, 3, 4, 3, 2, 1, 2, 1, 3],\\
                [2, 1, 2, 1, 4, 3, 2, 1, 2, 1, 3, 2, 1, 2, 4],\,
                [1, 3, 2, 1, 4, 3, 2, 1, 2, 1, 3, 2, 1, 4, 3]
\big\},
\end{multline*}}%
each of them giving rise to $m/2$ elements of
$\Fix_{NC^m(H_4)}(\phi^{p})$ since $i$ ranges from $1$ to $m/2$,
and one obtains $30$ pairs $(w_{i_1},w_{i_2})$ of solutions in 
\eqref{eq:H4''E}:
{\small
\begin{multline} \label{eq:H4sol4}
(w_{i_1},w_{i_2})\in\big\{                  (     [2],\, [2, 1, 3, 2, 1, 2, 1, 3, 2]),\ (
                      [2],\, [2, 1, 2, 3, 4, 3, 2, 1, 2]),\ (
                            [3],\, [3, 2, 1, 2, 3]),\\ (
             [3],\, [1, 3, 2, 1, 4, 3, 2, 1, 2, 1, 3, 2, 1, 4, 3]),\ (
                      [4],\, [1, 4, 3, 2, 1, 2, 1, 3, 4]),\\ (
                      [4],\, [2, 1, 2, 3, 4, 3, 2, 1, 2]),\ (
                               [2, 1, 2],\, [3]),\ (
                   [2, 1, 2],\, [1, 4, 3, 2, 1, 2, 1, 3, 4]),\\ (
                [3, 2, 1, 2, 3],\, [2, 1, 3, 2, 1, 2, 1, 3, 2]),\ (
          [3, 2, 1, 2, 3],\, [1, 3, 2, 1, 2, 3, 4, 3, 2, 1, 2, 1, 3]),\\ (
                         [1, 3, 2, 1, 2, 1, 3],\, [2]),\ (
                         [1, 3, 2, 1, 2, 1, 3],\, [4]),\ (
                         [2, 1, 2, 3, 2, 1, 2],\, [4]),\\ (
                      [2, 1, 2, 3, 2, 1, 2],\, [2, 1, 2]),\ (
                         [1, 2, 3, 4, 3, 2, 1],\, [2]),\ (
                   [1, 2, 3, 4, 3, 2, 1],\, [3, 2, 1, 2, 3]),\\ (
             [2, 1, 3, 2, 1, 2, 1, 3, 2],\, [1, 3, 2, 1, 2, 1, 3]),\ (
             [2, 1, 3, 2, 1, 2, 1, 3, 2],\, [2, 1, 2, 3, 2, 1, 2]),\\ (
 [1, 4, 3, 2, 1, 2, 1, 3, 4],\, [2, 1, 2, 1, 4, 3, 2, 1, 2, 1, 3, 2, 1, 2, 4]),\\ (
 [1, 4, 3, 2, 1, 2, 1, 3, 4],\, [1, 3, 2, 1, 4, 3, 2, 1, 2, 1, 3, 2, 1, 4, 3]),\\ (
             [2, 1, 2, 3, 4, 3, 2, 1, 2],\, [2, 1, 2, 3, 2, 1, 2]),\\ (
 [2, 1, 2, 3, 4, 3, 2, 1, 2],\, [2, 1, 2, 1, 4, 3, 2, 1, 2, 1, 3, 2, 1, 2, 4]),\\ (
                [1, 2, 1, 4, 3, 2, 1, 2, 1, 3, 2, 1, 4],\, [3]),\\ (
       [1, 2, 1, 4, 3, 2, 1, 2, 1, 3, 2, 1, 4],\, [1, 2, 3, 4, 3, 2, 1]),\\ (
       [1, 3, 2, 1, 2, 3, 4, 3, 2, 1, 2, 1, 3],\, [1, 3, 2, 1, 2, 1, 3]),\\ (
       [1, 3, 2, 1, 2, 3, 4, 3, 2, 1, 2, 1, 3],\, [1, 2, 3, 4, 3, 2, 1]),\\ (
          [2, 1, 2, 1, 4, 3, 2, 1, 2, 1, 3, 2, 1, 2, 4],\, [2, 1, 2]),\\ (
               [2, 1, 2, 1, 4, 3, 2, 1, 2, 1, 3, 2, 1, 2, 4],\, 
                 [1, 2, 1, 4, 3, 2, 1, 2, 1, 3, 2, 1, 4]),\\ (
               [1, 3, 2, 1, 4, 3, 2, 1, 2, 1, 3, 2, 1, 4, 3],\, 
                 [1, 2, 1, 4, 3, 2, 1, 2, 1, 3, 2, 1, 4]),\\ (
               [1, 3, 2, 1, 4, 3, 2, 1, 2, 1, 3, 2, 1, 4, 3],\, 
                 [1, 3, 2, 1, 2, 3, 4, 3, 2, 1, 2, 1, 3])
\big\},
\end{multline}}%
each of them giving rise to $\binom {m/2}2$ elements of
$\Fix_{NC^m(H_4)}(\phi^{p})$ since $1\le i_1<i_2\le \frac m2$.

In total, we obtain 
$1+(15+8)\frac m2+30\binom {m/2}2=\frac {(3m+2)(5m+2)}4$ elements in
$\Fix_{NC^m(H_4)}(\phi^p)$, which agrees with the limit in
\eqref{eq:H4.7}.

\smallskip
Finally, we turn to \eqref{eq:H4.1}. By Remark~\ref{rem:1},
the only choices for $h_2$ and $m_2$ to be considered
are $h_2=2$ and $m_2=4$, respectively $h_2=m_2=2$. 
These correspond to the choices $p=15m/2$,
respectively $p=15m/4$, out of which only $p=15m/4$
has not yet been discussed and belongs to the current case.
The corresponding action of $\phi^p$ is given by 
\begin{multline*}
\phi^p\big((w_0;w_1,\dots,w_m)\big)\\=
(*;
c^{4}w_{\frac m4+1}c^{-4},c^{4}w_{\frac m4+2}c^{-4},
\dots,c^{4}w_{m}c^{-4},
c^3w_{1}c^{-3},\dots,
c^3w_{\frac {m}4}c^{-3}\big),
\end{multline*}
so that we have to solve
$$t_1(c^{11}t_1c^{-11})(c^{7}t_1c^{-7})(c^{3}t_1c^{-3})=c$$
for $t_1$ with $\ell_T(t_1)$.
A computation with Stembridge's {\sl Maple} package {\tt coxeter}
\cite{StemAZ} finds no solution. 
Hence, 
the left-hand side of \eqref{eq:1} is equal to $1$, as required.

\subsection*{\sc Case $G_{32}$}
The degrees are $12,18,24,30$, and hence we have
$$
\Cat^m(G_{32};q)=\frac 
{[30m+30]_q\, [30m+24]_q\, [30m+18]_q\, [30m+12]_q} 
{[30]_q\, [24]_q\, [18]_q\, [12]_q} .
$$
Let $\zeta$ be a $30m$-th root of unity. 
The following cases on the right-hand side of \eqref{eq:1}
occur:
{\refstepcounter{equation}\label{eq:G32}}
\alphaeqn
\begin{align} 
\label{eq:G32.2}
\lim_{q\to\zeta}\Cat^m(G_{32};q)&=m+1,
\quad\text{if }\zeta=\zeta_{30},\zeta_{15},\zeta_{10},\zeta_5,\\
\label{eq:G32.3}
\lim_{q\to\zeta}\Cat^m(G_{32};q)&=\tfrac {5m+4}4,
\quad\text{if }\zeta=\zeta_{24},\zeta_8,\ 4\mid m,\\
\label{eq:G32.4}
\lim_{q\to\zeta}\Cat^m(G_{32};q)&=\tfrac {5m+3}3,
\quad\text{if }\zeta=\zeta_{18},\zeta_9,\ 3\mid m,\\
\label{eq:G32.5}
\lim_{q\to\zeta}\Cat^m(G_{32};q)&=\tfrac {(5m+4)(5m+2)}{8},
\quad\text{if }\zeta=\zeta_{12},\zeta_{4},\ 2\mid m,\\
\label{eq:G32.6}
\lim_{q\to\zeta}\Cat^m(G_{32};q)&=\Cat^m(G_{32}),
\quad\text{if }\zeta=\zeta_6,\zeta_3,-1,1,\\
\label{eq:G32.1}
\lim_{q\to\zeta}\Cat^m(G_{32};q)&=1,
\quad\text{otherwise.}
\end{align}
\reseteqn

We must now prove that the left-hand side of \eqref{eq:1} in
each case agrees with the values exhibited in 
\eqref{eq:G32}. The only cases not covered by
Lemmas~\ref{lem:2} and \ref{lem:3} are the ones in 
\eqref{eq:G32.3}, \eqref{eq:G32.4}, 
\eqref{eq:G32.5}, 
and \eqref{eq:G32.1}.

\smallskip
We begin with the case in \eqref{eq:G32.3}.
By Lemma~\ref{lem:1}, 
we are free to choose $p=5m/4$ if $\zeta=\zeta_{24}$, 
and we are free to choose $p=15m/4$ if $\zeta=\zeta_{8}$. 
In particular, in all both cases,
$m$ must be divisible by $4$.

We start with the case that $p=5m/4$.
From \eqref{eq:Aktion}, we infer
\begin{multline*}
\phi^p\big((w_0;w_1,\dots,w_m)\big)\\=
(*;
c^{2}w_{\frac {3m}4+1}c^{-2},c^{2}w_{\frac {3m}4+2}c^{-2},
\dots,c^{2}w_{m}c^{-2},
cw_{1}c^{-1},\dots,
cw_{\frac {3m}4}c^{-1}\big).
\end{multline*}
Supposing that 
$(w_0;w_1,\dots,w_m)$ is fixed by $\phi^p$, we obtain
the system of equations
{\refstepcounter{equation}\label{eq:G32A}}
\alphaeqn
\begin{align} \label{eq:G32Aa}
w_i&=c^2w_{\frac {3m}4+i}c^{-2}, \quad i=1,2,\dots,\tfrac {m}4,\\
w_i&=cw_{i-\frac {m}4}c^{-1}, \quad i=\tfrac {m}4+1,\tfrac {m}4+2,\dots,m.
\label{eq:G32Ab}
\end{align}
\reseteqn
There are two distinct possibilities for choosing
the $w_i$'s, $1\le i\le m$: 
{\refstepcounter{equation}\label{eq:G32C}}
\alphaeqn
\begin{enumerate}
\item[(i)]
all the $w_i$'s are equal to $\ep$ (and $w_0=c$), 
\item[(ii)]
there is an $i$ with $1\le i\le \frac m4$ such that
\begin{equation} \label{eq:G32Cii}
\ell_T(w_i)=\ell_T(w_{i+\frac m4})=\ell_T(w_{i+\frac {2m}4})
=\ell_T(w_{i+\frac {3m}4})=1,
\end{equation}
and the other $w_j$'s, $1\le j\le m$, are equal to $\ep$.
\end{enumerate}
\reseteqn

Moreover, since $(w_0;w_1,\dots,w_m)\in NC^m(G_{32})$, we must have 
$$w_iw_{i+\frac m4}w_{i+\frac {2m}4}w_{i+\frac {3m}4}= c.$$
Together with equations~\eqref{eq:G32A}--\eqref{eq:G32C}, 
this implies that
\begin{equation} \label{eq:G32D}
w_i=c^5w_ic^{-5}\quad\text{and}\quad 
w_i(cw_ic^{-1})(c^2w_ic^{-2})(c^3w_ic^{-3})=c.
\end{equation}
With the help of {\tt CHEVIE}, one obtains five solutions for $w_i$ in 
\eqref{eq:G32D}:
{\small
\begin{equation} \label{eq:G32sol1}
w_i\in\big\{[1],\,
[2],\,
[3],\,
[4],\,
[27]\big\},
\end{equation}}%
where we have again used the short notation of {\tt CHEVIE}
referring to the internal
ordering of the roots of $G_{32}$ in {\tt CHEVIE}.
Each of the above solutions for $w_i$ gives rise to $m/4$ elements of
$\Fix_{NC^m(G_{32})}(\phi^{p})$ since $i$ ranges from $1$ to $m/4$.

In total, we obtain 
$1+5\frac m4=\frac {5m+4}4$ elements in
$\Fix_{NC^m(G_{32})}(\phi^p)$, which agrees with the limit in
\eqref{eq:G32.3}.

In the case that $p=15m/4$, we infer from \eqref{eq:Aktion} that
\begin{multline*}
\phi^p\big((w_0;w_1,\dots,w_m)\big)\\=
(*;
c^{4}w_{\frac {m}4+1}c^{-4},c^{4}w_{\frac {m}4+2}c^{-4},
\dots,c^{4}w_{m}c^{-4},
c^3w_{1}c^{-3},\dots,
c^3w_{\frac {m}4}c^{-3}\big).
\end{multline*}
Supposing that 
$(w_0;w_1,\dots,w_m)$ is fixed by $\phi^p$, we obtain
the system of equations
{\refstepcounter{equation}\label{eq:G32'''A}}
\alphaeqn
\begin{align} \label{eq:G32'''Aa}
w_i&=c^4w_{\frac {m}4+i}c^{-4}, \quad i=1,2,\dots,\tfrac {3m}4,\\
w_i&=c^3w_{i-\frac {3m}4}c^{-3}, \quad i=\tfrac {3m}4+1,\tfrac {3m}4+2,\dots,m.
\label{eq:G32'''Ab}
\end{align}
\reseteqn
There are two distinct possibilities for choosing
the $w_i$'s, $1\le i\le m$: 
either all the $w_i$'s are equal to $\ep$, or
there is an $i$ with $1\le i\le \frac m4$ such that
$$\ell_T(w_i)=\ell_T(w_{i+\frac m4})=
\ell_T(w_{i+\frac {2m}4})=\ell_T(w_{i+\frac {3m}4})=1.$$
Writing $t_1,t_2,t_3,t_4$ for 
$w_i,w_{i+\frac m4},w_{i+\frac {2m}4},w_{i+\frac {3m}4}$, 
respectively, the equations \eqref{eq:G32'''A} 
reduce to
{\refstepcounter{equation}\label{eq:G32'''B}}
\alphaeqn
\begin{align} \label{eq:G32'''Ba}
t_1&=c^4t_2c^{-4},\\
\label{eq:G32'''Bb}
t_2&=c^4t_3c^{-4},\\
\label{eq:G32'''Bc}
t_3&=c^4t_4c^{-4},\\
t_4&=c^3t_1c^{-3}.
\label{eq:G32'''Bd}
\end{align}
\reseteqn
One of these equations is in fact superfluous: if we substitute
\eqref{eq:G32'''Bb}--\eqref{eq:G32'''Bd} in \eqref{eq:G32'''Ba}, then
we obtain $t_1=c^{15}t_1c^{-15}$ which is automatically satisfied 
due to Lemma~\ref{lem:4} with $d=2$.

Since $(w_0;w_1,\dots,w_m)\in NC^m(G_{32})$, we must have $t_1t_2t_3t_4=c$.
Combining this with \eqref{eq:G32'''B}, we infer that
\begin{equation} \label{eq:G32'''D}
t_1(c^{11}t_1c^{-11})(c^{7}t_1c^{-7})(c^{3}t_1c^{-3})=c.
\end{equation}
Using that $c^5t_1c^{-5}=t_1$, due to Lemma~\ref{lem:4}
with $d=6$, we see that
this equation is equivalent with \eqref{eq:G32D}. 
Therefore, we are facing exactly the same enumeration 
problem here as for
$p=5m/4$, and, consequently, the number of solutions to \eqref{eq:G32'''D} is the same, namely
$\frac {5m+4}4$, as required.

\smallskip
Next we consider the case in \eqref{eq:G32.4}.
By Lemma~\ref{lem:1}, 
we are free to choose $p=5m/3$ if $\zeta=\zeta_{18}$, 
and we are free to choose $p=10m/3$ if $\zeta=\zeta_{9}$.
In particular, in both cases,
$m$ must be divisible by $3$.

We start with the case that $p=5m/3$.
From \eqref{eq:Aktion}, we infer
$$
\phi^p\big((w_0;w_1,\dots,w_m)\big)\\=
(*;
c^{2}w_{\frac m3+1}c^{-2},c^{2}w_{\frac m3+2}c^{-2},
\dots,c^{2}w_{m}c^{-2},
cw_{1}c^{-1},\dots,
cw_{\frac m3}c^{-1}\big).
$$
Supposing that 
$(w_0;w_1,\dots,w_m)$ is fixed by $\phi^p$, we obtain
the system of equations
{\refstepcounter{equation}\label{eq:G32''''A}}
\alphaeqn
\begin{align} \label{eq:G32''''Aa}
w_i&=c^2w_{\frac m3+i}c^{-2}, \quad i=1,2,\dots,\tfrac {2m}3,\\
w_i&=cw_{i-\frac {2m}3}c^{-1}, \quad i=\tfrac {2m}3+1,\tfrac {2m}3+2,\dots,m.
\label{eq:G32''''Ab}
\end{align}
\reseteqn
There are two distinct possibilities for choosing
the $w_i$'s, $1\le i\le m$: 
{\refstepcounter{equation}\label{eq:G32''''C}}
\alphaeqn
\begin{enumerate}
\item[(i)]
all the $w_i$'s are equal to $\ep$ (and $w_0=c$), 
\item[(ii)]
there is an $i$ with $1\le i\le \frac m2$ such that
\begin{equation} \label{eq:G32''''Cii}
\ell_T(w_i)=\ell_T(w_{i+\frac m3})=\ell_T(w_{i+\frac {2m}3})=1,
\end{equation}
and the other $w_j$'s, $1\le j\le m$, are equal to $\ep$.
\end{enumerate}
\reseteqn

Moreover, since $(w_0;w_1,\dots,w_m)\in NC^m(G_{32})$, we must have 
$w_iw_{i+\frac m3}w_{i+\frac {2m}3}\le_T c$.
Together with equations~\eqref{eq:G32''''A}--\eqref{eq:G32''''C}, 
this implies that
\begin{equation} \label{eq:G32''''D}
w_i=c^5w_ic^{-5}\quad\text{and}\quad 
w_i(c^3w_ic^{-3})(cw_ic^{-1})\le_T c.
\end{equation}
With the help of {\tt CHEVIE}, one obtains three solutions for $w_i$ in 
\eqref{eq:G32''''D}:
{\small
$$w_i\in\big\{[1],\,
[2],\,
[3],\,
[4],\,
[27]\big\},$$}%
where we have again used the short notation of {\tt CHEVIE}
referring to the internal
ordering of the roots of $G_{32}$ in {\tt CHEVIE}.
Each of the above solutions for $w_i$ gives rise to $m/3$ elements of
$\Fix_{NC^m(G_{32})}(\phi^{p})$ since $i$ ranges from $1$ to $m/3$.

In total, we obtain 
$1+5\frac m3=\frac {5m+3}3$ elements in
$\Fix_{NC^m(G_{32})}(\phi^p)$, which agrees with the limit in
\eqref{eq:G32.4}.

In the case that $p=10m/3$, we infer from \eqref{eq:Aktion} that
\begin{multline*}
\phi^p\big((w_0;w_1,\dots,w_m)\big)\\=
(*;
c^{4}w_{\frac {2m}3+1}c^{-4},c^{4}w_{\frac {2m}3+2}c^{-4},
\dots,c^{4}w_{m}c^{-4},
c^3w_{1}c^{-3},\dots,
c^3w_{\frac {2m}3}c^{-3}\big).
\end{multline*}
Supposing that 
$(w_0;w_1,\dots,w_m)$ is fixed by $\phi^p$, we obtain
the system of equations
{\refstepcounter{equation}\label{eq:G32'''''''A}}
\alphaeqn
\begin{align} \label{eq:G32'''''''Aa}
w_i&=c^4w_{\frac {2m}3+i}c^{-4}, \quad i=1,2,\dots,\tfrac {m}3,\\
w_i&=c^3w_{i-\frac {m}3}c^{-3}, \quad i=\tfrac {m}3+1,\tfrac {m}3+2,\dots,m.
\label{eq:G32'''''''Ab}
\end{align}
\reseteqn
There are two distinct possibilities for choosing
the $w_i$'s, $1\le i\le m$: 
either all the $w_i$'s are equal to $\ep$, or
there is an $i$ with $1\le i\le \frac m3$ such that
$$\ell_T(w_i)=\ell_T(w_{i+\frac m3})=\ell_T(w_{i+\frac {2m}3})=1.$$
Writing $t_1,t_2,t_3$ for $w_i,w_{i+\frac m3},w_{i+\frac {2m}3}$, 
respectively, the equations \eqref{eq:G32'''''''A} 
reduce to
{\refstepcounter{equation}\label{eq:G32'''''''B}}
\alphaeqn
\begin{align} \label{eq:G32'''''''Ba}
t_1&=c^4t_3c^{-4},\\
\label{eq:G32'''''''Bb}
t_2&=c^3t_1c^{-3},\\
t_3&=c^3t_2c^{-3}.
\label{eq:G32'''''''Bc}
\end{align}
\reseteqn
One of these equations is in fact superfluous: if we substitute
\eqref{eq:G32'''''''Bb} and \eqref{eq:G32'''''''Bc} in \eqref{eq:G32'''''''Ba}, then
we obtain $t_1=c^{10}t_1c^{-10}$ which is automatically satisfied 
due to Lemma~\ref{lem:4} with $d=2$.

Since $(w_0;w_1,\dots,w_m)\in NC^m(G_{32})$, we must have
$t_1t_2t_3\le_T c$.
Combining this with \eqref{eq:G32'''''''B}, we infer that
\begin{equation} \label{eq:G32'''''''D}
t_1(c^{3}t_1c^{-3})(c^{6}t_1c^{-6})\le_T c.
\end{equation}
Using that $c^5t_1c^{-5}=t_1$, due to Lemma~\ref{lem:4}
with $d=4$, we see that
this equation is equivalent with \eqref{eq:G32''''D}. 
Therefore, we are facing exactly the same enumeration 
problem here as for
$p=5m/3$, and, consequently, the number of solutions to \eqref{eq:G32'''''''D} is the same, namely
$\frac {5m+3}3$, as required.

\smallskip
Next we discuss the case in \eqref{eq:G32.5}.
By Lemma~\ref{lem:1}, 
we are free to choose $p=5m/2$ if $\zeta=\zeta_{12}$, 
and we are free to choose $p=15m/2$ if $\zeta=\zeta_{4}$. 
In particular,
$m$ must be divisible by $2$.
From \eqref{eq:Aktion}, we infer
\begin{multline} \label{eq:55m22Aktion}
\phi^p\big((w_0;w_1,\dots,w_m)\big)\\=
(*;
c^{3}w_{\frac m2+1}c^{-3},c^{3}w_{\frac m2+2}c^{-3},
\dots,c^{3}w_{m}c^{-3},
c^2w_{1}c^{-2},\dots,
c^2w_{\frac m2}c^{-2}\big).
\end{multline}
Supposing that 
$(w_0;w_1,\dots,w_m)$ is fixed by $\phi^p$, we obtain
the system of equations
{\refstepcounter{equation}\label{eq:G32'A}}
\alphaeqn
\begin{align} \label{eq:G32'Aa}
w_i&=c^3w_{\frac m2+i}c^{-3}, \quad i=1,2,\dots,\tfrac {m}2,\\
w_i&=c^2w_{i-\frac {m}2}c^{-2}, \quad i=\tfrac {m}2+1,\tfrac {m}2+2,\dots,m.
\label{eq:G32'Ab}
\end{align}
\reseteqn
There are four distinct possibilities for choosing
the $w_i$'s, $1\le i\le m$: 
{\refstepcounter{equation}\label{eq:G32'C}}
\alphaeqn
\begin{enumerate}
\item[(i)]
all the $w_i$'s are equal to $\ep$ (and $w_0=c$), 
\item[(ii)]
there is an $i$ with $1\le i\le \frac m2$ such that
\begin{equation} \label{eq:G32'Cii}
\ell_T(w_i)=\ell_T(w_{i+\frac m2})=2,
\end{equation}
and all other $w_j$'s are equal to $\ep$,
\item[(iii)]
there is an $i$ with $1\le i\le \frac m2$ such that
\begin{equation} \label{eq:G32'Ciii}
\ell_T(w_i)=\ell_T(w_{i+\frac m2})=1,
\end{equation}
and the other $w_j$'s, $1\le j\le m$, are equal to $\ep$,
\item[(iv)]
there are $i_1$ and $i_2$ with $1\le i_1<i_2\le \frac m2$ such that
\begin{equation} \label{eq:G32'Civ}
\ell_T(w_{i_1})=\ell_T(w_{i_2})=
\ell_T(w_{i_1+\frac m2})=\ell_T(w_{i_2+\frac m2})=1,
\end{equation}
and all other $w_j$'s are equal to $\ep$.
\end{enumerate}
\reseteqn

Moreover, since $(w_0;w_1,\dots,w_m)\in NC^m(G_{32})$, we must have 
$w_iw_{i+\frac m2}\le_T c$, respectively 
$w_{i_1}w_{i_2}w_{i_1+\frac m2}w_{i_2+\frac m2}=c$.
Together with equations~\eqref{eq:G32'A}--\eqref{eq:G32'C}, 
this implies that
\begin{equation} \label{eq:G32'D}
w_i=c^3w_ic^{-3}\quad\text{and}\quad 
w_i(c^2w_ic^{-2})=c,
\end{equation}
respectively that
\begin{equation} \label{eq:G32'D2}
w_i=c^3w_ic^{-3},\quad 
w_i(c^2w_ic^{-2})\le_T c,\quad\text{and}\quad
\ell_T(w_i)=1,
\end{equation}
respectively that
\begin{multline} \label{eq:G32'E}
w_{i_1}=c^3w_{i_1}c^{-3},\quad 
w_{i_2}=c^3w_{i_2}c^{-3},\\
w_{i_1}w_{i_2}(c^2w_{i_1}c^{-2})(c^2w_{i_2}c^{-2})=c,
\quad\text{and}\quad \ell_T(w_{i_1})=\ell_T(w_{i_2})=1.
\end{multline}
With the help of {\tt CHEVIE}, one obtains ten solutions for $w_i$ in 
\eqref{eq:G32'D}:
{\small
\begin{equation} \label{eq:G32sol2}
w_i\in\big\{[ 3, 4 ],\,
[ 1, 20 ],\,
[ 4, 7 ],\,
[ 1, 2 ],\,
[ 2, 3 ],\,
[ 4, 27 ],\,
[ 3, 23 ],\,
[ 1, 23 ],\,
[ 2, 16 ],\,
[ 12, 27 ]
\big\},
\end{equation}}%
where we have again used the short notation of {\tt CHEVIE}
referring to the internal
ordering of the roots of $G_{32}$ in {\tt CHEVIE},
one obtains solutions for $w_i$ in 
\eqref{eq:G32'D2}:
{\small
\begin{equation*}
w_i\in\big\{[ 4 ]\,\
[ 3 ]\,\
[ 20 ]\,\
[ 1 ]\,\
[ 7 ]\,\
[ 2 ]\,\
[ 16 ]\,\
[ 12 ]\,\
[ 27 ]\,\
[ 23 ]
\big\},
\end{equation*}}%
each of them giving rise to $m/2$ elements of
$\Fix_{NC^m(G_{32})}(\phi^{p})$ since $i$ ranges from $1$ to $m/2$,
and one obtains 25 pairs $(w_{i_1},w_{i_2})$ satisfying 
\eqref{eq:G32'E}:
{\small
\begin{multline} \label{eq:G32sol3}
(w_{i_1},w_{i_2})\in\big\{( [ 4 ],\, [ 20 ] ),\
( [ 4 ],\, [ 7 ] ),\
( [ 4 ],\, [ 27 ] ),\
( [ 3 ],\, [ 4 ] ),\
( [ 3 ],\, [ 12 ] ),\
( [ 3 ],\, [ 23 ] ),\
( [ 20 ],\, [ 3 ] ),\\
( [ 20 ],\, [ 1 ] ),\
( [ 1 ],\, [ 20 ] ),\
( [ 1 ],\, [ 2 ] ),\
( [ 1 ],\, [ 23 ] ),\
( [ 7 ],\, [ 4 ] ),\
( [ 7 ],\, [ 1 ] ),\
( [ 2 ],\, [ 3 ] ),\
( [ 2 ],\, [ 7 ] ),\\
( [ 2 ],\, [ 16 ] ),\
( [ 16 ],\, [ 4 ] ),\
( [ 16 ],\, [ 2 ] ),\
( [ 12 ],\, [ 2 ] ),\
( [ 12 ],\, [ 27 ] ),\
( [ 27 ],\, [ 1 ] ),\
( [ 27 ],\, [ 16 ] ),\\
( [ 27 ],\, [ 12 ] ),\
( [ 23 ],\, [ 3 ] ),\
( [ 23 ],\, [ 27 ] )
\big\},
\end{multline}}%
each of them giving rise to $\binom {m/2}2$ elements of
$\Fix_{NC^m(G_{32})}(\phi^{p})$ since $1\le i_1<i_2\le \frac m2$. 

In total, we obtain 
$1+20\frac m2+25\binom {m/2}2=\frac {(5m+4)(5m+2)}8$ elements in
$\Fix_{NC^m(G_{32})}(\phi^p)$, which agrees with the limit in
\eqref{eq:G32.5}.

\smallskip
If $p=15m/2$, then, from \eqref{eq:Aktion}, we infer
\begin{multline*}
\phi^p\big((w_0;w_1,\dots,w_m)\big)\\=
(*;
c^{8}w_{\frac m2+1}c^{-8},c^{8}w_{\frac m2+2}c^{-8},
\dots,c^{8}w_{m}c^{-8},
c^7w_{1}c^{-7},\dots,
c^7w_{\frac m2}c^{-7}\big).
\end{multline*}
Using that $c^5wc^{-5}=w$ for all $w\in NC(G_{32}$, due to
Lemma~\ref{lem:4} with $d=6$, we see that this action is
identical with the one in \eqref{eq:55m22Aktion}.
Therefore, we are facing exactly the same enumeration 
problem here as for
$p=5m/2$, and, consequently, the number of elements
in $\Fix_{NC^m(G_{32})}(\phi^p)$ is the same, namely
$\frac {(5m+4)(5m+2)}8$, as required.

\smallskip
Finally, we turn to \eqref{eq:G32.1}. By Remark~\ref{rem:1},
the only choices for $h_2$ and $m_2$ to be considered
are $h_2=2$ and $m_2=4$, $h_2=m_2=2$, $h_2=m_2=3$, 
$h_2=6$ and $m_2=4$, $h_2=6$ and $m_2=3$, 
respectively $h_2=6$ and $m_2=2$. 
These correspond to the choices $p=15m/4$, $p=15m/2$, 
$p=10m/3$, $p=5m/4$, $p=5m/3$,
respectively $p=5m/2$, all of which 
have already been discussed as they do not belong to 
\eqref{eq:G32.1}. Hence, 
\eqref{eq:1} must necessarily hold, as required.

\subsection*{\sc Case $G_{33}$}
The degrees are $4,6,10,12,18$, and hence we have
$$
\Cat^m(G_{33};q)=\frac 
{[18m+18]_q\, [18m+12]_q\, [18m+10]_q\, [18m+6]_q\, [18m+4]_q} 
{[18]_q\, [12]_q\, [10]_q\, [6]_q\, [4]_q} .
$$
Let $\zeta$ be a $18m$-th root of unity. 
The following cases on the right-hand side of \eqref{eq:1}
occur:
{\refstepcounter{equation}\label{eq:G33}}
\alphaeqn
\begin{align} 
\label{eq:G33.2}
\lim_{q\to\zeta}\Cat^m(G_{33};q)&=m+1,
\quad\text{if }\zeta=\zeta_{18},\zeta_9,\\
\label{eq:G33.3}
\lim_{q\to\zeta}\Cat^m(G_{33};q)&=\tfrac {3m+2}2,
\quad\text{if }\zeta=\zeta_{12},\ 2\mid m,\\
\label{eq:G33.4}
\lim_{q\to\zeta}\Cat^m(G_{33};q)&=\tfrac {9m+5}5,
\quad\text{if }\zeta=\zeta_{10},\zeta_5,\ 5\mid m,\\
\label{eq:G33.5}
\lim_{q\to\zeta}\Cat^m(G_{33};q)&=\tfrac {(m+1)(3m+2)(3m+1)}{2},
\quad\text{if }\zeta=\zeta_{6},\zeta_{3},\\
\label{eq:G33.6}
\lim_{q\to\zeta}\Cat^m(G_{33};q)&=\tfrac {(3m+2)(9m+2)}{4},
\quad\text{if }\zeta=\zeta_{4},\ 2\mid m,\\
\label{eq:G33.7}
\lim_{q\to\zeta}\Cat^m(G_{33};q)&=\Cat^m(G_{33}),
\quad\text{if }\zeta=-1\text{ or }\zeta=1,\\
\label{eq:G33.1}
\lim_{q\to\zeta}\Cat^m(G_{33};q)&=1,
\quad\text{otherwise.}
\end{align}
\reseteqn

We must now prove that the left-hand side of \eqref{eq:1} in
each case agrees with the values exhibited in 
\eqref{eq:G33}. The only cases not covered by
Lemmas~\ref{lem:2} and \ref{lem:3} are the ones in 
\eqref{eq:G33.3}, \eqref{eq:G33.4}, 
\eqref{eq:G33.6}, 
and \eqref{eq:G33.1}.

\smallskip
We begin with the case in \eqref{eq:G33.3}.
By Lemma~\ref{lem:1}, 
we are free to choose $p=3m/2$. 
In particular,
$m$ must be divisible by $2$.
From \eqref{eq:Aktion}, we infer
\begin{equation*} 
\phi^p\big((w_0;w_1,\dots,w_m)\big)=
(*;
c^{2}w_{\frac m2+1}c^{-2},c^{2}w_{\frac m2+2}c^{-2},
\dots,c^{2}w_{m}c^{-2},
cw_{1}c^{-1},\dots,
cw_{\frac m2}c^{-1}\big).
\end{equation*}
Supposing that 
$(w_0;w_1,\dots,w_m)$ is fixed by $\phi^p$, we obtain
the system of equations
{\refstepcounter{equation}\label{eq:G33'A}}
\alphaeqn
\begin{align} \label{eq:G33'Aa}
w_i&=c^2w_{\frac m2+i}c^{-2}, \quad i=1,2,\dots,\tfrac {m}2,\\
w_i&=cw_{i-\frac {m}2}c^{-1}, \quad i=\tfrac {m}2+1,\tfrac {m}2+2,\dots,m.
\label{eq:G33'Ab}
\end{align}
\reseteqn
There are four distinct possibilities for choosing
the $w_i$'s, $1\le i\le m$: 
{\refstepcounter{equation}\label{eq:G33'C}}
\alphaeqn
\begin{enumerate}
\item[(i)]
all the $w_i$'s are equal to $\ep$ (and $w_0=c$), 
\item[(ii)]
there is an $i$ with $1\le i\le \frac m2$ such that
\begin{equation} \label{eq:G33'Cii}
\ell_T(w_i)=\ell_T(w_{i+\frac m2})=2,
\end{equation}
and the other $w_j$'s, $1\le j\le m$, are equal to $\ep$,
\item[(iii)]
there is an $i$ with $1\le i\le \frac m2$ such that
\begin{equation} \label{eq:G33'Ciii}
\ell_T(w_i)=\ell_T(w_{i+\frac m2})=1,
\end{equation}
and the other $w_j$'s, $1\le j\le m$, are equal to $\ep$,
\item[(iv)]
there are $i_1$ and $i_2$ with $1\le i_1<i_2\le \frac m2$ such that
\begin{equation} \label{eq:G33'Civ}
\ell_T(w_{i_1})=\ell_T(w_{i_2})=
\ell_T(w_{i_1+\frac m2})=\ell_T(w_{i_2+\frac m2})=1,
\end{equation}
and the other $w_j$'s, $1\le j\le m$, are equal to $\ep$.
\end{enumerate}
\reseteqn

Moreover, since $(w_0;w_1,\dots,w_m)\in NC^m(G_{33})$, we must have 
$w_iw_{i+\frac m2}\le_T c$, respectively 
$w_{i_1}w_{i_2}w_{i_1+\frac m2}w_{i_2+\frac m2}\le_T c$.
Together with equations~\eqref{eq:G33'A}--\eqref{eq:G33'C}, 
this implies that
\begin{equation} \label{eq:G33'D}
w_i=c^3w_ic^{-3}\quad\text{and}\quad 
w_i(cw_ic^{-1})\le_T c,\quad\text{and}\quad
\ell_T(w_i)=2,
\end{equation}
respectively that
\begin{equation} \label{eq:G33'D2}
w_i=c^3w_ic^{-3},\quad 
w_i(cw_ic^{-1})\le_T c,\quad\text{and}\quad
\ell_T(w_i)=1,
\end{equation}
respectively that
\begin{equation} \label{eq:G33'E}
w_{i_1}=c^3w_{i_1}c^{-3},\quad 
w_{i_1}(cw_{i_1}c^{-1})\le_T c,
\quad\text{and}\quad \ell_T(w_{i_1})=1.
\end{equation}
With the help of {\tt CHEVIE}, one obtains three solutions for $w_i$ in 
\eqref{eq:G33'D}:
{\small
\begin{equation*} 
w_i\in\big\{[ 1, 44 ],\,
[ 2, 4 ],\,
[ 5, 42 ]
\big\},
\end{equation*}}%
where we have again used the short notation of {\tt CHEVIE}
referring to the internal
ordering of the roots of $G_{33}$ in {\tt CHEVIE}.
Each of them gives rise to $m/2$ elements of
$\Fix_{NC^m(G_{33})}(\phi^{p})$ since $i$ ranges from $1$ to $m/2$.

There are no solutions to \eqref{eq:G33'D2} and to
\eqref{eq:G33'E}.

In total, we obtain 
$1+3\frac m2=\frac {3m+2}2$ elements in
$\Fix_{NC^m(G_{33})}(\phi^p)$, which agrees with the limit in
\eqref{eq:G33.3}.

\smallskip
Next we turn to the case in \eqref{eq:G33.4}.
By Lemma~\ref{lem:1}, 
we are free to choose $p=9m/5$ if $\zeta=\zeta_{10}$, 
and we are free to choose $p=18m/5$ if $\zeta=\zeta_{5}$. 
In particular, in all both cases,
$m$ must be divisible by $5$.

We start with the case that $p=9m/5$.
From \eqref{eq:Aktion}, we infer
\begin{multline*}
\phi^p\big((w_0;w_1,\dots,w_m)\big)\\=
(*;
c^{2}w_{\frac {m}5+1}c^{-2},c^{2}w_{\frac {m}5+2}c^{-2},
\dots,c^{2}w_{m}c^{-2},
cw_{1}c^{-1},\dots,
cw_{\frac {m}5}c^{-1}\big).
\end{multline*}
Supposing that 
$(w_0;w_1,\dots,w_m)$ is fixed by $\phi^p$, we obtain
the system of equations
{\refstepcounter{equation}\label{eq:G33A}}
\alphaeqn
\begin{align} \label{eq:G33Aa}
w_i&=c^2w_{\frac {m}5+i}c^{-2}, \quad i=1,2,\dots,\tfrac {4m}5,\\
w_i&=cw_{i-\frac {4m}5}c^{-1}, \quad i=\tfrac {4m}5+1,\tfrac {4m}5+2,\dots,m.
\label{eq:G33Ab}
\end{align}
\reseteqn
There are two distinct possibilities for choosing
the $w_i$'s, $1\le i\le m$: 
{\refstepcounter{equation}\label{eq:G33C}}
\alphaeqn
\begin{enumerate}
\item[(i)]
all the $w_i$'s are equal to $\ep$ (and $w_0=c$), 
\item[(ii)]
there is an $i$ with $1\le i\le \frac m5$ such that
\begin{equation} \label{eq:G33Cii}
\ell_T(w_i)=\ell_T(w_{i+\frac m5})=\ell_T(w_{i+\frac {2m}5})
=\ell_T(w_{i+\frac {3m}5})=\ell_T(w_{i+\frac {4m}5})=1,
\end{equation}
and the other $w_j$'s, $1\le j\le m$, are equal to $\ep$.
\end{enumerate}
\reseteqn

Moreover, since $(w_0;w_1,\dots,w_m)\in NC^m(G_{33})$, we must have 
$$w_iw_{i+\frac m5}w_{i+\frac {2m}5}w_{i+\frac {3m}5}
w_{i+\frac {4m}5}=c.$$
Together with equations~\eqref{eq:G33A}--\eqref{eq:G33C}, 
this implies that
\begin{equation} \label{eq:G33D}
w_i=c^9w_ic^{-9}\quad\text{and}\quad 
w_i(c^7w_ic^{-7})(c^5w_ic^{-5})(c^3w_ic^{-3})(cw_ic^{-1})=c.
\end{equation}
With the help of {\tt CHEVIE}, one obtains nine solutions for $w_i$ in 
\eqref{eq:G33D}:
{\small
\begin{equation} \label{eq:G33sol1}
w_i\in\big\{[ 5 ],\,
[ 4 ],\,
[ 1 ],\,
[ 2 ],\,
[ 10 ],\,
[ 27 ],\,
[ 42 ],\,
[ 44 ],\,
[ 52 ]
\big\},
\end{equation}}%
where we have again used the short notation of {\tt CHEVIE}
referring to the internal
ordering of the roots of $G_{33}$ in {\tt CHEVIE}.
Each of the above solutions for $w_i$ gives rise to $m/5$ elements of
$\Fix_{NC^m(G_{33})}(\phi^{p})$ since $i$ ranges from $1$ to $m/5$.

In total, we obtain 
$1+9\frac m5=\frac {9m+5}5$ elements in
$\Fix_{NC^m(G_{33})}(\phi^p)$, which agrees with the limit in
\eqref{eq:G33.4}.

In the case that $p=18m/5$, we infer from \eqref{eq:Aktion} that
\begin{multline*}
\phi^p\big((w_0;w_1,\dots,w_m)\big)\\=
(*;
c^{4}w_{\frac {2m}5+1}c^{-4},c^{4}w_{\frac {2m}5+2}c^{-4},
\dots,c^{4}w_{m}c^{-4},
c^3w_{1}c^{-3},\dots,
c^3w_{\frac {2m}5}c^{-3}\big).
\end{multline*}
Supposing that 
$(w_0;w_1,\dots,w_m)$ is fixed by $\phi^p$, we obtain
the system of equations
{\refstepcounter{equation}\label{eq:G33'''A}}
\alphaeqn
\begin{align} \label{eq:G33'''Aa}
w_i&=c^4w_{\frac {2m}5+i}c^{-4}, \quad i=1,2,\dots,\tfrac {3m}5,\\
w_i&=c^3w_{i-\frac {3m}5}c^{-3}, \quad i=\tfrac {3m}5+1,\tfrac {3m}5+2,\dots,m.
\label{eq:G33'''Ab}
\end{align}
\reseteqn
There are two distinct possibilities for choosing
the $w_i$'s, $1\le i\le m$: 
either all the $w_i$'s are equal to $\ep$, or
there is an $i$ with $1\le i\le \frac m5$ such that
$$\ell_T(w_i)=\ell_T(w_{i+\frac m5})=
\ell_T(w_{i+\frac {2m}5})=\ell_T(w_{i+\frac {3m}5})=
\ell_T(w_{i+\frac {4m}5})=1.$$
Writing $t_1,t_2,t_3,t_4,t_5$ for 
$w_i,w_{i+\frac m5},w_{i+\frac {2m}5},w_{i+\frac {3m}5},
w_{i+\frac {4m}5}$, 
respectively, the equations\break \eqref{eq:G33'''A} 
reduce to
{\refstepcounter{equation}\label{eq:G33'''B}}
\alphaeqn
\begin{align} \label{eq:G33'''Ba}
t_1&=c^4t_3c^{-4},\\
\label{eq:G33'''Bb}
t_2&=c^4t_4c^{-4},\\
\label{eq:G33'''Bc}
t_3&=c^4t_5c^{-4},\\
\label{eq:G33'''Bd}
t_4&=c^3t_1c^{-3},\\
t_5&=c^3t_2c^{-3}.
\label{eq:G33'''Be}
\end{align}
\reseteqn
One of these equations is in fact superfluous: if we substitute
\eqref{eq:G33'''Bb}--\eqref{eq:G33'''Be} in \eqref{eq:G33'''Ba}, then
we obtain $t_1=c^{18}t_1c^{-18}$ which is automatically satisfied 
since $c^{18}=\ep$.

Since $(w_0;w_1,\dots,w_m)\in NC^m(G_{33})$, we must have $t_1t_2t_3t_4t_5=c$.
Combining this with \eqref{eq:G33'''B}, we infer that
\begin{equation} \label{eq:G33'''D}
t_1(c^{7}t_1c^{-7})(c^{14}t_1c^{-14})(c^{3}t_1c^{-3})
(c^{10}t_1c^{-10})=c.
\end{equation}
Using that $c^9t_1c^{-9}=t_1$, due to Lemma~\ref{lem:4}
with $d=2$, we see that
this equation is equivalent with \eqref{eq:G33D}. 
Therefore, we are facing exactly the same enumeration 
problem here as for
$p=9m/5$, and, consequently, the number of solutions to \eqref{eq:G33'''D} is the same, namely
$\frac {9m+5}5$, as required.

\smallskip
Next we consider the case in \eqref{eq:G33.6}.
By Lemma~\ref{lem:1}, 
we are free to choose $p=9m/2$.
In particular, 
$m$ must be divisible by $2$.
From \eqref{eq:Aktion}, we infer
\begin{multline*}
\phi^p\big((w_0;w_1,\dots,w_m)\big)\\=
(*;
c^{5}w_{\frac m2+1}c^{-5},c^{5}w_{\frac m2+2}c^{-5},
\dots,c^{5}w_{m}c^{-5},
c^4w_{1}c^{-4},\dots,
c^4w_{\frac m2}c^{-4}\big).
\end{multline*}
Supposing that 
$(w_0;w_1,\dots,w_m)$ is fixed by $\phi^p$, we obtain
the system of equations
{\refstepcounter{equation}\label{eq:G33''''A}}
\alphaeqn
\begin{align} \label{eq:G33''''Aa}
w_i&=c^5w_{\frac m2+i}c^{-5}, \quad i=1,2,\dots,\tfrac {m}2,\\
w_i&=c^4w_{i-\frac {m}2}c^{-4}, \quad i=\tfrac {m}2+1,\tfrac {m}2+2,\dots,m.
\label{eq:G33''''Ab}
\end{align}
\reseteqn
There are four distinct possibilities for choosing
the $w_i$'s, $1\le i\le m$: 
{\refstepcounter{equation}\label{eq:G33''''C}}
\alphaeqn
\begin{enumerate}
\item[(i)]
all the $w_i$'s are equal to $\ep$ (and $w_0=c$), 
\item[(ii)]
there is an $i$ with $1\le i\le \frac m2$ such that
\begin{equation} \label{eq:G33''''Cii}
\ell_T(w_i)=\ell_T(w_{i+\frac m2})=2,
\end{equation}
and the other $w_j$'s, $1\le j\le m$, are equal to $\ep$,
\item[(iii)]
there is an $i$ with $1\le i\le \frac m2$ such that
\begin{equation} \label{eq:G33''''Ciii}
\ell_T(w_i)=\ell_T(w_{i+\frac m2})=1,
\end{equation}
and the other $w_j$'s, $1\le j\le m$, are equal to $\ep$,
\item[(iv)]
there are $i_1$ and $i_2$ with $1\le i_1<i_2\le \frac m2$ such that
\begin{equation} \label{eq:G33''''Civ}
\ell_T(w_{i_1})=\ell_T(w_{i_2})=
\ell_T(w_{i_1+\frac m2})=\ell_T(w_{i_2+\frac m2})=1,
\end{equation}
and the other $w_j$'s, $1\le j\le m$, are equal to $\ep$.
\end{enumerate}
\reseteqn

Moreover, since $(w_0;w_1,\dots,w_m)\in NC^m(G_{33})$, we must have 
$w_iw_{i+\frac m2}\le_T c$, respectively 
$w_{i_1}w_{i_2}w_{i_1+\frac m2}w_{i_2+\frac m2}\le_T c$.
Together with equations~\eqref{eq:G33''''A}--\eqref{eq:G33''''C}, 
this implies that
\begin{equation} \label{eq:G33''''D}
w_i=c^9w_ic^{-9}\quad\text{and}\quad 
w_i(c^4w_ic^{-4})\le_T c,\quad\text{and}\quad
\ell_T(w_i)=2,
\end{equation}
respectively that
\begin{equation} \label{eq:G33''''D2}
w_i=c^9w_ic^{-9},\quad 
w_i(c^4w_ic^{-4})\le_T c,\quad\text{and}\quad
\ell_T(w_i)=1,
\end{equation}
respectively that
\begin{multline} \label{eq:G33''''E}
w_{i_1}=c^9w_{i_1}c^{-9},\quad 
w_{i_2}=c^9w_{i_2}c^{-9},\\
w_{i_1}w_{i_2}(c^4w_{i_1}c^{-4})(c^4w_{i_2}c^{-4})\le_T c,
\quad\text{and}\quad \ell_T(w_{i_1})=\ell_T(w_{i_2})=1.
\end{multline}
With the help of {\tt CHEVIE}, one obtains 21 solutions for $w_i$ in 
\eqref{eq:G33''''D}:
{\small
\begin{multline*}
w_i\in\big\{[ 4, 5 ],\,
[ 1, 20 ],\,
[ 5, 7 ],\,
[ 1, 2 ],\,
[ 2, 4 ],\,
[ 10, 259 ],\,
[ 27, 208 ],\,
[ 27, 44 ],\,
[ 4, 39 ],\,
[ 5, 42 ],\,
[ 1, 39 ],\\
[ 5, 52 ],\,
[ 10, 208 ],\,
[ 1, 44 ],\,
[ 42, 113 ],\,
[ 4, 113 ],\,
[ 2, 27 ],\,
[ 10, 42 ],\,
[ 2, 49 ],\,
[ 52, 61 ],\,
[ 44, 123 ]
\big\},
\end{multline*}}%
where we have again used the short notation of {\tt CHEVIE}
referring to the internal
ordering of the roots of $G_{33}$ in {\tt CHEVIE},
one obtains 18 solutions for $w_i$ in 
\eqref{eq:G33''''D2}:
{\small
\begin{multline*}
w_i\in\big\{[ 5 ],\,
[ 4 ],\,
[ 20 ],\,
[ 208 ],\,
[ 259 ],\,
[ 1 ],\,
[ 2 ],\,
[ 7 ],\,
[ 10 ],\,
[ 27 ],\\
[ 39 ],\,
[ 42 ],\,
[ 49 ],\,
[ 44 ],\,
[ 52 ],\,
[ 113 ],\,
[ 61 ],\,
[ 123 ]
\big\},
\end{multline*}}%
each of them giving rise to $m/2$ elements of
$\Fix_{NC^m(G_{33})}(\phi^{p})$ since $i$ ranges from $1$ to $m/2$,
and one obtains 54 pairs $(w_{i_1},w_{i_2})$ satisfying 
\eqref{eq:G33''''E}:
{\tiny
\begin{multline*}
(w_{i_1},w_{i_2})\in\big\{( [ 5 ],\, [ 20 ] ),\
( [ 5 ],\, [ 7 ] ),\
( [ 5 ],\, [ 42 ] ),\
( [ 5 ],\, [ 52 ] ),\
( [ 4 ],\, [ 5 ] ),\
( [ 4 ],\, [ 10 ] ),\
( [ 4 ],\, [ 39 ] ),\
( [ 4 ],\, [ 113 ] ),\
( [ 20 ],\, [ 4 ] ),\\
( [ 20 ],\, [ 1 ] ),\
( [ 208 ],\, [ 27 ] ),\
( [ 208 ],\, [ 52 ] ),\
( [ 259 ],\, [ 10 ] ),\
( [ 259 ],\, [ 27 ] ),\
( [ 1 ],\, [ 20 ] ),\
( [ 1 ],\, [ 2 ] ),\
( [ 1 ],\, [ 39 ] ),\
( [ 1 ],\, [ 44 ] ),\
( [ 2 ],\, [ 4 ] ),\\
( [ 2 ],\, [ 7 ] ),\
( [ 2 ],\, [ 27 ] ),\
( [ 2 ],\, [ 49 ] ),\
( [ 7 ],\, [ 5 ] ),\
( [ 7 ],\, [ 1 ] ),\
( [ 10 ],\, [ 208 ] ),\
( [ 10 ],\, [ 259 ] ),\
( [ 10 ],\, [ 2 ] ),\
( [ 10 ],\, [ 42 ] ),\
( [ 27 ],\, [ 5 ] ),\\
( [ 27 ],\, [ 208 ] ),\
( [ 27 ],\, [ 44 ] ),\
( [ 27 ],\, [ 61 ] ),\
( [ 39 ],\, [ 4 ] ),\
( [ 39 ],\, [ 42 ] ),\
( [ 42 ],\, [ 1 ] ),\
( [ 42 ],\, [ 27 ] ),\
( [ 42 ],\, [ 113 ] ),\
( [ 42 ],\, [ 123 ] ),\\
( [ 49 ],\, [ 5 ] ),\
( [ 49 ],\, [ 2 ] ),\
( [ 44 ],\, [ 4 ] ),\
( [ 44 ],\, [ 259 ] ),\
( [ 44 ],\, [ 52 ] ),\
( [ 44 ],\, [ 123 ] ),\
( [ 52 ],\, [ 1 ] ),\
( [ 52 ],\, [ 10 ] ),\
( [ 52 ],\, [ 49 ] ),\\
( [ 52 ],\, [ 61 ] ),\
( [ 113 ],\, [ 42 ] ),\
( [ 113 ],\, [ 44 ] ),\
( [ 61 ],\, [ 2 ] ),\
( [ 61 ],\, [ 52 ] ),\
( [ 123 ],\, [ 10 ] ),\
( [ 123 ],\, [ 44 ] )
\big\},
\end{multline*}}%
each of them giving rise to $\binom {m/2}2$ elements of
$\Fix_{NC^m(G_{33})}(\phi^{p})$ since $1\le i_1<i_2\le m$.

In total, we obtain 
$1+(21+18)\frac m2+54\binom {m/2}2=\frac {(3m+2)(9m+2)}4$ elements in
$\Fix_{NC^m(G_{33})}(\phi^p)$, which agrees with the limit in
\eqref{eq:G33.6}.

\smallskip
Finally, we turn to \eqref{eq:G33.1}. By Remark~\ref{rem:1},
the only choices for $h_2$ and $m_2$ to be considered
are $h_2=1$ and $m_2=5$, $h_2=2$ and $m_2=5$, 
$h_2=2$ and $m_2=4$, 
respectively $h_2=m_2=2$. 
These correspond to the choices $p=18m/5$, $p=9m/5$, 
$p=9m/4$, respectively $p=9m/2$, out of which only $p=9m/4$ 
has not yet been discussed and belongs to the current case.
The corresponding action of $\phi^p$ is given by 
\begin{multline*}
\phi^p\big((w_0;w_1,\dots,w_m)\big)\\=
(*;
c^{3}w_{\frac {3m}4+1}c^{-3},c^{3}w_{\frac {3m}4+2}c^{-3},
\dots,c^{3}w_{m}c^{-3},
c^2w_{1}c^{-2},\dots,
c^2w_{\frac {3m}4}c^{-2}\big),
\end{multline*}
so that we have to solve
$$t_1(c^{2}t_1c^{-2})(c^{4}t_1c^{-4})(c^{6}t_1c^{-6})\le_T c$$
for $t_1$ with $\ell_T(t_1)$.
A computation with the help of {\tt CHEVIE} finds no solution. 
Hence, 
the left-hand side of \eqref{eq:1} is equal to $1$, as required.

\subsection*{\sc Case $G_{34}$}
The degrees are $6,12,18,24,30,42$, and hence we have
\begin{multline*}
\Cat^m(G_{34};q)=\frac 
{[42m+42]_q\, [42m+30]_q\, [42m+24]_q} 
{[42]_q\, [30]_q\, [24]_q}\\
\times
\frac {[42m+18]_q\,
 [42m+12]_q\, [42m+6]_q}
{[18]_q\, [12]_q\, [6]_q} .
\end{multline*}
Let $\zeta$ be a $42m$-th root of unity. 
The following cases on the right-hand side of \eqref{eq:1}
occur:
{\refstepcounter{equation}\label{eq:G34}}
\alphaeqn
\begin{align} 
\label{eq:G34.2}
\lim_{q\to\zeta}\Cat^m(G_{34};q)&=m+1,
\quad\text{if }\zeta=\zeta_{42},\zeta_{21},\zeta_{14},\zeta_7,\\
\label{eq:G34.3}
\lim_{q\to\zeta}\Cat^m(G_{34};q)&=\tfrac {7m+5}5,
\quad\text{if }\zeta=\zeta_{30},\zeta_{15},\zeta_{10},\zeta_5,\ 
5\mid m,\\
\label{eq:G34.4}
\lim_{q\to\zeta}\Cat^m(G_{34};q)&=\tfrac {7m+4}4,
\quad\text{if }\zeta=\zeta_{24},\zeta_8,\ 4\mid m,\\
\label{eq:G34.5}
\lim_{q\to\zeta}\Cat^m(G_{34};q)&=\tfrac {7m+3}3,
\quad\text{if }\zeta=\zeta_{18},\zeta_9,\ 3\mid m,\\
\label{eq:G34.6}
\lim_{q\to\zeta}\Cat^m(G_{34};q)&=\tfrac {(7m+4)(7m+2)}{8},
\quad\text{if }\zeta=\zeta_{12},\zeta_4,\ 2\mid m,\\
\label{eq:G34.7}
\lim_{q\to\zeta}\Cat^m(G_{34};q)&=\Cat^m(G_{34}),
\quad\text{if }\zeta=\zeta_6,\zeta_3,-1,1,\\
\label{eq:G34.1}
\lim_{q\to\zeta}\Cat^m(G_{34};q)&=1,
\quad\text{otherwise.}
\end{align}
\reseteqn

We must now prove that the left-hand side of \eqref{eq:1} in
each case agrees with the values exhibited in 
\eqref{eq:G34}. The only cases not covered by
Lemmas~\ref{lem:2} and \ref{lem:3} are the ones in 
\eqref{eq:G34.3}, \eqref{eq:G34.4}, 
\eqref{eq:G34.5}, \eqref{eq:G34.6}, 
and \eqref{eq:G34.1}.

\smallskip
We begin with the case in \eqref{eq:G34.3}.
By Lemma~\ref{lem:1}, 
we are free to choose $p=7m/5$ if $\zeta=\zeta_{30}$, 
we are free to choose $p=14m/5$ if $\zeta=\zeta_{15}$, 
we are free to choose $p=21m/5$ if $\zeta=\zeta_{10}$, 
and we are free to choose $p=42m/5$ if $\zeta=\zeta_{5}$. 
In particular, in all cases,
$m$ must be divisible by $5$.

We start with the case that $p=7m/5$.
From \eqref{eq:Aktion}, we infer
\begin{multline*}
\phi^p\big((w_0;w_1,\dots,w_m)\big)\\=
(*;
c^{2}w_{\frac {3m}5+1}c^{-2},c^{2}w_{\frac {3m}5+2}c^{-2},
\dots,c^{2}w_{m}c^{-2},
cw_{1}c^{-1},\dots,
cw_{\frac {3m}5}c^{-1}\big).
\end{multline*}
Supposing that 
$(w_0;w_1,\dots,w_m)$ is fixed by $\phi^p$, we obtain
the system of equations
{\refstepcounter{equation}\label{eq:G34A}}
\alphaeqn
\begin{align} \label{eq:G34Aa}
w_i&=c^2w_{\frac {3m}5+i}c^{-2}, \quad i=1,2,\dots,\tfrac {2m}5,\\
w_i&=cw_{i-\frac {2m}5}c^{-1}, \quad i=\tfrac {2m}5+1,\tfrac {2m}5+2,\dots,m.
\label{eq:G34Ab}
\end{align}
\reseteqn
There are two distinct possibilities for choosing
the $w_i$'s, $1\le i\le m$: 
{\refstepcounter{equation}\label{eq:G34C}}
\alphaeqn
\begin{enumerate}
\item[(i)]
all the $w_i$'s are equal to $\ep$ (and $w_0=c$), 
\item[(ii)]
there is an $i$ with $1\le i\le \frac m5$ such that
\begin{equation} \label{eq:G34Cii}
\ell_T(w_i)=\ell_T(w_{i+\frac m5})=\ell_T(w_{i+\frac {2m}5})
=\ell_T(w_{i+\frac {3m}5})=\ell_T(w_{i+\frac {4m}5})=1,
\end{equation}
and the other $w_j$'s, $1\le j\le m$, are equal to $\ep$.
\end{enumerate}
\reseteqn

Moreover, since $(w_0;w_1,\dots,w_m)\in NC^m(G_{34})$, we must have 
$$w_iw_{i+\frac m5}w_{i+\frac {2m}5}w_{i+\frac {3m}5}
w_{i+\frac {4m}5}\le_T c.$$
Together with equations~\eqref{eq:G34A}--\eqref{eq:G34C}, 
this implies that
\begin{equation} \label{eq:G34D}
w_i=c^7w_ic^{-7}\quad\text{and}\quad 
w_i(c^4w_ic^{-4})(cw_ic^{-1})(c^5w_ic^{-5})(c^2w_ic^{-2})\le_T c.
\end{equation}
With the help of {\tt CHEVIE}, one obtains seven solutions for $w_i$ in 
\eqref{eq:G34D}:
{\small
$$w_i\in\big\{[ 4 ],\,
[ 5 ],\,
[ 6 ],\,
[ 2 ],\,
[ 11 ],\,
[ 44 ],\,
[ 63 ],\,
[ 74 ]\big\},$$}%
where we have again used the short notation of {\tt CHEVIE}
referring to the internal
ordering of the roots of $G_{34}$ in {\tt CHEVIE}.
Each of the above solutions for $w_i$ gives rise to $m/5$ elements of
$\Fix_{NC^m(G_{34})}(\phi^{p})$ since $i$ ranges from $1$ to $m/5$.

In total, we obtain 
$1+7\frac m5=\frac {7m+5}5$ elements in
$\Fix_{NC^m(G_{34})}(\phi^p)$, which agrees with the limit in
\eqref{eq:G34.3}.

In the case that $p=14m/5$, we infer from \eqref{eq:Aktion} that
\begin{multline*}
\phi^p\big((w_0;w_1,\dots,w_m)\big)\\=
(*;
c^{3}w_{\frac {m}5+1}c^{-3},c^{3}w_{\frac {m}5+2}c^{-3},
\dots,c^{3}w_{m}c^{-3},
c^2w_{1}c^{-2},\dots,
c^2w_{\frac {m}5}c^{-2}\big).
\end{multline*}
Supposing that 
$(w_0;w_1,\dots,w_m)$ is fixed by $\phi^p$, we obtain
the system of equations
{\refstepcounter{equation}\label{eq:G34'''A}}
\alphaeqn
\begin{align} \label{eq:G34'''Aa}
w_i&=c^3w_{\frac {m}5+i}c^{-3}, \quad i=1,2,\dots,\tfrac {4m}5,\\
w_i&=c^2w_{i-\frac {4m}5}c^{-2}, \quad i=\tfrac {4m}5+1,\tfrac {4m}5+2,\dots,m.
\label{eq:G34'''Ab}
\end{align}
\reseteqn
There are two distinct possibilities for choosing
the $w_i$'s, $1\le i\le m$: 
either all the $w_i$'s are equal to $\ep$, or
there is an $i$ with $1\le i\le \frac m5$ such that
$$\ell_T(w_i)=\ell_T(w_{i+\frac m5})=
\ell_T(w_{i+\frac {2m}5})=\ell_T(w_{i+\frac {3m}5})=
\ell_T(w_{i+\frac {4m}5})=1.$$
Writing $t_1,t_2,t_3,t_4,t_5$ for 
$w_i,w_{i+\frac m5},w_{i+\frac {2m}5},w_{i+\frac {3m}5},
w_{i+\frac {4m}5}$, 
respectively, the equations\break \eqref{eq:G34'''A} 
reduce to
{\refstepcounter{equation}\label{eq:G34'''B}}
\alphaeqn
\begin{align} \label{eq:G34'''Ba}
t_1&=c^3t_2c^{-3},\\
\label{eq:G34'''Bb}
t_2&=c^3t_3c^{-3},\\
\label{eq:G34'''Bc}
t_3&=c^3t_4c^{-3},\\
\label{eq:G34'''Bd}
t_4&=c^3t_5c^{-3},\\
t_5&=c^2t_1c^{-2}.
\label{eq:G34'''Be}
\end{align}
\reseteqn
One of these equations is in fact superfluous: if we substitute
\eqref{eq:G34'''Bb}--\eqref{eq:G34'''Be} in \eqref{eq:G34'''Ba}, then
we obtain $t_1=c^{14}t_1c^{-14}$ which is automatically satisfied 
due to Lemma~\ref{lem:4} with $d=3$.

Since $(w_0;w_1,\dots,w_m)\in NC^m(G_{34})$, we must have
$t_1t_2t_3t_4t_5\le_T c$.
Combining this with \eqref{eq:G34'''B}, we infer that
\begin{equation} \label{eq:G34'''D}
t_1(c^{11}t_1c^{-11})(c^{8}t_1c^{-8})(c^{5}t_1c^{-5})
(c^{2}t_1c^{-2})\le_T c.
\end{equation}
Using that $c^7t_1c^{-7}=t_1$, due to Lemma~\ref{lem:4}
with $d=6$, we see that
this equation is equivalent with \eqref{eq:G34D}. 
Therefore, we are facing exactly the same enumeration 
problem here as for
$p=7m/5$, and, consequently, the number of solutions to \eqref{eq:G34'''D} is the same, namely
$\frac {7m+5}5$, as required.

In the case that $p=21m/5$, we infer from \eqref{eq:Aktion} that
\begin{multline*}
\phi^p\big((w_0;w_1,\dots,w_m)\big)\\=
(*;
c^{5}w_{\frac {4m}5+1}c^{-5},c^{5}w_{\frac {4m}5+2}c^{-5},
\dots,c^{5}w_{m}c^{-5},
c^4w_{1}c^{-4},\dots,
c^4w_{\frac {4m}5}c^{-4}\big).
\end{multline*}
Supposing that 
$(w_0;w_1,\dots,w_m)$ is fixed by $\phi^p$, we obtain
the system of equations
{\refstepcounter{equation}\label{eq:G34'''''A}}
\alphaeqn
\begin{align} \label{eq:G34'''''Aa}
w_i&=c^5w_{\frac {4m}5+i}c^{-5}, \quad i=1,2,\dots,\tfrac {m}5,\\
w_i&=c^4w_{i-\frac {m}5}c^{-4}, \quad i=\tfrac {m}5+1,\tfrac {m}5+2,\dots,m.
\label{eq:G34'''''Ab}
\end{align}
\reseteqn
There are two distinct possibilities for choosing
the $w_i$'s, $1\le i\le m$: 
either all the $w_i$'s are equal to $\ep$, or
there is an $i$ with $1\le i\le \frac m5$ such that
$$\ell_T(w_i)=\ell_T(w_{i+\frac m5})=
\ell_T(w_{i+\frac {2m}5})=\ell_T(w_{i+\frac {3m}5})=
\ell_T(w_{i+\frac {4m}5})=1.$$
Writing $t_1,t_2,t_3,t_4,t_5$ for 
$w_i,w_{i+\frac m5},w_{i+\frac {2m}5},w_{i+\frac {3m}5},
w_{i+\frac {4m}5}$, 
respectively, the equations\break \eqref{eq:G34'''''A} 
reduce to
{\refstepcounter{equation}\label{eq:G34'''''B}}
\alphaeqn
\begin{align} \label{eq:G34'''''Ba}
t_1&=c^5t_5c^{-5},\\
\label{eq:G34'''''Bb}
t_2&=c^4t_1c^{-4},\\
\label{eq:G34'''''Bc}
t_3&=c^4t_2c^{-4},\\
\label{eq:G34'''''Bd}
t_4&=c^4t_3c^{-4},\\
t_5&=c^4t_4c^{-4}.
\label{eq:G34'''''Be}
\end{align}
\reseteqn
One of these equations is in fact superfluous: if we substitute
\eqref{eq:G34'''''Bb}--\eqref{eq:G34'''''Be} in \eqref{eq:G34'''''Ba}, then
we obtain $t_1=c^{21}t_1c^{-21}$ which is automatically satisfied 
due to Lemma~\ref{lem:4} with $d=6$.

Since $(w_0;w_1,\dots,w_m)\in NC^m(G_{34})$, we must have
$t_1t_2t_3t_4t_5\le_T c$.
Combining this with \eqref{eq:G34'''''B}, we infer that
\begin{equation} \label{eq:G34'''''D}
t_1(c^{4}t_1c^{-4})(c^{8}t_1c^{-8})(c^{12}t_1c^{-12})
(c^{16}t_1c^{-16})\le_T c.
\end{equation}
Using that $c^7t_1c^{-7}=t_1$, due to Lemma~\ref{lem:4}
with $d=6$, we see that
this equation is equivalent with \eqref{eq:G34D}. 
Therefore, we are facing exactly the same enumeration 
problem here as for
$p=7m/5$, and, consequently, the number of solutions to \eqref{eq:G34'''''D} is the same, namely
$\frac {7m+5}5$, as required.

In the case that $p=42m/5$, we infer from \eqref{eq:Aktion} that
\begin{multline*}
\phi^p\big((w_0;w_1,\dots,w_m)\big)\\=
(*;
c^{9}w_{\frac {3m}5+1}c^{-9},c^{9}w_{\frac {3m}5+2}c^{-9},
\dots,c^{9}w_{m}c^{-9},
c^8w_{1}c^{-8},\dots,
c^8w_{\frac {3m}5}c^{-8}\big).
\end{multline*}
Supposing that 
$(w_0;w_1,\dots,w_m)$ is fixed by $\phi^p$, we obtain
the system of equations
{\refstepcounter{equation}\label{eq:G34''''''A}}
\alphaeqn
\begin{align} \label{eq:G34''''''Aa}
w_i&=c^9w_{\frac {3m}5+i}c^{-9}, \quad i=1,2,\dots,\tfrac {2m}5,\\
w_i&=c^8w_{i-\frac {2m}5}c^{-8}, \quad i=\tfrac {2m}5+1,\tfrac {2m}5+2,\dots,m.
\label{eq:G34''''''Ab}
\end{align}
\reseteqn
There are two distinct possibilities for choosing
the $w_i$'s, $1\le i\le m$: 
either all the $w_i$'s are equal to $\ep$, or
there is an $i$ with $1\le i\le \frac m5$ such that
$$\ell_T(w_i)=\ell_T(w_{i+\frac m5})=
\ell_T(w_{i+\frac {2m}5})=\ell_T(w_{i+\frac {3m}5})=
\ell_T(w_{i+\frac {4m}5})=1.$$
Writing $t_1,t_2,t_3,t_4,t_5$ for 
$w_i,w_{i+\frac m5},w_{i+\frac {2m}5},w_{i+\frac {3m}5},
w_{i+\frac {4m}5}$, 
respectively, the equations\break \eqref{eq:G34''''''A} 
reduce to
{\refstepcounter{equation}\label{eq:G34''''''B}}
\alphaeqn
\begin{align} \label{eq:G34''''''Ba}
t_1&=c^9t_4c^{-9},\\
\label{eq:G34''''''Bb}
t_2&=c^9t_5c^{-9},\\
\label{eq:G34''''''Bc}
t_3&=c^8t_1c^{-8},\\
\label{eq:G34''''''Bd}
t_4&=c^8t_2c^{-8},\\
t_5&=c^8t_3c^{-8}.
\label{eq:G34''''''Be}
\end{align}
\reseteqn
One of these equations is in fact superfluous: if we substitute
\eqref{eq:G34''''''Bb}--\eqref{eq:G34''''''Be} in \eqref{eq:G34''''''Ba}, then
we obtain $t_1=c^{42}t_1c^{-42}$ which is automatically satisfied 
since $c^{42}=\ep$.

Since $(w_0;w_1,\dots,w_m)\in NC^m(G_{34})$, we must have
$t_1t_2t_3t_4t_5\le_T c$.
Combining this with \eqref{eq:G34''''''B}, we infer that
\begin{equation} \label{eq:G34''''''D}
t_1(c^{25}t_1c^{-25})(c^{8}t_1c^{-8})(c^{33}t_1c^{-33})
(c^{16}t_1c^{-16})\le_T c.
\end{equation}
Using that $c^7t_1c^{-7}=t_1$, due to Lemma~\ref{lem:4}
with $d=6$, we see that
this equation is equivalent with \eqref{eq:G34D}. 
Therefore, we are facing exactly the same enumeration 
problem here as for
$p=7m/5$, and, consequently, the number of solutions to \eqref{eq:G34''''''D} is the same, namely
$\frac {7m+5}5$, as required.

\smallskip
Next we consider the case in \eqref{eq:G34.4}.
By Lemma~\ref{lem:1}, 
we are free to choose $p=7m/4$ if $\zeta=\zeta_{24}$,
and we are free to choose $p=21m/4$ if $\zeta=\zeta_{8}$. 
In both cases,
$m$ must be divisible by $4$.

We start with the case that $p=7m/4$.
From \eqref{eq:Aktion}, we infer
\begin{equation*} 
\phi^p\big((w_0;w_1,\dots,w_m)\big)=
(*;
c^{2}w_{\frac m4+1}c^{-2},c^{2}w_{\frac m4+2}c^{-2},
\dots,c^{2}w_{m}c^{-2},
cw_{1}c^{-1},\dots,
cw_{\frac m4}c^{-1}\big).
\end{equation*}
Supposing that 
$(w_0;w_1,\dots,w_m)$ is fixed by $\phi^p$, we obtain
the system of equations
{\refstepcounter{equation}\label{eq:G34'A}}
\alphaeqn
\begin{align} \label{eq:G34'Aa}
w_i&=c^2w_{\frac m4+i}c^{-2}, \quad i=1,2,\dots,\tfrac {3m}4,\\
w_i&=cw_{i-\frac {3m}4}c^{-1}, \quad i=\tfrac {3m}4+1,\tfrac {3m}4+2,\dots,m.
\label{eq:G34'Ab}
\end{align}
\reseteqn
There are two distinct possibilities for choosing
the $w_i$'s, $1\le i\le m$: 
{\refstepcounter{equation}\label{eq:G34'C}}
\alphaeqn
\begin{enumerate}
\item[(i)]
all the $w_i$'s are equal to $\ep$ (and $w_0=c$), 
\item[(ii)]
there is an $i$ with $1\le i\le \frac m2$ such that
\begin{equation} \label{eq:G34'Cii}
\ell_T(w_i)=\ell_T(w_{i+\frac {m}4})
=\ell_T(w_{i+\frac {2m}4})=\ell_T(w_{i+\frac {3m}4})=1,
\end{equation}
and the other $w_j$'s, $1\le j\le m$, are equal to $\ep$.
\end{enumerate}
\reseteqn

Moreover, since $(w_0;w_1,\dots,w_m)\in NC^m(G_{34})$, we must have 
$$w_iw_{i+\frac {m}4}w_{i+\frac {2m}4}w_{i+\frac {3m}4}\le_T c.$$
Together with equations~\eqref{eq:G34'A}--\eqref{eq:G34'C}, 
this implies that
\begin{equation} \label{eq:G34'D}
w_i=c^7w_ic^{-7}\quad\text{and}\quad 
w_i(c^5w_ic^{-5})(c^3w_ic^{-3})(cw_ic^{-1})\le_T c.
\end{equation}
With the help of {\tt CHEVIE}, one obtains seven solutions for $w_i$ in 
\eqref{eq:G34'D}:
{\small
\begin{equation*}
w_i\in\big\{[ 1 ],\,
[ 2 ],\,
[ 28 ],\,
[ 34 ],\,
[ 61 ],\,
[ 46 ],\,
[ 168 ]
\big\},
\end{equation*}}%
where we have again used the short notation of {\tt CHEVIE}
referring to the internal
ordering of the roots of $G_{34}$ in {\tt CHEVIE}.
Each of them gives rise to $m/4$ elements of
$\Fix_{NC^m(G_{34})}(\phi^{p})$ since $i$ ranges from $1$ to $m/4$.

In total, we obtain 
$1+7\frac m4=\frac {7m+4}4$ elements in
$\Fix_{NC^m(G_{34})}(\phi^p)$, which agrees with the limit in
\eqref{eq:G34.4}.

\smallskip
If $p=21m/4$, then, from \eqref{eq:Aktion}, we infer
\begin{multline*} 
\phi^p\big((w_0;w_1,\dots,w_m)\big)\\=
(*;
c^{6}w_{\frac {3m}4+1}c^{-6},c^{6}w_{\frac {3m}4+2}c^{-6},
\dots,c^{6}w_{m}c^{-6},
c^5w_{1}c^{-5},\dots,
c^5w_{\frac {3m}4}c^{-5}\big).
\end{multline*}
Supposing that 
$(w_0;w_1,\dots,w_m)$ is fixed by $\phi^p$, we obtain
the system of equations
{\refstepcounter{equation}\label{eq:G34'''''''A}}
\alphaeqn
\begin{align} \label{eq:G34'''''''Aa}
w_i&=c^6w_{\frac {3m}4+i}c^{-6}, \quad i=1,2,\dots,\tfrac {m}4,\\
w_i&=c^5w_{i-\frac {3m}4}c^{-5}, \quad i=\tfrac {m}4+1,\tfrac {m}4+2,\dots,m.
\label{eq:G34'''''''Ab}
\end{align}
\reseteqn
There are two distinct possibilities for choosing
the $w_i$'s, $1\le i\le m$: 
{\refstepcounter{equation}\label{eq:G34'''''''C}}
\alphaeqn
\begin{enumerate}
\item[(i)]
all the $w_i$'s are equal to $\ep$ (and $w_0=c$), 
\item[(ii)]
there is an $i$ with $1\le i\le \frac m2$ such that
\begin{equation} \label{eq:G34'''''''Cii}
\ell_T(w_i)=\ell_T(w_{i+\frac {m}4})
=\ell_T(w_{i+\frac {2m}4})=\ell_T(w_{i+\frac {3m}4})=1,
\end{equation}
and the other $w_j$'s, $1\le j\le m$, are equal to $\ep$.
\end{enumerate}
\reseteqn

Moreover, since $(w_0;w_1,\dots,w_m)\in NC^m(G_{34})$, we must have 
$$w_iw_{i+\frac {m}4}w_{i+\frac {2m}4}w_{i+\frac {3m}4}\le_T c.$$
Together with equations~\eqref{eq:G34'''''''A}--\eqref{eq:G34'''''''C}, 
this implies that
\begin{equation} \label{eq:G34'''''''D}
w_i=c^{21}w_ic^{-21}\quad\text{and}\quad 
w_i(c^5w_ic^{-5})(c^{10}w_ic^{-10})(c^{15}w_ic^{-15})\le_T c.
\end{equation}
Using that $c^7t_1c^{-7}=t_1$, due to Lemma~\ref{lem:4}
with $d=6$, we see that
this equation is equivalent with \eqref{eq:G34'D}. 
Therefore, we are facing exactly the same enumeration 
problem here as for
$p=7m/4$, and, consequently, the number of solutions to \eqref{eq:G34''''''D} is the same, namely
$\frac {7m+4}4$, as required.

\smallskip
Our next case is the case in \eqref{eq:G34.5}.
By Lemma~\ref{lem:1}, 
we are free to choose $p=7m/3$ if $\zeta=\zeta_{18}$, 
and we are free to choose $p=14m/3$ if $\zeta=\zeta_{9}$.
In both cases,
$m$ must be divisible by $3$.

We start with the case that $p=7m/3$.
From \eqref{eq:Aktion}, we infer
\begin{multline*}
\phi^p\big((w_0;w_1,\dots,w_m)\big)\\=
(*;
c^{3}w_{\frac {2m}3+1}c^{-3},c^{3}w_{\frac {2m}3+2}c^{-3},
\dots,c^{3}w_{m}c^{-3},
c^2w_{1}c^{-2},\dots,
c^2w_{\frac {2m}3}c^{-2}\big).
\end{multline*}
Supposing that 
$(w_0;w_1,\dots,w_m)$ is fixed by $\phi^p$, we obtain
the system of equations
{\refstepcounter{equation}\label{eq:G34*'''A}}
\alphaeqn
\begin{align} \label{eq:G34*'''Aa}
w_i&=c^3w_{\frac {2m}3+i}c^{-3}, \quad i=1,2,\dots,\tfrac {m}3,\\
w_i&=c^2w_{i-\frac {m}3}c^{-2}, \quad i=\tfrac {m}3+1,\tfrac {m}3+2,\dots,m.
\label{eq:G34*'''Ab}
\end{align}
\reseteqn
There are four distinct possibilities for choosing
the $w_i$'s, $1\le i\le m$: 
{\refstepcounter{equation}\label{eq:G34*'''C}}
\alphaeqn
\begin{enumerate}
\item[(i)]
all the $w_i$'s are equal to $\ep$ (and $w_0=c$), 
\item[(ii)]
there is an $i$ with $1\le i\le \frac m2$ such that
\begin{equation} \label{eq:G34*'''Cii}
\ell_T(w_i)=\ell_T(w_{i+\frac m3})=\ell_T(w_{i+\frac {2m}3})=2,
\end{equation}
and all other $w_j$'s are equal to $\ep$,
\item[(iii)]
there is an $i$ with $1\le i\le \frac m2$ such that
\begin{equation} \label{eq:G34*'''Ciii}
\ell_T(w_i)=\ell_T(w_{i+\frac m3})=\ell_T(w_{i+\frac {2m}3})=1,
\end{equation}
and the other $w_j$'s, $1\le j\le m$, are equal to $\ep$,
\item[(iv)]
there are $i_1$ and $i_2$ with $1\le i_1<i_2\le \frac m2$ such that
\begin{equation} \label{eq:G34*'''Civ}
\ell_T(w_{i_1})=\ell_T(w_{i_2})=
\ell_T(w_{i_1+\frac m3})=\ell_T(w_{i_2+\frac m3})=
\ell_T(w_{i_1+\frac {2m}3})=\ell_T(w_{i_2+\frac {2m}3})=1,
\end{equation}
and all other $w_j$'s are equal to $\ep$.
\end{enumerate}
\reseteqn

Moreover, since $(w_0;w_1,\dots,w_m)\in NC^m(G_{34})$, we must have 
$w_iw_{i+\frac m3}w_{i+\frac {2m}3}\le_T c$, respectively 
$w_{i_1}w_{i_2}w_{i_1+\frac {m}3}w_{i_2+\frac {m}3}
w_{i_1+\frac {2m}3}w_{i_2+\frac {2m}3}=c$.
Together with equations~\eqref{eq:G34*'''A}--\eqref{eq:G34*'''C}, 
this implies that
\begin{equation} \label{eq:G34*'''D}
w_i=c^7w_ic^{-7}\quad\text{and}\quad 
w_i(c^2w_ic^{-2})(c^4w_ic^{-4})=c,
\end{equation}
respectively that
\begin{equation} \label{eq:G34*'''D2}
w_i=c^7w_ic^{-7},\quad 
w_i(c^2w_ic^{-2})(c^4w_ic^{-4})\le_T c,\quad\text{and}\quad
\ell_T(w_i)=1,
\end{equation}
respectively that
\begin{multline} \label{eq:G34*'''E}
w_{i_1}=c^7w_{i_1}c^{-7},\quad 
w_{i_2}=c^7w_{i_2}c^{-7},\\
\text{and}\quad 
w_{i_1}w_{i_2}(c^2w_{i_1}c^{-2})(c^2w_{i_2}c^{-2})
(c^4w_{i_1}c^{-4})(c^4w_{i_2}c^{-4})=c.
\end{multline}
With the help of {\tt CHEVIE}, one obtains 21 solutions for $w_i$ in 
\eqref{eq:G34*'''D2}:
{\small
\begin{equation*}
w_i\in\big\{[4]\,\
[5]\,\
[6]\,\
[11]\,\
[44]\,\
[63]\,\
[74]
\big\},
\end{equation*}}%
where we have again used the short notation of {\tt CHEVIE}
referring to the internal
ordering of the roots of $G_{34}$ in {\tt CHEVIE}.
Each of them gives rise to $m/3$ elements of
$\Fix_{NC^m(G_{34})}(\phi^{p})$ since $i$ ranges from $1$ to $m/3$.

There are no solutions $w_i$ in \eqref{eq:G34*'''D} and
for $(w_{i_1},w_{i_2})$ in \eqref{eq:G34*'''E}.

In total, we obtain 
$1+7\frac m3=\frac {7m+3}3$ elements in
$\Fix_{NC^m(G_{34})}(\phi^p)$, which agrees with the limit in
\eqref{eq:G34.5}.

In the case that $p=14m/3$, we infer from \eqref{eq:Aktion} that
\begin{multline*}
\phi^p\big((w_0;w_1,\dots,w_m)\big)\\=
(*;
c^{5}w_{\frac {m}3+1}c^{-5},c^{5}w_{\frac {m}3+2}c^{-5},
\dots,c^{5}w_{m}c^{-5},
c^4w_{1}c^{-4},\dots,
c^4w_{\frac {m}3}c^{-4}\big).
\end{multline*}
Supposing that 
$(w_0;w_1,\dots,w_m)$ is fixed by $\phi^p$, we obtain
the system of equations
{\refstepcounter{equation}\label{eq:G34**A}}
\alphaeqn
\begin{align} \label{eq:G34**Aa}
w_i&=c^5w_{\frac {m}3+i}c^{-5}, \quad i=1,2,\dots,\tfrac {2m}3,\\
w_i&=c^4w_{i-\frac {2m}3}c^{-4}, \quad i=\tfrac {2m}3+1,\tfrac {2m}3+2,\dots,m.
\label{eq:G34**Ab}
\end{align}
\reseteqn
There are four distinct possibilities for choosing
the $w_i$'s, $1\le i\le m$: 
{\refstepcounter{equation}\label{eq:G34**C}}
\alphaeqn
\begin{enumerate}
\item[(i)]
all the $w_i$'s are equal to $\ep$ (and $w_0=c$), 
\item[(ii)]
there is an $i$ with $1\le i\le \frac m2$ such that
\begin{equation} \label{eq:G34**Cii}
\ell_T(w_i)=\ell_T(w_{i+\frac m3})=\ell_T(w_{i+\frac {2m}3})=2,
\end{equation}
and all other $w_j$'s are equal to $\ep$,
\item[(iii)]
there is an $i$ with $1\le i\le \frac m2$ such that
\begin{equation} \label{eq:G34**Ciii}
\ell_T(w_i)=\ell_T(w_{i+\frac m3})=\ell_T(w_{i+\frac {2m}3})=1,
\end{equation}
and the other $w_j$'s, $1\le j\le m$, are equal to $\ep$,
\item[(iv)]
there are $i_1$ and $i_2$ with $1\le i_1<i_2\le \frac m2$ such that
\begin{equation} \label{eq:G34**Civ}
\ell_T(w_{i_1})=\ell_T(w_{i_2})=
\ell_T(w_{i_1+\frac m3})=\ell_T(w_{i_2+\frac m3})=
\ell_T(w_{i_1+\frac {2m}3})=\ell_T(w_{i_2+\frac {2m}3})=1,
\end{equation}
and all other $w_j$'s are equal to $\ep$.
\end{enumerate}
\reseteqn

Moreover, since $(w_0;w_1,\dots,w_m)\in NC^m(G_{34})$, we must have 
$w_iw_{i+\frac m3}w_{i+\frac {2m}3}\le_T c$, respectively 
$w_{i_1}w_{i_2}w_{i_1+\frac {m}3}w_{i_2+\frac {m}3}
w_{i_1+\frac {2m}3}w_{i_2+\frac {2m}3}=c$.
Together with equations~\eqref{eq:G34**A}--\eqref{eq:G34**C}, 
this implies that
\begin{equation} \label{eq:G34**D}
w_i=c^{14}w_ic^{-14}\quad\text{and}\quad 
w_i(c^9w_ic^{-9})(c^4w_ic^{-4})=c,
\end{equation}
respectively that
\begin{equation} \label{eq:G34**D2}
w_i=c^{14}w_ic^{-14},\quad 
w_i(c^9w_ic^{-9})(c^4w_ic^{-4})\le_T c,\quad\text{and}\quad
\ell_T(w_i)=1,
\end{equation}
respectively that
\begin{multline} \label{eq:G34**E}
w_{i_1}=c^{14}w_{i_1}c^{-14},\quad 
w_{i_2}=c^{14}w_{i_2}c^{-14},\\
\text{and}\quad 
w_{i_1}w_{i_2}(c^9w_{i_1}c^{-9})(c^9w_{i_2}c^{-9})
(c^4w_{i_1}c^{-4})(c^4w_{i_2}c^{-4})=c.
\end{multline}
Using that $c^7wc^{-7}=w$ for all $w\in NC(G_{34}$, due to Lemma~\ref{lem:4}
with $d=6$, we see that
this equation is equivalent with \eqref{eq:G34*'''D}. 
Therefore, we are facing exactly the same enumeration 
problem here as for
$p=7m/3$, and, consequently, the number of solutions to \eqref{eq:G34**D} is the same, namely
$\frac {7m+3}3$, as required.

\smallskip
Next we consider the case in \eqref{eq:G34.6}.
By Lemma~\ref{lem:1}, 
we are free to choose $p=7m/2$ if $\zeta=\zeta_{12}$,
and we are free to choose $p=21m/2$ if $\zeta=\zeta_{4}$.
In both cases, 
$m$ must be divisible by $2$.

We begin with the case that $p=7m/2$.
From \eqref{eq:Aktion}, we infer
\begin{multline} \label{eq:7m2Aktion}
\phi^p\big((w_0;w_1,\dots,w_m)\big)\\=
(*;
c^{4}w_{\frac m2+1}c^{-4},c^{4}w_{\frac m2+2}c^{-4},
\dots,c^{4}w_{m}c^{-4},
c^3w_{1}c^{-3},\dots,
c^3w_{\frac m2}c^{-3}\big).
\end{multline}
Supposing that 
$(w_0;w_1,\dots,w_m)$ is fixed by $\phi^p$, we obtain
the system of equations
{\refstepcounter{equation}\label{eq:G34''''A}}
\alphaeqn
\begin{align} \label{eq:G34''''Aa}
w_i&=c^4w_{\frac m2+i}c^{-4}, \quad i=1,2,\dots,\tfrac {m}2,\\
w_i&=c^3w_{i-\frac {m}2}c^{-3}, \quad i=\tfrac {m}2+1,\tfrac {m}2+2,\dots,m.
\label{eq:G34''''Ab}
\end{align}
\reseteqn
There are several distinct possibilities for choosing
the $w_i$'s, $1\le i\le m$, which we summarise as follows: 
{\refstepcounter{equation}\label{eq:G34''''C}}
\alphaeqn
\begin{enumerate}
\item[(i)]
all the $w_i$'s are equal to $\ep$ (and $w_0=c$), 
\item[(ii)]
there is an $i$ with $1\le i\le \frac m2$ such that
\begin{equation} \label{eq:G34''''Cii}
1\le \ell_T(w_i)=\ell_T(w_{i+\frac m2})\le 3,
\end{equation}
and the other $w_j$'s, $1\le j\le m$, are equal to $\ep$,
\item[(iii)]
there are $i_1$ and $i_2$ with $1\le i_1<i_2\le \frac m2$ such that
\begin{equation} \label{eq:G34''''Ciii}
\ell_1=\ell_T(w_{i_1})=\ell_T(w_{i_1+\frac m2})\ge1,\quad
\ell_2=\ell_T(w_{i_2})=
=\ell_T(w_{i_2+\frac m2})\ge1,\quad
\ell_1+\ell_2\le3,
\end{equation}
and the other $w_j$'s, $1\le j\le m$, are equal to $\ep$,
\item[(iv)]
there are $i_1,i_2,i_3$ with $1\le i_1<i_2<i_3\le \frac m2$ such that
\begin{equation} \label{eq:G34''''Civ}
\ell_T(w_{i_1})=\ell_T(w_{i_2})=\ell_T(w_{i_3})=
\ell_T(w_{i_1+\frac m2})=\ell_T(w_{i_2+\frac m2})=
\ell_T(w_{i_3+\frac m2})=1,
\end{equation}
and all other $w_j$'s are equal to $\ep$.
\end{enumerate}
\reseteqn

Moreover, since $(w_0;w_1,\dots,w_m)\in NC^m(G_{34})$, we must have 
$w_iw_{i+\frac m2}\le_T c$, respectively 
$w_{i_1}w_{i_2}w_{i_1+\frac m2}w_{i_2+\frac m2}\le_T c$,
respectively
$$
w_{i_1}w_{i_2}w_{i_3}w_{i_1+\frac m2}w_{i_2+\frac m2}
w_{i_3+\frac m2}=c
$$.
Together with equations~\eqref{eq:G34''''A}--\eqref{eq:G34''''C}, 
this implies that
\begin{equation} \label{eq:G34''''D}
w_i=c^7w_ic^{-7}\quad\text{and}\quad 
w_i(c^4w_ic^{-4})\le_T c,
\end{equation}
respectively that
\begin{equation} \label{eq:G34''''E}
w_{i_1}=c^7w_{i_1}c^{-7},\quad 
w_{i_2}=c^7w_{i_2}c^{-7},\quad \text{and}\quad
w_{i_1}w_{i_2}(c^4w_{i_1}c^{-4})(c^4w_{i_2}c^{-4})\le_T c,
\end{equation}
respectively that
\begin{multline} \label{eq:G34''''F}
w_{i_1}=c^7w_{i_1}c^{-7},\quad 
w_{i_2}=c^7w_{i_2}c^{-7},\quad 
w_{i_3}=c^7w_{i_3}c^{-7},\\ \text{and}\quad
w_{i_1}w_{i_2}w_{i_3}
(c^4w_{i_1}c^{-4})(c^4w_{i_2}c^{-4})(c^4w_{i_3}c^{-4})=c.
\end{multline}
With the help of {\tt CHEVIE}, one obtains 14 solutions for $w_i$ in 
\eqref{eq:G34''''D} with $\ell_T(w_i)=1$:
{\small
\begin{equation*}
w_i\in\big\{[ 14 ],\,
[ 35 ],\,
[ 1 ],\,
[ 2 ],\,
[ 10 ],\,
[ 24 ],\,
[ 28 ],\,
[ 34 ],\,
[ 39 ],\,
[ 56 ],\,
[ 61 ],\,
[ 46 ],\,
[ 168 ],\,
[ 105 ]
\big\},
\end{equation*}}%
where we have again used the short notation of {\tt CHEVIE}
referring to the internal
ordering of the roots of $G_{34}$ in {\tt CHEVIE},
one obtains 21 solutions for $w_i$ in 
\eqref{eq:G34''''D} with $\ell_T(w_i)=2$:
{\small
\begin{multline*}
w_i\in\big\{[ 1, 14 ],\,
[ 1, 35 ],\,
[ 1, 24 ],\,
[ 1, 34 ],\,
[ 14, 61 ],\,
[ 2, 39 ],\,
[ 10, 34 ],\,
[ 2, 56 ],\,
[ 2, 35 ],\,
[ 14, 46 ],\,
[ 35, 168 ],\\
[ 34, 56 ],\,
[ 2, 61 ],\,
[ 28, 56 ],\,
[ 10, 28 ],\,
[ 10, 61 ],\,
[ 34, 105 ],\,
[ 28, 105 ],\,
[ 24, 61 ],\,
[ 39, 168 ],\,
[ 24, 46 ]
\big\},
\end{multline*}}%
each of them giving rise to $m/2$ elements of
$\Fix_{NC^m(G_{34})}(\phi^{p})$ since $i$ ranges from $1$ to $m/2$,
and one obtains 49 pairs $(w_{i_1},w_{i_2})$ satisfying 
\eqref{eq:G34''''E}:
{\tiny
\begin{multline*}
(w_{i_1},w_{i_2})\in\big\{( [ 14 ],\, [ 1 ] ),\
( [ 14 ],\, [ 61 ] ),\
( [ 14 ],\, [ 46 ] ),\
( [ 35 ],\, [ 1 ] ),\
( [ 35 ],\, [ 46 ] ),\
( [ 35 ],\, [ 168 ] ),\
( [ 1 ],\, [ 14 ] ),\
( [ 1 ],\, [ 35 ] ),\
( [ 1 ],\, [ 24 ] ),\\
( [ 1 ],\, [ 34 ] ),\
( [ 2 ],\, [ 35 ] ),\
( [ 2 ],\, [ 39 ] ),\
( [ 2 ],\, [ 56 ] ),\
( [ 2 ],\, [ 61 ] ),\
( [ 10 ],\, [ 28 ] ),\
( [ 10 ],\, [ 34 ] ),\
( [ 10 ],\, [ 61 ] ),\
( [ 24 ],\, [ 28 ] ),\
( [ 24 ],\, [ 61 ] ),\\
( [ 24 ],\, [ 46 ] ),\
( [ 28 ],\, [ 1 ] ),\
( [ 28 ],\, [ 10 ] ),\
( [ 28 ],\, [ 56 ] ),\
( [ 28 ],\, [ 105 ] ),\
( [ 34 ],\, [ 39 ] ),\
( [ 34 ],\, [ 56 ] ),\
( [ 34 ],\, [ 46 ] ),\
( [ 34 ],\, [ 105 ] ),\\
( [ 39 ],\, [ 1 ] ),\
( [ 39 ],\, [ 2 ] ),\
( [ 39 ],\, [ 168 ] ),\
( [ 56 ],\, [ 2 ] ),\
( [ 56 ],\, [ 34 ] ),\
( [ 56 ],\, [ 168 ] ),\
( [ 61 ],\, [ 10 ] ),\
( [ 61 ],\, [ 24 ] ),\
( [ 61 ],\, [ 168 ] ),\\
( [ 61 ],\, [ 105 ] ),\
( [ 46 ],\, [ 14 ] ),\
( [ 46 ],\, [ 2 ] ),\
( [ 46 ],\, [ 10 ] ),\
( [ 46 ],\, [ 24 ] ),\
( [ 168 ],\, [ 14 ] ),\
( [ 168 ],\, [ 35 ] ),\
( [ 168 ],\, [ 28 ] ),\\
( [ 168 ],\, [ 39 ] ),\
( [ 105 ],\, [ 2 ] ),\
( [ 105 ],\, [ 28 ] ),\
( [ 105 ],\, [ 34 ] )
\big\},
\end{multline*}}%
each of them giving rise to $\binom {m/2}2$ elements of
$\Fix_{NC^m(G_{34})}(\phi^{p})$ since $1\le i_1<i_2\le m$.

There are no solutions for $w_i$ with $\ell_T(w_i)=3$ 
in \eqref{eq:G34''''D}, and hence no solutions for
$(w_{i_1},w_{i_2})$ with $\ell_T(w_{i_1})+\ell_T(w_{i_2})=3$
in \eqref{eq:G34''''E}, and no solutions for
$(w_{i_1},w_{i_2},w_{i_3})$ in \eqref{eq:G34''''F}.

In total, we obtain 
$1+(14+21)\frac m2+49\binom {m/2}2=\frac {(7m+2)(7m+4)}8$ elements in
$\Fix_{NC^m(G_{34})}(\phi^p)$, which agrees with the limit in
\eqref{eq:G34.6}.

If $p=21m/2$, from \eqref{eq:Aktion}, we infer
\begin{multline*}
\phi^p\big((w_0;w_1,\dots,w_m)\big)\\=
(*;
c^{11}w_{\frac m2+1}c^{-11},c^{11}w_{\frac m2+2}c^{-11},
\dots,c^{11}w_{m}c^{-11},
c^{10}w_{1}c^{-10},\dots,
c^{10}w_{\frac m2}c^{-10}\big).
\end{multline*}
Using that $c^7wc^{-7}=w$ for all $w\in NC(G_{34}$, due to Lemma~\ref{lem:4}
with $d=6$, we see that
this action is identical with the one in \eqref{eq:7m2Aktion}. 
Therefore, we are facing exactly the same enumeration 
problem here as for
$p=7m/2$, and, consequently, the number of elements in
$\Fix_{NC^m(G_{34})}(\phi^p)$ is the same, namely
$\frac {(7m+2)(7m+4)}8$, as required.

\smallskip
Finally, we turn to \eqref{eq:G34.1}. By Remark~\ref{rem:1},
the only choices for $h_2$ and $m_2$ to be considered
are $h_2=1$ and $m_2=5$, $h_2=2$ and $m_2=5$, 
$h_2=2$ and $m_2=4$, $h_2=m_2=2$, 
$h_2=3$ and $m_2=5$, $h_2=m_2=3$,
$h_2=6$ and $m_2=6$, $h_2=6$ and $m_2=5$, 
$h_2=6$ and $m_2=4$, $h_2=6$ and $m_2=3$,  
respectively $h_2=6$ and $m_2=2$, . 
These correspond to the choices $p=42m/5$, $p=21m/5$, 
$p=21m/4$, $p=21m/2$, $p=14m/3$, $p=14m/5$, 
$p=7m/6$, $p=7m/5$, $p=7m/4$, $p=7m/3$, 
respectively $p=7m/2$, out of which only $p=7m/6$ 
has not yet been discussed and belongs to the current case.
The corresponding action of $\phi^p$ is given by 
\begin{multline*}
\phi^p\big((w_0;w_1,\dots,w_m)\big)\\=
(*;
c^{2}w_{\frac {5m}6+1}c^{-2},c^{2}w_{\frac {5m}6+2}c^{-2},
\dots,c^{2}w_{m}c^{-2},
cw_{1}c^{-1},\dots,
cw_{\frac {5m}6}c^{-1}\big),
\end{multline*}
so that we have to solve
$$t_1(ct_1c^{-1})(c^{2}t_1c^{-2})(c^{3}t_1c^{-3})
(c^{4}t_1c^{-4})(c^{5}t_1c^{-5})=c.$$
A computation with the help of {\tt CHEVIE} finds no solution. 
Hence, 
the left-hand side of \eqref{eq:1} is equal to $1$, as required.

\subsection*{\sc Case $G_{35}=E_6$}
The degrees are $2,5,6,8,9,12$, and hence we have
$$
\Cat^m(E_6;q)=\frac 
{[12m+12]_q\, [12m+9]_q\, [12m+8]_q\, [12m+6]_q\, [12m+5]_q\, [12m+2]_q} 
{[12]_q\, [9]_q\, [8]_q\, [6]_q\, [5]_q\, [2]_q} .
$$
Let $\zeta$ be a $12m$-th root of unity. 
The following cases on the right-hand side of \eqref{eq:1}
occur:
{\refstepcounter{equation}\label{eq:E6}}
\alphaeqn
\begin{align} 
\label{eq:E6.2}
\lim_{q\to\zeta}\Cat^m(E_6;q)&=m+1,
\quad\text{if }\zeta=\zeta_{12},\\
\label{eq:E6.3}
\lim_{q\to\zeta}\Cat^m(E_6;q)&=\tfrac {4m+3}3,
\quad\text{if }\zeta=\zeta_{9},\ 3\mid m,\\
\label{eq:E6.4}
\lim_{q\to\zeta}\Cat^m(E_6;q)&=\tfrac {3m+2}2,
\quad\text{if }\zeta=\zeta_{8},\ 2\mid m,\\
\label{eq:E6.7}
\lim_{q\to\zeta}\Cat^m(E_6;q)&=(m+1)(2m+1),
\quad\text{if }\zeta= \zeta_{6},\\
\label{eq:E6.5}
\lim_{q\to\zeta}\Cat^m(E_6;q)&=\tfrac {12m+5}5,
\quad\text{if }\zeta=\zeta_{5},\ 5\mid m,\\
\label{eq:E6.6}
\lim_{q\to\zeta}\Cat^m(E_6;q)&=\tfrac {(m+1)(3m+2)}2,
\quad\text{if }\zeta= \zeta_{4},\\
\label{eq:E6.8}
\lim_{q\to\zeta}\Cat^m(E_6;q)&=\tfrac {(m+1)(4m+3)(2m+1)}3,
\quad\text{if }\zeta= \zeta_{3},\\
\label{eq:E6.9}
\lim_{q\to\zeta}\Cat^m(E_6;q)&=\tfrac {(m+1)(3m+2)(2m+1)(6m+1)}2,
\quad\text{if }\zeta= -1,\\
\lim_{q\to\zeta}\Cat^m(E_6;q)&=\Cat^m(E_6),
\quad\text{if }\zeta=1,\\
\label{eq:E6.1}
\lim_{q\to\zeta}\Cat^m(E_6;q)&=1,
\quad\text{otherwise.}
\end{align}
\reseteqn

We must now prove that the left-hand side of \eqref{eq:1} in
each case agrees with the values exhibited in 
\eqref{eq:E6}. The only cases not covered by
Lemmas~\ref{lem:2} and \ref{lem:3} are the ones in 
\eqref{eq:E6.3}, \eqref{eq:E6.4}, \eqref{eq:E6.5},
and \eqref{eq:E6.1}.

\smallskip
We begin with the case in \eqref{eq:E6.3}.
By Lemma~\ref{lem:1}, we are free to choose $p=4m/3$. In particular,
$m$ must be divisible by $3$.
From \eqref{eq:Aktion}, we infer
\begin{multline*}
\phi^p\big((w_0;w_1,\dots,w_m)\big)\\=
(*;
c^{2}w_{\frac {2m}3+1}c^{-2},c^{2}w_{\frac {2m}3+2}c^{-2},
\dots,c^{2}w_{m}c^{-2},
cw_{1}c^{-1},\dots,
cw_{\frac {2m}3}c^{-1}\big).
\end{multline*}
Supposing that 
$(w_0;w_1,\dots,w_m)$ is fixed by $\phi^p$, we obtain
the system of equations
{\refstepcounter{equation}\label{eq:E6A}}
\alphaeqn
\begin{align} \label{eq:E6Aa}
w_i&=c^2w_{\frac {2m}3+i}c^{-2}, \quad i=1,2,\dots,\tfrac {m}3,\\
w_i&=cw_{i-\frac {m}3}c^{-1}, \quad i=\tfrac {m}3+1,\tfrac {m}3+2,\dots,m.
\label{eq:E6Ab}
\end{align}
\reseteqn
There are four distinct possibilities for choosing
the $w_i$'s, $1\le i\le m$: 
{\refstepcounter{equation}\label{eq:E6C}}
\alphaeqn
\begin{enumerate}
\item[(i)]
all the $w_i$'s are equal to $\ep$ (and $w_0=c$), 
\item[(ii)]
there is an $i$ with $1\le i\le \frac m3$ such that
\begin{equation} \label{eq:E6Cii}
\ell_T(w_i)=\ell_T(w_{i+\frac {m}3})=\ell_T(w_{i+\frac {2m}3})=2,
\end{equation}
and all other $w_j$'s are equal to $\ep$,
\item[(iii)]
there is an $i$ with $1\le i\le \frac m3$ such that
\begin{equation} \label{eq:E6Ciii}
\ell_T(w_i)=\ell_T(w_{i+\frac m3})=\ell_T(w_{i+\frac {2m}3})=1,
\end{equation}
and the other $w_j$'s, $1\le j\le m$, are equal to $\ep$,
\item[(iv)]
there are $i_1$ and $i_2$ with $1\le i_1<i_2\le \frac m3$ such that
\begin{equation} \label{eq:E6Civ}
\ell_T(w_{i_1})=\ell_T(w_{i_2})=
\ell_T(w_{i_1+\frac m3})=\ell_T(w_{i_2+\frac m3})=
\ell_T(w_{i_1+\frac {2m}3})=\ell_T(w_{i_2+\frac {2m}3})=1,
\end{equation}
and all other $w_j$'s are equal to $\ep$.
\end{enumerate}
\reseteqn

Moreover, since $(w_0;w_1,\dots,w_m)\in NC^m(E_6)$, we must have 
$w_iw_{i+\frac m3}w_{i+\frac {2m}3}\le_T c$, respectively 
$w_{i_1}w_{i_2}w_{i_1+\frac m3}w_{i_2+\frac m3}
w_{i_1+\frac {2m}3}w_{i_2+\frac {2m}3}=c$.
Together with equations~\eqref{eq:E6A}--\eqref{eq:E6C}, 
this implies that
\begin{equation} \label{eq:E6D}
w_i=c^4w_ic^{-4}\quad\text{and}\quad 
w_i(cw_ic^{-1})(c^2w_ic^{-2})=c,
\end{equation}
respectively that
\begin{equation} \label{eq:E6D2}
w_i=c^4w_ic^{-4},\quad 
w_i(cw_ic^{-1})(c^2w_ic^{-2})\le_T c,\quad\text{and}\quad
\ell_T(w_i)=1,
\end{equation}
respectively that
\begin{equation} \label{eq:E6E}
w_{i_1}=c^4w_{i_1}c^{-4},\quad 
w_{i_1}(cw_{i_1}c^{-1})(c^2w_{i_1}c^{-2})\le_T c,
\quad\text{and}\quad \ell_T(w_{i_1})=1.
\end{equation}
With the help of Stembridge's {\sl Maple} package {\tt coxeter}
\cite{StemAZ}, one obtains four solutions for $w_i$ in 
\eqref{eq:E6D}:
{\small
\begin{equation*} 
w_i\in\big\{      [1, 2, 3, 4, 5, 6, 5, 4, 2, 3],\,
                                   [3, 4],\,
                                [2, 4, 2, 5],\,
                          [1, 3, 4, 5, 4, 3, 1, 6]
\big\},
\end{equation*}}%
where we have again used the short notation of {\tt coxeter},
$\{s_1,s_2,s_3,s_4,s_5,s_6\}$ being a simple system of generators of 
$E_6$,
corresponding to the Dynkin diagram displayed in Figure~\ref{fig:E6}.
Each of the above solutions for $w_i$ gives rise to $m/3$ elements of
$\Fix_{NC^m(E_6)}(\phi^{p})$ since $i$ ranges from $1$ to $m/3$.

\begin{figure}[h]
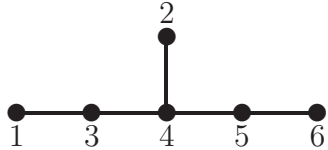

$$
\Einheit1cm
\Pfad(0,0),1111\endPfad
\Pfad(2,0),2\endPfad
\DickPunkt(0,0)
\DickPunkt(1,0)
\DickPunkt(2,0)
\DickPunkt(3,0)
\DickPunkt(4,0)
\DickPunkt(2,1)
\Label\u{1}(0,0)
\Label\u{3}(1,0)
\Label\u{4}(2,0)
\Label\u{5}(3,0)
\Label\u{6}(4,0)
\Label\o{\raise5pt\hbox{2}}(2,1)
\hskip4cm
$$ 
\caption{\scriptsize The Dynkin diagram for $E_6$}
\label{fig:E6}
\end{figure}

There are no solutions for $w_i$ in \eqref{eq:E6D2} and
for $w_{i_1}$ in \eqref{eq:E6E}.

In total, we obtain 
$1+4\frac m3=\frac {4m+3}3$ elements in
$\Fix_{NC^m(E_6)}(\phi^p)$, which agrees with the limit in
\eqref{eq:E6.3}.

\smallskip
Next we discuss the case in \eqref{eq:E6.4}.
By Lemma~\ref{lem:1}, we are free to choose $p=3m/2$. In particular,
$m$ must be divisible by $2$.
From \eqref{eq:Aktion}, we infer
$$
\phi^p\big((w_0;w_1,\dots,w_m)\big)\\=
(*;
c^{2}w_{\frac m2+1}c^{-2},c^{2}w_{\frac m2+2}c^{-2},
\dots,c^{2}w_{m}c^{-2},
cw_{1}c^{-1},\dots,
cw_{\frac m2}c^{-1}\big).
$$
Supposing that 
$(w_0;w_1,\dots,w_m)$ is fixed by $\phi^p$, we obtain
the system of equations
{\refstepcounter{equation}\label{eq:E6'A}}
\alphaeqn
\begin{align} \label{eq:E6'Aa}
w_i&=c^2w_{\frac m2+i}c^{-2}, \quad i=1,2,\dots,\tfrac {m}2,\\
w_i&=cw_{i-\frac {m}2}c^{-1}, \quad i=\tfrac {m}2+1,\tfrac {m}2+2,\dots,m.
\label{eq:E6'Ab}
\end{align}
\reseteqn
There are several distinct possibilities for choosing
the $w_i$'s, $1\le i\le m$, which we summarise as follows: 
{\refstepcounter{equation}\label{eq:E6'C}}
\alphaeqn
\begin{enumerate}
\item[(i)]
all the $w_i$'s are equal to $\ep$ (and $w_0=c$), 
\item[(ii)]
there is an $i$ with $1\le i\le \frac m2$ such that
\begin{equation} \label{eq:E6'Cii}
\ell_T(w_i)=\ell_T(w_{i+\frac m2})=3,
\end{equation}
and all other $w_j$'s are equal to $\ep$,
\item[(iii)]
there is a $j$ with $1\le j\le \frac m2$ such that
\begin{equation} \label{eq:E6'Ciii}
1\le \ell_T(w_j)=\ell_T(w_{j+\frac m2})\le 2.
\end{equation}
\end{enumerate}
\reseteqn

Moreover, since $(w_0;w_1,\dots,w_m)\in NC^m(E_6)$, we must have 
$w_iw_{i+\frac m2}=c$, respectively 
$w_{j}w_{j+\frac m2}\le_T c$.
Together with equations~\eqref{eq:E6'A}--\eqref{eq:E6'C}, 
this implies that
\begin{equation} \label{eq:E6'D}
w_i=c^3w_ic^{-3}\quad\text{and}\quad 
w_i(cw_ic^{-1})=c,
\end{equation}
respectively that
\begin{equation} \label{eq:E6'D2}
w_j=c^3w_jc^{-3},\quad 
w_j(cw_jc^{-1})\le_T c,\quad\text{and}\quad
1\le \ell_T(w_j)\le 2.
\end{equation}
With the help of Stembridge's {\sl Maple} package {\tt coxeter}
\cite{StemAZ}, one obtains three solutions for $w_i$ in 
\eqref{eq:E6'D}:
{\small
\begin{equation*} 
w_i\in\big\{                               [1, 3, 4, 3, 5],\,
                      [1, 2, 3, 4, 3, 5, 6, 5, 4, 3, 1],\,
                            [2, 3, 4, 5, 4, 2, 6]
\big\},
\end{equation*}}%
where we used again {\tt coxeter}'s short notation, $\{s_1,s_2,s_3,s_4,s_5,s_6\}$ being a simple system of generators of 
$E_6$,
corresponding to the Dynkin diagram displayed in Figure~\ref{fig:E6}.
Each of these solutions for $w_i$ gives rise to $m/2$ elements of
$\Fix_{NC^m(E_6)}(\phi^{p})$ since $i$ ranges from $1$ to $m/2$.

There are no solutions for $w_j$ in \eqref{eq:E6'D2}.

In total, we obtain 
$1+3\frac m2=\frac {3m+2}2$ elements in
$\Fix_{NC^m(E_6)}(\phi^p)$, which agrees with the limit in
\eqref{eq:E6.4}.

\smallskip
Finally we discuss the case in \eqref{eq:E6.5}.
By Lemma~\ref{lem:1}, we are free to choose $p=12m/5$. In particular,
$m$ must be divisible by $5$.
From \eqref{eq:Aktion}, we infer
\begin{multline*}
\phi^p\big((w_0;w_1,\dots,w_m)\big)\\=
(*;
c^{3}w_{\frac {3m}5+1}c^{-3},c^{3}w_{\frac {3m}5+2}c^{-3},
\dots,c^{3}w_{m}c^{-3},
c^2w_{1}c^{-2},\dots,
c^2w_{\frac {3m}5}c^{-2}\big).
\end{multline*}
Supposing that 
$(w_0;w_1,\dots,w_m)$ is fixed by $\phi^p$, we obtain
the system of equations
{\refstepcounter{equation}\label{eq:E6''A}}
\alphaeqn
\begin{align} \label{eq:E6''Aa}
w_i&=c^3w_{\frac {3m}5+i}c^{-3}, \quad i=1,2,\dots,\tfrac {2m}5,\\
w_i&=c^2w_{i-\frac {2m}5}c^{-2}, \quad i=\tfrac {2m}5+1,\tfrac {2m}5+2,\dots,m.
\label{eq:E6''Ab}
\end{align}
\reseteqn
There are two distinct possibilities for choosing
the $w_i$'s, $1\le i\le m$: 
\begin{enumerate}
\item[(i)]
all the $w_i$'s are equal to $\ep$ (and $w_0=c$), 
\item[(ii)]
there is an $i$ with $1\le i\le \frac m5$ such that
\begin{equation} \label{eq:E6''C}
\ell_T(w_i)=\ell_T(w_{i+\frac m5})=\ell_T(w_{i+\frac {2m}5})=
\ell_T(w_{i+\frac {3m}5})=\ell_T(w_{i+\frac {4m}5})=1,
\end{equation}
and the other $w_j$'s, $1\le j\le m$, are equal to $\ep$.
\end{enumerate}

Moreover, since $(w_0;w_1,\dots,w_m)\in NC^m(E_6)$, we must have 
$$w_iw_{i+\frac {m}5}w_{i+\frac {2m}5}w_{i+\frac {3m}5}
w_{i+\frac {4m}5}\le_T c.$$
Together with equations~\eqref{eq:E6''A}--\eqref{eq:E6''C}, 
this implies that
\begin{equation} \label{eq:E6''D}
w_i=c^{12}w_ic^{-12}\quad\text{and}\quad 
w_i(c^7w_ic^{-7})(c^2w_ic^{-2})(c^9w_ic^{-9})(c^4w_ic^{-4})\le_T c.
\end{equation}
Here, the first equation is automatically satisfied since 
$c^{12}=\ep$.

With the help of Stembridge's {\sl Maple} package {\tt coxeter}
\cite{StemAZ}, one obtains 12 solutions for $w_i$ in 
\eqref{eq:E6''D}:
{\small
\begin{multline*}
w_i\in\big\{                                     [1],\,
                                     [3],\,
                                     [4],\,
                                     [5],\,
                                     [6],\,
                                  [2, 4, 2],\,
                                  [3, 4, 3],\,
                               [2, 4, 5, 4, 2],\,
                            [1, 3, 4, 5, 4, 3, 1],\\
                         [1, 3, 4, 5, 6, 5, 4, 3, 1],\,
                         [2, 3, 4, 5, 6, 5, 4, 2, 3],\,
                      [1, 2, 3, 4, 5, 6, 5, 4, 2, 3, 1]
\big\},
\end{multline*}}%
where $\{s_1,s_2,s_3,s_4,s_5,s_6\}$ is a simple system of generators of 
$E_6$,
corresponding to the Dynkin diagram displayed in Figure~\ref{fig:E6},
and each of them gives rise to $m/5$ elements of
$\Fix_{NC^m(E_6)}(\phi^{p})$ since $i$ ranges from $1$ to $m/5$.

In total, we obtain 
$1+12\frac m5=\frac {12m+5}5$ elements in
$\Fix_{NC^m(E_6)}(\phi^p)$, which agrees with the limit in
\eqref{eq:E6.5}.

\smallskip
Finally, we turn to \eqref{eq:E6.1}. By Remark~\ref{rem:1},
the only choices for $h_2$ and $m_2$ to be considered
are $h_2=1$ and $m_2=5$ and $h_2=2$ and $m_2=5$. 
These correspond to the choices $p=12m/5$, respectively $p=6m/5$,
out which only $p=6m/5$ has not yet been discussed
 and belongs to the current case.
The corresponding action of $\phi^p$ is given by 
\begin{multline*}
\phi^p\big((w_0;w_1,\dots,w_m)\big)\\=
(*;
c^{2}w_{\frac {4m}5+1}c^{-2},c^{2}w_{\frac {4m}5+2}c^{-2},
\dots,c^{2}w_{m}c^{-2},
cw_{1}c^{-1},\dots,
cw_{\frac {4m}5}c^{-1}\big),
\end{multline*}
so that we have to solve 
$$
t_1(ct_1c^{-1})(c^{2}t_1c^{-2})(c^{3}t_1c^{-3})
(c^{4}t_1c^{-4})\le_T c
$$
for $t_1$ with $\ell_T(t_1)=1$.
A computation with the help of 
Stembridge's {\sl Maple} package {\tt coxeter}
\cite{StemAZ} finds no solution. 
Hence, 
the left-hand side of \eqref{eq:1} is equal to $1$, as required.

\subsection*{\sc Case $G_{36}=E_7$}
The degrees are $2,6,8,10,12,14,18$, and hence we have
\begin{multline*}
\Cat^m(E_7;q)=\frac 
{[18m+18]_q\, [18m+14]_q\, [18m+12]_q} 
{[18]_q\, [14]_q\, [12]_q}\\
\times
\frac 
{[18m+10]_q\, [18m+8]_q\, [18m+6]_q\, [18m+2]_q} 
{[10]_q\, [8]_q\, [6]_q\, [2]_q} .
\end{multline*}
Let $\zeta$ be a $18m$-th root of unity. 
The following cases on the right-hand side of \eqref{eq:1}
occur:
{\refstepcounter{equation}\label{eq:E7}}
\alphaeqn
\begin{align} 
\label{eq:E7.2}
\lim_{q\to\zeta}\Cat^m(E_7;q)&=m+1,
\quad\text{if }\zeta=\zeta_{18},\zeta_{9},\\
\label{eq:E7.3}
\lim_{q\to\zeta}\Cat^m(E_7;q)&=\tfrac {9m+7}7,
\quad\text{if }\zeta=\zeta_{14},\zeta_{7},\ 7\mid m,\\
\label{eq:E7.4}
\lim_{q\to\zeta}\Cat^m(E_7;q)&=\tfrac {3m+2}2,
\quad\text{if }\zeta=\zeta_{12},\ 2\mid m,\\
\label{eq:E7.5}
\lim_{q\to\zeta}\Cat^m(E_7;q)&=\tfrac {9m+5}5,
\quad\text{if }\zeta=\zeta_{10},\zeta_{5},\ 5\mid m,\\
\label{eq:E7.6}
\lim_{q\to\zeta}\Cat^m(E_7;q)&=\tfrac {9m+4}4,
\quad\text{if }\zeta=\zeta_{8},\ 4\mid m,\\
\label{eq:E7.10}
\lim_{q\to\zeta}\Cat^m(E_7;q)&=\frac {(m+1)(3m+2)(3m+1)}2,
\quad\text{if }\zeta= \zeta_{6},\zeta_{3},\\
\label{eq:E7.9}
\lim_{q\to\zeta}\Cat^m(E_7;q)&=\frac{(3m+2)(9m+4)}8,
\quad\text{if }\zeta= \zeta_{4},\ 2\mid m,\\
\lim_{q\to\zeta}\Cat^m(E_7;q)&=\Cat^m(E_7),
\quad\text{if }\zeta=-1\text{ or }\zeta=1,\\
\label{eq:E7.1}
\lim_{q\to\zeta}\Cat^m(E_7;q)&=1,
\quad\text{otherwise.}
\end{align}
\reseteqn

We must now prove that the left-hand side of \eqref{eq:1} in
each case agrees with the values exhibited in 
\eqref{eq:E7}. The only cases not covered by
Lemmas~\ref{lem:2} and \ref{lem:3} are the ones in 
\eqref{eq:E7.3}, \eqref{eq:E7.4}, \eqref{eq:E7.5}, 
\eqref{eq:E7.6}, \eqref{eq:E7.9}, and \eqref{eq:E7.1}.

\smallskip
We begin with the case in \eqref{eq:E7.3}.
By Lemma~\ref{lem:1}, we are free to choose $p=9m/7$ if 
$\zeta=\zeta_{14}$, respectively $p=18m/7$ if 
$\zeta=\zeta_7$. In both cases, $m$ must be divisible by $7$.

We start with the case that $p=9m/7$.
From \eqref{eq:Aktion}, we infer
\begin{multline*}
\phi^p\big((w_0;w_1,\dots,w_m)\big)\\=
(*;
c^{2}w_{\frac {5m}7+1}c^{-2},c^{2}w_{\frac {5m}7+2}c^{-2},
\dots,c^{2}w_{m}c^{-2},
cw_{1}c^{-1},\dots,
cw_{\frac {5m}7}c^{-1}\big).
\end{multline*}
Supposing that 
$(w_0;w_1,\dots,w_m)$ is fixed by $\phi^p$, we obtain
the system of equations
{\refstepcounter{equation}\label{eq:E7'''A}}
\alphaeqn
\begin{align} \label{eq:E7'''Aa}
w_i&=c^2w_{\frac {5m}7+i}c^{-2}, \quad i=1,2,\dots,\tfrac {2m}7,\\
w_i&=cw_{i-\frac {2m}7}c^{-1}, \quad i=\tfrac {2m}7+1,\tfrac {2m}7+2,\dots,m.
\label{eq:E7'''Ab}
\end{align}
\reseteqn
There are two distinct possibilities for choosing
the $w_i$'s, $1\le i\le m$: 
\begin{enumerate}
\item[(i)]
all the $w_i$'s are equal to $\ep$ (and $w_0=c$), 
\item[(ii)]
there is an $i$ with $1\le i\le \frac m7$ such that
\begin{multline} \label{eq:E7'''C}
\ell_T(w_i)=\ell_T(w_{i+\frac m7})=\ell_T(w_{i+\frac {2m}7})=
\ell_T(w_{i+\frac {3m}7})\\=\ell_T(w_{i+\frac {4m}7})=
\ell_T(w_{i+\frac {5m}7})=\ell_T(w_{i+\frac {6m}7})=1,
\end{multline}
and the other $w_j$'s, $1\le j\le m$, are equal to $\ep$.
\end{enumerate}

Moreover, since $(w_0;w_1,\dots,w_m)\in NC^m(E_7)$, we must have 
$$w_iw_{i+\frac {m}7}w_{i+\frac {2m}7}w_{i+\frac {3m}7}
w_{i+\frac {4m}7}w_{i+\frac {5m}7}
w_{i+\frac {6m}7}=c.$$
Together with equations~\eqref{eq:E7'''A}--\eqref{eq:E7'''C}, 
this implies that
\begin{equation} \label{eq:E7'''D}
w_i=c^{9}w_ic^{-9}\quad\text{and}\quad 
w_i(c^5w_ic^{-5})(cw_ic^{-1})(c^6w_ic^{-6})(c^2w_ic^{-2})
(c^7w_ic^{-7})(c^3w_ic^{-3})=c.
\end{equation}
Here, the first equation is automatically satisfied due to
Lemma~\ref{lem:4} with $d=2$.

With the help of Stembridge's {\sl Maple} package {\tt coxeter}
\cite{StemAZ}, one obtains 9 solutions for $w_i$ in 
\eqref{eq:E7'''D}:
{\small
\begin{multline} \label{eq:E7sol1}
w_i\in\big\{                         [4],\,
                                     [5],\,
                                     [6],\,
                                     [7],\,
                                  [3, 4, 3],\,
                               [2, 4, 5, 4, 2],\,
                         [1, 3, 4, 5, 6, 5, 4, 3, 1],\\
                      [2, 3, 4, 5, 6, 7, 6, 5, 4, 2, 3],\,
                   [1, 2, 3, 4, 5, 6, 7, 6, 5, 4, 2, 3, 1]
\big\},
\end{multline}}%
where we have again used the short notation of {\tt coxeter},
$\{s_1,s_2,s_3,s_4,s_5,s_6,s_7\}$ being a simple system of generators of 
$E_7$,
corresponding to the Dynkin diagram displayed in Figure~\ref{fig:E7}.
Each of the above solutions for $w_i$ gives rise to $m/7$ elements of
$\Fix_{NC^m(E_7)}(\phi^{p})$ since $i$ ranges from $1$ to $m/7$.

\begin{figure}[h]
$$
\Einheit1cm
\Pfad(0,0),11111\endPfad
\Pfad(2,0),2\endPfad
\DickPunkt(0,0)
\DickPunkt(1,0)
\DickPunkt(2,0)
\DickPunkt(3,0)
\DickPunkt(4,0)
\DickPunkt(5,0)
\DickPunkt(2,1)
\Label\u{1}(0,0)
\Label\u{3}(1,0)
\Label\u{4}(2,0)
\Label\u{5}(3,0)
\Label\u{6}(4,0)
\Label\u{7}(5,0)
\Label\o{\raise5pt\hbox{2}}(2,1)
\hskip5cm
$$ 
\caption{\scriptsize The Dynkin diagram for $E_7$}
\label{fig:E7}
\end{figure}

In total, we obtain 
$1+9\frac m7=\frac {9m+7}7$ elements in
$\Fix_{NC^m(E_7)}(\phi^p)$, which agrees with the limit in
\eqref{eq:E7.3}.

In the case that $p=18m/7$, we infer from \eqref{eq:Aktion} that
\begin{multline*}
\phi^p\big((w_0;w_1,\dots,w_m)\big)\\=
(*;
c^{3}w_{\frac {3m}7+1}c^{-3},c^{3}w_{\frac {3m}7+2}c^{-3},
\dots,c^{3}w_{m}c^{-3},
c^2w_{1}c^{-2},\dots,
c^2w_{\frac {3m}7}c^{-2}\big).
\end{multline*}
Supposing that 
$(w_0;w_1,\dots,w_m)$ is fixed by $\phi^p$, we obtain
the system of equations
{\refstepcounter{equation}\label{eq:E7'''A'}}
\alphaeqn
\begin{align} \label{eq:E7'''Aa'}
w_i&=c^3w_{\frac {3m}7+i}c^{-3}, \quad i=1,2,\dots,\tfrac {4m}7,\\
w_i&=c^2w_{i-\frac {4m}7}c^{-2}, \quad i=\tfrac {4m}7+1,\tfrac {4m}7+2,\dots,m.
\label{eq:E7'''Ab'}
\end{align}
\reseteqn
There are two distinct possibilities for choosing
the $w_i$'s, $1\le i\le m$: 
\begin{enumerate}
\item[(i)]
all the $w_i$'s are equal to $\ep$ (and $w_0=c$), 
\item[(ii)]
there is an $i$ with $1\le i\le \frac m7$ such that
\begin{multline} \label{eq:E7'''C'}
\ell_T(w_i)=\ell_T(w_{i+\frac m7})=\ell_T(w_{i+\frac {2m}7})=
\ell_T(w_{i+\frac {3m}7})\\=\ell_T(w_{i+\frac {4m}7})=
\ell_T(w_{i+\frac {5m}7})=\ell_T(w_{i+\frac {6m}7})=1,
\end{multline}
and the other $w_j$'s, $1\le j\le m$, are equal to $\ep$.
\end{enumerate}

Moreover, since $(w_0;w_1,\dots,w_m)\in NC^m(E_7)$, we must have 
$$w_iw_{i+\frac {m}7}w_{i+\frac {2m}7}w_{i+\frac {3m}7}
w_{i+\frac {4m}7}w_{i+\frac {5m}7}
w_{i+\frac {6m}7}=c.$$
Together with equations~\eqref{eq:E7'''A'}--\eqref{eq:E7'''C'}, 
this implies that
\begin{multline} \label{eq:E7'''D'}
w_i=c^{18}w_ic^{-18}\\\text{and}\quad 
w_i(c^5w_ic^{-5})(c^{10}w_ic^{-10})(c^{15}w_ic^{-15})(c^2w_ic^{-2})
(c^7w_ic^{-7})(c^{12}w_ic^{-12})=c.
\end{multline}
Here, the first equation is automatically satisfied since 
$c^{18}=\ep$. Due to Lemma~\ref{lem:4} with $d=2$, we have
$c^{9}w_ic^{-9}=w_i$, hence also $c^{10}w_ic^{-10}=cw_ic^{-1}$, 
etc., so that \eqref{eq:E7'''D'} reduces to \eqref{eq:E7'''D}.
Therefore, we are facing exactly the same enumeration 
problem here as for
$p=9m/7$, and, consequently, the number of solutions to \eqref{eq:E7'''D'} is the same, namely
$\frac {9m+7}7$, as required.

\smallskip
Next we discuss the case in \eqref{eq:E7.4}.
By Lemma~\ref{lem:1}, we are free to choose $p=3m/2$. In particular,
$m$ must be divisible by $2$.
From \eqref{eq:Aktion}, we infer
$$
\phi^p\big((w_0;w_1,\dots,w_m)\big)\\=
(*;
c^{2}w_{\frac m2+1}c^{-2},c^{2}w_{\frac m2+2}c^{-2},
\dots,c^{2}w_{m}c^{-2},
cw_{1}c^{-1},\dots,
cw_{\frac m2}c^{-1}\big).
$$
Supposing that 
$(w_0;w_1,\dots,w_m)$ is fixed by $\phi^p$, we obtain
the system of equations
{\refstepcounter{equation}\label{eq:E7'A}}
\alphaeqn
\begin{align} \label{eq:E7'Aa}
w_i&=c^2w_{\frac m2+i}c^{-2}, \quad i=1,2,\dots,\tfrac {m}2,\\
w_i&=cw_{i-\frac {m}2}c^{-1}, \quad i=\tfrac {m}2+1,\tfrac {m}2+2,\dots,m.
\label{eq:E7'Ab}
\end{align}
\reseteqn
There are several distinct possibilities for choosing
the $w_i$'s, $1\le i\le m$, which we summarise as follows: 
\begin{enumerate}
\item[(i)]
all the $w_i$'s are equal to $\ep$ (and $w_0=c$), 
\item[(ii)]
there is an $i$ with $1\le i\le \frac m2$ such that
\begin{equation} \label{eq:E7'C}
1\le \ell_T(w_i)=\ell_T(w_{i+\frac m2})\le 3.
\end{equation}
\end{enumerate}

Moreover, since $(w_0;w_1,\dots,w_m)\in NC^m(E_7)$, we must have 
$w_iw_{i+\frac m2}\le_T c$.
Together with equations~\eqref{eq:E7'A}--\eqref{eq:E7'C}, 
this implies that
\begin{equation} \label{eq:E7'D}
w_i=c^3w_ic^{-3},\quad 
w_i(cw_ic^{-1})\le_T c,\quad\text{and}\quad
1\le \ell_T(w_i)\le 3.
\end{equation}
With the help of Stembridge's {\sl Maple} package {\tt coxeter}
\cite{StemAZ}, one obtains three solutions for $w_i$ in 
\eqref{eq:E7'D} with $\ell_T(w_i)=2$:
{\small
\begin{equation*}
w_i\in\big\{                  [1, 3, 4, 5, 4, 3],\,
                          [2, 3, 4, 5, 6, 5, 4, 2],\,
                 [1, 2, 4, 2, 3, 4, 5, 6, 7, 6, 5, 4, 3, 1]
\big\},
\end{equation*}}%
where we used again {\tt coxeter}'s short notation, $\{s_1,s_2,s_3,s_4,s_5,s_6,s_7\}$ being a simple system of generators of 
$E_7$,
corresponding to the Dynkin diagram displayed in 
Figure~\ref{fig:E7}, and none if $\ell_T(w_i)=1$ or
$\ell_T(w_i)=3$.
Each of the solutions for $w_i$ gives rise to $m/2$ elements of
$\Fix_{NC^m(E_7)}(\phi^{p})$ since $i$ ranges from $1$ to $m/2$.

In total, we obtain 
$1+3\frac m2=\frac {3m+2}2$ elements in
$\Fix_{NC^m(E_7)}(\phi^p)$, which agrees with the limit in
\eqref{eq:E7.4}.

\smallskip
Next we consider the case in \eqref{eq:E7.5}.
By Lemma~\ref{lem:1}, we are free to choose $p=9m/5$ if 
$\zeta=\zeta_{10}$, respectively $p=18m/5$ if 
$\zeta=\zeta_5$. In both cases, $m$ must be divisible by $5$.

We start with the case that $p=9m/5$.
From \eqref{eq:Aktion}, we infer
\begin{equation*}
\phi^p\big((w_0;w_1,\dots,w_m)\big)=
(*;
c^{2}w_{\frac {m}5+1}c^{-2},c^{2}w_{\frac {m}5+2}c^{-2},
\dots,c^{2}w_{m}c^{-2},
cw_{1}c^{-1},\dots,
cw_{\frac {m}5}c^{-1}\big).
\end{equation*}
Supposing that 
$(w_0;w_1,\dots,w_m)$ is fixed by $\phi^p$, we obtain
the system of equations
{\refstepcounter{equation}\label{eq:E7A}}
\alphaeqn
\begin{align} \label{eq:E7Aa}
w_i&=c^2w_{\frac {m}5+i}c^{-2}, \quad i=1,2,\dots,\tfrac {4m}5,\\
w_i&=cw_{i-\frac {4m}5}c^{-1}, \quad i=\tfrac {4m}5+1,\tfrac {4m}5+2,\dots,m.
\label{eq:E7Ab}
\end{align}
\reseteqn
There are two distinct possibilities for choosing
the $w_i$'s, $1\le i\le m$: 
{\refstepcounter{equation}\label{eq:E7C}}
\alphaeqn
\begin{enumerate}
\item[(i)]
all the $w_i$'s are equal to $\ep$ (and $w_0=c$), 
\item[(ii)]
there is an $i$ with $1\le i\le \frac m5$ such that
\begin{equation} \label{eq:E7Cii}
\ell_T(w_i)=\ell_T(w_{i+\frac {m}5})=\ell_T(w_{i+\frac {2m}5})=
\ell_T(w_{i+\frac {3m}5})=\ell_T(w_{i+\frac {4m}5})=1,
\end{equation}
and the other $w_j$'s, $1\le j\le m$, are equal to $\ep$.
\end{enumerate}
\reseteqn

Moreover, since $(w_0;w_1,\dots,w_m)\in NC^m(E_7)$, we must have 
$$w_iw_{i+\frac m5}w_{i+\frac {2m}5}
w_{i+\frac {3m}5}w_{i+\frac {4m}5}\le_T c.$$
Together with equations~\eqref{eq:E7A}--\eqref{eq:E7C}, 
this implies that
\begin{equation} \label{eq:E7D}
w_i=c^9w_ic^{-9}\quad\text{and}\quad 
w_i(c^7w_ic^{-7})(c^5w_ic^{-5})(c^3w_ic^{-3})(cw_ic^{-1})\le_T c.
\end{equation}
With the help of Stembridge's {\sl Maple} package {\tt coxeter}
\cite{StemAZ}, one obtains 9 solutions for $w_i$ in 
\eqref{eq:E7D}:
{\small
\begin{multline*}
w_i\in\big\{                       [4],\,
                                     [5],\,
                                     [6],\,
                                     [7],\,
                                  [3, 4, 3],\,
                               [2, 4, 5, 4, 2],\,
                         [1, 3, 4, 5, 6, 5, 4, 3, 1],\\
                      [2, 3, 4, 5, 6, 7, 6, 5, 4, 2, 3],\,
                   [1, 2, 3, 4, 5, 6, 7, 6, 5, 4, 2, 3, 1]
\big\},
\end{multline*}}%
where $\{s_1,s_2,s_3,s_4,s_5,s_6,s_7\}$ is a simple system of generators of 
$E_7$,
corresponding to the Dynkin diagram displayed in Figure~\ref{fig:E7},
and each of them gives rise to $m/5$ elements of
$\Fix_{NC^m(E_7)}(\phi^{p})$ since $i$ ranges from $1$ to 
$m/5$.\footnote{Miraculously, these are exactly the same solutions
as in the case of \eqref{eq:E7.3}. We have no explanation for this
phenomenon.}

In total, we obtain 
$1+9\frac m5=\frac {9m+5}5$ elements in
$\Fix_{NC^m(E_7)}(\phi^p)$, which agrees with the limit in
\eqref{eq:E7.5}.

In the case that $p=18m/5$, we infer from \eqref{eq:Aktion} that
\begin{multline*}
\phi^p\big((w_0;w_1,\dots,w_m)\big)\\=
(*;
c^{4}w_{\frac {2m}5+1}c^{-4},c^{4}w_{\frac {2m}5+2}c^{-4},
\dots,c^{4}w_{m}c^{-4},
c^3w_{1}c^{-3},\dots,
c^3w_{\frac {2m}5}c^{-3}\big).
\end{multline*}
Supposing that 
$(w_0;w_1,\dots,w_m)$ is fixed by $\phi^p$, we obtain
the system of equations
{\refstepcounter{equation}\label{eq:E7A'}}
\alphaeqn
\begin{align} \label{eq:E7Aa'}
w_i&=c^4w_{\frac {2m}5+i}c^{-4}, \quad i=1,2,\dots,\tfrac {3m}5,\\
w_i&=c^3w_{i-\frac {3m}5}c^{-3}, \quad i=\tfrac {3m}5+1,\tfrac {3m}5+2,\dots,m.
\label{eq:E7Ab'}
\end{align}
\reseteqn
There are two distinct possibilities for choosing
the $w_i$'s, $1\le i\le m$: 
{\refstepcounter{equation}\label{eq:E7C'}}
\alphaeqn
\begin{enumerate}
\item[(i)]
all the $w_i$'s are equal to $\ep$ (and $w_0=c$), 
\item[(ii)]
there is an $i$ with $1\le i\le \frac m5$ such that
\begin{equation} \label{eq:E7Cii'}
\ell_T(w_i)=\ell_T(w_{i+\frac {m}5})=\ell_T(w_{i+\frac {2m}5})=
\ell_T(w_{i+\frac {3m}5})=\ell_T(w_{i+\frac {4m}5})=1,
\end{equation}
and the other $w_j$'s, $1\le j\le m$, are equal to $\ep$.
\end{enumerate}
\reseteqn

Moreover, since $(w_0;w_1,\dots,w_m)\in NC^m(E_7)$, we must have 
$$w_iw_{i+\frac m5}w_{i+\frac {2m}5}
w_{i+\frac {3m}5}w_{i+\frac {4m}5}\le_T c.$$
Together with equations~\eqref{eq:E7A'}--\eqref{eq:E7C'}, 
this implies that
\begin{equation} \label{eq:E7D'}
w_i=c^{18}w_ic^{-18}\quad\text{and}\quad 
w_i(c^7w_ic^{-7})(c^{14}w_ic^{-14})(c^3w_ic^{-3})
(c^{10}w_ic^{-10})\le_T c.
\end{equation}
Here, the first equation is automatically satisfied since 
$c^{18}=\ep$. Due to Lemma~\ref{lem:4} with $d=2$, we have
$c^{9}w_ic^{-9}=w_i$, hence also $c^{14}w_ic^{-14}=c^5w_ic^{-5}$, 
etc., so that \eqref{eq:E7D'} reduces to \eqref{eq:E7D}.
Therefore, we are facing exactly the same enumeration 
problem here as for
$p=9m/5$, and, consequently, the number of solutions to \eqref{eq:E7D'} is the same, namely
$\frac {9m+5}5$, as required.

\smallskip
Our next case is the case in \eqref{eq:E7.6}.
By Lemma~\ref{lem:1}, we are free to choose $p=9m/4$. In particular,
$m$ must be divisible by $4$.
From \eqref{eq:Aktion}, we infer
\begin{multline*}
\phi^p\big((w_0;w_1,\dots,w_m)\big)\\=
(*;
c^{3}w_{\frac {3m}4+1}c^{-3},c^{3}w_{\frac {3m}4+2}c^{-3},
\dots,c^{3}w_{m}c^{-3},
c^2w_{1}c^{-2},\dots,
c^2w_{\frac {3m}4}c^{-2}\big).
\end{multline*}
Supposing that 
$(w_0;w_1,\dots,w_m)$ is fixed by $\phi^p$, we obtain
the system of equations
{\refstepcounter{equation}\label{eq:E7''A}}
\alphaeqn
\begin{align} \label{eq:E7''Aa}
w_i&=c^3w_{\frac {3m}4+i}c^{-3}, \quad i=1,2,\dots,\tfrac {m}4,\\
w_i&=c^2w_{i-\frac {m}4}c^{-2}, \quad i=\tfrac {m}4+1,\tfrac {m}4+2,\dots,m.
\label{eq:E7''Ab}
\end{align}
\reseteqn
There are several distinct possibilities for choosing
the $w_i$'s, $1\le i\le m$, which we summarise as follows: 
\begin{enumerate}
\item[(i)]
all the $w_i$'s are equal to $\ep$ (and $w_0=c$), 
\item[(ii)]
there is an $i$ with $1\le i\le \frac m4$ such that
\begin{equation} \label{eq:E7''C}
\ell_T(w_i)=\ell_T(w_{i+\frac m4})=\ell_T(w_{i+\frac {2m}4})=
\ell_T(w_{i+\frac {3m}4})=1,
\end{equation}
and the other $w_j$'s, $1\le j\le m$, are equal to $\ep$,
\end{enumerate}

Moreover, since $(w_0;w_1,\dots,w_m)\in NC^m(E_7)$, we must have 
$w_iw_{i+\frac {m}4}w_{i+\frac {2m}4}w_{i+\frac {3m}4}\le_T c$.
Together with equations~\eqref{eq:E7''A}--\eqref{eq:E7''C}, 
this implies that
\begin{equation} \label{eq:E7''D}
w_i=c^{9}w_ic^{-9}\quad\text{and}\quad 
w_i(c^2w_ic^{-2})(c^4w_ic^{-4})(c^6w_ic^{-6})\le_T c.
\end{equation}
Here, the first equation in \eqref{eq:E7''D} 
is automatically satisfied due to
Lemma~\ref{lem:4} with $d=2$.

With the help of Stembridge's {\sl Maple} package {\tt coxeter}
\cite{StemAZ}, one obtains 9 solutions for $w_i$ in 
\eqref{eq:E7''D} with $\ell_T(w_i)=1$:
{\small
\begin{multline*}
w_i\in\big\{                         [1],\,
                                     [3],\,
                                  [2, 4, 2],\,
                               [3, 4, 5, 4, 3],\,
                            [1, 3, 4, 5, 4, 3, 1],\,
                            [2, 4, 5, 6, 5, 4, 2],\,
                         [2, 3, 4, 5, 6, 5, 4, 2, 3],\\
                      [1, 3, 4, 5, 6, 7, 6, 5, 4, 3, 1],\,
                [1, 4, 2, 3, 4, 5, 6, 7, 6, 5, 4, 2, 3, 1, 4]
\big\},
\end{multline*}}%
where $\{s_1,s_2,s_3,s_4,s_5,s_6,s_7\}$ is a simple system of generators of 
$E_7$,
corresponding to the Dynkin diagram displayed in Figure~\ref{fig:E7},
and each of them gives rise to $m/4$ elements of
$\Fix_{NC^m(E_7)}(\phi^{p})$ since $i$ ranges from $1$ to $m/4$.
Hence, we obtain 
$1+9\frac m4=\frac {9m+4}4$ elements in
$\Fix_{NC^m(E_7)}(\phi^p)$, which agrees with the limit in
\eqref{eq:E7.6}.

\smallskip
Finally we discuss the case in \eqref{eq:E7.9}.
By Lemma~\ref{lem:1}, we are free to choose $p=9m/2$. In particular,
$m$ must be divisible by $2$.
From \eqref{eq:Aktion}, we infer
\begin{multline*}
\phi^p\big((w_0;w_1,\dots,w_m)\big)\\=
(*;
c^{5}w_{\frac {m}2+1}c^{-5},c^{5}w_{\frac {m}2+2}c^{-5},
\dots,c^{5}w_{m}c^{-5},
c^4w_{1}c^{-4},\dots,
c^4w_{\frac {m}2}c^{-4}\big).
\end{multline*}
Supposing that 
$(w_0;w_1,\dots,w_m)$ is fixed by $\phi^p$, we obtain
the system of equations
{\refstepcounter{equation}\label{eq:E7''''A}}
\alphaeqn
\begin{align} \label{eq:E7''''Aa}
w_i&=c^5w_{\frac {m}2+i}c^{-5}, \quad i=1,2,\dots,\tfrac {m}2,\\
w_i&=c^4w_{i-\frac {m}2}c^{-4}, \quad i=\tfrac {m}2+1,\tfrac {m}2+2,\dots,m.
\label{eq:E7''''Ab}
\end{align}
\reseteqn
There are several distinct possibilities for choosing
the $w_i$'s, $1\le i\le m$: 
{\refstepcounter{equation}\label{eq:E7''''C}}
\alphaeqn
\begin{enumerate}
\item[(i)]
all the $w_i$'s are equal to $\ep$ (and $w_0=c$), 
\item[(ii)]
there is an $i$ with $1\le i\le \frac m2$ such that
\begin{equation} \label{eq:E7''''Cii}
\ell_T(w_i)=\ell_T(w_{i+\frac m2})=3,
\end{equation}
and the other $w_j$'s, $1\le j\le m$, are equal to $\ep$,
\item[(iii)]
there is an $i$ with $1\le i\le \frac m2$ such that
\begin{equation} \label{eq:E7''''Ciii}
\ell_T(w_i)=\ell_T(w_{i+\frac m2})=2,
\end{equation}
and the other $w_j$'s, $1\le j\le m$, are equal to $\ep$,
\item[(iv)]
there is an $i$ with $1\le i\le \frac m2$ such that
\begin{equation} \label{eq:E7''''Civ}
\ell_T(w_i)=\ell_T(w_{i+\frac m2})=1,
\end{equation}
and the other $w_j$'s, $1\le j\le m$, are equal to $\ep$,
\item[(v)]
there are $i_1$ and $i_2$ with $1\le i_1<i_2\le \frac m2$ such that
\begin{equation} \label{eq:E7''''Cv}
\ell_T(w_{i_1})=\ell_T(w_{i_2})=
\ell_T(w_{i_1+\frac m2})=\ell_T(w_{i_2+\frac m2})=1,
\end{equation}
and the other $w_j$'s, $1\le j\le m$, are equal to $\ep$,
\item[(vi)]
there are $i_1$ and $i_2$ with $1\le i_1,i_2\le \frac m2$ such that
\begin{equation} \label{eq:E7''''Cvi}
\ell_T(w_{i_1})=
\ell_T(w_{i_1+\frac m2})=2\quad\text{and}\quad
\ell_T(w_{i_2})=\ell_T(w_{i_2+\frac m2})=1,
\end{equation}
and the other $w_j$'s, $1\le j\le m$, are equal to $\ep$,
\item[(vii)]
there are $i_1,i_2,i_3$ with $1\le i_1<i_2<i_3\le \frac m2$ such that
\begin{equation} \label{eq:E7''''Cvii}
\ell_T(w_{i_1})=\ell_T(w_{i_2})=\ell_T(w_{i_3})=
\ell_T(w_{i_1+\frac m2})=\ell_T(w_{i_2+\frac m2})=
\ell_T(w_{i_3+\frac m2})=1,
\end{equation}
and the other $w_j$'s, $1\le j\le m$, are equal to $\ep$.
\end{enumerate}
\reseteqn

Moreover, since $(w_0;w_1,\dots,w_m)\in NC^m(E_7)$, we must have 
$w_iw_{i+\frac {m}2}\le_T c$,
respectively 
$w_{i_1}w_{i_2}w_{i_1+\frac {m}2}w_{i_2+\frac {m}2}\le_T c$,
respectively 
$$w_{i_1}w_{i_2}w_{i_3}
w_{i_1+\frac {m}2}w_{i_2+\frac {m}2}w_{i_3+\frac {m}2}\le_T c.$$
Together with equations~\eqref{eq:E7''''A}--\eqref{eq:E7''''C}, 
this implies that
\begin{equation} \label{eq:E7''''D}
w_i=c^{9}w_ic^{-9}\quad\text{and}\quad 
w_i(c^4w_ic^{-4})\le_T c,
\end{equation}
respectively that
\begin{equation} \label{eq:E7''''E}
w_{i_1}=c^{9}w_{i_1}c^{-9},\quad 
w_{i_2}=c^{9}w_{i_2}c^{-9},\quad\text{and}\quad  w_{i_1}w_{i_2}(c^4w_{i_1}c^{-4})(c^4w_{i_2}c^{-4})\le_T c,
\end{equation}
respectively that
\begin{multline} \label{eq:E7''''F}
w_{i_1}=c^{9}w_{i_1}c^{-9},\quad 
w_{i_2}=c^{9}w_{i_2}c^{-9},\quad 
w_{i_3}=c^{9}w_{i_3}c^{-9},\\
\quad\text{and}\quad w_{i_1}w_{i_2}w_{i_3}
(c^4w_{i_1}c^{-4})(c^4w_{i_2}c^{-4})(c^4w_{i_3}c^{-4})\le_T c.
\end{multline}
Here, the first equation in \eqref{eq:E7''''D},
the first two in \eqref{eq:E7''''E},
and the first three in \eqref{eq:E7''''F}, are all automatically 
satisfied due to Lemma~\ref{lem:4} with $d=2$.

With the help of Stembridge's {\sl Maple} package {\tt coxeter}
\cite{StemAZ}, one obtains 9 solutions for $w_i$ in 
\eqref{eq:E7''''D}  with $\ell_T(w_i)=1$:
{\small
\begin{multline*}
w_i\in\big\{                          [1],\,
                                     [3],\,
                                  [2, 4, 2],\,
                               [3, 4, 5, 4, 3],\,
                            [1, 3, 4, 5, 4, 3, 1],\,
                            [2, 4, 5, 6, 5, 4, 2],\,
                         [2, 3, 4, 5, 6, 5, 4, 2, 3],\\
                      [1, 3, 4, 5, 6, 7, 6, 5, 4, 3, 1],\,
                [1, 4, 2, 3, 4, 5, 6, 7, 6, 5, 4, 2, 3, 1, 4]
\big\},
\end{multline*}}%
where $\{s_1,s_2,s_3,s_4,s_5,s_6,s_7\}$ is a simple system of generators of $E_7$,
corresponding to the Dynkin diagram displayed in Figure~\ref{fig:E7},
and each of them gives rise to $m/2$ elements of
$\Fix_{NC^m(E_7)}(\phi^{p})$ since $i$ ranges from $1$ to $m/2$.
Furthermore, one obtains 12 solutions for $w_i$ in 
\eqref{eq:E7''''D}  with $\ell_T(w_i)=2$:
{\small
\begin{multline*}
w_i\in\big\{                    [1, 2, 4, 2],\,
                             [1, 3, 4, 5, 4, 3],\,
                          [1, 2, 4, 5, 6, 5, 4, 2],\,
                          [1, 3, 1, 4, 5, 4, 3, 1],\,
                          [2, 3, 4, 5, 6, 5, 4, 2],\\
                    [1, 3, 1, 4, 5, 6, 7, 6, 5, 4, 3, 1],\,
                       [2, 3, 4, 3, 5, 6, 5, 4, 2, 3],\,
                 [1, 2, 4, 2, 3, 4, 5, 6, 7, 6, 5, 4, 3, 1],\\
                    [2, 3, 4, 2, 5, 4, 2, 6, 5, 4, 2, 3],\,
              [1, 3, 1, 4, 3, 5, 4, 3, 1, 6, 7, 6, 5, 4, 3, 1],\\
           [1, 3, 4, 2, 3, 4, 5, 4, 2, 6, 7, 6, 5, 4, 2, 3, 1, 4],\,
        [1, 2, 4, 2, 3, 4, 5, 4, 6, 5, 4, 3, 7, 6, 5, 4, 2, 3, 1, 4]
\big\},
\end{multline*}}%
each of them giving rise to $m/2$ elements of
$\Fix_{NC^m(E_7)}(\phi^{p})$ since $i$ ranges from $1$ to $m/2$,
and one obtains 27 pairs $(w_{i_1},w_{i_2})$ of solutions in 
\eqref{eq:E7''''E} with $\ell_T(w_{i_1})=\ell_T(w_{i_2})=1$:
{\small
\begin{multline*}
w_i\in\big\{                             (  [1], [2, 4, 2]),\ (
                            [1], [3, 4, 5, 4, 3]),\ (
                         [1], [2, 4, 5, 6, 5, 4, 2]),\ (
                         [3], [1, 3, 4, 5, 4, 3, 1]),\\ (
                         [3], [2, 4, 5, 6, 5, 4, 2]),\ (
                   [3], [1, 3, 4, 5, 6, 7, 6, 5, 4, 3, 1]),\ (
                               [2, 4, 2], [1]),\\ (
                   [2, 4, 2], [2, 3, 4, 5, 6, 5, 4, 2, 3]),\ (
                [2, 4, 2], [1, 3, 4, 5, 6, 7, 6, 5, 4, 3, 1]),\\ (
                   [3, 4, 5, 4, 3], [1, 3, 4, 5, 4, 3, 1]),\ (
                [3, 4, 5, 4, 3], [2, 3, 4, 5, 6, 5, 4, 2, 3]),\\ (
             [3, 4, 5, 4, 3], [1, 3, 4, 5, 6, 7, 6, 5, 4, 3, 1]),\ (
                         [1, 3, 4, 5, 4, 3, 1], [1]),\\ (
                         [1, 3, 4, 5, 4, 3, 1], [3]),\ (
    [1, 3, 4, 5, 4, 3, 1], [1, 4, 2, 3, 4, 5, 6, 7, 6, 5, 4, 2, 3, 1, 4]),\\ (
                         [2, 4, 5, 6, 5, 4, 2], [1]),\ (
             [2, 4, 5, 6, 5, 4, 2], [2, 3, 4, 5, 6, 5, 4, 2, 3]),\\ (
    [2, 4, 5, 6, 5, 4, 2], [1, 4, 2, 3, 4, 5, 6, 7, 6, 5, 4, 2, 3, 1, 4]),\ (
                      [2, 3, 4, 5, 6, 5, 4, 2, 3], [3]),\\ (
                   [2, 3, 4, 5, 6, 5, 4, 2, 3], [2, 4, 2]),\ (
                [2, 3, 4, 5, 6, 5, 4, 2, 3], [3, 4, 5, 4, 3]),\\ (
                   [1, 3, 4, 5, 6, 7, 6, 5, 4, 3, 1], [3]),\ (
             [1, 3, 4, 5, 6, 7, 6, 5, 4, 3, 1], [3, 4, 5, 4, 3]),\\ (
               [1, 3, 4, 5, 6, 7, 6, 5, 4, 3, 1], 
                 [1, 4, 2, 3, 4, 5, 6, 7, 6, 5, 4, 2, 3, 1, 4]),\\ (
          [1, 4, 2, 3, 4, 5, 6, 7, 6, 5, 4, 2, 3, 1, 4], [2, 4, 2]),\\ (
    [1, 4, 2, 3, 4, 5, 6, 7, 6, 5, 4, 2, 3, 1, 4], [1, 3, 4, 5, 4, 3, 1]),\\ (
    [1, 4, 2, 3, 4, 5, 6, 7, 6, 5, 4, 2, 3, 1, 4], [2, 4, 5, 6, 5, 4, 2])
\big\},
\end{multline*}}%
each of them giving rise to $\binom {m/2}2$ elements of
$\Fix_{NC^m(E_7)}(\phi^{p})$ since $1\le i_1<i_2\le \frac m2$.

There are no solutions for $w_i$ in \eqref{eq:E7''''D} with
$\ell_T(w_i)=3$, and hence there are no solutions for
$w_{i_1},w_{i_2}$ in \eqref{eq:E7''''E} if we are in case (vi),
or for $w_{i_1},w_{i_2},w_{i_3}$ in \eqref{eq:E7''''F}.

In total, we obtain 
$1+(9+12)\frac m2+27\binom {m/2}2=\frac {(3m+2)(9m+4)}8$ elements in
$\Fix_{NC^m(E_7)}(\phi^p)$, which agrees with the limit in
\eqref{eq:E7.9}.

\smallskip
Finally, we turn to \eqref{eq:E7.1}. By Remark~\ref{rem:1},
the only choices for $h_2$ and $m_2$ to be considered
are $h_2=1$ and $m_2=7$, $h_2=1$ and $m_2=5$, 
$h_2=2$ and $m_2=7$, $h_2=2$ and $m_2=5$, 
$h_2=2$ and $m_2=4$, respectively $h_2=m_2=2$. 
These correspond to the choices $p=18m/7$, $p=18m/5$, $p=9m/7$,
$p=9m/5$, $p=9m/4$, respectively $p=9m/2$, 
all of which have already been discussed
 as they do not belong to \eqref{eq:E7.1}. Hence, 
\eqref{eq:1} must necessarily hold, as required.

\subsection*{\sc Case $G_{37}=E_8$}
The degrees are $2,8,12,14,18,20,24,30$, and hence we have
\begin{multline*}
\Cat^m(E_8;q)=\frac 
{[30m+30]_q\, [30m+24]_q\, [30m+20]_q\, [30m+18]_q} 
{[30]_q\, [24]_q\, [20]_q\, [18]_q}\\
\times
\frac 
{[30m+14]_q\, [30m+12]_q\, [30m+8]_q\, [30m+2]_q} 
{[14]_q\, [12]_q\, [8]_q\, [2]_q} .
\end{multline*}
Let $\zeta$ be a $30m$-th root of unity. 
The following cases on the right-hand side of \eqref{eq:1}
occur:
{\refstepcounter{equation}\label{eq:E8}}
\alphaeqn
\begin{align} 
\label{eq:E8.2}
\lim_{q\to\zeta}\Cat^m(E_8;q)&=m+1,
\quad\text{if }\zeta=\zeta_{30},\zeta_{15},\\
\label{eq:E8.3}
\lim_{q\to\zeta}\Cat^m(E_8;q)&=\tfrac {5m+4}4,
\quad\text{if }\zeta=\zeta_{24},\ 4\mid m,\\
\label{eq:E8.4}
\lim_{q\to\zeta}\Cat^m(E_8;q)&=\tfrac {3m+2}2,
\quad\text{if }\zeta=\zeta_{20},\ 2\mid m,\\
\label{eq:E8.5}
\lim_{q\to\zeta}\Cat^m(E_8;q)&=\tfrac {5m+3}3,
\quad\text{if }\zeta=\zeta_{18},\zeta_{9},\ 3\mid m,\\
\label{eq:E8.6}
\lim_{q\to\zeta}\Cat^m(E_8;q)&=\tfrac {15m+7}7,
\quad\text{if }\zeta=\zeta_{14},\zeta_{7},\ 7\mid m,\\
\label{eq:E8.7}
\lim_{q\to\zeta}\Cat^m(E_8;q)&=\tfrac {(5m+4)(5m+2)}8,
\quad\text{if }\zeta= \zeta_{12},\ 2\mid m,\\
\label{eq:E8.8}
\lim_{q\to\zeta}\Cat^m(E_8;q)&=\tfrac {(m+1)(3m+2)}2,
\quad\text{if }\zeta= \zeta_{10},\zeta_{5},\\
\label{eq:E8.9}
\lim_{q\to\zeta}\Cat^m(E_8;q)&=\frac{(5m+4)(15m+4)}{16},
\quad\text{if }\zeta= \zeta_{8},\ 4\mid m,\\
\label{eq:E8.10}
\lim_{q\to\zeta}\Cat^m(E_8;q)&=\frac{(m+1)(5m+4)(5m+3)(5m+2)}{24},
\quad\text{if }\zeta= \zeta_{6},\zeta_3,\\
\label{eq:E8.11}
\lim_{q\to\zeta}\Cat^m(E_8;q)&=\frac {(5m+4)(3m+2)(5m+2)(15m+4)}{64},
\quad\text{if }\zeta= \zeta_{4},\ 2\mid m,\\
\lim_{q\to\zeta}\Cat^m(E_8;q)&=\Cat^m(E_8),
\quad\text{if }\zeta=-1\text{ or }\zeta=1,\\
\label{eq:E8.1}
\lim_{q\to\zeta}\Cat^m(E_8;q)&=1,
\quad\text{otherwise.}
\end{align}
\reseteqn

We must now prove that the left-hand side of \eqref{eq:1} in
each case agrees with the values exhibited in 
\eqref{eq:E8}. The only cases not covered by
Lemmas~\ref{lem:2} and \ref{lem:3} are the ones in 
\eqref{eq:E8.3}, \eqref{eq:E8.4}, \eqref{eq:E8.5}, 
\eqref{eq:E8.6}, \eqref{eq:E8.7}, \eqref{eq:E8.9}, 
\eqref{eq:E8.11}, and \eqref{eq:E8.1}.

\smallskip
We begin with the case in \eqref{eq:E8.3}.
By Lemma~\ref{lem:1}, we are free to choose $p=5m/4$. In particular,
$m$ must be divisible by $4$.
From \eqref{eq:Aktion}, we infer
\begin{multline*}
\phi^p\big((w_0;w_1,\dots,w_m)\big)\\=
(*;
c^{2}w_{\frac {3m}4+1}c^{-2},c^{2}w_{\frac {3m}4+2}c^{-2},
\dots,c^{2}w_{m}c^{-2},
cw_{1}c^{-1},\dots,
cw_{\frac {3m}4}c^{-1}\big).
\end{multline*}
Supposing that 
$(w_0;w_1,\dots,w_m)$ is fixed by $\phi^p$, we obtain
the system of equations
{\refstepcounter{equation}\label{eq:E8'''A}}
\alphaeqn
\begin{align} \label{eq:E8'''Aa}
w_i&=c^2w_{\frac {3m}4+i}c^{-2}, \quad i=1,2,\dots,\tfrac {m}4,\\
w_i&=cw_{i-\frac {m}4}c^{-1}, \quad i=\tfrac {m}4+1,\tfrac {m}4+2,\dots,m.
\label{eq:E8'''Ab}
\end{align}
\reseteqn
There are several distinct possibilities for choosing
the $w_i$'s, $1\le i\le m$, which we summarise as follows: 
\begin{enumerate}
\item[(i)]
all the $w_i$'s are equal to $\ep$ (and $w_0=c$), 
\item[(ii)]
there is an $i$ with $1\le i\le \frac m4$ such that
\begin{equation} \label{eq:E8'''C}
1\le\ell_T(w_i)=\ell_T(w_{i+\frac m4})=\ell_T(w_{i+\frac {2m}4})=
\ell_T(w_{i+\frac {3m}4})\le2.
\end{equation}
\end{enumerate}

Moreover, since $(w_0;w_1,\dots,w_m)\in NC^m(E_8)$, we must have 
$$w_iw_{i+\frac {m}4}w_{i+\frac {2m}4}w_{i+\frac {3m}4}\le_T c.$$
Together with equations~\eqref{eq:E8'''A}--\eqref{eq:E8'''C}, 
this implies that
\begin{equation} \label{eq:E8'''D}
w_i=c^{5}w_ic^{-5}\quad\text{and}\quad 
w_i(cw_ic^{-1})(c^2w_ic^{-2})(c^3w_ic^{-3})\le_T c.
\end{equation}

With the help of Stembridge's {\sl Maple} package {\tt coxeter}
\cite{StemAZ}, one obtains 5 solutions for $w_i$ in 
\eqref{eq:E8'''D} with $\ell_T(w_i)=2$:
{\small
\begin{multline*} 
w_i\in\big\{   [1, 2, 3, 4, 5, 6, 7, 8, 7, 6, 5, 4, 2, 3, 1, 4],\,
                                [3, 4, 3, 5],\,
                             [2, 4, 5, 4, 2, 6],\\
                       [1, 3, 4, 5, 6, 5, 4, 3, 1, 7],\,
                    [2, 3, 4, 5, 6, 7, 6, 5, 4, 2, 3, 8]
\big\},
\end{multline*}}%
where we have again used the short notation of {\tt coxeter},
$\{s_1,s_2,s_3,s_4,s_5,s_6,s_7,s_8\}$ being a simple system of generators of 
$E_8$,
corresponding to the Dynkin diagram displayed in Figure~\ref{fig:E8}.
Each of the above solutions for $w_i$ gives rise to $m/4$ elements of
$\Fix_{NC^m(E_8)}(\phi^{p})$ since $i$ ranges from $1$ to $m/4$.

\begin{figure}[h]
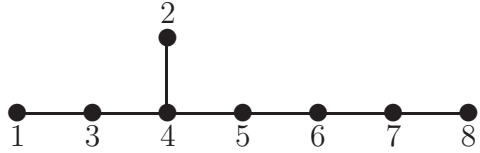

$$
\Einheit1cm
\Pfad(0,0),111111\endPfad
\Pfad(2,0),2\endPfad
\DickPunkt(0,0)
\DickPunkt(1,0)
\DickPunkt(2,0)
\DickPunkt(3,0)
\DickPunkt(4,0)
\DickPunkt(5,0)
\DickPunkt(6,0)
\DickPunkt(2,1)
\Label\u{1}(0,0)
\Label\u{3}(1,0)
\Label\u{4}(2,0)
\Label\u{5}(3,0)
\Label\u{6}(4,0)
\Label\u{7}(5,0)
\Label\u{8}(6,0)
\Label\o{\raise5pt\hbox{2}}(2,1)
\hskip6cm
$$ 
\caption{\scriptsize The Dynkin diagram for $E_8$}
\label{fig:E8}
\end{figure}

There are no solutions for $w_i$ in 
\eqref{eq:E8'''D} with $\ell_T(w_i)=1$.

In total, we obtain 
$1+5\frac m4=\frac {5m+4}4$ elements in
$\Fix_{NC^m(E_8)}(\phi^p)$, which agrees with the limit in
\eqref{eq:E8.3}.

\smallskip
Next we discuss the case in \eqref{eq:E8.4}.
By Lemma~\ref{lem:1}, we are free to choose $p=3m/2$. In particular,
$m$ must be divisible by $2$.
From \eqref{eq:Aktion}, we infer
$$
\phi^p\big((w_0;w_1,\dots,w_m)\big)\\=
(*;
c^{2}w_{\frac m2+1}c^{-2},c^{2}w_{\frac m2+2}c^{-2},
\dots,c^{2}w_{m}c^{-2},
cw_{1}c^{-1},\dots,
cw_{\frac m2}c^{-1}\big).
$$
Supposing that 
$(w_0;w_1,\dots,w_m)$ is fixed by $\phi^p$, we obtain
the system of equations
{\refstepcounter{equation}\label{eq:E8'A}}
\alphaeqn
\begin{align} \label{eq:E8'Aa}
w_i&=c^2w_{\frac m2+i}c^{-2}, \quad i=1,2,\dots,\tfrac {m}2,\\
w_i&=cw_{i-\frac {m}2}c^{-1}, \quad i=\tfrac {m}2+1,\tfrac {m}2+2,\dots,m.
\label{eq:E8'Ab}
\end{align}
\reseteqn
There are several distinct possibilities for choosing
the $w_i$'s, $1\le i\le m$, which we summarise as follows: 
\begin{enumerate}
\item[(i)]
all the $w_i$'s are equal to $\ep$ (and $w_0=c$), 
\item[(ii)]
there is an $i$ with $1\le i\le \frac m2$ such that
\begin{equation} \label{eq:E8'C}
1\le \ell_T(w_i)=\ell_T(w_{i+\frac m2})\le 4.
\end{equation}
\end{enumerate}

Moreover, since $(w_0;w_1,\dots,w_m)\in NC^m(E_8)$, we must have 
$w_iw_{i+\frac m2}\le_T c$.
Together with equations~\eqref{eq:E8'A}--\eqref{eq:E8'C}, 
this implies that
\begin{equation} \label{eq:E8'D}
w_i=c^3w_ic^{-3},\quad 
w_i(cw_ic^{-1})\le_T c,\quad\text{and}\quad
1\le \ell_T(w_i)\le 4.
\end{equation}
With the help of Stembridge's {\sl Maple} package {\tt coxeter}
\cite{StemAZ}, one obtains three solutions for $w_i$ in 
\eqref{eq:E8'D} with $\ell_T(w_i)=4$:
{\small
\begin{multline*} 
w_i\in\big\{ [1, 2, 3, 4, 2, 5, 4, 2, 6, 5, 7, 8, 7, 6, 5, 4, 2, 3],\,
                       [2, 3, 4, 2, 5, 6, 5, 4, 2, 7],\\
              [1, 2, 4, 2, 3, 4, 5, 4, 6, 7, 6, 5, 4, 3, 1, 8]
\big\},
\end{multline*}}%
where we used again {\tt coxeter}'s short notation, $\{s_1,s_2,s_3,s_4,s_5,s_6,s_7,s_8\}$ being a simple system of generators of 
$E_8$,
corresponding to the Dynkin diagram displayed in 
Figure~\ref{fig:E8}.
Each of these solutions for $w_i$ gives rise to $m/2$ elements of
$\Fix_{NC^m(E_8)}(\phi^{p})$ since $i$ ranges from $1$ to $m/2$.
There are no solutions for $w_i$ in 
\eqref{eq:E8'D} with $1\le\ell_T(w_i)\le3$.

In total, we obtain 
$1+3\frac m2=\frac {3m+2}2$ elements in
$\Fix_{NC^m(E_8)}(\phi^p)$, which agrees with the limit in
\eqref{eq:E8.4}.

\smallskip
Next we consider the case in \eqref{eq:E8.5}.
If $\zeta=\zeta_{18}$,
then, by Lemma~\ref{lem:1}, we are free to choose $p=5m/3$, 
whereas, for $\zeta=\zeta_9$, we can choose $p=10m/3$. In particular,
in both cases $m$ must be divisible by $3$.

First, let $p=5m/3$.
From \eqref{eq:Aktion}, we infer
\begin{equation*}
\phi^p\big((w_0;w_1,\dots,w_m)\big)=
(*;
c^{2}w_{\frac {m}3+1}c^{-2},c^{2}w_{\frac {m}3+2}c^{-2},
\dots,c^{2}w_{m}c^{-2},
cw_{1}c^{-1},\dots,
cw_{\frac {m}3}c^{-1}\big).
\end{equation*}
Supposing that 
$(w_0;w_1,\dots,w_m)$ is fixed by $\phi^p$, we obtain
the system of equations
{\refstepcounter{equation}\label{eq:E8A}}
\alphaeqn
\begin{align} \label{eq:E8Aa}
w_i&=c^2w_{\frac {m}3+i}c^{-2}, \quad i=1,2,\dots,\tfrac {2m}3,\\
w_i&=cw_{i-\frac {2m}3}c^{-1}, \quad i=\tfrac {2m}3+1,\tfrac {2m}3+2,\dots,m.
\label{eq:E8Ab}
\end{align}
\reseteqn
There several distinct possibilities for choosing
the $w_i$'s, $1\le i\le m$, which we summarise as follows: 
\begin{enumerate}
\item[(i)]
all the $w_i$'s are equal to $\ep$ (and $w_0=c$), 
\item[(ii)]
there is an $i$ with $1\le i\le \frac m3$ such that
\begin{equation} \label{eq:E8C}
1\le \ell_T(w_i)=\ell_T(w_{i+\frac {m}3})=
\ell_T(w_{i+\frac {2m}3})\le 2.
\end{equation}
\end{enumerate}

Moreover, since $(w_0;w_1,\dots,w_m)\in NC^m(E_8)$, we must have 
$$w_iw_{i+\frac m3}w_{i+\frac {2m}3}\le_T c.$$
Together with equations~\eqref{eq:E8A}--\eqref{eq:E8C}, 
this implies that
\begin{equation} \label{eq:E8D}
w_i=c^5w_ic^{-5}\quad\text{and}\quad 
w_i(c^3w_ic^{-3})(cw_ic^{-1})\le_T c.
\end{equation}
With the help of Stembridge's {\sl Maple} package {\tt coxeter}
\cite{StemAZ}, one obtains five solutions for $w_i$ in 
\eqref{eq:E8D} with $\ell_T(w_i)=2$:
{\small
\begin{multline*}
w_i\in\big\{  [1, 2, 3, 4, 5, 6, 7, 8, 7, 6, 5, 4, 2, 3, 1, 4],\,
                                [3, 4, 3, 5],\,
                             [2, 4, 5, 4, 2, 6],\\
                       [1, 3, 4, 5, 6, 5, 4, 3, 1, 7],\,
                    [2, 3, 4, 5, 6, 7, 6, 5, 4, 2, 3, 8]
\big\},
\end{multline*}}%
where $\{s_1,s_2,s_3,s_4,s_5,s_6,s_7,s_8\}$ is a simple system of generators of 
$E_8$,
corresponding to the Dynkin diagram displayed in Figure~\ref{fig:E8},
and each of them gives rise to $m/3$ elements of
$\Fix_{NC^m(E_8)}(\phi^{p})$ since $i$ ranges from $1$ to 
$m/3$. There are no solutions for $w_i$ in 
\eqref{eq:E8D} with $\ell_T(w_i)=1$.

In total, we obtain 
$1+5\frac m3=\frac {5m+3}3$ elements in
$\Fix_{NC^m(E_8)}(\phi^p)$, which agrees with the limit in
\eqref{eq:E8.5}.

Now let $p=10m/3$.
From \eqref{eq:Aktion}, we infer
\begin{multline*}
\phi^p\big((w_0;w_1,\dots,w_m)\big)\\=
(*;
c^{4}w_{\frac {2m}3+1}c^{-4},c^{4}w_{\frac {2m}3+2}c^{-4},
\dots,c^{4}w_{m}c^{-4},
c^3w_{1}c^{-3},\dots,
c^3w_{\frac {2m}3}c^{-3}\big).
\end{multline*}
Supposing that 
$(w_0;w_1,\dots,w_m)$ is fixed by $\phi^p$, we obtain
the system of equations
{\refstepcounter{equation}\label{eq:E8A'}}
\alphaeqn
\begin{align} \label{eq:E8Aa'}
w_i&=c^4w_{\frac {2m}3+i}c^{-4}, \quad i=1,2,\dots,\tfrac {m}3,\\
w_i&=c^3w_{i-\frac {m}3}c^{-3}, \quad i=\tfrac {m}3+1,\tfrac {m}3+2,\dots,m.
\label{eq:E8Ab'}
\end{align}
\reseteqn
There several distinct possibilities for choosing
the $w_i$'s, $1\le i\le m$, which we summarise as follows: 
\begin{enumerate}
\item[(i)]
all the $w_i$'s are equal to $\ep$ (and $w_0=c$), 
\item[(ii)]
there is an $i$ with $1\le i\le \frac m3$ such that
\begin{equation} \label{eq:E8C'}
1\le \ell_T(w_i)=\ell_T(w_{i+\frac {m}3})=
\ell_T(w_{i+\frac {2m}3})\le 2.
\end{equation}
\end{enumerate}

Moreover, since $(w_0;w_1,\dots,w_m)\in NC^m(E_8)$, we must have 
$$w_iw_{i+\frac m3}w_{i+\frac {2m}3}\le_T c.$$
Together with equations~\eqref{eq:E8A'}--\eqref{eq:E8C'}, 
this implies that
\begin{equation} \label{eq:E8D'}
w_i=c^{10}w_ic^{-10}\quad\text{and}\quad 
w_i(c^3w_ic^{-3})(c^6w_ic^{-6})\le_T c.
\end{equation}
Due to Lemma~\ref{lem:4} with $d=2$, we have
$c^{15}w_ic^{-15}=w_i$, hence also $c^{5}w_ic^{-5}=w_i$, so that 
\eqref{eq:E8D'} reduces to \eqref{eq:E8D}.
Therefore, we are facing exactly the same enumeration 
problem here as for
$p=5m/3$, and, consequently, the number of solutions to \eqref{eq:E8D'} is the same, namely
$\frac {5m+3}3$, as required.

\smallskip
Our next case is the case in \eqref{eq:E8.6}. If $\zeta=\zeta_{14}$,
then, by Lemma~\ref{lem:1}, we are free to choose $p=15m/7$, 
whereas, for $\zeta=\zeta_7$, we can choose $p=30m/7$. In particular,
in both cases $m$ must be divisible by $7$.

First, let $p=15m/7$.
From \eqref{eq:Aktion}, we infer
\begin{multline*}
\phi^p\big((w_0;w_1,\dots,w_m)\big)\\=
(*;
c^{3}w_{\frac {6m}7+1}c^{-3},c^{3}w_{\frac {6m}7+2}c^{-3},
\dots,c^{3}w_{m}c^{-3},
c^2w_{1}c^{-2},\dots,
c^2w_{\frac {6m}7}c^{-2}\big).
\end{multline*}
Supposing that 
$(w_0;w_1,\dots,w_m)$ is fixed by $\phi^p$, we obtain
the system of equations
{\refstepcounter{equation}\label{eq:E8'''''A}}
\alphaeqn
\begin{align} \label{eq:E8'''''Aa}
w_i&=c^3w_{\frac {6m}7+i}c^{-3}, \quad i=1,2,\dots,\tfrac {m}7,\\
w_i&=c^2w_{i-\frac {m}7}c^{-2}, \quad i=\tfrac {m}7+1,\tfrac {m}7+2,\dots,m.
\label{eq:E8'''''Ab}
\end{align}
\reseteqn
There are two distinct possibilities for choosing
the $w_i$'s, $1\le i\le m$: 
\begin{enumerate}
\item[(i)]
all the $w_i$'s are equal to $\ep$ (and $w_0=c$), 
\item[(ii)]
there is an $i$ with $1\le i\le \frac m7$ such that
\begin{multline} \label{eq:E8'''''C}
\ell_T(w_i)=\ell_T(w_{i+\frac m7})=\ell_T(w_{i+\frac {2m}7})=
\ell_T(w_{i+\frac {3m}7})\\=\ell_T(w_{i+\frac {4m}7})=
\ell_T(w_{i+\frac {5m}7})=\ell_T(w_{i+\frac {6m}7})=1,
\end{multline}
and the other $w_j$'s, $1\le j\le m$, are equal to $\ep$.
\end{enumerate}

Moreover, since $(w_0;w_1,\dots,w_m)\in NC^m(E_8)$, we must have 
$$w_iw_{i+\frac {m}7}w_{i+\frac {2m}7}w_{i+\frac {3m}7}
w_{i+\frac {4m}7}w_{i+\frac {5m}7}
w_{i+\frac {6m}7}\le_T c.$$
Together with the 
equations~\eqref{eq:E8'''''A}--\eqref{eq:E8'''''C}, 
this implies that
\begin{multline} \label{eq:E8'''''D}
w_i=c^{15}w_ic^{-15}\\\text{and}\quad 
w_i(c^2w_ic^{-2})(c^4w_ic^{-4})(c^6w_ic^{-6})(c^8w_ic^{-8})
(c^{10}w_ic^{-10})(c^{12}w_ic^{-12})\le_T c.
\end{multline}
Here, the first equation is automatically satisfied due to
Lemma~\ref{lem:4} with $d=2$.

With the help of Stembridge's {\sl Maple} package {\tt coxeter}
\cite{StemAZ}, one obtains 15 solutions for $w_i$ in 
\eqref{eq:E8'''''D}:
{\small
\begin{multline*}
w_i\in\big\{                         [4],\,
                                     [5],\,
                                     [6],\,
                                     [7],\,
                                     [8],\,
                                  [3, 4, 3],\,
                               [2, 4, 5, 4, 2],\,
                               [3, 4, 5, 4, 3],\,
                            [2, 4, 5, 6, 5, 4, 2],\\
                         [1, 3, 4, 5, 6, 5, 4, 3, 1],\,
                      [1, 3, 4, 5, 6, 7, 6, 5, 4, 3, 1],\\
                      [2, 3, 4, 5, 6, 7, 6, 5, 4, 2, 3],\,
                   [2, 3, 4, 5, 6, 7, 8, 7, 6, 5, 4, 2, 3],\\
                [1, 2, 3, 4, 5, 6, 7, 8, 7, 6, 5, 4, 2, 3, 1],\,
             [1, 4, 2, 3, 4, 5, 6, 7, 8, 7, 6, 5, 4, 2, 3, 1, 4]
\big\},
\end{multline*}}%
where we have again used the short notation of {\tt coxeter},
$\{s_1,s_2,s_3,s_4,s_5,s_6,s_7,s_8\}$ being a simple system of generators of  $E_8$,
corresponding to the Dynkin diagram displayed in Figure~\ref{fig:E8},
and each of them gives rise to $m/7$ elements of
$\Fix_{NC^m(E_8)}(\phi^{p})$ since $i$ ranges from $1$ to $m/7$.

In total, we obtain 
$1+15\frac m7=\frac {15m+7}7$ elements in
$\Fix_{NC^m(E_8)}(\phi^p)$, which agrees with the limit in
\eqref{eq:E8.6}.

Now let $p=30m/7$.
From \eqref{eq:Aktion}, we infer
\begin{multline*}
\phi^p\big((w_0;w_1,\dots,w_m)\big)\\=
(*;
c^{5}w_{\frac {5m}7+1}c^{-5},c^{5}w_{\frac {5m}7+2}c^{-5},
\dots,c^{5}w_{m}c^{-5},
c^4w_{1}c^{-4},\dots,
c^4w_{\frac {5m}7}c^{-4}\big).
\end{multline*}
Supposing that 
$(w_0;w_1,\dots,w_m)$ is fixed by $\phi^p$, we obtain
the system of equations
{\refstepcounter{equation}\label{eq:E8''''''A}}
\alphaeqn
\begin{align} \label{eq:E8''''''Aa}
w_i&=c^5w_{\frac {5m}7+i}c^{-5}, \quad i=1,2,\dots,\tfrac {2m}7,\\
w_i&=c^4w_{i-\frac {2m}7}c^{-4}, \quad i=\tfrac {2m}7+1,\tfrac {2m}7+2,\dots,m.
\label{eq:E8''''''Ab}
\end{align}
\reseteqn
There are two distinct possibilities for choosing
the $w_i$'s, $1\le i\le m$: 
\begin{enumerate}
\item[(i)]
all the $w_i$'s are equal to $\ep$ (and $w_0=c$), 
\item[(ii)]
there is an $i$ with $1\le i\le \frac m7$ such that
\begin{multline} \label{eq:E8''''''C}
\ell_T(w_i)=\ell_T(w_{i+\frac m7})=\ell_T(w_{i+\frac {2m}7})=
\ell_T(w_{i+\frac {3m}7})\\=\ell_T(w_{i+\frac {4m}7})=
\ell_T(w_{i+\frac {5m}7})=\ell_T(w_{i+\frac {6m}7})=1,
\end{multline}
and the other $w_j$'s, $1\le j\le m$, are equal to $\ep$.
\end{enumerate}

Moreover, since $(w_0;w_1,\dots,w_m)\in NC^m(E_8)$, we must have 
$$w_iw_{i+\frac {m}7}w_{i+\frac {2m}7}w_{i+\frac {3m}7}
w_{i+\frac {4m}7}w_{i+\frac {5m}7}
w_{i+\frac {6m}7}\le_T c.$$
Together with the 
equations~\eqref{eq:E8''''''A}--\eqref{eq:E8''''''C}, 
this implies that
\begin{multline} \label{eq:E8''''''D}
w_i=c^{30}w_ic^{-30}\\\text{and}\quad 
w_i(c^{17}w_ic^{-17})(c^4w_ic^{-4})(c^{21}w_ic^{-21})(c^8w_ic^{-8})
(c^{25}w_ic^{-25})(c^{12}w_ic^{-12})\le_T c.
\end{multline}
Here, the first equation is automatically satisfied since $c^{30}=\ep$. Moreover, due to Lemma~\ref{lem:4} with $d=2$, we have
$c^{17}w_ic^{-17}=c^{2}w_ic^{-2}$, etc., so that 
\eqref{eq:E8''''''D} reduces to \eqref{eq:E8'''''D}.
Therefore, we are facing exactly the same enumeration 
problem here as for
$p=15m/7$, and, consequently, the number of solutions to \eqref{eq:E8''''''D} is the same, namely
$\frac {15m+7}7$, as required.

\smallskip
We now turn to the case in \eqref{eq:E8.7}.
By Lemma~\ref{lem:1}, we are free to choose $p=5m/2$. In particular,
$m$ must be divisible by $2$.
From \eqref{eq:Aktion}, we infer
\begin{multline*}
\phi^p\big((w_0;w_1,\dots,w_m)\big)\\=
(*;
c^{3}w_{\frac {m}2+1}c^{-3},c^{3}w_{\frac {m}2+2}c^{-3},
\dots,c^{3}w_{m}c^{-3},
c^2w_{1}c^{-2},\dots,
c^2w_{\frac {m}2}c^{-2}\big).
\end{multline*}
Supposing that 
$(w_0;w_1,\dots,w_m)$ is fixed by $\phi^p$, we obtain
the system of equations
{\refstepcounter{equation}\label{eq:E8'''''''A}}
\alphaeqn
\begin{align} \label{eq:E8'''''''Aa}
w_i&=c^3w_{\frac {m}2+i}c^{-3}, \quad i=1,2,\dots,\tfrac {m}2,\\
w_i&=c^2w_{i-\frac {m}2}c^{-2}, \quad i=\tfrac {m}2+1,\tfrac {m}2+2,\dots,m.
\label{eq:E8'''''''Ab}
\end{align}
\reseteqn
There are several distinct possibilities for choosing
the $w_i$'s, $1\le i\le m$: 
{\refstepcounter{equation}\label{eq:E8'''''''C}}
\alphaeqn
\begin{enumerate}
\item[(i)]
all the $w_i$'s are equal to $\ep$ (and $w_0=c$), 
\item[(ii)]
there is an $i$ with $1\le i\le \frac m2$ such that
\begin{equation} \label{eq:E8'''''''Cii}
1\le\ell_T(w_i)=\ell_T(w_{i+\frac m2})\le 4,
\end{equation}
and the other $w_j$'s, $1\le j\le m$, are equal to $\ep$,
\item[(iii)]
there are $i_1$ and $i_2$ with $1\le i_1<i_2\le \frac m2$ such that
\begin{equation} \label{eq:E8'''''''Ciii}
\ell_1:=\ell_T(w_{i_1})=\ell_T(w_{i_1+\frac m2})\ge1,\quad
\ell_2:=\ell_T(w_{i_2})=\ell_T(w_{i_2+\frac m2})\ge1,
\quad\text{and}\quad
\ell_1+\ell_2\le4,
\end{equation}
and the other $w_j$'s, $1\le j\le m$, are equal to $\ep$,
\item[(iv)]
there are $i_1,i_2,i_3$ with $1\le i_1<i_2<i_3\le \frac m2$ such that
\begin{multline} \label{eq:E8'''''''Civ}
\ell_1:=\ell_T(w_{i_1})=\ell_T(w_{i_1+\frac m2})\ge1,\quad
\ell_2:=\ell_T(w_{i_2})=\ell_T(w_{i_2+\frac m2})\ge1,\\
\ell_3:=\ell_T(w_{i_3})=\ell_T(w_{i_3+\frac m2})\ge1,
\quad\text{and}\quad
\ell_1+\ell_2+\ell_3\le4,
\end{multline}
and the other $w_j$'s, $1\le j\le m$, are equal to $\ep$,
\item[(v)]
there are $i_1,i_2,i_3,i_4$ with $1\le i_1<i_2<i_3<i_4\le \frac m2$ such 
that
\begin{multline} \label{eq:E8'''''''Cv}
\ell_T(w_{i_1})=\ell_T(w_{i_2})=\ell_T(w_{i_3})=\ell_T(w_{i_4})\\=
\ell_T(w_{i_1+\frac m2})=\ell_T(w_{i_2+\frac m2})=
\ell_T(w_{i_3+\frac m2})=\ell_T(w_{i_4+\frac m2})=1,
\end{multline}
and all other $w_j$'s are equal to $\ep$.
\end{enumerate}
\reseteqn

Moreover, since $(w_0;w_1,\dots,w_m)\in NC^m(E_8)$, we must have 
$w_iw_{i+\frac {m}2}\le_T c$,
respectively 
$w_{i_1}w_{i_2}w_{i_1+\frac {m}2}w_{i_2+\frac {m}2}\le_T c$,
respectively 
$$w_{i_1}w_{i_2}w_{i_3}
w_{i_1+\frac {m}2}w_{i_2+\frac {m}2}w_{i_3+\frac {m}2}\le_T c,$$
respectively 
$$w_{i_1}w_{i_2}w_{i_3}w_{i_4}
w_{i_1+\frac {m}2}w_{i_2+\frac {m}2}w_{i_3+\frac {m}2}
w_{i_4+\frac {m}2}=c.$$
Together with the 
equations~\eqref{eq:E8'''''''A}--\eqref{eq:E8'''''''C}, 
this implies that
\begin{equation} \label{eq:E8'''''''D}
w_i=c^{5}w_ic^{-5}\quad\text{and}\quad 
w_i(c^2w_ic^{-2})\le_T c,
\end{equation}
respectively that
\begin{equation} \label{eq:E8'''''''E}
w_{i_1}=c^{5}w_{i_1}c^{-5},\quad 
w_{i_2}=c^{5}w_{i_2}c^{-5},\quad\text{and}\quad  w_{i_1}w_{i_2}(c^2w_{i_1}c^{-2})(c^2w_{i_2}c^{-2})\le_T c,
\end{equation}
respectively that
\begin{multline} \label{eq:E8'''''''F}
w_{i_1}=c^{5}w_{i_1}c^{-5},\quad 
w_{i_2}=c^{5}w_{i_2}c^{-5},\quad 
w_{i_3}=c^{5}w_{i_3}c^{-5},\\
\quad\text{and}\quad w_{i_1}w_{i_2}w_{i_3}
(c^2w_{i_1}c^{-2})(c^2w_{i_2}c^{-2})(c^2w_{i_3}c^{-2})\le_T c,
\end{multline}
respectively that
\begin{multline} \label{eq:E8'''''''G}
w_{i_1}=c^{5}w_{i_1}c^{-5},\quad 
w_{i_2}=c^{5}w_{i_2}c^{-5},\quad 
w_{i_3}=c^{5}w_{i_3}c^{-5},\quad 
w_{i_4}=c^{5}w_{i_4}c^{-5},\\
\quad\text{and}\quad w_{i_1}w_{i_2}w_{i_3}w_{i_4}
(c^2w_{i_1}c^{-2})(c^2w_{i_2}c^{-2})(c^2w_{i_3}c^{-2})
(c^2w_{i_4}c^{-2})= c.
\end{multline}

With the help of Stembridge's {\sl Maple} package {\tt coxeter}
\cite{StemAZ}, one obtains 10 solutions for $w_i$ in 
\eqref{eq:E8'''''''D} with $\ell_T(w_i)=2$:
{\small
\begin{multline} \label{eq:fL2}
w_i\in\big\{   [1, 4, 2, 3, 4, 5, 6, 7, 6, 5, 4, 2, 3, 4],\,
    [3, 1, 5, 4, 2, 3, 4, 5, 6, 7, 8, 7, 6, 5, 4, 2, 3, 1, 4, 5],\\
        [1, 2, 3, 4, 5, 6, 7, 8, 7, 6, 5, 4, 2, 3, 1, 4],\,
            [3, 4, 3, 5],\,
            [2, 4, 5, 4, 2, 6],\,
              [1, 3, 4, 5, 6, 5, 4, 3, 1, 7],\\
            [2, 3, 4, 5, 6, 7, 6, 5, 4, 2, 3, 8],\,
             [2, 4, 2, 3, 4, 5, 6, 5, 4, 3],\,
          [1, 3, 4, 5, 4, 2, 3, 1, 4, 5, 6, 7, 6, 5, 4, 2],\\
  [2, 3, 1, 4, 5, 6, 5, 4, 2, 3, 1, 4, 5, 6, 7, 8, 7, 6, 5, 4, 3, 1]
\big\},
\end{multline}}%
and one obtains 10 solutions for $w_i$ in 
\eqref{eq:E8'''''''D} with $\ell_T(w_i)=4$:
{\small
\begin{multline} \label{eq:fL4}
w_i\in\big\{  [1, 2, 3, 4, 5, 6, 7, 6, 5, 4, 2, 3, 4, 8] ,\,  
      [1, 2, 4, 2, 3, 4, 5, 4, 3, 6, 5, 7, 6, 5, 4, 2, 3, 4],\\
       [1, 2, 3, 1, 4, 5, 6, 7, 8, 7, 6, 5, 4, 2, 3, 1, 4, 5],\\
 [1, 3, 1, 4, 5, 4, 2, 3, 4, 5, 6, 5, 4, 2, 7, 6, 8, 7, 6, 5, 4, 2, 3, 1, 4, 5]     ,\\
 [1, 2, 3, 4, 2, 5, 4, 2, 3, 4, 6, 5, 7, 6, 5, 4, 3, 8, 7, 6, 5, 4, 2, 3, 1, 4],\,
     [4, 2, 3, 4, 5, 6],\\
[2, 3, 1, 4, 2, 5, 4, 6, 5, 4, 2, 3, 1, 4, 5, 6, 7, 8, 7, 6, 5, 4, 3, 1],\,
    [1, 5, 4, 2, 3, 1, 4, 5, 6, 7],\\
      [3, 1, 6, 5, 4, 2, 3, 1, 4, 3, 5, 6, 7, 8],\,
  [2, 3, 4, 2, 3, 5, 4, 6, 5, 4, 2, 7, 6, 5, 4, 2, 3, 8]
\big\},
\end{multline}}%
where we have again used the short notation of {\tt coxeter},
$\{s_1,s_2,s_3,s_4,s_5,s_6,s_7,s_8\}$ being a simple system of generators of  $E_8$,
corresponding to the Dynkin diagram displayed in Figure~\ref{fig:E8},
and each of them gives rise to $m/2$ elements of
$\Fix_{NC^m(E_8)}(\phi^{p})$ since $i$ ranges from $1$ to $m/2$.
There are no solutions for $w_i$ in 
\eqref{eq:E8'''''''D} with $\ell_T(w_i)=1$ or $\ell_T(w_i)=3$.

Consequently, there are no solutions in Cases~(iv) and (v),
and the only possible solutions occurring
in Case~(iii) are pairs $(w_{i_1},w_{i_2})$ of elements of
\eqref{eq:fL2} whose product is in \eqref{eq:fL4}.
Another computation using 
Stembridge's {\sl Maple} package {\tt coxeter}
finds the following 25 pairs meeting that description:
{\allowdisplaybreaks\small
\begin{multline*} 
w_i\in\big\{      (       [1, 2, 3, 4, 5, 6, 7, 8, 7, 6, 5, 4, 2, 3, 1, 4], \,
               [1, 4, 2, 3, 4, 5, 6, 7, 6, 5, 4, 2, 3, 4]),\\ (
       [2, 4, 5, 4, 2, 6], \,[1, 4, 2, 3, 4, 5, 6, 7, 6, 5, 4, 2, 3, 4]),\\ (
 [3, 4, 3, 5],\, [3, 1, 5, 4, 2, 3, 4, 5, 6, 7, 8, 7, 6, 5, 4, 2, 3, 1, 4, 5]),\\ (
       [1, 3, 4, 5, 6, 5, 4, 3, 1, 7], \,
         [3, 1, 5, 4, 2, 3, 4, 5, 6, 7, 8, 7, 6, 5, 4, 2, 3, 1, 4, 5]),\\ (
       [3, 1, 5, 4, 2, 3, 4, 5, 6, 7, 8, 7, 6, 5, 4, 2, 3, 1, 4, 5], \,
         [1, 2, 3, 4, 5, 6, 7, 8, 7, 6, 5, 4, 2, 3, 1, 4]),\\ (
             [2, 3, 4, 5, 6, 7, 6, 5, 4, 2, 3, 8], \,
               [1, 2, 3, 4, 5, 6, 7, 8, 7, 6, 5, 4, 2, 3, 1, 4]),\\ (
             [1, 3, 4, 5, 4, 2, 3, 1, 4, 5, 6, 7, 6, 5, 4, 2],\,
               [1, 2, 3, 4, 5, 6, 7, 8, 7, 6, 5, 4, 2, 3, 1, 4]),\\ (
       [1, 2, 3, 4, 5, 6, 7, 8, 7, 6, 5, 4, 2, 3, 1, 4], [3, 4, 3, 5]),\ (
                [2, 4, 2, 3, 4, 5, 6, 5, 4, 3],\, [3, 4, 3, 5]),\\ (
    [2, 3, 1, 4, 5, 6, 5, 4, 2, 3, 1, 4, 5, 6, 7, 8, 7, 6, 5, 4, 3, 1], \,
[3, 4, 3, 5]),\\ (
       [1, 4, 2, 3, 4, 5, 6, 7, 6, 5, 4, 2, 3, 4],\, [2, 4, 5, 4, 2, 6]),\ (
                      [3, 4, 3, 5], \,[2, 4, 5, 4, 2, 6]),\\ (
    [1, 3, 4, 5, 4, 2, 3, 1, 4, 5, 6, 7, 6, 5, 4, 2], \,[2, 4, 5, 4, 2, 6]),\\ (
       [3, 1, 5, 4, 2, 3, 4, 5, 6, 7, 8, 7, 6, 5, 4, 2, 3, 1, 4, 5], \,
         [1, 3, 4, 5, 6, 5, 4, 3, 1, 7]),\\ (
             [2, 4, 5, 4, 2, 6],\, [1, 3, 4, 5, 6, 5, 4, 3, 1, 7]),\\ (
    [2, 3, 1, 4, 5, 6, 5, 4, 2, 3, 1, 4, 5, 6, 7, 8, 7, 6, 5, 4, 3, 1], \,
      [1, 3, 4, 5, 6, 5, 4, 3, 1, 7]),\\ (
                [1, 4, 2, 3, 4, 5, 6, 7, 6, 5, 4, 2, 3, 4], \,
                  [2, 3, 4, 5, 6, 7, 6, 5, 4, 2, 3, 8]),\\ (
    [1, 3, 4, 5, 6, 5, 4, 3, 1, 7],\, [2, 3, 4, 5, 6, 7, 6, 5, 4, 2, 3, 8]),\\ (
    [2, 4, 2, 3, 4, 5, 6, 5, 4, 3],\, [2, 3, 4, 5, 6, 7, 6, 5, 4, 2, 3, 8]),\ (
             [2, 4, 5, 4, 2, 6], \,[2, 4, 2, 3, 4, 5, 6, 5, 4, 3]),\\ (
    [2, 3, 4, 5, 6, 7, 6, 5, 4, 2, 3, 8],\, [2, 4, 2, 3, 4, 5, 6, 5, 4, 3]),\\ (
             [1, 2, 3, 4, 5, 6, 7, 8, 7, 6, 5, 4, 2, 3, 1, 4], \,
               [1, 3, 4, 5, 4, 2, 3, 1, 4, 5, 6, 7, 6, 5, 4, 2]),\\ (
             [1, 3, 4, 5, 6, 5, 4, 3, 1, 7], \,
               [1, 3, 4, 5, 4, 2, 3, 1, 4, 5, 6, 7, 6, 5, 4, 2]),\\ (
    [3, 4, 3, 5], \,
      [2, 3, 1, 4, 5, 6, 5, 4, 2, 3, 1, 4, 5, 6, 7, 8, 7, 6, 5, 4, 3, 1]),\\ (
    [2, 3, 4, 5, 6, 7, 6, 5, 4, 2, 3, 8], \,
      [2, 3, 1, 4, 5, 6, 5, 4, 2, 3, 1, 4, 5, 6, 7, 8, 7, 6, 5, 4, 3, 1])
\big\},
\end{multline*}}%
where $\{s_1,s_2,s_3,s_4,s_5,s_6,s_7,s_8\}$ is a simple system of generators of $E_8$,
corresponding to the Dynkin diagram displayed in Figure~\ref{fig:E8},
and each of them gives rise to $\binom {m/2}2$ elements of
$\Fix_{NC^m(E_8)}(\phi^{p})$ since $1\le i_1<i_2\le \frac m2$.

In total, we obtain 
$$
1+20\frac m2
+25\binom {m/2}2
=\frac {(5m+4)(5m+2)}{8}
$$
elements in
$\Fix_{NC^m(E_8)}(\phi^p)$, which agrees with the limit in
\eqref{eq:E8.7}.

\smallskip
Next we consider the case in \eqref{eq:E8.9}.
By Lemma~\ref{lem:1}, we are free to choose $p=15m/4$. In particular,
$m$ must be divisible by $4$.
From \eqref{eq:Aktion}, we infer
\begin{multline*}
\phi^p\big((w_0;w_1,\dots,w_m)\big)\\=
(*;
c^{4}w_{\frac {m}4+1}c^{-4},c^{4}w_{\frac {m}4+2}c^{-4},
\dots,c^{4}w_{m}c^{-4},
c^3w_{1}c^{-3},\dots,
c^3w_{\frac {m}4}c^{-3}\big).
\end{multline*}
Supposing that 
$(w_0;w_1,\dots,w_m)$ is fixed by $\phi^p$, we obtain
the system of equations
{\refstepcounter{equation}\label{eq:E8''A}}
\alphaeqn
\begin{align} \label{eq:E8''Aa}
w_i&=c^4w_{\frac {m}4+i}c^{-4}, \quad i=1,2,\dots,\tfrac {3m}4,\\
w_i&=c^3w_{i-\frac {3m}4}c^{-3}, \quad i=\tfrac {3m}4+1,\tfrac {3m}4+2,\dots,m.
\label{eq:E8''Ab}
\end{align}
\reseteqn
There are several distinct possibilities for choosing
the $w_i$'s, $1\le i\le m$, which we summarise as follows: 
{\refstepcounter{equation}\label{eq:E8''C}}
\alphaeqn
\begin{enumerate}
\item[(i)]
all the $w_i$'s are equal to $\ep$ (and $w_0=c$), 
\item[(ii)]
there is an $i$ with $1\le i\le \frac m4$ such that
\begin{equation} \label{eq:E8''Cii}
1\le\ell_T(w_i)=\ell_T(w_{i+\frac m4})=\ell_T(w_{i+\frac {2m}4})=
\ell_T(w_{i+\frac {3m}4})\le2,
\end{equation}
and the other $w_j$'s, $1\le j\le m$, are equal to $\ep$,
\item[(iii)]
there are $i_1$ and $i_2$ with $1\le i_1<i_2\le \frac m4$ such that
\begin{multline} \label{eq:E8''Ciii}
\ell_T(w_{i_1})=\ell_T(w_{i_2})=
\ell_T(w_{i_1+\frac m4})=\ell_T(w_{i_2+\frac m4})\\=
\ell_T(w_{i_1+\frac {2m}4})=\ell_T(w_{i_2+\frac {2m}4})=
\ell_T(w_{i_1+\frac {3m}4})=\ell_T(w_{i_2+\frac {3m}4})=1,
\end{multline}
and all other $w_j$ are equal to $\ep$.
\end{enumerate}
\reseteqn

Moreover, since $(w_0;w_1,\dots,w_m)\in NC^m(E_8)$, we must have 
$$w_iw_{i+\frac {m}4}w_{i+\frac {2m}4}w_{i+\frac {3m}4}\le_T c,$$
or
$$
w_{i_1}w_{i_2}w_{i_1+\frac {m}4}w_{i_2+\frac {m}4}
w_{i_1+\frac {2m}4}w_{i_2+\frac {2m}4}
w_{i_1+\frac {3m}4}w_{i_2+\frac {3m}4}=c.
$$
Together with equations~\eqref{eq:E8''A}--\eqref{eq:E8''C}, 
this implies that
\begin{equation} \label{eq:E8''D}
w_i=c^{15}w_ic^{-15}\quad\text{and}\quad 
w_i(c^{11}w_ic^{-11})(c^7w_ic^{-7})
(c^3w_ic^{-3})\le_T c,
\end{equation}
or that
\begin{multline} \label{eq:E8''E}
w_{i_1}=c^{15}w_{i_1}c^{-15},\quad
w_{i_1}=c^{15}w_{i_2}c^{-15},\\
\quad\text{and}\quad 
w_{i_1}w_{i_2}
(c^{11}w_{i_1}c^{-11})(c^{11}w_{i_2}c^{-11})
(c^7w_{i_1}c^{-7})(c^7w_{i_2}c^{-7})
(c^3w_{i_1}c^{-3})(c^3w_{i_2}c^{-3})=c.
\end{multline}
Here, the first equation in \eqref{eq:E8''D} and the first
two equations in \eqref{eq:E8''E} are 
automatically satisfied due to Lemma~\ref{lem:4} with $d=2$.

With the help of Stembridge's {\sl Maple} package {\tt coxeter}
\cite{StemAZ}, one obtains 30 solutions for $w_i$ in 
\eqref{eq:E8''D} with $\ell_T(w_i)=1$:
{\small
\begin{multline*}
w_i\in\big\{                         [4],\,
                                     [5],\,
                                     [6],\,
                                     [7],\,
                                     [8],\,
                                  [3, 4, 3],\,
                                  [4, 5, 4],\,
                                  [5, 6, 5],\,
                                  [6, 7, 6],\,
                                  [7, 8, 7],\,
                               [2, 4, 5, 4, 2],\,
                               [3, 4, 5, 4, 3],\\
                            [2, 4, 5, 6, 5, 4, 2],\,
                         [1, 3, 4, 5, 6, 5, 4, 3, 1],\,
                         [4, 2, 3, 4, 5, 4, 2, 3, 4],\,
                      [1, 3, 4, 5, 6, 7, 6, 5, 4, 3, 1],\\
                      [2, 3, 4, 5, 6, 7, 6, 5, 4, 2, 3],\,
                   [1, 2, 3, 4, 5, 6, 7, 6, 5, 4, 2, 3, 1],\,
                   [5, 4, 2, 3, 4, 5, 6, 5, 4, 2, 3, 4, 5],\\
                   [2, 3, 4, 5, 6, 7, 8, 7, 6, 5, 4, 2, 3],\,
                [1, 5, 4, 2, 3, 4, 5, 6, 5, 4, 2, 3, 1, 4, 5],\\
                [1, 2, 3, 4, 5, 6, 7, 8, 7, 6, 5, 4, 2, 3, 1],\,
                [4, 2, 3, 4, 5, 6, 7, 8, 7, 6, 5, 4, 2, 3, 4],\\
             [1, 4, 2, 3, 4, 5, 6, 7, 8, 7, 6, 5, 4, 2, 3, 1, 4],\,
          [3, 1, 4, 2, 3, 4, 5, 6, 7, 8, 7, 6, 5, 4, 2, 3, 1, 4, 3],\\
          [1, 6, 5, 4, 2, 3, 4, 5, 6, 7, 6, 5, 4, 2, 3, 1, 4, 5, 6],\,
       [3, 1, 6, 5, 4, 2, 3, 4, 5, 6, 7, 6, 5, 4, 2, 3, 1, 4, 3, 5, 6],\\
    [4, 3, 1, 5, 4, 2, 3, 4, 5, 6, 7, 8, 7, 6, 5, 4, 2, 3, 1, 4, 3, 5, 4],\\
 [3, 1, 7, 6, 5, 4, 2, 3, 4, 5, 6, 7, 8, 7, 6, 5, 4, 2, 3, 1, 4, 3, 5, 6, 7],\\
[2, 4, 3, 1, 7, 6, 5, 4, 2, 3, 4, 5, 6, 7, 8, 7, 6, 5, 4, 2, 3, 1, 4, 3, 5, 4, 2, 6, 7]
\big\},
\end{multline*}}%
one obtains 45 solutions for $w_i$ in 
\eqref{eq:E8''D} with $\ell_T(w_i)=2$ and $w_i$ of type $A_1^2$
(as a parabolic Coxeter element; see the end of Section~\ref{sec:prel}):
{\allowdisplaybreaks\tiny
\begin{multline} \label{eq:E8sol4}
w_i\in\big\{[1, 2, 3, 1, 4, 5, 6, 5, 4, 2, 3, 4, 5, 7, 6, 5, 4, 2, 3, 1, 4, 5, 8, 7, 6, 5, 4, 2, 3, 1],\\
[1, 2, 3, 4, 2, 3, 1, 7, 6, 5, 4, 2, 3, 1, 4, 3, 5, 4, 2, 6, 7, 8, 7, 6, 5, 4, 2, 3, 1, 4, 3, 5, 4, 2, 6, 7],\\
[1, 2, 3, 4, 2, 5, 4, 6, 5, 4, 2, 7, 6, 5, 4, 2, 3, 1],\,
[1, 2, 3, 4, 3, 5, 4, 6, 5, 4, 3, 7, 6, 5, 4, 2, 3, 1],\\
[1, 2, 3, 4, 5, 4, 6, 5, 4, 7, 8, 7, 6, 5, 4, 2, 3, 1],\,
[1, 2, 3, 4, 5, 4, 6, 7, 6, 5, 4, 2, 3, 1],\\
[1, 2, 3, 4, 5, 6, 5, 4, 2, 3, 4, 5, 7, 6, 5, 4, 2, 3, 4, 5, 8, 7, 6, 5, 4, 2, 3, 1],\,
[1, 2, 6, 5, 4, 2, 3, 1, 7, 6, 5, 4, 2, 3, 1, 4, 3, 5, 4, 2, 6, 5, 7, 6],\\
[1, 3, 1, 4, 2, 3, 4, 5, 4, 6, 5, 7, 6, 5, 4, 2, 8, 7, 6, 5, 4, 2, 3, 1, 4, 3],\\
[1, 3, 1, 4, 2, 7, 6, 5, 4, 2, 3, 1, 8, 7, 6, 5, 4, 2, 3, 1, 4, 3, 5, 4, 2, 6, 5, 4, 3, 1, 7, 6, 8, 7],\\
[1, 3, 4, 3, 5, 4, 3, 6, 5, 4, 3, 1],\,
[1, 3, 4, 5, 4, 6, 5, 4, 7, 6, 5, 4, 3, 1],\,
[1, 4, 2, 3, 1, 4, 3, 5, 4, 3, 1, 6, 7, 8, 7, 6, 5, 4, 2, 3, 1, 4],\\
[1, 4, 2, 3, 4, 5, 6, 5, 7, 6, 5, 8, 7, 6, 5, 4, 2, 3, 1, 4],\,
[1, 4, 5, 4, 2, 3, 4, 5, 6, 5, 4, 2, 3, 1, 4, 5],\\
[1, 5, 4, 2, 3, 1, 4, 5, 6, 5, 4, 2, 3, 1, 4, 5],\,
[1, 5, 4, 2, 3, 4, 5, 6, 5, 4, 2, 3, 1, 4, 5, 8],\\
[1, 5, 6, 5, 4, 2, 3, 4, 5, 6, 7, 6, 5, 4, 2, 3, 1, 4, 5, 6],\,
[2, 3, 1, 4, 2, 3, 4, 5, 4, 6, 5, 7, 6, 5, 4, 3, 8, 7, 6, 5, 4, 2, 3, 1, 4, 3],\\
[2, 3, 1, 6, 5, 4, 2, 3, 1, 4, 3, 5, 6, 7, 6, 5, 4, 2, 3, 1, 4, 3, 5, 6],\\
[2, 3, 4, 2, 3, 1, 4, 7, 6, 5, 4, 2, 3, 4, 5, 6, 7, 8, 7, 6, 5, 4, 2, 3, 1, 4, 3, 5, 4, 2, 6, 7],\\
[2, 3, 4, 2, 3, 5, 4, 2, 3, 6, 7, 8, 7, 6, 5, 4, 2, 3],\,
[2, 3, 4, 5, 4, 6, 5, 4, 7, 6, 5, 4, 2, 3],\,
[2, 3, 4, 5, 6, 5, 7, 6, 5, 8, 7, 6, 5, 4, 2, 3],\\
[2, 4, 2, 3, 4, 5, 4, 3, 6, 7, 8, 7, 6, 5, 4, 2, 3, 4],\,
[2, 4, 2, 5, 4, 2, 6, 5, 4, 2],\\
[2, 4, 3, 1, 6, 7, 6, 5, 4, 2, 3, 4, 5, 6, 7, 8, 7, 6, 5, 4, 2, 3, 1, 4, 3, 5, 4, 2, 6, 7],\,
[2, 4, 5, 4, 2, 7, 8, 7],\\
[3, 1, 4, 2, 3, 4, 5, 6, 5, 7, 8, 7, 6, 5, 4, 2, 3, 1, 4, 3],\,
[3, 1, 4, 7, 6, 5, 4, 2, 3, 4, 5, 6, 7, 8, 7, 6, 5, 4, 2, 3, 1, 4, 3, 5, 6, 7],\\
[3, 1, 5, 6, 5, 4, 2, 3, 4, 5, 6, 7, 6, 5, 4, 2, 3, 1, 4, 3, 5, 6],\,
[3, 1, 6, 7, 6, 5, 4, 2, 3, 4, 5, 6, 7, 8, 7, 6, 5, 4, 2, 3, 1, 4, 3, 5, 6, 7],\\
[3, 4, 2, 3, 1, 5, 4, 2, 3, 1, 4, 8, 7, 6, 5, 4, 2, 3, 1, 4, 3, 5, 4, 2, 6, 5, 4, 3, 7, 8],\,
[3, 4, 2, 3, 4, 5, 4, 2, 6, 7, 8, 7, 6, 5, 4, 2, 3, 4],\\
[3, 4, 3, 6, 7, 6],\,
[3, 4, 5, 4, 3, 7, 8, 7],\,
[4, 2, 3, 1, 5, 4, 2, 3, 1, 4, 3, 5, 6, 7, 8, 7, 6, 5, 4, 2, 3, 1, 4, 3, 5, 4],\\
[4, 2, 3, 4, 5, 4, 2, 3, 4, 7],\,
[4, 2, 3, 4, 5, 6, 7, 6, 8, 7, 6, 5, 4, 2, 3, 4],\,
[4, 3, 1, 5, 4, 2, 3, 4, 5, 6, 7, 6, 8, 7, 6, 5, 4, 2, 3, 1, 4, 3, 5, 4],\\
[4, 5, 4, 2, 3, 4, 5, 6, 5, 4, 2, 3, 4, 5],\,
[4, 5, 4, 8],\,
[4, 7, 8, 7],\\
[5, 4, 2, 3, 4, 5, 6, 5, 4, 2, 3, 4, 5, 8],\,
[5, 4, 3, 1, 6, 5, 4, 2, 3, 1, 4, 3, 5, 4, 6, 5]
\big\},
\end{multline}}%
and one obtains 20 solutions for $w_i$ in 
\eqref{eq:E8''D} with $\ell_T(w_i)=2$ and $w_i$ of type $A_2$:
{\small
\begin{multline} \label{eq:E8sol5}
w_i\in\big\{[1, 2, 3, 1, 4, 5, 6, 7, 8, 7, 6, 5, 4, 2, 3, 1, 4, 3],\,
[1, 2, 3, 4, 5, 6, 7, 6, 5, 4, 2, 3, 1, 8],\\
[1, 2, 3, 4, 5, 6, 7, 8, 7, 6, 5, 4, 2, 3, 1, 4],\
[1, 2, 4, 5, 4, 2, 3, 4, 5, 6, 5, 4, 3, 1],\\
[1, 2, 4, 5, 6, 5, 4, 2, 3, 4, 5, 6, 7, 6, 5, 4, 3, 1],\,
[1, 3, 1, 4, 5, 6, 5, 4, 2, 3, 1, 4, 5, 6, 7, 6, 5, 4, 2, 3],\\
[1, 3, 1, 4, 5, 6, 7, 6, 5, 4, 2, 3, 1, 4, 5, 6, 7, 8, 7, 6, 5, 4, 2, 3],\,
[1, 3, 4, 5, 6, 5, 4, 3, 1, 7],\\
[1, 4, 2, 3, 1, 4, 5, 4, 6, 7, 8, 7, 6, 5, 4, 2, 3, 1, 4, 3, 5, 4],\\
[2, 3, 4, 3, 1, 5, 6, 7, 6, 5, 4, 2, 3, 1, 4, 3, 5, 6, 7, 8, 7, 6, 5, 4, 2, 3, 1, 4],\,
[2, 3, 4, 5, 6, 7, 6, 5, 4, 2, 3, 8],\\
[2, 3, 4, 5, 6, 7, 8, 7, 6, 5, 4, 2, 3, 4],\,
[2, 4, 5, 4, 2, 6],\,
[3, 4, 2, 3, 4, 5, 4, 2],\,
[3, 4, 3, 5],\\
[3, 4, 5, 4, 2, 3, 4, 5, 6, 5, 4, 2],\,
[4, 5],\,
[5, 6],\,
[6, 7],\,
[7, 8]
\big\},
\end{multline}}%
where $\{s_1,s_2,s_3,s_4,s_5,s_6,s_7,s_8\}$ is a simple system of generators of 
$E_8$,
corresponding to the Dynkin diagram displayed in Figure~\ref{fig:E8},
and each of them gives rise to $m/4$ elements of
$\Fix_{NC^m(E_8)}(\phi^{p})$ since $i$ ranges from $1$ to $m/4$.

The number of solutions in Case~(iii) can be computed
from our knowledge of the solutions in Case~(ii) according to type,
using some elementary counting arguments. 
Namely, the number of
solutions of \eqref{eq:E8''E} is equal to
$$
45\cdot 2+20\cdot 3=150,
$$
since an element of type $A_1^2$ can be decomposed in two ways 
into a product of two elements of absolute length $1$, while for
an element of type $A_2$ this can be done in $3$
ways. 

In total, we obtain 
$1+(30+45+20)\frac m4+150\binom {m/4}2=\frac {(5m+4)(15m+4)}{16}$ elements in
$\Fix_{NC^m(E_8)}(\phi^p)$, which agrees with the limit in
\eqref{eq:E8.9}.

\smallskip
Finally we discuss the case in \eqref{eq:E8.11}.
By Lemma~\ref{lem:1}, we are free to choose $p=15m/2$. In particular,
$m$ must be divisible by $2$.
From \eqref{eq:Aktion}, we infer
\begin{multline*}
\phi^p\big((w_0;w_1,\dots,w_m)\big)\\=
(*;
c^{8}w_{\frac {m}2+1}c^{-8},c^{8}w_{\frac {m}2+2}c^{-8},
\dots,c^{8}w_{m}c^{-8},
c^7w_{1}c^{-7},\dots,
c^7w_{\frac {m}2}c^{-7}\big).
\end{multline*}
Supposing that 
$(w_0;w_1,\dots,w_m)$ is fixed by $\phi^p$, we obtain
the system of equations
{\refstepcounter{equation}\label{eq:E8''''A}}
\alphaeqn
\begin{align} \label{eq:E8''''Aa}
w_i&=c^8w_{\frac {m}2+i}c^{-8}, \quad i=1,2,\dots,\tfrac {m}2,\\
w_i&=c^7w_{i-\frac {m}2}c^{-7}, \quad i=\tfrac {m}2+1,\tfrac {m}2+2,\dots,m.
\label{eq:E8''''Ab}
\end{align}
\reseteqn
There are several distinct possibilities for choosing
the $w_i$'s, $1\le i\le m$: 
{\refstepcounter{equation}\label{eq:E8''''C}}
\alphaeqn
\begin{enumerate}
\item[(i)]
all the $w_i$'s are equal to $\ep$ (and $w_0=c$), 
\item[(ii)]
there is an $i$ with $1\le i\le \frac m2$ such that
\begin{equation} \label{eq:E8''''Cii}
1\le\ell_T(w_i)=\ell_T(w_{i+\frac m2})\le 4,
\end{equation}
and the other $w_j$'s, $1\le j\le m$, are equal to $\ep$,
\item[(iii)]
there are $i_1$ and $i_2$ with $1\le i_1<i_2\le \frac m2$ such that
\begin{equation} \label{eq:E8''''Ciii}
\ell_1:=\ell_T(w_{i_1})=\ell_T(w_{i_1+\frac m2})\ge1,\quad
\ell_2:=\ell_T(w_{i_2})=\ell_T(w_{i_2+\frac m2})\ge1,\quad\text{and}\quad
\ell_1+\ell_2\le4,
\end{equation}
and the other $w_j$'s, $1\le j\le m$, are equal to $\ep$,
\item[(iv)]
there are $i_1,i_2,i_3$ with $1\le i_1<i_2<i_3\le \frac m2$ such that
\begin{multline} \label{eq:E8''''Civ}
\ell_1:=\ell_T(w_{i_1})=\ell_T(w_{i_1+\frac m2})\ge1,\quad
\ell_2:=\ell_T(w_{i_2})=\ell_T(w_{i_2+\frac m2})\ge1,\\
\ell_3:=\ell_T(w_{i_3})=\ell_T(w_{i_3+\frac m2})\ge1,
\quad\text{and}\quad
\ell_1+\ell_2+\ell_3\le4,
\end{multline}
and the other $w_j$'s, $1\le j\le m$, are equal to $\ep$,
\item[(v)]
there are $i_1,i_2,i_3,i_4$ with $1\le i_1<i_2<i_3<i_4\le \frac m2$ such that
\begin{multline} \label{eq:E8''''Cv}
\ell_T(w_{i_1})=\ell_T(w_{i_2})=\ell_T(w_{i_3})=\ell_T(w_{i_4})\\=
\ell_T(w_{i_1+\frac m2})=\ell_T(w_{i_2+\frac m2})=
\ell_T(w_{i_3+\frac m2})=\ell_T(w_{i_4+\frac m2})=1,
\end{multline}
and all other $w_j$'s are equal to $\ep$.
\end{enumerate}
\reseteqn

Moreover, since $(w_0;w_1,\dots,w_m)\in NC^m(E_8)$, we must have 
$w_iw_{i+\frac {m}2}\le_T c$,
respectively 
$w_{i_1}w_{i_2}w_{i_1+\frac {m}2}w_{i_2+\frac {m}2}\le_T c$,
respectively 
$$w_{i_1}w_{i_2}w_{i_3}
w_{i_1+\frac {m}2}w_{i_2+\frac {m}2}w_{i_3+\frac {m}2}\le_T c,$$
respectively 
$$w_{i_1}w_{i_2}w_{i_3}w_{i_4}
w_{i_1+\frac {m}2}w_{i_2+\frac {m}2}w_{i_3+\frac {m}2}
w_{i_4+\frac {m}2}=c.$$
Together with equations~\eqref{eq:E8''''A}--\eqref{eq:E8''''C}, 
this implies that
\begin{equation} \label{eq:E8''''D}
w_i=c^{15}w_ic^{-15}\quad\text{and}\quad 
w_i(c^7w_ic^{-7})\le_T c,
\end{equation}
respectively that
\begin{equation} \label{eq:E8''''E}
w_{i_1}=c^{15}w_{i_1}c^{-15},\quad 
w_{i_2}=c^{15}w_{i_2}c^{-15},\quad\text{and}\quad  w_{i_1}w_{i_2}(c^7w_{i_1}c^{-7})(c^7w_{i_2}c^{-7})\le_T c,
\end{equation}
respectively that
\begin{multline} \label{eq:E8''''F}
w_{i_1}=c^{15}w_{i_1}c^{-15},\quad 
w_{i_2}=c^{15}w_{i_2}c^{-15},\quad 
w_{i_3}=c^{15}w_{i_3}c^{-15},\\
\quad\text{and}\quad w_{i_1}w_{i_2}w_{i_3}
(c^7w_{i_1}c^{-7})(c^7w_{i_2}c^{-7})(c^7w_{i_3}c^{-7})\le_T c,
\end{multline}
respectively that
\begin{multline} \label{eq:E8''''G}
w_{i_1}=c^{15}w_{i_1}c^{-15},\quad 
w_{i_2}=c^{15}w_{i_2}c^{-15},\quad 
w_{i_3}=c^{15}w_{i_3}c^{-15},\quad 
w_{i_4}=c^{15}w_{i_4}c^{-15},\\
\quad\text{and}\quad w_{i_1}w_{i_2}w_{i_3}w_{i_4}
(c^7w_{i_1}c^{-7})(c^7w_{i_2}c^{-7})(c^7w_{i_3}c^{-7})
(c^7w_{i_4}c^{-7})=c.
\end{multline}
Here, the first equation in \eqref{eq:E8''''D},
the first two in \eqref{eq:E8''''E},
the first three in \eqref{eq:E8''''F},
and the first four in \eqref{eq:E8''''G}, are all automatically 
satisfied due to Lemma~\ref{lem:4} with $d=2$.

With the help of Stembridge's {\sl Maple} package {\tt coxeter}
\cite{StemAZ}, one obtains 45 solutions for $w_i$ in 
\eqref{eq:E8''''D} with $\ell_T(w_i)=1$:
{\small
\begin{multline*}
w_i\in\big\{                         [1],\,
                                     [3],\,
                                     [4],\,
                                     [5],\,
                                     [6],\,
                                     [7],\,
                                     [8],\,
                                  [2, 4, 2],\,
                                  [3, 4, 3],\,
                                  [4, 5, 4],\,
                                  [5, 6, 5],\,
                                  [6, 7, 6],\,
                                  [7, 8, 7],\\
                               [2, 4, 5, 4, 2],\,
                               [3, 4, 5, 4, 3],\,
                            [1, 3, 4, 5, 4, 3, 1],\,
                            [2, 4, 5, 6, 5, 4, 2],\,
                            [3, 4, 5, 6, 5, 4, 3],\\
                         [1, 3, 4, 5, 6, 5, 4, 3, 1],\,
                         [4, 2, 3, 4, 5, 4, 2, 3, 4],\,
                         [2, 3, 4, 5, 6, 5, 4, 2, 3],\,
                         [2, 4, 5, 6, 7, 6, 5, 4, 2],\\
                      [1, 3, 4, 5, 6, 7, 6, 5, 4, 3, 1],\,
                      [4, 2, 3, 4, 5, 6, 5, 4, 2, 3, 4],\,
                      [2, 3, 4, 5, 6, 7, 6, 5, 4, 2, 3],\\
                   [1, 2, 3, 4, 5, 6, 7, 6, 5, 4, 2, 3, 1],\,
                   [1, 3, 4, 5, 6, 7, 8, 7, 6, 5, 4, 3, 1],\,
                   [5, 4, 2, 3, 4, 5, 6, 5, 4, 2, 3, 4, 5],\\
                   [4, 2, 3, 4, 5, 6, 7, 6, 5, 4, 2, 3, 4],\,
                   [2, 3, 4, 5, 6, 7, 8, 7, 6, 5, 4, 2, 3],\,
                [1, 5, 4, 2, 3, 4, 5, 6, 5, 4, 2, 3, 1, 4, 5],\\
                [1, 4, 2, 3, 4, 5, 6, 7, 6, 5, 4, 2, 3, 1, 4],\,
                [1, 2, 3, 4, 5, 6, 7, 8, 7, 6, 5, 4, 2, 3, 1],\\
                [4, 2, 3, 4, 5, 6, 7, 8, 7, 6, 5, 4, 2, 3, 4],\,
             [1, 5, 4, 2, 3, 4, 5, 6, 7, 6, 5, 4, 2, 3, 1, 4, 5],\\
             [1, 4, 2, 3, 4, 5, 6, 7, 8, 7, 6, 5, 4, 2, 3, 1, 4],\,
          [3, 1, 4, 2, 3, 4, 5, 6, 7, 8, 7, 6, 5, 4, 2, 3, 1, 4, 3],\\
          [1, 6, 5, 4, 2, 3, 4, 5, 6, 7, 6, 5, 4, 2, 3, 1, 4, 5, 6],\,
          [1, 5, 4, 2, 3, 4, 5, 6, 7, 8, 7, 6, 5, 4, 2, 3, 1, 4, 5],\\
       [3, 1, 6, 5, 4, 2, 3, 4, 5, 6, 7, 6, 5, 4, 2, 3, 1, 4, 3, 5, 6],\\
       [3, 1, 5, 4, 2, 3, 4, 5, 6, 7, 8, 7, 6, 5, 4, 2, 3, 1, 4, 3, 5],\\
    [4, 3, 1, 5, 4, 2, 3, 4, 5, 6, 7, 8, 7, 6, 5, 4, 2, 3, 1, 4, 3, 5, 4],\\
    [3, 1, 6, 5, 4, 2, 3, 4, 5, 6, 7, 8, 7, 6, 5, 4, 2, 3, 1, 4, 3, 5, 6],\\
 [3, 1, 7, 6, 5, 4, 2, 3, 4, 5, 6, 7, 8, 7, 6, 5, 4, 2, 3, 1, 4, 3, 5, 6, 7],\\
[2, 4, 3, 1, 7, 6, 5, 4, 2, 3, 4, 5, 6, 7, 8, 7, 6, 5, 4, 2, 3, 1, 4, 3, 5, 4, 2, 6, 7]
\big\},
\end{multline*}}%
one obtains 150 solutions for $w_i$ in 
\eqref{eq:E8''''D} with $\ell_T(w_i)=2$ and $w_i$ of type $A_1^2$:
{\allowdisplaybreaks\tiny
\begin{multline*}
w_i\in\big\{ [1, 2, 3, 1, 4, 3, 5, 4, 6, 5, 7, 6, 5, 4, 3, 1, 8, 7, 6, 5, 4, 2, 3, 1],\\
[1, 2, 3, 1, 4, 5, 6, 5, 4, 2, 3, 4, 5, 7, 6, 5, 4, 2, 3, 1, 4, 5, 8, 7, 6, 5, 4, 2, 3, 1],\,
[1, 2, 3, 1, 4, 5, 6, 7, 8, 7, 6, 5, 4, 2, 3, 1],\\
[1, 2, 3, 4, 2, 3, 1, 7, 6, 5, 4, 2, 3, 1, 4, 3, 5, 4, 2, 6, 7, 8, 7, 6, 5, 4, 2, 3, 1, 4, 3, 5, 4, 2, 6, 7],\\
[1, 2, 3, 4, 2, 3, 5, 4, 6, 5, 7, 6, 5, 4, 2, 3, 8, 7, 6, 5, 4, 2, 3, 1],\,
[1, 2, 3, 4, 2, 5, 4, 2, 6, 7, 8, 7, 6, 5, 4, 2, 3, 1],\\
[1, 2, 3, 4, 2, 5, 4, 6, 5, 4, 2, 7, 6, 5, 4, 2, 3, 1],\,
[1, 2, 3, 4, 2, 5, 4, 6, 5, 4, 2, 7, 8, 7, 6, 5, 4, 2, 3, 1],\\
[1, 2, 3, 4, 3, 5, 4, 6, 5, 4, 3, 7, 6, 5, 4, 2, 3, 1],\,
[1, 2, 3, 4, 3, 5, 4, 6, 5, 7, 6, 5, 4, 3, 8, 7, 6, 5, 4, 2, 3, 1],\\
[1, 2, 3, 4, 5, 4, 2, 3, 4, 6, 5, 7, 6, 5, 4, 2, 3, 4, 8, 7, 6, 5, 4, 2, 3, 1],\,
[1, 2, 3, 4, 5, 4, 6, 5, 4, 7, 8, 7, 6, 5, 4, 2, 3, 1],\\
[1, 2, 3, 4, 5, 4, 6, 7, 6, 5, 4, 2, 3, 1],\,
[1, 2, 3, 4, 5, 4, 6, 7, 8, 7, 6, 5, 4, 2, 3, 1],\\
[1, 2, 3, 4, 5, 6, 5, 4, 2, 3, 4, 5, 7, 6, 5, 4, 2, 3, 4, 5, 8, 7, 6, 5, 4, 2, 3, 1],\,
[1, 2, 3, 4, 5, 6, 5, 7, 8, 7, 6, 5, 4, 2, 3, 1],\\
[1, 2, 4, 2, 3, 1, 7, 6, 5, 4, 2, 3, 1, 4, 3, 5, 4, 2, 6, 7],\,
[1, 2, 4, 2, 3, 1, 8, 7, 6, 5, 4, 2, 3, 1, 4, 3, 5, 4, 2, 6, 7, 8],\\
[1, 2, 4, 2, 3, 4, 5, 4, 3, 6, 7, 6, 5, 4, 2, 3, 1, 4],\,
[1, 2, 4, 2, 3, 4, 5, 4, 3, 6, 7, 8, 7, 6, 5, 4, 2, 3, 1, 4],\\
[1, 2, 4, 2, 3, 4, 5, 4, 6, 5, 4, 3, 7, 6, 5, 4, 2, 3, 1, 4],\,
[1, 2, 4, 2, 3, 4, 5, 4, 6, 5, 7, 6, 5, 4, 3, 8, 7, 6, 5, 4, 2, 3, 1, 4],\\
[1, 2, 4, 5, 4, 2, 3, 4, 5, 6, 5, 4, 3, 7, 6, 5, 4, 2, 3, 1, 4, 5],\,
[1, 2, 4, 5, 4, 2],\,
[1, 2, 4, 5, 6, 5, 4, 2],\\
[1, 2, 6, 5, 4, 2, 3, 1, 7, 6, 5, 4, 2, 3, 1, 4, 3, 5, 4, 2, 6, 5, 7, 6],\,
[1, 3, 1, 4, 2, 3, 4, 5, 4, 6, 5, 7, 6, 5, 4, 2, 8, 7, 6, 5, 4, 2, 3, 1, 4, 3],\\
[1, 3, 1, 4, 2, 5, 4, 2, 3, 1, 8, 7, 6, 5, 4, 2, 3, 1, 4, 3, 5, 4, 2, 6, 5, 4, 3, 1, 7, 8],\\
[1, 3, 1, 4, 2, 7, 6, 5, 4, 2, 3, 1, 8, 7, 6, 5, 4, 2, 3, 1, 4, 3, 5, 4, 2, 6, 5, 4, 3, 1, 7, 6, 8, 7],\\
[1, 3, 1, 4, 3, 1, 5, 6, 7, 8, 7, 6, 5, 4, 3, 1],\,
[1, 3, 1, 4, 3, 5, 4, 3, 1, 6, 5, 4, 3, 1],\,
[1, 3, 1, 4, 3, 5, 4, 3, 1, 6, 7, 8, 7, 6, 5, 4, 3, 1],\\
[1, 3, 1, 4, 5, 4, 2, 3, 4, 5, 6, 5, 4, 2, 7, 8, 7, 6, 5, 4, 2, 3, 1, 4, 3, 5],\\
[1, 3, 1, 4, 5, 4, 2, 3, 4, 5, 6, 5, 7, 6, 5, 4, 2, 8, 7, 6, 5, 4, 2, 3, 1, 4, 3, 5],\\
[1, 3, 1, 4, 5, 6, 5, 4, 2, 3, 4, 5, 6, 7, 6, 5, 4, 2, 8, 7, 6, 5, 4, 2, 3, 1, 4, 3, 5, 6],\\
[1, 3, 1, 4, 5, 6, 5, 4, 3, 1],\,
[1, 3, 1, 4, 5, 6, 7, 6, 5, 4, 3, 1],\,
[1, 3, 4, 2, 3, 4, 5, 4, 2, 6, 7, 8, 7, 6, 5, 4, 2, 3, 1, 4],\\
[1, 3, 4, 2, 3, 4, 5, 4, 6, 5, 7, 6, 5, 4, 2, 8, 7, 6, 5, 4, 2, 3, 1, 4],\,
[1, 3, 4, 3, 5, 4, 3, 1],\\
[1, 3, 4, 3, 5, 4, 3, 6, 5, 4, 3, 1],\,
[1, 3, 4, 5, 4, 2, 3, 4, 5, 6, 5, 4, 2, 7, 6, 5, 4, 2, 3, 1, 4, 5],\\
[1, 3, 4, 5, 4, 2, 3, 4, 5, 6, 5, 4, 2, 7, 8, 7, 6, 5, 4, 2, 3, 1, 4, 5],\\
[1, 3, 4, 5, 4, 2, 3, 4, 5, 6, 5, 7, 6, 5, 4, 2, 8, 7, 6, 5, 4, 2, 3, 1, 4, 5],\,
[1, 3, 4, 5, 4, 3, 1, 7],\\
[1, 3, 4, 5, 4, 6, 5, 4, 7, 6, 5, 4, 3, 1],\,
[1, 3, 4, 5, 4, 6, 7, 6, 5, 4, 3, 1],\,
[1, 3, 4, 5, 4, 6, 7, 8, 7, 6, 5, 4, 3, 1],\,
[1, 3, 4, 5, 6, 5, 4, 3, 1, 8],\\
[1, 3, 4, 5, 6, 5, 7, 6, 5, 4, 3, 1],\,
[1, 3, 4, 5, 6, 7, 6, 8, 7, 6, 5, 4, 3, 1],\,
[1, 4, 2, 3, 1, 4, 3, 5, 4, 3, 1, 6, 7, 8, 7, 6, 5, 4, 2, 3, 1, 4],\\
[1, 4, 2, 3, 1, 4, 3, 5, 4, 6, 5, 7, 6, 5, 4, 3, 1, 8, 7, 6, 5, 4, 2, 3, 1, 4],\,
[1, 4, 2, 3, 1, 4, 5, 6, 7, 8, 7, 6, 5, 4, 2, 3, 1, 4],\\
[1, 4, 2, 3, 4, 5, 4, 2, 3, 4, 6, 5, 7, 6, 5, 4, 2, 3, 4, 8, 7, 6, 5, 4, 2, 3, 1, 4],\,
[1, 4, 2, 3, 4, 5, 6, 5, 7, 6, 5, 4, 2, 3, 1, 4],\\
[1, 4, 2, 3, 4, 5, 6, 5, 7, 6, 5, 8, 7, 6, 5, 4, 2, 3, 1, 4],\,
[1, 4, 5, 4, 2, 3, 4, 5, 6, 5, 4, 2, 3, 1, 4, 5],\\
[1, 4, 5, 4, 2, 3, 4, 5, 6, 7, 6, 5, 4, 2, 3, 1, 4, 5],\,
[1, 4, 5, 4, 2, 3, 4, 5, 6, 7, 8, 7, 6, 5, 4, 2, 3, 1, 4, 5],\\
[1, 4],\,
[1, 5, 4, 2, 3, 1, 4, 5, 6, 5, 4, 2, 3, 1, 4, 5],\,
[1, 5, 4, 2, 3, 1, 4, 5, 6, 7, 8, 7, 6, 5, 4, 2, 3, 1, 4, 5],\\
[1, 5, 4, 2, 3, 1, 4, 5],\,
[1, 5, 4, 2, 3, 4, 5, 6, 5, 4, 2, 3, 1, 4, 5, 8],\,
[1, 5, 4, 2, 3, 4, 5, 6, 7, 6, 8, 7, 6, 5, 4, 2, 3, 1, 4, 5],\\
[1, 5, 6, 5, 4, 2, 3, 4, 5, 6, 7, 6, 5, 4, 2, 3, 1, 4, 5, 6],\,
[1, 6, 5, 4, 2, 3, 1, 4, 5, 6],\,
[1, 6],\,
[1, 7, 6, 5, 4, 2, 3, 1, 4, 5, 6, 7],\,
[1, 8],\\
[2, 3, 1, 4, 2, 3, 4, 5, 4, 6, 5, 7, 6, 5, 4, 3, 8, 7, 6, 5, 4, 2, 3, 1, 4, 3],\\
[2, 3, 1, 4, 5, 6, 5, 4, 2, 3, 4, 5, 6, 7, 6, 5, 4, 3, 8, 7, 6, 5, 4, 2, 3, 1, 4, 3, 5, 6],\,
[2, 3, 1, 6, 5, 4, 2, 3, 1, 4, 3, 5, 6, 7, 6, 5, 4, 2, 3, 1, 4, 3, 5, 6],\\
[2, 3, 4, 2, 3, 1, 4, 7, 6, 5, 4, 2, 3, 4, 5, 6, 7, 8, 7, 6, 5, 4, 2, 3, 1, 4, 3, 5, 4, 2, 6, 7],\\
[2, 3, 4, 2, 3, 5, 4, 2, 3, 6, 7, 8, 7, 6, 5, 4, 2, 3],\,
[2, 3, 4, 2, 3, 5, 4, 6, 5, 4, 2, 3, 7, 6, 5, 4, 2, 3],\\
[2, 3, 4, 2, 3, 5, 4, 6, 5, 4, 2, 3, 7, 8, 7, 6, 5, 4, 2, 3],\,
[2, 3, 4, 2, 3, 5, 4, 6, 5, 7, 6, 5, 4, 2, 3, 8, 7, 6, 5, 4, 2, 3],\\
[2, 3, 4, 2, 5, 4, 2, 6, 5, 4, 2, 3],\,
[2, 3, 4, 2, 5, 4, 2, 6, 7, 8, 7, 6, 5, 4, 2, 3],\,
[2, 3, 4, 2, 5, 4, 6, 5, 4, 2, 7, 6, 5, 4, 2, 3],\\
[2, 3, 4, 2, 5, 4, 6, 5, 4, 2, 7, 8, 7, 6, 5, 4, 2, 3],\,
[2, 3, 4, 2, 5, 6, 5, 4, 2, 3],\,
[2, 3, 4, 2, 5, 6, 7, 6, 5, 4, 2, 3],\\
[2, 3, 4, 3, 5, 4, 6, 5, 4, 3, 7, 6, 5, 4, 2, 3],\,
[2, 3, 4, 3, 5, 6, 7, 6, 5, 4, 2, 3],\,
[2, 3, 4, 3, 5, 6, 7, 8, 7, 6, 5, 4, 2, 3],\\
[2, 3, 4, 5, 4, 6, 5, 4, 2, 3],\,
[2, 3, 4, 5, 4, 6, 5, 4, 7, 6, 5, 4, 2, 3],\,
[2, 3, 4, 5, 6, 5, 4, 2, 3, 8],\,
[2, 3, 4, 5, 6, 5, 7, 6, 5, 8, 7, 6, 5, 4, 2, 3],\\
[2, 3, 4, 5, 6, 5, 7, 8, 7, 6, 5, 4, 2, 3],\,
[2, 3, 4, 5, 6, 7, 6, 8, 7, 6, 5, 4, 2, 3],\,
[2, 4, 2, 3, 4, 5, 4, 3, 6, 5, 4, 2, 3, 4],\\
[2, 4, 2, 3, 4, 5, 4, 3, 6, 7, 6, 5, 4, 2, 3, 4],\,
[2, 4, 2, 3, 4, 5, 4, 3, 6, 7, 8, 7, 6, 5, 4, 2, 3, 4],\,
[2, 4, 2, 3, 4, 5, 4, 6, 5, 4, 3, 7, 6, 5, 4, 2, 3, 4],\\
[2, 4, 2, 5, 4, 2, 6, 5, 4, 2],\,
[2, 4, 2, 5, 6, 5, 4, 2],\,
[2, 4, 2, 5, 6, 7, 6, 5, 4, 2],\,
[2, 4, 2, 6],\,
[2, 4, 2, 8],\\
[2, 4, 3, 1, 6, 7, 6, 5, 4, 2, 3, 4, 5, 6, 7, 8, 7, 6, 5, 4, 2, 3, 1, 4, 3, 5, 4, 2, 6, 7],\,
[2, 4, 5, 4, 2, 7, 8, 7],\\
[2, 4, 5, 4, 2, 7],\,
[2, 4, 5, 4, 2, 8],\,
[2, 4, 5, 4, 6, 5, 4, 2],\,
[2, 4, 5, 6, 5, 7, 6, 5, 4, 2],\,
[3, 1, 4, 2, 3, 4, 5, 6, 5, 7, 8, 7, 6, 5, 4, 2, 3, 1, 4, 3],\\
[3, 1, 4, 5, 4, 2, 3, 4, 5, 6, 7, 8, 7, 6, 5, 4, 2, 3, 1, 4, 3, 5],\,
[3, 1, 4, 5, 4, 6, 5, 4, 2, 3, 4, 5, 6, 7, 8, 7, 6, 5, 4, 2, 3, 1, 4, 3, 5, 6],\\
[3, 1, 4, 6, 5, 4, 2, 3, 4, 5, 6, 7, 8, 7, 6, 5, 4, 2, 3, 1, 4, 3, 5, 6],\,
[3, 1, 4, 7, 6, 5, 4, 2, 3, 4, 5, 6, 7, 8, 7, 6, 5, 4, 2, 3, 1, 4, 3, 5, 6, 7],\\
[3, 1, 5, 4, 2, 3, 4, 5, 6, 7, 6, 8, 7, 6, 5, 4, 2, 3, 1, 4, 3, 5],\,
[3, 1, 5, 6, 5, 4, 2, 3, 4, 5, 6, 7, 6, 5, 4, 2, 3, 1, 4, 3, 5, 6],\\
[3, 1, 5, 6, 5, 4, 2, 3, 4, 5, 6, 7, 8, 7, 6, 5, 4, 2, 3, 1, 4, 3, 5, 6],\,
[3, 1, 6, 5, 4, 2, 3, 1, 4, 3, 5, 6],\\
[3, 1, 6, 7, 6, 5, 4, 2, 3, 4, 5, 6, 7, 8, 7, 6, 5, 4, 2, 3, 1, 4, 3, 5, 6, 7],\,
[3, 1, 7, 6, 5, 4, 2, 3, 1, 4, 3, 5, 6, 7],\\
[3, 1, 8, 7, 6, 5, 4, 2, 3, 1, 4, 3, 5, 6, 7, 8],\,
[3, 4, 2, 3, 1, 5, 4, 2, 3, 1, 4, 8, 7, 6, 5, 4, 2, 3, 1, 4, 3, 5, 4, 2, 6, 5, 4, 3, 7, 8],\\
[3, 4, 2, 3, 4, 5, 4, 2, 6, 5, 4, 2, 3, 4],\,
[3, 4, 2, 3, 4, 5, 4, 2, 6, 7, 8, 7, 6, 5, 4, 2, 3, 4],\,
[3, 4, 3, 5, 4, 3],\,
[3, 4, 3, 6, 7, 6],\\
[3, 4, 3, 6],\,
[3, 4, 3, 7],\,
[3, 4, 5, 4, 3, 7, 8, 7],\,
[3, 4, 5, 4, 6, 5, 4, 3],\,
[3, 4, 5, 6, 5, 4, 3, 8],\,
[3, 5],\,
[3, 7],\\
[4, 2, 3, 1, 5, 4, 2, 3, 1, 4, 3, 5, 6, 7, 8, 7, 6, 5, 4, 2, 3, 1, 4, 3, 5, 4],\,
[4, 2, 3, 4, 5, 4, 2, 3, 4, 7],\,
[4, 2, 3, 4, 5, 6, 5, 4, 2, 3, 4, 8],\\
[4, 2, 3, 4, 5, 6, 5, 7, 6, 5, 4, 2, 3, 4],\,
[4, 2, 3, 4, 5, 6, 7, 6, 8, 7, 6, 5, 4, 2, 3, 4],\,
[4, 2, 3, 4],\\
[4, 3, 1, 5, 4, 2, 3, 4, 5, 6, 7, 6, 8, 7, 6, 5, 4, 2, 3, 1, 4, 3, 5, 4],\,
[4, 5, 4, 2, 3, 4, 5, 6, 5, 4, 2, 3, 4, 5],\,
[4, 5, 4, 8],\,
[4, 7, 8, 7],\\
[4, 7],\,
[4, 8],\,
[5, 4, 2, 3, 4, 5, 6, 5, 4, 2, 3, 4, 5, 8],\,
[5, 4, 2, 3, 4, 5],\\
[5, 4, 3, 1, 6, 5, 4, 2, 3, 1, 4, 3, 5, 4, 6, 5],\,
[5, 8],\,
[6, 5, 4, 2, 3, 4, 5, 6]
\big\},
\end{multline*}}%
one obtains 100 solutions for $w_i$ in 
\eqref{eq:E8''''D} with $\ell_T(w_i)=2$ and $w_i$ of type $A_2$:
{\allowdisplaybreaks\tiny
\begin{multline*}
w_i\in\big\{[1, 2, 3, 1, 4, 5, 4, 6, 7, 8, 7, 6, 5, 4, 2, 3, 1, 4, 3, 5],\,
[1, 2, 3, 1, 4, 5, 6, 5, 4, 7, 6, 5, 4, 2, 3, 1, 4, 3, 5, 6],\\
[1, 2, 3, 1, 4, 5, 6, 5, 4, 7, 8, 7, 6, 5, 4, 2, 3, 1, 4, 3, 5, 6],\,
[1, 2, 3, 1, 4, 5, 6, 7, 8, 7, 6, 5, 4, 2, 3, 1, 4, 3],\\
[1, 2, 3, 4, 3, 1, 5, 6, 7, 6, 5, 4, 3, 8, 7, 6, 5, 4, 2, 3, 1, 4, 3, 5, 4, 2, 6, 7],\,
[1, 2, 3, 4, 5, 4, 6, 7, 6, 5, 4, 2, 3, 1, 4, 5],\\
[1, 2, 3, 4, 5, 4, 6, 7, 8, 7, 6, 5, 4, 2, 3, 1, 4, 5],\,
[1, 2, 3, 4, 5, 6, 7, 6, 5, 4, 2, 3, 1, 4],\\
[1, 2, 3, 4, 5, 6, 7, 6, 5, 4, 2, 3, 1, 8],\,
[1, 2, 3, 4, 5, 6, 7, 6, 5, 4, 2, 3],\,
[1, 2, 3, 4, 5, 6, 7, 8, 7, 6, 5, 4, 2, 3, 1, 4],\\
[1, 2, 3, 4, 5, 6, 7, 8, 7, 6, 5, 4, 2, 3],\,
[1, 2, 4, 2, 3, 1, 4, 5, 6, 7, 6, 5, 4, 2, 3, 1, 4, 3, 5, 6, 7, 8, 7, 6, 5, 4, 2, 3],\\
[1, 2, 4, 2, 3, 4, 5, 6, 7, 6, 5, 4, 3, 1],\,
[1, 2, 4, 2, 3, 4, 5, 6, 7, 8, 7, 6, 5, 4, 3, 1],\,
[1, 2, 4, 5, 4, 2, 3, 4, 5, 6, 5, 4, 3, 1],\\
[1, 2, 4, 5, 4, 2, 3, 4, 5, 6, 7, 6, 5, 4, 3, 1],\,
[1, 2, 4, 5, 4, 2, 3, 4, 5, 6, 7, 8, 7, 6, 5, 4, 3, 1],\,
[1, 2, 4, 5, 6, 5, 4, 2, 3, 4, 5, 6, 7, 6, 5, 4, 3, 1],\\
[1, 3, 1, 4, 5, 4, 2, 3, 1, 4, 5, 6, 7, 8, 7, 6, 5, 4, 2, 3],\,
[1, 3, 1, 4, 5, 6, 5, 4, 2, 3, 1, 4, 5, 6, 7, 6, 5, 4, 2, 3],\\
[1, 3, 1, 4, 5, 6, 5, 4, 2, 3, 1, 4, 5, 6, 7, 8, 7, 6, 5, 4, 2, 3],\,
[1, 3, 1, 4, 5, 6, 7, 6, 5, 4, 2, 3, 1, 4, 5, 6, 7, 8, 7, 6, 5, 4, 2, 3],\\
[1, 3, 4, 3, 1, 5, 4, 2, 3, 1, 4, 5, 6, 7, 8, 7, 6, 5, 4, 2, 3, 4],\,
[1, 3, 4, 5, 4, 2, 3, 1, 4, 5, 6, 5, 4, 2],\\
[1, 3, 4, 5, 4, 2, 3, 1, 4, 5, 6, 7, 6, 5, 4, 2],\,
[1, 3, 4, 5, 4, 3, 1, 6, 7, 6],\,
[1, 3, 4, 5, 4, 3, 1, 6],\,
[1, 3, 4, 5, 4, 3],\\
[1, 3, 4, 5, 6, 5, 4, 2, 3, 1, 4, 5, 6, 7, 6, 5, 4, 2],\,
[1, 3, 4, 5, 6, 5, 4, 3, 1, 7, 8, 7],\,
[1, 3, 4, 5, 6, 5, 4, 3, 1, 7],\\
[1, 3, 4, 5, 6, 5, 4, 3],\,
[1, 3, 4, 5, 6, 7, 6, 5, 4, 3, 1, 8],\,
[1, 4, 2, 3, 1, 4, 5, 4, 6, 7, 8, 7, 6, 5, 4, 2, 3, 1, 4, 3, 5, 4],\\
[1, 4, 2, 3, 4, 5, 6, 5, 7, 6, 5, 4, 2, 3, 1, 4, 5, 6],\,
[1, 4, 2, 3, 4, 5, 6, 7, 6, 5, 4, 2, 3, 1, 4, 5],\,
[1, 4, 2, 3, 4, 5, 6, 7, 6, 5, 4, 2, 3, 1, 4, 8],\\
[1, 4, 2, 3, 4, 5, 6, 7, 6, 5, 4, 2, 3, 4],\,
[1, 4, 2, 3, 4, 5, 6, 7, 8, 7, 6, 5, 4, 2, 3, 1, 4, 5],\,
[1, 4, 2, 3, 4, 5, 6, 7, 8, 7, 6, 5, 4, 2, 3, 4],\\
[1, 5, 4, 2, 3, 4, 5, 6, 5, 4, 2, 3, 1, 4, 5, 7, 8, 7],\,
[1, 5, 4, 2, 3, 4, 5, 6, 5, 4, 2, 3, 1, 4, 5, 7],\,
[1, 5, 4, 2, 3, 4, 5, 6, 5, 4, 2, 3, 4, 5],\\
[1, 5, 4, 2, 3, 4, 5, 6, 7, 6, 5, 4, 2, 3, 1, 4, 5, 6],\,
[1, 5, 4, 2, 3, 4, 5, 6, 7, 6, 5, 4, 2, 3, 1, 4, 5, 8],\\
[2, 3, 1, 4, 2, 3, 1, 4, 5, 6, 7, 6, 5, 8, 7, 6, 5, 4, 2, 3, 1, 4, 3, 5, 4, 2, 6, 7],\,
[2, 3, 1, 4, 5, 6, 5, 4, 2, 3, 1, 4, 5, 6, 7, 6, 5, 4, 3, 1],\\
[2, 3, 1, 4, 5, 6, 5, 4, 2, 3, 1, 4, 5, 6, 7, 8, 7, 6, 5, 4, 3, 1],\,
[2, 3, 1, 4, 5, 6, 7, 6, 5, 4, 2, 3, 1, 4, 5, 6, 7, 8, 7, 6, 5, 4, 3, 1],\\
[2, 3, 4, 3, 1, 5, 6, 7, 6, 5, 4, 2, 3, 1, 4, 3, 5, 6, 7, 8, 7, 6, 5, 4, 2, 3, 1, 4],\,
[2, 3, 4, 5, 4, 6, 5, 4, 2, 3, 4, 5],\\
[2, 3, 4, 5, 6, 5, 4, 2, 3, 4],\,
[2, 3, 4, 5, 6, 5, 4, 2, 3, 7, 8, 7],\,
[2, 3, 4, 5, 6, 5, 4, 2, 3, 7],\\
[2, 3, 4, 5, 6, 5, 4, 2],\,
[2, 3, 4, 5, 6, 7, 6, 5, 4, 2, 3, 4],\,
[2, 3, 4, 5, 6, 7, 6, 5, 4, 2, 3, 8],\,
[2, 3, 4, 5, 6, 7, 6, 5, 4, 2],\\
[2, 3, 4, 5, 6, 7, 8, 7, 6, 5, 4, 2, 3, 4],\,
[2, 4, 2, 3, 1, 7, 6, 5, 4, 2, 3, 4, 5, 6, 7, 8, 7, 6, 5, 4, 2, 3, 1, 4, 3, 5, 6, 7],\\
[2, 4, 2, 3, 4, 5, 4, 3],\,
[2, 4, 2, 3, 4, 5, 6, 5, 4, 3],\,
[2, 4, 2, 5, 6, 5],\,
[2, 4, 2, 5],\,
[2, 4, 5, 4, 2, 3, 4, 5, 6, 5, 4, 3],\\
[2, 4, 5, 4, 2, 6, 7, 6],\,
[2, 4, 5, 4, 2, 6],\,
[2, 4, 5, 6, 5, 4, 2, 7],\,
[3, 1, 4, 2, 3, 4, 5, 6, 5, 7, 8, 7, 6, 5, 4, 2, 3, 1, 4, 3, 5, 6],\\
[3, 1, 4, 2, 3, 4, 5, 6, 7, 8, 7, 6, 5, 4, 2, 3, 1, 4, 3, 5],\,
[3, 1, 4, 2, 3, 4, 5, 6, 7, 8, 7, 6, 5, 4, 2, 3, 1, 4],\\
[3, 1, 5, 4, 2, 3, 4, 5, 6, 7, 6, 8, 7, 6, 5, 4, 2, 3, 1, 4, 3, 5, 6, 7],\,
[3, 1, 5, 4, 2, 3, 4, 5, 6, 7, 8, 7, 6, 5, 4, 2, 3, 1, 4, 3, 5, 4],\\
[3, 1, 5, 4, 2, 3, 4, 5, 6, 7, 8, 7, 6, 5, 4, 2, 3, 1, 4, 3, 5, 6],\,
[3, 1, 5, 4, 2, 3, 4, 5, 6, 7, 8, 7, 6, 5, 4, 2, 3, 1, 4, 5],\\
[3, 1, 6, 5, 4, 2, 3, 4, 5, 6, 7, 6, 5, 4, 2, 3, 1, 4, 3, 5, 6, 8],\,
[3, 1, 6, 5, 4, 2, 3, 4, 5, 6, 7, 6, 5, 4, 2, 3, 1, 4, 5, 6],\\
[3, 1, 6, 5, 4, 2, 3, 4, 5, 6, 7, 8, 7, 6, 5, 4, 2, 3, 1, 4, 3, 5, 6, 7],\,
[3, 4, 2, 3, 4, 5, 4, 2],\,
[3, 4, 2, 3, 4, 5, 6, 5, 4, 2],\\
[3, 4, 2, 3, 4, 5, 6, 7, 6, 5, 4, 2],\,
[3, 4, 3, 1, 5, 4, 2, 3, 4, 5, 6, 7, 8, 7, 6, 5, 4, 2, 3, 1, 4, 5],\,
[3, 4, 3, 5, 6, 5],\\
[3, 4, 3, 5],\,
[3, 4, 5, 4, 2, 3, 4, 5, 6, 5, 4, 2],\,
[3, 4, 5, 4, 3, 6],\,
[3, 4, 5, 4],\,
[3, 4],\\
[4, 2, 3, 1, 4, 5, 4, 2, 3, 1, 4, 3, 5, 6, 7, 8, 7, 6, 5, 4, 3, 1],\,
[4, 2, 3, 4, 5, 4, 2, 3, 4, 6, 7, 6],\,
[4, 2, 3, 4, 5, 4, 2, 3, 4, 6],\\
[4, 2, 3, 4, 5, 6, 5, 4, 2, 3, 4, 5],\,
[4, 2, 3, 4, 5, 6, 5, 4, 2, 3, 4, 7, 8, 7],\,
[4, 2, 3, 4, 5, 6, 5, 4, 2, 3, 4, 7],\\
[4, 2, 3, 4, 5, 6, 7, 6, 5, 4, 2, 3, 4, 8],\,
[4, 5],\,
[5, 6],\,
[6, 7],\,
[7, 8]
\big\},
\end{multline*}}%
one obtains 75 solutions for $w_i$ in 
\eqref{eq:E8''''D} with $\ell_T(w_i)=3$ and $w_i$ of type $A_1^3$:
{\allowdisplaybreaks\tiny
\begin{multline*}
w_i\in\big\{ [1, 2, 3, 1, 4, 2, 5, 4, 6, 5, 4, 2, 7, 8, 7, 6, 5, 4, 2, 3, 1],\,
[1, 2, 3, 1, 4, 3, 5, 4, 3, 6, 5, 4, 3, 7, 6, 5, 4, 3, 1, 8, 7, 6, 5, 4, 2, 3, 1],\\
[1, 2, 3, 1, 4, 5, 4, 6, 5, 4, 2, 3, 4, 5, 7, 6, 5, 4, 2, 3, 1, 4, 5, 8, 7, 6, 5, 4, 2, 3, 1],\,
[1, 2, 3, 1, 4, 5, 4, 6, 7, 8, 7, 6, 5, 4, 2, 3, 1],\\
[1, 2, 3, 4, 2, 3, 1, 6, 7, 6, 5, 4, 2, 3, 1, 4, 3, 5, 4, 2, 6, 7, 8, 7, 6, 5, 4, 2, 3, 1, 4, 3, 5, 4, 2, 6, 7],\\
[1, 2, 3, 4, 2, 3, 5, 4, 6, 5, 4, 7, 6, 5, 4, 2, 3, 8, 7, 6, 5, 4, 2, 3, 1],\\
[1, 2, 3, 4, 2, 5, 4, 2, 3, 4, 6, 5, 4, 3, 7, 6, 5, 4, 2, 3, 4, 8, 7, 6, 5, 4, 2, 3, 1],\\
[1, 2, 3, 4, 3, 1, 5, 6, 7, 6, 5, 4, 2, 3, 1, 4, 3, 5, 6, 8, 7, 6, 5, 4, 2, 3, 1],\\
[1, 2, 3, 4, 3, 5, 4, 3, 6, 5, 4, 3, 7, 6, 5, 4, 2, 3, 1],\,
[1, 2, 3, 4, 3, 5, 4, 6, 5, 4, 7, 6, 5, 4, 3, 8, 7, 6, 5, 4, 2, 3, 1],\\
[1, 2, 3, 4, 5, 4, 6, 5, 4, 2, 3, 4, 5, 7, 6, 5, 4, 2, 3, 4, 5, 8, 7, 6, 5, 4, 2, 3, 1],\,
[1, 2, 3, 4, 5, 6, 5, 4, 2, 3, 4, 5, 7, 6, 5, 4, 2, 3, 1],\\
[1, 2, 3, 4, 5, 6, 5, 4, 3, 1, 7, 6, 5, 4, 2, 3, 1, 4, 3, 5, 4, 6, 5, 8, 7, 6, 5, 4, 2, 3, 1],\\
[1, 2, 4, 2, 3, 1, 4, 3, 5, 4, 3, 1, 6, 5, 7, 6, 5, 4, 3, 1, 8, 7, 6, 5, 4, 2, 3, 1, 4],\\
[1, 2, 4, 2, 3, 1, 5, 6, 5, 7, 6, 5, 4, 2, 3, 1, 4, 3, 5, 4, 2, 6, 7],\,
[1, 2, 4, 2, 3, 1, 6, 7, 6, 8, 7, 6, 5, 4, 2, 3, 1, 4, 3, 5, 4, 2, 6, 7, 8],\\
[1, 2, 4, 2, 5, 4, 2, 3, 4, 5, 6, 5, 4, 3, 7, 6, 5, 4, 2, 3, 1, 4, 5],\,
[1, 2, 4, 2, 5, 6, 5, 4, 2],\,
[1, 2, 4, 5, 4, 2, 8],\\
[1, 3, 1, 4, 2, 3, 4, 5, 4, 6, 5, 4, 7, 6, 5, 4, 2, 8, 7, 6, 5, 4, 2, 3, 1, 4, 3],\\
[1, 3, 1, 4, 2, 3, 5, 4, 2, 3, 1, 8, 7, 6, 5, 4, 2, 3, 1, 4, 3, 5, 4, 2, 6, 5, 4, 3, 1, 7, 8],\\
[1, 3, 1, 4, 2, 3, 7, 6, 5, 4, 2, 3, 1, 8, 7, 6, 5, 4, 2, 3, 1, 4, 3, 5, 4, 2, 6, 5, 4, 3, 1, 7, 6, 8, 7],\\
[1, 3, 1, 4, 2, 5, 4, 2, 3, 1, 6, 7, 6, 8, 7, 6, 5, 4, 2, 3, 1, 4, 3, 5, 4, 2, 6, 5, 4, 3, 1, 7, 8],\\
[1, 3, 1, 4, 3, 1, 5, 6, 7, 6, 8, 7, 6, 5, 4, 3, 1],\,
[1, 3, 1, 4, 5, 4, 6, 5, 4, 2, 3, 4, 5, 6, 7, 6, 5, 4, 2, 8, 7, 6, 5, 4, 2, 3, 1, 4, 3, 5, 6],\\
[1, 3, 1, 4, 5, 4, 6, 7, 6, 5, 4, 3, 1],\,
[1, 3, 4, 3, 5, 4, 3, 1, 7],\,
[1, 3, 4, 3, 5, 4, 3, 6, 5, 4, 3, 1, 8],\\
[1, 3, 4, 5, 4, 2, 3, 1, 4, 5, 6, 5, 4, 2, 7, 8, 7, 6, 5, 4, 2, 3, 1, 4, 5],\,
[1, 4, 2, 3, 1, 4, 5, 6, 5, 7, 6, 5, 8, 7, 6, 5, 4, 2, 3, 1, 4],\\
[1, 4, 2, 3, 4, 5, 4, 2, 3, 4, 6, 7, 8, 7, 6, 5, 4, 2, 3, 1, 4],\,
[1, 4, 2, 3, 4, 5, 6, 7, 6, 5, 4, 2, 3, 4, 5, 6, 8, 7, 6, 5, 4, 2, 3, 1, 4],\\
[1, 4, 5, 4, 2, 3, 1, 4, 5, 6, 5, 4, 2, 3, 1, 4, 5],\,
[1, 4, 5, 4, 2, 3, 1, 4, 5, 6, 7, 8, 7, 6, 5, 4, 2, 3, 1, 4, 5],\\
[1, 4, 5, 4, 2, 3, 4, 5, 6, 5, 4, 2, 3, 1, 4, 5, 8],\,
[1, 4, 5, 4, 2, 3, 4, 5, 6, 7, 6, 8, 7, 6, 5, 4, 2, 3, 1, 4, 5],\\
[1, 4, 5, 4, 6, 5, 4, 2, 3, 1, 4, 5, 6],\,
[1, 4, 8],\,
[1, 5, 4, 2, 3, 1, 4, 5, 7],\\
[1, 5, 4, 2, 3, 4, 5, 6, 5, 4, 2, 3, 4, 5, 7, 6, 5, 4, 2, 3, 1, 4, 5],\,
[1, 6, 7, 6, 5, 4, 2, 3, 1, 4, 5, 6, 7],\\
[2, 3, 1, 4, 2, 5, 6, 5, 4, 2, 3, 4, 5, 6, 7, 6, 5, 4, 3, 8, 7, 6, 5, 4, 2, 3, 1, 4, 3, 5, 6],\\
[2, 3, 1, 5, 6, 5, 4, 2, 3, 1, 4, 3, 5, 6, 7, 6, 5, 4, 2, 3, 1, 4, 3, 5, 6],\\
[2, 3, 4, 2, 3, 1, 4, 6, 7, 6, 5, 4, 2, 3, 4, 5, 6, 7, 8, 7, 6, 5, 4, 2, 3, 1, 4, 3, 5, 4, 2, 6, 7],\,
[2, 3, 4, 2, 3, 5, 4, 2, 3, 6, 7, 6, 8, 7, 6, 5, 4, 2, 3],\\
[2, 3, 4, 2, 3, 5, 4, 3, 6, 5, 4, 2, 3, 7, 8, 7, 6, 5, 4, 2, 3],\,
[2, 3, 4, 2, 3, 5, 4, 6, 5, 4, 7, 6, 5, 4, 2, 3, 8, 7, 6, 5, 4, 2, 3],\\
[2, 3, 4, 2, 3, 5, 6, 7, 6, 5, 4, 2, 3],\,
[2, 3, 4, 3, 5, 6, 5, 7, 8, 7, 6, 5, 4, 2, 3],\,
[2, 3, 4, 5, 4, 6, 5, 4, 2, 3, 8],\\
[2, 3, 4, 5, 6, 5, 4, 2, 3, 4, 5, 7, 6, 5, 4, 2, 3],\,
[2, 4, 2, 3, 1, 4, 5, 4, 3, 8, 7, 6, 5, 4, 2, 3, 1, 4, 3, 5, 4, 2, 6, 7, 8],\\
[2, 4, 2, 3, 1, 4, 5, 6, 7, 6, 5, 4, 3, 8, 7, 6, 5, 4, 2, 3, 1, 4, 3, 5, 4, 2, 6, 7, 8],\,
[2, 4, 2, 3, 4, 5, 4, 3, 6, 5, 4, 2, 3, 4, 8],\\
[2, 4, 2, 3, 4, 5, 4, 3, 6, 7, 6, 8, 7, 6, 5, 4, 2, 3, 4],\,
[3, 1, 4, 2, 3, 4, 5, 6, 7, 6, 5, 4, 2, 3, 4, 5, 6, 8, 7, 6, 5, 4, 2, 3, 1, 4, 3],\\
[3, 1, 4, 5, 4, 2, 3, 4, 5, 6, 7, 6, 8, 7, 6, 5, 4, 2, 3, 1, 4, 3, 5],\,
[3, 1, 4, 5, 4, 6, 5, 4, 2, 3, 1, 4, 3, 5, 6],\\
[3, 1, 4, 5, 4, 8, 7, 6, 5, 4, 2, 3, 1, 4, 3, 5, 6, 7, 8],\,
[3, 1, 4, 6, 7, 6, 5, 4, 2, 3, 4, 5, 6, 7, 8, 7, 6, 5, 4, 2, 3, 1, 4, 3, 5, 6, 7],\\
[3, 1, 5, 6, 5, 7, 6, 5, 4, 2, 3, 1, 4, 3, 5, 6, 7],\,
[3, 1, 6, 5, 4, 2, 3, 1, 4, 3, 5, 6, 8],\\
[3, 1, 6, 5, 4, 2, 3, 4, 5, 6, 7, 6, 5, 4, 2, 3, 4, 5, 6, 8, 7, 6, 5, 4, 2, 3, 1, 4, 3, 5, 6],\,
[3, 1, 7, 8, 7, 6, 5, 4, 2, 3, 1, 4, 3, 5, 6, 7, 8],\\
[3, 4, 2, 3, 1, 5, 4, 2, 3, 1, 4, 7, 8, 7, 6, 5, 4, 2, 3, 1, 4, 3, 5, 4, 2, 6, 5, 4, 3, 7, 8],\,
[3, 4, 3, 5, 4, 3, 7, 8, 7],\\
[3, 4, 5, 4, 6, 5, 4, 3, 8],\,
[4, 2, 3, 1, 5, 4, 2, 3, 1, 4, 3, 5, 6, 7, 6, 8, 7, 6, 5, 4, 2, 3, 1, 4, 3, 5, 4],\\
[4, 2, 3, 4, 5, 4, 2, 3, 4, 6, 5, 4, 2, 3, 4],\,
[4, 2, 3, 4, 5, 4, 2, 3, 4, 6, 7, 8, 7, 6, 5, 4, 2, 3, 4],\,
[4, 2, 3, 4, 6],\\
[4, 5, 4, 2, 3, 4, 5, 6, 5, 4, 2, 3, 4, 5, 8],\,
[5, 4, 2, 3, 4, 5, 7, 8, 7],\,
[5, 4, 3, 1, 6, 5, 4, 2, 3, 1, 4, 3, 5, 4, 6, 5, 8],\,
[5, 6, 5, 4, 2, 3, 4, 5, 6]
\big\},
\end{multline*}}%
one obtains 165 solutions for $w_i$ in 
\eqref{eq:E8''''D} with $\ell_T(w_i)=3$ and $w_i$ of type $A_1*A_2$:
{\allowdisplaybreaks\tiny
\begin{multline*}
w_i\in\big\{
 [1, 3, 1, 4, 5, 4, 2, 3, 4, 5, 6, 5, 4, 2, 7, 8, 7, 6, 5, 4, 2, 3, 1, 4, 5],\\
[1, 2, 3, 1, 4, 2, 5, 4, 2, 6, 5, 4, 2, 7, 8, 7, 6, 5, 4, 2, 3, 1, 4, 3, 5],\\
[1, 2, 3, 1, 4, 2, 5, 4, 6, 5, 4, 2, 3, 4, 5, 7, 6, 5, 4, 3, 1, 8, 7, 6, 5, 4, 2, 3, 1],\\
[1, 2, 3, 1, 4, 3, 5, 4, 6, 5, 4, 3, 7, 6, 5, 4, 3, 8, 7, 6, 5, 4, 2, 3, 1, 4, 3, 5, 6],\\
[1, 2, 3, 1, 4, 3, 5, 4, 6, 5, 7, 6, 5, 4, 3, 8, 7, 6, 5, 4, 2, 3, 1, 4, 3],\,
[1, 2, 3, 1, 4, 3, 5, 4, 6, 5, 7, 6, 5, 4, 3, 8, 7, 6, 5, 4, 2, 3, 1],\\
[1, 2, 3, 1, 4, 5, 4, 6, 5, 4, 7, 6, 5, 4, 2, 3, 1, 4, 3, 5, 6],\,
[1, 2, 3, 1, 4, 5, 4, 6, 5, 4, 7, 8, 7, 6, 5, 4, 2, 3, 1, 4, 3, 5, 6],\\
[1, 2, 3, 1, 4, 5, 6, 5, 4, 2, 3, 4, 5, 7, 6, 5, 4, 2, 3, 4, 5, 8, 7, 6, 5, 4, 2, 3, 1],\,
[1, 2, 3, 1, 4, 5, 6, 5, 7, 8, 7, 6, 5, 4, 2, 3, 1, 4, 3],\\
[1, 2, 3, 4, 2, 3, 5, 4, 6, 5, 4, 7, 6, 5, 4, 2, 3, 4, 5, 8, 7, 6, 5, 4, 2, 3, 1],\,
[1, 2, 3, 4, 2, 3, 5, 4, 6, 5, 7, 6, 5, 4, 2, 3, 4, 8, 7, 6, 5, 4, 2, 3, 1],\\
[1, 2, 3, 4, 2, 5, 4, 2, 3, 4, 6, 5, 7, 6, 5, 4, 3, 8, 7, 6, 5, 4, 2, 3, 1],\,
[1, 2, 3, 4, 2, 5, 4, 2, 6, 5, 4, 2, 7, 6, 5, 4, 2, 3, 1, 4, 5],\\
[1, 2, 3, 4, 2, 5, 4, 2, 6, 5, 4, 2, 7, 8, 7, 6, 5, 4, 2, 3, 1, 4, 5],\,
[1, 2, 3, 4, 2, 5, 4, 2, 6, 5, 7, 8, 7, 6, 5, 4, 2, 3, 1],\\
[1, 2, 3, 4, 2, 5, 4, 2, 6, 7, 8, 7, 6, 5, 4, 2, 3, 1, 4],\,
[1, 2, 3, 4, 2, 5, 4, 6, 5, 4, 2, 3, 4, 5, 7, 6, 5, 4, 3, 8, 7, 6, 5, 4, 2, 3, 1],\\
[1, 2, 3, 4, 2, 5, 4, 6, 5, 4, 2, 7, 6, 5, 4, 2, 3, 1, 8],\,
[1, 2, 3, 4, 3, 1, 5, 6, 5, 7, 6, 5, 4, 3, 8, 7, 6, 5, 4, 2, 3, 1, 4, 3, 5, 4, 2, 6, 7],\\
[1, 2, 3, 4, 3, 5, 4, 3, 6, 5, 4, 3, 7, 6, 5, 4, 2, 3, 1, 4, 5],\,
[1, 2, 3, 4, 3, 5, 4, 6, 5, 4, 3, 7, 6, 5, 4, 2, 3, 1, 4],\\
[1, 2, 3, 4, 3, 5, 4, 6, 5, 4, 3, 7, 6, 5, 4, 2, 3],\,
[1, 2, 3, 4, 3, 5, 4, 6, 5, 7, 6, 5, 4, 3, 8, 7, 6, 5, 4, 2, 3, 1, 4],\\
[1, 2, 3, 4, 5, 4, 2, 3, 4, 6, 5, 7, 6, 5, 4, 2, 3, 4, 5, 8, 7, 6, 5, 4, 2, 3, 1],\\
[1, 2, 3, 4, 5, 4, 2, 3, 4, 6, 5, 7, 6, 5, 4, 2, 3, 4, 8, 7, 6, 5, 4, 2, 3, 1, 4],\\
[1, 2, 3, 4, 5, 4, 6, 5, 7, 8, 7, 6, 5, 4, 2, 3, 1],\,
[1, 2, 3, 4, 5, 4, 6, 7, 6, 5, 4, 2, 3, 1, 8],\\
[1, 2, 3, 4, 5, 6, 5, 4, 3, 1, 7, 6, 5, 4, 2, 3, 1, 4, 3, 5, 6],\,
[1, 2, 3, 4, 5, 6, 5, 7, 8, 7, 6, 5, 4, 2, 3],\\
[1, 2, 4, 2, 3, 1, 4, 3, 8, 7, 6, 5, 4, 2, 3, 1, 4, 3, 5, 4, 2, 6, 5, 4, 3, 7, 8],\\
[1, 2, 4, 2, 3, 1, 4, 5, 6, 5, 7, 6, 5, 4, 2, 3, 1, 4, 3, 5, 6, 7, 8, 7, 6, 5, 4, 2, 3],\\
[1, 2, 4, 2, 3, 1, 6, 7, 6, 5, 4, 2, 3, 1, 4, 3, 5, 4, 2, 6, 5, 7, 6],\,
[1, 2, 4, 2, 3, 4, 5, 4, 2, 3, 4, 6, 5, 7, 6, 5, 4, 2, 8, 7, 6, 5, 4, 2, 3, 1, 4],\\
[1, 2, 4, 2, 3, 4, 5, 4, 3, 6, 5, 7, 6, 5, 4, 2, 3, 1, 4],\,
[1, 2, 4, 2, 3, 4, 5, 4, 3, 6, 5, 7, 6, 5, 8, 7, 6, 5, 4, 2, 3, 1, 4],\\
[1, 2, 4, 2, 3, 4, 5, 4, 3, 6, 7, 6, 5, 4, 2, 3, 1, 4, 8],\,
[1, 2, 4, 2, 3, 4, 5, 4, 3, 6, 7, 6, 5, 4, 2, 3, 4],\\
[1, 2, 4, 2, 3, 4, 5, 4, 3, 6, 7, 8, 7, 6, 5, 4, 2, 3, 4],\,
[1, 2, 4, 2, 3, 4, 5, 4, 6, 5, 4, 3, 7, 6, 5, 4, 2, 3, 1, 4, 5],\\
[1, 2, 4, 2, 3, 4, 5, 4, 6, 5, 4, 3, 7, 6, 5, 4, 2, 3, 4],\,
[1, 2, 4, 2, 3, 4, 5, 6, 5, 7, 6, 5, 4, 3, 1],\,
[1, 2, 4, 5, 4, 2, 3, 4, 5, 6, 5, 4, 3, 1, 8],\\
[1, 2, 4, 5, 4, 2, 3, 4, 5, 6, 7, 6, 8, 7, 6, 5, 4, 3, 1],\,
[1, 2, 4, 5, 4, 2, 6],\,
[1, 2, 4, 5, 4, 6, 5, 4, 2, 3, 4, 5, 6, 7, 6, 5, 4, 3, 1],\\
[1, 3, 1, 4, 2, 3, 4, 5, 4, 6, 5, 4, 7, 6, 5, 4, 2, 8, 7, 6, 5, 4, 2, 3, 1, 4, 3, 5, 6],\\
[1, 3, 1, 4, 2, 3, 4, 5, 4, 6, 5, 7, 6, 5, 4, 2, 8, 7, 6, 5, 4, 2, 3, 1, 4, 3, 5],\\
[1, 3, 1, 4, 2, 3, 4, 5, 4, 6, 5, 7, 6, 5, 4, 2, 8, 7, 6, 5, 4, 2, 3, 1, 4],\\
[1, 3, 1, 4, 2, 3, 4, 5, 6, 7, 6, 5, 4, 2, 3, 1, 8, 7, 6, 5, 4, 2, 3, 1, 4, 3, 5, 4, 2, 6, 7],\\
[1, 3, 1, 4, 2, 5, 4, 2, 3, 1, 7, 8, 7, 6, 5, 4, 2, 3, 1, 4, 3, 5, 4, 2, 6, 5, 4, 3, 1, 7, 6, 8, 7],\,
[1, 3, 1, 4, 3, 1, 5, 4, 6, 7, 8, 7, 6, 5, 4, 3, 1],\\
[1, 3, 1, 4, 3, 5, 4, 3, 1, 6, 5, 4, 3, 1, 7, 8, 7],\,
[1, 3, 1, 4, 3, 5, 4, 3, 6, 5, 4, 2, 3, 1, 4, 5, 6, 7, 8, 7, 6, 5, 4, 2, 3],\\
[1, 3, 1, 4, 3, 5, 4, 3, 6, 5, 4, 3, 1],\,
[1, 3, 1, 4, 5, 4, 2, 3, 1, 4, 5, 6, 7, 6, 8, 7, 6, 5, 4, 2, 3],\\
[1, 3, 1, 4, 5, 4, 2, 3, 4, 5, 6, 5, 4, 2, 7, 6, 8, 7, 6, 5, 4, 2, 3, 1, 4, 3, 5],\\
[1, 3, 1, 4, 5, 4, 2, 3, 4, 5, 6, 5, 7, 6, 5, 4, 2, 8, 7, 6, 5, 4, 2, 3, 1, 4, 3, 5, 6],\\
[1, 3, 1, 4, 5, 4, 2, 3, 4, 5, 6, 5, 7, 6, 5, 4, 2, 8, 7, 6, 5, 4, 2, 3, 1, 4, 5],\,
[1, 3, 1, 4, 5, 6, 5, 4, 3, 1, 7],\\
[1, 3, 1, 6, 5, 4, 2, 3, 1, 4, 3, 5, 4, 6, 5, 7, 6, 5, 4, 2, 3],\,
[1, 3, 4, 2, 3, 4, 5, 4, 2, 3, 4, 6, 5, 7, 6, 5, 4, 3, 8, 7, 6, 5, 4, 2, 3, 1, 4],\\
[1, 3, 4, 2, 3, 4, 5, 4, 2, 6, 5, 7, 6, 5, 8, 7, 6, 5, 4, 2, 3, 1, 4],\,
[1, 3, 4, 2, 3, 4, 5, 4, 6, 5, 7, 6, 5, 4, 2, 8, 7, 6, 5, 4, 2, 3, 1, 4, 5],\\
[1, 3, 4, 2, 3, 5, 4, 2, 3, 1, 4, 8, 7, 6, 5, 4, 2, 3, 1, 4, 3, 5, 4, 2, 6, 7, 8],\\
[1, 3, 4, 2, 3, 5, 6, 7, 6, 5, 4, 2, 3, 1, 4, 8, 7, 6, 5, 4, 2, 3, 1, 4, 3, 5, 4, 2, 6, 7, 8],\\
[1, 3, 4, 3, 1, 5, 4, 2, 3, 1, 4, 5, 6, 7, 6, 8, 7, 6, 5, 4, 2, 3, 4],\,
[1, 3, 4, 3, 5, 4, 2, 3, 1, 4, 5, 6, 5, 4, 2],\\
[1, 3, 4, 3, 5, 4, 2, 3, 1, 4, 5, 6, 7, 6, 5, 4, 2],\,
[1, 3, 4, 3, 5, 4, 3],\,
[1, 3, 4, 5, 4, 2, 3, 4, 5, 6, 5, 4, 2, 7, 6, 5, 4, 2, 3, 1, 4, 5, 8],\\
[1, 3, 4, 5, 4, 2, 3, 4, 5, 6, 5, 4, 2, 7, 6, 8, 7, 6, 5, 4, 2, 3, 1, 4, 5],\,
[1, 3, 4, 5, 4, 6, 5, 7, 6, 5, 4, 3, 1],\\
[1, 3, 4, 5, 4, 6, 7, 6, 5, 4, 3, 1, 8],\,
[1, 3, 4, 5, 6, 5, 4, 3, 8],\
[1, 4, 2, 3, 1, 4, 3, 5, 4, 3, 1, 6, 5, 7, 6, 5, 8, 7, 6, 5, 4, 2, 3, 1, 4],\\
[1, 4, 2, 3, 1, 4, 3, 5, 4, 3, 6, 7, 8, 7, 6, 5, 4, 2, 3, 1, 4],\,
[1, 4, 2, 3, 1, 4, 3, 5, 4, 6, 5, 7, 6, 5, 4, 3, 8, 7, 6, 5, 4, 2, 3, 1, 4],\\
[1, 4, 2, 3, 1, 4, 3, 5, 6, 7, 8, 7, 6, 5, 4, 3, 1],\,
[1, 4, 2, 3, 4, 5, 4, 3, 1, 6, 7, 8, 7, 6, 5, 4, 2, 3, 1, 4, 3, 5, 4],\\
[1, 4, 2, 3, 4, 5, 6, 5, 7, 6, 5, 4, 2, 3, 4],\,
[1, 4, 3, 1, 5, 4, 2, 3, 1, 4, 3, 5, 4, 6, 7, 8, 7, 6, 5, 4, 2, 3, 4],\\
[1, 4, 5, 4, 2, 3, 4, 5, 6, 5, 4, 2, 3, 1, 4, 5, 7, 8, 7],\,
[1, 4, 5, 4, 2, 3, 4, 5, 6, 5, 4, 2, 3, 1, 4, 5, 7],\\
[1, 4, 5, 4, 2, 3, 4, 5, 6, 5, 4, 2, 3, 4, 5],\,
[1, 4, 5, 4, 2, 3, 4, 5, 6, 7, 6, 5, 4, 2, 3, 1, 4, 5, 8],\\
[1, 5, 4, 2, 3, 1, 4, 3, 5, 4, 6, 5, 4, 3, 1],\,
[1, 5, 4, 2, 3, 1, 4, 3, 5, 4, 6, 7, 8, 7, 6, 5, 4, 3, 1],\,
[1, 5, 4, 2, 3, 1, 4, 5, 6, 5, 4, 2, 3, 1, 4, 5, 7, 8, 7],\\
[1, 5, 4, 2, 3, 4, 5, 6, 5, 4, 2, 3, 4, 5, 8],\,
[1, 5, 4, 2, 3, 4, 5],\,
[1, 6, 5, 4, 2, 3, 4, 5, 6],\\
[2, 3, 1, 4, 2, 3, 1, 4, 5, 6, 5, 7, 6, 5, 8, 7, 6, 5, 4, 2, 3, 1, 4, 3, 5, 4, 2, 6, 7],\\
[2, 3, 1, 4, 2, 3, 4, 5, 4, 6, 5, 4, 7, 6, 5, 4, 3, 8, 7, 6, 5, 4, 2, 3, 1, 4, 3, 5, 6],\\
[2, 3, 1, 4, 2, 5, 4, 2, 6, 5, 4, 2, 3, 1, 4, 5, 6, 7, 8, 7, 6, 5, 4, 3, 1],\,
[2, 3, 1, 4, 2, 5, 6, 5, 4, 2, 3, 1, 4, 5, 6, 7, 8, 7, 6, 5, 4, 3, 1],\\
[2, 3, 1, 4, 2, 5, 6, 7, 6, 5, 4, 2, 3, 1, 4, 5, 6, 7, 8, 7, 6, 5, 4, 3, 1],\\
[2, 3, 1, 4, 3, 1, 7, 6, 5, 4, 2, 3, 1, 4, 3, 5, 4, 2, 6, 5, 7, 6, 8, 7, 6, 5, 4, 2, 3, 1, 4],\\
[2, 3, 1, 4, 5, 4, 6, 5, 4, 2, 3, 1, 4, 5, 6, 7, 6, 5, 4, 3, 1],\,
[2, 3, 1, 4, 5, 4, 6, 5, 4, 2, 3, 1, 4, 5, 6, 7, 8, 7, 6, 5, 4, 3, 1],\\
[2, 3, 4, 2, 3, 1, 5, 6, 7, 6, 5, 4, 2, 3, 1, 4, 3, 5, 6, 7, 8, 7, 6, 5, 4, 2, 3, 1, 4],\,
[2, 3, 4, 2, 3, 5, 4, 2, 3, 6, 5, 7, 6, 5, 8, 7, 6, 5, 4, 2, 3],\\
[2, 3, 4, 2, 3, 5, 4, 2, 3, 6, 5, 7, 8, 7, 6, 5, 4, 2, 3],\,
[2, 3, 4, 2, 3, 5, 4, 2, 6, 7, 8, 7, 6, 5, 4, 2, 3],\\
[2, 3, 4, 2, 3, 5, 4, 6, 5, 4, 2, 3, 7, 6, 5, 4, 2, 3, 8],\,
[2, 3, 4, 2, 3, 5, 4, 6, 5, 4, 2, 3, 7, 6, 8, 7, 6, 5, 4, 2, 3],\\
[2, 3, 4, 2, 3, 5, 4, 6, 5, 4, 2, 7, 6, 5, 4, 2, 3],\,
[2, 3, 4, 2, 3, 5, 4, 6, 5, 4, 2, 7, 8, 7, 6, 5, 4, 2, 3],\,
[2, 3, 4, 2, 3, 5, 4, 6, 5, 4, 3, 7, 6, 5, 4, 2, 3],\\
[2, 3, 4, 2, 5, 4, 2, 6, 5, 4, 2, 3, 4],\,
[2, 3, 4, 2, 5, 4, 2, 6, 5, 4, 2, 3, 7, 8, 7],\,
[2, 3, 4, 2, 5, 4, 2, 6, 5, 7, 8, 7, 6, 5, 4, 2, 3],\\
[2, 3, 4, 2, 5, 4, 2, 6, 7, 8, 7, 6, 5, 4, 2, 3, 4],\,
[2, 3, 4, 2, 5, 4, 6, 5, 4, 2, 3],\,
[2, 3, 4, 2, 5, 4, 6, 5, 4, 2, 7, 6, 5, 4, 2, 3, 8],\\
[2, 3, 4, 2, 5, 4, 6, 5, 4, 7, 6, 5, 4, 2, 3],\,
[2, 3, 4, 2, 5, 6, 5, 4, 2, 3, 7],\,
[2, 3, 4, 3, 5, 4, 6, 5, 4, 3, 7, 6, 5, 4, 2, 3, 4],\\
[2, 3, 4, 3, 5, 4, 6, 5, 4, 7, 6, 5, 4, 2, 3],\,
[2, 3, 4, 3, 5, 6, 7, 6, 5, 4, 2, 3, 8],\,
[2, 3, 4, 5, 4, 6, 5, 4, 2, 3, 4, 5, 8],\\
[2, 3, 4, 5, 4, 6, 5, 4, 2],\,
[2, 3, 4, 5, 6, 5, 4, 2, 3, 4, 8],\,
[2, 3, 4, 5, 6, 5, 7, 6, 8, 7, 6, 5, 4, 2, 3],\\
[2, 3, 4, 5, 6, 7, 6, 8, 7, 6, 5, 4, 2, 3, 4],\,
[2, 4, 2, 3, 1, 4, 7, 6, 5, 4, 2, 3, 4, 5, 6, 7, 8, 7, 6, 5, 4, 2, 3, 1, 4, 3, 5, 6, 7],\\
[2, 4, 2, 3, 1, 6, 7, 6, 5, 4, 2, 3, 4, 5, 6, 7, 8, 7, 6, 5, 4, 2, 3, 1, 4, 3, 5, 6, 7],\,
[2, 4, 2, 3, 1, 7, 6, 5, 4, 2, 3, 1, 4, 3, 5, 6, 7],\\
[2, 4, 2, 3, 1, 8, 7, 6, 5, 4, 2, 3, 1, 4, 3, 5, 6, 7, 8],\,
[2, 4, 2, 3, 4, 5, 4, 3, 6, 5, 4, 2, 3, 4, 7, 8, 7],\,
[2, 4, 2, 3, 4, 5, 4, 3, 6, 5, 4, 2, 3, 4, 7],\\
[2, 4, 2, 3, 4, 5, 4, 3, 6, 5, 7, 6, 5, 4, 2, 3, 4],\,
[2, 4, 2, 3, 4, 5, 4, 3, 6, 7, 6, 5, 4, 2, 3, 4, 8],\,
[2, 4, 2, 3, 4, 5, 6, 5, 4, 3, 8],\\
[2, 4, 2, 5, 4, 6, 5, 4, 2],\,
[2, 4, 2, 5, 6, 5, 4, 2, 7],\,
[2, 4, 2, 5, 8],\,
[2, 4, 5, 4, 2, 3, 4, 5, 6, 5, 4, 3, 8],\\
[2, 4, 5, 4, 2, 7, 8],\,
[2, 4, 5, 6, 5, 4, 2, 3, 1, 7, 6, 5, 4, 2, 3, 1, 4, 3, 5, 6, 7],\,
[3, 1, 4, 5, 4, 2, 3, 4, 5, 6, 7, 6, 8, 7, 6, 5, 4, 2, 3, 1, 4, 3, 5, 6, 7],\\
[3, 1, 4, 5, 4, 2, 3, 4, 5, 6, 7, 8, 7, 6, 5, 4, 2, 3, 1, 4, 3, 5, 6],\,
[3, 1, 4, 5, 6, 5, 4, 2, 3, 4, 5, 6, 7, 8, 7, 6, 5, 4, 2, 3, 1, 4, 3, 5, 6],\\
[3, 1, 4, 6, 5, 4, 2, 3, 4, 5, 6, 7, 8, 7, 6, 5, 4, 2, 3, 1, 4, 3, 5, 6, 7],\,
[3, 1, 5, 4, 2, 3, 4, 5, 6, 7, 6, 8, 7, 6, 5, 4, 2, 3, 1, 4, 3, 5, 4],\\
[3, 1, 5, 4, 2, 3, 4, 5, 6, 7, 6, 8, 7, 6, 5, 4, 2, 3, 1, 4, 5],\,
[3, 1, 5, 6, 5, 4, 2, 3, 1, 4, 3, 5, 4, 6, 5],\\
[3, 1, 5, 6, 5, 4, 2, 3, 4, 5, 6, 7, 6, 5, 4, 2, 3, 1, 4, 3, 5, 6, 8],\,
[3, 1, 5, 6, 5, 4, 2, 3, 4, 5, 6, 7, 6, 5, 4, 2, 3, 1, 4, 5, 6],\\
[3, 1, 6, 5, 4, 2, 3, 1, 4, 5, 6],\,
[3, 1, 7, 6, 5, 4, 2, 3, 1, 4, 5, 6, 7],\\
[3, 4, 2, 3, 4, 5, 4, 2, 3, 1, 4, 8, 7, 6, 5, 4, 2, 3, 1, 4, 3, 5, 6, 7, 8],\,
[3, 4, 2, 3, 4, 5, 4, 2, 6, 5, 4, 2, 3, 4, 7, 8, 7],\\
[3, 4, 2, 3, 4, 5, 4, 2, 7],\,
[3, 4, 2, 3, 4, 5, 6, 5, 7, 6, 5, 4, 2],\,
[3, 4, 3, 1, 5, 4, 2, 3, 4, 5, 6, 7, 6, 8, 7, 6, 5, 4, 2, 3, 1, 4, 5],\\
[3, 4, 3, 5, 4, 2, 3, 4, 5, 6, 5, 4, 2],\,
[3, 4, 3, 6, 7],\,
[3, 4, 5, 4, 3, 1, 6, 5, 4, 2, 3, 1, 4, 5, 6],\,
[3, 4, 7],\\
[4, 2, 3, 1, 4, 5, 4, 2, 3, 1, 4, 3, 5, 6, 7, 6, 8, 7, 6, 5, 4, 3, 1],\,
[4, 2, 3, 4, 5, 6, 5, 4, 2, 3, 4, 5, 8],\,
[4, 5, 8],\,
[4, 7, 8],\\
[1, 3, 1, 4, 5, 6, 5, 7, 6, 5, 4, 2, 3, 1, 4, 5, 6, 7, 8, 7, 6, 5, 4, 2, 3]
\big\},
\end{multline*}}%
one obtains 90 solutions for $w_i$ in 
\eqref{eq:E8''''D} with $\ell_T(w_i)=3$ and $w_i$ of type $A_3$:
{\allowdisplaybreaks\tiny
\begin{multline*}
w_i\in\big\{[1, 2, 3, 1, 4, 3, 1, 5, 6, 7, 6, 5, 8, 7, 6, 5, 4, 2, 3, 1, 4, 3, 5, 4, 2, 6, 7],\\
[1, 2, 3, 1, 4, 3, 5, 4, 6, 5, 7, 6, 5, 4, 3, 1, 8, 7, 6, 5, 4, 2, 3, 1, 4],\,
[1, 2, 3, 1, 4, 5, 4, 6, 7, 8, 7, 6, 5, 4, 2, 3, 1, 4, 3, 5, 4],\\
[1, 2, 3, 1, 4, 5, 4, 6, 7, 8, 7, 6, 5, 4, 2, 3, 1, 4, 3, 5, 6],\,
[1, 2, 3, 1, 4, 5, 4, 6, 7, 8, 7, 6, 5, 4, 2, 3, 1, 4, 5],\\
[1, 2, 3, 1, 4, 5, 6, 5, 4, 7, 6, 5, 4, 2, 3, 1, 4, 3, 5, 6, 8],\,
[1, 2, 3, 1, 4, 5, 6, 5, 7, 8, 7, 6, 5, 4, 2, 3, 1, 4, 3, 5, 6],\\
[1, 2, 3, 1, 4, 5, 6, 7, 8, 7, 6, 5, 4, 2, 3, 1, 4, 3, 5],\,
[1, 2, 3, 1, 4, 5, 6, 7, 8, 7, 6, 5, 4, 2, 3, 1, 4],\\
[1, 2, 3, 4, 2, 3, 5, 4, 6, 5, 7, 6, 5, 4, 2, 3, 8, 7, 6, 5, 4, 2, 3],\,
[1, 2, 3, 4, 2, 5, 4, 2, 6, 7, 8, 7, 6, 5, 4, 2, 3],\\
[1, 2, 3, 4, 2, 5, 4, 6, 5, 4, 2, 7, 6, 5, 4, 2, 3],\,
[1, 2, 3, 4, 2, 5, 4, 6, 5, 4, 2, 7, 8, 7, 6, 5, 4, 2, 3],\,
[1, 2, 3, 4, 5, 4, 6, 7, 6, 5, 4, 2, 3, 1, 4, 5, 8],\\
[1, 2, 3, 4, 5, 6, 7, 6, 5, 4, 2, 3, 1, 4, 5],\,
[1, 2, 3, 4, 5, 6, 7, 6, 5, 4, 2, 3, 1, 4, 8],\,
[1, 2, 3, 4, 5, 6, 7, 6, 5, 4, 2, 3, 4],\\
[1, 2, 3, 4, 5, 6, 7, 6, 5, 4, 2, 3, 8],\,
[1, 2, 3, 4, 5, 6, 7, 8, 7, 6, 5, 4, 2, 3, 1, 4, 5],\\
[1, 2, 3, 4, 5, 6, 7, 8, 7, 6, 5, 4, 2, 3, 4],\,
[1, 2, 4, 2, 3, 1, 4, 3, 5, 4, 3, 1, 6, 7, 8, 7, 6, 5, 4, 3, 1],\\
[1, 2, 4, 2, 3, 1, 4, 5, 6, 7, 6, 5, 4, 2, 3, 1, 4, 5, 6, 7, 8, 7, 6, 5, 4, 2, 3],\,
[1, 2, 4, 2, 3, 1, 7, 6, 5, 4, 2, 3, 1, 4, 3, 5, 4, 2, 6, 7, 8],\\
[1, 2, 4, 2, 3, 4, 5, 4, 6, 5, 4, 7, 6, 5, 4, 3, 1],\,
[1, 2, 4, 2, 3, 4, 5, 4, 6, 7, 6, 5, 4, 3, 1],\,
[1, 2, 4, 2, 3, 4, 5, 4, 6, 7, 8, 7, 6, 5, 4, 3, 1],\\
[1, 2, 4, 2, 3, 4, 5, 6, 7, 6, 5, 4, 3, 1, 8],\,
[1, 2, 4, 5, 4, 2, 3, 4, 5, 6, 5, 4, 3, 1, 7, 8, 7],\,
[1, 2, 4, 5, 4, 2, 3, 4, 5, 6, 5, 4, 3, 1, 7],\\
[1, 2, 4, 5, 4, 2, 3, 4, 5, 6, 5, 4, 3],\,
[1, 2, 4, 5, 4, 2, 3, 4, 5, 6, 5, 7, 6, 5, 4, 3, 1],\,
[1, 2, 4, 5, 4, 2, 3, 4, 5, 6, 7, 6, 5, 4, 3, 1, 8],\\
[1, 3, 1, 4, 5, 4, 2, 3, 1, 4, 5, 6, 5, 7, 6, 5, 8, 7, 6, 5, 4, 2, 3],\,
[1, 3, 1, 4, 5, 4, 2, 3, 1, 4, 5, 6, 5, 7, 8, 7, 6, 5, 4, 2, 3],\\
[1, 3, 1, 4, 5, 4, 2, 3, 1, 4, 5, 6, 7, 8, 7, 6, 5, 4, 2, 3, 4],\,
[1, 3, 1, 4, 5, 6, 5, 4, 2, 3, 1, 4, 5, 6, 7, 6, 5, 4, 2, 3, 8],\\
[1, 3, 1, 4, 5, 6, 5, 4, 2, 3, 1, 4, 5, 6, 7, 6, 5, 4, 2],\,
[1, 3, 1, 4, 5, 6, 5, 4, 2, 3, 1, 4, 5, 6, 7, 6, 8, 7, 6, 5, 4, 2, 3],\\
[1, 3, 4, 2, 3, 4, 5, 4, 2, 6, 7, 8, 7, 6, 5, 4, 2, 3, 4],\,
[1, 3, 4, 5, 4, 2, 3, 1, 4, 5, 6, 5, 4, 2, 7],\\
[1, 3, 4, 5, 4, 2, 3, 1, 4, 5, 6, 5, 7, 6, 5, 4, 2],\,
[1, 3, 4, 5, 4, 2, 3, 4, 5, 6, 5, 4, 2],\\
[1, 3, 4, 5, 4, 3, 1, 6, 7],\,
[1, 3, 4, 5, 4, 3, 6],\,
[1, 3, 4, 5, 6, 5, 4, 3, 1, 7, 8],\,
[1, 4, 2, 3, 1, 4, 5, 6, 7, 8, 7, 6, 5, 4, 2, 3, 1, 4, 5],\\
[1, 4, 2, 3, 4, 5, 6, 7, 6, 5, 4, 2, 3, 1, 4, 5, 6],\,
[1, 4, 2, 3, 4, 5, 6, 7, 6, 5, 4, 2, 3, 1, 4, 5, 8],\,
[1, 4, 2, 3, 4, 5, 6, 7, 6, 5, 4, 2, 3, 4, 8],\\
[1, 5, 4, 2, 3, 1, 4, 5, 6, 7, 6],\,
[1, 5, 4, 2, 3, 1, 4, 5, 6],\,
[1, 5, 4, 2, 3, 4, 5, 6, 5, 4, 2, 3, 1, 4, 5, 7, 8],\,
[1, 6, 5, 4, 2, 3, 1, 4, 5, 6, 7],\\
[2, 3, 1, 4, 2, 3, 4, 5, 4, 6, 5, 7, 6, 5, 4, 3, 8, 7, 6, 5, 4, 2, 3, 1, 4],\,
[2, 3, 1, 4, 5, 6, 5, 4, 2, 3, 1, 4, 5, 6, 7, 6, 5, 4, 3, 1, 8],\\
[2, 3, 1, 4, 5, 6, 5, 4, 2, 3, 1, 4, 5, 6, 7, 6, 8, 7, 6, 5, 4, 3, 1],\,
[2, 3, 1, 4, 5, 6, 5, 4, 2, 3, 4, 5, 6, 7, 6, 5, 4, 3, 1],\,
[2, 3, 4, 2, 5, 4, 2, 6, 5, 4, 2],\\
[2, 3, 4, 2, 5, 6, 5, 4, 2],\,
[2, 3, 4, 2, 5, 6, 7, 6, 5, 4, 2],\,
[2, 3, 4, 3, 1, 5, 6, 7, 6, 5, 4, 2, 3, 1, 4, 5, 6, 7, 8, 7, 6, 5, 4, 3, 1],\\
[2, 3, 4, 5, 6, 5, 4, 2, 3, 4, 5],\,
[2, 3, 4, 5, 6, 5, 4, 2, 3, 4, 7, 8, 7],\,
[2, 3, 4, 5, 6, 5, 4, 2, 3, 4, 7],\\
[2, 3, 4, 5, 6, 5, 4, 2, 3, 7, 8],\,
[2, 3, 4, 5, 6, 5, 4, 2, 7],\,
[2, 3, 4, 5, 6, 7, 6, 5, 4, 2, 3, 4, 8],\,
[2, 4, 2, 3, 4, 5, 4, 3, 6],\\
[2, 4, 2, 3, 4, 5, 4, 6, 5, 4, 3],\,
[2, 4, 2, 5, 6],\,
[2, 4, 5, 4, 2, 6, 7],\,
[3, 1, 4, 2, 3, 4, 5, 6, 7, 8, 7, 6, 5, 4, 2, 3, 1, 4, 3, 5, 6],\\
[3, 1, 4, 2, 3, 4, 5, 6, 7, 8, 7, 6, 5, 4, 2, 3, 1, 4, 5],\,
[3, 1, 4, 5, 4, 2, 3, 4, 5, 6, 7, 8, 7, 6, 5, 4, 2, 3, 1, 4, 5],\\
[3, 1, 5, 4, 2, 3, 4, 5, 6, 7, 8, 7, 6, 5, 4, 2, 3, 1, 4, 3, 5, 6, 7],\,
[3, 1, 6, 5, 4, 2, 3, 1, 4, 3, 5, 6, 7, 8, 7],\\
[3, 1, 6, 5, 4, 2, 3, 1, 4, 3, 5, 6, 7],\,
[3, 1, 7, 6, 5, 4, 2, 3, 1, 4, 3, 5, 6, 7, 8],\\
[3, 4, 2, 3, 1, 4, 5, 4, 2, 3, 4, 5, 6, 7, 8, 7, 6, 5, 4, 3, 1],\,
[3, 4, 2, 3, 4, 5, 4, 2, 6, 7, 6],\,
[3, 4, 2, 3, 4, 5, 4, 2, 6],\,
[3, 4, 2, 3, 4, 5, 4, 6, 5, 4, 2],\\
[3, 4, 2, 3, 4, 5, 6, 5, 4, 2, 7],\,
[3, 4, 3, 5, 6],\,
[3, 4, 5],\,
[4, 2, 3, 4, 5, 4, 2, 3, 4, 6, 7],\\
[4, 2, 3, 4, 5, 6, 5, 4, 2, 3, 4, 7, 8],\,
[4, 2, 3, 4, 5, 6, 5],\,
[4, 2, 3, 4, 5],\,
[5, 4, 2, 3, 4, 5, 6]
\big\},
\end{multline*}}%
one obtains 15 solutions for $w_i$ in 
\eqref{eq:E8''''D} with $\ell_T(w_i)=4$ and $w_i$ of type 
$A_1^2*A_2$:
{\allowdisplaybreaks\small
\begin{multline} \label{eq:E8sol6}
w_i\in\big\{ [1, 2, 3, 1, 4, 5, 4, 6, 5, 4, 2, 3, 4, 5, 7, 6, 5, 4, 2, 3, 4, 5, 8, 7, 6, 5, 4, 2, 3, 1],\\
[1, 2, 3, 4, 3, 1, 5, 6, 5, 7, 6, 5, 4, 2, 3, 1, 4, 3, 5, 4, 6, 5, 8, 7, 6, 5, 4, 2, 3, 1],\\
[1, 2, 3, 4, 5, 4, 6, 5, 4, 2, 3, 4, 5, 7, 6, 5, 4, 2, 3, 1, 4, 5],\\
[1, 2, 3, 4, 5, 4, 6, 5, 4, 3, 1, 7, 6, 5, 4, 2, 3, 1, 4, 3, 5, 6],\\
[1, 3, 1, 4, 2, 3, 4, 5, 6, 5, 7, 6, 5, 4, 2, 3, 1, 8, 7, 6, 5, 4, 2, 3, 1, 4, 3, 5, 4, 2, 6, 7],\\
[1, 3, 1, 4, 2, 3, 5, 4, 2, 3, 1, 7, 8, 7, 6, 5, 4, 2, 3, 1, 4, 3, 5, 4, 2, 6, 5, 4, 3, 1, 7, 6, 8, 7],\\
[1, 4, 3, 1, 5, 4, 2, 3, 1, 4, 3, 5, 4, 6, 7, 6, 8, 7, 6, 5, 4, 2, 3, 4],\\
[1, 4, 5, 4, 2, 3, 1, 4, 5, 6, 5, 4, 2, 3, 1, 4, 5, 7, 8, 7],\\
[1, 4, 5, 4, 2, 3, 4, 5, 6, 5, 4, 2, 3, 4, 5, 8],\\
[2, 4, 2, 3, 1, 4, 5, 4, 3, 6, 7, 6, 8, 7, 6, 5, 4, 2, 3, 1, 4, 3, 5, 4, 2, 6, 7, 8],\\
[2, 4, 2, 3, 1, 4, 6, 7, 6, 5, 4, 2, 3, 4, 5, 6, 7, 8, 7, 6, 5, 4, 2, 3, 1, 4, 3, 5, 6, 7],\\
[3, 1, 4, 2, 3, 4, 5, 6, 5, 7, 6, 5, 4, 2, 3, 4, 5, 6, 8, 7, 6, 5, 4, 2, 3, 1, 4, 3, 5, 6],\\
[3, 1, 5, 6, 5, 4, 2, 3, 1, 4, 3, 5, 4, 6, 5, 8],\\
[3, 4, 2, 3, 4, 5, 4, 2, 3, 1, 4, 7, 8, 7, 6, 5, 4, 2, 3, 1, 4, 3, 5, 6, 7, 8],\\
[4, 2, 3, 4, 5, 4, 2, 3, 4, 6, 5, 4, 2, 3, 4, 7, 8, 7]
\big\},
\end{multline}}%
one obtains 45 solutions for $w_i$ in 
\eqref{eq:E8''''D} with $\ell_T(w_i)=4$ and $w_i$ of type $A_1*A_3$:
{\allowdisplaybreaks\tiny
\begin{multline} \label{eq:E8sol7}
w_i\in\big\{
[1, 2, 3, 1, 4, 2, 5, 4, 2, 6, 5, 4, 2, 7, 8, 7, 6, 5, 4, 2, 3, 1, 4, 5],\\
[1, 2, 3, 1, 4, 2, 5, 4, 6, 5, 4, 2, 3, 4, 5, 7, 6, 5, 4, 3, 8, 7, 6, 5, 4, 2, 3, 1],\\
[1, 2, 3, 1, 4, 3, 1, 5, 6, 5, 7, 6, 5, 8, 7, 6, 5, 4, 2, 3, 1, 4, 3, 5, 4, 2, 6, 7],\\
[1, 2, 3, 1, 4, 3, 5, 4, 6, 5, 4, 7, 6, 5, 4, 3, 8, 7, 6, 5, 4, 2, 3, 1, 4, 3, 5, 6],\\
[1, 2, 3, 1, 4, 5, 4, 6, 5, 4, 7, 6, 5, 4, 2, 3, 1, 4, 3, 5, 6, 8],\,
[1, 2, 3, 4, 2, 3, 5, 4, 6, 5, 4, 7, 6, 5, 4, 2, 3, 8, 7, 6, 5, 4, 2, 3],\\
[1, 2, 3, 4, 2, 3, 5, 4, 6, 5, 7, 6, 5, 4, 2, 3, 4, 5, 8, 7, 6, 5, 4, 2, 3, 1],\,
[1, 2, 3, 4, 2, 5, 4, 2, 3, 4, 6, 5, 4, 7, 6, 5, 4, 3, 8, 7, 6, 5, 4, 2, 3, 1],\\
[1, 2, 3, 4, 2, 5, 4, 2, 6, 5, 4, 2, 7, 6, 5, 4, 2, 3, 1, 4, 5, 8],\,
[1, 2, 3, 4, 3, 5, 4, 6, 5, 4, 3, 7, 6, 5, 4, 2, 3, 1, 4, 5],\\
[1, 2, 3, 4, 3, 5, 4, 6, 5, 4, 3, 7, 6, 5, 4, 2, 3, 4],\,
[1, 2, 3, 4, 5, 6, 5, 4, 2, 3, 4, 5, 7, 6, 5, 4, 2, 3],\\
[1, 2, 4, 2, 3, 1, 4, 5, 6, 5, 7, 6, 5, 4, 2, 3, 1, 4, 5, 6, 7, 8, 7, 6, 5, 4, 2, 3],\,
[1, 2, 4, 2, 3, 4, 5, 4, 3, 6, 7, 6, 5, 4, 2, 3, 4, 8],\\
[1, 2, 4, 5, 4, 2, 3, 4, 5, 6, 5, 4, 3, 8],\,
[1, 3, 1, 4, 2, 3, 4, 5, 4, 6, 5, 7, 6, 5, 4, 2, 8, 7, 6, 5, 4, 2, 3, 1, 4, 3, 5, 6],\\
[1, 3, 1, 4, 2, 3, 4, 5, 4, 6, 5, 7, 6, 5, 4, 2, 8, 7, 6, 5, 4, 2, 3, 1, 4, 5],\,
[1, 3, 1, 4, 5, 4, 2, 3, 1, 4, 5, 6, 7, 6, 8, 7, 6, 5, 4, 2, 3, 4],\\
[1, 3, 4, 2, 3, 5, 4, 2, 3, 1, 4, 6, 7, 6, 8, 7, 6, 5, 4, 2, 3, 1, 4, 3, 5, 4, 2, 6, 7, 8],\,
[1, 3, 4, 3, 5, 4, 2, 3, 1, 4, 5, 6, 5, 4, 2, 7],\\
[1, 3, 4, 3, 5, 4, 2, 3, 4, 5, 6, 5, 4, 2],\,
[1, 4, 2, 3, 1, 4, 3, 5, 4, 3, 6, 5, 7, 6, 5, 8, 7, 6, 5, 4, 2, 3, 1, 4],\\
[1, 4, 2, 3, 4, 5, 4, 2, 3, 4, 6, 5, 7, 6, 5, 8, 7, 6, 5, 4, 2, 3, 1, 4],\,
[1, 4, 2, 3, 4, 5, 4, 2, 3, 4, 6, 7, 8, 7, 6, 5, 4, 2, 3, 4],\\
[1, 4, 5, 4, 2, 3, 4, 5, 6, 5, 4, 2, 3, 1, 4, 5, 7, 8],\,
[1, 5, 4, 2, 3, 1, 4, 3, 5, 4, 6, 5, 4, 3, 1, 7, 8, 7],\\
[2, 3, 1, 4, 2, 5, 6, 5, 4, 2, 3, 1, 4, 5, 6, 7, 6, 8, 7, 6, 5, 4, 3, 1],\,
[2, 3, 1, 4, 5, 4, 6, 5, 4, 2, 3, 1, 4, 5, 6, 7, 6, 5, 4, 3, 1, 8],\\
[2, 3, 1, 4, 5, 4, 6, 5, 4, 2, 3, 4, 5, 6, 7, 6, 5, 4, 3, 1],\,
[2, 3, 4, 2, 3, 1, 5, 6, 7, 6, 5, 4, 2, 3, 1, 4, 5, 6, 7, 8, 7, 6, 5, 4, 3, 1],\\
[2, 3, 4, 2, 3, 5, 4, 2, 3, 6, 5, 7, 6, 8, 7, 6, 5, 4, 2, 3],\,
[2, 3, 4, 2, 3, 5, 4, 2, 6, 5, 7, 8, 7, 6, 5, 4, 2, 3],\\
[2, 3, 4, 2, 3, 5, 4, 6, 5, 4, 7, 6, 5, 4, 2, 3],\,
[2, 3, 4, 2, 5, 4, 2, 6, 5, 4, 2, 3, 4, 7, 8, 7],\\
[2, 3, 4, 5, 6, 5, 4, 2, 3, 4, 5, 8],\,
[2, 4, 2, 3, 1, 4, 5, 4, 8, 7, 6, 5, 4, 2, 3, 1, 4, 3, 5, 6, 7, 8],\\
[2, 4, 2, 3, 1, 5, 6, 5, 7, 6, 5, 4, 2, 3, 1, 4, 3, 5, 6, 7],\,
[2, 4, 2, 3, 4, 5, 4, 3, 6, 5, 4, 2, 3, 4, 7, 8],\,
[2, 4, 2, 3, 4, 5, 4, 6, 5, 4, 3, 8],\\
[3, 1, 4, 2, 3, 4, 5, 6, 7, 6, 5, 4, 2, 3, 4, 5, 6, 8, 7, 6, 5, 4, 2, 3, 1, 4],\,
[3, 1, 4, 5, 4, 2, 3, 4, 5, 6, 7, 6, 8, 7, 6, 5, 4, 2, 3, 1, 4, 5],\\
[3, 1, 4, 5, 4, 2, 3, 4, 5, 6, 7, 8, 7, 6, 5, 4, 2, 3, 1, 4, 3, 5, 6, 7],\,
[3, 1, 4, 5, 4, 6, 5, 4, 2, 3, 1, 4, 3, 5, 6, 7, 8, 7] ,\\
[3, 1, 4, 5, 4, 6, 5, 4, 2, 3, 1, 4, 5, 6],\,
[3, 4, 2, 3, 1, 4, 5, 4, 2, 3, 4, 5, 6, 7, 6, 8, 7, 6, 5, 4, 3, 1]
\big\},
\end{multline}}%
one obtains 5 solutions for $w_i$ in 
\eqref{eq:E8''''D} with $\ell_T(w_i)=4$ and $w_i$ of type $A_2^2$:
{\allowdisplaybreaks\small
\begin{multline} \label{eq:E8sol8}
w_i\in\big\{
[1, 2, 3, 4, 2, 5, 4, 2, 3, 4, 6, 5, 7, 6, 5, 4, 3, 8, 7, 6, 5, 4, 2, 3, 1, 4],\\
[1, 2, 4, 2, 3, 4, 5, 4, 3, 6, 5, 7, 6, 5, 4, 2, 3, 4],\\
[1, 3, 1, 4, 5, 4, 2, 3, 4, 5, 6, 5, 4, 2, 7, 6, 8, 7, 6, 5, 4, 2, 3, 1, 4, 5],\\
[2, 3, 1, 4, 2, 5, 4, 6, 5, 4, 2, 3, 1, 4, 5, 6, 7, 8, 7, 6, 5, 4, 3, 1],\\
[2, 3, 4, 2, 3, 5, 4, 6, 5, 4, 2, 7, 6, 5, 4, 2, 3, 8]
\big\},
\end{multline}}%
one obtains 18 solutions for $w_i$ in 
\eqref{eq:E8''''D} with $\ell_T(w_i)=4$ and $w_i$ of type $A_4$:
{\allowdisplaybreaks\tiny
\begin{multline} \label{eq:E8sol9}
w_i\in\big\{
[1, 2, 3, 1, 4, 3, 5, 4, 6, 5, 7, 6, 5, 4, 3, 8, 7, 6, 5, 4, 2, 3, 1, 4],\,
[1, 2, 3, 1, 4, 5, 6, 7, 8, 7, 6, 5, 4, 2, 3, 1, 4, 3, 5, 6],\\
[1, 2, 3, 4, 2, 5, 4, 2, 6, 5, 7, 8, 7, 6, 5, 4, 2, 3],\,
[1, 2, 3, 4, 2, 5, 4, 2, 6, 7, 8, 7, 6, 5, 4, 2, 3, 4],\\
[1, 2, 3, 4, 2, 5, 4, 6, 5, 4, 2, 7, 6, 5, 4, 2, 3, 8],\,
[1, 2, 3, 4, 5, 6, 7, 6, 5, 4, 2, 3, 1, 4, 5, 8],\,
[1, 2, 4, 2, 3, 4, 5, 4, 6, 5, 7, 6, 5, 4, 3, 1],\\
[1, 2, 4, 2, 3, 4, 5, 4, 6, 7, 6, 5, 4, 3, 1, 8],\,
[1, 2, 4, 5, 4, 2, 3, 4, 5, 6, 5, 4, 3, 1, 7, 8],\\
[1, 3, 1, 4, 5, 4, 2, 3, 1, 4, 5, 6, 5, 7, 6, 8, 7, 6, 5, 4, 2, 3],\,
[1, 4, 2, 3, 1, 4, 3, 5, 4, 6, 7, 8, 7, 6, 5, 4, 3, 1],\\
[1, 5, 4, 2, 3, 4, 5, 6],\,
[2, 3, 4, 2, 5, 4, 6, 5, 4, 2],\,
[2, 3, 4, 2, 5, 6, 5, 4, 2, 7],\,
[2, 3, 4, 5, 6, 5, 4, 2, 3, 4, 7, 8],\\
[2, 4, 2, 3, 1, 7, 6, 5, 4, 2, 3, 1, 4, 3, 5, 6, 7, 8],\,
[3, 1, 6, 5, 4, 2, 3, 1, 4, 5, 6, 7],\,
[3, 4, 2, 3, 4, 5, 4, 2, 6, 7]
\big\},
\end{multline}}%
and one obtains 5 solutions for $w_i$ in 
\eqref{eq:E8''''D} with $\ell_T(w_i)=4$ and $w_i$ of type $D_4$:
{\small
\begin{multline} \label{eq:E8sol10}
w_i\in\big\{
[1, 2, 3, 1, 4, 5, 6, 7, 8, 7, 6, 5, 4, 2, 3, 1, 4, 5],\,
[1, 2, 3, 4, 5, 6, 7, 6, 5, 4, 2, 3, 4, 8],\\
[1, 5, 4, 2, 3, 1, 4, 5, 6, 7],\,
[3, 1, 6, 5, 4, 2, 3, 1, 4, 3, 5, 6, 7, 8],\,
[4, 2, 3, 4, 5, 6]
\big\},
\end{multline}}%
where $\{s_1,s_2,s_3,s_4,s_5,s_6,s_7,s_8\}$ is a simple system of generators of $E_8$,
corresponding to the Dynkin diagram displayed in Figure~\ref{fig:E8},
and each of them gives rise to $m/2$ elements of
$\Fix_{NC^m(E_8)}(\phi^{p})$ since $i$ ranges from $1$ to $m/2$.
There are no solutions  for $w_i$ in \eqref{eq:E8''''D} with 
$w_i$ of type $A_1^4$.

Letting the computer find all solutions in cases (iii)--(v) would
take years. However, the number of these solutions can be computed
from our knowledge of the solutions in Case~(ii) according to type,
if this information is combined with the decomposition numbers in
the sense of \cite{KratCB,KratCF,KrMuAB} (see the end of
Section~\ref{sec:prel}) and some elementary 
(multiset) permutation counting. The decomposition numbers for
$A_2$, $A_3$, $A_4$, and $D_4$ of which we make use can be found in
the appendix of \cite{KratCF}.

To begin with, the number of
solutions of \eqref{eq:E8''''E} with $\ell_1=\ell_2=1$ is equal to
$$
n_{1,1}:=150\cdot 2+100\cdot N_{A_2}(A_1,A_1)=600,
$$
since an element of type $A_1^2$ can be decomposed in two ways 
into a product of two elements of absolute length $1$, while for
an element of type $A_2$ this can be done in $ N_{A_2}(A_1,A_1)=3$
ways. Similarly, the number of
solutions of \eqref{eq:E8''''E} with $\ell_1=2$ and $\ell_2=1$ is equal to
$$
n_{2,1}:=75\cdot 3+165\cdot(1+N_{A_2}(A_1,A_1))
+90\cdot N_{A_3}(A_2,A_1)=1425,
$$
the number of
solutions of \eqref{eq:E8''''E} with $\ell_1=3$ and $\ell_2=1$ is equal to
\begin{multline*}
n_{3,1}:=15\cdot (2+N_{A_2}(A_1,A_1))
+45\cdot(1+N_{A_3}(A_2,A_1))
+5\cdot(2N_{A_2}(A_1,A_1))\\
+18\cdot(N_{A_4}(A_3,A_1)+N_{A_4}(A_1*A_2,A_1))
+5\cdot(N_{D_4}(A_3,A_1)+N_{D_4}(A_1^3,A_1))=660,
\end{multline*}
the number of
solutions of \eqref{eq:E8''''E} with $\ell_1=\ell_2=2$ is equal to
\begin{multline*}
n_{2,2}:=15\cdot (2+2N_{A_2}(A_1,A_1))
+45\cdot(2N_{A_3}(A_2,A_1))
+5\cdot(2+N_{A_2}(A_1,A_1)^2)\\
+18\cdot(N_{A_4}(A_2,A_2)+N_{A_4}(A_1^2,A_1^2)+2N_{A_4}(A_2,A_1^2))\\
+5\cdot(N_{D_4}(A_2,A_2)+2N_{D_4}(A_2,A_1^2))=1195,
\end{multline*}
the number of
solutions of \eqref{eq:E8''''F} with $\ell_1=\ell_2=\ell_3=1$ is 
equal to
$$
n_{1,1,1}:=75\cdot 3!+165\cdot(3N_{A_2}(A_1,A_1))
+90N_{A_3}(A_1,A_1,A_1)=3375,
$$
the number of
solutions of \eqref{eq:E8''''F} with $\ell_1=2$ and 
$\ell_2=\ell_3=1$ is equal to
\begin{multline*}
n_{2,1,1}:=15\cdot (2+N_{A_2}(A_1,A_1)+2\cdot2\cdot N_{A_2}(A_1,A_1))
+45\cdot(2N_{A_3}(A_2,A_1)+N_{A_3}(A_1,A_1,A_1))\\
+5\cdot(2N_{A_2}(A_1,A_1)+2N_{A_2}(A_1,A_1)^2)
+18\cdot(N_{A_4}(A_2,A_1,A_1)+N_{A_4}(A_1^2,A_1,A_1))\\
+5\cdot(N_{D_4}(A_2,A_1,A_1)+N_{D_4}(A_1^2,A_1,A_1))=2850,
\end{multline*}
and the number of
solutions of \eqref{eq:E8''''G} is equal to
\begin{multline*}
n_{1,1,1,1}:=15\cdot (12N_{A_2}(A_1,A_1))
+45\cdot(4N_{A_3}(A_1,A_1,A_1))
+5\cdot(6N_{A_2}(A_1,A_1)^2)\\
+18\cdot N_{A_4}(A_1,A_1,A_1,A_1)
+5\cdot N_{D_4}(A_1,A_1,A_1,A_1)=6750.
\end{multline*}

In total, we obtain 
\begin{multline*}
1+(45+150+100+75+165+90+15+45+5+18+5)\frac m2
+(n_{1,1}+2n_{2,1}+2n_{3,1}+n_{2,2})\binom {m/2}2\\
+(n_{1,1,1}+3n_{2,1,1})\binom {m/2}3
+n_{1,1,1,1}\binom {m/2}4
=\frac {(5m+4)(3m+2)(5m+2)(15m+4)}{64}
\end{multline*}
elements in
$\Fix_{NC^m(E_8)}(\phi^p)$, which agrees with the limit in
\eqref{eq:E8.11}.

\smallskip
Finally, we turn to \eqref{eq:E8.1}. By Remark~\ref{rem:1},
the only choices for $h_2$ and $m_2$ to be considered
are $h_2=1$ and $m_2=7$, $h_2=2$ and $m_2=8$, 
$h_2=2$ and $m_2=7$, $h_2=2$ and $m_2=4$, 
respectively $h_2=m_2=2$. These correspond to the choices
$p=30m/7$, $p=15m/8$, $p=15m/7$, $p=15m/4$, respectively
$15m/2$, out of which only $p=15m/8$
has not yet been discussed and belongs to the current case.
The corresponding action of $\phi^p$ is given by 
\begin{multline*}
\phi^p\big((w_0;w_1,\dots,w_m)\big)\\=
(*;
c^{2}w_{\frac {m}8+1}c^{-2},c^{2}w_{\frac {m}8+2}c^{-2},
\dots,c^{2}w_{m}c^{-2},
cw_{1}c^{-1},\dots,
cw_{\frac {m}8}c^{-1}\big),
\end{multline*}
so that we have to solve 
\begin{equation} \label{eq:8prod} 
t_1(c^{13}t_1c^{-13})(c^{11}t_1c^{-11})(c^{9}t_1c^{-9})
(c^{7}t_1c^{-7})(c^{5}t_1c^{-5})(c^{3}t_1c^{-3})(ct_1c^{-1})=c
\end{equation}
for $t_1$ with $\ell_T(t_1)=1$.
A computation with the help of 
Stembridge's {\sl Maple} package {\tt coxeter}
\cite{StemAZ} finds no solution. 
Hence, 
the left-hand side of \eqref{eq:1} is equal to $1$, as required.

\section{Cyclic sieving II}
\label{sec:siev2}

In this section we present the second cyclic sieving conjecture due to
Bessis and Reiner \cite[Conj.~6.5]{BeReAA}.

Let $\psi:NC^m(W)\to NC^m(W)$ be the map defined by
\begin{equation} \label{eq:psi}
(w_0;w_1,\dots,w_m)\mapsto
\big(cw_mc^{-1};w_0,w_1,\dots,w_{m-1}\big).
\end{equation}
For $m=1$, we have $w_0=cw_1^{-1}$, so that
this action reduces to the inverse of the {\it Kreweras complement\/} 
$K_{\text {id}}^{c}$ as defined by Armstrong
\cite[Def.~2.5.3]{ArmDAA}.

It is easy to see that $\psi^{(m+1)h}$ acts as the identity,
where $h$ is the Coxeter number of $W$ (see \eqref{eq2:Aktion} below).
By slight abuse of notation as before, let $C_2$ be the cyclic group of order
$(m+1)h$ generated by $\psi$. 

Given these definitions, we are now in the position to state
the second cyclic sieving conjecture of Bessis and Reiner.
By the results of \cite{KratCG} and of this paper, it becomes the
following theorem.

\begin{theorem} \label{thm2:1}
For an irreducible well-generated complex reflection group $W$ and any $m\ge1$, 
the triple $(NC^m(W),\Cat^m(W;q),C_2)$, where $\Cat^m(W;q)$ is
the $q$-analogue of the Fu\ss--Catalan number defined in
\eqref{eq:FCZahl}, exhibits the cyclic sieving phenomenon.
\end{theorem}

By definition of the cyclic sieving phenomenon, we have to
prove that 
\begin{equation} \label{eq2:1}
\vert\Fix_{NC^m(W)}(\psi^{p})\vert = 
\Cat^m(W;q)\big\vert_{q=e^{2\pi i p/(m+1)h}}, 
\end{equation}
for all $p$ in the range $0\le p<(m+1)h$.

\section{Auxiliary results II}
\label{sec:aux2}

This section collects several auxiliary results which allow us to
reduce the problem of proving Theorem~\ref{thm2:1}, respectively the
equivalent statement
\eqref{eq2:1}, for the 26 exceptional groups listed in
Section~\ref{sec:prel} to a finite problem. While Lemmas~\ref{lem2:2}
and \ref{lem2:3} cover special choices of the parameters,
Lemmas~\ref{lem2:1} and
\ref{lem2:6} afford an inductive procedure. More precisely, 
if we assume that we have already verified Theorem~\ref{thm2:1} for all 
groups of smaller rank, then Lemmas~\ref{lem2:1} and \ref{lem2:6}, together 
with Lemmas~\ref{lem2:2} and \ref{lem2:7}, 
reduce the verification of Theorem~\ref{thm2:1}
for the group that we are currently considering to a finite problem;
see Remark~\ref{rem2:1}.
The final lemma of this section, Lemma~\ref{lem2:8}, disposes of 
complex reflection groups with a special property satisfied by their degrees.

Let $p=a(m+1)+b$, $0\le b<m+1$. We have
\begin{multline}
\psi^p\big((w_0;w_1,\dots,w_m)\big)\\
=(c^{a+1}w_{m-b+1}c^{-a-1};c^{a+1}w_{m-b+2}c^{-a-1},
\dots,c^{a+1}w_{m}c^{-a-1},\\
c^{a}w_{0}c^{-a},\dots,
c^{a}w_{m-b}c^{-a}\big).
\label{eq2:Aktion}
\end{multline}

\begin{lemma} \label{lem2:1}
It suffices to check \eqref{eq2:1} for $p$ a divisor of $(m+1)h$.
More precisely, let $p$ be a divisor of $(m+1)h$, and let $k$ be another positive integer with 
$\gcd(k,(m+1)h/p)=1$, then we have
\begin{equation} \label{eq2:2}
\Cat^m(W;q)\big\vert_{q=e^{2\pi i p/(m+1)h}}
= 
\Cat^m(W;q)\big\vert_{q=e^{2\pi i kp/(m+1)h}}
\end{equation}
and
\begin{equation} \label{eq2:3}
\vert\Fix_{NC^m(W)}(\psi^{p})\vert =
\vert\Fix_{NC^m(W)}(\psi^{kp})\vert 	.
\end{equation}
\end{lemma}

\begin{proof}
For \eqref{eq2:3}, this follows in the same way as \eqref{eq:3}
in Lemma~\ref{lem:1}.

For \eqref{eq2:2}, we must argue differently than in 
Lemma~\ref{lem:1}. Let us write $\zeta=e^{2\pi i p/(m+1)h}$.
For a given group $W$, we write $S_1(W)$ for the set of all
indices $i$ such that $\zeta^{d_i-h}=1$, and we write
$S_2(W)$ for the set of all
indices $i$ such that $\zeta^{d_i}=1$. 
By the rule of de l'Hospital, we have
\begin{equation} \label{eq:S1S2}
\Cat^m(W;q)\big\vert_{q=e^{2\pi i p/(m+1)h}}=
\begin{cases}
0&\text{if }\vert S_1(W)\vert>\vert S_2(W)\vert,\\
\frac {\prod_{i\in S_1(W)}(mh+d_i)} 
{\prod_{i\in S_2(W)} d_i}
\frac {\prod_{i\notin S_1(W)}(1-\zeta^{d_i-h})} 
{\prod_{i\notin S_2(W)}(1-\zeta^{d_i})},
&\text{if }\vert S_1(W)\vert=\vert S_2(W)\vert.
\end{cases}
\end{equation}
Since, by Theorem~\ref{thm:0}, $\Cat^m(W;q)$ is a polynomial in $q$,
the case 
$\vert S_1(W)\vert<\vert S_2(W)\vert$ cannot occur.

We claim 
that, for the case where $\vert S_1(W)\vert=\vert S_2(W)\vert$, 
the factors in the quotient of products
$$
\frac {\prod_{i\notin S_1(W)}(1-\zeta^{d_i-h})} 
{\prod_{i\notin S_2(W)}(1-\zeta^{d_i})}
$$
cancel pairwise. If we assume the correctness of the claim,
it is obvious that we get the same result
if we replace $\zeta$ by $\zeta^k$, where $\gcd(k,(m+1)h/p)=1$,
hence establishing \eqref{eq2:2}.

In order to see that our claim is indeed valid,
we proceed in a case-by-case fashion,
making appeal to the classification of irreducible well-generated complex 
reflection groups, which we recalled in Section~\ref{sec:prel}.
First of all, since $d_n=h$, the set $S_1(W)$ is always non-empty
as it contains the element $n$. Hence, if we want to have
$\vert S_1(W)\vert=\vert S_2(W)\vert$, the set $S_2(W)$ must be
non-empty as well. In other words, the integer $(m+1)h/p$ must divide at least
one of the degrees $d_1,d_2,\dots,d_n$. In particular, this implies that,
for each fixed reflection group $W$ of exceptional type, only a finite number
of values of $(m+1)h/p$ has to be checked. Writing $M$ for $(m+1)h/p$,
what needs to be checked is whether the 
{\it multi}sets (that is, multiplicities of elements must be taken
into account)
$$
\{(d_i-h)\text{ mod }M:i\notin S_1(W)\}
\quad 
\text{and}
\quad 
\{d_i\text{ mod }M:i\notin S_2(W)\}
$$
are the same. Since, for a fixed irreducible well-generated complex
reflection group, there is only a finite number of possibilities for
$M$, this amounts to a routine verification.
\end{proof}

\begin{lemma} \label{lem2:2}
Let $p$ be a divisor of $(m+1)h$.
If $p$ is divisible by $m+1$, then \eqref{eq2:1} is true.
\end{lemma}

\begin{proof}
According to \eqref{eq2:Aktion}, the action of $\psi^p$ on
$NC^m(W)$ is described by
\begin{multline*}
\psi^p\big((w_0;w_1,\dots,w_m)\big)\\
=(c^{p/(m+1)}w_{0}c^{-p/(m+1)};c^{p/(m+1)}w_{1}c^{-p/(m+1)},\dots,
c^{p/(m+1)}w_{m}c^{-p/(m+1)}\big).
\end{multline*}
Hence, if $(w_0;w_1,\dots,w_m)$ is fixed by $\psi^p$, then
each individual $w_i$ must be fixed under conjugation by $c^{p/(m+1)}$.

Using the notation $W'=\Cent_W(c^{p/(m+1)})$, 
the previous observation means that $w_i\in W'$, 
$i=1,2,\dots,m$. By the theorem of Springer cited in the proof of
Lemma~\ref{lem:2} and by \eqref{eq:7},
the tuples $(w_0;w_1,\dots,w_m)$ fixed by $\psi^p$
are in fact identical with the elements of $NC^m(W')$, which
implies that
\begin{equation} \label{eq2:6}
\vert\Fix_{NC^m(W)}(\psi^{p})\vert=\vert NC^m(W')\vert.
\end{equation}
Application of
Theorem~\ref{thm:2} with $W$ replaced by $W'$ and of the
``limit rule" \eqref{eq:limit} then yields that
\begin{equation} \label{eq2:5}
\vert NC^m(W')\vert=
\underset{\frac {(m+1)h} p\mid d_i}{\prod_{1\le i\le n}} \frac {mh+d_i} {d_i}=\Cat^m(W;q)\big\vert_{q=e^{2\pi i p/(m+1)h}}.
\end{equation}
Combining \eqref{eq2:6} and \eqref{eq2:5}, we obtain \eqref{eq2:1}.
This finishes the proof of the lemma.
\end{proof}

\begin{lemma} \label{lem2:3}
Equation \eqref{eq2:1} holds for all divisors $p$ of $m+1$.
\end{lemma}

\begin{proof}
We have
$$
\Cat^m(W;q)\big\vert_{q=e^{2\pi i p/(m+1)h}}=
\begin{cases}
0&\text{if }p<m+1,\\
m+1&\text{if }p=m+1.
\end{cases}
$$
Here, the first case follows from \eqref{eq:S1S2} and the fact
that we have $S_1(W)\supseteq\{n\}$ 
and $S_2(W)=\emptyset$ if $p\mid (m+1)$ and
$p<m+1$.

On the other hand, if $(w_0;w_1,\dots,w_m)$ is fixed by $\psi^p$,
then, because of the action \eqref{eq2:Aktion}, we must have
$w_0=w_{p}=\dots=w_{m-p+1}$ and $w_0=cw_{m-p+1}c^{-1}$. 
In particular,
$w_0\in\Cent_W(c)$. By the theorem of Springer cited in the
proof of Lemma~\ref{lem:2}, the subgroup $\Cent_W(c)$
is itself a complex reflection group whose degrees are those degrees
of $W$ that are divisible by $h$. The only such degree is $h$
itself, hence $\Cent_W(c)$ is the cyclic group generated by
$c$. Moreover, by \eqref{eq:7}, we obtain that $w_0=\ep$ or
$w_0=c$. If $p=m+1$, the set $\Fix_{NC^m(W)}(\psi^p)$ consists of the 
$m+1$ elements $(w_0;w_1,\dots,w_m)$ obtained by choosing 
$w_i=c$ for a particular $i$ between $0$ and $m$, all other $w_j$'s 
being equal to $\ep$. If $p<m+1$, then
there is no element in $\Fix_{NC^m(W)}(\psi^p)$.
\end{proof}

\begin{lemma} \label{lem2:6}
Let $W$ be an irreducible 
well-generated complex reflection group of rank $n$, 
and let $p=m_1h_1$ be a divisor of $(m+1)h$, where $m+1=m_1m_2$ and
$h=h_1h_2$. We assume that 
$\gcd(h_1,m_2)=1$. Suppose that
Theorem~\ref{thm2:1} has already been verified for all 
irreducible well-generated
complex reflection groups with rank $<n$. 
If $h_2$ does not divide all degrees $d_i$,
then equation~\eqref{eq2:1} is satisfied.
\end{lemma}

\begin{proof}
Let us write $h_1=am_2+b$, with $0\le b<m_2$. The condition 
$\gcd(h_1,m_2)=1$ translates into $\gcd(b,m_2)=1$.
From \eqref{eq2:Aktion}, we infer that
\begin{multline} \label{eq2:m2Aktion}
\psi^p\big((w_0;w_1,\dots,w_m)\big)\\=
(c^{a+1}w_{m-m_1b+1}c^{-a-1};c^{a+1}w_{m-m_1b+2}c^{-a-1},
\dots,c^{a+1}w_{m}c^{-a-1},\\
c^aw_{0}c^{-a},\dots,
c^aw_{m-m_1b}c^{-a}\big).
\end{multline}
Supposing that 
$(w_0;w_1,\dots,w_m)$ is fixed by $\psi^p$, we obtain
the system of equations
\begin{align*} 
w_i&=c^{a+1}w_{i+m-m_1b+1}c^{-a-1}, \quad i=0,1,\dots,m_1b-1,\\
w_i&=c^aw_{i-m_1b}c^{-a}, \quad i=m_1b,m_1b+1,\dots,m,
\end{align*}
which, after iteration, implies in particular that
$$
w_i=c^{b(a+1)+(m_2-b)a}w_ic^{-b(a+1)-(m_2-b)a}=c^{h_1}w_ic^{-h_1},
\quad i=0,1,\dots,m.
$$
It is at this point where we need $\gcd(b,m_2)=1$.
The last equation shows that each $w_i$, $i=0,1,\dots,m$, 
lies in $\Cent_{W}(c^{h_1})$. 
By the theorem of Springer cited in the
proof of Lemma~\ref{lem:2}, this centraliser subgroup 
is itself a complex reflection group, $W'$ say, 
whose degrees are those degrees
of $W$ that are divisible by $h/h_1=h_2$. Since, by assumption, $h_2$
does not divide {\em all\/} degrees, $W'$ has 
rank strictly less than $n$. Again by assumption, we know that
Theorem~\ref{thm2:1} is true for $W'$, so that in particular,
$$
\vert\Fix_{NC^m(W')}(\psi^{p})\vert = 
\Cat^m(W';q)\big\vert_{q=e^{2\pi i p/(m+1)h}}.
$$
The arguments above together with \eqref{eq:7} show that 
$\Fix_{NC^m(W)}(\psi^{p})=\Fix_{NC^m(W')}(\psi^{p})$.
On the other hand, it is straightforward to see that
$$\Cat^m(W;q)\big\vert_{q=e^{2\pi i p/(m+1)h}}=
\Cat^m(W';q)\big\vert_{q=e^{2\pi i p/(m+1)h}}.$$
This proves \eqref{eq2:1} for our particular $p$, as required. 
\end{proof}

\begin{lemma} \label{lem2:7}
Let $W$ be an irreducible 
well-generated complex reflection group of rank $n$, 
and let $p=m_1h_1$ be a divisor of $(m+1)h$, where $m+1=m_1m_2$ and
$h=h_1h_2$. We assume that 
$\gcd(h_1,m_2)=1$. If $m_2>n$ then
$$
\Fix_{NC^m(W)}(\psi^{p})=\emptyset.
$$
\end{lemma}

\begin{proof}
Let us suppose that $(w_0;w_1,\dots,w_m)\in 
\Fix_{NC^m(W)}(\psi^{p})$ and that there exists a $j\ge1$ such
that $w_j\ne\ep$. By \eqref{eq2:m2Aktion}, it then follows
for such a $j$ that
also $w_k\ne\ep$ for all $k\equiv j-lm_1b$~(mod~$m+1$), where, as before,
$b$ is defined as the unique integer with $h_1=am_2+b$ and
$0\le b<m_2$. Since, by assumption, $\gcd(b,m_2)=1$, there are
exactly $m_2$ such $k$'s which are distinct mod~$m+1$. 
However, this implies that the sum of the absolute lengths
of the $w_i$'s, $0\le i\le m$, is at least $m_2>n$, a
contradiction. This leaves as only possibility 
$(w_0;w_1,\dots,w_m)=(c;\ep,\dots,\ep)$. However,
this is clearly not an element of $\Fix_{NC^m(W)}(\psi^{p})$
unless $p$ is divisible by $m+1$. This is impossible since
$$
\frac {p} {m+1}=\frac {m_1h_1} {m_1m_2}=\frac {h_1} {m_2}
$$
is not an integer by our hypotheses.
\end{proof}

\begin{remark} \label{rem2:1}
(1)
If we put ourselves in the situation of the assumptions of
Lemma~\ref{lem2:6}, then we may conclude that equation~\eqref{eq2:1} 
only needs to be checked for pairs $(m_2,h_2)$ subject to the 
following restrictions: 
\begin{equation} 
m_2\ge2,\quad 
\gcd(h_1,m_2)=1,\quad
\text{and $h_2$ divides all degrees of $W$}. 
\label{eq2:restr}
\end{equation}
Indeed, Lemmas~\ref{lem2:2} and \ref{lem2:6} 
together
imply that equation~\eqref{eq2:1} is always satisfied except if
$m_2\ge2$, $h_2$ divides all degrees of $W$, 
and $\gcd(h_1,m_2)=1$. 

\smallskip
(2) Still putting ourselves in the situation of Lemma~\ref{lem2:6},
if $m_2>n$ and $m_2h_2$ does not divide any of the degrees of $W$,
then equation~\eqref{eq2:1} is satisfied. Indeed, Lemma~\ref{lem2:7}
says that in this case the left-hand side of \eqref{eq2:1} equals
$0$, while it is obvious that in this case the right-hand side of 
\eqref{eq2:1} equals $0$ as well. 

\smallskip
(3) It should be observed that this leaves a finite
number of choices for $m_2$ to consider, whence a finite number of
choices for $(m_1,m_2,h_1,h_2)$. Altogether, there remains a finite
number of choices for $p=h_1m_1$ to be checked.
\end{remark}

\begin{lemma} \label{lem2:8}
Let $W$ be an irreducible 
well-generated complex reflection group of rank $n$ 
with the property that $d_i\mid h$ for $i=1,2,\dots,n$.
Then Theorem~{\em\ref{thm2:1}} is true for this group $W$.
\end{lemma}

\begin{proof}
By Lemma~\ref{lem2:1}, we may restrict ourselves to divisors
$p$ of $(m+1)h$. 

Suppose that $e^{2\pi ip/(m+1)h}$ is a $d_i$-th root of unity
for some $i$. In other words, $(m+1)h/p$ divides $d_i$.
Since $d_i$ is a divisor of $h$ by assumption, 
the integer $(m+1)h/p$ also divides $h$. But this is equivalent to
saying that $m+1$ divides $p$, and equation \eqref{eq2:1} holds by 
Lemma~\ref{lem2:2}.

Now assume that $(m+1)h/p$ does not divide any of the $d_i$'s.
In this case, it follows from \eqref{eq:S1S2} and the fact that
we have  $S_1(W)\supseteq\{n\}$ 
and $S_2(W)=\emptyset$ that
the right-hand side of \eqref{eq2:1} equals $0$. 
Inspection of the classification of all irreducible 
well-generated complex reflection groups, which we recalled in
Section~\ref{sec:prel}, reveals that all groups
satisfying the hypotheses of the lemma have rank $n\le2$.
Except for the groups contained in the infinite series $G(d,1,n)$ 
and $G(e,e,n)$ for which Theorem~\ref{thm:1} has been established in
\cite{KratCG}, these are the groups $G_5,G_6,G_9,G_{10},
G_{14},G_{17},G_{18},G_{21}$. We now discuss these groups case
by case, keeping the notation of Lemma~\ref{lem2:6}.
In order to simplify the argument, we note that Lemma~\ref{lem2:7}
implies that equation~\eqref{eq2:1} holds if $m_2>2$, so that
in the following arguments we always may assume that $m_2=2$.


\smallskip
{\sc Case $G_5$}. The degrees are $6,12$, and
therefore Remark~\ref{rem2:1}.(1) implies that equation~\eqref{eq2:1}
is always satisfied.

\smallskip
{\sc Case $G_6$}. The degrees are $4,12$, and
therefore, according to Remark~\ref{rem2:1}.(1), we need only consider
the case where $h_2=4$ and $m_2=2$, that is, $p=3(m+1)/2$. Then \eqref{eq2:m2Aktion} becomes
\begin{equation} \label{eq2:3m2Aktion}
\psi^p\big((w_0;w_1,\dots,w_m)\big)
=(c^{2}w_{\frac {m+1}2}c^{-2};
c^{2}w_{\frac {m+3}2}c^{-2},
\dots,c^{2}w_{m}c^{-2},
cw_{0}c^{-1},\dots,
cw_{\frac {m-1}2}c^{-1}\big).
\end{equation}
If $(w_0;w_1,\dots,w_m)$ is fixed by $\psi^p$, 
there must exist an $i$
with $0\le i\le \frac {m-1}2$ such that $\ell_T(w_i)=1$,
$w_icw_{i}c^{-1}=c$, and all $w_j$, $j\ne i,\frac {m+1}2+i$,
equal $\ep$. However, 
with the help of {\tt CHEVIE}, one verifies that there is no such
solution to this equation. Hence,
the left-hand side of \eqref{eq2:1} is equal to $0$, as required.

\smallskip
{\sc Case $G_9$}. The degrees are $8,24$, and
therefore, according to Remark~\ref{rem2:1}.(1), we need only consider
the case where $h_2=8$ and $m_2=2$, that is, $p=3(m+1)/2$. 
This is the same $p$ as for $G_6$. Again, 
{\tt CHEVIE} finds no solution. Hence,
the left-hand side of \eqref{eq2:1} is equal to $0$, as required.

\smallskip
{\sc Case $G_{10}$}. The degrees are $12,24$, and
therefore, according to Remark~\ref{rem2:1}.(1), we need only consider
the case where $h_2=12$ and $m_2=2$, that is, $p=3(m+1)/2$. 
This is the same $p$ as for $G_6$. Again, 
{\tt CHEVIE} finds no solution. Hence,
the left-hand side of \eqref{eq2:1} is equal to $0$, as required.

\smallskip
{\sc Case $G_{14}$}. The degrees are $6,24$, and
therefore Remark~\ref{rem2:1}.(1) implies that equation~\eqref{eq2:1}
is always satisfied.

\smallskip
{\sc Case $G_{17}$}. The degrees are $20,60$, and
therefore, according to Remark~\ref{rem2:1}.(1), we need only consider
the cases where $h_2=20$ and $m_2=2$, respectively that $h_2=4$
and $m_2=2$. In the first case, $p=3(m+1)/2$, which is
the same $p$ as for $G_6$. Again, 
{\tt CHEVIE} finds no solution. 
In the second case, $p=15(m+1)/2$. Then \eqref{eq2:m2Aktion} becomes
\begin{multline} \label{eq2:15m2Aktion}
\psi^p\big((w_0;w_1,\dots,w_m)\big)\\
=(c^{8}w_{\frac {m+1}2}c^{-8};
c^{8}w_{\frac {m+3}2}c^{-8},
\dots,c^{8}w_{m}c^{-8},
c^7w_{0}c^{-7},\dots,
c^7w_{\frac {m-1}2}c^{-7}\big).
\end{multline}
By Lemma~\ref{lem:4}, every element of $NC(W)$ is fixed under 
conjugation by $c^3$, and, thus, on elements fixed by $\psi^p$,
the above action of $\psi^p$
reduces to the one in \eqref{eq2:3m2Aktion}. This action was already
discussed in the first case.
Hence, in both cases,
the left-hand side of \eqref{eq2:1} is equal to $0$, as required.

\smallskip
{\sc Case $G_{18}$}. The degrees are $30,60$, and
therefore Remark~\ref{rem2:1}.(1) implies that equation~\eqref{eq2:1}
is always satisfied.

\smallskip
{\sc Case $G_{21}$}. The degrees are $12,60$, and
therefore, according to Remark~\ref{rem2:1}.(1), we need only consider
the cases where $h_2=5$ and $m_2=2$, respectively that $h_2=15$
and $m_2=2$. In the first case, $p=5(m+1)/2$, 
so that \eqref{eq2:m2Aktion} becomes
\begin{multline} \label{eq2:5m2Aktion}
\psi^p\big((w_0;w_1,\dots,w_m)\big)\\
=(c^{3}w_{\frac {m+1}2}c^{-3};
c^{3}w_{\frac {m+3}2}c^{-3},
\dots,c^{3}w_{m}c^{-3},
c^2w_{0}c^{-2},\dots,
c^2w_{\frac {m-1}2}c^{-2}\big).
\end{multline}
If $(w_0;w_1,\dots,w_m)$ is fixed by $\psi^p$, 
there must exist an $i$
with $0\le i\le \frac {m-1}2$ such that $\ell_T(w_i)=1$ and
$w_ic^2w_{i}c^{-2}=c$. However, 
with the help of {\tt CHEVIE}, one verifies that there is no such
solution to this equation. 
In the second case, $p=15(m+1)/2$. Then \eqref{eq2:m2Aktion} becomes
the action in \eqref{eq2:15m2Aktion}.
By Lemma~\ref{lem:4}, every element of $NC(W)$ is fixed under 
conjugation by $c^5$, and, thus, on elements fixed by $\psi^p$,
the action of $\psi^p$ in
\eqref{eq2:15m2Aktion} reduces to the one in the first case. 
Hence, in both cases,
the left-hand side of \eqref{eq2:1} is equal to $0$, as required.

\smallskip
This completes the proof of the lemma.
\end{proof}

\section{Case-by-case verification of Theorem~\ref{thm2:1}}
\label{sec:Beweis2}

We now perform a case-by-case verification of Theorem~\ref{thm2:1}.
It should be observed that the action of $\psi$ (given in 
\eqref{eq:psi}) is exactly the same as the action of $\phi$
(given in \eqref{eq:phi}) with $m$ replaced by $m+1$ 
{\it on the components $w_1,w_2,\dots,w_{m+1}$}, that is,
if we disregard the $0$-th component of the elements of the
generalised non-crossing partitions involved. The only difference
which arises is that, while the $(m+1)$-tuples 
$(w_0;w_1,\dots,w_m)$ in \eqref{eq:psi} must satisfy 
$w_0w_1\cdots w_m=c$, for  $w_1,w_2,\dots,w_{m+1}$ in \eqref{eq:phi}
we only must have $w_1w_2\cdots w_{m+1}\le_T c$. The condition
for $(w_0;w_1,\dots,w_m)$ of being in $\Fix_{NC^m(W)}(\psi^p)$
is therefore exactly the same as the condition on
$w_1,w_2,\dots,w_{m+1}$ for the element
$(\ep;w_1,\dots,w_m,w_{m+1})$ being in $\Fix_{NC^{m+1}(W)}(\phi^p)$.
Consequently, we may use the counting results from 
Section~\ref{sec:Beweis1}, except that we have to restrict our
attention to those elements $(w_0;w_1,\dots,w_m,w_{m+1})\in 
NC^{m+1}(W)$ for which $w_1w_2\cdots w_{m+1}=c$, or, equivalently,
$w_0=\ep$.

As before, we write $\zeta_d$ for a primitive $d$-th root of 
unity.

\subsection*{\sc Case $G_4$}
The degrees are $4,6$, and hence we have
$$
\Cat^m(G_4;q)=\frac 
{[6m+6]_q\, [6m+4]_q} 
{[6]_q\, [4]_q} .
$$
Let $\zeta$ be a $6(m+1)$-th root of unity. 
As before, in what follows
we abbreviate the assertion that ``$\zeta$ is a primitive $d$-th root of
unity" as ``$\zeta=\zeta_d$."
The following cases on the right-hand side of \eqref{eq2:1}
occur:
{\refstepcounter{equation}\label{eq2:G4}}
\alphaeqn
\begin{align} 
\label{eq2:G4.2}
\lim_{q\to\zeta}\Cat^m(G_4;q)&=m+1,
\quad\text{if }\zeta=\zeta_6,\zeta_3,\\
\label{eq2:G4.3}
\lim_{q\to\zeta}\Cat^m(G_4;q)&=\tfrac {3m+3}2,
\quad\text{if }\zeta=\zeta_4,\ 2\mid (m+1),\\
\label{eq2:G4.4}
\lim_{q\to\zeta}\Cat^m(G_4;q)&=\Cat^m(G_4),
\quad\text{if }\zeta=-1\text{ or }\zeta=1,\\
\label{eq2:G4.1}
\lim_{q\to\zeta}\Cat^m(G_4;q)&=0,
\quad\text{otherwise.}
\end{align}
\reseteqn

We must now prove that the left-hand side of \eqref{eq2:1} in
each case agrees with the values exhibited in 
\eqref{eq2:G4}. The only cases not covered by
Lemmas~\ref{lem2:2} and \ref{lem2:3} are the ones in \eqref{eq2:G4.3}
and \eqref{eq2:G4.1}. On the other hand, the only case left
to consider according to Remark~\ref{rem2:1} is
the case where $h_2=m_2=2$, that is the case \eqref{eq2:G4.3} where
$p=3(m+1)/2$. In particular, $m+1$ must be divisible by $2$.
The action of $\psi^p$ is the same as the one in 
\eqref{eq2:3m2Aktion}. Hence, the counting problem is the same as
there, except that the underlying group now is $G_4$. 
With the help of {\tt CHEVIE}, one finds that each of the
$3$ (complex) reflections in $G_4$ which are less than the (chosen) 
Coxeter element is a valid choice for $w_i$,
and each of these choices gives rise to $(m+1)/2$ elements in
$\Fix_{NC^m(G_4)}(\psi^p)$
since the index $i$ ranges from $0$ to $(m-1)/2$. 

Hence, in total, we obtain 
$3\frac {m+1}2=\frac {3m+3}2$ elements in
$\Fix_{NC^m(G_4)}(\psi^p)$, which agrees with the limit in
\eqref{eq2:G4.3}.

\subsection*{\sc Case $G_8$}
The degrees are $8,12$, and hence we have
$$
\Cat^m(G_8;q)=\frac 
{[12m+12]_q\, [12m+8]_q} 
{[12]_q\, [8]_q} .
$$
Let $\zeta$ be a $12(m+1)$-th root of unity. 
The following cases on the right-hand side of \eqref{eq2:1}
occur:
{\refstepcounter{equation}\label{eq2:G8}}
\alphaeqn
\begin{align} 
\label{eq2:G8.2}
\lim_{q\to\zeta}\Cat^m(G_8;q)&=m+1,
\quad\text{if }\zeta=\zeta_{12},\zeta_6,\zeta_3,\\
\label{eq2:G8.3}
\lim_{q\to\zeta}\Cat^m(G_8;q)&=\tfrac {3m+3}2,
\quad\text{if }\zeta=\zeta_8,\ 2\mid (m+1),\\
\label{eq2:G8.4}
\lim_{q\to\zeta}\Cat^m(G_8;q)&=\Cat^m(G_8),
\quad\text{if }\zeta=\zeta_4,-1,1,\\
\label{eq2:G8.1}
\lim_{q\to\zeta}\Cat^m(G_8;q)&=0,
\quad\text{otherwise.}
\end{align}
\reseteqn

We must now prove that the left-hand side of \eqref{eq2:1} in
each case agrees with the values exhibited in 
\eqref{eq2:G8}. The only cases not covered by
Lemmas~\ref{lem2:2} and \ref{lem2:3} are the ones in \eqref{eq2:G8.3}
and \eqref{eq2:G8.1}. On the other hand, the only case left
to consider according to Remark~\ref{rem2:1} is
the case where $h_2=4$ and $m_2=2$, that is the case \eqref{eq2:G8.3} where
$p=3(m+1)/2$. In particular, $m+1$ must be divisible by $2$.
The action of $\psi^p$ is the same as the one in 
\eqref{eq2:3m2Aktion}. Hence, the counting problem is the same as
there, except that the underlying group now is $G_8$. 
With the help of {\tt CHEVIE}, one finds that each of the
$3$ (complex) reflections in $G_8$ which are less than the (chosen) 
Coxeter element is a valid choice for $w_i$,
and each of these choices gives rise to $(m+1)/2$ elements in
$\Fix_{NC^m(G_8)}(\psi^p)$
since the index $i$ ranges from $0$ to $(m-1)/2$. 

Hence, in total, we obtain 
$3\frac {m+1}2=\frac {3m+3}2$ elements in
$\Fix_{NC^m(G_8)}(\psi^p)$, which agrees with the limit in
\eqref{eq2:G8.3}.

\subsection*{\sc Case $G_{16}$}
The degrees are $20,30$, and hence we have
$$
\Cat^m(G_{16};q)=\frac 
{[30m+30]_q\, [30m+20]_q} 
{[30]_q\, [20]_q} .
$$
Let $\zeta$ be a $30(m+1)$-th root of unity. 
The following cases on the right-hand side of \eqref{eq2:1}
occur:
{\refstepcounter{equation}\label{eq2:G16}}
\alphaeqn
\begin{align} 
\label{eq2:G16.2}
\lim_{q\to\zeta}\Cat^m(G_{16};q)&=m+1,
\quad\text{if }\zeta=\zeta_{30},\zeta_{15},\zeta_6,\zeta_3,\\
\label{eq2:G16.3}
\lim_{q\to\zeta}\Cat^m(G_{16};q)&=\tfrac {3m+3}2,
\quad\text{if }\zeta=\zeta_{20},\zeta_4,\ 2\mid (m+1),\\
\label{eq2:G16.4}
\lim_{q\to\zeta}\Cat^m(G_{16};q)&=\Cat^m(G_{16}),
\quad\text{if }\zeta=\zeta_{10},\zeta_5,-1,1,\\
\label{eq2:G16.1}
\lim_{q\to\zeta}\Cat^m(G_{16};q)&=0,
\quad\text{otherwise.}
\end{align}
\reseteqn

We must now prove that the left-hand side of \eqref{eq2:1} in
each case agrees with the values exhibited in 
\eqref{eq2:G16}. The only cases not covered by
Lemmas~\ref{lem2:2} and \ref{lem2:3} are the ones in \eqref{eq2:G16.3}
and \eqref{eq2:G16.1}. On the other hand, the only cases left
to consider according to Remark~\ref{rem2:1} are
the cases where $h_2=10$ and $m_2=2$, respectively $h_2=m_2=2$.
Both cases belong to \eqref{eq2:G16.3}. In the first case,
we have $p=3(m+1)/2$, while in the second case we have 
$p=15(m+1)/2$. 
In particular, $m+1$ must be divisible by $2$. In the first case,
the action of $\psi^p$ is the same as the one in 
\eqref{eq2:3m2Aktion}. Hence, the counting problem is the same as
there, except that the underlying group now is $G_{16}$. 
With the help of {\tt CHEVIE}, one finds that each of the
$3$ (complex) reflections in $G_{16}$ which are less than the (chosen) 
Coxeter element is a valid choice for $w_i$,
and each of these choices gives rise to $(m+1)/2$ elements in
$\Fix_{NC^m(G_{16})}(\psi^p)$
since the index $i$ ranges from $0$ to $(m-1)/2$. 
On the other hand, if $p=15(m+1)/2$, 
then the action of $\psi^p$ is the same as the one in \eqref{eq2:15m2Aktion}. 
By Lemma~\ref{lem:4}, every element of $NC(G_{16})$ is fixed under 
conjugation by $c^3$, and, thus, on elements fixed by $\psi^p$,
the action of $\psi^p$
reduces to the one in the first case. 

Hence, in total, we obtain 
$3\frac {m+1}2=\frac {3m+3}2$ elements in
$\Fix_{NC^m(G_{16})}(\psi^p)$, which agrees with the limit in
\eqref{eq2:G16.3}.

\subsection*{\sc Case $G_{20}$}
The degrees are $12,30$, and hence we have
$$
\Cat^m(G_{20};q)=\frac 
{[30m+30]_q\, [30m+12]_q} 
{[30]_q\, [12]_q} .
$$
Let $\zeta$ be a $30(m+1)$-th root of unity. 
The following cases on the right-hand side of \eqref{eq2:1}
occur:
{\refstepcounter{equation}\label{eq2:G20}}
\alphaeqn
\begin{align} 
\label{eq2:G20.2}
\lim_{q\to\zeta}\Cat^m(G_{20};q)&=m+1,
\quad\text{if }\zeta=\zeta_{30},\zeta_{15},\zeta_{10},\zeta_5,\\
\label{eq2:G20.3}
\lim_{q\to\zeta}\Cat^m(G_{20};q)&=\tfrac {5m+5}2,
\quad\text{if }\zeta=\zeta_{12},\zeta_4,\ 2\mid (m+1),\\
\label{eq2:G20.4}
\lim_{q\to\zeta}\Cat^m(G_{20};q)&=\Cat^m(G_{20}),
\quad\text{if }\zeta=\zeta_6,\zeta_3,-1,1,\\
\label{eq2:G20.1}
\lim_{q\to\zeta}\Cat^m(G_{20};q)&=0,
\quad\text{otherwise.}
\end{align}
\reseteqn

We must now prove that the left-hand side of \eqref{eq2:1} in
each case agrees with the values exhibited in 
\eqref{eq2:G20}. The only cases not covered by
Lemmas~\ref{lem2:2} and \ref{lem2:3} are the ones in \eqref{eq2:G20.3}
and \eqref{eq2:G20.1}. On the other hand, the only cases left
to consider according to Remark~\ref{rem2:1} are
the cases where $h_2=6$ and $m_2=2$, respectively $h_2=m_2=2$.
Both cases belong to \eqref{eq2:G20.3}. In the first case,
we have $p=5(m+1)/2$, while in the second case we have 
$p=15(m+1)/2$. 
In particular, $m+1$ must be divisible by $2$. In the first case,
the action of $\psi^p$ is the same as the one in 
\eqref{eq2:5m2Aktion}. Hence, the counting problem is the same as
there, except that the underlying group now is $G_{20}$. 
With the help of {\tt CHEVIE}, one finds that each of the
$5$ (complex) reflections in $G_{20}$ which are less than the (chosen) 
Coxeter element is a valid choice for $w_i$,
and each of these choices gives rise to $(m+1)/2$ elements in
$\Fix_{NC^m(G_{20})}(\psi^p)$
since the index $i$ ranges from $0$ to $(m-1)/2$. 
On the other hand, if $p=15(m+1)/2$, then the action of $\psi^p$ is 
the same as the one in \eqref{eq2:15m2Aktion}. 
By Lemma~\ref{lem:4}, every element of $NC(G_{20})$ is fixed under 
conjugation by $c^5$, and, thus, on elements fixed by $\psi^p$,
the action of $\psi^p$
reduces to the one in the first case. 

Hence, in total, we obtain 
$5\frac {m+1}2=\frac {5m+5}2$ elements in
$\Fix_{NC^m(G_{20})}(\psi^p)$, which agrees with the limit in
\eqref{eq2:G20.3}.

\subsection*{\sc Case $G_{23}=H_3$}
The degrees are $2,6,10$, and hence we have
$$
\Cat^m(H_3;q)=\frac 
{[10m+10]_q\, [10m+6]_q\, [10m+2]_q} 
{[10]_q\, [6]_q\, [2]_q} .
$$
Let $\zeta$ be a $10(m+1)$-th root of unity. 
The following cases on the right-hand side of \eqref{eq2:1}
occur:
{\refstepcounter{equation}\label{eq2:H3}}
\alphaeqn
\begin{align} 
\label{eq2:H3.2}
\lim_{q\to\zeta}\Cat^m(H_3;q)&=m+1,
\quad\text{if }\zeta=\zeta_{10},\zeta_5,\\
\label{eq2:H3.3}
\lim_{q\to\zeta}\Cat^m(H_3;q)&=\tfrac {5m+5}3,
\quad\text{if }\zeta= \zeta_{6},\zeta_3,\ 3\mid (m+1),\\
\label{eq2:H3.4}
\lim_{q\to\zeta}\Cat^m(H_3;q)&=\Cat^m(H_3),
\quad\text{if }\zeta=-1\text{ or }\zeta=1,\\
\label{eq2:H3.1}
\lim_{q\to\zeta}\Cat^m(H_3;q)&=0,
\quad\text{otherwise.}
\end{align}
\reseteqn

We must now prove that the left-hand side of \eqref{eq2:1} in
each case agrees with the values exhibited in 
\eqref{eq2:H3}. The only cases not covered by
Lemmas~\ref{lem2:2} and \ref{lem2:3} are the ones 
in \eqref{eq2:H3.3} and \eqref{eq2:H3.1}.
On the other hand, the only cases left
to consider according to Remark~\ref{rem2:1} are
the cases where $h_2=1$ and $m_2=3$, $h_2=2$
and $m_2=3$, and $h_2=m_2=2$. 
These correspond to the choices $p=10(m+1)/3$,
$p=5(m+1)/3$, respectively $p=5(m+1)/2$.
The first two cases belong to \eqref{eq2:H3.3}, while
$p=5(m+1)/2$ belongs to \eqref{eq2:H3.1}.

In the case that $p=5(m+1)/3$, the action of $\psi^p$ is given by
\begin{multline*}
\psi^p\big((w_0;w_1,\dots,w_m)\big)\\
=(c^{2}w_{\frac {m+1}3}c^{-2};c^{2}w_{\frac {m+4}3}c^{-2},
\dots,c^{2}w_{m}c^{-2},
cw_{0}c^{-1},\dots,
cw_{\frac {m-2}3}c^{-1}\big).
\end{multline*}
Hence, for an $i$ with $0\le i\le \frac {m-2} {3}$,
we must find an element $w_i=t_1$, where $t_1$ satisfies
\eqref{eq:H3D}, and all other $w_j$, $j\notin\big\{i,i+\frac {m+1} {3},
i+\frac {2(m+1)} {3}\big\}$, are set equal to $\ep$.
We have found
five solutions to the counting problem \eqref{eq:H3D} in
\eqref{eq:H3sol1}.
Each of them gives rise to $(m+1)/3$ elements in
$\Fix_{NC^m(H_3)}(\psi^p)$
since the index $i$ ranges from $0$ to $(m-2)/3$. 
On the other hand, if $p=10(m+1)/3$, then the action of $\psi^p$ is given by
\begin{multline*}
\psi^p\big((w_0;w_1,\dots,w_m)\big)\\
=(c^{4}w_{\frac {2m+2}3}c^{-4};c^{4}w_{\frac {2m+5}3}c^{-4},
\dots,c^{4}w_{m}c^{-4},
c^3w_{0}c^{-3},\dots,
c^3w_{\frac {2m-1}3}c^{-3}\big).
\end{multline*}
By Lemma~\ref{lem:4}, every element of $NC(H_3)$ is fixed under 
conjugation by $c^5$, and, thus, on elements fixed by $\psi^p$,
the action of $\psi^p$
reduces to the one in the first case. 

Hence, in total, we obtain 
$5\frac {m+1}3=\frac {5m+5}3$ elements in
$\Fix_{NC^m(H_3)}(\psi^p)$, which agrees with the limit in
\eqref{eq2:H3.3}.

If $p=5(m+1)/2$, then the action of $\psi^p$ is 
the same as the one in \eqref{eq2:5m2Aktion}.
The computation at the end of Case~$H_3$
in Section~\ref{sec:Beweis1} did not find any solutions,
which is in agreement with \eqref{eq2:H3.1}.

\subsection*{\sc Case $G_{24}$}
The degrees are $4,6,14$, and hence we have
$$
\Cat^m(G_{24};q)=\frac 
{[14m+14]_q\, [14m+6]_q\, [14m+4]_q} 
{[14]_q\, [6]_q\, [4]_q} .
$$
Let $\zeta$ be a $14(m+1)$-th root of unity. 
The following cases on the right-hand side of \eqref{eq2:1}
occur:
{\refstepcounter{equation}\label{eq2:G24}}
\alphaeqn
\begin{align} 
\label{eq2:G24.2}
\lim_{q\to\zeta}\Cat^m(G_{24};q)&=m+1,
\quad\text{if }\zeta=\zeta_{14},\zeta_7,\\
\label{eq2:G24.3}
\lim_{q\to\zeta}\Cat^m(G_{24};q)&=\tfrac {7m+7}3,
\quad\text{if }\zeta=\zeta_{6},\zeta_3,\ 3\mid (m+1),\\
\label{eq2:G24.5}
\lim_{q\to\zeta}\Cat^m(G_{24};q)&=\Cat^m(G_{24}),
\quad\text{if }\zeta=-1\text{ or }\zeta=1,\\
\label{eq2:G24.1}
\lim_{q\to\zeta}\Cat^m(G_{24};q)&=0,
\quad\text{otherwise.}
\end{align}
\reseteqn

We must now prove that the left-hand side of \eqref{eq2:1} in
each case agrees with the values exhibited in 
\eqref{eq2:G24}. The only cases not covered by
Lemmas~\ref{lem2:2} and \ref{lem2:3} are the ones in \eqref{eq2:G24.3}
and \eqref{eq2:G24.1}. 
On the other hand, the only cases left
to consider according to Remark~\ref{rem2:1} are 
the cases where $h_2=1$ and $m_2=3$, $h_2=2$ and $m_2=3$, 
and $h_2=m_2=2$. 
These correspond to the choices $p=14(m+1)/3$,
$p=7(m+1)/3$, respectively $p=7(m+1)/2$.
The first two cases belong to \eqref{eq2:G24.3}, while
$p=7(m+1)/2$ belongs to \eqref{eq2:G24.1}.

In the case that $p=14(m+1)/3$ or $p=7(m+1)/3$, we have found
seven solutions to the counting problem \eqref{eq:G24D} in
\eqref{eq:G24sol1},
and each of them gives rise to $(m+1)/3$ elements in
$\Fix_{NC^m(G_{24})}(\psi^p)$
(in the style as discussed in Case~$H_3$).
Hence, in total, we obtain 
$7\frac {m+1}3=\frac {7m+7}3$ elements in
$\Fix_{NC^m(G_{24})}(\psi^p)$, which agrees with the limit in
\eqref{eq2:G24.3}.

If $p=7(m+1)/2$, the relevant counting problem is \eqref{eq:G24''D}.
However, no element
$(w_0;w_1,\dots,w_m)\in \Fix_{NC^m(G_{24})}(\psi^p)$ can be
produced in this way since the counting problem imposes the
restriction that $\ell_T(w_0)+\ell_T(w_1)+\dots+\ell_T(w_m)$
be even, which contradicts the fact that $\ell_T(c)=n=3$. 
This is in agreement with the limit in
\eqref{eq2:G24.1}.

\subsection*{\sc Case $G_{25}$}
The degrees are $6,9,12$, and hence we have
$$
\Cat^m(G_{25};q)=\frac 
{[12m+12]_q\, [12m+9]_q\, [12m+6]_q} 
{[12]_q\, [9]_q\, [6]_q} .
$$
Let $\zeta$ be a $12(m+1)$-th root of unity. 
The following cases on the right-hand side of \eqref{eq2:1}
occur:
{\refstepcounter{equation}\label{eq2:G25}}
\alphaeqn
\begin{align} 
\label{eq2:G25.2}
\lim_{q\to\zeta}\Cat^m(G_{25};q)&=m+1,
\quad\text{if }\zeta=\zeta_{12},\zeta_4,\\
\label{eq2:G25.3}
\lim_{q\to\zeta}\Cat^m(G_{25};q)&=\tfrac {4m+4}3,
\quad\text{if }\zeta=\zeta_{9},\ 3\mid (m+1),\\
\label{eq2:G25.4}
\lim_{q\to\zeta}\Cat^m(G_{25};q)&=(m+1)(2m+1),
\quad\text{if }\zeta=\zeta_6,-1\\
\label{eq2:G25.5}
\lim_{q\to\zeta}\Cat^m(G_{25};q)&=\Cat^m(G_{25}),
\quad\text{if }\zeta=\zeta_3,1,\\
\label{eq2:G25.1}
\lim_{q\to\zeta}\Cat^m(G_{25};q)&=0,
\quad\text{otherwise.}
\end{align}
\reseteqn

We must now prove that the left-hand side of \eqref{eq2:1} in
each case agrees with the values exhibited in 
\eqref{eq2:G25}. The only cases not covered by
Lemmas~\ref{lem2:2} and \ref{lem2:3} are the ones in \eqref{eq2:G25.3}
and \eqref{eq2:G25.1}. 
On the other hand, the only case left
to consider according to Remark~\ref{rem2:1} is the case
where $h_2=m_2=3$. This corresponds to the choice $p=4(m+1)/3$, 
which belongs to \eqref{eq2:G25.3}.
We have found
four solutions to the counting problem \eqref{eq:G25D} in
\eqref{eq:G25sol1},
and each of them gives rise to $(m+1)/3$ elements in
$\Fix_{NC^m(G_{25})}(\psi^p)$
(in the style as discussed in Case~$H_3$).
Hence, in total, we obtain 
$4\frac {m+1}3=\frac {4m+4}3$ elements in
$\Fix_{NC^m(G_{25})}(\psi^p)$, which agrees with the limit in
\eqref{eq2:G25.3}.

\subsection*{\sc Case $G_{26}$}
The degrees are $6,12,18$, and hence we have
$$
\Cat^m(G_{26};q)=\frac 
{[18m+18]_q\, [18m+12]_q\, [18m+6]_q} 
{[18]_q\, [12]_q\, [6]_q} .
$$
Let $\zeta$ be a $14(m+1)$-th root of unity. 
The following cases on the right-hand side of \eqref{eq2:1}
occur:
{\refstepcounter{equation}\label{eq2:G26}}
\alphaeqn
\begin{align} 
\label{eq2:G26.2}
\lim_{q\to\zeta}\Cat^m(G_{26};q)&=m+1,
\quad\text{if }\zeta=\zeta_{18},\zeta_9,\\
\label{eq2:G26.5}
\lim_{q\to\zeta}\Cat^m(G_{26};q)&=\Cat^m(G_{26}),
\quad\text{if }\zeta=\zeta_6,\zeta_3,-1,1,\\
\label{eq2:G26.1}
\lim_{q\to\zeta}\Cat^m(G_{26};q)&=0,
\quad\text{otherwise.}
\end{align}
\reseteqn

We must now prove that the left-hand side of \eqref{eq2:1} in
each case agrees with the values exhibited in 
\eqref{eq2:G26}. The only case not covered by
Lemmas~\ref{lem2:2} and \ref{lem2:3} is the one in \eqref{eq2:G26.1}. 
On the other hand, the only cases left
to consider according to Remark~\ref{rem2:1} are the cases
where $h_2=6$ and $m_2=2$, respectively $h_2=m_2=2$. 
These correspond to the choices $p=3(m+1)/2$,
respectively $p=9(m+1)/2$, both of which belong to 
\eqref{eq2:G26.1}. The relevant counting problem is
\eqref{eq:G26D}. However, no element
$(w_0;w_1,\dots,w_m)\in \Fix_{NC^m(G_{26})}(\psi^p)$ can be
produced in this way since the counting problem imposes the
restriction that $\ell_T(w_0)+\ell_T(w_1)+\dots+\ell_T(w_m)$
be even, which is absurd. This is in agreement with the limit in
\eqref{eq2:G26.1}.

\subsection*{\sc Case $G_{27}$}
The degrees are $6,12,30$, and hence we have
$$
\Cat^m(G_{27};q)=\frac 
{[30m+30]_q\, [30m+12]_q\, [30m+6]_q} 
{[30]_q\, [12]_q\, [6]_q} .
$$
Let $\zeta$ be a $14(m+1)$-th root of unity. 
The following cases on the right-hand side of \eqref{eq2:1}
occur:
{\refstepcounter{equation}\label{eq2:G27}}
\alphaeqn
\begin{align} 
\label{eq2:G27.2}
\lim_{q\to\zeta}\Cat^m(G_{27};q)&=m+1,
\quad\text{if }\zeta=\zeta_{30},\zeta_{15},\zeta_{10},\zeta_5,\\
\label{eq2:G27.4}
\lim_{q\to\zeta}\Cat^m(G_{27};q)&=\Cat^m(G_{27}),
\quad\text{if }\zeta=\zeta_6,\zeta_3,-1,1,\\
\label{eq2:G27.1}
\lim_{q\to\zeta}\Cat^m(G_{27};q)&=0,
\quad\text{otherwise.}
\end{align}
\reseteqn

We must now prove that the left-hand side of \eqref{eq2:1} in
each case agrees with the values exhibited in 
\eqref{eq2:G27}. The only case not covered by
Lemmas~\ref{lem2:2} and \ref{lem2:3} is the one in \eqref{eq2:G27.1}.
On the other hand, the only cases left
to consider according to Remark~\ref{rem2:1} are the cases
where $h_2=6$ and $m_2=3$, $h_2=m_2=3$, $h_2=6$ and $m_2=2$, 
respectively $h_2=m_2=2$. These correspond to the choices
$p=5(m+1)/3$, $10(m+1)/3$, $5(m+1)/2$, respectively $15(m+1)/2$, 
all of which belong to \eqref{eq2:G27.1}.

If $p=5(m+1)/3$ or $p=10(m+1)/3$, the computation with the
help of {\tt CHEVIE}
at the end of Case~$G_{27}$
in Section~\ref{sec:Beweis1} did not find any solutions
for the corresponding counting problem. This is in agreement with
the limit in \eqref{eq2:G27.1}.

In the case that $5(m+1)/2$ or $15(m+1)/2$, 
the relevant counting problem is
\eqref{eq:G27D}. However, no element
$(w_0;w_1,\dots,w_m)\in \Fix_{NC^m(G_{27})}(\psi^p)$ can be
produced in this way since the counting problem imposes the
restriction that $\ell_T(w_0)+\ell_T(w_1)+\dots+\ell_T(w_m)$
be even, which is absurd. This is again 
in agreement with the limit in \eqref{eq2:G27.1}.

\subsection*{\sc Case $G_{28}=F_4$}
The degrees are $2,6,8,12$, and hence we have
$$
\Cat^m(F_4;q)=\frac 
{[12m+12]_q\, [12m+8]_q\, [12m+6]_q\, [12m+2]_q} 
{[12]_q\, [8]_q\, [6]_q\, [2]_q} .
$$
Let $\zeta$ be a $12(m+1)$-th root of unity. 
The following cases on the right-hand side of \eqref{eq2:1}
occur:
{\refstepcounter{equation}\label{eq2:F4}}
\alphaeqn
\begin{align} 
\label{eq2:F4.2}
\lim_{q\to\zeta}\Cat^m(F_4;q)&=m+1,
\quad\text{if }\zeta=\zeta_{12},\\
\label{eq2:F4.3}
\lim_{q\to\zeta}\Cat^m(F_4;q)&=\tfrac {3m+3}2,
\quad\text{if }\zeta=\zeta_{8},\ 2\mid (m+1),\\
\label{eq2:F4.6}
\lim_{q\to\zeta}\Cat^m(F_4;q)&=(m+1)(2m+1),
\quad\text{if }\zeta= \zeta_{6},\zeta_{3},\\
\label{eq2:F4.5}
\lim_{q\to\zeta}\Cat^m(F_4;q)&=\tfrac {(m+1)(3m+2)}2,
\quad\text{if }\zeta= \zeta_{4},\\
\label{eq2:F4.8}
\lim_{q\to\zeta}\Cat^m(F_4;q)&=\Cat^m(F_4),
\quad\text{if }\zeta=-1\text{ or }\zeta=1,\\
\label{eq2:F4.1}
\lim_{q\to\zeta}\Cat^m(F_4;q)&=0,
\quad\text{otherwise.}
\end{align}
\reseteqn

We must now prove that the left-hand side of \eqref{eq2:1} in
each case agrees with the values exhibited in 
\eqref{eq2:F4}. 
The only cases not covered by
Lemmas~\ref{lem2:2} and \ref{lem2:3} are the ones in 
\eqref{eq2:F4.3} and \eqref{eq2:F4.1}.
On the other hand, according to Remark~\ref{rem2:1},
the are no choices for $h_2$ and $m_2$ left to be considered.

\subsection*{\sc Case $G_{29}$}
The degrees are $4,8,12,20$, and hence we have
$$
\Cat^m(G_{29};q)=\frac 
{[20m+20]_q\, [20m+12]_q\, [20m+8]_q\, [20m+4]_q} 
{[20]_q\, [12]_q\, [8]_q\, [4]_q} .
$$
Let $\zeta$ be a $20(m+1)$-th root of unity. 
The following cases on the right-hand side of \eqref{eq2:1}
occur:
{\refstepcounter{equation}\label{eq2:G29}}
\alphaeqn
\begin{align} 
\label{eq2:G29.2}
\lim_{q\to\zeta}\Cat^m(G_{29};q)&=m+1,
\quad\text{if }\zeta=\zeta_{20},\zeta_{10},\zeta_5,\\
\label{eq2:G29.5}
\lim_{q\to\zeta}\Cat^m(G_{29};q)&=\Cat^m(G_{29}),
\quad\text{if }\zeta=\zeta_4,-1,1,\\
\label{eq2:G29.1}
\lim_{q\to\zeta}\Cat^m(G_{29};q)&=0,
\quad\text{otherwise.}
\end{align}
\reseteqn

We must now prove that the left-hand side of \eqref{eq2:1} in
each case agrees with the values exhibited in 
\eqref{eq2:G29}. The only case not covered by
Lemmas~\ref{lem2:2} and \ref{lem2:3} is the one in \eqref{eq2:G29.1}.
On the other hand, the only cases left
to consider according to Remark~\ref{rem2:1},
the only choices for $h_2$ and $m_2$ to be considered
are $h_2=1$ and $m_2=3$, $h_2=2$ and $m_2=3$, 
$h_2=4$ and $m_2=3$, $h_2=4$ and $m_2=2$, 
respectively $h_2=m_2=4$. 
These correspond to the choices $p=20(m+1)/3$, $p=10(m+1)/3$, 
$p=5(m+1)/3$, $p=5(m+1)/2$, 
respectively $p=5(m+1)/4$, all of which belong to \eqref{eq2:G29.1}.

In the case that $p=20(m+1)/3$, $p=10(m+1)/3$, or $p=5(m+1)/3$,
the relevant counting problem is
\eqref{eq:G29D}. However, no element
$(w_0;w_1,\dots,w_m)\in \Fix_{NC^m(G_{27})}(\psi^p)$ can be
produced in this way since the counting problem imposes the
restriction that $\ell_T(w_0)+\ell_T(w_1)+\dots+\ell_T(w_m)$
be divisible by $3$, which is absurd. This is 
in agreement with the limit in \eqref{eq2:G29.1}.

In the case that $p=5(m+1)/2$,
the relevant counting problem is
\eqref{eq:G29'D}, for which we did not find any solutions.
This is again
in agreement with the limit in \eqref{eq2:G29.1}.

In the case that $p=5(m+1)/4$, 
the computation at the end of Case~$G_{29}$
in Section~\ref{sec:Beweis1} did not find any solutions,
which is as well in agreement with the limit in \eqref{eq2:G29.1}.

\subsection*{\sc Case $G_{30}=H_4$}
The degrees are $2,12,20,30$, and hence we have
$$
\Cat^m(H_4;q)=\frac 
{[30m+30]_q\, [30m+20]_q\, [30m+12]_q\, [30m+2]_q} 
{[30]_q\, [20]_q\, [12]_q\, [2]_q} .
$$
Let $\zeta$ be a $30(m+1)$-th root of unity. 
The following cases on the right-hand side of \eqref{eq2:1}
occur:
{\refstepcounter{equation}\label{eq2:H4}}
\alphaeqn
\begin{align} 
\label{eq2:H4.2}
\lim_{q\to\zeta}\Cat^m(H_4;q)&=m+1,
\quad\text{if }\zeta=\zeta_{30},\zeta_{15},\\
\label{eq2:H4.3}
\lim_{q\to\zeta}\Cat^m(H_4;q)&=\tfrac {3m+3}2,
\quad\text{if }\zeta=\zeta_{20},\ 2\mid (m+1),\\
\label{eq2:H4.4}
\lim_{q\to\zeta}\Cat^m(H_4;q)&=\tfrac {5m+5}2,
\quad\text{if }\zeta=\zeta_{12},\ 2\mid (m+1),\\
\label{eq2:H4.5}
\lim_{q\to\zeta}\Cat^m(H_4;q)&=\tfrac {(m+1)(3m+2)}2,
\quad\text{if }\zeta= \zeta_{10},\zeta_{5},\\
\label{eq2:H4.6}
\lim_{q\to\zeta}\Cat^m(H_4;q)&=\tfrac {(m+1)(5m+2)}2,
\quad\text{if }\zeta= \zeta_{6},\zeta_{3},\\
\label{eq2:H4.7}
\lim_{q\to\zeta}\Cat^m(H_4;q)&=\tfrac {(m+1)(15m+1)}4,
\quad\text{if }\zeta= \zeta_{4},\ 2\mid (m+1),\\
\label{eq2:H4.8}
\lim_{q\to\zeta}\Cat^m(H_4;q)&=\Cat^m(H_4),
\quad\text{if }\zeta=-1\text{ or }\zeta=1,\\
\label{eq2:H4.1}
\lim_{q\to\zeta}\Cat^m(H_4;q)&=0,
\quad\text{otherwise.}
\end{align}
\reseteqn

We must now prove that the left-hand side of \eqref{eq2:1} in
each case agrees with the values exhibited in 
\eqref{eq2:H4}. The only cases not covered by
Lemmas~\ref{lem2:2} and \ref{lem2:3} are the ones in 
\eqref{eq2:H4.3}, \eqref{eq2:H4.4}, \eqref{eq2:H4.7},
and \eqref{eq2:H4.1}.
On the other hand, the only cases left
to consider according to Remark~\ref{rem2:1} are the cases
where $h_2=2$ and $m_2=4$, respectively $h_2=m_2=2$. 
These correspond to the choices $p=15(m+1)/2$,
respectively $p=15(m+1)/4$, out of which the first  
belongs to \eqref{eq2:H4.7}, while the second belongs to
\eqref{eq2:H4.1}.

In the case that $p=15(m+1)/2$, the action of $\psi^p$ is the same as
the one in \eqref{eq2:15m2Aktion}.
We have found
eight solutions to the counting problem \eqref{eq:H4''D} in
\eqref{eq:H4sol3}, 
each of them giving rise to $(m+1)/2$ elements in
$\Fix_{NC^m(H_4)}(\psi^p)$
since the index $i$ (in \eqref{eq:H4''D}) 
ranges from $0$ to $(m-1)/2$,
and we have found
30 solutions to the counting problem \eqref{eq:H4''E} in
\eqref{eq:H4sol4}, 
each of them giving rise to $\binom {(m+1)/2}2$ elements in
$\Fix_{NC^m(H_4)}(\psi^p)$
since $0\le i_1<i_2\le (m-1)/2$
(in \eqref{eq:H4''E}).  Hence, we obtain 
$8\frac {m+1}2+30\binom {(m+1)/2}2=\frac {(m+1)(15m+1)}4$ elements in
$\Fix_{NC^m(H_4)}(\psi^p)$, which agrees with the limit in
\eqref{eq2:H4.7}.

If $p=15(m+1)/4$, the computation at the end of Case~$H_4$
in Section~\ref{sec:Beweis1} did not find any solutions,
which is in agreement with the limit in \eqref{eq2:H4.1}.

\subsection*{\sc Case $G_{32}$}
The degrees are $12,18,24,30$, and hence we have
$$
\Cat^m(G_{32};q)=\frac 
{[30m+30]_q\, [30m+24]_q\, [30m+18]_q\, [30m+12]_q} 
{[30]_q\, [24]_q\, [18]_q\, [12]_q} .
$$
Let $\zeta$ be a $30(m+1)$-th root of unity. 
The following cases on the right-hand side of \eqref{eq2:1}
occur:
{\refstepcounter{equation}\label{eq2:G32}}
\alphaeqn
\begin{align} 
\label{eq2:G32.2}
\lim_{q\to\zeta}\Cat^m(G_{32};q)&=m+1,
\quad\text{if }\zeta=\zeta_{30},\zeta_{15},\zeta_{10},\zeta_5,\\
\label{eq2:G32.3}
\lim_{q\to\zeta}\Cat^m(G_{32};q)&=\tfrac {5m+5}4,
\quad\text{if }\zeta=\zeta_{24},\zeta_8,\ 4\mid (m+1),\\
\label{eq2:G32.5}
\lim_{q\to\zeta}\Cat^m(G_{32};q)&=\tfrac {(5m+5)(5m+3)}{8},
\quad\text{if }\zeta=\zeta_{12},\zeta_{4},\ 2\mid (m+1),\\
\label{eq2:G32.6}
\lim_{q\to\zeta}\Cat^m(G_{32};q)&=\Cat^m(G_{32}),
\quad\text{if }\zeta=\zeta_6,\zeta_3,-1,1,\\
\label{eq2:G32.1}
\lim_{q\to\zeta}\Cat^m(G_{32};q)&=0,
\quad\text{otherwise.}
\end{align}
\reseteqn

We must now prove that the left-hand side of \eqref{eq2:1} in
each case agrees with the values exhibited in 
\eqref{eq2:G32}. The only cases not covered by
Lemmas~\ref{lem2:2} and \ref{lem2:3} are the ones in 
\eqref{eq2:G32.3}, 
\eqref{eq2:G32.5}, 
and \eqref{eq2:G32.1}.
On the other hand, the only cases left
to consider according to Remark~\ref{rem2:1} are the cases
where $h_2=2$ and $m_2=4$,   
$h_2=6$ and $m_2=4$, $h_2=m_2=3$, $h_2=6$ and $m_2=3$, 
$h_2=m_2=2$, respectively $h_2=6$ and $m_2=2$. 
These correspond to the choices $p=15(m+1)/4$,  
$p=5(m+1)/4$, $p=10(m+1)/3$, $p=5(m+1)/3$, $p=15(m+1)/2$,
respectively $p=5(m+1)/2$, out of which the first two belong
to \eqref{eq2:G32.3}, the next two belong to \eqref{eq2:G32.1},
and the last two belong to \eqref{eq2:G32.5}. 

In the case that $p=15(m+1)/4$ or $p=5(m+1)/4$, we have found
five solutions to the counting problem \eqref{eq:G32D} in
\eqref{eq:G32sol1},
each of them giving rise to $(m+1)/4$ elements in
$\Fix_{NC^m(G_{32})}(\psi^p)$.
Hence, we obtain 
$5\frac {m+1}4=\frac {5m+5}4$ elements in
$\Fix_{NC^m(G_{32})}(\psi^p)$, which agrees with the limit in
\eqref{eq2:G32.3}.

In the case that $p=10(m+1)/3$ or $p=5(m+1)/3$, 
the relevant counting problem is
\eqref{eq:G32''''D}. However, no element
$(w_0;w_1,\dots,w_m)\in \Fix_{NC^m(G_{32})}(\psi^p)$ can be
produced in this way since the counting problem imposes the
restriction that $\ell_T(w_0)+\ell_T(w_1)+\dots+\ell_T(w_m)$
be divisible by $3$, which is absurd. This is 
in agreement with the limit in \eqref{eq2:G32.1}.

In the case that $p=15(m+1)/2$ or $p=5(m+1)/2$, we have found
ten solutions to the counting problem \eqref{eq:G32'D} in
\eqref{eq:G32sol2}, 
each of them giving rise to $(m+1)/2$ elements in
$\Fix_{NC^m(G_{32})}(\psi^p)$,
and we have found
25 solutions to the counting problem \eqref{eq:G32'E} in
\eqref{eq:G32sol3}, 
each of them giving rise to $\binom {(m+1)/2}2$ elements in
$\Fix_{NC^m(G_{32})}(\psi^p)$.
Hence, we obtain 
$10\frac {m+1}2+25\binom {(m+1)/2}2=\frac {(5m+5)(5m+3)}8$ 
elements in
$\Fix_{NC^m(G_{32})}(\psi^p)$, which agrees with the limit in
\eqref{eq2:G32.5}.

\subsection*{\sc Case $G_{33}$}
The degrees are $4,6,10,12,18$, and hence we have
$$
\Cat^m(G_{33};q)=\frac 
{[18m+18]_q\, [18m+12]_q\, [18m+10]_q\, [18m+6]_q\, [18m+4]_q} 
{[18]_q\, [12]_q\, [10]_q\, [6]_q\, [4]_q} .
$$
Let $\zeta$ be a $18(m+1)$-th root of unity. 
The following cases on the right-hand side of \eqref{eq2:1}
occur:
{\refstepcounter{equation}\label{eq2:G33}}
\alphaeqn
\begin{align} 
\label{eq2:G33.2}
\lim_{q\to\zeta}\Cat^m(G_{33};q)&=m+1,
\quad\text{if }\zeta=\zeta_{18},\zeta_9,\\
\label{eq2:G33.4}
\lim_{q\to\zeta}\Cat^m(G_{33};q)&=\tfrac {9m+9}5,
\quad\text{if }\zeta=\zeta_{10},\zeta_5,\ 5\mid (m+1),\\
\label{eq2:G33.5}
\lim_{q\to\zeta}\Cat^m(G_{33};q)&=\tfrac {(m+1)(3m+2)(3m+1)}{2},
\quad\text{if }\zeta=\zeta_{6},\zeta_{3},\\
\label{eq2:G33.7}
\lim_{q\to\zeta}\Cat^m(G_{33};q)&=\Cat^m(G_{33}),
\quad\text{if }\zeta=-1\text{ or }\zeta=1,\\
\label{eq2:G33.1}
\lim_{q\to\zeta}\Cat^m(G_{33};q)&=0,
\quad\text{otherwise.}
\end{align}
\reseteqn

We must now prove that the left-hand side of \eqref{eq2:1} in
each case agrees with the values exhibited in 
\eqref{eq2:G33}. The only cases not covered by
Lemmas~\ref{lem2:2} and \ref{lem2:3} are the ones in 
\eqref{eq2:G33.4} 
and \eqref{eq2:G33.1}.
On the other hand, the only cases left
to consider according to Remark~\ref{rem2:1} are the cases
where $h_2=1$ and $m_2=5$, $h_2=2$ and $m_2=5$, 
$h_2=2$ and $m_2=4$, 
respectively $h_2=m_2=2$. 
These correspond to the choices $p=18(m+1)/5$, $p=9(m+1)/5$, 
$p=9(m+1)/4$, respectively $p=9(m+1)/2$, out of which 
the first two belong to \eqref{eq2:G33.4},
while the others belong to \eqref{eq2:G33.1}.

In the case that $p=18(m+1)/5$ or $p=9(m+1)/5$, we have found
nine solutions to the counting problem \eqref{eq:G33D} in
\eqref{eq:G33sol1}. Hence, we obtain 
$9\frac {m+1}5=\frac {9m+9}5$ elements in
$\Fix_{NC^m(G_{33})}(\psi^p)$, which agrees with the limit in
\eqref{eq2:G33.4}.

If $p=9(m+1)/4$, 
the computation at the end of Case~$G_{33}$
in Section~\ref{sec:Beweis1} did not find any solutions,
which is again in agreement with the limit in \eqref{eq2:G32.1}.

In the case that $p=9(m+1)/2$, 
the relevant counting problems are \eqref{eq:G33''''D}
and \eqref{eq:G33''''E}. However, no element
$(w_0;w_1,\dots,w_m)\in \Fix_{NC^m(G_{33})}(\psi^p)$ can be
produced in this way since the counting problem imposes the
restriction that $\ell_T(w_0)+\ell_T(w_1)+\dots+\ell_T(w_m)$
be even, which is absurd. This is 
in agreement with the limit in \eqref{eq2:G33.1}.

\subsection*{\sc Case $G_{34}$}
The degrees are $6,12,18,24,30,42$, and hence we have
\begin{multline*}
\Cat^m(G_{34};q)=\frac 
{[42m+42]_q\, [42m+30]_q\, [42m+24]_q} 
{[42]_q\, [30]_q\, [24]_q}\\
\times
\frac {[42m+18]_q\,
 [42m+12]_q\, [42m+6]_q}
{[18]_q\, [12]_q\, [6]_q} .
\end{multline*}
Let $\zeta$ be a $42(m+1)$-th root of unity. 
The following cases on the right-hand side of \eqref{eq2:1}
occur:
{\refstepcounter{equation}\label{eq2:G34}}
\alphaeqn
\begin{align} 
\label{eq2:G34.2}
\lim_{q\to\zeta}\Cat^m(G_{34};q)&=m+1,
\quad\text{if }\zeta=\zeta_{42},\zeta_{21},\zeta_{14},\zeta_7,\\
\label{eq2:G34.7}
\lim_{q\to\zeta}\Cat^m(G_{34};q)&=\Cat^m(G_{34}),
\quad\text{if }\zeta=\zeta_6,\zeta_3,-1,1,\\
\label{eq2:G34.1}
\lim_{q\to\zeta}\Cat^m(G_{34};q)&=0,
\quad\text{otherwise.}
\end{align}
\reseteqn

We must now prove that the left-hand side of \eqref{eq2:1} in
each case agrees with the values exhibited in 
\eqref{eq2:G34}. The only case not covered by
Lemmas~\ref{lem2:2} and \ref{lem2:3} is the one in 
\eqref{eq2:G34.1}.
On the other hand, the only cases left
to consider according to Remark~\ref{rem2:1} are the cases
where $h_2=1$ and $m_2=5$, $h_2=2$ and $m_2=5$, 
$h_2=3$ and $m_2=5$, $h_2=6$ and $m_2=5$,
$h_2=2$ and $m_2=4$, $h_2=6$ and $m_2=4$, 
$h_2=m_2=3$, $h_2=6$ and $m_2=3$, 
$h_2=m_2=2$, 
$h_2=6$ and $m_2=2$,   
respectively $h_2=6$ and $m_2=6$. 
These correspond to the choices 
$p=42(m+1)/5$, $p=21(m+1)/5$, $p=14(m+1)/5$, $p=7(m+1)/5$, 
$p=21(m+1)/4$, $p=7(m+1)/4$, 
$p=14(m+1)/3$, $p=7(m+1)/3$, $p=21(m+1)/2$,  
$p=7(m+1)/2$, 
respectively $p=7(m+1)/6$, all of which belong to \eqref{eq2:G34.1}. 

In the case that $p=42(m+1)/5$, $p=21(m+1)/5$, 
$p=14(m+1)/5$, or $p=7(m+1)/5$, 
the relevant counting problem is \eqref{eq:G34D}.
However, no element
$(w_0;w_1,\dots,w_m)\in \Fix_{NC^m(G_{34})}(\psi^p)$ can be
produced in this way since the counting problem imposes the
restriction that $\ell_T(w_0)+\ell_T(w_1)+\dots+\ell_T(w_m)$
be divisible by $5$, which is absurd. This is 
in agreement with the limit in \eqref{eq2:G34.1}.

In the case that $p=21(m+1)/4$ or $p=7(m+1)/4$, 
the relevant counting problem is \eqref{eq:G34'D}.
However, no element
$(w_0;w_1,\dots,w_m)\in \Fix_{NC^m(G_{34})}(\psi^p)$ can be
produced in this way since the counting problem imposes the
restriction that $\ell_T(w_0)+\ell_T(w_1)+\dots+\ell_T(w_m)$
be divisible by $4$, which is absurd. This is 
in agreement with the limit in \eqref{eq2:G34.1}.

In the case that $p=14(m+1)/3$ or $p=7(m+1)/3$, 
the relevant counting problems are \eqref{eq:G34*'''D}
and \eqref{eq:G34*'''E}. However, the computations with the
help of {\tt CHEVIE}
performed in Case~$G_{34}$ in Section~\ref{sec:Beweis1}
did not find any solutions for \eqref{eq:G34*'''D}
or \eqref{eq:G34*'''E}. This is 
in agreement with the limit in \eqref{eq2:G34.1}.

In the case that $p=21(m+1)/2$, 
the relevant counting problems are \eqref{eq:G34''''D},
\eqref{eq:G34''''E}, and \eqref{eq:G34''''F}. However, 
the computations with the
help of {\tt CHEVIE}
performed in Case~$G_{34}$ in Section~\ref{sec:Beweis1}
found no $w_i$ with $\ell_T(w_i)=3$ 
in \eqref{eq:G34''''D}, and hence no solutions for
$(w_{i_1},w_{i_2})$ with $\ell_T(w_{i_1})+\ell_T(w_{i_2})=3$
in \eqref{eq:G34''''E}, and no solutions for
$(w_{i_1},w_{i_2},w_{i_3})$ in \eqref{eq:G34''''F}.
This is 
in agreement with the limit in \eqref{eq2:G34.1}.

If $p=7(m+1)/6$, the computation at the end of Case~$G_{34}$
in Section~\ref{sec:Beweis1} did not find any solutions, which
is also in agreement with the limit in \eqref{eq2:G34.1}.

\subsection*{\sc Case $G_{35}=E_6$}
The degrees are $2,5,6,8,9,12$, and hence we have
$$
\Cat^m(E_6;q)=\frac 
{[12m+12]_q\, [12m+9]_q\, [12m+8]_q\, [12m+6]_q\, [12m+5]_q\, [12m+2]_q} 
{[12]_q\, [9]_q\, [8]_q\, [6]_q\, [5]_q\, [2]_q} .
$$
Let $\zeta$ be a $12(m+1)$-th root of unity. 
The following cases on the right-hand side of \eqref{eq2:1}
occur:
{\refstepcounter{equation}\label{eq2:E6}}
\alphaeqn
\begin{align} 
\label{eq2:E6.2}
\lim_{q\to\zeta}\Cat^m(E_6;q)&=m+1,
\quad\text{if }\zeta=\zeta_{12},\\
\label{eq2:E6.3}
\lim_{q\to\zeta}\Cat^m(E_6;q)&=\tfrac {4m+4}3,
\quad\text{if }\zeta=\zeta_{9},\ 3\mid (m+1),\\
\label{eq2:E6.4}
\lim_{q\to\zeta}\Cat^m(E_6;q)&=\tfrac {3m+3}2,
\quad\text{if }\zeta=\zeta_{8},\ 2\mid (m+1),\\
\label{eq2:E6.7}
\lim_{q\to\zeta}\Cat^m(E_6;q)&=(m+1)(2m+1),
\quad\text{if }\zeta= \zeta_{6},\\
\label{eq2:E6.6}
\lim_{q\to\zeta}\Cat^m(E_6;q)&=\tfrac {(m+1)(3m+2)}2,
\quad\text{if }\zeta= \zeta_{4},\\
\label{eq2:E6.8}
\lim_{q\to\zeta}\Cat^m(E_6;q)&=\tfrac {(m+1)(4m+3)(2m+1)}3,
\quad\text{if }\zeta= \zeta_{3},\\
\label{eq2:E6.9}
\lim_{q\to\zeta}\Cat^m(E_6;q)&=\tfrac {(m+1)(3m+2)(2m+1)(6m+1)}2,
\quad\text{if }\zeta= -1,\\
\lim_{q\to\zeta}\Cat^m(E_6;q)&=\Cat^m(E_6),
\quad\text{if }\zeta=1,\\
\label{eq2:E6.1}
\lim_{q\to\zeta}\Cat^m(E_6;q)&=0,
\quad\text{otherwise.}
\end{align}
\reseteqn

We must now prove that the left-hand side of \eqref{eq2:1} in
each case agrees with the values exhibited in 
\eqref{eq2:E6}. The only cases not covered by
Lemmas~\ref{lem2:2} and \ref{lem2:3} are the ones in 
\eqref{eq2:E6.3}, \eqref{eq2:E6.4}, 
and \eqref{eq2:E6.1}.
On the other hand, the only cases left
to consider according to Remark~\ref{rem2:1} are the cases
where $h_2=1$ and $m_2=5$, respectively $h_2=2$ and $m_2=5$. 
These correspond to the choices $p=12(m+1)/5$,
respectively $p=6(m+1)/5$, both of which belong to \eqref{eq2:E6.1}.

In the case that $p=12(m+1)/5$, 
the relevant counting problem is \eqref{eq:E6''D}. 
However, no element
$(w_0;w_1,\dots,w_m)\in \Fix_{NC^m(E_{6})}(\psi^p)$ can be
produced in this way since the counting problem imposes the
restriction that $\ell_T(w_0)+\ell_T(w_1)+\dots+\ell_T(w_m)$
be divisible by $5$, which is absurd. This is 
in agreement with the limit in \eqref{eq2:E6.1}.

If $p=6(m+1)/5$, the computation at the end of Case~$E_6$
in Section~\ref{sec:Beweis1} did not find any solutions,
which is also in agreement with the limit in \eqref{eq2:E6.1}.

\subsection*{\sc Case $G_{36}=E_7$}
The degrees are $2,6,8,10,12,14,18$, and hence we have
\begin{multline*}
\Cat^m(E_7;q)=\frac 
{[18m+18]_q\, [18m+14]_q\, [18m+12]_q} 
{[18]_q\, [14]_q\, [12]_q}\\
\times
\frac 
{[18m+10]_q\, [18m+8]_q\, [18m+6]_q\, [18m+2]_q} 
{[10]_q\, [8]_q\, [6]_q\, [2]_q} .
\end{multline*}
Let $\zeta$ be a $18(m+1)$-th root of unity. 
The following cases on the right-hand side of \eqref{eq2:1}
occur:
{\refstepcounter{equation}\label{eq2:E7}}
\alphaeqn
\begin{align} 
\label{eq2:E7.2}
\lim_{q\to\zeta}\Cat^m(E_7;q)&=m+1,
\quad\text{if }\zeta=\zeta_{18},\zeta_{9},\\
\label{eq2:E7.3}
\lim_{q\to\zeta}\Cat^m(E_7;q)&=\tfrac {9m+9}7,
\quad\text{if }\zeta=\zeta_{14},\zeta_{7},\ 7\mid (m+1),\\
\label{eq2:E7.10}
\lim_{q\to\zeta}\Cat^m(E_7;q)&=\frac {(m+1)(3m+2)(3m+1)}2,
\quad\text{if }\zeta= \zeta_{6},\zeta_{3},\\
\lim_{q\to\zeta}\Cat^m(E_7;q)&=\Cat^m(E_7),
\quad\text{if }\zeta=-1\text{ or }\zeta=1,\\
\label{eq2:E7.1}
\lim_{q\to\zeta}\Cat^m(E_7;q)&=0,
\quad\text{otherwise.}
\end{align}
\reseteqn

We must now prove that the left-hand side of \eqref{eq2:1} in
each case agrees with the values exhibited in 
\eqref{eq2:E7}. The only cases not covered by
Lemmas~\ref{lem2:2} and \ref{lem2:3} are the ones in 
\eqref{eq2:E7.3} and \eqref{eq2:E7.1}.
On the other hand, the only cases left
to consider according to Remark~\ref{rem2:1} are the cases
where $h_2=1$ and $m_2=7$, $h_2=2$ and $m_2=7$,
$h_2=1$ and $m_2=5$, 
$h_2=2$ and $m_2=5$, 
$h_2=2$ and $m_2=4$, respectively $h_2=m_2=2$. 
These correspond to the choices $p=18(m+1)/7$, 
$p=9(m+1)/7$, $p=18(m+1)/5$,
$p=9(m+1)/5$, $p=9(m+1)/4$, respectively $p=9(m+1)/2$, 
out of which the first two belong to \eqref{eq2:E6.3}, 
and all others belong to \eqref{eq2:E6.1}.

In the case that $p=18(m+1)/7$ or $p=9(m+1)/7$, we have found
nine solutions to the counting problem \eqref{eq:E7'''D} in
\eqref{eq:E7sol1}. Hence, we obtain 
$9\frac {m+1}7=\frac {9m+9}7$ elements in
$\Fix_{NC^m(E_7)}(\psi^p)$, which agrees with the limit in
\eqref{eq2:E7.3}.

In the case that $p=18(m+1)/5$ or $p=9(m+1)/5$,
the relevant counting problem is \eqref{eq:E7D}. 
However, no element
$(w_0;w_1,\dots,w_m)\in \Fix_{NC^m(E_{7})}(\psi^p)$ can be
produced in this way since the counting problem imposes the
restriction that $\ell_T(w_0)+\ell_T(w_1)+\dots+\ell_T(w_m)$
be divisible by $5$, which is absurd. This is 
in agreement with the limit in \eqref{eq2:E7.1}.

In the case that $p=9(m+1)/4$,
the relevant counting problem is \eqref{eq:E7''D}. 
However, no element
$(w_0;w_1,\dots,w_m)\in \Fix_{NC^m(E_{7})}(\psi^p)$ can be
produced in this way since the counting problem imposes the
restriction that $\ell_T(w_0)+\ell_T(w_1)+\dots+\ell_T(w_m)$
be divisible by $4$, which is absurd.
This is again
in agreement with the limit in \eqref{eq2:E7.1}.

In the case that $p=9(m+1)/2$,
the relevant counting problems are \eqref{eq:E7''''D},
\eqref{eq:E7''''E}, and \eqref{eq:E7''''F}. 
However, no element
$(w_0;w_1,\dots,w_m)\in \Fix_{NC^m(E_{7})}(\psi^p)$ can be
produced in this way since the counting problem imposes the
restriction that $\ell_T(w_0)+\ell_T(w_1)+\dots+\ell_T(w_m)$
be even, which is absurd. This is also
in agreement with the limit in \eqref{eq2:E7.1}.

\subsection*{\sc Case $G_{37}=E_8$}
The degrees are $2,8,12,14,18,20,24,30$, and hence we have
\begin{multline*}
\Cat^m(E_8;q)=\frac 
{[30m+30]_q\, [30m+24]_q\, [30m+20]_q\, [30m+18]_q} 
{[30]_q\, [24]_q\, [20]_q\, [18]_q}\\
\times
\frac 
{[30m+14]_q\, [30m+12]_q\, [30m+8]_q\, [30m+2]_q} 
{[14]_q\, [12]_q\, [8]_q\, [2]_q} .
\end{multline*}
Let $\zeta$ be a $30(m+1)$-th root of unity. 
The following cases on the right-hand side of \eqref{eq2:1}
occur:
{\refstepcounter{equation}\label{eq2:E8}}
\alphaeqn
\begin{align} 
\label{eq2:E8.2}
\lim_{q\to\zeta}\Cat^m(E_8;q)&=m+1,
\quad\text{if }\zeta=\zeta_{30},\zeta_{15},\\
\label{eq2:E8.3}
\lim_{q\to\zeta}\Cat^m(E_8;q)&=\tfrac {5m+5}4,
\quad\text{if }\zeta=\zeta_{24},\ 4\mid (m+1),\\
\label{eq2:E8.4}
\lim_{q\to\zeta}\Cat^m(E_8;q)&=\tfrac {3m+3}2,
\quad\text{if }\zeta=\zeta_{20},\ 2\mid (m+1),\\
\label{eq2:E8.7}
\lim_{q\to\zeta}\Cat^m(E_8;q)&=\tfrac {(5m+5)(5m+3)}{8},
\quad\text{if }\zeta= \zeta_{12},\ 2\mid (m+1),\\
\label{eq2:E8.8}
\lim_{q\to\zeta}\Cat^m(E_8;q)&=\tfrac {(m+1)(3m+2)}2,
\quad\text{if }\zeta= \zeta_{10},\zeta_{5},\\
\label{eq2:E8.9}
\lim_{q\to\zeta}\Cat^m(E_8;q)&=\frac{(5m+5)(15m+7)}{16},
\quad\text{if }\zeta= \zeta_{8},\ 4\mid (m+1),\\
\label{eq2:E8.10}
\lim_{q\to\zeta}\Cat^m(E_8;q)&=\frac{(m+1)(5m+4)(5m+3)(5m+2)}{24},
\quad\text{if }\zeta= \zeta_{6},\zeta_3,\\
\label{eq2:E8.11}
\lim_{q\to\zeta}\Cat^m(E_8;q)&=\frac {(m+1)(5m+3)(15m+7)(15m+1)}{64},
\quad\text{if }\zeta= \zeta_{4},\ 2\mid (m+1),\\
\lim_{q\to\zeta}\Cat^m(E_8;q)&=\Cat^m(E_8),
\quad\text{if }\zeta=-1\text{ or }\zeta=1,\\
\label{eq2:E8.1}
\lim_{q\to\zeta}\Cat^m(E_8;q)&=0,
\quad\text{otherwise.}
\end{align}
\reseteqn

We must now prove that the left-hand side of \eqref{eq2:1} in
each case agrees with the values exhibited in 
\eqref{eq2:E8}. The only cases not covered by
Lemmas~\ref{lem2:2} and \ref{lem2:3} are the ones in 
\eqref{eq2:E8.3}, \eqref{eq2:E8.4}, 
\eqref{eq2:E8.7}, \eqref{eq2:E8.9}, 
\eqref{eq2:E8.11}, and \eqref{eq2:E8.1}.
On the other hand, the only cases left
to consider according to Remark~\ref{rem2:1} are the cases
where $h_2=2$ and $m_2=8$, $h_2=1$ and $m_2=7$,  
$h_2=2$ and $m_2=7$, $h_2=2$ and $m_2=4$, 
respectively $h_2=m_2=2$. These correspond to the choices
$p=15(m+1)/8$, $p=30(m+1)/7$, $p=15(m+1)/7$, $p=15(m+1)/4$, respectively
$15(m+1)/2$, out of which the first three belong to \eqref{eq2:E8.1},
the fourth
belongs to \eqref{eq2:E8.9}, and the last belongs to 
\eqref{eq2:E8.11}.

If $p=15(m+1)/8$, the relevant counting problem is \eqref{eq:8prod}. 
However, the computation at the end of Case~$E_8$
in Section~\ref{sec:Beweis1} did not find any solutions,
which is in agreement with the limit in \eqref{eq2:E8.1}.
Hence, 
the left-hand side of \eqref{eq2:1} is equal to $0$, as required.

In the case that $p=30(m+1)/7$ or $p=15(m+1)/7$,
the relevant counting problem is \eqref{eq:E8'''''D}. 
However, no element
$(w_0;w_1,\dots,w_m)\in \Fix_{NC^m(E_{8})}(\psi^p)$ can be
produced in this way since the counting problem imposes the
restriction that $\ell_T(w_0)+\ell_T(w_1)+\dots+\ell_T(w_m)$
be divisible by $7$, which is absurd. This is also
in agreement with the limit in \eqref{eq2:E8.1}.

In the case that $p=15(m+1)/4$, the relevant counting problems
are \eqref{eq:E8''D} and \eqref{eq:E8''E}. 
We have found
45 solutions $w_i$ to \eqref{eq:E8''D} of type $A_1^2$ in 
\eqref{eq:E8sol4}, and we have found 20 solutions $w_i$ to
\eqref{eq:E8''D} of type $A_2$ in
\eqref{eq:E8sol5}, which implied $150$ solutions for 
$(w_{i_1},w_{i_2})$
to \eqref{eq:E8''E}. The first two give rise to 
to $(45+20)\frac{m+1}4=65\frac{m+1}4$ elements in
$\Fix_{NC^m(E_{8})}(\psi^p)$,
while the third give rise to $150\binom {(m+1)/4}2$ elements in
$\Fix_{NC^m(E_{8})}(\psi^p)$.
Hence, we obtain 
$65\frac {m+1}4+150\binom {(m+1)/4}2=\frac {(5m+5)(15m+7)}{16}$ elements in
$\Fix_{NC^m(E_8)}(\psi^p)$, which agrees with the limit in
\eqref{eq2:E8.9}.

In the case that $p=15(m+1)/2$, the relevant counting problems
are \eqref{eq:E8''''D}, \eqref{eq:E8''''E}, \eqref{eq:E8''''F}, 
and \eqref{eq:E8''''G}. 
We have found
15 solutions $w_i$ to \eqref{eq:E8''''D} of type $A_1^2*A_2$ in 
\eqref{eq:E8sol6}, 
we have found 45 solutions $w_i$ to
\eqref{eq:E8''''D} of type $A_1*A_3$ in
\eqref{eq:E8sol7}, 
we have found 5 solutions $w_i$ to
\eqref{eq:E8''''D} of type $A_2^2$ in
\eqref{eq:E8sol8}, 
we have found 18 solutions $w_i$ to
\eqref{eq:E8''''D} of type $A_4$ in
\eqref{eq:E8sol9}, 
we have found 5 solutions $w_i$ to
\eqref{eq:E8''''D} of type $D_4$ in
\eqref{eq:E8sol10}, each giving rise to $(m+1)/2$ 
elements in $\Fix_{NC^m(E_{8})}(\psi^p)$.
Using the notation from there,
these imply $2n_{3,1}+n_{2,2}=2\cdot 660+1195=2515$ solutions for 
$(w_{i_1},w_{i_2})$ to \eqref{eq:E8''''E} with 
$\ell_T(w_{i_1})+\ell_T(w_{i_2})=4$, each giving rise to 
$\binom {(m+1)/2}2$ elements in $\Fix_{NC^m(E_{8})}(\psi^p)$.
They also imply $3n_{2,1,1}=3\cdot 2850=8550$ solutions for 
$(w_{i_1},w_{i_2},w_{i_3})$ to \eqref{eq:E8''''F} with 
$\ell_T(w_{i_1})+\ell_T(w_{i_2})+\ell_T(w_{i_3})=4$, 
each giving rise to $\binom {(m+1)/2}3$ elements in
$\Fix_{NC^m(E_{8})}(\psi^p)$.
Finally, they imply as well $n_{1,1,1,1}=6750$ solutions for 
$(w_{i_1},w_{i_2},w_{i_3},w_{i_4})$ to \eqref{eq:E8''''G}, 
each giving rise to $\binom {(m+1)/2}4$ elements in
$\Fix_{NC^m(E_{8})}(\psi^p)$.

In total, we obtain 
\begin{multline*}
(15+45+5+18+5)\frac {m+1}2+2515\binom {(m+1)/2}2+
8550\binom {(m+1)/2}3+6750\binom {(m+1)/2}4\\
=\frac {(m+1)(5m+3)(15m+7)(15m+1)}{64}
\end{multline*}
elements in
$\Fix_{NC^m(E_8)}(\psi^p)$, which agrees with the limit in
\eqref{eq2:E8.11}.

\section*{Acknowledgements}
The authors thank an anonymous referee for a very careful 
reading of the paper \cite{KrMuAE}, and for the many 
pertinent suggestions which
helped to improve that paper and also this manuscript considerably.


\begin{thebibliography} {10}

\bibitem{AndrAF}
G. E. Andrews, {\em The Theory of Partitions},
Encyclopedia of Math.\ and its Applications, vol.~2,
Addison--Wesley, Reading, 1976.

\bibitem{ArmDAA} D.~Armstrong, {\em Generalized noncrossing partitions and 
combinatorics of Coxeter groups}, 
Mem.\ Amer.\ Math.\ Soc., vol.~202, no.~949, Amer.\ Math.\ Soc.,
Providence, R.I., 2009.

\bibitem{ArSTAA}
D. Armstrong, C. Stump and H. Thomas,
{\em A uniform bijection between nonnesting and noncrossing partitions},
Trans.\ Amer.\ Math.\ Soc.\ (to appear).

\bibitem{AthaAG} C. A. Athanasiadis, {\em Generalized Catalan 
numbers, Weyl groups and arrangements of hyperplanes}, Bull.\
London Math.\ Soc.\ {\bf 36} (2004), 294--302.

\bibitem{AthaAH} C. A. Athanasiadis, {\em On a refinement of 
the generalized Catalan numbers for Weyl groups}, Trans.\ Amer.\
Math.\ Soc.\ {\bf 357} (2005), 179--196.

\bibitem{BesDAA} D.    Bessis, {\em The dual braid
monoid}, Ann.\ Sci.\ \'Ecole Norm.\ Sup.\ (4)  {\bf 36} (2003),
647--683. 

\bibitem{BesDAB} D.    Bessis, {\em Finite complex reflection groups
are $K(\pi,1)$}, preprint,  {\tt ar$\chi$iv:math/0610777}.

\bibitem{BeCoAA}
D. Bessis and R. Corran, {\em Non-crossing partitions of type
$(e,e,r)$},  Adv.\ Math.\ {\bf202} (2006), 1--49. 

\bibitem{BeReAA}
D. Bessis and V. Reiner, {\em Cyclic sieving and noncrossing 
partitions for complex reflection groups}, Ann.\ Comb.\
{\bf 15} (2011), 197--222. 


\bibitem{BRWaAA} T.    Brady and C. Watt, {\em
$K(\pi,1)$'s for Artin groups of finite type}, Geom.\ Dedicata  {\bf
94} (2002), 225--250. 


\bibitem{ChaFAA} F.    Chapoton, {\em Enumerative
properties of generalized associahedra}, S\'eminaire Lotharingien
Combin.\ {\bf 51} (2004), Article~B51b, 16~pp.  

\bibitem{EdelAA} P.    Edelman, {\em Chain enumeration and noncrossing partitions}, 
Discrete Math.\ {\bf 31} (1981), 171--180.

\bibitem{FoReAA} S.    Fomin and N. Reading, {\em
Generalized cluster complexes and Coxeter combinatorics}, 
Int.\ Math.\ Res.\ Notices {\bf 44} (2005), 2709--2757.

\bibitem{chevAA} M. Geck, G. Hiss, F. L\"ubeck, G. Malle and
G. Pfeiffer, {\em {\tt CHEVIE} --- a system for computing and processing
generic character tables for finite groups of Lie type},
Appl.\ Algebra Engrg.\ Comm.\ Comput.\ {\bf 7} (1996), 175--210.

\bibitem{GoGrAA}
I. Gordon and S. Griffeth, {\em Catalan numbers for complex reflection groups},
Amer.\ J. Math.\ (to appear).

\bibitem{HumpAC} J. E. Humphreys, {\em  Reflection groups
and Coxeter groups}, Cambridge University Press,
Cambridge, 1990. 

\bibitem{KratCB} C.    Krattenthaler,
{\em The $F$-triangle of the generalised cluster complex}, in:
Topics in Discrete Mathematics, 
dedicated to Jarik Ne\v set\v ril on the occasion of 
his 60th birthday, M.~Klazar, J.~Kratochvil, M.~Loebl, 
J.~Matou\v sek, R.~Thomas and P.~Valtr (eds.),
Springer--Verlag, Berlin, New York, 2006, pp.~93--126.

\bibitem{KratCF} C.    Krattenthaler,
{\em The $M$-triangle of generalised non-crossing partitions for the types 
$E_7$ and $E_8$}, S\'e\-mi\-naire Lotharingien Combin.\ {\bf 54}
(2006), Article~B54l, 34~pages.

\bibitem{KratCG} C.    Krattenthaler,
{\em Non-crossing partitions on an annulus}, in preparation.

\bibitem{KrMuAB} C. Krattenthaler and T. W. M\"uller,
{\em Decomposition numbers for finite Coxeter groups
and generalised non-crossing partitions}, 
Trans.\ Amer.\ Math.\ Soc.\ {\bf 362} (2010), 2723--2787.

\bibitem{KrMuAE} C. Krattenthaler and T. W. M\"uller,
{\em Cyclic sieving for generalised non-crossing partitions associated
with complex reflection groups of exceptional type},
in: W80, volume in memory of Herb Wilf,
I.~Kotsireas, E.~Zima (eds.), Springer--Verlag (to appear); 
{\tt ar$\chi$iv:1001.0028}.

\bibitem{KrewAC} G.    Kreweras, {\em Sur les partitions
non crois\'ees d'un cycle}, Discrete Math.\ {\bf 1} (1972),
333--350. 

\bibitem{LeMiAA} G. I. Lehrer and J. Michel,
{\em Invariant theory and eigenspaces for unitary reflection groups}, 
C. R. Math. Acad. Sci. Paris {\bf 336} (2003), 795--800.

\bibitem{LeTaAA} G. I. Lehrer and D. E. Taylor,
{\em Unitary reflection groups}, Cambridge University Press,
Cambridge, 2009.

\bibitem{LoehAA}
N. A. Loehr, {\em Conjectured statistics for the higher $q,t$-Catalan
sequences}, Electron.\ J. Combin.\ {\bf 12} (2005), Art.~\#R9, 54~pp.

\bibitem{MaMiAA}
G. Malle and J. Michel,
{\em Constructing representations of Hecke algebras for complex 
reflection groups}, LMS J. Comput. Math. {\bf 13} (2010), 426--450.

\bibitem{MariAA}
I. Marin, {\em The cubic Hecke algebra on at most 5 strands},
preprint, {\tt  ar$\chi$iv:1110.6621}.

\bibitem{MichAA}
J. Michel, {\em The {\sl GAP}-part of the {\tt CHEVIE} system},
{\sl GAP}~3-package available for
download from 
{\tt http://people.math.jussieu.fr/jmichel/chevie/chevie.html}.

\bibitem{OS} P.~Orlik and L.~Solomon, {\em Unitary reflection groups and 
cohomology}, Invent. Math. {\bf 59} (1980), 77--94.


\bibitem{ReSWAA}
V. Reiner, D. Stanton and D. White, {\em The cyclic sieving 
phenomenon}, J. Combin. Theory Ser.~A {\bf 108} (2004), 17--50.

\bibitem{RipoAA}
V. Ripoll, {\em Orbites d'Hurwitz des factorisations primitives d'un
\'el\'ement de Coxeter}, J. Algebra {\bf 323} (2010), 1432--1453.

\bibitem{ShToAA} 
G. C. Shephard and J. A. Todd, {\em Finite unitary reflection
groups}, Canad. J. Math. {\bf 6} (1954), 274--304.

\bibitem{SpriAA}
T. A. Springer, {\em Regular elements of finite reflection groups},
Invent. Math. {\bf 25} (1974), 159--198.

\bibitem{StemAL} J. R. Stembridge, {\em Some hidden relations
involving the  
ten symmetry classes of plane partitions}, J.~Combin.\ Theory Ser.~A
{\bf 68} (1994), 372--409.

\bibitem{StemAP} J. R. Stembridge, {\em Canonical bases and
self-evacuating  tableaux}, Duke Math.\ J. {\bf 82} (1996). 585--606,

\bibitem{StemAZ} J. R. Stembridge,  {\em coxeter}, {\sl
Maple} package for working with root systems and finite Coxeter groups;
available at {\tt http://www.math.lsa.umich.edu/\~{}jrs}.


\end{thebibliography}
\end{document}
